\pdfoutput=1
\documentclass{book}         

\usepackage{savesym}
\usepackage{amsthm,amsmath}
\usepackage{amssymb,latexsym,graphicx}
\usepackage{amscd}
 \usepackage[all,cmtip]{xy}
\usepackage{tikz-cd}
\usepackage{accents}
\usepackage{cite}
\usepackage{mathtools}
\usepackage{stmaryrd} 

\usepackage{calligra}
\DeclareMathAlphabet{\mathcalligra}{T1}{calligra}{m}{n}
\DeclareMathAlphabet{\mathpzc}{OT1}{pzc}{m}{it}
\usepackage{hycolor}
\usepackage{xcolor}
\usepackage[
                       colorlinks=true,
                       linkcolor=black, 
                       citecolor=black, 
                       urlcolor=blue,
                     ]{hyperref}
\usepackage{soul} 

\usepackage{multibib}
\newcites{intro}{References}        
\newcites{symptop}{References}  
\newcites{refFH}{References}        
\newcites{refCG}{References}    
\newcites{refRF}{References}     

\usepackage{makeidx}
\makeindex

\usepackage{stackengine} 
\usepackage{booktabs} 

\newtheorem{theorem}{Theorem}[section]
\newtheorem{corollary}[theorem]{Corollary}

\newtheorem{lemma}[theorem]{Lemma}
\newtheorem{proposition}[theorem]{Proposition}
\newtheorem{conjecture}[theorem]{Conjecture}
\theoremstyle{definition}
\newtheorem{definition}[theorem]{Definition}

\newtheorem{assumption}[theorem]{Assumption}

\newtheorem{remark}[theorem]{Remark}
\newtheorem{exercise}[theorem]{Exercise}
\newtheorem{example}[theorem]{Example}
\newtheorem{NOTATION}[theorem]{Notation}

\theoremstyle{remark}

\newenvironment{dedication}
  {\clearpage           
   \thispagestyle{empty}
   \vspace*{\stretch{1}}
   \itshape             
   \raggedleft          
  }
  {\par 
   \vspace{\stretch{3}} 
   \clearpage           
  }

\newcommand{\comp}[1]{#1^{\rm C}}  
\newcommand{\C}{{\mathbb{C}}}
\newcommand{\D}{{\mathbb{D}}}

\newcommand{\N}{{\mathbb{N}}}

\newcommand{\Q}{{\mathbb{Q}}}
\newcommand{\R}{{\mathbb{R}}}
\renewcommand{\SS}{{\mathbb{S}}}
\newcommand{\T}{{\mathbb{T}}}
\newcommand{\Z}{{\mathbb{Z}}}
\newcommand{\Aa}{{\mathcal{A}}}   
\newcommand{\Bb}{{\mathcal{B}}}
\newcommand{\Cc}{{\mathcal{C}}}   

\newcommand{\Ee}{{\mathcal{E}}}
\newcommand{\Ff}{{\mathcal{F}}}
\newcommand{\Hh}{{\mathcal{H}}}

\newcommand{\Jj}{{\mathcal{J}}}

\newcommand{\Ll}{{\mathcal{L}}}   
\newcommand{\Mm}{{\mathcal{M}}}   

\newcommand{\Oo}{{\mathcal{O}}}
\newcommand{\Pp}{{\mathcal{P}}}

\newcommand{\Rr}{{\mathcal{R}}}
\newcommand{\Ss}{{\mathcal{S}}}
\newcommand{\Tt}{{\mathcal{T}}}
\newcommand{\Uu}{{\mathcal{U}}}
\newcommand{\Vv}{{\mathcal{V}}}

\newcommand{\Xx}{{\mathcal{X}}}

\newcommand{\EEE}{\mbf{E}}       

\newcommand{\AAA}{\mbf{A}}       
\newcommand{\AAAdot}{\mbf{\dot A}} 
\newcommand{\BBB}{\mbf{B}}       
\newcommand{\rot}{\mathrm{rot}}  
\newcommand{\rrr}{\mbf{r}}       %
\newcommand{\Crm}{{\mathrm{C}}}

\newcommand{\Ann}{{\rm Ann }}            
\newcommand{\pann}[1]{#1^\perp}        
\newcommand{\Bdual}[1]{#1^\prime}     
\newcommand{\Hdual}[1]{#1^*}              
\newcommand{\coker}{{\rm coker\, }}  
\newcommand{\im}{{\rm im\, }}             
\newcommand{\dom}{{\rm dom\, }}           
\newcommand{\DIV}{{\rm div\, }}           
\newcommand{\sign}{{\rm sign\, }}         
\newcommand{\id}{{\rm id}}                
\newcommand{\Id}{{\rm Id}}
\newcommand{\rank}{{\rm rank}}            
\newcommand{\codim}{{\rm codim\, }}       
\newcommand{\diag}{{\rm diag}}            
\newcommand{\dist}{{\rm dist}}            
\newcommand{\INT}{{\rm int\, }}           
\newcommand{\supp}{{\rm supp\, }}         
\newcommand{\grad}{\mathop{\mathrm{grad}}}    
\newcommand{\Fred}{\Ff}              
\newcommand{\INDEX}{\mathop{\mathrm{index}}}     
\newcommand{\IND}{{\rm ind}}                  
\newcommand{\INDsign}{{\rm ind}^\sigma}       
\newcommand{\CZ}{{\mu_{\rm CZ}}}              
\newcommand{\CZcan}{{\mu^{\rm CZ}}}           
\newcommand{\RStrans}{{\mu^{\rm RS}_\xi}}     
\newcommand{\CZtrans}{{\mu^{\rm CZ}_\xi}}     
\newcommand{\mucan}{{\mu_{\rm can}}}          
\newcommand{\RS}{{\mu_{\rm RS}}}              

\newcommand{\LVF}{{Y}}               
\newcommand{\Lflow}{\theta}          
\newcommand{\LVFcan}{{Y_{\rm can}}}  
\newcommand{\LVFfrad}{{Y_{\rm rad}}} 
\newcommand{\LM}{{\tau}}             
\newcommand{\LMpath}{{\eta}}         
\newcommand{\Ffreg}{\Ff_{\rm reg}}   
\newcommand{\Uureg}{\Uu_{\rm reg}}   
\newcommand{\pertau}{{\tau}}         
\newcommand{\lpz}{{z}}                   
\newcommand{\natinc}{\iota}               
\newcommand{\Ppall}{\Pp_{\rm all}}   
\newcommand{\LIP}{{\rm LIP}}         
\newcommand{\RVF}{{R}}               
\newcommand{\Rflow}{\vartheta}       
\newcommand{\ror}{{r}}               
\newcommand{\rorDUMMY}{{r}}       
\newcommand{\GRO}{\Pp_{\mspace{-3mu}\pm}} 
\newcommand{\RFH}{{\rm RFH}}         
\newcommand{\mean}{{\rm mean}}       
\newcommand{\ppgamma}{\gamma_{\rm prime}} 

\newcommand{\PhiPSShom}{\Phi^{\mathrm{PSS}}}
\newcommand{\phiPSShom}{\phi^{\mathrm{PSS}}}
\newcommand{\PsiPSShom}{\Psi^{\mathrm{PSS}}}
\newcommand{\psiPSShom}{\psi^{\mathrm{PSS}}}

\newcommand{\SF}{{\mu^{\rm spec}}} 
 
\newcommand{\wind}{{\rm wind}}     
\newcommand{\NULL}{{\rm null}}       

\newcommand{\SsPp}{{\mathcal{SP}}}        
\newcommand{\Per}{{\rm Per}}          
\newcommand{\Fix}{{\rm Fix}}          
\newcommand{\Crit}{{\rm Crit}}        
\newcommand{\Hom}{{\rm Hom}}          
\newcommand{\End}{{\rm End}}          
\newcommand{\SB}{{\rm SB}}            
\newcommand{\Ho}{{\rm H}}              
\newcommand{\Io}{{\rm I}}              
\newcommand{\SP}{{\rm SP}}             
\newcommand{\HM}{{\rm HM}}             
\newcommand{\CM}{{\rm CM}}             
\newcommand{\HF}{{\rm HF}}             
\newcommand{\CF}{{\rm CF}}             

\newcommand{\lambdacan}{\lambda_{\rm can}} 
\newcommand{\omegacan}{\omega_{\rm can}}   

\newcommand{\PD}{{\rm PD}}             
\newcommand{\Pd}{{\rm Pd}}             
\newcommand{\RP}{\R{\rm P}}            
\newcommand{\CP}{\C{\rm P}}            
\newcommand{\Hess}{\mathrm{Hess}}          
\newcommand{\cat}{\mathrm{cat}}            
\newcommand{\cupp}{\mathrm{cup}}           
\newcommand{\Pcont}{\mathcal{P}_0} 
 
\newcommand{\Eig}{\mathrm{Eig}}            
\newcommand{\Arnold}{{Arnol$'$d}}           
\newcommand{\interior}[1]{\accentset{\circ}{#1}} 
\newcommand{\pt}{\textsf{pt}}                    
\renewcommand{\d}{{\rm d}}

\newcommand{\norm}{{\rm norm}}

\newcommand{\loc}{{\rm loc}}

\newcommand{\eps}{{\varepsilon}}

\newcommand{\GL}{{\rm GL}}
\renewcommand{\O}{{\rm O}}
\newcommand{\SO}{{\rm SO}}
\newcommand{\U}{{\rm U}}

\newcommand{\Sp}{{\rm Sp}}

\renewcommand{\o}{{\mathfrak o}}

\newcommand{\cc}{{\mathfrak c}}

\newcommand{\Hhreg}{{\mathcal{H}}_{\rm reg}}

\newcommand{\Vvreg}{{\mathcal{V}}_{\rm reg}}

\newcommand{\Zreg}{Z_{\rm reg}}

\newcommand{\delbar}{\bar\p}
\newcommand{\inner}[2]{\langle #1, #2\rangle}   
\newcommand{\INNER}[2]{\left\langle #1, #2\right\rangle}

\newcommand{\mbf}[1]{\text{\boldmath $#1$}}  

\newcommand{\ba}{{\mbf{a}}}

\def\NABLA#1{{\mathop{\nabla\kern-.5ex\lower1ex\hbox{$#1$}}}}
\def\Nabla#1{\nabla\kern-.5ex{}_{#1}}
\def\Tabla#1{\Tilde\nabla\kern-.5ex{}_{#1}}
\def\abs#1{\mathopen|#1\mathclose|}   
\def\Abs#1{\left|#1\right|}            
\def\norm#1{\mathopen\|#1\mathclose\|}
\def\Norm#1{\left\|#1\right\|}

\renewcommand{\Tilde}{\widetilde}

\newcommand{\p}{{\partial}}

\newcommand{\INTO}{\hookrightarrow}              
\newcommand{\IMTO}{\looparrowright}              

\renewcommand{\1}{{{\mathchoice {\rm 1\mskip-4mu l} {\rm 1\mskip-4mu l}
{\rm 1\mskip-4.5mu l} {\rm 1\mskip-5mu l}}}}

\newcommand{\Index}[1]{#1\index{#1}}

\newlength\eqshift
\renewcommand\theequation{\thesection.\arabic{equation}}
\let\savetheequation\theequation

\makeatletter
\renewcommand*\env@matrix[1][\arraystretch]{%
  \edef\arraystretch{#1}%
  \hskip -\arraycolsep
  \let\@ifnextchar\new@ifnextchar
  \array{*\c@MaxMatrixCols c}}
\makeatother

\makeatletter
\let\save@mathaccent\mathaccent
\newcommand*\if@single[3]{%
  \setbox0\hbox{${\mathaccent"0362{#1}}^H$}%
  \setbox2\hbox{${\mathaccent"0362{\kern0pt#1}}^H$}%
  \ifdim\ht0=\ht2 #3\else #2\fi
  }
\newcommand*\rel@kern[1]{\kern#1\dimexpr\macc@kerna}
\newcommand*\widebar[1]{\@ifnextchar^{{\wide@bar{#1}{0}}}{\wide@bar{#1}{1}}}
\newcommand*\wide@bar[2]{\if@single{#1}{\wide@bar@{#1}{#2}{1}}{\wide@bar@{#1}{#2}{2}}}
\newcommand*\wide@bar@[3]{%
  \begingroup
  \def\mathaccent##1##2{%
    \let\mathaccent\save@mathaccent
    \if#32 \let\macc@nucleus\first@char \fi
    \setbox\z@\hbox{$\macc@style{\macc@nucleus}_{}$}%
    \setbox\tw@\hbox{$\macc@style{\macc@nucleus}{}_{}$}%
    \dimen@\wd\tw@
    \advance\dimen@-\wd\z@
    \divide\dimen@ 3
    \@tempdima\wd\tw@
    \advance\@tempdima-\scriptspace
    \divide\@tempdima 10
    \advance\dimen@-\@tempdima
    \ifdim\dimen@>\z@ \dimen@0pt\fi
    \rel@kern{0.6}\kern-\dimen@
    \if#31
      \overline{\rel@kern{-0.6}\kern\dimen@\macc@nucleus\rel@kern{0.4}\kern\dimen@}%
      \advance\dimen@0.4\dimexpr\macc@kerna
      \let\final@kern#2%
      \ifdim\dimen@<\z@ \let\final@kern1\fi
      \if\final@kern1 \kern-\dimen@\fi
    \else
      \overline{\rel@kern{-0.6}\kern\dimen@#1}%
    \fi
  }%
  \macc@depth\@ne
  \let\math@bgroup\@empty \let\math@egroup\macc@set@skewchar
  \mathsurround\z@ \frozen@everymath{\mathgroup\macc@group\relax}%
  \macc@set@skewchar\relax
  \let\mathaccentV\macc@nested@a
  \if#31
    \macc@nested@a\relax111{#1}%
  \else
    \def\gobble@till@marker##1\endmarker{}%
    \futurelet\first@char\gobble@till@marker#1\endmarker
    \ifcat\noexpand\first@char A\else
      \def\first@char{}%
    \fi
    \macc@nested@a\relax111{\first@char}%
  \fi
  \endgroup
}
\makeatother

\newcommand{\widebarsub}[2]{{{\widebar{#1}}_{{\mspace{-1.7mu} #2}}}}

\newcommand{\Jbar}{{\widebar{J}}}     

\long\def\symbolfootnote[#1]#2{\begingroup%
\def\thefootnote{\fnsymbol{footnote}}\footnote[#1]{#2}\endgroup}
\begin{document}
\sloppy
\author{Joa Weber \\ UNICAMP }
\title{Topological Methods in the \\ Quest for Periodic Orbits
       \\
       \vspace{.5cm}
       {\Large \mbox{ } Lecture Notes}\large\,\footnote{
       February 18, 2018. Revised version of
       the~{\href{https://impa.br/publicacoes/coloquios/}
                        {$31^{\text{\underline{o}}}$ CBM-08}}
       {\href{https://www.math.stonybrook.edu/\%7Ejoa/PUBLICATIONS/00-publications.html}{Lecture
         Notes}}.
       Notable modifications and extensions occured,
       in increasing order, in Sections~\ref{sec:loops},~\ref{sec:Baire},
       and~\ref{sec:Thom-Smale-transversality}. There are the new
       Appendices~\ref{sec:function-spaces} and~\ref{sec:functional-analysis}.
       }\Large
       \\
       {\Large MM613 2016-2}
       }
\date{}

\begin{titlepage}
\maketitle 
\thispagestyle{empty}
\newpage
\mbox{ }
\end{titlepage}

\frontmatter 

%
%
%

\begin{dedication}
To the cycles of life  
\\\mbox{ }
\\
unimaginable in variety \\ stunning surprises

\end{dedication}

\newpage
\thispagestyle{empty}

%
%
%

\chapter{Preface}




The present text originates from lecture notes written
during the graduate course
``MM613 M\'{e}todos Topol\'{o}gicos da Mec\^{a}nica Hamiltoniana''
held from august to november 2016 at UNICAMP.
The manuscript has then been extended in order to serve
as accompanying text for an advanced mini-course
during the \textit{$31^{\rm st}$ Col\'{o}quio Brasileiro de Matem\'{a}tica},
IMPA, Rio de Janeiro, in august 2017.

\subsection*{Scope}
We aim to present some steps in the history
of the problem of detecting closed orbits
in Hamiltonian dynamics. This not only relates
to symplectic geometry, but also to an odd
cousin, called contact geometry and leading to Reeb dynamics.
Ultimately we'd like to introduce the reader to Rabinowitz-Floer
homology, an active area of contemporary research.

When we started to write these lecture notes
we aimed in the introduction
``\textit{The following text is meant to provide an introductory overview,
throwing in some details occasionally,
preferably such which are usually omitted.}''
Obviously we failed: In the end our text contains quite a lot of details and,
as it turned out, basically all of them can be found \emph{somewhere}
in the literature..

\subsection*{Content}
There are two parts, Hamiltonian dynamics on a symplectic manifold
and Reeb dynamics on a contact manifold,
each one coming with, maybe largely motivated by, a famous
conjecture: The {\Arnold} conjecture on existence of
$1$-periodic Hamiltonian trajectories and the Weinstein conjecture
concerning existence of closed characteristics (embedded circles whose
tangent spaces are lines of the characteristic line bundle) on certain
closed energy hypersurfaces
of a symplectic manifold, namely, those of contact type.
While in general closed characteristics integrate
the Hamiltonian vector field, it is a consequence
of the contact condition that they simultaneously integrate
a Reeb vector field. Closed characteristics are images
of periodic Hamiltonian trajectories -- whatever period but on a given energy level.

Part one recalls basics of symplectic geometry, in particular,
we review the Conley-Zehnder index from various
angles. Then we present the construction of Floer homology,
rather detailed, as analogous steps are used in the construction of
Rabinowitz-Floer homology. Floer homology was deviced to prove the
{\Arnold} conjecture.
Part two recalls basics of contact geometry
and reviews the construction of Rabinowitz-Floer homology.
The Weinstein conjecture is reconfirmed for certain classes
of hypersurfaces in exact symplectic manifolds.

It goes without saying that the references simply reflect the
knowledge, not to say ignorance, of the author.
They are not meant to be exhaustive. Certainly many more 
people contributed to the many research fields, and all their facets,
touched upon in these 
notes.

\subsection*{Audience}
The intended audience are graduate students.
Necessary background includes basic knowledge of manifolds,
differential geometry, and functional analysis.

\subsection*{Acknowledgements}
It is a pleasure to thank brazilian tax payers for the
excellent research and teaching opportunities at UNICAMP
and for generous financial support:
``O presente trabalho foi realizado com apoio do CNPq, Conselho
Nacional de Desenvolvimento Cient\'{\i}fico e Tecnol\'{o}gico - Brasil,
e da FAPESP, Funda\c{c}\~{a}o de Amparo \`{a} Pesquisa do Estado
de S\~{a}o Paulo - Brasil.''

I'd like to thank Leonardo Soriani Alves
for interest in and many pleasant conversations throughout
the lecture course
``MM613 M\'{e}todos Topol\'{o}gicos da Mec\^{a}nica Hamiltoniana''
held in the second semester of 2016 at UNICAMP.

\vspace{\baselineskip}
\begin{flushright}\noindent
Campinas, \hfill \mbox{ }  {\it Joa Weber} \\
February 2018  \hfill \mbox{ }  
\end{flushright}

\tableofcontents

\mainmatter 

\cleardoublepage
\phantomsection
\chapter{Introduction}
\chaptermark{Introduction}
The quest for periodic orbits of dynamical systems - for instance
periodic geodesics or periodic trajectories of particles in a magnetic field -
dates back to the foundational work by
Hamilton~\citeintro{Hamilton:1835a} and
Jacobi~\citeintro{Jacobi:2009a} around 1840 and by
Poincar\'{e}~\citeintro{Poincare:1895a} around 1900, followed by work,
among many others, by Lusternik-Schnirelmann~\citeintro{lusternik:1930a}
in the 1920s, Kolmogorov-{\Arnold}-Moser
 \citeintro{Kolmogorov:1954a,Arnold:1963a,Moser:1962a}
around the 1960s, and Rabinowitz~\citeintro{Rabinowitz:1979a,rabinowitz:1986a}
and Conley-Zehnder\citeintro{Conley:1983a} around the early 1980s. Floer's
approach~\citeintro{floer:1989a} to infinite dimensional Morse theory in
the second half of the 1980s, combining the Conley-Zehnder
approach with Gromov's $J$-holomorphic curves introduced
in his 1985 landmark paper~\citeintro{gromov:1985a},
marked a breakthrough in the efforts to prove the
\emph{{Arnol$\, '$d} conjecture}: The number of 1-periodic orbits of a
Hamiltonian vector field on a closed symplectic manifold $M$ is bounded
below by the Lusternik-Schnirelmann category of $M$ or, in the
non-degenerate case, by the sum of the Betti numbers of $M$. At about the same time
Hofer entered the stage and together with Wysocki, Zehnder,
Eliashberg, among others, contactized the symplectic world, eventually
leading to the (occasionally so-called) theory of everything~\citeintro{Eliashberg:2000a}:
Symplectic Field Theory -- SFT.


\subsection*{Departing from Poincar\'{e}'s last geometric theorem}

We briefly sketch how Poincar\'{e}'s last geometric
theorem\index{theorem!Poincar\'{e}'s last geometric --}
inspired the {\Arnold} conjecture. For many more facets and further
related results along these developments see the excellent
presentations~\cite[Ch.~6]{hofer:2011a}
and~\cite[App.~9]{Arnold:1978a}. The following result was announced by
Poincar\'e~\citeintro{Poincare:1912a} shortly before his death in 1912
and proved by Birkhoff~\citeintro{Birkhoff:1913a} shortly thereafter. Let
$\D^n\subset\R^n$ be the closed unit disk and $\SS^{n-1}=\p\D^n$ the unit sphere.

\begin{theorem}[\Index{Poincar\'{e}}-\Index{Birkhoff}]
\label{thm:P-B}
Every area and orientation preserving homeomorphism $h$
of an annulus $A:=\SS^1\times[a,b]$ rotating the two
boundaries in opposite directions\footnote{
  This so-called \textbf{\Index{twist condition}}
  excludes rotations (they have no fixed points in general).
  }
possesses at least 2 fixed points in the interior. 
\end{theorem}

\begin{exercise}\label{exc:jhghj78}
Show that $h$ in the Poincar\'{e}-Birkhoff Theorem~\ref{thm:P-B}
is homotopic to the identity.
[{Hint: Identify each of the two boundary components of the annulus $A$
to a point to obtain a space homeomorphic to $\SS^2$ equipped with an induced
homeomorphism $\tilde h$. Apply the Hopf degree theorem.\footnote{
  \textbf{\Index{Hopf degree theorem}:}
  Two maps of a closed connected oriented
  $n$-dimensional manifold $Q$ into $\SS^n$
  are homotopic if and only if they have the same degree.
  See e.g.~{\cite[Ch.3 \S 6]{guillemin:1974a}}
  or~{\cite[Ch.5 Thm.~1.10]{hirsch:1976a}}.
  }}]
\end{exercise}

\Index{Lefschetz fixed point theory},
introduced in 1926~\citeintro{Lefschetz:1926a},
cf.~{\cite[Ch.5 \S 2 Excs.]{hirsch:1976a}} or~{\cite[Ch.3 \S
4]{guillemin:1974a}}, guarantees existence of a fixed
point for a continuous map $h:X\to X$ on a compact topological space
$X$ whenever a certain integer $L_h$, called the Lefschetz number,
is non-zero. Key properties concerning applications
are, firstly, that $L_h$ is a homotopy invariant
and, secondly, if $X$ is a closed manifold then
$L_\id$ is the Euler characteristic $\chi(X)$.

For the Poincar\'{e}-Birkhoff Theorem~\ref{thm:P-B}
Lefschetz theory fails, as $\chi(A)=0$.
A direct proof of the existence of one fixed point of $h$
is given in the beautyful presentation~\cite[\S 8.2]{mcduff:1998a}
where, furthermore, existence of infinitely many
periodic points\footnote{
  A \textbf{\Index{periodic point}} $x$ of $h$ is a fixed point of one of the
  iterates of $h$, that is $h^k(x):=(h\circ\dots\circ h)(x)=x$ for some $k\in\Z$.
  }
of $h$ is proved whenever the boundary twist is 'sufficiently strong'.

\begin{exercise}
Show that any continuous map $f:\SS^2\to\SS^2$ \Index{homotopic
to the identity}, in symbols $f\sim \id$, has at least one fixed point.
This result is sharp even for homeomorphisms: Find a homeomorphism
of $\SS^2$ with exactly one fixed point.
[Hint: Consider the Riemann sphere $\R^2\cup\{\infty\}$
and translations on $\R^2$.]
\end{exercise}

\begin{remark}\label{rem:2FIX}
By~\citeintro{Nikishin:1974a,Simon:1974a}
one gets back to at least two guaranteed fixed points, if
one requires a homeomorphism $f\sim\id$ on $\SS^2$
to preserve, in addition, a regular measure.
So any diffeomorphism $f$ of $\SS^2$ leaving
an area form $\omega$ invariant, that is $f^*\omega=\omega$,
admits at least two\footnote{
  Note that $\deg f=1=\deg \id$ and apply the
  Hopf degree theorem.
  }
fixed points; cf. Section~\ref{sec:LIPS}.
\end{remark}

In dimension \emph{two}, but \emph{not} in higher dimension,
the diffeomorphisms of a surface that preserve an area form
are the \textbf{symplectomorphisms} of the form.

\subsection*{Arriving at the {\Arnold} conjecture}

In~\cite[App.~9]{Arnold:1978a} {\Arnold}
suggested to glue together two copies of the annulus
in the Poincar\'{e}-Birkhoff Theorem~\ref{thm:P-B}
along their boundaries each of which equipped
with the same area and orientation preserving map $h$
which, in addition, is now assumed to be a diffeomorphism
and not too far $C^1$-away from the identity.
This results in the $2$-\Index{torus} $\T^2:=\SS^1\times\SS^1$
equipped with an area and orientation preserving diffeomorphism, say
$\tilde h$, which is $C^1$-close to $\id$ and by the twist condition
satisfies a condition illustratively called
``preservation of center of mass''.
Note that Lefschetz theory does not predict any fixed point for
$\tilde h$ since $\chi(\T^2)=0$.
However, due to the additional $C^1$-close-to-$\id$ condition,
the fixed points of $\tilde h$ correspond precisely to the critical
points of a function $F$ on $\T^2$ called the generating function of
$\tilde h$. The number of critical points of $F$ is bounded below
by the \textbf{Lusternik-Schnirelmann category}
$$
     \abs{\Crit F}\ge \cat(\T^2)=3>\cupp_\R(\T^2)=2,
$$
more modestly, by the \textbf{cuplength} plus one,
or via Morse theory by the \textbf{sum of the Betti numbers}
$\SB(\T^2)=4$ in the \textbf{non-degenerate case}, that is
in case all fixed points of $h$, equivalently all critical points of
$F$, are non-degenerate.
See e.g.~\cite{weber:2015-MORSELEC-In_Preparation} for
basics on Lusternik-Schnirelmann and Morse theory.
So the number of fixed points of $\tilde h$ is at least three.
But this number is even by symmetry of the construction
(the fixed points come in pairs). Consequently $\tilde h$ has at least four
fixed points. So $h$ has at least two and this reconfirms the
Poincar\'{e}-Birkhoff Theorem~\ref{thm:P-B} for diffeomorphisms
and under the \emph{additional
assumption} of $h$ being $C^1$-close to $\id$.
\\
Hence one might conjecture,
as {\Arnold} did in~\citeintro{Arnold:1976a},
that the torus result should remain true without the
$C^1$-close-to-$\id$ condition and, furthermore, not only for
``doublings'' $\tilde h$ of $h$. It is important to observe that
$h\sim\id$ leads to the fact that $\tilde h$ is a Hamiltonian
diffeomorphism\footnote{
  Some authors use the terminology
  \emph{$\tilde h$ is \Index{homologous to the identity}}.
  }
for the area form, that is it is the time-1-map of the flow generated
by the Hamiltonian vector field $X_H$ for some function
$H:\SS^1\times\T^2\to\T^2$. Fixed points of $\tilde h$ are then
in bijection with 1-periodic orbits of $X_H$.

\vspace{.2cm}
\noindent
\textbf{\Index{Arnol$'$d conjecture}.} Suppose $(M,\omega)$ is
a\index{conjecture!\Arnold}
\emph{closed} symplectic manifold and $H:\R\times M\to\R$ is a smooth time-$1$
periodic function $H_t(x):=H(t,x)=H_{t+1}(x)$, denoted
$H:\SS^1\times M\to\R$.
Consider the time-dependent Hamiltonian equation
$$
     \dot z=X_{H_t}(z),\quad z:\R\to M,
$$
and the set $\Pcont(H)=\Pcont(H;M,\omega)$ of all contractible\footnote{
  Multiplying $X_H$ by a small constant implies that all
  1-periodic solutions are very short, hence contractible;
  see Proposition~\ref{prop:HZ-C2small}.
  Firstly, this inspires the conjecture that it is the contractible solutions
  which are related to the topology of $M$.
  Secondly, this has the consequence that any Floer complex on a component of
  the free loop space that consists of non-contractible loops
  is chain homotopy equivalent to the trivial (no generators) complex.
  }
and 1-periodic\footnote{
  Which non-zero period one picks does not
  matter by Remark~\ref{rem:period-1}, so choose $1$.
  }
solutions. The {\Arnold} conjecture states that the
number $\abs{\Pp_0(H)}$ of contractible $1$-periodic solutions
is bounded below by the least number of critical points that a
function on $M$ must have, that is by the infimum $\Crit(M)$ over
all functions $f:M\to\R$ of the number $\Crit f$ of critical points.
The commonly addressed weaker forms of the {\Arnold}
conjecture are suggested by Lusternik-Schnirelmann and Morse theory,
respectively. They state that
\begin{equation}\label{eq:deg-AC}
     \Abs{\Pcont(H)}\ge \cupp_\R(M) +1
\end{equation}
in general and that
\begin{equation}\label{eq:non-deg-AC}
     \Abs{\Pcont(H)}\ge \SB(M)
\end{equation}
in case all contractible 1-periodic solutions are non-degenerate.

\vspace{.2cm}
\noindent
As we tried to stress, the {\Arnold} conjecture for $\T^2$ is the
\emph{differentiable} generalization of the Poincar\'{e}-Birkhoff
Theorem~\ref{thm:P-B}. Are there \emph{topological} 
generalizations, that is topological analogues of
the {\Arnold} conjecture, as well? There are -- in dimension \emph{two}
-- and these are extremely far reaching indeed; see discussion
towards the end of $\S 6.1$ in~\cite{hofer:2011a}.
For instance, they led to the affirmative
solution~\citeintro{Franks:1992a,Bangert:1993a}
of the longstanding open question if all Riemannian 2-spheres
carry infinitely many geometrically distinct periodic
geodesics.\footnote{
  Or, equivalently and shorter, if they carry infinitely many closed geodesics.
  In our terminology \textbf{\Index{periodic}} indicates maps
  (parametrizations) defined on $\R$ (or on $\R/\tau\Z$ to emphasize a
  period $\tau\not=0$), hence analysis, whereas
  \textbf{\Index{closed}} refers to closed (compact and
  without boundary) 1-dimensional submanifolds
  or, more general, immersed circles, hence geometry.
  But one and the same immersed circle
  can be parametrized, even insisting on constant speed,
  by choosing any of its points as initial value at time zero or,
  giving up on injectivity, by running at $k$-fold speed. All of these
  different maps are \textbf{\Index{geometrically equal}},
  meaning same image.
  }

\subsubsection*{Floer homology -- period one}
Cornerstones in the confirmation of the {\Arnold} conjecture
were the solution by Conley and Zehnder~\citeintro{Conley:1983a}
for Hamiltonians $H_t$ on the standard torus $(\T^n,\omega_0)$ and the
solution by Floer~\citeintro{floer:1988a,floer:1989a} for
$\omega$-aspherical (and other) closed symplectic manifolds
$(M,\omega)$; see e.g.~\citerefFH{salamon:1999a}
or~\cite{hofer:2011a} for detailed accounts of further contributions.
Floer's seminal contribution was to develop a meaningful Morse
(homology) theory for the \textbf{\Index{symplectic action} functional}
\begin{equation}\label{eq:symp-action-fctl}
     \Aa_H:\Ll_0 M\to\R,\quad
     z\mapsto \int_{\D} \bar{z}^*\omega -\int_0^1 H_t(z(t))\, dt,
\end{equation}
on the component $\Ll_0 M$ of the \textbf{\Index{free loop space}}
$\Ll M=C^\infty(\SS^1,M)$ that consists of contractible
smooth $1$-periodic loops $z:\SS^1=\p\D\to M$,
where $\bar{z}:\D\to M$ is any smooth extension of $z$.
Floer mastered the obstructions presented by
\begin{itemize}
\item 
  \emph{infinite} Morse index of the critical points
  $z\in\Crit\Aa_H=\Pcont(H)$\footnote{
    For time-dependent vector fields take the analytic view point
    of periodic solution \emph{maps}. (Looking
    at images in $M$ is useless without
    recording for each point simultaneously time.)
    }
  (which by definition are the generators of the Floer chain groups),
  cf. Ex.~\ref{ex:inf-M-index};
\item 
  the fact that the formal \emph{downward} gradient equation
  for the $L^2$-gradient
  $$
     \grad\Aa_H(z)=-J_t(z)\bigl(\dot z-X_{H_t}(z)\bigr)
  $$
  \emph{does not} generate a flow on loop space, not even a
  semi-flow; see Remark~\ref{rem:no-flow}.
  Here $J_t$ is a family of $\omega$-compatible almost complex structures.
\end{itemize}
By definition and for generic $H$ the Floer chain group
$\CF_*(M,\omega,H)$
is the free abelian group generated by $\Pcont(H)$
and graded by the canonical Conley-Zehnder index
assuming that the first Chern class $c_1(M)$ vanishes.
Roughly speaking, the Floer boundary operator counts downward flow lines
and the Floer isomorphism
$$
     \HF_{n-\ell}(M,\omega;H,J)\cong\Ho_\ell(M)
$$
to singular homology of $M$ proves the {\Arnold}
conjecture~(\ref{eq:non-deg-AC}) for \emph{closed}
symplectic manifolds that are $\omega$-aspherical;
see Definition~\ref{def:I-omega}.
Floer homology of the closed manifold $M$ of dimension $2n$
is restricted to degrees in $[-n,n]$.

\vspace{.2cm}
\noindent
\textbf{Floer homology of cotangent bundles.}
Floer homology of non-compact symplectic manifolds can be highly
different, if it can be defined. For instance, given a closed
orientable Riemannian manifold $(Q,g)$ consider the cotangent bundle
$T^*Q$ equipped with the canonical symplectic structure
$\omegacan=d\lambdacan="dp\wedge dq\,"$. 
It is convenient to identify $T^*Q\cong TQ$ via $g$ and
abbreviate $g_q(v,v)$ by $\abs{v}^2$.
Now consider a \textbf{\Index{mechanical Hamiltonian}}
\begin{equation}\label{eq:H_V}
     H_{V_t}(q,v)=\frac12\abs{v}^2+V_t(q),\qquad
     \text{$q\in Q$,\, $v\in T_q Q$,\, $t\in\SS^1$,}
\end{equation}
of the form \emph{kinetic} plus \emph{potential energy} where the
\textbf{\Index{potential}} $V(t,q)=:V_t(q)$ is a smooth function
on $\SS^1\times Q$. The action functional~(\ref{eq:symp-action-fctl})
takes on the form
\begin{equation}\label{eq:symp-action-fctl-T^*M}
     \Aa_V:\Ll TQ\to\R,\quad
     z=(q,v)\mapsto\int_0^1 g\left(v(t),\dot q(t)\right) -H_{V_t}(z(t))\, dt,
\end{equation}
being defined on arbitrary loops, not just contractible ones;
cf.~(\ref{eq:action-can-cot}).
Its critical points are of the form $z_x=(x,\dot x)$ where $x$
is a \textbf{\Index{perturbed 1-periodic geodesic}}, that is an element of
the set\index{geodesic!perturbed 1-periodic --}
\begin{equation}\label{eq:P(V)}
     \Pp(V):=\{x\in\Ll Q\mid-\Nabla{t}\dot x-\nabla V_t(x)=0 \}.
\end{equation}
By the Morse index theorem the Morse index of
a periodic geodesic is finite; still true after perturbation
by a zero order term. In~\citerefFH{weber:2002a}
it is shown that for generic $V$ the canonical Conley-Zehnder index
is well defined and equal to
\begin{equation}\label{eq:CZ-ind-Mind-intro}
     \CZcan(z_x)=\IND_{\Ss_V}(x)\in\N_0
\end{equation}
the \textbf{\Index{Morse index}}; cf.~(\ref{conv:CZ=Morse-conv}):
The number of negative eigenvalues, counted with multiplicities,
of the Hessian at a critical point $x$ of the classical action functional
given by $\Ss_V(\gamma)=\int_0^1\frac12\abs{\dot\gamma}^2-V_t(\gamma)\, dt$
for $\gamma\in\Ll Q$.
The upshot is that the \textbf{Floer homology
of the cotangent bundle}, graded by $\CZcan$, is naturally isomorphic
to singular integral homology of the free loop space:
That is\index{Floer homology!of $T^*Q$}
\begin{equation*}
     \HF_*(\Aa_V):=\HF_*(T^*Q,\omegacan;H_V,\Jbar_g)\simeq\Ho_*(\Ll Q),
\end{equation*}
at least if the orientable manifold $Q$ carries a spin structure or,
equivalently, if the first and second Stiefel-Whitney classes of $Q$
are both trivial; cf.~Section~\ref{sec:FH-T^*M}.
If $Q$ is not simply connected, there is a separate isomorphism
for each component $\Ll_\alpha Q$
of the free loop space.
If $Q$ is not orientable, choose $\Z_2$ coefficients.

\subsection*{Weinstein conjecture}
Given a symplectic manifold $(M,\omega)$,
consider an \emph{autonomous}\footnote{
  In case of a time-independent vector field $X$
  the skinnier geometric view point makes sense and one is looking for
  \textbf{\Index{closed characteristics}}, namely, closed
  1-dimensional submanifolds $P$ that integrate $X$, that is along which
  $X$ is a non-vanishing section of the tangent bundle $TP$.
  }
Hamiltonian $F:M\to\R$, also called an energy function. For such $F$
the Hamiltonian flow $\phi^F$ generated by the Hamiltonian vector
field $X_F$ on $M$ is energy preserving: Energy level
sets $F^{-1}(c)$ are invariant under $\phi^F$.
It is a natural question if there exists a Hamiltonian
flow trajectory that closes up in finite time $T$
on a given, say closed, \emph{regular} level set $\Sigma:=F^{-1}(c)$.
Observe that by regularity there are no zeroes of $X_F$ or,
equivalently, no constant flow trajectories.
Restricting the non-degenerate $2$-form $\omega$ to
the odd-dimensional submanifold $\Sigma$ yields the so-called
\textbf{\Index{characteristic line bundle}}
$$
     \Ll_\Sigma:=\ker\omega|_\Sigma\to\Sigma
$$
which even comes with a non-vanishing section, namely $X_F$.
Therefore flow lines of $X_F$ are integral curves
of the distribution $\Ll_\Sigma$ and those that close
up are called the \textbf{\Index{closed characteristics}} $P$
of the energy surface, in symbols $TP=\Ll_\Sigma|_P$.

On $\R^{2n}$ equipped with the canonical symplectic form
$\omegacan="dp\wedge dq\,"$ 
existence of a closed characteristic was confirmed
on convex and star-shaped $\Sigma$ by, respectively,
Weinstein~\citerefCG{Weinstein:1978a}
and Rabinowitz~\citerefCG{Rabinowitz:1978a}.
Weinstein then isolated key \emph{geometric} features of these
hypersurfaces, and of the slightly more general class treated by
Rabinowitz in~\citeintro{Rabinowitz:1979a},
thereby coining the notion of \emph{contact type} hypersurfaces
in~\citeintro{Weinstein:1979a}
and formulating the influential\footnote{
  For more background and context we recommend the fine
  survey in~\citerefCG{Hutchings:2010a}.
  }

\vspace{.2cm}
\noindent
\textbf{\Index{Weinstein conjecture}.}
\textit{A closed hypersurface of contact type with trivial first real
cohomology carries a closed characteristic.}

\subsubsection*{Rabinowitz-Floer homology -- free period fixed energy}
For about three decades the potential of the variational setup
used by Rabinowitz in his breakthrough result~\citerefCG{Rabinowitz:1978a},
cf.~\citeintro{Rabinowitz:1979a}, went widely unnoticed. Given an
autonomous Hamiltonian system $(V,\omega,F:M\to\R)$, his idea was to
incorporate a Lagrange multiplier $\LM$ into the standard action
functional~(\ref{eq:symp-action-fctl}) whose presence causes that the
critical points are periodic Hamiltonian trajectories of \emph{whatever period}
and constrained to a \emph{fixed energy} level surface, namely $\Sigma:=F^{-1}(0)$.
Only around 2007 the Rabinowitz action functional
\begin{equation*}
     \Aa^F:\Ll V\times\R,\quad
     (\lpz,\LM)\mapsto
     \int_{\SS^1} \lpz^*\lambda-\LM\int_0^1 F(\lpz(t))\, dt,
\end{equation*}
on certain exact symplectic manifolds $(V,\omega=d\lambda)$, namely
convex ones, was brought to new, if not spectacular, honours by Cieliebak and
Frauenfelder in their landmark construction~\citeintro{Cieliebak:2009a}
of a Floer type homology theory:
Rabinowitz-Floer homology $\RFH(\Sigma,V):=\HF(\Aa^F)$
associated to certain closed hypersurfaces
$\Sigma=F^{-1}(0)$, for instance such of restricted contact type that bound a
closed submanifold-with-boundary $M\subset V$,
written as a regular level.

The power of their theory is shown by the fact that
Rabinowitz-Floer homology of the archetype example of the unit
bundle $\Sigma=S^*Q$ in the cotangent bundle
$(V,\lambda)=(T^*Q,\lambdacan)$ over a closed Riemannian manifold $Q$,
not only represents the homology of the loop space of $Q$,
but \emph{simultaneously} its cohomology.

\subsection*{Symplectic and contact topology}
For an overview of the development of
symplectic and contact topology, starting with
Lagrange's 1808 formulation of classical mechanics and culminating
in the moduli space techniques initiated by
Gromov~\citeintro{gromov:1985a} and
Floer~\citeintro{Floer:1986a,floer:1989a}
in the mid 1980's we recommend the
article~\citeintro{2016arXiv161102676N}.
The article also explains the origin of the
adjective \emph{symplectic} as the greek version
of the originally advocated latin adjective \emph{complex}.
The latter was abandoned as it was already used in the prominent
notion of complex number; see also\index{symplectic!origin of name}
the wiktionary entry~\href{https://en.wiktionary.org/wiki/symplectic}{'symplectic'}.

\section*{Notation and conventions}
\addcontentsline{toc}{section}{Notation and sign conventions}
\vspace{.2cm}
\begin{tabular}{lll}
\toprule
  Symbol
  & Terminology
  & Remark
\\
\midrule
  $\N$, $\N_0$
  & positive integers, including
      $0$\index{$\N_0=\{0,1,2,\dots\}$, $\N=\{1,2,3,\dots\}$}
  & $\{1,2,3,\dots\}$, $\{0,1,2,3,\dots\}$
\\
  $\R^*$, $\Z^*$
  & non-zero reals,\index{$\R^*:=\R\setminus\{0\}$ non-zero reals}
      integers\index{$\Z^*:=\Z\setminus\{0\}$ non-zero integers}
  & $\R\setminus\{0\}$, $\Z\setminus\{0\}$
\\
\midrule
  $N$
  & \Index{manifold} (mf)
  & modeled on $\R^k$ 
\\
  $N$
  & \Index{manifold-with-boundary}
  & modeled on $\R^{k-1}\times\{x_k\ge 0\}$
\\
  $Q$
  & closed manifold
  & compact, no boundary, $\dim$ $n$
\\
  $(M,\omega)$
  & symplectic mf/mf-w-bdy
  & $\dim M=2n$
\\
  $(V,\lambda)$
  & exact symplectic mf/mf-w-bdy
  & $\dim V=2n$, $\omega=d\lambda$
\\
  $(W,\alpha)$
  & contact mf/mf-w-bdy
  & $\dim W=2n-1$
\\
\midrule
  $F$, $\phi^F$
  & autonomous Ham. and flow
  & $\frac{d}{dt}\phi_t=X_F\circ\phi_t$, $\phi_0=\id$
\\
  
  & \hfill generates $1$-param. group:
  & $\phi_{t+s}=\phi_s\circ\phi_s$
\\
  $\cc$
  & orbit, \Index{integral curve}
  & inj.\,imm.\,submf. $\cc\subset N$ of dim.
\\
  
  & of auton. vector field $X$
  & $1$ or $0$ s.t. $X$ is tangent to $\cc$
\\
  
  & \hfill closed orbit\index{orbit!closed --}
  & \hfill $\cc\cong\SS^1$ or $\cc=\{\pt\}$
\\
  
  & \hfill constant or point orbit\index{orbit!point --}
  & \hfill $\cc=\{\pt\}$
\\
  $P$
  & \Index{closed characteristic} on $S$
  & $TP=\Ll_S|_P$,
     Rmk.\,~\ref{rem:vhjgj676}
\\
\midrule
  $\gamma$
  & path (Def.~\ref{def:paths_and_curves-continuous})
  & smooth map $\gamma:\R\to N$
\\
  $\cc$
  & curve
  & image of a path, subset of $N$
\\
  $\alpha$
  & finite path
  & smooth map $\alpha:[a,b]\to N$
\\

  & \hfill that closes up
  & \hfill $\alpha(a)=\alpha(b)$ with all derivatives
\\
  $\gamma$, $\gamma_\tau$
  & $\tau$-periodic \Index{loop} ($\tau\not=0$)
  & $\gamma:\R/\pertau\Z\to N$, Def.~\ref{def:prime-part}
\\
  $\gamma_0$
  & point path (domain$=\{\pt\}$)
  & $\gamma_0:\R^0=\{0\}\to N$
\\
  $u$
  & \Index{trajectory} (domain $\R$)
  & $u:\R\to \text{mf}$ s.t. $\dot u=X_t(u)$
\\
  
  & \Index{flow line}
  & image $u(\R)$ of solution
\\
  $z$\index{orbit!periodic --}
  & \Index{periodic orbit} (circle domain)
  & $z:\R/\pertau\Z\to N$, $\dot z=X_t(z)$, $\tau\not=0$
\\
  
  & \hfill constant periodic orbit
  & \hfill $z:\R/\pertau\Z\to \{\pt\}\subset N$
\\
\midrule
  $H_t$, $\psi^{H}$
  & non-auton. Hamiltonian, flow
  & $\frac{d}{dt}\psi_t=X_{H_t}\circ\psi_t$, $\psi_0=\id$
\\
  
  & \hfill is not a $1$-param. group: 
  & $\psi_t=\psi_{t,0}$; see~(\ref{eq:Ham-flow})
\\
  
  & Hamiltonian path/loop
  & Remark~\ref{rem:closed-vs-periodic}
\\
  $\Pp(H)$
  & $1$-periodic orbits of $X_H$
  & $\Pp(H)^*$, non-constant ones
\\
\bottomrule
\end{tabular}
\index{$\R_+:=[0,\infty)$}

\begin{NOTATION}\label{not:notations_and_signs}
Unless
mentioned otherwise, the following conventions apply throughout.
All quantities, including homotopies and paths,\index{homotopies are smooth}
\textbf{\Index{smooth}}, that is of class $C^\infty$.
The empty set $\emptyset$ generates the trivial group $\{0\}$.
It is often convenient to set $\inf\emptyset:=\infty$.
Vector spaces are real.\index{vector spaces are real}
Neighborhoods\index{empty set!generates trivial group}
are open.\index{$\inf\emptyset:=\infty$}
To help readability we sometimes omit parentheses of arguments of
maps, usually for linear maps, but also for flows
we usually write $\phi_t p$ instead of $\phi_t(p)$.

Given a differentiable map $f:X\to Y$ between Banach spaces,
we denote by $df(x)\in\Ll(X,Y)$ the (Fr\'{e}chet)
\textbf{\Index{differential}} of $f$ at $x$; see e.g.~\cite{ambrosetti:1993a}.
A \textbf{\Index{Banach manifold}} is a Hausdorff\,\footnote{
  Points are separated by open sets:
  any two points admit open disjoint neighborhoods.
  } 
topological space $\Xx$\index{Hausdorff space}\index{topology!Hausdorff --}
which is\index{manifold!of class $C^k$}
\emph{locally modeled on a Banach space
$X$};\footnote{
  $\Xx$ is covered by the open domains of a collection $\Aa$, called
  \textbf{atlas}, of homeomorphisms $\varphi_i:\Xx\supset U_i\to
  V_i\subset X$, called \textbf{local coordinate charts}, such
  that all transition maps $\varphi_j\circ{\varphi_i}^{-1}:V_i\to V_j$
  are $C^\infty$ diffeomorphisms. In a
  \textbf{$\mbf{C^k}$ manifold} they are all of class $C^k$.
  }
see e.g.~\cite{Abraham:1967a,lang:2001a}.
In the finite dimensional case $\dim \Xx:=\dim X=n\in\N$
we speak of a \textbf{manifold} and add the
requirement\index{topology!base of --}
of being \Index{second countable}.\footnote{
  There is a countable base of the topology. A base
  is a collection $\Bb$ of open sets that generates the topology:
  Any open set of the topology is a union of members of $\Bb$.
  }
For manifolds we choose the model space $\R^n$
and we denote them\index{manifold}
by roman font letters such as $N$.
A finite dimensional manifold is metrizable, hence admits a countable
atlas; cf. Remark~\ref{rem:sep-B-countable}\,c).
To define a \textbf{\Index{manifold-with-boundary}} $N$
replace $\R^n$ by its closed upper half space. The boundary $\p N$
might be empty though.
A \textbf{closed manifold}, here usually denoted
by\index{manifold!closed $Q=Q^n$}
$Q=Q^n$,\index{closed manifold $Q=Q^n$}\index{$Q=Q^n$ closed manifold}
is a compact manifold (hence no boundary by definition of manifold).

For a map $\gamma:\R\to N$, a path, denote time shift
and uniform speed change~by
$$
     \gamma_{(T)}:=\gamma(T+\cdot),\qquad
     \gamma^\mu:=\gamma(\mu\cdot),
$$
whereas subindex $\gamma_\tau:[0,\tau]\to N$,
$\tau\in\Per(\gamma)\setminus\{0\}$,
denotes a divisor part, see~(\ref{eq:divisor-part}), and simultaneously the
induced loop
$
     \gamma_\tau:\R/\tau\Z\to N
$,
but subindex $u_s(\cdot):=u(s,\cdot)$ also denotes the operation of
\Index{freezing a variable}.\index{variable!freeze a --}

Given a map $f:X\to Y$ between sets, a
\textbf{\Index{pre-image}} is a subset of $X$
of the form $f^{-1}(B):=\{x\in X\mid f(x)\in B\}$ where
$B$ is a subset of $Y$. Often we simply write $f^{-1} B$.
The pre-image of a point is denoted by $f^{-1}(y):=f^{-1}(\{y\})$.

\vspace{.2cm}\noindent
\textit{Linear space.}
On $\R^{2n}$ there are two natural structures, the
\textbf{\Index{euclidean metric}}
$\langle v,w\rangle_0:=\sum_{j=1}^{2n}v_jw_j$ and the
\textbf{\Index{standard almost complex structure}}
$$
     J_0:=\begin{pmatrix} 0&-\1 \\ \1&0 \end{pmatrix}; \qquad
     \Jbar_0:=-J_0=\begin{pmatrix} 0&\1 \\ -\1&0 \end{pmatrix}.
$$
The matrizes $\pm J_0$ represent multiplication by $\pm i$ under
the natural isomorphism
$$
     \R^{2n}\stackrel{\simeq}{\longrightarrow}\C^n ,\quad
     z=(x,y)=(x_1,\dots,x_n,y_1,\dots,y_n)\mapsto x+iy .
$$
We shall use this isomorphism freely whenever convenient,
even writing $\R^{2n}=\C^n$ as real vector spaces and $J_0=i$.
Furthermore, given the coordinates $z=(x,y)\in\R^{2n}$
it is natural to combine them in the form
$$
     \omega_0:=\sum_j dx_j\wedge dy_j
$$
called\index{$\omega_0=dx\wedge dy$ standard sympl. form}
the\index{symplectic form!standard}
\textbf{standard symplectic form}. While the 2-form $\omega_0$ is
exact for several choices\index{differential form!primitive of --}
of\index{primitive of differential form}
primitives,\footnote{
  A differential form $\lambda$ is called a \textbf{primitive} of $\omega$
  if its exterior derivative $d\lambda$ is $\omega$.
  }
such as for instance $\sum_j x_jdy_j$, the natural \textbf{\Index{radial vector field}}
$\LVF_0(z)=z=\sum_j\left(x_j\p_{x_j}+y_j\p_{y_j}\right)$
is compatible with the $\omega_0$-primitive
\begin{equation}\label{eq:R2n-omega-standard}
     \lambda_0:=\frac12\sum_{j=1}^n\left(x_jdy_j-y_jdx_j\right),\qquad
     d\lambda_0=\omega_0=:dx\wedge dy,
\end{equation}
in the sense that
$
     i_{\LVF_0}\omega_0:=\omega_0(\LVF_0,\cdot)=\lambda_0
$.\footnote{
  We define the \textbf{\Index{wedge product}}\index{product!wedge}
  by $dx_j\wedge dy_j:=\frac12\left(dx_j\otimes dy_j-dy_j\otimes dx_j\right)$
  as in~\cite{guillemin:1974a}.
  }
Hence
$
     L_{\LVF_0}\omega_0=d i_{\LVF_0}\omega_0=\omega_0
$.
These identities play a crucial role in the history of the Weinstein
Conjecture~\ref{conj:Weinstein} and the development of the 
notion of contact type hypersurfaces.

On the other hand, on cotangent bundles, say $T^*Q\ni(q,p)$,
there is a canonical globally defined $1$-form, the
\textbf{\Index{Liouville form}} $\lambdacan$,
see~(\ref{eq:tautological-1-form}),
the\index{$\omegacan=d\lambdacan=dp\wedge dq$ canonical symplectic form}
\textbf{canonical symplectic form}\index{symplectic form!canonical}
$\omegacan:=d\lambdacan$,
and\index{canonical!symplectic form}
the\index{canonical!fiberwise radial vector field}
canonical\index{canonical!$1$-form}
\textbf{\Index{fiberwise radial vector field}}
$\LVFfrad$, see~(\ref{eq:Liouv-VF-vert-cot-bdl}).
In natural cotangent bundle coordinates
$(q_1,\dots,q_n,p^1,\dots,p^n)$ these structures are of the form
$\LVFfrad=\sum_j p^j\p_{p^j}$ and
\begin{equation*}
     \lambdacan:=\sum_{j=1}^n p^j\, dq_j=:p\, dq
     ,\quad
     \omegacan:={\color{cyan} +}\:
     d\lambdacan=\sum_{j=1}^n dp^j\wedge dq_j
     =: dp\wedge dq .
\end{equation*}
Note that $\omegacan(\LVFcan,\cdot)=\lambdacan$ where $\LVFcan:=2\LVFfrad$
is the \textbf{canonical Liouville vector field}.
Of course, these definitions make sense on $\R^{2n}\simeq T^*\R^n$.
For better readability we often use the notation
$(q_1,\dots,q_n,p_1,\dots,p_n)$.

The two natural symplectic structures $\omega_0$ and $\omegacan$
on $\R^{2n}$ are compatible with $J_0=i$ and 
$\Jbar_0=-i$, respectively, in the
sense that the two\index{$\Jbar_0:=-J_0=-i$}
compositions\index{$(\R^{2n},\omegacan,\Jbar_0,\langle\cdot ,\cdot\rangle_0)$}
\begin{equation}\label{eq:R2n-comp-triples}
     \omega_0(\cdot,J_0\cdot)=\inner{\cdot}{\cdot}_0 ,\qquad
     \omegacan(\cdot, {\color{cyan} \Jbar_0}\cdot)=\inner{\cdot}{\cdot}_0 ,
\end{equation}
both reproduce the euclidean metric. 

\begin{remark}[Canonical normalization of Conley-Zehnder index]
In Hamiltonian dynamics of classical physical Hamiltonians
on cotangent bundles, e.g. on $\R^{2n}\simeq T^* \R^n$,
the second choice in~(\ref{eq:R2n-comp-triples})
is natural since the dynamics is governed by
\textbf{\Index{Hamilton's equations}}~\citeintro[Eq.\,(A.)]{Hamilton:1835a}
given by
\begin{equation}\label{eq:HamEqs}
     \begin{pmatrix}\dot q\\\dot p\end{pmatrix}
     =\begin{pmatrix}\p_{p}H\\-\p_{q}H\end{pmatrix}
     = {\color{cyan} \Jbar_0} \Nabla{} H
\end{equation}
and exhibiting most prominently $\Jbar_0$. For the most basic physical
system, the harmonic oscillator on $\R^2$, which is also most basic
mathematically in the sense that in place of $\Nabla{} H$ one has
the identity $\1$, the system is linear and given by
$$
     \dot \zeta=\Jbar_0\zeta ,\qquad
     \zeta(t)=e^{\Jbar_0 t}=e^{-i t} ,\quad t\in[0,1].
$$
Hence it is natural
to favorize $\Jbar_0$ and the finite path $t\mapsto e^{\Jbar_0 t}$ concerning
sign conventions and normalize the index function by associating the value 
$1$\index{symplectic path!normalizing --}\index{symplectic path!distinct --}
to the \emph{distinct symplectic path}
$e^{-i t}$ in $\Sp(2)$,\index{$e^{-i t}$ distinct symplectic path}
as we do in~(\ref{eq:CZ-normalization-canonical}).\footnote{
  As a rotation $e^{-i t}$ is mathematically
  negatively oriented (counter-clockwise is positive).
  }
However, probably since $e^{i t}$ is also rather distinct
in the sense that it represents the mathematically
positive sense of rotation (counter-clockwise), just like $i=J_0$,
basically all of the original papers on the Conley-Zehnder
index use the normalization 
$$
     \CZ(\{ e^{i t} \}_{t\in[0,1]})=1.
$$
This is\index{standard normalization!of Conley-Zehnder index}
the\index{normalization!standard}\index{Conley-Zehnder index!standard normalization}
\textbf{standard normalization}
or the \textbf{counter-clockwise normalization}
and $\CZ$ the (standard) Conley-Zehnder index.
For compatibility with the literature
our review in Section~\ref{sec:linear-theory}
of the various variants of Maslov-type indices
and the various constructions of each of them
uses the standard normalization.

A simple method to deal with the need,
when dealing with $\omegacan$,
for a Conley-Zehnder index normalized
clockwise is to introduce a new name and notation.
We call the Conley-Zehnder index based
on\index{canonical normalization!of Conley-Zehnder index}
the\index{normalization!canonical}\index{Conley-Zehnder index!canonical normalization}
\textbf{canonical normalization}
\begin{equation}\label{eq:CZ-normalization-canonical}
     \boxed{\CZcan({\color{magenta} \{ e^{-i t} \}_{t\in[0,1]}})=1.}
\end{equation}
the \textbf{canonical Conley-Zehnder index},
denoted by $\CZcan$ for distinction.
It is just the negative $\CZcan=-\CZ$ of the standard
Conley-Zehnder index.
\end{remark}

By \Index{$\overline{\R^{2n}}\times\R^{2n}$} we denote the
vector space $\R^{2n}\times\R^{2n}$ equipped
with the almost complex structure $-J_0\oplus J_0$
and the symplectic form $-\omega_0\oplus\omega_0$.

\vspace{.2cm}\noindent
\textit{Manifolds.}
Suppose $M$ is a manifold. Let $\SS^1\subset\C\cong\R^2$ be the unit
circle which we usually identify with $\R/\Z$ through the map
$\R/\Z\to\SS^1$, $t\mapsto e^{i2\pi t}$.
We slightly abuse notation to express periodicity in time
$t\in\R/\Z$. We usually write
$$
     H:\SS^1\times M\to\R\quad\text{or}\quad
     H:\R/\Z\times M\to\R
$$
to denote a function $H:\R\times M\to\R$ with
$H_{t+1}=H_t$ $\forall t$ where $H_t:=H(t,\cdot)$.
Given a symplectic form $\omega$ on $M$, the identities of 1-forms
\begin{equation}\label{conv:Ham-VF}
     \boxed{dH_t(\cdot)={\color{blue}-}\omega(X_{H_t},\cdot)}
     ,\qquad X_t:=X_{H_t}=X_{H_t}^\omega,
\end{equation}
one for each $t\in\SS^1$, determine the family of
\textbf{\Index{Hamiltonian vector field}s} $X_{t}$.\footnote{
  In our previous papers~\citerefFH{weber:2002a,salamon:2006a}
  we used twice opposite signs, firstly for $\omegacan$
  and secondly in~(\ref{eq:R2n-comp-triples}). Hence the Hamiltonian
  vector field there and here is the same.
  }
The set $\Pp_0(H)$ of contractible $1$-periodic
Hamiltonian trajectories is precisely the set of critical points of
the\index{symplectic action functional!perturbed --}
\textbf{perturbed symplectic action functional}
$$
     \Aa_H=\Aa_H^\omega:C^\infty_{\rm contr}(\SS^1,M)\to\R,\quad
     z\mapsto {\color{cyan} +} {\int_{\D} \bar{z}^*\omega}  \,\,
     {\color{blue} -} \int_0^1 H_t(z(t))\, dt.
$$
Here $\bar{z}:\D\to M$ denotes an extension\footnote{
  To avoid that $\Aa_H(z)$ depends on the extension $\bar z$, suppose that
  $\omega$ vanishes over $\pi_2(M)$.
  }
of the contractible loop $z:\SS^1\to M$ and the two signs arise as follows.
The sign {\color{cyan} $"+"$} is due to the requirement that 
on cotangent bundles (convention
$\omegacan:={\color{cyan} +}\: d\lambdacan=dp\wedge dq$)
the first integral should reduce to $\int_{\SS^1}\lambdacan$.
Since the critical points of $\Aa_H$ should be
orbits of $X_H$, as opposed to $-X_H$, the sign choice {\color{blue} $"-"$}
in~(\ref{conv:Ham-VF}) dictates the second sign {\color{blue} $"-"$} in $\Aa_H$.

Suppose $J_t$ is a
family of \textbf{\Index{almost complex structure}s} on $TM$,
that is each $J_t$ is a section of the endomorphism bundle
$\End(TM)\to M$ with ${J_t}^2=-\1$. Assume, in addition, that each
$J_t$ is \textbf{$\mbf{\omega}$-compatible}\,\footnote{
  In the euclidean case the convention
  $g_{J_0}:=\omega_0(\cdot,{\color{cyan}J_0}\cdot)$
  leads to the euclidean metric, so the opposite convention
  $g_{J_0}^\prime:=\omega_0({\color{red}J_0}\cdot,\cdot)$
  is negative definite and therefore not an inner product.
  }
in the sense that
$$
     \boxed{g_{J_t}:=\langle\cdot ,\cdot\rangle_t:=\omega(\cdot, {\color{cyan}J_t}\cdot)}
$$
defines a Riemannian metric on $M$, one for each $t$. 
Such a triple $(\omega,J_t,g_{J_t})$ is called a
\textbf{\Index{compatible triple}} and for \emph{such} there is the
identity\index{$(\omega,J,g_J)$ compatible triple}
\index{almost complex structure!$\omega$-compatible}
\begin{equation}\label{eq:signs-HamEq}
     X_{H_t}^\omega={\color{blue} +\: \color{cyan} J_t}\circ \nabla^{g_{J_t}} H_t,
\end{equation}
one for each $t\in\SS^1$.
Two compatible triples in $\R^{2n}$ are shown in~(\ref{eq:R2n-comp-triples}).

\begin{exercise}
Suppose $J$ is an $\omega$-compatible almost
complex structure. Show that both,
the associated Riemannian metric $g_J$
and $\omega$ itself, are $J$-invariant,
that is $g_J(J\cdot,J\cdot)=g_J(\cdot,\cdot)$
and $\omega(J\cdot,J\cdot)=\omega(\cdot,\cdot)$.
\end{exercise}

\vspace{.2cm}\noindent
\textit{Cotangent bundles.}
Consider a cotangent bundle
$(T^*Q,\omegacan:={\color{cyan} +}\: d\lambdacan)$
over a closed Riemannian manifold $(Q,g)$ of dimension $n$.
By exactness of $\omegacan$ there is no need to restrict to contractible loops.
Just define\footnote{
  Here all signs are dictated: By physics (integrand should be
  $p\,dq-H\,dt$) as well as by mathematics (the cousin $\Ss_V$ of
  $\Aa$, given by~(\ref{eq:classical-action}), is bounded
  below which suggests that the downward gradient flow
  encodes homology and the upward flow cohomology).
  }
\begin{equation}\label{eq:action-can-cot}
     \Aa_H^{\lambdacan}:\Ll T^*Q\to\R,\quad
     z
     =(q,p)\mapsto{\color{cyan}+}
     \int_0^1\INNER{p(t)}{\dot q(t)} -H_{t}(q(t),p(t))\, dt.
\end{equation}
It is convenient to use $g$ to identify $T^*Q$ with $TQ$ via the inverse of the
isomorphism $\xi\mapsto g(\xi,\cdot)$, again denoted by $g$.
For Hamiltonians $H_{V,g}:\SS^1\times TQ\to\R$ of the form
\emph{kinetic plus potential energy}
for some potential $V_{t+1}=V_t:Q\to\R$, see (\ref{eq:H_V}),
the critical points of $\Aa_V:=\Aa_{H_{V,g}}^{\lambdacan}$ on $\Ll TQ$
are precisely of the form $z_x:=(x,\dot x)$ with
$x\in\Pp(V)=\{-\Nabla{t}\dot x-\nabla V(x)=0\}$.
Hence $x$ is a (perturbed) 1-periodic
geodesic in the Riemannian manifold $(Q,g)$ and as such admits a Morse
index $\IND_{\Ss_V}(x)\in\N_0$ and a nullity $\NULL_{\Ss_V}(x)\in\N_0$.

Suppose the nullity of $x\in\Pp(V)$ is zero and the vector
bundle $x^*TQ\to\SS^1$ is
trivial; pick an orthonormal trivialization.
Then the linearized Hamiltonian flow
along $x$ provides a finite path $\Psi_x:[0,1]\to\Sp(2n)$
of symplectic matrices that starts at $\1_{2n}$
and whose endpoint does not admit $1$ in its
spectrum (by the nullity assumption). Thus $\Psi_x$ has a well defined
canonical Conley-Zehnder index $\CZcan(\Psi_x)$.
Recall from~(\ref{eq:CZ-normalization-canonical})
that $\CZcan$ is based on the canonical $\Jbar_0$ (clockwise)
normalization and equal to $-\CZ$.
In other words, compared to our previous
papers~\citerefFH{weber:2002a,salamon:2006a}
we use {\color{magenta} the opposite} \texttt{(signature)} axiom:
\begin{enumerate}\label{label:signature-can}
\index{\rm\texttt{(signature}$\texttt{)}_{\texttt{can}}$}
\index{Conley-Zehnder index!canonical}
\item[$\rm\texttt{(signature)}_{\texttt{can}}$]
  If $S=S^T\in\R^{2n\times 2n}$ is a symmetric matrix
  of norm $\norm{S}<2\pi$, then
  \begin{equation}\label{eq:signature_can}
     \CZcan\left(\{[0,1]\ni t\mapsto e^{t{\color{magenta} \Jbar_0} S}\right)
     =\frac12\sign(S):=\frac{n^+(S)-n^-(S)}{2}.
  \end{equation}
\end{enumerate}
Since $\CZ(\Psi_x)$ does not depend on the
choice of trivialization, one defines
$\CZcan(z_x):=\CZcan(\Psi_x)$.
The relation to the Morse index is
\begin{equation}\label{conv:CZ=Morse-conv}
     \CZcan(z_x)
     ={\color{magenta} +\:}\IND_{\Ss_V}(x),
\end{equation}
as shown in~\citerefFH{weber:2002a}.\footnote{
  The identity $\CZ=-\IND_{\Ss_V}$ in~\citerefFH{weber:2002a}
  uses the anti-clockwise normalization of $\CZ$.
  }
If $x^*TM\to\SS^1$ is not orientable a correction term enters.

\begin{remark}[Homology or cohomology?]\label{rem:grad-HF-T*Q}
As the energy functional $\Ss_V:\Ll Q\to\R$ given
by~(\ref{eq:classical-action}) is bounded below, the \emph{downward} gradient
direction is the right choice to construct Morse homology, whereas
counting upwards naturally suits cohomology.
The functional is Morse for generic $V$ and the critical points are given by
$\Crit\,\Ss_V=\Pp(V)$. For each of them there is a finite Morse index
$\IND_{\Ss_V}(x)$ and the index function $\IND_{\Ss_V}$ decreases
along isolated connecting \emph{downward} gradient flow lines under the
Morse-Smale condition. Thus $\IND_{\Ss_V}$ is a natural grading of
Morse homology $\HM_*(\Ll Q,\Ss_V)$.

Because the critical points of $\Ss_V$ coincide with those
of $\Aa_V$ under the correspondence $x\mapsto z_x=(x,\dot x)$
and there is the identity~(\ref{conv:CZ=Morse-conv})
of indices and the identity $\Ss_V(x)=\Aa_V(z_x)$ of functionals,
it is natural to use the canonical Conley-Zehnder index
$\CZcan$ and the downward gradient of $\Aa_V$ to construct Floer
homology of cotangent bundles.
\end{remark}
\end{NOTATION}

\bibliographystyleintro{alpha}
\cleardoublepage
\phantomsection
\addcontentsline{toc}{section}{References}

\begin{bibliographyintro}{}
\end{bibliographyintro}

\cleardoublepage
\phantomsection

\cleardoublepage
\phantomsection
\part{Hamiltonian dynamics}\label{sec:Ham-Dyn}
\chapter{Symplectic geometry}
\chaptermark{Symplectic geometry}
\label{sec:symp-top}

Consider a manifold $M$ of finite dimension. A
\textbf{\Index{Riemannian metric}} $g$
on $M$ is a family of symmetric non-degenerate
bilinear forms $g_x$ on $T_xM$ varying smoothly in $x$.
To define a \textbf{\Index{symplectic form}} $\omega$ on
$M$ replace symmetric by
skew-symmetric\footnote{
  Such non-degenerate skew-symmetric $\omega_x$ is called
  a \textbf{\Index{symplectic bilinear form}} on $T_xM$.
  }
-- consequently the dimension is necessarily even, say $2n$ -- and
impose, in addition, the \emph{global} condition that the
non-degenerate differential 2-form $\omega$ is closed
($d\omega=0$). 
Symplectic manifolds are orientable since the $2n$-form
$\omega^{\wedge n}$ nowhere vanishes by non-degeneracy of $\omega$,
in other words $\omega^{\wedge n}$ is a volume form.
Thus, if the manifold $M$ is \emph{closed}, then the differential form
$\omega$ cannot be exact by Stoke's theorem.
Thus the global condition $d\omega=0$ implies
that $[\omega]\not=0$. Hence the second real cohomology of a closed
symplectic manifold is necessarily non-trivial.
For existence of symplectic structures
see e.g.~\citesymptop{Gompf:2001a,Salamon:2013a}.

In contrast to Riemannian geometry there are
\emph{no local invariants} in symplectic geometry:
By Darboux's Theorem a symplectic manifold
looks locally like the prototype symplectic vector space
$(\R^{2n},\omega_0)$.
In contrast to Riemannian geometry\footnote{
  The space of Riemannian metrics is convex, hence contractible, thus
  topologically trivial.
  }
the global theory is rich, already for the space $\Sp(2n)$ of \emph{linear}
symplectic transformations of $(\R^{2n},\omega_0)$.
In Chapter~\ref{sec:symp-top} we follow
mainly~\cite{mcduff:1998a}.

\begin{exercise}
Show that only one of the unit spheres
$\SS^k\subset\R^{k+1}$, $k\in\N$, carries a symplectic form.
Which of the tori $\T^k:=(\SS^1)^{\times k}$
carry a symplectic form? How about the real projective plane $\RP^2$?
And, in contrast, the $\CP^k$'s?
\end{exercise}

\begin{exercise}[The unit $2$-sphere $\SS^2\subset\R^3$]\label{exc:2-sphere}
Show that\index{symplectic form!on $\SS^2$}
$$
     \omega_p(x,y):=\langle p,x\times y\rangle,\qquad
     p\in\SS^2,\quad x,y\in(\R p)^\perp,
$$
defines a symplectic form on $\SS^2$ and that the unit tangent
bundle $T^1 \SS^2$ is diffeomorphic to 
$\SO(3)$.\index{$\SO(3)\cong T^1\SS^2$}
[Hint: The three columns of any matrix $\ba\in\SO(3)$ are of the
form $p,v,p\times v$ where $p\perp v$ are unit vectors.]
\end{exercise}

\begin{exercise}
Show that $\omega$ on $\SS^2$ defined above is
in cylindrical coordinates given by
$
     \omega_{\rm cyl}=d\theta\wedge dz
$, for
$
     (\theta,z)\in[0,2\pi)\times(-1,1)
$,
and in spherical coordinates by
$
     \omega_{\rm sph}=(\sin\varphi)\, d\theta\wedge d\varphi
$, for
$
     (\theta,\varphi)\in[0,2\pi)\times(0,\pi)
$.
\end{exercise}

\section{Linear theory}\label{sec:linear-theory}
The \textbf{\Index{symplectic linear group}} $\Sp(2n)$
consists of all real $2n\times 2n$ matrices $\Psi$
that preserve the standard symplectic structure
$\omega_0="dx\wedge dy\,"$, 
that is
\begin{equation}\label{eq:pres-symp-lin}
     \omega_0=\Psi^*\omega_0:=\omega_0(\Psi\cdot,\Psi\cdot).
\end{equation}
Observe that this identity implies that $\det\Psi=1$.

\begin{exercise}\label{exc:pres-symp-lin}
Show that~(\ref{eq:pres-symp-lin}) is equivalent to
$$
     \Psi^TJ_0\Psi=J_0
$$
where $\Psi^T$ is the transposed matrix.
[Hint: Compatibility with euclidean metric.]
\end{exercise}

Consider the  group $\GL(2n,\R)$ of invertible real $2n\times 2n$
matrices. The orthogonal group $\O(2n)\subset\GL(2n)$ is the
subgroup of those matrices that preserve the euclidean metric.
The linear map $\R^{2n}\to\C^n$, $z=(x,y)\mapsto x+iy$,
is an isomorphism of vector spaces that identifies $J_0$ and the imaginary unit $i$.
Under this identification $X+iY\in\GL(n,\C)$ corresponds~to
$$
     \begin{pmatrix} X&-Y\\Y&X\end{pmatrix}\in\GL(2n,\R).
$$
Similarly the unitary group $\U(n)$ is a subgroup of $\GL(2n,\R)$,
in fact of $\Sp(2n)$.

\begin{exercise}\label{exc:Un}
Show that the identities of real $n\times n$ matrices
$$
     X^TY=Y^TX,\qquad X^TX+Y^TY=\1,
$$
are precisely the condition that $X+iY\in\U(n)$.
\end{exercise}

\begin{exercise}\label{exc:O-Sp-Gl=U}
Prove that the intersection of any two of $\O(2n)$, $\Sp(2n)$, and
$\GL(n,\C)$ is precisely $\U(n)$ as indicated by Figure~\ref{fig:fig-O-Sp-Gl=U}.
\end{exercise}
\begin{figure}[h]
  \centering
  \includegraphics[width=0.5\textwidth]
                             {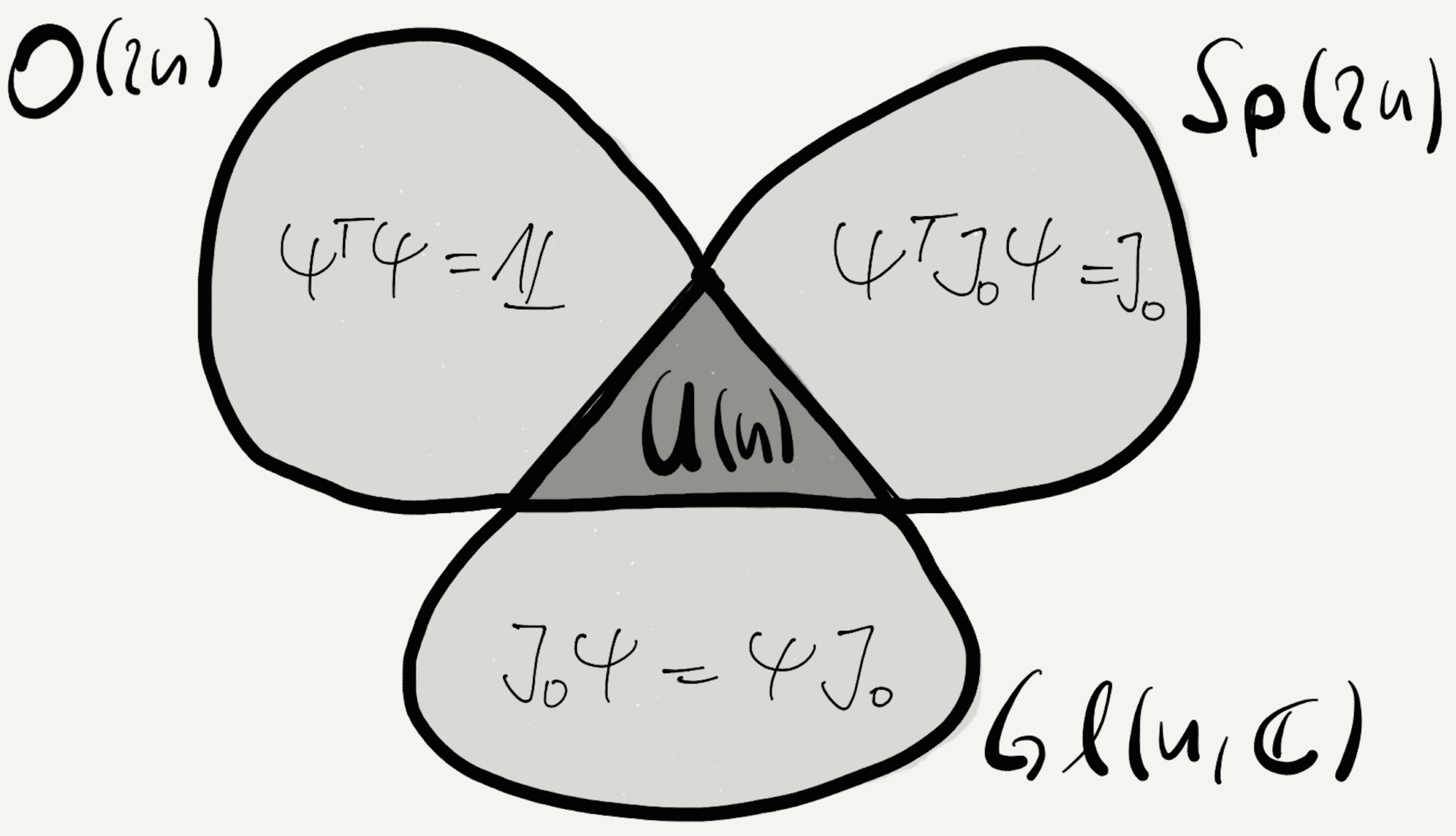}
  \caption{Relation among four classical matrix Lie groups}
  \label{fig:fig-O-Sp-Gl=U}
\end{figure}

The \textbf{eigenvalues of a symplectic matrix} occur 
either\index{eigenvalues! of symplectic matrix} as
pairs $\lambda,\lambda^{-1}\in\R\setminus\{0\}$ or
$\lambda,\bar\lambda\in\SS^1$ or as complex quadruples
$$
     \lambda,\lambda^{-1},\bar\lambda,{\bar\lambda}^{-1}.
$$
In particular, both $1$ and $-1$ occur with even multiplicity.

\subsection{Topology of $\Sp(2)$}
Major topological properties of $\Sp(2n)$,
such as the fundamental group being $\Z$
or the existence of certain cycles, can
nicely be visualized using the
Gel{$'$}fand-Lidski{\u\i}~\citesymptop{Gelfand:1958a}
homeomorphism between $\Sp(2)$ and the open
solid 2-torus in $\R^3$. It is a diffeomorphism
away from the center circle $\U(1)$; for details
see also~\citesymptop[App.~D]{weber:1999a}.
Consider the sets $\Cc_\pm$ of all symplectic matrices
which have $\pm1$ in their spectrum.
$$
     \textit{The set $\Cc:=\Cc_+$ is called
     the \textbf{\Index{Maslov cycle}} of $\Sp(2n)$.}
$$
It consists of disjoint subsets
\index{cycle!Maslov --}
which are submanifolds, called strata. For $n=1$
there are only two strata one of which
contains only one element, namely the identity matrix ${\rm E}=\1$;
see Figure~\ref{fig:fig-Maslov-cycle} which shows
$\Cc_\pm$ in the Gel{$'$}fand-Lidski{\u\i}
parametrization of $\Sp(2)$, namely, as the open solid $2$-torus in $\R^3$.

\begin{figure}[h]
  \centering
  \includegraphics
                             [height=4cm]
                             {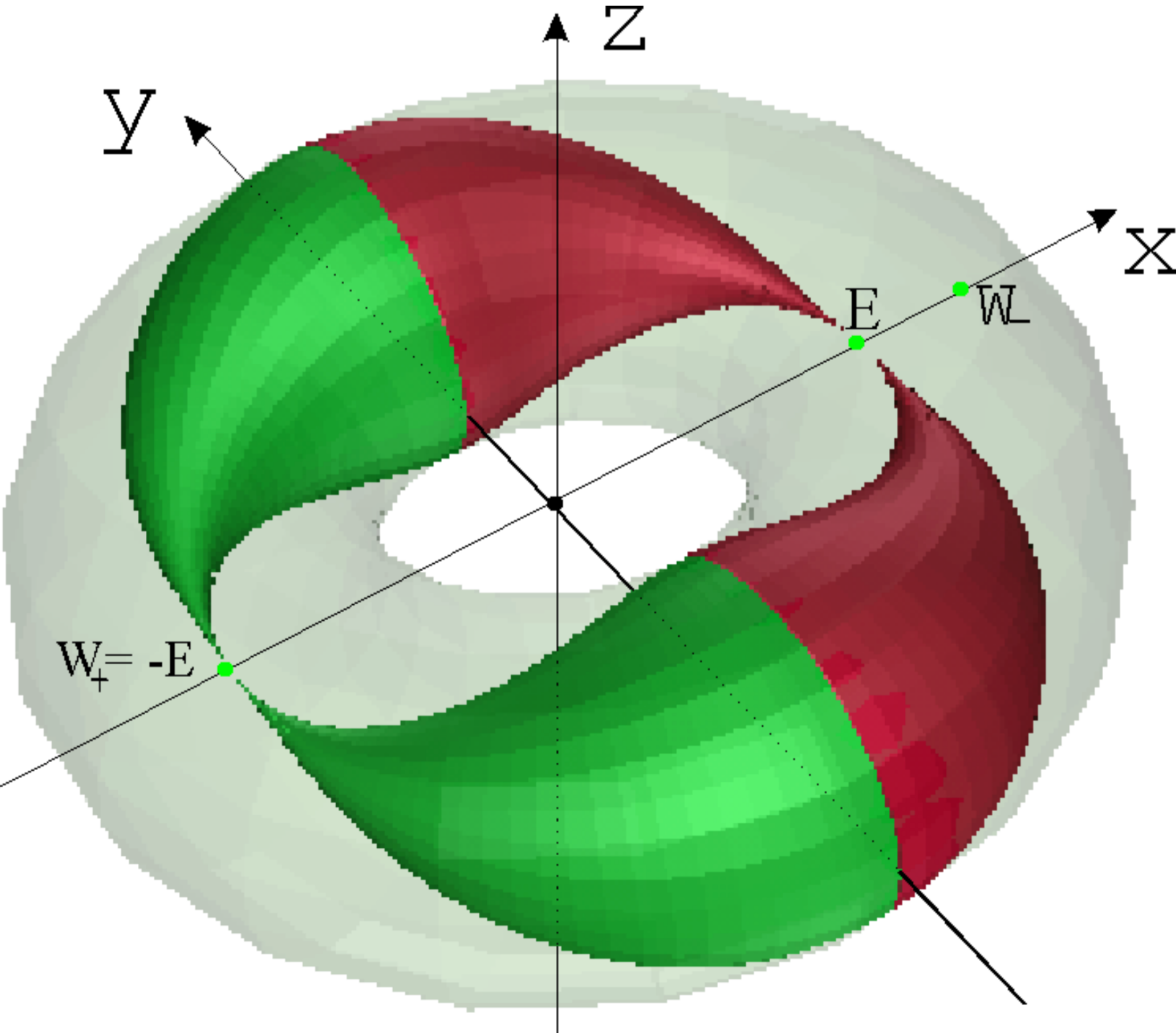}
  \caption{Subsets ${\color{green} \Cc_-}$ and ${\color{red} \Cc_+}$
                 of $\Sp(2)=\SS^1\times \INT\D$. Notation ${\rm E}=\1$}
  \label{fig:fig-Maslov-cycle}
\end{figure}

For the spectrum of the elements $\Psi$ of $\Sp(2)$
there are three possibilities:
\begin{enumerate}
\item[\boldmath\bf (pos. hyp.)]
  real positive pairs $\lambda,\lambda^{-1}>0$;
  \hfill
   those $\Psi$ with $\lambda=1$ are ${\color{red} \Cc_+}$;
\item[\boldmath\bf (neg. hyp.)]
  real negative pairs $\lambda,\lambda^{-1}<0$;
  \hfill
   those $\Psi$ with $\lambda=-1$ are ${\color{green} \Cc_-}$;
\item[\bf (elliptic)]
  complex pairs
  $\lambda,\bar\lambda\in\SS^1\setminus\{{\color{green} -1},{\color{red} +1}\}$;
  \hfill
  those enclosed by ${\color{green} \Cc_-}\cup {\color{red} \Cc_+}$.
\end{enumerate}

\begin{exercise}[Eigenvalues of first and second kind]
\label{exc:1st-2nd-kind}
\index{eigenvalue!of first/second kind}
Note that the set enclosed by ${\color{green} \Cc_-}\cup {\color{red} \Cc_+}$
has \emph{two} connected components. What distinguishes them?\footnote{
  Those $\Psi$ whose eigenvalue of the first kind
  lies in the upper half plane, e.g. $J_0$, lie in the same
  component, those where this location
  is the lower half plane lie in the other component.
  }
Suppose $\lambda,\bar\lambda\in\SS^1\setminus\{-1,+1\}$
are eigenvalues of $\Psi\in\Sp(2n)$.
Thus $\lambda\not=\bar\lambda$ and the eigenvectors
$\xi_\lambda,\xi_{\bar\lambda}=\overline{\xi_\lambda}\in\C^n\setminus\R^n$
are linearly independent. Show that
$\omega_0\left(\overline{\xi_\lambda},\xi_\lambda\right)\in i\R\setminus\{0\}$.
If the imaginary part of this quantity is positive, then
$\lambda$ is called \textbf{of the first kind}, otherwise
\textbf{of the second kind}. Show that one of $\lambda,\bar\lambda$
is of the first kind and the other one of the second kind.
\end{exercise}

For $\Sp(2n)$, $2n\ge 4$, where the
eigenvalues can be quadruples
the notion of eigenvalues of the first and second kind
becomes important concerning \emph{stability properties
of Hamiltonian trajectories}: Two pairs of eigenvalues on $\SS^1$
can meet and leave $\SS^1$, if and only if, eigenvalues of different
kind meet.

\subsection{Maslov index $\mu$}\label{sec:MI-mu}
Loops are continuous throughout Section~\ref{sec:MI-mu}.
The map
\begin{equation}\label{eq:str-def-retr-Sp-U}
     h:[0,1]\times \Sp(2n)\to\Sp(2n),\quad
     (t,\Psi)\mapsto (\Psi\Psi^T)^{-t/2} \Psi,
\end{equation}
is a strong deformation retraction of $\Sp(2n)$ onto
$\Sp(2n)\cap\GL(n,\C)\simeq\U(n)$; cf. Figure~\ref{fig:fig-O-Sp-Gl=U}.
So the quotient space $\Sp(2n)/\U(n)$ is contractible.
It is well known that the determinant map $\det:\U(n)\to\SS^1$
induces an isomorphism of fundamental groups. 
Consequently the fundamental group of $\Sp(2n)$
is given by the integers. 
Define a map
$
     \rho:\Sp(2n)
     \to\Sp(2n)\cap\GL(n,\C)
     \to\SS^1
$
by
\begin{equation}\label{eq:map-rho}
     \rho:\Psi\mapsto h(\Psi,1)=
     \begin{pmatrix} X&-Y\\Y&X\end{pmatrix}
     \mapsto \det\bigl(\underbrace{X+iY}_{\in\U(n)}\bigr).
\end{equation}

\subsubsection*{Maslov index -- degree}
An explicit isomorphism
$[\mu]:\pi_1(\Sp(2n))\to\Z$ is realized by the
\textbf{\Index{Maslov index}} $\mu$
which assigns to every loop
$\Phi:\SS^1\to\Sp(2n)$ of symplectic matrices
the integer
$$
     \mu(\Phi):=\deg\left(\SS^1\stackrel{\Phi}{\longrightarrow}\Sp(2n)
     \stackrel{\rho}{\longrightarrow}\SS^1\right).
$$

\begin{exercise}\label{exc:Maslov}
Show that the Maslov index satisfies the following axioms:
\begin{enumerate}
\item[\texttt{(homotopy)}]
  Two loops in $\Sp(2n)$ are homotopic iff they have the
  same Maslov index.
\item[\texttt{(product)}]
  For any two loops $\Phi_1,\Phi_2:\SS^1\to\Sp(2n)$ we have
  $$
     \mu(\Phi_1\Phi_2)=\mu(\Phi_1)+\mu(\Phi_2).
  $$
  Consequently $\mu(\1)=0$, hence $\mu(\Phi^{-1})=-\mu(\Phi)$
  where $\Phi^{-1}(t):=\Phi(t)^{-1}$.
\item[\texttt{(direct sum)}]
  If $n=n^\prime+n^{\prime\prime}$, then
  $\mu(\Phi^\prime\oplus\Phi^{\prime\prime})=\mu(\Phi^\prime)
  +\mu(\Phi^{\prime\prime})$.
\item[\texttt{(normalization)}]
  The loop $\Phi:\R/\Z\to\U(1)$, $t\mapsto
  e^{i 2\pi t}$,
  has Maslov index $1$.
\end{enumerate}
Show that these axioms determine $\mu$ uniquely.
[Hints: The problem reduces to loops in $\U(n)$ by~(\ref{eq:str-def-retr-Sp-U}).
On diagonal matrizes and products of matrizes
$\det$ behaves nicely.
A complex matrix $\Phi(t)\in\U(n)$ is triangularizable
via conjugation by a unitary matrix, continuously in $t$.
Diagonal elements are loops $\SS^1\to \U(1)$.]
\end{exercise}

\subsubsection*{Maslov index -- intersection number with Maslov cycle}
Looking at the Maslov cycle $\Cc$ in Figure~\ref{fig:fig-MasCyc-Sp_1},
alternatively at the Robbin-Salamon cycle
$\overline{\Sp}_1$,\footnote{
  Write $\Psi\in\Sp(2n)$ in the form of a block matrix
  $\Psi=\begin{pmatrix}A&B\\C&D\end{pmatrix}$ with
  four $n\times n$ matrices and consider the function
  $\chi:\Sp(2n)\to\R$, $\Psi\mapsto\det B$. By definition
  the \textbf{\Index{Robbin-Salamon cycle}} $\overline{\Sp}_1$ is the pre-image
  $\chi^{-1}(0)$,
  \index{cycle!Robbin-Salamon --}
  i.e. $\overline{\Sp}_1$ consists of all $\Psi$ with $\det B$=0.
  Actually $\Sp(2n)$ is partitioned by the submanifolds
  $\Sp_k(2n)$ of codimension $k(k+1)/2$ which consist
  of those $\Psi$ with $\rank\, B=n-k$
  and $\overline{\Sp}_1$ is the complement of
  the codimension zero stratum $\Sp_0(2n)$.
  }
suggests that the Maslov index $\mu(\Phi)$
should be half the intersection number with either cycle
of generic, that is transverse, loops $\Phi:\SS^1\to\Sp(2n)$.\footnote{
  This is indeed the case: The \textbf{\Index{Robbin-Salamon index}} $\RS$
  \index{$\RS$ Robbin-Salamon index}
  of a generic loop is the intersection number with $\Sp_1(2n)$;
  see definition in~\citesymptop[\S 4]{Robbin:1993a}.
  The equality $\mu=\frac12 \RS$
  follows either by the formula
  in~\citesymptop[Rmk.~5.3]{Robbin:1993a}
  or by the fact that $\frac12\RS$
  satisfies, by~\citesymptop[Thm.~4.1]{Robbin:1993a},
  the first three axioms for $\mu$ in Exercise~\ref{exc:Maslov}.
  To confirm \texttt{(normalization)}
  calculate $\RS(t\mapsto e^{i2\pi t})=2$.
  Hint: $B(t)=-\sin (2\pi t)$, intersection form
  above~\citesymptop[Thm.\,4.1]{Robbin:1993a}.]
  }

\subsection{Conley-Zehnder index $\CZ$ of symplectic path}
\label{sec:CZ}
Paths and loops are continuous throughout Section~\ref{sec:CZ}.
Consider the map
\begin{equation}\label{eq:map-SP*}
     \Sp(2n)\to\R,\quad
     \Psi\mapsto\det\left(\Psi-\1\right)
\end{equation}
and note that the pre-image of $0$
is precisely the Maslov cycle $\Cc$.
Let $\Sp^*$, $\Sp^*_+$, and $\Sp^*_-$ be the subsets of $\Sp(2n)$ on
which this map is, respectively, different from zero, positive, and negative.
Thus we obtain the partitions
$$
     \Sp(2n)
     =\Sp^*_+\,\mathop{\dot{\cup}}\,\Cc\,\mathop{\dot{\cup}}\,\Sp^*_-,
     ,\qquad
     \Sp^*=\Sp^*_+\,\mathop{\dot{\cup}}\,\Sp^*_-.
$$
Geometrically the \textbf{\Index{Conley-Zehnder index}}
of admissible paths $\Psi$, namely
$$
     \CZ:\SsPp^*(2n):=\{\text{$\Psi:[0,1]\stackrel{C^0}{\to}\Sp(2n)$ $\mid$
     $\Psi(0)=\1$, $\Psi(1)\in \Sp^*$}\}\to\Z
$$
can be defined as intersection number with the Maslov
cycle $\Cc$ of generic paths in $\Sp(2n)$ starting at the
identity and ending away from the Maslov cycle.
Here generic not only means transverse to the codimension one stratum,
but also in the sense that the paths depart from $\1$
immediately into $\Sp^*_-$. The need for the latter condition
can be read off from Figure~\ref{fig:fig-CZ-index} easily.
\begin{figure}[h]
\hfill
\begin{minipage}[b]{.44\linewidth}
  \centering
  \includegraphics[width=0.85\textwidth]
                             {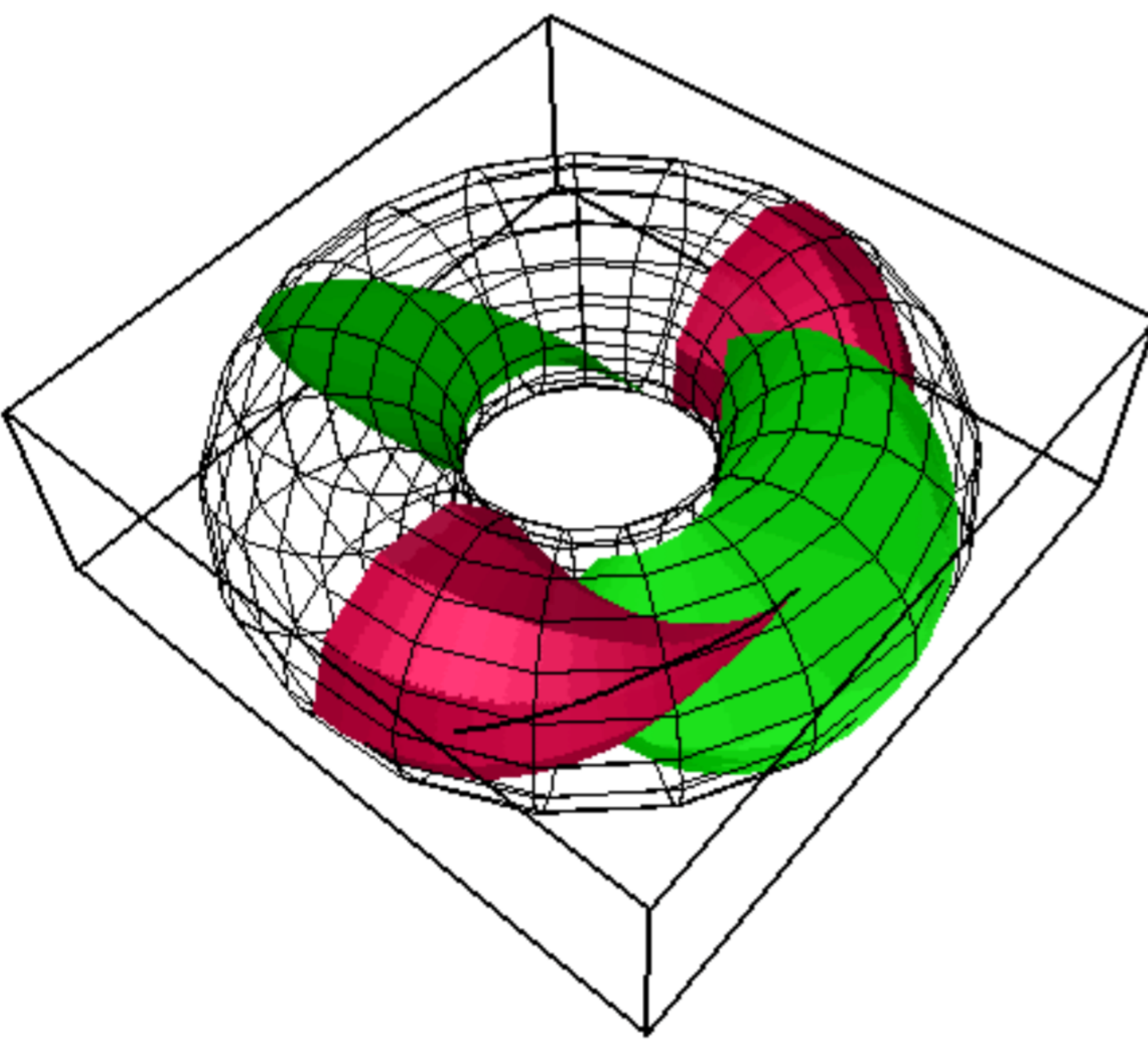}
  \caption{Cycles {\color{red} $\Cc$}, {\color{green} $\Sp_1(2)$}}
  \label{fig:fig-MasCyc-Sp_1}
\end{minipage}
\hfill
\begin{minipage}[b]{.54\linewidth}
  \centering
  \includegraphics[width=0.95\textwidth]
                             {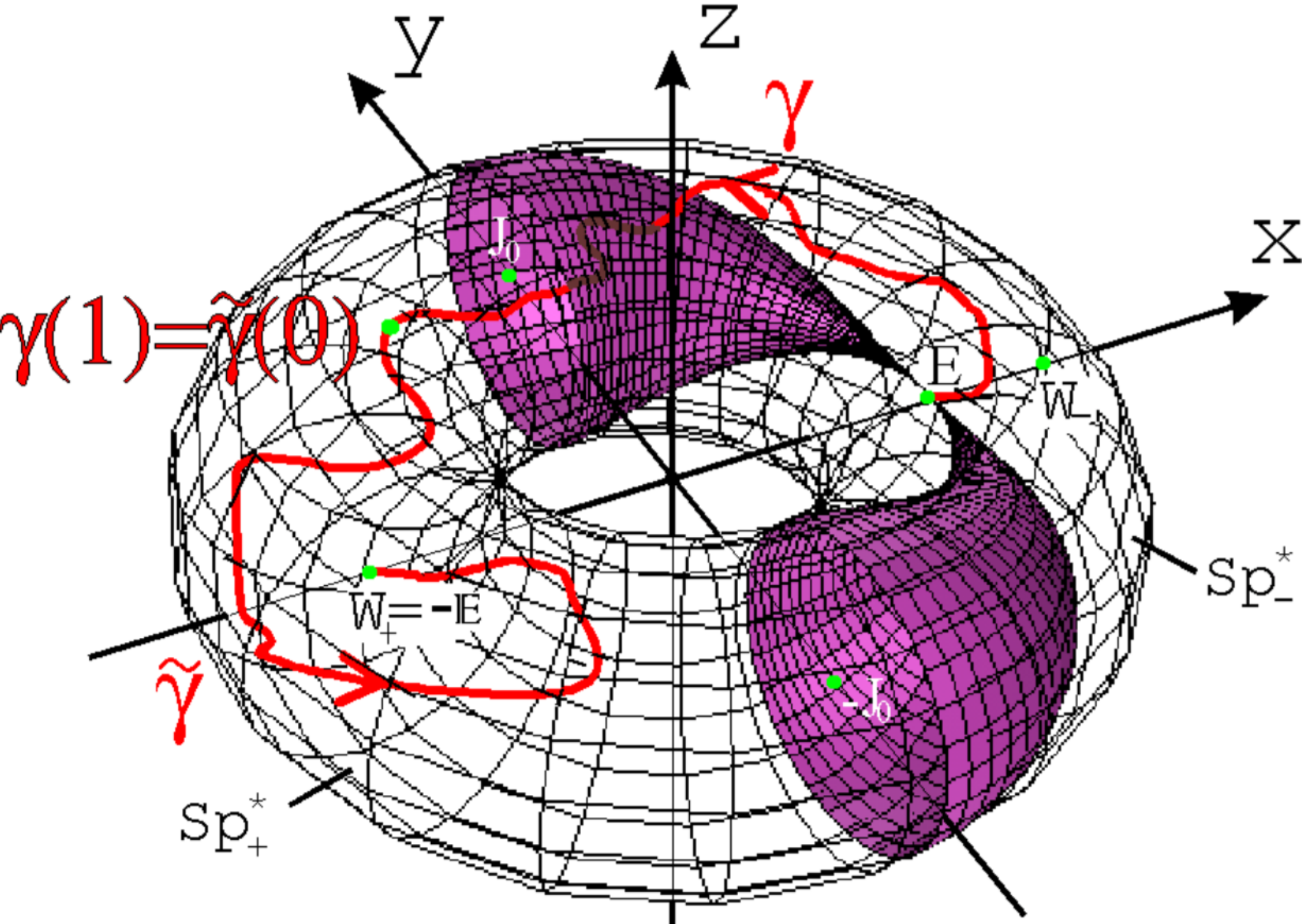}
  \caption{CZ-index $1$ of path $\Psi=\gamma$}
  \label{fig:fig-CZ-index}
\end{minipage}
\hfill
\end{figure}

\begin{exercise}\label{exc:CZ-index}
Following the original definition by Conley and
Zehnder~\citesymptop[\S 1]{conley:1984a}, pick
a path $\Psi\in\SsPp^*(2n)$. Then its endpoint
lies in one of the two \emph{connected} open
sets $\Sp^*_+$ or $\Sp^*_-$. Show that these sets contain, respectively,
the matrices
$$
     W_+:=-\1,\qquad
     W_-:=\diag\left(2,-1,\dots,-1,\tfrac12,-1,\dots,-1\right).
$$
Extend $\Psi$ from its endpoint to the corresponding matrix
inside the component $\Sp^*_\pm$ of the endpoint.
Consider the extended path $\tilde\Psi:[0,2]\to\Sp(2n)$. Define
$$
     \CZ(\Psi):=\deg\left(\rho^2\circ\tilde\Psi\right)
$$
and show that $\rho\circ\Psi(0)=1\in\SS^1$ and that
$$
     \rho (W_+)=\det(-\1)=(-1)^n,\quad
     \rho (W_-)=\det\left(1,-1,\dots,-1\right)=(-1)^{n-1}.
$$
Clearly taking the square of $\rho$ yields $+1$ in either case.
Hence the path $\rho^2\circ\tilde\Psi:[0,2]\to\SS^1$ closes up at time
$1$ and therefore taking the degree makes sense.
\end{exercise}

The Conley-Zehnder index $\CZ:\SsPp^*(2n)\to\Z$
satisfies certain axioms, similar to those
of the Maslov index $\mu$ in Exercise~\ref{exc:Maslov}.

\begin{theorem}[Conley-Zehnder index]\label{thm:CZ}
For\index{$\CZ$ Conley-Zehnder index}
$\Psi\in\SsPp^*(2n)$ the following holds.
\begin{enumerate}
\item[\rm\texttt{(homotopy*)}]
  The Conley-Zehnder index is constant on the components of $\SsPp^*(2n)$.
  \index{\rm \texttt{(homotopy)}}
\item[\rm\texttt{(loop*)}]
  $
     \CZ(\Phi\Psi)=2\mu(\Phi)+\CZ(\Psi)
  $
  for any loop $\Phi:\SS^1\to\Sp(2n)$.
  \index{\rm \texttt{(loop)}}
\item[\rm\texttt{(signature*)}]
  If $S=S^T\in\R^{2n\times 2n}$ is a symmetric matrix
  of norm $\norm{S}<2\pi$, then
  \index{\rm \texttt{(signature)}}
  $$
     \CZ\left(t\mapsto e^{tJ_0 S}\right)
     =\frac12\sign(S):=\frac{n^+(S)-n^-(S)}{2}
  $$
  where $\sign(S)$ is the \textbf{\Index{signature}} of $S$
  and $n^\pm(S)$ is the number of positive/negative eigenvalues of $S$.
\item[\rm\texttt{(direct sum)}]
  If $n=n^\prime+n^{\prime\prime}$, then
  $\mu(\Psi^\prime\oplus\Psi^{\prime\prime})=\mu(\Psi^\prime)
  +\mu(\Psi^{\prime\prime})$. \index{\rm \texttt{(direct sum)}}
\item[\rm\texttt{(naturality)}]
  $\CZ(\Theta\Psi \Theta^{-1})=\CZ(\Psi)$ for any path
  $\Theta:[0,1]\to\Sp(2n)$.\index{\rm \texttt{(naturality)}}
\item[\rm\texttt{(determinant)}]
  $
     (-1)^{n-\CZ(\Psi)}=\sign\det\left(\1-\Psi(1)\right)
  $.\index{\rm \texttt{(determinant)}}
\item[\rm\texttt{(inverse)}]
  $
     \CZ(\Psi^{-1})=\CZ(\Psi^{T})=-\CZ(\Psi)
  $.\index{\rm \texttt{(inverse)}}
\end{enumerate}
\end{theorem}

The {\rm\texttt{(signature)}} axiom normalizes $\CZ$. The $*$-axioms
determine $\CZ$ uniquely; see e.g.~\citerefFH[\S 2.4]{salamon:1999a}.

\begin{remark}[Canonical Maslov and Conley-Zehnder indices]
\label{rem:canonical-Maslov-CZ}
Clockwise\index{canonical Maslov index}\index{Maslov index!canonical --}
rotation\index{$\mucan$ canonical Maslov index}
appears\index{$\CZcan$ canonical Conley-Zehnder index}
naturally in Hamiltonian dynamics,
cf.~(\ref{eq:HamEqs}), since $\Jbar_0:=-J_0$ is compatible with
$\omegacan=dp\wedge dq$, not $J_0$.
Thus it is natural and convenient to introduce versions of
the Maslov and Conley-Zehnder indices normalized clockwise
and denoted by $\mucan$ and $\CZcan$, respectively,
namely\index{\rm\texttt{(normalization}$\texttt{)}_{\texttt{can}}$}
\begin{enumerate}
\item[$\rm\texttt{(normalization)}_{\texttt{can}}$]
  \mbox{ }
  \newline \vspace{-.85cm}
  \begin{equation}\label{eq:CZ-normalization-canonical-88}
     \mucan({\color{magenta} \{ e^{-i 2\pi t} \}_{t\in[0,1]}}):=1
     ,\qquad
     \CZcan({\color{magenta} \{ e^{-i t} \}_{t\in[0,1]}}):=1.
  \end{equation}
\end{enumerate}
These indices are just the negatives of the standard anti-clockwise normalized
indices, that is $\mucan=-\mu$ and $\CZcan=-\CZ$.
They satisfy corresponding versions of the previously stated
axioms, e.g. \texttt{(loop)} becomes
$\rm\texttt{(loop)}_{\texttt{can}}$
$\CZcan(\Phi\Psi)=2\mucan(\Phi)+\CZcan(\Psi)$ and
$\rm\texttt{(signature)}_{\texttt{can}}$ is displayed
in~(\ref{eq:signature_can}).

In the present text we use for Floer, and also Rabinowitz-Floer,
homology the canonical (clockwise) version $\CZcan$ of the
Conley-Zehnder index, because on cotangent bundles (see
Section~\ref{sec:FH-T^*M}) these theories relate canonically to the
classical action functional $\Ss_V$ which requires no choices at all
to establish Morse homology; see Remark~\ref{rem:grad-HF-T*Q}.

The reason why we introduced here in great detail
the standard (counter-clockwise) version $\CZ$ is better comparability
with the literature. It spares the reader continuously translating
between the normalizations. So while we explain $\CZ$, the reader can
conveniently consult the literature for details of proofs, and once
everything is established for $\CZ$ we simply note that
$\CZcan(\Psi)=-\CZ(\Psi)$.
\end{remark}

Symmetric matrizes are rather intimately tied to symplectic geometry:

\begin{exercise}
Show that for symmetric $S=S^T\in\R^{2n\times 2n}$
the matrizes $e^{t  J_0 S}$ and $e^{t \Jbar_0 S}$ are elements of
$\Sp(2n)$ whenever $t\in\R$.
\end{exercise}

\begin{exercise}[Symmetric matrizes]\label{exc:symm-symp}
More generally, given a path
$[0,1]\ni t\mapsto S(t)=S(t)^T\in\R^{2n \times 2n}$ of
symmetric matrizes, show that the path of matrizes
$\Psi:[0,1]\to\R^{2n \times 2n}$ determined by the initial value problem
\begin{equation}\label{eq:symm-symp}
     \frac{d}{dt}\Psi(t)
     =J_0 S(t)\Psi(t),\qquad \Psi(0)=\1,
\end{equation}
takes values in $\Sp(2n)$. 
Note that $\Psi\in\SsPp^*(2n)$ iff
$\det\left(\1-\Psi(1)\right)\not= 0$.
Vice versa, given a symplectic $C^1$ path $\Psi$, then
the family of matrizes defined by
\begin{equation}\label{eq:symp-symm}
     S(t):=-J_0\dot\Psi(t) \Psi(t)^{-1}
\end{equation}
is symmetric. 
\end{exercise}

\begin{exercise}\label{exc:Mas-co-or}
a)~The map~(\ref{eq:map-SP*}) provides a natural co-orientation\footnote{
  A \textbf{\Index{co-orientation}} is an orientation of the normal
  bundle (to the top-diml. stratum).
  }
of the Maslov cycle $\Cc$.
Does this co-orientation serve to define the Maslov index $\mu$
as intersection number with $\Cc$?
(For simplicity suppose $n=1$.)
[Hint: Given this co-orientation, calculate $\mu$ for any generic loop
winding around `the hole'~once.]
\\
b)~Is the situation better for the Robbin-Salamon cycle
$\overline{Sp}_1$? Suppose $n=1$ and co-orient $\overline{Sp}_1=\Sp_1$
by the increasing direction of the function $\chi$ in an earlier footnote.
Show that the intersection number with $\Sp_1$ of the loop
$\R/\Z\ni t\mapsto e^{i2\pi t}\in\U(1)\subset\Sp(2)$
is $-1$ at $t=0$ and $+1$ at $t=1/2$.
\\
c)~For $n=1$ consider the \textbf{\Index{parity}} $\nu(B,D)$ of $\Psi\in\Sp_1$
defined in~\citesymptop[Rmk.~4.5]{Robbin:1993a}. Check that for
$t\mapsto e^{i2\pi t}$ the parity is $-1$ at $t=0$ and $+1$ at
$t=1/2$. Check that the intersection number of loops
with $\Sp_1$ co-oriented by \emph{$\chi$-co-orientation times parity}
recovers the Maslov index $\mu$.
\end{exercise}

By now several alternative descriptions of the Conley-Zehnder index
have been found, for instance, the interpretation as intersection number
with the Maslov cycle of a symplectic path, even with arbitrary endpoints,
has been defined by Robbin and Salamon~\citesymptop{Robbin:1993a};
see Section~\ref{sec:RS-index}.
In case $n=1$ there is a description of $\CZ$ in terms of winding
numbers which we discuss right below.

For further details concerning Maslov, Conley-Zehnder, and other indices
see e.g.~\citesymptop{Arnold:1967a,conley:1984a,Robbin:1993a,
Gutt:2014a-arXiv-link} and\citerefFH{salamon:1999a}.

\subsubsection*{\boldmath Winding number descriptions of $\CZ$ in the 
case $n=1$}
For the following geometric and analytic construction
we recommend the presentations in \citesymptop[\S 8]{Hofer:2003a}
and \citesymptop[\S 2]{Hryniewicz:2015a}.
It is convenient to naturally identify $\R^2$ with $\C$ and $J_0$ with
$i$.

\vspace{.2cm}
\noindent
\textit{Geometric description (winding intervals
\citesymptop[\S 3]{Hofer:1999a}).}
A path $\Psi:[0,1]\to\Sp(2)$ with $\Psi(0)=\1$
uniquely determines via the identity
$$
     \Psi(t) z_s=r(t,s) e^{i\theta(t,s)},\qquad
     z_s:=e^{i2\pi s}\in\SS^1,
$$
two continuous functions $r$ and $\theta$. Note that $r>0$
and $\theta(0,s)=2\pi s$.
Define the winding number of the point $z_s:=e^{i2\pi s}\in\SS^1$
under the symplectic path $\Psi$, i.e. the change in argument
of $[0,1]\ni t\mapsto \Psi(t)z_s\in\C\setminus\{0\}$,
see Figure~\ref{fig:fig-CZ-winding},
\begin{figure}
  \centering
  \includegraphics
                             [height=4cm]
                             {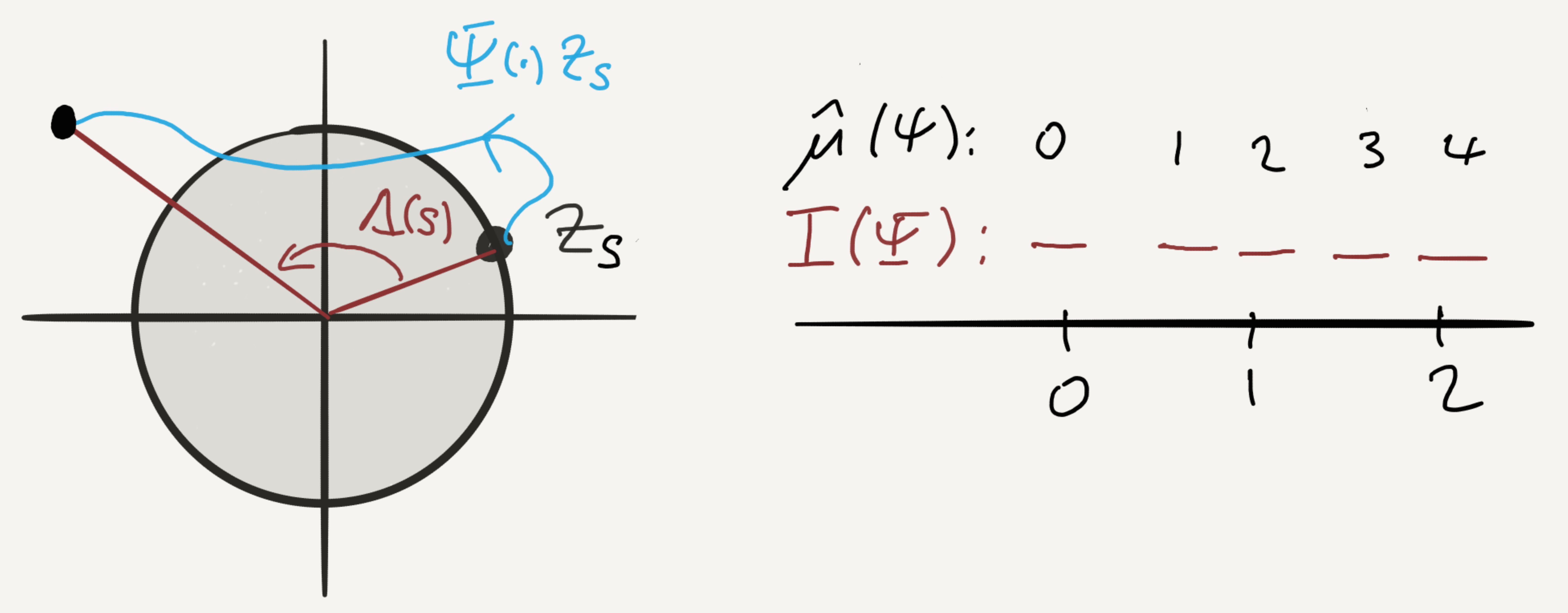}
  \caption{Winding number $\Delta(s)$ of point $z_s=e^{i2\pi s}\in\SS^1$
                under path $\Psi$}
  \label{fig:fig-CZ-winding}
\end{figure}
by 
$$
     \Delta(s):=\frac{\theta(1,s)-2\pi s}{2\pi}\in\R.
$$
The \textbf{winding interval of the symplectic path}
$\Psi$ \index{winding interval of path} is the union
$$
     I(\Psi):=\{\Delta(s)\mid s\in[0,1]\}
$$
of the winding numbers under $\Psi$ of the elements of $\SS^1$.
The interval $I(\Psi)$ is compact, its boundary is disjoint from the integers
iff $\Psi(1)\notin\Cc$, that is iff $\Psi\in\SsPp^*(2)$, and most
importantly its length $\abs{I(\Psi)}<1/2$ is less then $1/2$.
Thus, for $\Psi\in\SsPp^*(2)$, the winding interval
either lies between two consecutive integers or contains
precisely one of them in its interior. Thus one can define
$$
     \mu^\prime(\Psi):=
     \begin{cases}
        2k&\text{, if $k\in I(\Psi)$,}
        \\
        2k+1&\text{, if $I(\Psi)\subset(k,k+1)$,}
     \end{cases}
$$
for some integer $k\in\Z$. One verifies the $*$-axioms in
Theorem~\ref{thm:CZ} to get that
$
     \mu^\prime=\CZ
$
is the Conley-Zehnder index itself.

Observe that the winding number $\Delta(s)$ is an integer $k$
iff $\Psi(1)z_s=\lambda z_s$ is a positive multiple of $z_s$.
But the latter means that $\lambda$ is a positive eigenvalue of
$\Psi(1)$. Thus $k\in I(\Psi)$ which shows
that positive hyperbolic paths are of \emph{even} Conley-Zehnder index.
Similar considerations show that negative hyperbolic and elliptic paths
both have \emph{odd} Conley-Zehnder indices.

\vspace{.2cm}
\noindent
\textit{Analytic description (eigenvalue winding numbers,
\citesymptop[\S 3]{Hofer:1995b}).}
The integer $\mu^\prime(\Psi)$ can be characterized in terms of the spectral
properties of the unbounded self-adjoint differential operator on
$L^2$ with dense domain $W^{1,2}$, namely
$$
     L_S:=-J_0\frac{d}{dt}-S(t):L^2(\SS^1,\R^{2})\supset W^{1,2}\to L^2
$$
where the family $S$ of symmetric matrices corresponds to $\Psi$
via~(\ref{eq:symp-symm}).
Here we assume that the symplectic path $\Psi$ is defined on $\R$
and satisfies $\Psi(t+1)=\Psi(t)\Psi(1)$.
This extra condition corresponds to periodicity $S(t+1)=S(t)$.

The spectrum $\sigma(L_S)$ of the operator $L_S$ consists,
by compactness of the resolvent, of countably many
isolated real eigenvalues of finite multiplicity
accumulating precisely at $\pm\infty$.
Suppose that $v:\SS^1\to\R^2$ is eigenfunction
associated to an eigenvalue $\lambda$.
Note that $v:\SS^1\to\R^2$ cannot have any zero.
Thus we can write $v(t)=\rho(t) e^{i\vartheta(t)}$
and define its winding number by
$\wind(v):=\frac{\vartheta(1)-\vartheta(0)}{2\pi}$.
This integer only depends on the eigenvalue $\lambda$,
but not on the choice of eigenvector. So it is denoted by
$\wind(\lambda)$ and called the \textbf{\boldmath winding number of the
eigenvalue $\lambda$}. \index{winding number of eigenvalue}
For each integer $k$ there are precisely two eigenvalues (counted
with mulitplicities) whose winding number is $k$.
If there is only one such eigenvalue, its multiplicity is 2.
Moreover, if $\lambda_1\le\lambda_2$, then
$\wind(\lambda_1)\le\wind(\lambda_2)$.
\\
Let $\lambda_-(S)<0$ be the largest negative eigenvalue
and $\lambda_+(S)\in\N_0$ the next larger one.
Define the maximal winding number among the negative eigenvalues
of the operator $L_S$ and its parity by
$$
     \alpha(S):=\wind(\lambda_-)\in\Z,\qquad
     p(S):=
     \begin{cases}
        0&\text{, if $\wind(\lambda_-)=\wind(\lambda_+)$,}
       \\
        1&\text{, if $\wind(\lambda_-)<\wind(\lambda_+)$.}
     \end{cases}
$$

\begin{theorem}
If $\Psi\in\SsPp^*(2)$, then $2\alpha(S)+p(S)=\mu^\prime(\Psi)$.
\end{theorem}

\subsection{Lagrangian subspaces}
A \textbf{\Index{symplectic vector space}} $(V,\omega)$
is a real vector space with a non-degenerate
skew-symmetric bilinear form. So $\dim V=2n$
is necessarily even. 

\begin{exercise}
Show that, firstly, each symplectic vector space
admits a \textbf{\Index{symplectic basis}},
that is vectors $u_1,\dots,u_n,v_1\dots,v_n$ such that
$$
     \omega(u_j,u_k)=\omega(v_j,v_k)=0,\quad
     \omega(u_j,v_k)=\delta_{jk},
$$
and, secondly, there is a \textbf{linear \Index{symplectomorphism}} -- a
vector space isomorphism preserving the symplectic forms -- to
$(\R^{2n},\omega_0)$.
\end{exercise}

The \textbf{\Index{symplectic complement}}
of a vector subspace $W\subset V$ is defined by
$$
     W^\omega:=\ker\, \omega
     :=\{v\in V\mid\text{$\omega(v,w)=0$ $\forall w\in W$}\}.
$$
In contrast to the orthogonal complement, the symplectic
complement is not necessarily disjoint to $V$, but $V$ and $V^\omega$
are still of complementary dimension (as non-degeneracy
is imposed in both worlds) and $(W^\omega)^\omega=W$.
Thus the maximal dimension of $W\cap W^\omega$ is $n=\frac12\dim V$.
Such $W$, that is those with $W=W^\omega$,
are called \textbf{\Index{Lagrangian subspace}s}.
Equivalently these are characterized as the $n$ dimensional subspaces
restricted to which $\omega$ vanishes identically.

A\index{subspace!isotropic --}\index{isotropic subspace}
subspace\index{subspace!coisotropic --}\index{coisotropic subspace}
$W\subset V$ is called \textbf{isotropic} if $W\subset
W^\omega$, in other words, if $\omega$ vanishes on $W$,
and \textbf{coisotropic} if $W^\omega\subset W$.

\begin{exercise}[Graphs of symmetric matrizes are Lagrangian]
\label{exc:symm-Lag}
Show that
$$
     \Gamma_S:=\{(x,Sx)\mid x\in\R^n\}
$$
is a Lagrangian subspace of $(\R^{2n},\omega_0)$
if and only if $S=S^T\in\R^{n\times n}$ is symmetric.
\end{exercise}

\begin{exercise}[Natural structures on $W\oplus W^*$]
\label{exc:W-plus-W*}
Let $W$ be a real vector space and $W^*$ its
dual\index{$W^*:=\Ll(W,\R)$!dual space of $W$}
space.\footnote{
  Here $W^*:=\Ll(W,\R)$ is the \textbf{\Index{dual space}} of the real
  vector space $W$.
  }
Show that on $W\oplus W^*$ a symplectic form $\Omega_0$ is naturally
given by $\left((v,\eta),(\tilde v,\tilde \eta)\right)
\mapsto\tilde\eta(v)=\eta(\tilde v)$.
Show that both summands of $W\oplus W^*$ are Lagrangian.
\index{$\Omega_0$ on $W\oplus W^*$}
\index{$(\Omega_0,\Jbar_g,G_g)$ compatible triple}
Now pick, in addition, an \textbf{\Index{inner product}} on $W$,
that is a non-degenerate symmetric bilinear form $g$ on $W$. This provides a
natural isomorphism $W\to W^*$, $v\mapsto g(v,\cdot)$, again denoted
by $g$, which naturally leads to the inner product
$g^*=g(g^{-1}\cdot,g^{-1}\cdot)$ on $W^*$. Moreover, on $W\oplus W^*$
one obtains an inner product $G_g=g\oplus g^*$ and an
almost complex structure $\Jbar_g$; cf.~(\ref{eq:J_g-G_g}).
Show their compatibility in the sense that $\Omega_0(\cdot,\Jbar_g\cdot)=G_g$.
\end{exercise}

\begin{exercise}
Show that the graph $\Gamma_\Psi$ of a linear
symplectomorphism $\Psi:V\to V$ is a Lagrangian
subspace of the cartesian product $V\times V$
equipped with the symplectic form $(-\omega)\oplus\omega$
that sends $((v,w),(v^\prime,w^\prime))$ to
$-\omega(v,v^\prime)+\omega(w,w^\prime)$. 
Note that the \textbf{\Index{diagonal subspace}}
$\Delta:=\{(v,v)\mid v\in V\}$ is Lagrangian.
\end{exercise}

\subsection{Robbin-Salamon index -- degenerate endpoints}
\label{sec:RS-index}
A symplectic $C^1$ path $\Psi:[0,1]\to\Sp(2n)$ gives rise to the family
of symmetric matrizes $S(t)$ given by~(\ref{eq:symp-symm}) for which, in turn,
it is a solution to the ODE~(\ref{eq:symm-symp}).
A number $t\in[0,1]$ is called a \textbf{\Index{crossing}}
if $\det\left(\1-\Psi(t)\right)=0$ or, equivalently, if $1$ is
eigenvalue of $\Psi(t)$. In other words,  if $\Psi(t)$ hits the Maslov
cycle $\Cc$: The eigenspace
${\rm Eig}_1\Psi(t)=\ker\left(\1-\Psi(t)\right)\not= 0$ must be non-trivial.

At a crossing $t$ the quadradic form given by
$$
     \Gamma(\Psi,t):{\rm Eig}_1\Psi(t)\to\R,\quad
     \xi_0\mapsto\omega_0(\xi_0,\dot\Psi(t)\xi_0)=\inner{\xi_0}{S(t)\xi_0},
$$
is called the \textbf{\Index{crossing form}}.
A crossing is called \textbf{regular} \index{crossing!regular --}
if the crossing form is non-degenerate.
Regular crossings are isolated.
If all crossings are regular the \textbf{\Index{Robbin-Salamon index}}
$\RS(\Psi)$ was introduced in~\citesymptop{Robbin:1993a},
although here we repeat the presentation given
in~\citerefFH[\S\,2.4]{salamon:1999a},
as the sum over all crossings $t$ of the signatures
of the crossing forms where crossings at the boundary points $t=0,1$
are counted with the factor $\frac12$ only.
For the particular paths $\Psi\in\SsPp^*(2n)$ the Robbin-Salamon index
$$
     \RS(\Psi):=\frac12\sign S(0)+\sum_t\sign\Gamma(\Psi,t)=\CZ(\Psi)
$$
reproduces the Conley-Zehnder index.

\begin{exercise}\label{exc:RS-endpoint-factor}
  a)~Check the identity in the definition~of~$\Gamma$.
  b)~The factor $\frac12$ at the endpoints is introduced in order to make
  $\RS(\Psi)$ invariant under homotopies with fixed endpoints.
  To see what happens homotop the path $\gamma$ in
  Figure~\ref{fig:fig-CZ-index} to $\Psi(t)= e^{i\pi t}$
  and calculate the crossing forms at $t=0$ in both cases.
\end{exercise}

In~\citesymptop{Robbin:1993a} an index for a rather more general class
of paths is constructed: A \emph{relative} index
$\RS(\Lambda,\Lambda^\prime)$ for \emph{pairs}
of paths of \emph{Lagrangian subspaces} of a symplectic vector space $(V,\omega)$;
cf. Exercise~\ref{exc:symm-Lag}.
Here crossings are non-trivial intersections
$\Lambda(t)\cap\Lambda^\prime(t)\not=\{0\}$.
The Conley-Zehnder index on $\SsPp^*(2n)$ is recovered
by choosing the symplectic vector space
$(\R^{2n}\times\R^{2n},-\omega_0\oplus\omega_0)$
and the Lagrangian path given by the graphs $\Gamma_{\Psi(t)}$ of
$\Psi$ relative to the constant path given by the diagonal $\Delta$.
Indeed $\CZ(\Psi)=\RS(\Gamma_{\Psi},\Delta)$
by~\citesymptop[Rmk.\,5.4]{Robbin:1993a}.
Note that $\Gamma_{\Psi(t)}\cap\Delta\simeq {\rm Eig_1\,\Psi(t)}$.

\section{Symplectic vector bundles}\label{sec:SVB}
Suppose $E\to N$ is a vector bundle of real rank $2n$ over a
manifold-with-boundary of dimension $k$; where
$\p N=\emptyset$ is not excluded.
A \textbf{symplectic vector bundle}\index{vector bundle!symplectic --}
is a pair $(E,\omega)$ where $\omega$ is a family of
symplectic bilinear forms $\omega_x$, one on each fiber $E_x$.
Similarly a \textbf{complex vector bundle}\index{vector bundle!complex --}
is a pair $(E,J)$ where $J$ is a family of
\emph{complex structures $J_x$ on the fibers $E_x$},
that is $J_x^2=-\Id_{E_x}$.
Existence of a deformation retraction, such as $h$
in~(\ref{eq:str-def-retr-Sp-U}),
of $\Sp(2n)$ onto $\U(n)$ has the consequence that any
symplectic vector bundle $(E,\omega)$ is fiberwise homotopic,
thus isomorphic as a vector bundle,\footnote{
  A \textbf{vector bundle isomorphism}\index{vector bundle!isomorphism}
  is a diffeomorphism between the total spaces whose fiber restrictions
  are vector space isomorphisms.
  }
to a complex
vector bundle $(E,J_\omega)$ called the \textbf{underlying complex vector
bundle}.\index{vector bundle!underlying complex --}\footnote{
  The complex structure $J_\omega$, but not
  its isomorphism class, depends on $h$.
  } 
A \textbf{Hermitian vector bundle}\index{vector bundle!Hermitian --}
$(E,\omega,J,g_J)$ is a symplectic and complex
vector bundle $(E,\omega,J)$ 
such that $J$ is $\omega$-compatible,
that is $g_J:=\omega(\cdot , J\cdot)$
is a Riemannian bundle metric on $E$.

\begin{proposition}\label{prop:cx=symp-VB}
Two symplectic vector bundles $(E_1,\omega_1)$ and
$(E_2,\omega_2)$ are isomorphic if and only if their underlying
complex bundles are isomorphic.
\end{proposition}

Two proofs are given in~\cite[\S 2.6]{mcduff:1998a},
one based on the deformation retraction~(\ref{eq:str-def-retr-Sp-U}),
the other on constructing a homotopy equivalence
between $\Jj(V,\omega)$, the space of $\omega$-compatible
complex structures on a symplectic vector space, and the convex,
thus contractible, non-empty space of all inner products~on~$V$.

A \textbf{\Index{trivialization}} of a bundle $E$
is an isomorphism to the trivial bundle which preserves
the structure under consideration.
A \textbf{unitary trivialization}\index{trivialization!unitary --}
of a Hermitian vector bundle $E$ is a smooth map
\begin{equation}\label{eq:unitary-triv}
     \Phi:N\times\R^{2n}\to E,\quad
     (x,\xi)\mapsto\Phi(x,\xi)=:\Phi(x)\xi,
\end{equation}
which maps fibers linearly isomorphic to fibers,
that is $\Phi^{-1}$ is a vector bundle isomorphism to the trivial bundle,
and simultaneously identifies the compatible triple
$\omega,J,g_J$ on $E$
with the standard compatible triple
$\omega_0,J_0,\inner{\cdot}{\cdot}_0$ on $\R^{2n}$.

\begin{proposition}\label{prop:unitary-triv}
A Hermitian vector bundle $E\to\Sigma$ over a compact
Riemann surface $\Sigma$ with \emph{non-empty} boundary
$\p\Sigma$ admits a unitary trivialization.
\end{proposition}

The idea is to prove in a first step that for any
path $\gamma:[0,1]\to\Sigma$ the pull-back
bundle $\gamma^*T\Sigma\to[0,1]$ can be unitarily
trivialized even if one fixes in advance unitary
isomorphisms $\Phi_0:\R^{2n}\to E_{\gamma(0)}$ and
$\Phi_1:\R^{2n}\to E_{\gamma(1)}$ over the two endpoints
of $\gamma$. To see this construct unitary frames
over small subintervalls of $[0,1]$ starting with a unitary
basis of $E_{\gamma(t)}$ at some $t$, extend to a small
intervall via parallel transport, say with respect to some Riemannian
connection on $E$, and then exposed to the Gram-Schmidt process
over $\C$. The coupling of the resulting unitary trivializations over
the subintervals is based on the fact that the Lie group
$\U(n)$ is connected. In the second step one uses a parametrized
version of step one to deal with the case that $\Sigma$ is
diffeomorphic to the unit disk $\D\subset\R^2$.\,\footnote{
  Trivialize along rays starting at the origin: Extend a chosen
  frame sitting at the origin simultaneously along all rays, say by
  parallel transport, along an interval $[0,\eps]$. Now apply
  Gram-Schmidt to the family of frames
  and repeat the process on $[\eps/2,3\eps/2]$, and so on.
  }
Step three is to prove the general case by an induction
that starts at step two and whose induction step
is again by a parametrized version of step one, this
time for the disk with two open disks removed from its interior
(called a \textbf{\Index{pair of pants}}).

\subsection{Compatible almost complex structures}
\label{sec:comp-alm-cx-str}
Given a symplectic manifold $(M,\omega)$,
consider an endomorphism $J$ of $TM$
with $J^2=-\1$. Such $J$ is called an
\textbf{\Index{almost complex structure}} on $M$.\footnote{
  Such $J$ is called a \textbf{\Index{complex structure}}
  or an \textbf{\Index{integrable complex structure}} on $M$,
  \index{complex structure!integrable}
  if it arises from an atlas of $M$ consisting of complex
  differentiable coordinate charts to $(\C^n,i)$.
  }
If, in addition, the expression
$$
     g_J(\cdot,\cdot):=\omega(\cdot,J\cdot)
$$
defines a Riemannian metric on $M$,
then $J$ is called an \textbf{$\mbf{\omega}$-compatible almost
complex structure} on $M$. 
\index{almost complex structure!$\omega$-compatible}
The space $\Jj(M,\omega)$ of all such $J$
is non-empty and contractible by~\cite[Prop.~2.63]{mcduff:1998a}.

\begin{exercise}\label{exc:comp-alm-cx-strs}
Pick $J\in\Jj(M,\omega)$ and let $\nabla$
be the Levi-Civita connection associated to $g_J$.
\index{$\nabla=\nabla^g$ Levi-Civita connection}
Suppose $\xi$ is a smooth vector field on $M$,
show (i) and (ii):\footnote{
  Part~(iii) is less trivial;
  see~\cite[Le.\,C.7.1]{mcduff:2004a}.
  }
\begin{enumerate}
\item[(i)]
  $J$ preserves $g_J$ and $(\Nabla{\xi} J)J+J(\Nabla{\xi} J)=0$;
\item[(ii)]
  $J$ and $(\Nabla{\xi} J)$ are anti-symmetric with respect to $g_J$;
\item[(iii)]
  $J(\Nabla{J\xi} J)=\Nabla{\xi} J$.
\end{enumerate}
\end{exercise}

\subsection{First Chern class}
Up to isomorphism, symplectic\footnote{
  equivalently, complex vector bundles, by Proposition~\ref{prop:cx=symp-VB}  
  }
vector bundles $E$ over manifolds $N$ are classified
by a family $c_k(E)\in\Ho^{2k}(N)$ of integral cohomology classes of $N$
called \textbf{Chern classes}. If $N=\Sigma$ is a closed orientable
Riemannian surface, the first Chern class is uniquely
determined by the \textbf{first Chern number} which is the integer
obtained by evaluating the first Chern class on the fundamental
cycle $\Sigma$. Thus, slightly abusing notation, in case
$E\to\Sigma$ we shall denote the first Chern number
by $c_1(E)\in\Z$.
We cite again from~\cite{mcduff:1998a}.

\begin{theorem}\label{thm:c_1}
There exists a unique functor $c_1$, called the
\textbf{\Index{first Chern number}}, which assigns an integer
$c_1(E)\in\Z$ to every symplectic vector bundle $E$ over a
closed oriented Riemann surface $\Sigma$
and satisfies the following axioms.
\begin{enumerate}
\item[{\rm \texttt{(naturality)}}]
  Two symplectic vector bundles $E$ and $E^\prime$ over $\Sigma$
  are isomorphic iff they have the same rank and the same Chern number.
\item[{\rm \texttt{(functoriality)}}]
  For any smooth map $\varphi:\Sigma^\prime\to\Sigma$ of
  oriented Riemann surfaces and any symplectic vector bundle
  $E\so\Sigma$ it holds
  $
     c_1(\varphi^*E)=\deg(\varphi)\cdot c_1(E)
  $.
\item[{\rm \texttt{(additivity)}}]
  For any two symplectic vector bundles $E_1\to\Sigma$ and
  $E_2\to\Sigma$
  $$
     c_1(E_1\oplus E_2)=c_1(E_1\otimes E_2)=c_1(E_1)+c_1(E_2).
  $$
\item[{\rm \texttt{(normalization)}}]
  The Chern number of $\Sigma$ is
  $
     c_1(\Sigma):=c_1(T\Sigma)=2-2g
  $
  where $g$ is the genus.
\end{enumerate}
\end{theorem}

The proof is constructive, based on the Maslov index $\mu$ for
symplectic loops: Pick a splitting $\Sigma=\Sigma_1\cup_C\Sigma_2$
such that $\p\Sigma_1=C=-\p\Sigma_2$ as \emph{oriented} manifolds.
So the union $C=S^1 \mathop{\dot{\cup}}\dots\mathop{\dot{\cup}} S^1$
of, say $\ell$, embedded 1-spheres is oriented as the boundary of
$\Sigma_1$, say by the outward-normal-first convention.
Given a symplectic vector bundle $E$ over $\Sigma$,
pick unitary\footnote{
  a \textbf{symplectic trivialization} (just required to
  identify $\omega_0$ with $\omega$) is already
  fine\index{trivialization!symplectic}
  }
trivializations
\begin{equation}\label{eq:unitary-trivs}
     \Sigma_i\times\R^{2n}\to E_i,\quad
     (x,\xi)\mapsto\Phi_i(x)\xi,\quad i=1,2
\end{equation}
and consider the \textbf{\Index{overlap map}} $\Psi:C\to\Sp(2n)$
defined by $x\mapsto\Phi_1(x)^{-1}\Phi_2(x)$.

\begin{exercise}[First Chern number]\label{exc:c_1-Sigma}
Prove uniqueness in Theorem~\ref{thm:c_1}.
Show that the first Chern number of $E\to\Sigma$
is the degree of the composition
\begin{equation*}
\begin{split}
     c_1(E)
     =\deg\Bigl(C\stackrel{\Psi}{\longrightarrow}\Sp(2n)
     \xrightarrow[{(\ref{eq:map-rho})}]{\rho}
     \SS^1\Bigr)
     =\sum_{j=1}^\ell \mu(\gamma_j)
\end{split}
\end{equation*}
by verifying for $\deg(\rho\circ\Psi)$ the four axioms for the first
Chern number.
\\
(The second identity for the Maslov index $\mu$ is obvious:
Just pick an orientation preserving parametrization
$\gamma_j:\SS^1\to S^1$ for each connected component of $C$.)
\newline
[Hint: Show, or even just assume, first that $\deg(\rho\circ\Psi)$
is independent of the choice of, 
firstly, trivialization and, secondly, splitting. 
Use these two facts, whose proofs rely heavily on
Lemma~\ref{le:extend-degree} below, to verify the four axioms.]
\end{exercise}

\begin{lemma}\label{le:extend-degree}
Let $\Sigma$ be a compact oriented Riemann surface with
\emph{non-empty} boundary.
A smooth map $\Psi:\p\Sigma\to\Sp(2n)$ extends to
$\Sigma$ iff $\deg\left(\rho\circ\Psi\right)=0$.
\end{lemma}

\begin{exercise}[Obstruction to triviality]\label{exc:c_1-Obstruction}
Use the axioms to show that the first Chern number $c_1(E)$ vanishes
iff the symplectic vector bundle is \textbf{trivial}, that is isomorphic
to the trivial Hermitian bundle $\Sigma\times(\R^{2n},\omega_0,J_0,
\langle\cdot,\cdot\rangle_0)$.\index{vector bundle!trivial --}
\end{exercise}

\begin{exercise}[First Chern class]\label{exc:c_1-general-base}
Suppose $E$ is a symplectic vector bundle over any manifold $N$.
Observe that the first Chern number assigns an integer $c_1(f^*E)$ to every
smooth map $f:\Sigma\to N$ defined on a given closed oriented
Riemannian surface. Use the axioms to show that this integer
depends only on the homology class of $f$ and so the
first Chern number generalizes to an integral cohomology class
$c_1(E)\in\Ho^2(N)$ called the \textbf{\Index{first Chern class}} of $E$.
\end{exercise}

The \textbf{first Chern class of a symplectic manifold},
denoted by $c_1(M,\omega)$ or just by $c_1(M)$, is the first
Chern class of the tangent bundle $E=TM$.

\begin{exercise}[Splitting Lemma\footnote{
  There is a far more general theory behind called \emph{splitting principle};
  see e.g.~\cite[\S 21]{bott:1982a}.
  }\label{exc:Splitting-Lemma}]
Every symplectic vector bundle $E$ over a closed oriented Riemannian
surface $\Sigma$ decomposes as a direct sum of rank-2
symplectic vector bundles.
\newline
 [Hint: View $E$ as complex vector bundle with $\C$-dual $E^*$,
so $c_1(E)=-c_1(E^*)=-c_1(\Lambda^n E^*)=c_1((\Lambda^n E^*)^*)$;
cf.~\cite[p.414]{Griffiths:1978a}. Remember \texttt{(naturality)}.]
\end{exercise}

\begin{exercise}[Lagrangian subbundle]\label{exc:Lag-sub-bdle}
Suppose $E\to\Sigma$ is a symplectic vector bundle over a
closed oriented Riemannian surface. If $E$ admits a Lagrangian
subbundle $L$, then the first Chern number $c_1(E)=0$ vanishes.
(Consequently the vector bundle $E$ is unitarily trivial by
Exercise~\ref{exc:c_1-Obstruction}.)
\newline
[Hint: The unitary trivializations~(\ref{eq:unitary-trivs})
identify Lagrangian subspaces. 
Modify them so that each Lagrangian in $L$
gets identified with the horizontal Lagrangian $\R^n\times 0$.
Then the overlap map $\Psi$ will be of the form~(\ref{eq:map-rho})
with $Y=0$, so the determinant is real
and the degree therefore zero.]
\end{exercise}

\section{Hamiltonian trajectories}\label{sec:Ham-flows}

In the following we shall distinguish the analytic point
of view (maps) from the topological, often geometrical, point of view
(subsets, often submanifolds). We use the following terminology to
indicate
\begin{equation*}
\begin{split}
   &\textsf{maps: paths, loops, solutions, trajectories, \dots}
     \\
   &\textsf{subsets: curves, flow lines, closed characteristics, \dots.}
\end{split}
\end{equation*}
Sometimes it is convenient to use one and the same term in both
worlds and employ two adjectives to indicate the analysis or the
geometry point of view: 
Since periodicity is a property of maps, whereas closedness (compact
and no boundary) is a property of submanifolds, our convention is that
\begin{equation*}
\begin{split}
   &\textsf{``periodic'' qualifies maps: periodic orbits, periodic geodesics,}
     \\
   &\textsf{``closed'' qualifies subsets: closed orbits, closed geodesics.}
\end{split}
\end{equation*}
Closed geodesics are immersed circles and so are non-point closed orbits.

\subsection{Paths, periods, and loops}\label{sec:loops}
Suppose $N$ is a manifold.


\begin{definition}[Paths and curves, closed, simple, constant]
\label{def:paths_and_curves-continuous}
A \textbf{\Index{path}} is a smooth \emph{map} of the form
$\gamma:\R\to N$, whereas a \textbf{\Index{finite path}}\index{path!finite --}
is a smooth map $\alpha:[a,b]\to N$ that is defined on a compact
interval where $a\le b$.
The images $\cc=\gamma(\R)$ and
$\cc=\alpha([a,b])$ are connected subsets of $N$,
called \textbf{\Index{curve}s} in $N$.
\newline
In case $a=b$ we call $\alpha_p:\{a\}\to N$, $a\mapsto p$,
alternatively $\gamma_p:\R^0\to N$, $0\mapsto p$, a
\textbf{\Index{point path}} and its image curve\index{path!point --}
$\cc_p=\{p\}$ a \textbf{\Index{point}}.\index{curve!point --}
Note: Point paths are automatically embeddings and points
embedded submanifolds.

A\index{path!simple --}\index{curve!simple --}\index{simple path}
path $\beta$, finite or not, is called \textbf{simple} if it is
injective along the interior of the domain or, equivalently,
if it does not admit self-intersections $\beta(t)=\beta(s)$
at times $t\not= s$ in the interior of the domain.
The image of a simple path is called a \textbf{\Index{simple curve}}.
At the other extreme are paths, finite or
not,\index{path!constant --}\index{constant path}
whose images consist of a single point only.
These are called \textbf{constant paths}.
Note: A constant path is simple iff it is a point path.

We say that a\index{path!finite!closes up with order $\ell$}
finite\index{finite path!closes up with order $\ell$}
path $\alpha:[a,b]\to N$ \textbf{closes up with order $\mbf{\ell}$}
if initial and end point $\alpha(a)=\alpha(b)$ coincide,
together with all derivatives up to order $\ell$. 
In case all derivatives close up ($\ell=\infty$) we speak of a
\emph{finite path that closes up smoothly}, if in addition $a\not= b$ we speak
of a \textbf{\Index{loop} of period $\mbf{b-a}$}.\index{curve!closed --}
(For $a=b$ we already assigned the name point path.
A loop can be constant though.)

There are corresponding notions in other categories,
e.g. the elements of $C^k(\R,N)$ are called $C^k$ paths,
there are e.g. $W^{1,2}$ loops and so on.
\end{definition}

\begin{definition}[Paths, periodic and non-periodic]\label{def:loops_and_periods}
Consider a path $\gamma:\R\to N$.
If there is a real $\tau\not= 0$ such that\index{path!$\tau$-periodic --}
$\gamma(\tau+\cdot)=\gamma(\cdot)$,
then $\gamma$ is called a \textbf{\Index{periodic path}} and $\tau$
\textbf{a period} of\index{path!period of --}
$\gamma$.\index{period!of path}\index{path!periodic --}
One also says that the path $\gamma$
is\index{$\tau$-periodic path}\index{periodic!$\tau$- --}
\textbf{$\mbf{\tau}$-periodic}.
If there is no such $\tau\not=0$, then $\gamma$
is called \textbf{\Index{non-periodic}}.\index{path!non-periodic --}
By definition $\tau=0$ is considered a period
of any path, called the\index{path!trivial period of --}
\textbf{\Index{trivial period}}.\index{period!trivial --}
Let $\Per(\gamma)$ be the \textbf{\Index{set of all periods}} of
$\gamma$, including the trivial
period\index{$\Per(\gamma)$ period set of $\gamma:\R\to N$}
$0$.\index{periods!set of --}
\end{definition}

Observe that $\Per(\gamma)=\{0\}$ iff $\gamma$ is a non-periodic path
and $\Per(\gamma)=\R$ iff $\gamma$ is a constant path.
Do not confuse \emph{finite path that closes up} with \emph{periodic path} --
the domains $[a,b]$ and $\R$ are different.
However, a finite path $\alpha:[0,b]\to N$ closing up with all derivatives
comes with an\index{$\alpha^\#$ associated periodic path}
\textbf{associated $\mbf{b}$-periodic path}\index{path!associated periodic --}
$$
     \alpha^\#:\R\to N,\quad
     t\mapsto\alpha(t\bmod b),\qquad
     t\bmod 0:=0.
$$
Note that if $\alpha$ is a point path, then $\alpha^\#\equiv\alpha(0):\R\to N$ is
a constant path.
Vice versa, a non-constant $\tau$-periodic path $\gamma:\R\to N$ is
the infinite concatenation of the closed finite paths
$\alpha_k=\gamma|:[k\tau,(k+1)\tau]\to N$, $k\in\Z$.

\begin{exercise}
Given a path $\gamma:\R\to N$, show that $\Per(\gamma)$ is a closed
subgroup of $(\R,+)$.
With the convention $\inf\emptyset=\infty$
define\index{loop!minimal period of --}
the \textbf{minimal} or\index{period!minimal --}
\textbf{\Index{prime period}}\index{minimal period}
$$
     \tau_\gamma:=\inf\{\tau\in\Per(\gamma)\mid\tau>0\}
     \in[0,\infty].
$$
Show that the \textbf{period group}
$\Per(\gamma)$\index{period group of a path}
of\index{path!period group of --}
a\index{$\Per(\gamma)$ period group of a path}
path comes in three
flavors, namely\index{period!prime --}\index{loop!prime period of --}
\begin{equation}\label{eq:orbits-types}
     \Per(\gamma)
     =
     \begin{cases}
        \R&\text{, if $\tau_\gamma=0$, constant path (periodic),}
        \\
        \tau_\gamma\Z&
        \text{, if $\tau_\gamma>0$, non-constant periodic path,}
        \\
        \{0\}&\text{, if $\tau_\gamma=\infty$, non-periodic path (non-constant),}
     \end{cases}
\end{equation}
[Hint: Consult~\cite[Prop.\,1.3.1]{Palais:2009a} if you get stuck.]
If the period groups of two paths that have the same image curve are
equal, will they in general become equal after suitable time shift?
\end{exercise}

\begin{definition}[Prime and divisor parts, loops]
\label{def:prime-part}
A \textbf{\Index{divisor part}} 
of\index{path!periodic --!divisor part of -- --}
a\index{periodic path!divisor part of --}
periodic path $\gamma:\R\to N$ is a 
finite path of the form
\begin{equation}\label{eq:divisor-part}
     \gamma_\tau:[0,\abs{\tau}]\to N,\quad 
     t\mapsto
     \begin{cases}
        \gamma(t)&\text{, $\tau\ge0$,}\\
        \hat\gamma(t):=\gamma(-t)&\text{, $\tau<0$,}
     \end{cases}
\end{equation}
one for each period $\tau\in\Per(\gamma)$.
For $\tau\not=0$ the map on the quotient\footnote{
  By $u_\tau$ we also denote ``freezing the variable $\tau$'', but
  application context should be different. 
  }
\begin{equation}\label{eq:conv-or-path}
     \gamma_\tau:\R/\tau\Z\to N,\quad
     [t]\mapsto
     \begin{cases}
        \gamma(t)&\text{, $\tau>0$,}\\
        \hat\gamma(t)=\gamma(-t)&\text{, $\tau<0$,}
     \end{cases}
\end{equation}
is\index{$\gamma_\tau$ sign of period $\tau$ determines direction}
the \textbf{loop associated to the non-zero period}
$\tau\in\Per(\gamma)=\tau_\gamma\Z$
\textbf{of the non-constant path}
$\gamma$. To the trivial period $\tau=0$ we associate the point path
\[
     \gamma_0:\R^0=\{0\}\INTO N,\quad 0\mapsto\gamma(0),
\]
which we do not call a loop.
A \textbf{\Index{loop}} is a map of the
form~(\ref{eq:conv-or-path}). 
%
%
In case the minimal period $\tau_\gamma$ is positive and finite
we denote the associated divisor part and loop by
\begin{equation}\label{eq:prime-part}
     \ppgamma:[0,\tau_\gamma]\to N,\quad
     \ppgamma:\R/\tau_\gamma\Z\to N,\qquad
     t\mapsto\gamma(t),
\end{equation}
called the \textbf{prime part}\index{prime part!of periodic path}
and\index{path!periodic --!prime part of -- --}
the \textbf{prime loop} of a (non-constant) periodic path.
It is useful to call the loop $\ppgamma$ also the \textbf{prime loop}
of any of the loops $\gamma_\tau$ associated to a positive
period of $\gamma$.
A \textbf{\Index{simple loop}}\index{loop!simple --}
is an injective prime loop $\ppgamma$, equivalently,
the finite path $\ppgamma:[0,\tau_\gamma]\to N$ must be simple.
Observe that
$
     \gamma_{k\tau_\gamma}:\R/k\tau_\gamma\Z\to N
$
is a $k$-fold cover of $\ppgamma$.
\end{definition}

\begin{exercise}
Find a path whose prime part is not simple. Show that a simple loop
which is an immersion is an \textbf{\Index{embedding}}
(an injective immersion that is a homeomorphism onto its image)
of the unit circle.
\end{exercise}

Do not confuse the \emph{prime period $\tau_\gamma$} of a path with
the\index{time $T>0$ of first return}
\textbf{time $\mbf{T>0}$ of first return},
often called \textbf{time of first continuous return} or
\textbf{of order zero} and denoted by\index{first return!time of --}
$\mbf{T_0}$, namely $\gamma(T)=\gamma(0)$ but
$\gamma(t)\not=\gamma(0)$ at earlier times $t\in(0,T)$.\footnote{
  Similarly define the time $T_\ell>0$ of
  \textbf{first return of order $\mbf{\ell}$}
  by\index{first return of order $\ell$}
  the condition $\gamma(T_\ell)=\gamma(0)$,
  together with all derivatives up to order $\ell$,
  and this is not the case at any earlier time $t\in(0,T_\ell)$.    
  }
To see the difference consider a\index{first return!time of --}
figure eight with $\gamma(0)$ being the crossing point. 
For trajectories of smooth \emph{autonomous} vector fields both
notions coincide, prime parts are automatically simple, and prime
loops are\index{loop!embedded --}\index{embedded!loop}
circle embeddings; cf. Exercise~\ref{exc:loop-type-geod}!
For a periodic \emph{immersion} $\gamma:\R\to N$, thus non-constant,
an associated loop $\gamma_\tau$ is simple iff it is an embedding iff
it is the prime loop.

\begin{remark}[Negative periods]
Given a non-constant periodic path $\gamma$,
note that for negative periods $\tau$
divisor parts~(\ref{eq:divisor-part})
and\index{period!negative}
associated\index{path!backward --}\index{loop!running backwards}
loops~(\ref{eq:conv-or-path}) \emph{run backwards}, more precisely, they
follow\index{$\hat\tau:=-\tau$}
the\index{$\hat\gamma:=\gamma(-\cdot)$ time reversed path}
\textbf{\Index{time reversed path}}\index{path!time reversed --}
$$
     \hat\gamma(t):=\gamma(-t),\quad t\in\R,\qquad
     \hat\tau:=-\tau.
$$
Certainly $\gamma$ is $\tau$-periodic iff $\hat\gamma$
is $\hat\tau$-periodic and $\Per(\hat\gamma)=\Per(\gamma)$.
\end{remark}

\begin{definition}[Concatenation of finite paths and loops]\mbox{ }

(i)~Consider\index{paths!finite --!consecutive -- --}
two \textbf{consecutive} finite paths, that is\index{consecutive finite paths}
$\alpha:[a,b]\to N$ and $\beta:[b,c]\to N$ such that $\alpha$ ends at
the point $\alpha(b)=\beta(b)$ at which $\beta$ begins, together
with all derivatives.\footnote{
  If the domains are $[a,b]$ and $[c,d]$ replace $\beta$ by
  $\tilde\beta:[b,b+d-c]\ni t\mapsto \beta(t-b+c)$.
  }
The\index{$\beta\#\alpha$ concatenation of finite paths}
\textbf{concatenation} of two consecutive finite paths
is the finite path\index{concatenation!of finite paths}
defined by following first $\alpha$ and then $\beta$, notation
$$
     \beta\#\alpha:[a,c]\to N.
$$
Finite paths $\alpha:[0,b]\to N$ closing up with all derivatives are
self-concatenable: For $k\in\Z$ consider the
\textbf{$\mbf{k}$-fold concatenation}
$\alpha\#\dots\#\alpha:[0,\abs{kb}]\to N$; it is traversed backwards
in case $k<0$, see~(\ref{eq:divisor-part}). In case $b> 0$ denote by
\begin{equation}\label{eq:k-conc-path}
     \alpha^{\#k}:\R/kb\Z\to N,\qquad
     k\in\Z^*
\end{equation}
the \textbf{associated $\mbf{kb}$-periodic loop}; mind
convention~(\ref{eq:conv-or-path}) if $k<0$. To $k=0$ associate the
point path $\alpha^{{\#}0}(0):=\alpha(0)$ with domain $\R^0=\{0\}$.

(ii)~Suppose $\gamma_\tau$ is a loop with period $\tau$. Use
in~(\ref{eq:k-conc-path}) the finite path $\alpha$ given by the
divisor part $\gamma_\tau:[0,\abs{\tau}]\to N$
from~(\ref{eq:divisor-part}) to get the loop
$$
     \gamma_{k\tau}=\gamma^{\#k}_\tau:
     \R/k\tau\Z 
     \to N,
     \qquad \tau\in\Per(\gamma)\setminus\{0\},\quad
     k\in\Z^*,
$$
of period $k\tau$. It is
a \textbf{$\mbf{k}$-fold cover} of the $\tau$-periodic
loop\index{loop!${k}$-fold cover of --}
$\gamma_\tau$.\index{$\gamma^{\#k}_\tau$ ${k}$-fold cover of loop}
In particular, the loop
$
     \gamma^{\#k}_{\tau_\gamma}=\ppgamma^{\#k}:
     \R/k\tau_\gamma\Z \to N
$
is a $k$-fold cover of the prime loop of $\gamma_\tau$.
\end{definition}

\begin{definition}[Time shift and uniform change of speed]
Certainly\index{$z_{(T)}:=z(T+\cdot)$, $u_{(\sigma)}$ time-shift}
if a\index{$z^T:=z(T\cdot)$ speed change} 
path $\gamma:\R\to N$ is $\tau$-periodic, then so is any
\textbf{\Index{time shift}ed} path
$$
     \gamma_{(T)}:=\gamma(T+\cdot), \qquad T\in\R,\qquad
     \Per(\gamma_{(T)})=\Per(\gamma)=\tau_\gamma\Z.
$$
The operation \textbf{uniform change of speed}
applied to a loop\index{loop!uniform change of speed}
$\gamma$, namely\index{change of speed}
\begin{equation}\label{eq:k-speed}
     \gamma^\mu:=\gamma(\mu\cdot),\quad \mu\in\R,\qquad
     \text{hence $\gamma^0\equiv\gamma(0)$,}
\end{equation}
produces the new prime period
$\tau_{\gamma^\mu}=\tfrac{1}{\mu}\cdot\tau_\gamma$,
if $\mu\not=0$, and $\tau_{\gamma^0}=0$.
\end{definition}

\begin{remark}[In/compatibilities]
The operation of $k$-fold self-concatenation $\gamma^{\#k}_\tau$ of
a $\tau$-periodic loop $\gamma_\tau$ is compatible with ODEs and also
preserves periods in the sense that $\tau$ is still a period after the operation.
Period preservation also holds true for uniform \emph{integer} speed
changes $\gamma^k_\tau$, but for $k\not=1$ these do in general not map
ODE solutions to solutions. Examples are integral trajectories of a
vector field ($1^{\rm st}$ order ODE).
However, there are important cases where solutions are mapped to
solutions, e.g. periodic geodesics ($2^{\rm nd}$ order ODE).
\end{remark}

\begin{exercise}[Loops and periods -- immersed case]
An \textbf{\Index{immersion}} is a smooth map whose differential
is injective at every point.
Starting from Definition~\ref{def:paths_and_curves-continuous}
redo all definitions and constructions replacing
\emph{path} by \emph{immersed path}
and investigate if and how things change.
\end{exercise}

\begin{exercise}[Loops and periods -- embedded case]
Consider only immersed paths $\gamma:\R\IMTO N$.
Then a loop is an \emph{embedding} iff it is simple, in which case
it is prime\index{loop!embedded --}\index{embedded!loop}
(but prime is not sufficient).
Investigate if and how the previous constructions
change on the space of embedded loops.
A crucial observation is that, given a non-constant periodic
trajectory of an \emph{autonomous} smooth vector field on $N$,
an associated loop is embedded iff it is prime.
\end{exercise}

\subsection{Hamiltonian flows}\label{sec:Ham-flows-NEW}
Throughout $(M,\omega)$ is a symplectic manifold
and, as usual, everything is smooth.

\subsubsection{Autonomous Hamiltonians $\mbf{F}$}
Given a function $F:M\to\R$, by non-degeneracy of $\omega$ the
identity of $1$-forms
\begin{equation}\label{eq:X_H}
     dF=-i_{X_F}\omega:=-\omega(X_F,\cdot)
\end{equation}
determines a vector field $X_F=X_F^\omega$ on $M$, called the
\textbf{\Index{Hamiltonian vector field}} associated to $H$
or the \textbf{\Index{symplectic gradient}} of
$H$.\index{$X_H$ Hamiltonian vector field}\index{$X_H=J\nabla H$ for $\omega$-compatbile $J$}
The function $F$ is called the
\textbf{\Index{Hamiltonian}} of the dynamical system $(M,X_F)$
and it is called \textbf{autonomous}
since\index{autonomous Hamiltonian}\index{Hamiltonian!autonomous}
it does not depend on time.
For $\omega$-compatible almost complex structures $J$ the
Hamiltonian vector field is given by
\begin{equation}\label{eq:X_H-omega-comp-J}
     X_F=J\nabla F
\end{equation}
where the \textbf{\Index{gradient $\nabla F$}} is taken with respect
to the induced Riemannian metric, that is $\nabla F$ is determined
by $dF=g_J(\nabla F,\cdot)$.
We denote the flow generated by the Hamiltonian vector field
of an \emph{autonomous} Hamiltonian by $\phi=\{\phi_t\}$,
alternatively by $\phi^F=\{\phi_t^F\}$, as opposed to the
greek letter $\psi=\{\psi_t\}$ used in case of non-autonomous
Hamiltonians which are usually denoted
by\index{$\psi^H$ flow of non-auton. Ham. $H$}
$H=H_t$.\index{$\phi^F$ flow of auton. Ham. $F$}

An \textbf{\Index{energy level}} is a pre-image set
$F^{-1}(c)\subset M$ where $F$ is an autonomous Hamiltonian.
It is called an \textbf{\Index{energy surface}}
if\index{$S=F^{-1}(c)$ energy surface}
$c$ is a regular value of $F$, notation 
$
     S=F^{-1}(c)
$.
Hence energy surfaces $S$ contain no singularities (zeroes) of
$X_F$,\index{flow!stationary point of --}
equivalently, no \textbf{\Index{stationary point}s}\footnote{
  This means that $\phi_t p=p$ for \emph{all} times or, in other words,
  for the members of the \emph{whole} family $\{\phi_t\}$. For distinction
  we use the term \textbf{\Index{fixed point}} in the context of an
  individual map: If $\phi_Tp=p$, then $p$ is called a
  fixed point (of the map $\phi_T$).
  }
of the flow, and by the regular value theorem $S$ is a smooth
codimension one submanifold of~$M$. Most importantly, the Hamiltonian
flow preserves its energy levels:
\begin{equation}\label{eq:energy-preservation}
     \frac{d}{dt} F(\phi_t^F p)=0,\qquad
     F:M\to\R,
\end{equation}
for every initial condition $p\in M$.

\begin{exercise}
Show that~(\ref{eq:X_H}) determines $X_F$ uniquely.
Prove that $X_F$ is tangent to energy surfaces and,
slightly more general, prove~(\ref{eq:energy-preservation}).
\end{exercise}

\begin{remark}[Periodic vs closed and autonomous vs time-dependent]
\label{rem:closed-vs-periodic}\mbox{ }

a)~Suppose the flow\index{flow!complete}\index{complete flow}
of $X_F$ is complete, that is $\phi=\{\phi_t\}_{t\in\R}$.
The solution $z(t):=\phi_t z_0$, $t\in\R$, of $\dot z=X_F(z)$
with\index{Hamiltonian!path}
$z(0)=z_0$ is called a \textbf{Hamiltonian path} or\index{path!Hamiltonian}
a \textbf{\Index{flow trajectory}} -- note the domain
$\R$.\index{Hamiltonian!trajectory}
A solution\index{trajectory}
is either a line immersion $z:\R\IMTO M$ (embedded $z(\R)\cong\R$
or with self-tangencies $z(\R)\cong\SS^1$, but not self-transverse)
or it is constant $z(\R)=\{\pt\}$.
The image $\cc$ of a flow trajectory $z:\R\to M$
is called a \textbf{\Index{flow line}} or an
\textbf{\Index{integral curve}} of the Hamiltonian vector field.
If the solution path forms a loop $z:\R/\pertau\Z\to M$
(possibly constant but $\tau\not=0$ by definition of loop)
we call it a\index{Hamiltonian!loop}\index{loop!Hamiltonian --}
\textbf{Hamiltonian loop} or a \textbf{\Index{periodic orbit}}
of period $\tau$ (note the circle domain)
and its image $\cc=z(\R)$ a\index{orbit!periodic --}
\textbf{\Index{closed orbit}}.\index{orbit!closed --}

b)~\emph{Autonomous} (time-independent) vector fields $X$ on a
manifold $N$: Their\index{autonomous}
closed orbits are closed submanifolds of dimension,
either one (an embedded circle), or zero (a point).
For a non-constant periodic trajectory $z$ of $X$ the prime period
$\tau_z$ coincides with the time $T_0$ of first \emph{continuous} return.

c)~\emph{Non-autonomous} vector fields $X_t$ on $N$:
Here any geometric property of the image of a non-constant
trajectory $\gamma:\R\to M$ is lost, in general.\footnote{
  Unless one looks at the corresponding trajectory $t\mapsto(t,\gamma(t))$ in
  $\R\times N$.
  }
This is due to the possibility that $X_t$ could be
zero for a time interval of positive length during which the trajectory
will rest at a point, say $p$. Switching on again a suitable vector
field one can leave $p$ in any desired direction. 
Compared to the autonomous case $X$,
not only the immersion property is lost, but there
can be arbitrary self-intersections.
So $\gamma(\R)$ is in general nothing but a subset of $N$.
The same argument shows that the notion
of time of first return, even with infinite order,
is meaningless, it would not imply periodicity
of a non-constant trajectory $\gamma:\R\to M$.
For these reasons, in case of a non-autonomous vector field,
we do not call the image of a periodic trajectory
a closed orbit, it will be called just an image.
However, periodic solutions $\gamma:\R/\tau\Z\to M$ may
exist and these will still be called \textbf{\Index{periodic orbit}s}.

d)~Concerning solutions of autonomous $2^{\rm nd}$ order ODE's
see Exercise~\ref{exc:loop-type-geod}.
\end{remark}

\begin{remark}\label{rem:period-1}
Given $F:M\to\R$, note that $z$ is a ${\color{cyan}\tau}$-periodic
trajectory of $X_F$ iff $z^\tau=z(\tau\cdot)$ is a $1$-periodic
trajectory of ${\color{cyan}\tau} X_F$.
More generally, that $z$ is a ${\color{cyan}\tau}$-periodic
trajectory of a ${\color{cyan}\tau}$-periodic vector field
$X_t$ is equivalent to $z^\tau=z(\tau\cdot)$ being a
$1$-periodic trajectory of the $1$-periodic vector field
$X_{t\color{cyan}\tau}$.
\end{remark}

\begin{remark}[Multiple cover problem -- variable period]
\label{rem:mult-cover-problem}
This doesn't refer to change of speed, but to path concatenation:
Given a $\tau$-periodic trajectory $z:\R\to M$ of $X_t$, consider that
same map instead of on $[0,\tau]$ on the larger domain
$[0,k\tau]$, $k\in\N$, to get a periodic \emph{trajectory} $k$ times
covering $z$ -- \emph{same speed} but $k$-fold time.
\end{remark}

\begin{proposition}[$C^1$ and $C^2$ small Hamiltonians,
{\cite[\S 6.1]{hofer:2011a}}]\label{prop:HZ-C2small}
Suppose $M$ is a closed symplectic manifold.
Sufficiently $C^1$ small Hamiltonians $H:\SS^1\times M\to\R$
do not admit \emph{non-contractible} \underline{$1$}-periodic orbits.
Sufficiently $C^2$ small \underline{autonomous} Hamiltonians $F:M\to\R$
do not admit \underline{$1$}-periodic orbits at all --
except the constant ones sitting at the critical points.
\end{proposition}

\begin{proof}[Idea of proof]
Pick an $\omega$-compatible almost complex structure
to conclude that the length of a periodic orbit $z$ of period one is
small if the Hamiltonian is $C^1$ small, autonomous or not.
Indeed
$$
     \mathrm{length}(z)
     =\int_0^1\Abs{\dot z(t)}dt
     =\int_0^1\Abs{\nabla H_t(z(t))}dt.
$$
But a short loop $z$ in a compact manifold is contractible
and its image is covered by a Darboux chart.
For autonomous $F:M\to\R$ the argument
on page 185 in~\cite{hofer:2011a}
shows that $\dot z=0$ whenever the Hessian
of $F$ is sufficiently small.
\end{proof}

\subsubsection{Non-autonomous Hamiltonians $\mbf{H}$}
A time dependent Hamiltonian $H:\R\times M\to\R$,
notation $H_t(x):=H(t,x)$, generates a time dependent Hamiltonian
vector field $X_t:=X_{H_t}$ by considering~(\ref{eq:X_H}) for each time $t$.
One obtains a family $\psi_t=\psi_t^H$ of symplectomorphisms\footnote{
  A \textbf{\Index{symplectomorphism}}
  is a diffeomorphism preserving the symplectic form:
  $\psi^*\omega=\omega$.
  }
on $M$, called the \textbf{\Index{Hamiltonian flow}} generated by $H$, via
\begin{equation}\label{eq:Ham-flow}
     \frac{d}{dt}\psi_{t,0}=X_t\circ\psi_{t,0},\qquad
     \psi_{0,0}=\id,\qquad \psi_t:=\psi_{t,0}.
\end{equation}
The family\footnote{
  If $X_t$ depends on time, it is wise to keep track of the initial
  time $t_0$. As indicated in~(\ref{eq:Ham-flow}) we 
  shall always use $t_0=0$. The notation
  $\psi_{t,0}$ helps to remember that $\psi_{t+s,0}$
  is in general not a composition of $\psi_{s,0}$ and
  $\psi_{t,0}$. To obtain the composition law
  $\psi_{t,s}\psi_{s,r}=\psi_{t,r}$ one would have to allow for
  variable initial times, not just $t_0=0$.
  For simplicity $\psi_{t,0}=:\psi_t$.
  }
$\psi=\{\psi_t\}$ is called a\index{flow!complete}
\textbf{\Index{complete flow}} if it exists for all $t\in\R$.
Important examples are autonomous Hamiltonians $F$
and periodic in time Hamiltonians $H_{t+1}\equiv H_t$,
both on closed manifolds.
A \textbf{\Index{Hamiltonian trajectory}}, is a path of the form
$z(t)=\psi_t p$ with $p\in M$. In case $z$ is a loop we call it a
\textbf{Hamiltonian loop}.\index{Hamiltonian!loop}
In either case $z$ satisfies the 
\textbf{Hamiltonian equation}\index{Hamiltonian!equation}
$$
     \dot z(t)=X_t(z(t)),\qquad z(0)=p.
$$
Hamiltonian flows, autonomous or not, 
preserve the symplectic form.
By definition\footnote{
  Alternatively, defining $\Ll_X$ axiomatically,
  our definition becomes Thm.~2.2.24 in~{\cite{abraham:1978a}}.
  }
of the \textbf{\Index{Lie derivative}} and
\textbf{\Index{Cartan's formula}} one gets
\begin{equation}\label{eq:Lie-derivative}
     \frac{d}{dt}\psi_t^*\omega
     =:\psi_t^*\left(\Ll_{X_t}\omega\right)
     =\psi_t^*\left(i_{X_t}d\omega+di_{X_t}\omega\right).
\end{equation}
This shows that the family of diffeomorphisms $\psi_t$
generated by the family of vector fields $X_t$
preserves $\omega$, that is $\psi_t^*\omega=\omega$,
if and only if the $1$-form $i_{X_t}\omega$ is closed.\footnote{
  Such vector fields are called \textbf{symplectic},
  generalizing the Hamiltonian ones.\index{symplectic vector field}
  }
This holds, for instance, if $X_t$ is Hamiltonian
($di_{X_t}\omega=ddH=0$).

\subsubsection{Periodic orbits and their loop types}

\begin{remark}[Loops, periodic orbits, closed characteristics]\label{rem:loops}
\mbox{ }

  \vspace{.1cm}
  \textit{Topology (Subsets).}
  A loop $\gamma:\R/\tau\Z\to N$ is a
  \emph{simple}, thus \emph{prime}, loop
  if it admits no self-intersections, in symbols
  $\gamma^{-1}(\gamma(t))=\{t\}$ $\forall t$.
  Given two non-constant loops $\gamma$ and $\tilde \gamma$,
  if $\tilde \gamma=\gamma(k\cdot)$ for some integer $k$,
  one says that $\gamma$ is \textbf{$\mbf{k}$-fold covered} by
  $\tilde\gamma$, or a \textbf{\Index{multiply covered loop}}
  in case $\abs{k}>1$, in 
  symbols\index{$\gamma^k=\tilde\gamma$: loop
     $\gamma$ $k$-fold covered by $\tilde\gamma$}
  $\gamma^k=\tilde \gamma$.
  Two loops $\gamma$ and $\tilde \gamma$ are called
  \textbf{\Index{geometrically distinct}} if their image sets are not
  equal. Otherwise, they are
  \textbf{\Index{geometrically equivalent}},
  in symbols
  \[
     \gamma\sim\tilde \gamma.
  \]
  Geometrically equivalent loops, although having the same image set,
  certainly can be very different as maps. For instance, subloops of a
  figure-eight can be traversed a different number of times or in a
  different order.\footnote{
    While we require loops to be smooth, they do not need to be
    immersions. To go smoothly around a corner, just slow down to
    speed zero and accelerate again afterwards.
    }

  \vspace{.1cm}
  \textit{Analysis (Maps).} 
  Whereas any loop in a manifold $N$ is a periodic orbit of some
  periodic vector field $X_t$,
  only the rather restricted class of \emph{embedded} loops
  arises as (prime) periodic orbits of autonomous vector fields;
  see Exercise~\ref{exc:loops=X_t-orbits}.
  By~(\ref{eq:orbits-types}) the prime loop of a non-constant periodic
  trajectory of an autonomous vector field $X$ is a circle embedding.
  Moreover, still in the autonomous case, two periodic orbits
  $z,\tilde z$ are geometrically distinct iff their images are disjoint
  and, in the non-constant case, geometrically equivalent iff one
  $k$-fold covers the other one.

  \vspace{.1cm}
  \textit{Geometry (Submanifolds).} 
  Suppose $X$ is an autonomous vector field on a manifold $N$.
  Let $\Ppall(X)$ be the set of loop trajectories $z:\R/\tau\Z\to N$,
  in other words, periodic orbits, \emph{all} periods $\tau\not= 0$,
  constant solutions not excluded.
  Let $\Ppall^*(X)$ be the subset of the non-constant ones.
  The\index{$\Ppall(X)$ periodic orbits, all periods $\tau\not=0$}
  sets\index{$\Ppall^*(X)$ non-constant periodic orbits, all periods $\tau\not=0$}
  of
  equivalence 
  classes
  \begin{equation}\label{eq:closed-char-X}
     \Cc(X):=\Ppall(X)/\sim
     ,\qquad
     \Cc^*(X):=\Ppall^*(X)/\sim
  \end{equation}
  correspond to the set of closed orbits of $X$, respectively
  the non-point ones.
  The latter are disjoint embedded circles tangent to $X$,
  disjoint to each other. Representatives
  $y$ and $z$ of the same element of $\Cc^*(X)$ are multiple covers of a
  common simple periodic orbit $x$. In other words,
  the elements of the set $\Cc^*(X)$ are in bijection with those
  embedded circles $\SS^1\cong P\INTO N$ whose tangent bundle
  $TP=\R\cdot X|_P$ is spanned by the vector field $X$
  along\index{$\Cc(X)$ integral circles of $X$}\index{integral circles!of a vector field}
  $P$.\index{$\Cc^*(X)$ integral circles and zeroes of $X$}
  Technically one says that such $P$ are \textbf{integral
  submanifolds} of the (in general, singular) distribution $\R X\to N$ of 
  lines (and possibly points) along $N$.\index{integral submanifolds of distribution}
  We call these $P$ the \textbf{integral circles} of the vector
  field $X$. The set $\Cc^*(X)$ corresponds to the integral circles of
  $X$, whereas $\Cc(X)$ includes, in addition, the 0-dimensional
  integral submanifolds, namely, the zeroes, also called \textbf{singularities},
  of the vector field $X$.\index{singularity!of vector field}

  In Chapters~\ref{sec:contact-geometry} and~\ref{sec:RF}
  we will deal with the following special case:
  The manifold $N$ is a closed regular level set $S:=F^{-1}(c)$
  of an autonomous Hamiltonian $F$ on a symplectic manifold
  $(M,\omega)$ and $X=X_F$ is the Hamiltonian vector field.
  Note that in this case there are no zeroes of $X_F$, hence
  no constant solutions, on $F^{-1}(c)$ by regularity of the
  value $c$. In fact, the vector field $X_F$ spans what is called the
  \textbf{\Index{characteristic line bundle}}
  $\Ll_S:=(\ker\omega|_S)\to S$ and this is true for $X_K$ whenever
  $S$ is a regular level set of a Hamiltonian $K$; see~(\ref{eq:char-line-bdle}).
  Thus
  \begin{equation}\label{eq:Cc-clos-char-3}
      \Cc(S)=\Cc(S,\omega):=\Cc(X_F|_S)=\Cc^*(X_F|_S)
  \end{equation}
  denotes the set of integral circles $P$ of the characteristic
  distribution $\Ll_s$ of real lines along $S$.
  In\index{$\Cc(S)$ closed characteristics of energy surface $S\subset(M,\omega)$}
  this context the elements $P$ of $\Cc(S)$ are called the
  \textbf{\Index{closed characteristics} of the regular level set $\mbf{S}$}.
\end{remark}

\begin{exercise}[Loops are generated by vector fields]
\label{exc:loops=X_t-orbits}\mbox{ }\newline
a)  A loop $z$ in $N$ is the trajectory of some periodic vector field
  $X_{t+\tau}=X_t$.\index{periodic!vector field}\index{vector field!periodic}
\newline
b)  An embedded loop $z$ in $N$ is a trajectory
  of some autonomous vector field~$X$.
\newline
[Hints: First case $N=\R^k$, graph of $z$ in $[0,1]\times \R^k$, cutoff functions.]
\end{exercise}

\begin{exercise}[Prime periodic geodesics are self-transverse, not self-tangent]
\label{exc:loop-type-geod}
Let $N$ be a Riemannian manifold
with Levi-Civita connection $\nabla$.
Suppose the path $\gamma:\R\to N$ satisfies the $2^{\rm nd}$ order ODE
$\Nabla{t}\dot\gamma=0$. The equation implies that the speed
$\abs{\dot\gamma(t)}$ of a solution is constant in time $t$. Hence all
non-constant solutions are immersions, called
\textbf{\Index{geodesic}s}.
If a geodesic admits a positive period $\tau$ it is called
a \textbf{\Index{periodic geodesic}}
often denoted by $\gamma:\R/\tau\Z\to N$
or $\gamma_\tau$ to indicate the period in question.
Its image is an immersed circle, called
a\index{geodesic!periodic --}\index{geodesic!closed --}
\textbf{\Index{closed geodesic}}.
There are precisely two options. Such $\gamma$ is
\begin{itemize}
\item[-]
  either \Index{self-transverse}, then
   $\gamma=\ppgamma:\R/\tau_\gamma\Z\to N$
   is\index{geodesics are self-transverse}
   the prime loop of~$\gamma$,
   called a \textbf{prime periodic geodesic},
   or\index{geodesic!prime periodic --}
\item[-]
  a $k$-fold cover, where $k\in\N_{\ge2}$, of an
  underlying prime periodic geodesic.
\end{itemize}
\textbf{Self-transverse} means that whenever two arcs of $\gamma$
meet in $N$ they intersect transversely. In this case there is just a finite
number of intersection points by compactness of the domain.

  a)~Show that the two options are characterized by the two possibilities
  whether the set $\Tt$ of times $t_0$ such that $\gamma(t_0)$ has
  more than one pre-image\footnote{
    In symbols $\abs{\gamma^{-1}(\gamma(t_0))}\ge 2$. Such
    $\gamma(t_0)$ is called a \textbf{multiple} or a
    \textbf{double} ($=2$) \textbf{point}.\index{double point}
    }
  under $\gamma$
  is a finite set or an infinite set (thus equal to the circle domain itself).

  b)~Why are there no non-self-tangent
  self-intersections of trajectories of autonomous vector
  fields, but for geodesics they can appear?
\end{exercise}

\subsection{Conley-Zehnder index of periodic orbits}
Given a symplectic manifold $(M,\omega)$, consider a $1$-periodic family of
Hamiltonians $H_{t+1}=H_t:M\to\R$ with Hamiltonian flow
$\psi_t=\psi_{t,0}$.\index{$\Pp(H)$ $1$-periodic Hamiltonian orbits}
Let\index{periodic orbits!set of $1$- --}
$\Pp(H)$ be the \textbf{\Index{set of $1$-periodic orbits}}.

\begin{exercise}\label{exc:Pp-Fix-FH}
Check that $\Pp(H)\to \Fix\, \psi_1$, $z\mapsto z(0)$, provides
a bijection between the set of $1$-periodic orbits and the set of fixed
points of the time-1-map corresponding to initial time zero.
[Hint: Recall that $\psi_1$ abbreviates~$\psi_{1,0}$.]
\end{exercise}

A $1$-periodic orbit $z$ is called 
\textbf{non-degenerate}\index{non-degenerate!periodic orbit} 
if $1$ is not an eigenvalue of the linearized time-$1$-map, that is
\begin{equation}\label{eq:def-non-deg}
     \det\left(d\psi_1(p)-\1\right)\not=0,\qquad
     p:=z(0).
\end{equation}

\begin{exercise}\label{exc:nondeg->isolated}
Show that condition~(\ref{eq:def-non-deg}) implies that
$p$ is an isolated fixed point of $\psi_1$. Vice versa, would
isolatedness imply~(\ref{eq:def-non-deg})?\newline
[Hint: Condition~(\ref{eq:def-non-deg}) means that
the graph of $\psi_1$ in the product manifold $M\times M$
is transverse at $p$ to the \textbf{\Index{diagonal}}
$\Delta:=\{(p,p)\mid p\in M\}$.]
\end{exercise}

\begin{exercise}[Finite set]\label{exc:finite-set}
If the manifold $M$ is closed
and all $1$-periodic orbits are non-degenerate,
then the set $\Fix\,\psi_1$, hence $\Pp(H)$, is a finite set.
\end{exercise}

In addition to non-degeneracy, suppose the loop trajectory
$z:\SS^1=\p\D\to M$ is contractible.
Fix~an extension of $z$, namely a smooth map $v:\D\to M$
that coincides with $z$ on $\p\D$.
Moreover, pick an auxiliary $\omega$-compatible almost complex
structure $J\in\Jj(M,\omega)$, so the Hermitian
vector bundle $(E,\omega,J,g_J)$ with $E=v^*TM\to\D$
admits a unitary trivialization $\Phi_v$ by
Proposition~\ref{prop:unitary-triv},
that is $\Phi_v$ identifies the compatible
triples $(\omega_0,J_0,\inner{\cdot}{\cdot}_0)$ and $(\omega,J,g_J)$
where $J_0$ rotates counter-clockwise and corresponds to $i$.
Restriction to the boundary $\SS^1$ provides a
unitary trivialization, say $\Phi_z$, of the pull-back bundle
$z^*TM\to\SS^1$.
These choices provide a symplectic path
$\Psi_{z,v}\in\SsPp^*(2n)$ defined~by
\begin{equation}\label{eq:CZ-well-def-Ham}
\begin{tikzcd} 
\R^{2n}
\arrow[rr, dashed, "\Psi_{z,v}(t)"]
  &&\R^{2n}
        \arrow[d, "\Phi_v(z(t))"]
\\
T_{z(0)}M
\arrow[u, "\Phi_v(z(0))^{-1}"]
\arrow[rr, "d\psi_t(z(0))"]
  &&T_{z(t)}M
\end{tikzcd} 
\end{equation}
The \textbf{standard} and the
\textbf{canonical Conley-Zehnder indices of the non-degenerate
$1$-periodic orbit} $z$ are defined and related by
\begin{equation}\label{eq:def-CZ-orbit}
\begin{split}
     \underline{\CZcan(z):=\CZcan(\Psi_{z,v})}
     =-\CZ(\Psi_{z,v})
     =:-\CZ(z).
\end{split}
\end{equation}
In general these indices depend on the
spanning disk $v$, unless $c_1(M)|_{\pi_2(M)}=0$.
Sometimes it is useful to denote $\CZcan(z)$
by $\CZcan(z;H)$ or even by $\CZcan(z;H,\omega)$.

\begin{exercise}\label{exc:CZ-well-def-Ham}
Show that $\Psi_{z,v}(t)\in\Sp(2n)$ and that $\Psi_{z,v}(1)\in\Sp^*$.
Show that $\CZcan(\Psi_{z,v})$ does not depend on the particular
unitary trivialization $\Phi_z$.
Show that $\CZcan(z)$ is independent of the choice of $v$
if $c_1(M)$ vanishes on $\pi_2(M)$.
\end{exercise}

\begin{exercise}[Critical points are periodic orbits]
Suppose $z_0$ is a non-degenerate critical point
of a time independent function $H:M\to\R$.
Then~the constant $1$-periodic orbit $t\mapsto z_0$
is non-degenerate and the canonical Conley-Zehnder index and the Morse
index of $z_0$, see Section~\ref{sec:Morse-homology}, are related by
\begin{equation}\label{eq:mu_H=ind_-H}
\begin{split}
     \CZcan\left(z_0\right)
   :&=\CZcan\left(\Psi_{z_0}:t\mapsto e^{{\color{magenta} -}t J_0 S t}\right)
     =\frac12\sign(S)\\
   &=n-\IND_H(z_0)
     =\IND_{-H}(z_0)-n
\end{split}
\end{equation}
whenever $\norm{S}<2\pi$ and where $\dim M=2n$.
To obtain the first displayed formula, pick an
$\omega$-compatible almost complex structure $J$,
the induced metric $g_J$,
and an orthonormal basis of eigenvectors of the
Hessian of $H$ at $z_0$ and denote the corresponding
Hessian matrix by $S$.\footnote{
  One has $X^\omega_H=J\nabla H$ for
  $(M,\omega,J,g_J)$, but $X^{\omega_0}_K={\color{magenta} -}J_0\nabla K$ for
  $(\R^{2n},\omega_0,J_0,\inner{\cdot}{\cdot}_0$.
  }
Apply the axiom $\rm\texttt{(signature)}_{\texttt{can}}$.
\end{exercise}

\section{Cotangent bundles}\label{sec:cotangent-bundles}
Cotangent bundles are the phase spaces in
the Hamiltonian formulation of classical mechanics.
Given the tremendous success of the theory in
physics, not to mention daily life, one wouldn't risk much
predicting that these bundles should have a distinct position in the
mathematical world as well. Indeed
$$
     \textbf{Cotangent bundles $\mbf{\pi:T^*N\to N}$ over a manifold}
$$
\begin{enumerate}
\item[(S1)]
  admit\index{symplectic form!on cotangent bundle}
  a \textbf{\Index{canonical symplectic form}}
  and a canonical $1$-form given by\footnote{
     Local coordinates $\varphi:U\to\R^n$ on $N$ induce the
     diffeomorphism $T^*\varphi:T^*U\to \varphi(U)\times\R^n$
     taking $z=(q,p)$ to $(x,y):=(\varphi(q),(d\varphi(q)^{-1})^{*}p)\in\R^{2n}$
     and identifying $\omegacan$ with $\omegacan(\R^{2n})$.
     }
  \begin{equation*}
  \begin{split}
     \omegacan=d\lambdacan="dp\wedge dq\,",\qquad
     \lambdacan(z):T_zT^*N&\to\R,
     \\
     \zeta&\mapsto z\circ d\pi(z)\zeta
  \end{split}
  \end{equation*}
  The \textbf{\Index{Liouville form}} $\lambdacan$ is characterized by
  the property that
  \begin{equation}\label{eq:tautological-1-form}
     \sigma^*\lambdacan=\sigma
  \end{equation}
  for every $1$-form $\sigma\in\Omega^1(N)$, so one
  calls $\lambdacan$ the \textbf{\Index{tautological 1-form}};
\item[(S2)]
  admit a canonical Lagrangian subbundle of the tangent bundle,
  namely\footnote{
     The previous coordinates identify each $V_z$ symplectically
     with the Lagrangian $0\times\R^n$.
  }
  $$
     V:=\ker d\pi\subset T(T^*N);
  $$
\item[(S3)]
  admit \emph{along the zero section} $\Oo_N=\o(N)$ a natural
  Lagrangian splitting\footnote{
     $T_{\Oo_N}T^*N=H\oplus V$ is a direct sum:
     Linearize the composition $\pi\circ\o=\id:N\to T^*N\to N$
     of injection and surjection to get $H\cap V=\{0\}$. So
     $H+V=T_{\Oo_N}T^*N$ since ranks add up.
     }
  \begin{equation*}
  \begin{split}
     T_{\Oo_N}T^*N
     =\overbrace{\im d\o}^{=:H}\oplus\overbrace{\ker d\pi}^{=:V}
   &\stackrel{\cong}{\longrightarrow}
     TN\oplus T^*N.
     \\
     (h,v)
   &\;\mapsto\;
     (w_h,\theta_v)
  \end{split}
  \end{equation*}
  The isomorphisms\footnote{
     It suffices to show either surjectivity or
     injectivity (equal dimension of domain/codomain). As $d\o$ is
     injective it is an isomorphism $d\o:TN\to\im d\o$ onto its image
     with inverse $w$.
     Assuming $\theta_v\equiv 0$ with $v\in V$ means $v\in H^\omega$.
     So $v\in V\cap H=\{0\}$, as $H=H^\omega$ is Lagrangian: The restriction
     $\o^*\omegacan=d\o^*\lambdacan=d\o$
     is zero as $\o\in\Omega^1(N)$ is the zero section.
     }
  $w:\im d\o\to TN$ and $\theta:\ker d\pi\to T^*M$
  are determined in terms of the inclusion
  $\o:N\hookrightarrow T^*N$, $q\mapsto(q,0)$, by
  $$
     d\o (w_h)=h,\qquad
     \theta_v(\cdot):=\omegacan(v,d\o \cdot),
  $$
  and $\o^*\omegacan=(w,\theta)^*\Omega_{\rm can}$ where
  $$
     \Omega_{\rm can}(w_h\oplus\theta_v,w_{h^\prime}\oplus\theta_{v^\prime})
     :=\theta_v(w_{h^\prime})-\theta_{v^\prime}(w_h);
  $$
\item[(S4)]
  have trivial first Chern class in the sense that
  $c_1(T_QT^*N)=0$ for every {\color{cyan} oriented} closed\footnote{
    The argument involves Poincar\'{e} duality.
    Closedness: \textit{Push the obstruction to infinity}.
    }
  submanifold $Q$. If $Q=N$, then $c_1(TQ\oplus T^*Q)=0$.\footnote{
    If $Q=N=\Sigma$ use Exercise~\ref{exc:Lag-sub-bdle}.
    In general, let $\tilde E^*$ be the $\C$-dual of the vector
    bundle $\C^n\hookrightarrow \tilde E=T_Q T^*N\to Q$. Then
    $
       c_1(\tilde E)=-c_1(\tilde E^*)=-c_1(\Lambda^n \tilde E^*)
       =0\in\Ho^2(Q)
     $
     by~\cite[p.414]{Griffiths:1978a} and as the restriction
     of $\omegacan^{\wedge n}$ to $Q$ is a non-vanishing section;
     cf.~\citesymptop[App.~B.1.7]{weber:1999a}.
     }
\end{enumerate}

\begin{figure}
  \centering
  \includegraphics
                             [height=4cm]
                             {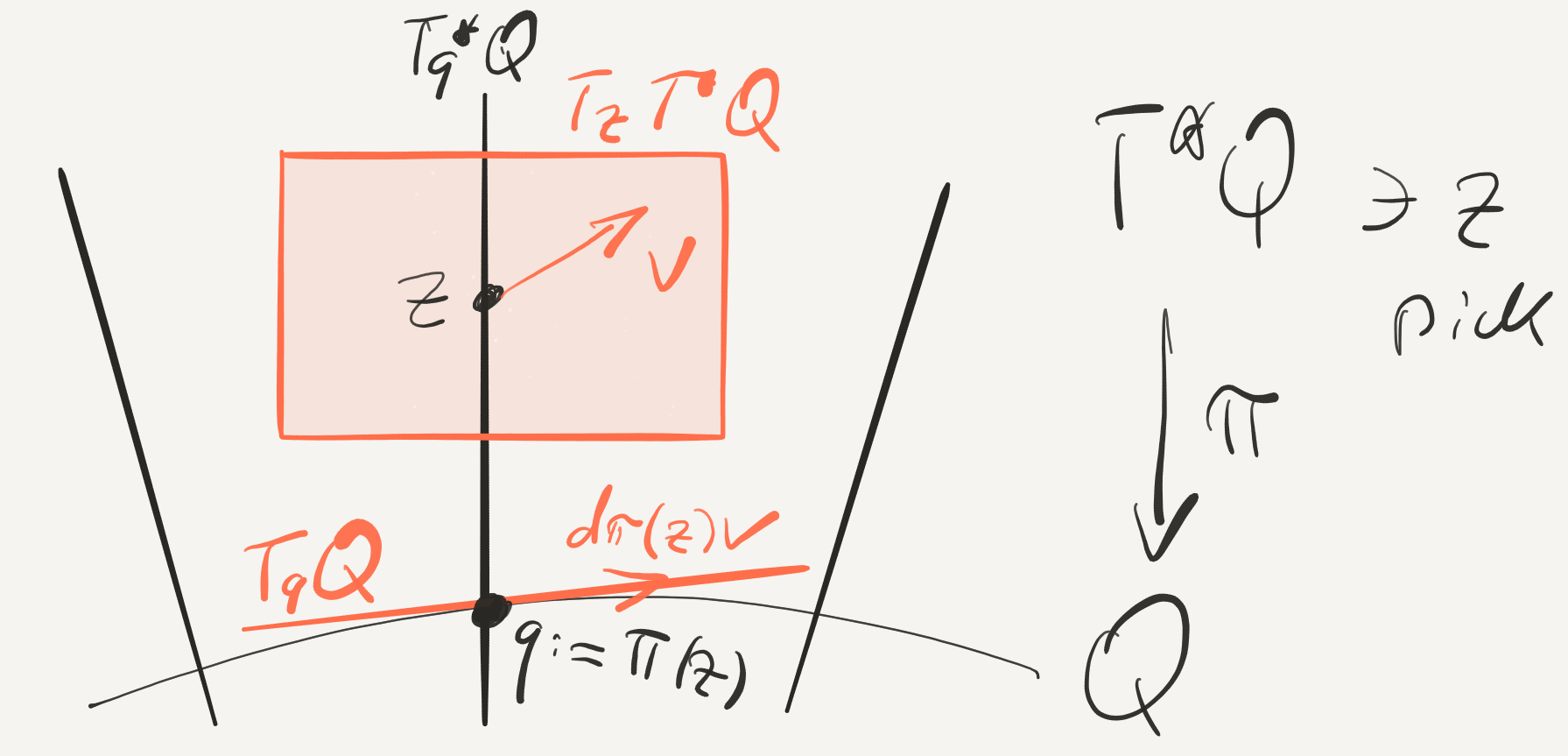}
  \caption{Liouville 1-form $\lambdacan\in\Omega^1(T^*Q)$ on cotangent
    bundle $\pi:T^*Q\to Q$}
  \label{fig:fig-Liouville-form}
\end{figure}

To get in all the other structures introduced in earlier sections
is rather simple:
$$
     \textbf{Pick a Riemannian metric $\mbf{g}$ on $\mbf{N}$ to obtain}
$$
\begin{itemize}
\item[(H1)]
  a \emph{global} Lagrangian splitting $T(T^*N)\cong\pi^*\left(TN\oplus T^*N\right)$
  defined pointwise by the isomorphism
  which takes the derivative of a curve $t\mapsto z(t)=(x(t),y(t))$ in $T^*N$
  to the pair of derivatives, namely
  \begin{equation*}\label{eq:g-Lag-split}
  \begin{split}
     T_{z(t)}T^*N \stackrel{\cong}{\longrightarrow} T_{x(t)}N\oplus T_{x(t)}^*N
     ,\quad
     \dot z(t)\mapsto \left(\dot x(t),\Nabla{t} y(t)\right).
  \end{split}
  \end{equation*}
  The isomorphism takes $\omegacan$ to $\Omega_{\rm can}$
  extending the one in~(S3);
\item[(H2)]
  a canonical Hermitian structure $(\omegacan,\Jbar_g,G_g)$ on
  $E=T(T^*N)$ defined pointwise for $z=(x,y)$
  on $T_xN\oplus T^*_x N$ by
  \begin{equation}\label{eq:J_g-G_g}
     \Jbar_g:=\begin{pmatrix} 0&g^{-1} \\ -g&0 \end{pmatrix},\qquad
     G_g:=\begin{pmatrix} g&0 \\ 0&g^* \end{pmatrix}.
  \end{equation}
  Concerning notation see Exercise~\ref{exc:W-plus-W*}.
  Throughout we denote by $g$ not only 
  the Riemannian metric on~$N$, but also the induced
  isomorphism
  $$
     g:TN\to T^*N,\quad v\mapsto g(v,\cdot).
  $$
\end{itemize}

\begin{exercise}
Prove properties~(S1--S4) and~(H1--H2).\footnote{
  For details of the proof of~(H1--H2)
  see e.g.~\citesymptop[App.~B.1.2--B.1.4]{weber:1999a}.
  }
\end{exercise}

\begin{exercise}\label{exc:lift-diffeo}
Any diffeomorphism $\psi:N\to N$ of a manifold
lifts to a symplectomorphism of the cotangent bundle
\begin{equation*}
\begin{tikzcd} 
\left( T^*N,\omegacan\right)
\arrow[rr, dashed, "\Psi:=T^*\psi", "\cong"']
\arrow[d, "\pi"']
  &&\left( T^*N,\omegacan\right)
        \arrow[d, "\pi"]
\\
N
\arrow[rr, "\psi", "\cong"']
  &&N
\end{tikzcd} 
\end{equation*}
defined by $\Psi(q,p)=\left(\psi(q),(d\psi(q)^{-1})^* p\right)$.
Prove that $\Psi^*\lambdacan=\lambdacan$.
What can one say if $(N,g)$ is a Riemannian manifold
and $\psi$ is an \textbf{\Index{isometry}} ($\psi^*g=g$)?
\end{exercise}

\begin{exercise}\label{exc:X-gen-diffeo}
Suppose the vector field $Y:N\to TN$ generates a
$1$-parameter group of diffeomorphisms $\psi_t:N\to N$.
Consider the corresponding group of symplectomorphisms
$\Psi_t$ of $(T^*N,\omegacan)$ and denote by
$X:T^*N\to TT^*N$ the generating vector field, that is 
$$
     X(q,p)
     =\left.\tfrac{d}{dt}\right|_{t=0}\Psi_t (q,p)
     =\left( Y(q),\left.\tfrac{d}{dt}\right|_{t=0} (d\psi_t(q)^{-1})^* p\right).
$$
Show that $X$ is the Hamiltonian vector
field of the function $H(q,p):=p\circ Y(q)$.
\newline
[Hint: Observe that
$0=\Ll_X\lambdacan=\omegacan (X,\cdot)+d\left(\lambdacan (X)\right)$
by the previous exercise and Cartan's formula and
that $\lambdacan (X(z)):=z\circ\left(d\pi(z) X(z)\right)$.]
\end{exercise}

\subsection{Electromagnetic flows -- twisted cotangent bundles}
\label{sec:electromagnetic}
Classical 
mechanics describes the motion of a particle of unit mass $m=1$
and unit charge $e=1$ located at time $t$ at position $\gamma(t)$ 
in a configuration space, a manifold $Q$.
The \textbf{\Index{electromagnetic system}}\footnote{
  \textbf{\Index{Maxwell's equations}} for a \textbf{\Index{magnetic field}}
  $\BBB$ and an \textbf{\Index{electric field}} $\EEE$ in $\R^3$ are
  $$
     \BBB=\rot\,\AAA=\nabla\times\AAA,\qquad
     \EEE=-\nabla V-\AAAdot.
  $$
  Here $\AAA$ is the \textbf{\Index{magnetic vector potential}}
  and $V$ the \textbf{\Index{electric scalar potential}}.
  The \textbf{\Index{Lorentz force} law} of a particle
  of mass $m$ and charge $e$ at position $\rrr(t)$
  is (cf.~\citesymptop[\S 1.1.2]{arnold:2006a})
  $$
     m\mbf{\ddot r}=e\left(\EEE+\mbf{\dot r}\times\BBB\right).
  $$
  }
is described by the following three structures on $Q$:
A Riemannian metric $g$ (providing kinetic energy),
a smooth one form $\theta$ (the \textbf{magnetic potential}),
and a smooth function $V$\index{potential!magnetic}\index{potential!electric} 
(the \textbf{electric potential}). 

\begin{exercise}[Reformulating Maxwell's equations]
\label{exc:Maxwell}
Let us reformulate (some of) Maxwell's equations on $\R^3$
in terms of quantities which are at home on manifolds -- differential
forms (see e.g. \cite{bott:1982a,warner:1983a}).
Consider $\R^3$ with its natural orientation and the euclidean metric $g_0$.
One-forms $\theta\in\Omega^1=\Omega^1(\R^3)$ are
in bijection with vector fields $\AAA\in\Xx=\Xx(\R^3)$,
just identify components:
$$
     \Omega^1\ni\theta=\theta_1 dx_1+\theta_2 dx_2+\theta_3 dx_3
     \;\stackrel{g_0^{-1}}{\mapsto}\;
     \theta_1\p_{x_1}+\theta_2\p_{x_2}+\theta_3\p_{x_3}=:\AAA\in\Xx.
$$
Let us indicate the isomorphism $\Omega^1\simeq\Xx$ by writing
$dx_j\mapsto\p_{x_j}$. Consider the isomorphism
$*:\Omega\to\Omega$ which relates the degree of the
differential forms by $\Omega^j\to\Omega^{3-j}$
and is determined on the elements of the natural bases by
\begin{equation*}
\begin{aligned}
     \Omega^0&\to\Omega^3,\quad& 1&\mapsto dx_1\wedge dx_2\wedge dx_3,\\
     \Omega^1&\to\Omega^2,\quad& dx_1&\mapsto dx_2\wedge dx_3,\\
     \Omega^2&\to\Omega^1,\quad& dx_1\wedge dx_2&\mapsto dx_3,\\
     \Omega^3&\to\Omega^0,\quad& dx_1\wedge dx_2\wedge dx_3&\mapsto 1,
\end{aligned}
\end{equation*}
and cyclic permutations. Note that $**=\1$.\footnote{
  More generally, on odd dimensional space $**=\1$,
  but on even dimensional space $**=\pm \1$
  on forms of even/odd degree.
  }
With $\nabla=(\p_{x_1},\p_{x_2},\p_{x_3})$ check\footnote{
  To $\BBB=B_1\p_{x_1}+B_2\p_{x_2}+B_3\p_{x_3}$ corresponds
  $\sigma=B_1dx_2\wedge dx_3+B_2dx_3\wedge dx_2+B_3dx_1\wedge dx_2$.
  }

\vspace{.5cm}
\begin{tabular}{llllll}
\toprule
  $\Xx$
  & $\stackrel{g_0}{\longrightarrow}$
  & $\Omega^1$
  & $\stackrel{*}{\longleftrightarrow}$
  & $\Omega^2$
  & Remark
\\
\midrule
  $\AAA$
  & 
  & $\theta$
  & 
  &
  & magnetic potential
\\
  $\nabla\times\AAA$
  & 
  & $*d\theta$
  & 
  & $d\theta$
  & exact magnetic field
\\
  $\BBB$
  & 
  & $*\sigma$
  & 
  & $\sigma$
  & magnetic field, requires
\\

  & 
  & 
  & 
  &
  & $\DIV\BBB=0$ ($ \Leftrightarrow d\sigma=0$)
\\
  $\mbf{V}$
  & 
  & $\nu$
  & 
  & 
  & velocity $\mbf{V}=\mbf{\dot r}$
\\
  $\mbf{V}\times \BBB$
  & 
  & $*(\nu\wedge*\sigma)$
  & 
  & $\nu\wedge*\sigma$
  &
\\

  & 
  & $=-i_{\mbf{V}}\sigma$
  & 
  & 
  & explicit calculation
\\
  $\mbf{Y}(\mbf{V})$
  & 
  & 
  & 
  & 
  & Lorentz force
\\
  & 
  & 
  & 
  & 
  & $\mbf{Y_r}\mbf{\dot r}:=\mbf{\dot r}\times\BBB_{\mbf{r}}$
\\
\bottomrule
\end{tabular}
\vspace{.5cm}

This shows that the Lorentz force
$\mbf{Y_r}\,\mbf{\dot r}:=\mbf{\dot r}\times\BBB_{\mbf{r}}$
experienced by a particle of unit mass and
unit charge is determined in terms of the closed
2-form $\sigma$ encoding the magnetic field by the identity 
\begin{equation}\label{eq:Lorentz-R3}
     i_{\mbf{Y}(\mbf{V})}g_0=-i_{\mbf{V}}\sigma.
\end{equation}
Concerning differential forms in electrodynamics
see e.g.~\citesymptop{Deschamps:1981a,Bott:1985a,Warnick:2014a}.
\end{exercise}

\begin{exercise}[Twisted symplectic structures
$\omega_\sigma=\omegacan+\pi^*\sigma$]
\label{excs:twisted-symp-str}
Suppose $\sigma$ is a closed 2-form on the closed manifold $Q$.
Denote by $\pi:T^*Q\to Q$ the bundle 
projection.\index{twisted symplectic structure}\index{symplectic structure!twisted --}
Show that $\omegacan+\pi^*\sigma$ is a symplectic form on $T^* Q$.
\newline
Now fix a Riemannian metric $g$ on $Q$ and check that the
identity
$$
     g_q(Y_q v,\cdot)={\color{magenta} -}\sigma_q(v,\cdot)
$$
pointwise at $q\in Q$ and $v\in T_qQ$
determines a fiber preserving anti-symmetric vector bundle map $Y:TQ\to TQ$,
i.e. fiberwise the map $Y_q:T_qQ\to T_qQ$ is linear and anti-symmetric.
The map $Y$ is the \textbf{\Index{Lorentz force}}
associated to the magnetic field
$\sigma$; see e.g.~\citesymptop{Contreras:2004a}.
We chose the {\color{magenta} minus} sign
in order to match the classical scenario~(\ref{eq:Lorentz-R3}) in
$\R^3$; cf.~\citesymptop{Barros:2005a}.
\end{exercise}

\begin{exercise}[Twisted geodesic flow]\label{exc:twisted-geod-flow}
A curve $\gamma$ solves the Euler-Lagrange
equations~(\ref{eq:EL-twisted}) associated to the Lagrangian
$L_\theta:TQ\to\R$ given by
$$
     L_\theta(q,v)=T(q,v)+\theta(q) v-V(q),\qquad
     T(q,v):=\tfrac12 \abs{v}^2:=\tfrac12 g_q(v,v),
$$
where $T$ is called \textbf{\Index{kinetic energy}},
if and only if (cf.~\citesymptop[\S 2]{Ginzburg:1996a})
the pair $(\gamma,g\dot\gamma)$
is an integral curve of the Hamiltonian vector field
$X_H^{\omega_{d\theta}}$ associated to the Hamiltonian $H$
and the twisted symplectic structure $\omega_{d\theta}$ on $T^*Q$ given by
$$
     H(q,p)=T(q,g^{-1}p)+V(q),\qquad
     \omega_{d\theta}:=\omegacan+\pi^*d\theta.
$$
The flow of $X_H^{\omega_{d\theta}}$ is called
\textbf{\Index{twisted geodesic flow}},
its flow lines\index{geodesic!twisted --}
\textbf{\Index{twisted geodesics}}.\index{geodesic flow!twisted --}
Note: In the Hamiltonian formulation there is no need that the magnetic
field $\sigma=d\theta$ is exact, any closed 2-form on $Q$ will do
by Exercise~\ref{excs:twisted-symp-str}.
\end{exercise}

Concerning existence of periodic electromagnetic
trajectories we recommend the (older) survey~\citesymptop{Ginzburg:1996a}
and the comments~\citesymptop{Ginzburg:2001b}.
For more recent results see e.g.~\citerefRF{Merry:2011a} and references therein.

\subsubsection*{Lagrangian and Hamiltonian formalism}
The \textbf{\Index{Lagrangian formulation}} of the dynamics is as follows.
Given two points $q_0,q_1\in Q$, the motion of the particle
is a curve $\gamma:[t_0,t_1]\to M$ with $\gamma(t_i)=q_i$
extremizing the \textbf{\Index{classical action}} functional
$$
     \Ss(\gamma)=\int_{t_0}^{t_1} L(\gamma(t),\dot\gamma(t))\, dt.
$$
Here\index{Lagrangian}
the\index{electromagnetic!Lagrangian}
\textbf{electromagnetic Lagrangian}\index{Lagrangian!electromagnetic}
$L:TQ\to\R$ of the system is given by
\begin{equation*}
\begin{split}
     L(q,v)=L_\theta(q,v):
   &=\frac12\, m\, g_q(v,v)+e\left(\theta_q v- V(q)\right)\\
   &=\frac12 \Abs{v}^2+\theta v- V.
\end{split}
\end{equation*}
The extremals (critical points) $\gamma$ are the solutions of the
\textbf{\Index{Euler-Lagrange equations}} which in local
coordinates can be written as
\begin{equation}\label{eq:EL-twisted}
     \frac{d}{dt} \p_v L(\gamma(t),\dot\gamma(t)) =\p_q L(\gamma(t),\dot\gamma(t)).
\end{equation}
For the physics behind we recommend~\citesymptop{Feynman:1964a}.
For the variational theory in the more general setting
of\index{Lagrangian!Tonelli --}
\textbf{\Index{Tonelli Lagrangians}}\footnote{
  A Lagrangian $L:TQ\to\R$ is \textbf{Tonelli} if it is
  fiberwise uniformly convex and superlinear.
  }
see e.g.~\citesymptop{mazzucchelli:2012a} or~\citerefRF{Abbondandolo:2013c}.

The \textbf{Hamiltonian description} of the system replaces the
Lagrangian $L_\theta$ by its Legendre transform $H_\theta:T^*Q \to\R$
called the\index{electromagnetic!Hamiltonian}
\textbf{electromagnetic Hamiltonian}\index{Hamiltonian!electromagnetic}
of\index{Hamiltonian}
the system and given by
\begin{equation*}
\begin{split}
     H_\theta(q,p)
   &=(\p_v L(q,v)) v-L(q,v)\\
   &=\frac12\Abs{v}^2+V\\
   &=\frac12\Abs{p-\theta}^2+V
\end{split}
\end{equation*}
where we substituted $v$ according to $p:=\p_v L_\theta(q,v)=g_qv+\theta_q$.
The dynamics of the particle on $T^*Q$ is then determined
by the Hamiltonian vector field $X_{H_\theta}^{\omegacan}$.

Alternatively, the dynamics of the same particle is described
by the Hamiltonian vector field $X_H^{\omega_{d\theta}}$ associated to
the standard \emph{non-magnetic} Hamiltonian
$$
     H(q,p)=\frac12\Abs{p}^2+V(q),\qquad
     \omega_{d\theta}:=\omegacan+\pi^*d\theta
     =d(\lambdacan+\pi^*\theta),
$$
and a \emph{magnetically} \textbf{\Index{twisted symplectic structure}}.
This\index{symplectic structure!twisted} alternative
description works not only for exact magnetic fields
$\sigma=d\theta$, but for any closed 2-form $\sigma$ on $Q$;
cf.~\citesymptop[Thm.\,2.1\,(ii)]{Ginzburg:1996a}.

\begin{exercise}
Hamiltonian dynamics of $(H_\theta,\omegacan)$ and
$(H,\omega_{d\theta})$ coincides.
\end{exercise}

For further details of the relation of symplectic geometry and classical
mechanics see e.g.~\cite{Arnold:1978a}
or~\citesymptop[Ch.~3]{Arnold:2001b}
or~\citesymptop[\S 1, \S 4]{arnold:2006a}.

\bibliographystylesymptop{alpha}
\cleardoublepage
\phantomsection
\addcontentsline{toc}{section}{References}

\begin{bibliographysymptop}{}
\end{bibliographysymptop}

\cleardoublepage
\phantomsection      
\chapter{Fixed period -- Floer homology}
\chaptermark{Fixed period -- Floer homology}
\label{sec:FH}

\subsubsection*{Towards Floer homology}

Consider a symplectic manifold $(M,\omega)$. Given an autonomous
Hamiltonian $F:M\to\R$, it is an interesting but rather
challenging problem to investigate the set ``of geometrically
distinct closed orbits'' of the Hamiltonian vector field
on a given regular energy level $S=F^{-1}(E)$.
A little thought reveals that this cannot be a set of Hamiltonian loops
$\R/\tau\Z\to M$, but rather it should be
a set of equivalence classes of such or, in geometric terms, of their
images -- embedded circles tangent to the Hamiltonian vector field
$X_F$. These circles form the set $\Cc(S,\omega)$ of
\emph{closed characteristics} of $X_F$ on $S$; cf.~(\ref{eq:Cc-clos-char-3}).

\vspace{.1cm}
\textit{Back to Hamiltonian loops (circle-parametrized closed orbits).}
Let us simplify the problem, firstly, by searching for loop trajectories
$\R/\tau\Z\to M$ of $X_F$ without taking any equivalence
classes at all, and secondly, by just focussing on the plain existence
problem. Let us break this further down into smaller pieces.

\vspace{.1cm}
\textit{Period one.} 
The problem reduces to detect Hamiltonian loops of period \emph{one}:
A $\tau$-periodic trajectory $z$ of $F$ on the level set $F^{-1}(E)$
corresponds to the $1$-periodic trajectory $z^\tau(t):=z(\tau t)$ of
$\tau F$ on the level set $(\tau F)^{-1}(\tau E)$ ($=F^{-1}(E)$);
indeed $\tau X_F=X_{\tau F}$.
So there is a little price to pay: One looses the freedom to fix
an energy value $E$.
 (The notion of  energy value/level is lost anyway as soon as one allows
 time-dependent Hamiltonians.)
So let us then investigate the set $\Pp(F)$ of $1$-periodic
loop trajectories $z:\R/\Z\to M$ of $X_F$. Two loops $z$ and $\tilde z$ are
geometrically distinct if their images are disjoint subsets of $M$,
otherwise they are equivalent $z\sim\tilde z$; cf.~Remark~\ref{rem:loops}.

There are two problems with autonomous Hamiltonians $F$
concerning \emph{non-constant} Hamiltonian loops $z$.
Suppose you found such $z$ of period~$1$.

\begin{enumerate}
\item[\rm (multiple covers)]
  How can you find out a Hamiltonian loop $z:\R/\Z\to M$
  does not \Index{multiply cover} another one $\tilde z$\,?
  In other words, is $1$ the prime period of $z:\R\to M$\,?
\item[\rm (degeneracy)]
  A non-constant Hamiltonian loop $z$ provides an $\SS^1$ family via
  time-shift\index{$z_{(T)}:=z(T+\cdot)$, $u_{(\sigma)}$ time-shift}
  $$
     \SS^1\ni\tau\mapsto z_{(\tau)}(\cdot):=z(\cdot+\tau).
  $$
\end{enumerate}
The multiple cover problem makes it difficult to decide
if a newly detected periodic orbit, say in the form of a
critical point of a functional on loop space, is geometrically
different from known ones.
The fact that the elements of $\Pp(F)$
are never isolated obstructs reformulating the problem
in terms of Morse theory, a powerful tool to analyze
sets of critical points in terms of topology.

\vspace{.1cm}
\textit{Non-autonomous $1$-periodic Hamiltonians and their
$1$-periodic orbits.} 
Both problems can be avoided by allowing time-dependent
Hamiltonians $H:\SS^1\times M\to\R$ and directing attention
to the set $\Pp(H)$ of $1$-periodic Hamiltonian loops
with respect\index{periodic orbits!set of $1$- --}
to\index{$\Pp(H)$ $1$-periodic Hamiltonian loops}
$X_{t+1}=X_t:=X_{H_t}$; see~(\ref{eq:Ham-flow}).
Is there a lower bound for the cardinality $\abs{\Pp(H)}$
uniformly in $H$? in fact, this problem reduces to study $\Pp_0(H)$,
the\index{$\Pp_0(H)$ contractible $1$-periodic ones}
\textbf{set of contractible 1-periodic Hamiltonian loops}:

\vspace{.1cm}
\textit{Contractible $1$-periodic Hamiltonian loops and closedness of $M$.}
For the special case of $C^2$ small autonomous Hamiltonians all
$1$-periodic orbits are constant by Proposition~\ref{prop:HZ-C2small},
so there are no non-contractible ones.
But a $C^2$ small Hamiltonian is obtained
by multiplying any given $H: M\to\R$
by a small constant $\eps>0$.
Really? Correct, at least, if $M$ is closed.
So from now on we assume closedness of $M$ --
conveniently guaranteeing completeness of flows --
and aim for a lower bound for the
cardinality of the set $\Pp_0(H)$, uniformly~in~$H$.

\begin{assumption}
Throughout Chapter~\ref{sec:FH} we assume, unless said
differently, that the symplectic manifold $(M,\omega)$ is closed and
symplectically aspherical, see~(\ref{eq:symp-asph-c_1}), and we study
the set $\Pp_0(H)$ for Hamiltonians~{$H:\SS^1\times M\to \R$}.
\end{assumption}

\textit{Arriving at Floer's ideas.}
It was well known before Floer that the contractible
$1$-periodic Hamiltonian loops
are precisely the critical points of a (possibly
multi-valued) functional $\Aa_H$,
but at the time it seemed that this functional was
``certainly not suitable for an existence proof.''~\citerefFH[(1.5)]{Moser:1976a}.
However, this changed with the success of Rabinowitz~\citerefCG{Rabinowitz:1978a}
in applying minimax methods to detect critical values of $\Aa_H$
-- and with Floer' s rather different idea~\citeintro{floer:1988a}
to overcome the obstruction presented by infinite
Morse index and coindex, namely, by looking at a \emph{relative}
index between critical points -- which is finite! Thereby
Floer discovered \textbf{\Index{relative Morse theory}} and
successfully reformulated the problem.\index{Morse theory!relative}
Floer's insights also included departing from looking at the
$L^2$ gradient equation as a formal ODE on the loop space,
but instead noticing that it represents a well posed PDE
  {\color{brown} Fredholm}
problem for maps from the cylinder $\R\times\SS^1$ into
the manifold $M$ itself (whenever suitably compactified by imposing
  {\color{brown} non-degenerate} 
Hamiltonian loops $z^\mp$ to sit at $\pm\infty$).
Assuming transversality, the dimension of the associated moduli space
is finite, as it is the Fredholm index which itself is given by the
spectral flow, the relative Morse index, along a flow cylinder.
Non-degeneracy amounts to $\Aa_H$ being a Morse functional,
this holds true for generic $H$, and in good cases (say $\Io_{c_1}=0$)
the relative index becomes the difference of an
absolute index associated to the non-degenerate Hamiltonian loops $z^\mp$ --
the Conley-Zehnder index of Section~\ref{sec:CZ}.
The fundamental results of Floer's construction are the
lower bounds~(\ref{eq:deg-AC}) and~(\ref{eq:non-deg-AC})
proving the {\Arnold} conjecture~\citeintro{floer:1989a} in many
cases; in general see~\citerefFH{Fukaya:1999a,Liu:1998a}.

\vspace{.1cm}
\textit{Back to autonomous Hamiltonians: Closed characteristics.}
Given an autonomous Hamiltonian $F:M\to\R$,
a natural approach to analyze the set $\Cc(X_F)$ of closed integral manifolds
of $X_F$, see~(\ref{eq:closed-char-X}), would certainly be 
-- in view of Floer's estimate $\abs{\Pp_0(H)}\ge \SB(M)$ --
to focus on $\Pp_0(F)$ in a first step:
Approximate $F$ in an appropriate topology
by non-degenerate, in general time-$1$-periodic, Hamiltonians $H_\nu\to F$.
By Floer's estimate every set $\Pp_0(H_\nu)$ is non-empty, so picking
one element $z_\nu$ for every $\nu$ provides a sequence
of Hamiltonian loops. Now one can try to extract a convergent
subsequence using the Arzel\`{a}-Ascoli Theorem~\ref{thm:AA}
and show that the limit loop, say $z$, satisfies the equation
$\dot z=X_F(z)$ of the limit Hamiltonian.
There are two problems.

I. Firstly, the Hamiltonian limit loop $z$, in fact already some or all of the
$z_\nu$, could be constant. This is excluded if the energy
hypersurface $S=F^{-1}(c)$ of $z$ is known to be regular,
that is if $c$ is a regular value of $F$.
Action filtered Floer homology may help sometimes;
see Section~\ref{sec:action-filtered-Floer homology}
and e.g.~\citerefFH{weber:2006a}.

II. The second problem are multiplicities. If you get two, or more, limit
solutions $z,\tilde z$ this way and suppose you even already know that they
are different and non-constant elements of $\Pp_0(F)$, say by having information
about the action values $\Aa_F(z)$ and $\Aa_F(\tilde z)$.
Even then, how would one decide whether
$z$ and $\tilde z$ are geometrically distinct or whether
one multiply covers the other one?
\newline
It helps looking at the particular Hamiltonian loops on a cotangent
bundle over a closed Riemannian manifold $Q$ which correspond to
geodesics $\gamma:\SS^1\to Q$ in the base manifold and remembering
Bott's analysis~\citerefFH{Bott:1956a} of how the Morse index changes
under iterations $\gamma^k(\cdot):=\gamma(k\cdot):\SS^1\to Q$.
For Hamiltonian loops the corresponding
index formulae have been pioneered by Long~\citerefFH{Long:2002a}.
For recent tremendous success of studying iterations,
in a slightly different direction though,\index{conjecture!Conley}
see Ginzburg's proof\citerefFH{Ginzburg:2010b} of the
\Index{Conley conjecture}.
See~\citerefFH{Ginzburg:2015b} for a recent survey about
existence of infinitely many simple periodic orbits.

\vspace{.1cm}
\textit{Cotangent bundles and loop spaces.}
If one gives up the compactness requirement
and looks at the class of symplectic manifolds
given by cotangent bundles $(T^*Q,\omegacan)$ over closed Riemannian
manifolds $(Q,g)$, say spin (thus orientable),
equipped with physical Hamiltonians $H$ of the form kinetic
plus potential energy and a natural almost complex structure $\Jbar_g$,
then Floer homology is totally different: It does not represent
singular homology of $T^*M$, but it is naturally isomorphic
to singular homology of the free loop space $\Ll Q$;
cf. Section~\ref{sec:FH-T^*M}. If $Q$ is not simply connected, then
there is one isomorphisms for each component of $\Ll Q$,
so here Floer homology even detects non-contractible periodic orbits.

\subsubsection*{Outline of Chapter~\ref{sec:FH}}

Consider a Hamiltonian $H:\SS^1\times M\to\R$.
Motivated by the Morse complex Floer's program,
see~\citeintro{floer:1989a},
is to use the symplectic action functional
\begin{equation*}\label{eq:action-A-in_preliminaries}
     \Aa_H:\Ll_0 M=C^\infty_{\rm contr}(\SS^1,M)\to\R,\quad
     z\mapsto\int_{\D} \bar{z}^*\omega -\int_0^1 H_t(z(t))\, dt,
\end{equation*}
cf.~(\ref{eq:action-A}), as a Morse function to construct a Morse type chain complex.

In Section~\ref{sec:toy-model}
we briefly recall the usual construction of the Morse complex
associated to a Morse function $f:Q\to\R$ on
a closed Riemannian manifold of dimension $n$
by using the critical points as generators,
the Morse index as grading, and counting downward
gradient flow lines to define a boundary operator.
We also recall the geometric realization
of the Morse cochain complex using the same
generators and grading, but counting upward flow lines.

Section~\ref{sec:symp-ac-fctl} is devoted to a detailed study
of the action functional $\Aa_H:\Ll_0 M\to\R$ starting with a list of serious
deficiency and explaining sign conventions.
Then we calculate the differential and, with the help of a family
$J_t$ of $\omega$-compatible almost complex structures,
also the $L^2$ gradient of $\Aa_H$.
This shows that the critical points are precisely given by the set
\begin{equation*}
     \Crit \Aa_H=\{z:\R/\Z\to M\mid
     \text{$\dot z=X_{H_t}(z)$, $z\sim\pt$}\}=:\Pp_0(H)
\end{equation*}
of $1$-periodic contractible Hamiltonian loops.
We calculate the Hessian operator
$A_z$ of $\Aa_H$ at a critical point $z$
to define non-degeneracy of critical points.
Next we insert an excursion to Baire's category theorem
trying to separate the surrounding and easily confusable notions of
``residual'' and ``second category'' subsets.
The purpose is to detail the informal notion of genericity
in theorems like the one asserting that $\Aa_H$ is Morse for generic
$H$ (``transversality on loops'').
\newline
Given Section~\ref{sec:symp-ac-fctl},
we define, for generic $H$, the Floer chain group $\CF_*(H)$
as the $\Z_2$ vector space generated by the finite set $\Pp_0(H)$ of
contractible $1$-periodic orbits and graded by the canonical
Conley-Zehnder index $\CZcan$ in~(\ref{eq:def-CZ-orbit}):
$$
     \CF_k(H)=\CF_k(M,\omega,H)
     :=\bigoplus_{z\in\Pp_0(H)\atop \CZcan(z)=k}
     \Z_2 z.
$$

Section~\ref{sec:DGF} introduces the substitute
for the non-existent downward gradient flow of $\Aa_H$ on $\Ll M$, namely,
solutions $u:\R\times\SS^1\to M$ to Floer's
elliptic PDE\footnote{
  Gromov's~\citeintro{gromov:1985a} $J$-holomorphic curve
  equation is $\p_su+J(u)\p_t u=0$, now add a lower order perturbation
  $\nabla H$.
  Here we have a $\Jbar$-holomorphic curve equation where
  $\Jbar:=-J$. Alternatively, time reflection $\tilde u(s,t):=u(-s,t)$
  relates the solutions $u$ of the displayed equation to solutions of
  the perturbed $J$-holomorphic curve equation
  $\p_s\tilde u+J_t(\tilde u)\p_t\tilde u+\nabla H_t(\tilde u)=0$.
  }
\begin{equation*}
\begin{split}
     0
    =\p_s u-J_t(u)\Bigl(\p_t u-X_{H_t}(u)\Bigr)
    =\p_s u-J_t(u)\p_t u-\nabla H_t(u),
\end{split}
\end{equation*}
called \textbf{\Index{Floer cylinders}} or
\textbf{\Index{Floer trajectories}}.
A Floer trajectory defines what we informally call a ``flow line'' or ``integral curve''
in the loop space, namely, the image set
$\{u(s,\cdot)\mid s\in\R\}\subset \Ll M$.
Any two trajectories producing the same flow line
differ by composition with time-shift $s\mapsto\sigma+ s$.
It is useful to switch view points using the correspondence
``cylinder in $M$'' $\leftrightarrow$ ``path in $\Ll M$'',
namely\index{$u_s:=u(s,\cdot)$ freeze variable}
$$
     \text{$u:\R\times\SS^1\to M$, $(s,t)\mapsto u(s,t)$}
     \quad\leftrightarrow\quad
     \text{$u:\R\to\Ll M$, $s\mapsto u_s(\cdot):=u(s,\cdot)$}.
$$
Back to Floer cylinders $u$.
Imposing as asymptotic boundary conditions
at $\pm\infty\times\SS^1$ Hamiltonian loops $z^\mp$
we call $u$ a \textbf{\Index{connecting trajectory}} from $z^-$ to $z^+$.
Let $\Mm(z^-,z^+)$ be the space of
all of them. At this point we insert Section~\ref{sec:Fredholm}
on relevant elements of Fredholm theory needed to analyze under which
conditions the spaces $\Mm(z^-,z^+)$ are manifolds and to calculate
their dimensions. Since $\Aa_H$ is Morse all asymptotic boundary
conditions are non-degenerate. This causes
the operators $D_u$ defined by linearizing Floer's
equation at any connecting trajectory $u$
to be Fredholm. For the manifold property of $\Mm(z^-,z^+)$
and the dimension formula $\CZcan(z^-)-\CZcan(z^+)$
one needs to make sure that $D_u$ is onto, often referred to as
``transversality on cylinders''.
But this can be achieved for generic $H$ again.
The necessary perturbation not only preserves the Morse property, but
even the set of critical points.
The machinery to deal with transversality issues,
be it non-degeneracy of the critical points of $\Aa_H$
or surjectivity of the Fredholm operators,
goes under the name Thom-Smale transversality theory
and will be discussed in detail guided by the
presentation~\citerefFH[App.\,B.4]{salamon:1999b} in finite dimensions.
In Section~\ref{sec:DGF}, also in~\ref{sec:FC},
we essentially follow~\citerefFH{salamon:1999a}.
\newline
Given Section~\ref{sec:DGF}, we define, for generic $(H,J)$,
Floer's boundary operator as the mod two count of flow
lines connecting Hamiltonian loops of index difference~$1$:
On basis elements $x\in\Pp_0(H)$ of canonical Conley-Zehnder index~$k$,~set
\begin{equation*}
\begin{split}
     \p x=\p^{\rm F}(M,\omega,H,J) x
     :=\sum_{y\in\Pp_0(H)\atop \CZcan(y)=k-1} \#_2(m_{xy})\, y
\end{split}
\end{equation*}
where $\#_2(m_{xy})$ is the number modulo two 
of flow lines connecting $x$ and $y$.

Section~\ref{sec:FC} is the heart of Chapter~\ref{sec:FH}.
First the property $\p^2=0$ is shown.
The resulting chain complex is denoted by
$
     \CF(H):=\left(\CF_*(H),\p\right)
$.
Its homology is a graded $\Z_2$ vector space denoted by
$
     \HF_*(H)=\HF_*(M,\omega,H;J)
$
and called \textbf{Floer homology}.
Next continuation isomorphisms are constructed
which naturally identify Floer homology under change
of Hamiltonian.
Then we discuss two methods of constructing
a natural isomorphisms to singular homology of
the manifold $M$ itself, namely by choosing for $H$
a $C^2$ small Morse function or, alternatively, by studying
``spiked disks''.
Section~\ref{sec:FC} concludes with a brief account
of action filtered Floer homology.

In Section~\ref{sec:FH-T^*M} we have
a glimpse at Floer homology for cotangent bundles, as
opposed to compact symplectic manifolds.

\subsubsection*{Preliminaries}
As a line in $\Ll M$ is a cylinder in $M$,
the formal $L^2$ gradient
equation for $\Aa_H$ on $\Ll M$ corresponds to
a PDE in $M$. Thus carrying out the program
of constructing a Morse complex
relies heavily on non-linear functional analysis.
So it is useful to impose in a first step conditions
in order to ``facilitate'' the analysis
and look for generalizations subsequently.
Closedness of $M$ we already mentioned.

\begin{itemize}
\item[(C1)]
  The symplectic manifold $(M,\omega)$ is closed.
\item[(C2)]
  Evaluating\index{$\mathrm{(C2)}$ $\omega$, $c_1(M)$ vanish on $\pi_2(M)$}
  $\omega$\index{assumption!(C2)}
  and $c_1(M)$ on $\pi_2(M)$, cf.~(\ref{eq:I-omega}),
  is identically zero
  \begin{equation}\label{eq:symp-asph-c_1}
      \Io_\omega=0=\Io_{c_1}.
  \end{equation}
\end{itemize}

Under these conditions we sketch
the construction of the Floer complex in
Sections~\ref{sec:symp-ac-fctl}--\ref{sec:FC};
for details, including history, see~\citerefFH{salamon:1999a},
or~\cite{mcduff:2004a}.

In Section~\ref{sec:FH-T^*M} the closedness condition~(C1)
gets dropped and we consider cotangent bundles $T^* Q$ equipped
with the canonical symplectic form $\omegacan=d\lambdacan$
over closed base manifolds $Q$.
Here both conditions in~(C2) are satisfied
automatically, the lack of compactness~(C1) will be
compensated by restricting to a class of Hamiltonians
that grow fiberwise sufficiently fast, such as
physical Hamiltonians of the form kinetic plus potential energy.

\begin{definition}\label{def:I-omega}
A symplectic manifold $(M,\omega)$ is called
\textbf{symplectically aspherical}, 
or\index{$\omega$-aspherical}
\textbf{$\mbf{\omega}$-aspherical} for short,
if\index{symplectically!aspherical}
the\index{$\Io_\omega:\pi_2(M)\to\R$, $ [v]\mapsto\int_{\SS^2} v^*\omega$}
homomorphism\index{$\Io_{c_1}:\pi_2(M)\to\Z$, $ [v]\mapsto c_1([v])$}
\begin{equation}\label{eq:I-omega}
     \Io_\omega:\pi_2(M)\to\R,\quad
     [v]\mapsto [\omega]([v]):=\int_{\SS^2} v^*\omega,
\end{equation}
vanishes for every class and each smooth representative $v:\SS^2\to M$.
We denote by $\Io_{c_1}:\pi_2(M)\to\Z$ the corresponding
evaluation homomorphism for the first Chern class
$c_1\in\Ho^2(M;\Z)$ of the (homotopic) complex vector
bundles $TM\to M$ associated to any family $J_t$ of $\omega$-compatible
almost complex structures.
\end{definition}

\begin{exercise}
Show that $\Io_\omega$ is well defined and a homomorphism of groups.
\end{exercise}

\begin{example}[Symplectically aspherical manifolds -- Condition~(C2)]
\label{ex:I-omega}
Examples of closed symplectically aspherical manifolds
are constructed in~\citerefFH{Gompf:1998a}.
It is shown that the conditions $\Io_\omega=0$
and $\Io_{c_1}=0$ are independent;
see also the survey~\citerefFH{Kedra:2008a}.
Since $M$ is closed $\Ho^2(M)$ is non-trivial, indeed
$[\omega]\not= 0$. Thus $\Io_\omega=0$
causes $\pi_1(M)\not=0$ via the Hurewicz homomorphism;
see~\cite[p.228]{hofer:2011a}.  
\begin{enumerate}
\item[(tori)]
  Since any torus $\T^\ell$ is \textbf{\Index{aspherical}},
  in fact $\pi_k(\T^\ell)=0$ for $k\ge 2$, any symplectic
  form $\omega$ on $\T^{2n}$ satisfies~(C2).
  Tori were treated in~\citeintro{Conley:1983a}.
\end{enumerate}
\end{example}

\section{Toy model}\label{sec:toy-model}

The Morse complex goes back to the work of Thom, Smale, and Milnor
in the 40s, 50s, and 60s, respectively, and was rediscovered in an
influential paper of Witten in 1982. It has been studied since
by many people. The standard reference is the 1993
monograph~\citerefFH{schwarz:1993a} by Schwarz.
For more on the history and references after 1993 see also our recent
lecture notes manuscript~\cite{weber:2015-MORSELEC-In_Preparation}
which covers in detail the dynamical systems
approach from~\citerefFH{weber:1993a}; cf.~\citerefFH{weber:2006b}.

\subsection{Morse homology}\label{sec:Morse-homology}
Given a closed manifold $Q$ of dimension $n$,
one can utilize gradient dynamical systems to recover the
\Index{integral singular homology} $\Ho_*(Q):=\Ho_*(Q;\Z)$.
Among all smooth functions on $Q$ there is an open and
dense subset consisting of \textbf{\Index{Morse function}s}
$f:Q\to\R$, that is all critical points $x$ are
\textbf{non-degenerate} in 
\index{critical point!non-degenerate} 
\index{non-degenerate!critical point!} 
the sense that all eigenvalues of the Hessian symmetric bilinear
form $\Hess_x f$ on $T_x Q$ are non-zero.
The number of negative eigenvalues, counted with
multiplicities, is called the \textbf{\Index{Morse index}} of $x$
denoted by $\IND_f(x)$. The \textbf{\Index{negative space}}
associated to the critical point $x$ is the subspace $E_x\subset T_xQ$
spanned by all eigenvectors associated to negative eigenvalues.
Non-degenerate critical points are isolated,
so by compactness of $Q$ they form a finite set $\Crit f$.

An \textbf{\Index{oriented critical point}} $o_x$,
\index{$o_x$, $\langle x\rangle$: oriented critical point}
also called an \textbf{\Index{orientation of a critical point}}
and alternatively denoted by $\langle x\rangle$, is a critical point $x$ 
together with a choice of orientation of its negative space $E_x$.
For each $k\in \Z$, let the \textbf{\Index{Morse chain group}}
$\CM_k(f)$ be the abelian group generated by the oriented critical
points $o_x$ of Morse index $k$ and subject to the relations
$o_x+\bar o_x=0$ where $\bar o_x$ is the opposite orientation of
$o_x$.\footnote{
  By convention the empty set $\emptyset$ generates the trivial group
  $\{0\}$; see Notation~\ref{not:notations_and_signs}.
  }
Let us denote by $[x]$ the equivalence class of an oriented
critical point under the relation $o_x+\bar
o_x=0$.\footnote{
  A choice $\langle\Crit_k f\rangle$ of an orientation
  for all critical points of index $k$ is a basis of $\CM_k(f)$.
  }

To define a boundary operator on $\CM_*(f)$
pick a Riemannian metric $g$ on $Q$ and
consider the corresponding \textbf{\Index{downward gradient flow}}
on $Q$, i.e. the 1-parameter group of diffeomorphisms
$\varphi=\{\varphi_t\}_{t\in\R}$ determined~by
\begin{equation}\label{eq:DGF}
      \frac{d}{dt}\varphi_t=-\nabla f\circ\varphi_t,\qquad
     \varphi_0=\id.
\end{equation}
By non-degeneracy of $x\in\Crit_k f$ the
\textbf{\Index{un/stable manifolds}} 
$W^{u/s}(x):=\{q\in\Q\mid\text{$\varphi_t q\to x$, as $t\to-/+\infty$}\}$
are embedded submanifolds of $Q$ of dimension/codimension $k$;
see e.g.~\citerefFH{Weber:2015c}.
Slightly perturbing the Morse function $f$ outside a small neighborhood
of its critical points leads to a function with the same critical points,
still Morse and still denoted by $f$, but whose flow satisfies in
addition the \textbf{\Index{Morse-Smale condition}}:\footnote{
  If the Morse-Smale condition~(\ref{eq:MS-Morse-homology}) holds
  true, we call $h:=(f,g)$ a \textbf{\Index{Morse-Smale pair}}.
  }
Namely, any intersection
\begin{equation}\label{eq:MS-Morse-homology}
     M_{xy}:=W^u(x)\pitchfork W^s(y),\qquad \dim M_{xy}=\IND_f(x)-\IND_f(y),
\end{equation}
of an unstable and a stable manifold is cut out transversely, hence a
manifold -- the \textbf{\Index{connecting manifold}} of $x$ and $y$.
The spaces of \textbf{\Index{connecting flow lines}}
$$
     m_{xy}:=M_{xy}\pitchfork f^{-1}(r),\qquad \dim m_{xy}=\IND_f(x)-\IND_f(y)-1,
$$
where $r\in(f(y),f(x))$ is any choice of a regular value of $f$,
are not only manifolds, but are what is called
\emph{compact up to broken flow lines}; cf. Figure~\ref{fig:fig-partner-pairs}.
Consequently in case of index difference one the $m_{xy}$ are finite
sets whose elements $u$ represent isolated flow lines running from
$x$ to $y$. Given such $u$ and an orientation $o_x$ of $E_x$, one can
define a \textbf{\Index{push-forward orientation}} $u_* o_x$
of $E_y$ that respects orientation reversal, that is
$u_* \bar o_x=\overline{u_* o_x}$.
Thus $u_*[x]:=[u_*o_x]$ is well defined
on the generators $[x]$ of the quotient group $\CM_k(f)$.
The \textbf{\Index{Morse boundary operator}} is then
defined on the generators by
$$
     \p_k=\p_k(f,g):\CM_k(f)\to\CM_{k-1}(f),\quad
     [x]\mapsto\sum_{y\in\Crit_{k-1} f}\;\sum_{u\in m_{xy}} u_*[x]
$$
and extended to the whole group by linearity.
That $\p^2=0$ boils down to the fact that
$\p^2[x]$ is a sum over all 1-fold broken flow lines $(u,v)$
where $u$ is a flow line from $x$ to some $y$ and $v$ is one
from \emph{that same} $y$ to some $z$ as indicated by
Figure~\ref{fig:fig-partner-pairs}.
As also indicated by the figure such broken flow lines correspond
precisely to the ends of a 1-dimensional manifold-with-boundary.
In other words, these broken flow lines appear in pairs and, moreover,
one partner provides the opposite coefficient
$v_*u_*[x]=-\tilde v_*\tilde u_*[x]$
in front of $[z]$ as the other one. So in sum each partner pair contributes zero,
but $\p^2[x]$ is precisely a sum of partner pair contributions.
For details of the facts above/below
see e.g.~\cite{weber:2015-MORSELEC-In_Preparation}.
\begin{figure}[h]
  \centering
  \includegraphics
                             {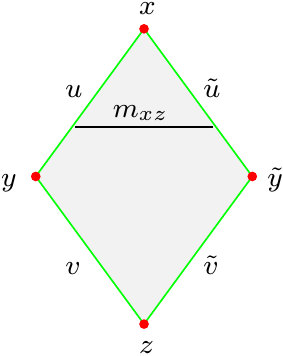}
  \caption{Partner pair property $(u,v)\sim(\tilde u,\tilde v)$
                leads to $\p^2=0$}
  \label{fig:fig-partner-pairs}
\end{figure}

The corresponding homology groups $\HM_k(Q;h)$
are actually independent of the Morse-Smale pair $h=(f,g)$ as one shows,
for instance,\footnote{
  Alternatively, use Po\'{z}niak cones;
  see~\cite{weber:2015-MORSELEC-In_Preparation}.
  }
by choosing a generic homotopy $f_t$ between 
two Morse functions $f^\alpha$ and $f^\beta$ and similarly $g_t$ among
the Riemannian metrics. Counting flow lines of the time-dependent
gradient equation --~just replace $h=(f,g)$ in~(\ref{eq:DGF}) by the
time-dependent pair $h_{\alpha\beta}=(f_t,g_t)$~-- 
provides a chain complex homomorphism $\psi^{\beta\alpha}_k(h_{\alpha\beta})$.
The induced maps on homology $\Psi^{\beta\alpha}_k:
\HM_k(f^\alpha,g^\alpha)\to\HM_k(f^\beta,g^\beta)$
are called \textbf{\Index{continuation maps}}.
They do not depend on the choice of homotopy $h_{\alpha\beta}$.
On the chain level the continuation maps have the trivial, but important,
property that the constant homotopy, denoted by $h_\alpha$,
induces the identity map, that is $\psi^{\alpha\alpha}_k(h_\alpha)=\1$.
Looking at homotopies of homotopies
not only shows that the $\Psi^{\beta\alpha}_k$
are independent of $h_{\alpha\beta}$, but also
provides the crucial relations
$$
     \Psi^{\gamma\beta}_k \Psi^{\beta\alpha}_k 
     =\Psi^{\gamma\alpha}_k,\qquad
     \Psi^{\alpha\alpha}_k=\1.
$$
A rather nice way to construct a natural isomorphism
\begin{equation}\label{eq:fund-Morse-hom}
     \Psi^h:\HM_k(Q;h)\stackrel{\cong}{\longrightarrow}\Ho_k(Q)
\end{equation}
to singular homology of $Q$ is via the
Abbondandolo-Majer
filtration~\citerefFH{abbondandolo:2006a}.\footnote{
  If $f(x)=k$ for $x\in\Crit_k f$, then
  $\{f<k+\frac12\}$ leads to such filtration; cf.~\cite[Thm.~7.4]{milnor:1965a}.
  }

\begin{remark}[$\Z_2$ coefficients]\label{rem:Z2-coeff-Morse}
The $k^{\rm th}$ Morse chain group with $\Z_2$ coefficients
is the $\Z_2$ vector space $\CM_k(f;\Z_2)$ whose canonical
basis $\Crit_k f$ are the critical points of Morse index $k$.
The Morse boundary operator is defined by
\begin{equation}\label{eq:bound-Morse-Z2}
     \p_k x:=\sum_{y\in\Crit_{k-1} f} \#_2(m_{xy}) \, y
\end{equation}
on the basis elements $x\in\Crit_k f$; here $\#_2$
denotes 'number of elements mod 2'.
\begin{equation*}
\begin{gathered}
  \textsf{The $\Z_2$ Morse boundary operator counts modulo two}
  \\
  \textsf{downward flow trajectories between critical points of index difference $1$.}
\end{gathered}
\end{equation*}
\end{remark}

\begin{exercise}[Closed orientable surfaces]
Calculate the $\Z_2$ Morse homology of your favorite
closed orientable surface.
[Hint: Embedd in $\R^3$; height function.]
\end{exercise}

\begin{exercise}[Real projective plane $\RP^2$]
Find a Morse function on $\RP^2$ with exactly
three critical points. Find the $\Z_2$
Morse complex and homology.
\newline
[Hint: Think of $\RP^2$ as unit disk $\D\subset\R^2$,
opposite boundary points identified.]
\end{exercise}

\subsection{Morse cohomology}\label{sec:Morse-cohom}
By
definition \textbf{\Index{cohomology}} arises from homology by
\emph{dualization}:
Any chain complex $\Crm=(\Crm_*,\p_*)$ comes naturally with a
cochain complex $\Crm^\#=(\Crm^*,\delta^*)$, the
\textbf{dual complex of $\mbf{\Crm}$}: It consists of the dual
spaces $\Crm^k:=\Crm_k^\#$ and transposed maps $\delta^k:=\p_{k+1}^\#$.
The cohomology $\Ho^*(\Crm^\#)$ of the cochain complex $\Crm^\#$
is called the \textbf{cohomology of $\mbf{\Crm}$}
and denoted by $\Ho^*(\Crm)$.

\subsubsection*{Morse cohomology}
For a Morse-Smale pair $h=(f,g)$ the
\textbf{\Index{Morse cochain groups}} are defined by
$$
     \CM^k(f):=\CM_k^\#(f):=\Hom(\CM_k(f),\Z)
$$
for any $k$ and the associated \textbf{\Index{Morse coboundary
operators}} $\delta^k=\delta^k(h)$ by
\begin{equation*}\label{eq:morse-coboundary}
\begin{tikzcd} [column sep=tiny]
\CM^{k+1}(f)
\arrow[r, equal]
  &\Bigl\{\CM_{k+1}(f)
    \arrow[r]
    \arrow[d, cyan, "\p_{k+1}"]
    &\Z\text{ linear}\Bigr\}
\\
\,\,\CM^{k}(f)\,\,
\arrow[r, equal]
\arrow[u,  "\delta^k:=", "(\p_{k+1})^\#"']
  &\Bigl\{\,\CM_{k}(f)\,
    \arrow[r, "\gamma"]
    &\Z\text{ linear}\Bigr\}.
\end{tikzcd}
\end{equation*}
The transposed map acts by $\delta^k\gamma=\gamma\circ\p_{k+1}$,
of course. The quotient space
$$
     \HM^k(Q;h):=\frac{\ker\delta^k}{\im\delta^{k-1}}
$$
is called the $\mbf{k^{\mathrm{th}}}$ \textbf{\Index{Morse cohomology}}
of $Q$ with $\Z$ coefficients.

From now on we restrict to $\Z_2$ coefficients for simplicity of
the presentation. Since $\Z_2$ is a field the Kronecker duality theorem
implies that the homomorphism induced on cohomology
$[\psi^{\beta\alpha}(h_{\alpha\beta})^\#]$ by the transpose 
is the transpose of the homology continuation isomorphism
$\Psi^{\beta\alpha}$. Consequently the transposes
$$
     (\Psi^{\beta\alpha})^\#
     =[\psi^{\beta\alpha}(h_{\alpha\beta})]^\#
     =[\psi^{\beta\alpha}(h_{\alpha\beta})^\#]:
     \HM^*(Q;h^\beta;\Z_2)\to\HM^*(Q;h^\alpha;\Z_2)
$$
are isomorphisms and satisfy the identities
\begin{equation}\label{eq:cont-maps-trans}
     (\Psi^{\beta\alpha})^\#(\Psi^{\gamma\beta})^\#=(\Psi^{\gamma\alpha})^\#
     ,\qquad
     (\Psi^{\alpha\alpha})^\#=\1.
\end{equation}

\subsubsection*{Geometric realization}
Restricting to $\Z_2$ coefficients
there are no orientations involved, so $\Bb_f:=\Crit f$ is a
\textbf{canonical basis} of $\CM_*(f;\Z_2)$.
\index{canonical basis!of $\CM_*(f;\Z_2)$}
By compactness of $Q$ the dimension
of $\CM_*(f;\Z_2)$, thus of its dual space $\CM^*(f;\Z_2)$, is finite.
Hence the dual set of $\Bb_f$, the set
$\Bb_f^\#:=\Crit^\# f:=\{\eta^x\mid x\in\Crit f\}$
that consists of the Dirac $\delta$-functionals\footnote{
  Dirac $\delta$-functionals are denoted by $\eta^x$ for distinction
  from the coboundary operator $\delta$.
  }
associated to the elements of $\Bb_f$,
is a basis of $\CM^*(f;\Z_2)$
called the \textbf{\Index{dual basis}} of $\Bb_f$.
By definition each functional is determined by its values
\begin{equation}\label{eq:Dirac-delta}
     \eta^{x}:\CM_*(f;\Z_2)\to\Z_2,\quad
     y\mapsto
     \begin{cases}
        1&\text{, $y=x$,}
        \\
        0&\text{, else,}
     \end{cases}
\end{equation}
on the basis elements $y\in\Crit f$. Since $\CM^*(f;\Z_2)$ is of
finite dimension any element $\omega$ can indeed be written
as a linear combination of the $\eta^x$'s, that is
\begin{equation}\label{eq:formal-sum}
     \omega=\sum_{x\in\Crit f}\omega_x\eta^x,\qquad
     \omega_x:=\omega(x)\in\Z_2.
\end{equation}
The dual basis, thus $\CM^*(f;\Z_2)$, inherits the Morse index grading
of $f$, i.e.
$
     \abs{\eta^x}:=\abs{x}:=\IND_f(x).
$
To geometrically identify the action of the coboundary operator
$\delta^k:=\p_{k+1}^\#$ on a cochain $\omega\in\CM^k(f;\Z_2)$
observe that
$$
     (\delta^k\omega)_x
     =(\delta^k\omega)(x)
     =\omega(\p_{k+1}x)
     =\sum_{y\in\Crit_k f}\#_2(m_{xy})\,\omega_y
$$
for every $x\in\Crit_{k+1} f$. Here we used
definition~(\ref{eq:bound-Morse-Z2}) of $\p_{k+1}$. Thus by~(\ref{eq:formal-sum})
\begin{equation*}
     \delta^k\omega
     =\sum_{x\in\Crit_{k+1} f}\left(\delta^k\omega\right)_x\eta^x
     =\sum_{x\in\Crit_{k+1} f}
     \Bigl(\sum_{y\in\Crit_k f}\#_2(m_{xy})\,\omega_y\Bigr)\eta^x
\end{equation*}
for every cochain $\omega\in\CM^k(f;\Z_2)$. In particular, we obtain that
\begin{equation}\label{eq:formal-sum-678}
     \delta^k\eta^y
     =\sum_{x\in\Crit_{k+1} f} \#_2(m_{xy})\,\eta^x
\end{equation}
for every basis element $\eta^y\in\Crit_k^\# f$.
But this means the
following.
\begin{equation*}
\begin{gathered}
  \textsf{The $\Z_2$ Morse {\color{red} co}boundary operator counts modulo two}
  \\
  \textsf{{\color{red} upward} flow trajectories between critical
               points of index difference $1$.}
\end{gathered}
\end{equation*}


\addtocounter{section}{+1}
\sectionmark{Symplectic action functional $\Aa_H$}
\addtocounter{section}{-1}
\section[Symplectic action functional $\Aa_H$ -- period one]
{Symplectic action $\Aa_H$ -- period one}
\sectionmark{Symplectic action functional $\Aa_H$}
\label{sec:symp-ac-fctl}

Suppose $(M,\omega)$ is a \emph{closed} symplectic manifold.
The through a Hamiltonian $H$
\textbf{perturbed \Index{symplectic action} functional}
on the space $\Ll_0 M$ of contractible smooth loops
$z:\SS^1\to M$ is defined by
\begin{equation}\label{eq:action-A}
     \Aa_H:\Ll_0 M\to\R,\quad
     z\mapsto\int_{\D} \bar{z}^*\omega -\int_0^1 H_t(z(t))\, dt,
\end{equation}
where $\bar{z}=v:\D\to M$ is a \textbf{\Index{spanning disk}},
i.e. a smooth extension of $z=v|_{\p\D}$.
Some remarks are in order. The symplectic action functional
\begin{itemize}
\item
  is \textbf{not well defined}, unless $M$ is $\omega$-aspherical
  ($dA_H$ makes sense though);\index{symplectic action!not well defined}
\item
  is \textbf{not bounded} below,\index{symplectic action!not bounded}
    neither above.\footnote{
    Figure out how term one behaves under replacing \emph{loops} $z$ by
    $z^k(t):=z(kt)$ und $v$ by $v^k$.
    }
  Unfortunately, common variational
  techniques build on at least semi-boundedness, say from below.

  One circumvents this problem by restricting
  attention to those $L^2$ (not $W^{1,2}$) gradient trajectories $\R\to\Ll_0 M$ along
  which the action
  remains bounded; see Remark~\ref{rem:non-bound-Aa}.
  It is the set $\Mm$ of these -- called the set of
  \textbf{finite energy trajectories} of the $L^2$ gradient $\grad\Aa_H$~--
  that carries the complete homology information of $M$
  whenever $\Aa_H$ is Morse. This brings in, through the
  back door, another common assumption in variational theory:
  Although in general $\Aa_H$ is not
  Palais-Smale with respect to the $W^{1,2}$ gradient
  (cf.\citerefFH[VI.1]{Hofer:1985a} and~\cite[\S\,3.3]{hofer:2011a}),
  it is sufficient that the \textbf{\Index{Palais-Smale condition}}
  holds \emph{on} $\Mm$;
\item
  has critical points of \textbf{infinite Morse index};
  cf. Example~\ref{ex:inf-M-index} ($H=0$).
  Unfortunately, therefore the symplectic action functional
  will not admit fundamental Morse theoretical tools such as
  the cell attachment theorem: The \index{theorem!Kuiper's --} 
  unit sphere in an infinite dimensional Hilbert space
  is contractible!\footnote{
    See Kakutani~\citerefFH{Kakutani:1943a}
    or apply \Index{Kuiper's theorem}~\citerefFH{Kuiper:1965a}.
    }

  One circumvents this problem by only looking at those eigenvalues
  that change sign along the path of Hessians associated to a flow
  trajectory connecting two critical points.
  Generically the sum of signed sign changes is well defined and finite,
  called \textbf{relative Morse index} or \textbf{spectral flow}.
\end{itemize}

\begin{remark}[Signs in $\Aa_H$ -- see Notation~\ref{not:notations_and_signs}
for a detailed discussion]
\label{rem:SIGNS-symp-action}\mbox{ }
\newline
\textit{Closed manifolds.} At the level of closed symplectic manifolds
the sign choices in~(\ref{eq:action-A}) are not relevant.
Changing the sign of $\omega$ is equivalent by~(\ref{conv:Ham-VF})
to changing the sign of $H$. 
But changing the sign of $H$, more precisely replacing $H=H_t$
by $\hat H=\hat H_t:=-H_{-t}$, results in
$\HF_k(H)\simeq\HF^{-k}(\hat H)$ induced by natural identification
of the two chain complexes.
Together with continuation $\HF^{-k}(\hat H)\simeq\HF^{-k}(H)$
and the natural isomorphisms to singular (co)homology
such change of sign induces nothing but the
Poincar\'{e} duality isomorphisms
$\Ho_{k+n}(M)\simeq\Ho^{2n-(k+n)}(M)$
of the closed symplectic, so orientable, manifold~$M$.
\newline
\textit{Cotangent bundles.} Motivated by classical mechanics
one would like to have as integrand $p\,dq-H\,dt$; this is the case for
convention~(\ref{eq:action-A}) with
$\omegacan=d\lambdacan$.
\end{remark}

\begin{remark}[Palais-Smale condition]\label{rem:PS-condition}
Suppose $f$ is a $C^1$ function on a Banach manifold
$\Bb$ equipped with a Riemannian metric,
see~\citerefFH{Palais:1966b},
and $\nabla f$ denotes the gradient.
A sequence $z_i\in \Bb$ along which
$f$ is bounded and $\nabla f$ converges to zero
is called a \textbf{\Index{Palais-Smale sequence}}.
One says that the \textbf{\Index{Palais-Smale condition}
holds on a subset} $U\subset\Bb$ if every
Palais-Smale sequence in $U$ admits
a subsequence converging to a critical point.
For a detailed account of the Palais-Smale
condition and its history see the
survey~\citerefFH{Mawhin:2010a}.
\end{remark}

\begin{remark}[Non-exact cases]\label{rem:non-exact}
Exactness of $\omega$ facilitates the definition of action
functionals on loops, but it can be dropped on the cost of either
\begin{itemize}
\item
  restricting to contractible loops and spanning in disks,
  cf.~(\ref{eq:action-A}), or
\item
  fixing one reference loop in each component of loop space
  and spanning in cylinders.
\end{itemize}
The thereby potentially arising multi-valuedness of the action can then
\begin{itemize}
\item
  either be ruled out by requiring $\omega$ to be symplectically aspherical
  (case of spanning disks) or symplectically atoroidal
  (case\index{$\omega$-atoroidal}
  of\index{atoroidal!symplectically --}\index{atoroidal!$\omega$- --}
  spanning\index{symplectically!atoroidal}
  cylinders),\footnote{
    \textbf{Symplectically} or \textbf{$\mbf{\omega}$-atoroidal} means that
    $\int_{\T^2} v^*\omega=0$ for every smooth map $v:\T^2\to M$;
    see~\citerefRF[\S 2.3]{Merry:2011a} for sufficient conditions in the cotangent bundle
    case $M=T^*Q$.
    }
  
\item
  or be accepted and dealt with by constructing chain complexes
  with coefficients in Novikov rings; cf.~\citerefFH{Hofer:1995a}.
\end{itemize}
\end{remark}

\subsection{Critical points and $L^2$ gradient}\label{sec:FH-gradient}

The critical points of the action functional $\Aa_H$ are
the $1$-periodic orbits of $X_{H}$ by
formula~(\ref{eq:dA_H}) for the differential
of $\Aa_H$ at any\footnote{
  The differential is well defined at any loop, contractible or not.
}
loop $z$.

\begin{remark}[Paying dynamics to get compactness]\label{rem:dyn-comp}
To turn the differential into a gradient
one needs to pick a Riemannian metric on the loop space.
Looking at the differential suggests the $W^{1,2}$
topology (absolutely continuous loops with square integrable
derivatives), but for this choice desirable compactness properties
fail, as mentioned above.
It was Floer's insight that taking $L^2$ gradient instead
provides sufficient compactness on relevant parts of loop space;
see Section~\ref{sec:FH-comp}.
The price to pay will be that the $L^2$ gradient
does not generate a flow on the whole loop space --
but it does on relevant parts. The relevant part
actually consists of the loops represented by the space
$\Mm$ of finite energy trajectories; cf.~(\ref{eq:Mm=Mm_xy}).
It is this space that carries the homology of $M$.
\end{remark}

Pick a $1$-periodic family $J_t$ of $\omega$-compatible
almost complex structures and let $g_{J_t}=\inner{\cdot}{\cdot}_t$ be
the associated family of Riemannian metrics on $M$.
Define the \textbf{$\mbf{L^2}$ inner product}
on the loop\index{$L^2$ inner product}
space\index{loop space!$L^2$ inner product}
at any loop $z$, contractible or not, by
$$
      \langle\cdot,\cdot\rangle=\langle\cdot,\cdot\rangle_{0,2}:
     T_z\Ll M\times T_z\Ll M\to\R,\quad
     (\xi,\eta)\mapsto\int_0^1 \langle\xi(t),\eta(t)\rangle_t\, dt ,
$$
where $\xi$ and $\eta$ are smooth vector fields along the loop $z$.

\begin{exercise}\label{exc:symp-action-fctl}
a)~Show that $\Aa_H$ is well defined, if $\Io_\omega=0$.
Recall that the identity $dH_t=-\omega(X_{H_t},\cdot)$
determines the vector field $X_{H_t}$.
Prove that
\begin{equation}\label{eq:dA_H}
\begin{split}
     d\Aa_H(z)\zeta
   &=\int_0^1\omega\left(\zeta,\dot z-X_{H_t}(z)\right) dt \\
  &=\int_0^1\INNER{\zeta}{-J_t(z)\bigl(\dot z-X_{H_t}(z)\bigr)}_t \, dt\\
   &=\int_0^1\INNER{\zeta}{-J_t(z)\dot z-\nabla H_t(z)}_t \, dt\\
   &=\bigl\langle\zeta,\underbrace{- J_t(z)\dot z-\nabla H(z)}
         _{\mathop{=:}\grad\Aa_H(z)}\bigr\rangle
\end{split}
\end{equation}
for every smooth vector field $\zeta$ along a contractible loop $z$ in
\begin{figure}
  \centering
  \includegraphics
                             {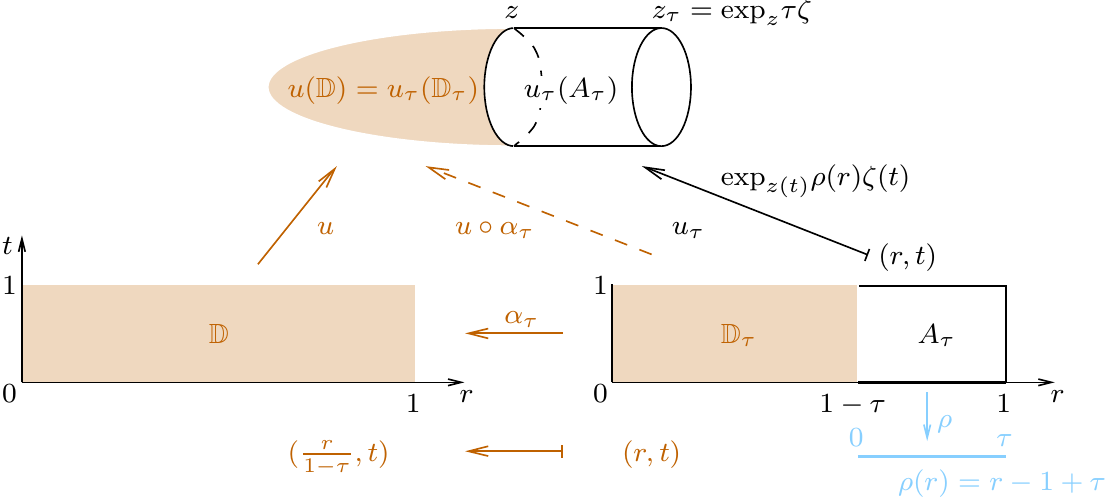}
  \caption{Derivative $d\Aa_H(z)\zeta
                 :=\left.\frac{d}{d\tau}\right|_0
                 \Aa_H(\exp_z\tau\zeta)$ and
                  spanning disks $u_\tau$}
  \label{fig:fig-dA_H}
\end{figure}
$M$; see also the hint to Exercise~\ref{prop:ENERGY-ACTION}.\footnote{
  Hint: To see identity one, pick an auxiliary Riemannian metric
  on $M$ with corresponding Levi-Civita connection $\nabla$
  and exponential map $\exp$. Given $z$ and $\zeta$, pick
  the families of loops $z_\tau$ and spanning disks $u_\tau$
  in Figure~\ref{fig:fig-dA_H}.
  (In the figure we have identified the closed unit disk
  in $\R^2$ minus the origin with the cylinder $(0,1]\times\SS^1$
  which we denote by $\D$! This abuses notation, but might facilitate
  reading.)
  In the calculation use that the integral is additive
  under the domain decomposition $\D=\D_\tau\cup A_\tau$
  to obtain a sum of three terms: One of them
 vanishes, as $\int_{\D_\tau} u_\tau^*\omega=\int_\D u^*\omega$
  is constant in $\tau$, and one of them leads easily to $X_{H_t}$.
  The third one is
  $$
     \left.\tfrac{d}{d\tau}\right|_{0}\int_{A_\tau} u_\tau^*\omega
     =\int_0^1
     \underbrace{
       \bigl.\tfrac{d}{d\tau}\bigr|_{0}{\textstyle \int_{1-\tau}^1} F(\tau,r,t)\, dr
       }_{-\left.F(\tau,r=1-\tau,t)\right|_{0}
       \tfrac{d}{d\tau}\bigr|_{0} (1-\tau)}
     dt
     =\int_0^1
     \underbrace{
       F(\tau=0,r=1,t)
     }_{\omega(\zeta,\dot z)}
     dt
  $$
  for $F=\omega\bigl(E_2(z,\rho\zeta)\zeta, E_1(z,\rho\zeta)\dot
  z+E_2(z,\rho\zeta)\rho\Nabla{t}\zeta\bigr)$
  and $E_i(z(t),\zeta(t)):T_{z(t)} M\to T_{\exp_{z(t)}\zeta(t)} M$
  for $i=1,2$ denoting the covariant partial derivatives of the exponential map
  ($E_i(z,0)=\1$ as shown
  e.g. in the section \textit{Analytic setup near hyperbolic singularities}
  in~\cite[App.]{weber:2015-MORSELEC-In_Preparation}).
}

b)~For general $\Io_\omega$ show that, although
$\Aa_H$ is not well defined, its linearization
does not depend on the choice of spanning disk.
Think about spanning disks as cylinders connecting
the periodic orbit $z$ with some fixed constant loop
$z_0(t)\equiv p\in M$. Extend the definition of $\Aa_H$
to components of the free loop space other than that
of the contractible loops. 
\end{exercise}

\begin{lemma}[Critical set and action spectrum -- closed manifold $M$]
\label{le:compactness-crit} \mbox{ }
\begin{enumerate}
\item[\rm a)]
 The set of critical points $\Crit \Aa_H =:\Pp_0(H)$, in fact, the set
$$
     \Pp(H):=\{z\in C^\infty(\SS^1,M) \mid \dot z=X_{H_t}(z)=J_t(z)\nabla H_t(z)\},
$$
is compact in the $C^1$ topology.
This builds on compactness of $M$.\footnote{
  Without compactness of $M$ one sometimes still
  gets compactness of the part $\Crit^{\le a}\Aa_H$ of the critical
  set below level $a$; cf. Section~\ref{sec:FH-T^*M} where $M=T^*Q$
  and~\citerefFH[App.\,A]{weber:2002a}.
  }
\item[\rm b)]
The set of critical values $\mbf{\mathfrak{S}(\Aa_H)}:=\Aa_H(\Crit\Aa_H)$,
called the \textbf{action spectrum},
is\index{$\mathfrak{S}(\Aa_H):=\Aa_H(\Crit\Aa_H)$!action spectrum}
a\index{action spectrum! $\mathfrak{S}(\Aa_H)$}
compact\,\footnote{
  In situations where $\Crit^{\le a}\Aa_H$ is compact $\forall a$
  one still gets closedness of $\mathfrak{S}(\Aa_H)$, easily.
  }
and measure-zero, thus nowhere dense, subset of the real
line~$\R$. The complement is open and dense; cf. Section~\ref{sec:Baire}.
\end{enumerate}
\end{lemma}

\begin{proof}
a) Pick a sequence $z_i\in\Pp(H)$ and consider the
sequence $w_i:=(p_i,v_i):=(z_i(0),\dot z_i(0))$ in the tangent bundle $TM$.
Pick a Riemannian metric on $M$.
As $\abs{\dot z_i(0)}=\abs{X_{H_{0}}\circ z_i(0)}
\le\max_{\SS^1\times M}\abs{X_H}$,
the sequence lives in a compact subset of $TM$.
Thus there is a subsequence, still denoted by $w_i$,
converging to an element $w=(p,v)\in TM$.
But the $z_i(0)$ are fixed points of the time-1-map
$\psi_1$ of the Hamiltonian flow~(\ref{eq:Ham-flow}) and so is the
limit $p$ by continuity of $\psi_1$.\footnote{
  Indeed $\psi_1p=\psi_1\lim_iz_i(0)=\lim_i\psi_1z_i(0)=lim_iz_i(0)=p$.
  }
Similarly $d\psi_1(p)v=v$ since $\psi_1$ is of class $C^1$.
Hence the $1$-periodic orbits $z_i$ converge in $C^1$
to the $1$-periodic orbit $z(t):=\psi_t p$.
\\
b) Compactness holds by~a) and continuity of $\Aa_H$.
Concerning nowhere dense the main step is to represent
$S:=\mathfrak{S}(\Aa_H)$ as a subset of the critical point set
$S^\prime$ of a smooth function on a finite dimensional manifold.
Then $S^\prime$, hence $S$, is of (Lebesgue) measure zero by
Sard's Theorem~\citerefFH{Sard:1942a} and therefore cannot
contain intervals.
But in this case $S$ is nowhere dense:
Any non-empty open set $U$ in $\R$ admits a non-empty
open subset $V$ disjoint from $S$.
To see this pick an open interval $I\subset U$.
As $S$ does not contain intervals, our
$I$ admits a point $p\notin S$. Since $S$ is closed $p$ admits an open
neighborhood $V\subset I$ disjoint from $S$.

\emph{Note}. Closed and nowhere dense does not imply measure zero:
There are closed subsets of $[0,1]$, called Cantor-like, which do not
contain intervals, but are of positive measure;
see e.g.~\cite[Ch.\,1 \S 6 Exc.\,4]{Stein:2005a}.
\newline
To carry out the reduction to Sard's Theorem
one can either apply the method of finite dimensional
approximation ($S=S^\prime$), see~\cite{milnor:1963a}
or e.g.~\citerefFH[Le.\,3.3]{weber:2002a} in a similar
situation, or adapt Sikorav's method described
in~\cite[\S\,5.2]{hofer:2011a}.
\end{proof}

\begin{exercise}
Use the Arzel\`{a}-Ascoli Theorem~\ref{thm:AA}
to prove Lemma~\ref{le:compactness-crit}\,a).
\end{exercise}

\subsubsection{The $\mbf{L^2}$ gradient $\mbf{\grad\Aa_H}$ as section of a Hilbert space bundle}
Concerning analysis one prefers to work in Banach
or even Hilbert spaces. As the gradient~(\ref{eq:dA_H}) of
$\Aa_H$ involves \emph{one} derivative of the loop,
the most natural space to consider
is the space
$$
     \Lambda M:=W^{1,2}_{\rm contr}(\SS^1,M)
$$
of contractible absolutely continuous loops $z:\SS^1\to M$
with square integrable derivative.\footnote{
  An absolutely continuous map $\SS^1\to M$
  admits a derivative almost everywhere.
  }
Although not a linear space, it is a Hilbert manifold,
modeled locally at a loop $z$ on the Hilbert space
$$
     W^{1,2}(\SS^1,z^*TM)=T_z\Lambda M
$$
of absolutely continuous
vector fields along $z$ with square integrable derivative.
A standard reference for the \textbf{geometry
of \Index{manifolds of maps}} is~\citerefFH{Eluiasson:1967a}.
If $z\in\Lambda M$, then $\grad\Aa_H(z)$
is an $L^2$ integrable vector field along $z$, that is
$$
     \grad\Aa_H(z)\in L^{2}(\SS^1,z^*TM)=:\Ee_z.
$$

\begin{remark}[No flow]\label{rem:no-flow}
To put it differently, the $L^2$ gradient $\grad\Aa_H$
of the action functional is not a tangent vector field to
the Hilbert manifold $W^{1,2}(\SS^1,M)$, nor to any
$W^{k,2}$, due to the \emph{loss of a
  derivative}.\footnote{
  Given $u\in W^{k,2}$ , then $\grad\Aa_H(u)$ lies in
  $W^{k-1,2}$. So $\grad\Aa_H$ is not a vector field on $W^{k,2}$,
  i.e. a section of $TW^{k,2}$, so the formal equation
  $\frac{d}{ds}u(s)=-\grad\Aa_H(u(s))$ isn't~an~ODE.
  }
The initial value problem is not well posed and $\grad\Aa_H$
does not generate a flow.\index{symplectic action!no flow}
The reason is that, by regularity, the loops of which a
solution cylinder is composed are smooth, so the flow
cannot pass any of the many non-smooth elements $z\in\Lambda M$.
\end{remark}

The union of all the Hilbert spaces $\Ee_z=L^{2}(\SS^1,z^*TM)$
forms a Hilbert space bundle $\Ee$ over $\Lambda M$.
A section is given by the $L^2$ gradient $\grad\Aa_H$.
Figure~\ref{fig:fig-L2-Hilbert-bdle}
illustrates the gradient section
and indicates the natural splitting
$$
     T_z\Ee M\cong T_z\Lambda M\oplus \Ee_z
     =W^{1,2}(\SS^1,z^*TM)\oplus L^2 (\SS^1,z^*TM)
$$
of the tangent bundle $T\Ee$ along the zero section
of $\Ee$, denoted again by $\Lambda M$.
\begin{figure}[h]
  \centering
  \includegraphics
                             {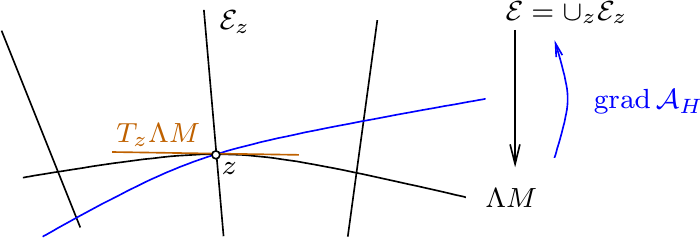}
  \caption{Hilbert bundle $\Ee\to\Lambda M$ over loop space and $L^2$
                gradient section}
  \label{fig:fig-L2-Hilbert-bdle}
\end{figure}

\begin{lemma}[Regularity]\label{le:regularity-crit}
Suppose $H\in C^k(\SS^1\times\Lambda M)$ for some integer $k\ge 2$.
Then any zero $z\in \Lambda M$ of $\grad\Aa_H$ is of class $C^k$.
\end{lemma}

\begin{proof}
By the Sobolev embedding theorem
$W^{1,2}(\SS^1)\hookrightarrow C^0(\SS^1)$
we get $z\in C^0$. By assumption
$z$ admits a weak derivative of class $L^2$, say $y$,
and $y=X_{H_t}(z)$ almost everywhere.
But the RHS,\index{RHS (right hand side)}\index{LHS (left hand side)}
hence $y$, is of class $C^0$ since $X_{H_t}(z(t))$ corresponds to $dH_t(z(t))$
under $\omega$. Thus the weak derivative is actually
the ordinary derivative $\dot z=y\in C^0$, hence $z\in C^1$.
More generally, assume $z\in C^\ell$ for some $\ell\in\N$. Then
$\dot z=X_{H_t}(z)\in C^{\min\{k-1,\ell\}}$, hence $z\in C^{\min\{k,\ell+1\}}$.
Thus the iteration terminates at $\ell=k-1$ and shows that $z\in\ C^k$.
\end{proof}

\subsection{Arzel\`{a}-Ascoli -- convergent subsequences}

\begin{theorem}[Arzel\`{a}-Ascoli Theorem]\label{thm:AA}
Suppose $(\Xx,d)$ is a compact metric space
and $C(\Xx)$ is the Banach space of
continuous functions on $\Xx$
equipped with the sup norm. Then the following is true.
A subset $\Ff$ of $C(\Xx)$ is pre-compact if and only if
the family $\Ff$ is \textbf{\Index{equicontinuous}}\footnote{
  $\forall\eps>0$ $\exists\delta>0$ such that
  $\abs{f(x)-f(y)}<\eps$
  whenever $d(x,y)<\delta$ and $f\in\Ff$.
  }
and \textbf{\Index{pointwise bounded}}\footnote{
  $\sup_{f\in\Ff}\abs{f(x)}<\infty$ for every $x\in\Xx$.
  }.
\index{theorem!Arzel\`{a}-Ascoli --}
\index{family!equicontinuous}
\index{family!pointwise bounded}
\end{theorem}

For a proof see e.g.~\cite[Thm.~A.5]{Rudin:1991b}.
The theorem generalizes to functions
taking values in a metric space.

\begin{exercise}\label{exc:AA}
Suppose $(\Xx,d)$ is a metric space
and $L>0$ is a constant.
a)~Show that any family $\Ff$ of Lipschitz continuous
functions on $\Xx$ with Lipschitz constant $L$
is equicontinuous.
\newline
b)~Show that any family of differentiable
functions on a closed manifold $Q$
whose derivative is bounded by $L$
is equicontinuous.
\end{exercise}

Further examples of equicontinuous families are
provided by $\alpha$-H\"older continuous functions.
In practice one often encounters families
of weakly differentiable functions on a compact
manifold $Q$ that are uniformly bounded in some
Sobolev space $W^{k,p}(Q)$. If $\alpha=k-\frac{n}{p}>0$
where $n=\dim Q$ then these functions are $\alpha$-H\"older
continuous by the Sobolev embedding theorem
which applies by compactness of $Q$.

\subsection{Hessian}\label{sec:Hess-A}
A Hessian is usually the second derivative
of a function or, more generally, of a section of a vector bundle
at a point in the domain.
However, this is in general only well defined
at a critical point, respectively a zero.
To extend the concept to general points
one chooses a connection or, equivalently,
a family of horizontal subspaces.
Furthermore, it is often convenient to
express the Hessian bilinear form via an inner product
as a linear operator, the Hessian operator.

Our setting is the following. Given a symplectic manifold
$(M,\omega)$ and a Hamiltonian $H:\SS^1\times M\to\R$,
pick a family $J_t=J_{t+1}$ of $\omega$-compatible
almost complex structures and denote by
$g_{t}$ the associated family of Riemannian metrics on $M$.
At each time $t$ consider the corresponding Levi-Civita connection
$\nabla^t$ with exponential map $\exp^t$ and
\textbf{\Index{parallel transport}} $\Tt^t_p(v):T_pM\to T_{\exp^t_p}M$
along the curve $[0,1]\ni\tau\mapsto\exp^t_p\tau v$.\footnote{
  To easy notation we write $\nabla=\nabla^t$ and $\exp=exp^t$ and
  so on. But we keep indicating time dependence of quantities
  which involve perturbation at some stage, such as $H_t$ and $J_t$.
}
Given a vector field $\zeta$ along an arbitrary loop $z$,\footnote{
  We assume tacitly that $\norm{\zeta(t)}$
  is smaller than the injectivity radius of $(M,g_t)$
  at $z(t)$.
  }
set
$$
     \exp_z\zeta:\SS^1\to M,\quad t\mapsto\exp^t_{z(t)}\zeta(t),
$$
to obtain a loop in $M$ homotopic to $z$
through $\tau\mapsto z_\tau:=\exp_z\tau\zeta$.

Now consider the map between Banach spaces defined near the origin by
\begin{equation}\label{eq:F_z-dA}
     f_z: T_z\Lambda M\to\Ee_z,\quad
     \zeta\mapsto\Tt_z(\zeta)^{-1}\grad\Aa_H(\exp_z\zeta).
\end{equation}
Since a $W^{1,2}$ vector field is in particular of class $L^2$,
there is the natural inclusion $T_z\Lambda M\subset \Ee_z$
which suggests to view this linearization
as an unbounded operator with dense domain.
The derivative at the origin
$df_z(0)\zeta=\left.\frac{d}{d\tau}\right|_{\tau=0}f_z(\tau\zeta)$
defines the \textbf{covariant \Index{Hessian operator}} of $\Aa_H$, namely
\begin{equation}\label{eq:Hess-A}
\begin{split}
     A_z:=df_z(0):L^2_z\supset W^{1,2}_z
   &\to L^2_z=L^2(\SS^1,z^*TM)
     \\
     \zeta
   &\mapsto -J_t(z)\Nabla{t}\zeta
     -(\Nabla{\zeta} J)(z)\dot z-\Nabla{\zeta}\nabla H_t(z),
\end{split}
\end{equation}
at any loop $z$.

\begin{figure}[h]
  \centering
  \includegraphics
                             {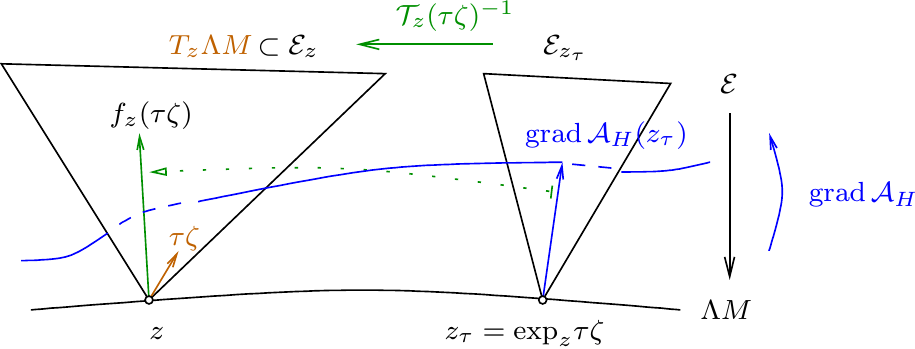}
  \caption{The Hessian operator
                $A_z\zeta=\left.\frac{d}{d\tau}\right|_{\tau=0}f_z(\tau\zeta)$
                 at any loop $z$}
  \label{fig:fig-parallel-transport-symp}
\end{figure}

\begin{exercise}\label{exc:symm-Hess}
Check that $df_z(0)\zeta$ is indeed given by~(\ref{eq:Hess-A}).
Show that $A_z$ is symmetric and even self-adjoint with
compact resolvent. Conclude that the eigenvalues of $A_z$ are
real, accumulate only at $\pm\infty$, and the eigenspaces
are of finite dimension. Show that $\dim\ker A_z\le 2n$.
\newline
[Hints: For symmetry see Exercise~\ref{exc:comp-alm-cx-strs}.
Self-adjointness is a regularity problem.
For compact resolvent see~\citerefFH[\S 2.3]{weber:2002a}.
To estimate $\dim\ker A_z$ see (\ref{eq:ker_A=Eig_1}).
]
\end{exercise}

The eigenvalues of $A_z$ are real by self-adjointness.
Let $z$ be a critical point of the symplectic action $\Aa_H$.
Observe that $A_z:=df_z(0)=D\grad\Aa_H(z)$ represents
the Hessian of $\Aa_H$ at the zero $z$ of the gradient.
The critical point $z$ is called
\textbf{\Index{non-degenerate}}\index{critical point!non-degenerate --}
if zero is not an eigenvalue of the Hessian $A_z$.
By self-adjointness of $A_z$ a critical point $z$ is non-degenerate
iff $A_z$ is surjective.
Below after having picked a trivialization we shall say more about the
spectrum. A \textbf{\Index{Morse function}} is a function all of whose
critical points are non-degenerate. A Hamiltonian $H$ is 
a \textbf{\Index{regular Hamiltonian}} if $\Aa_H$ is Morse.

\begin{exercise}\label{exc:non-deg-Crit-pp}
Given $z\in\Crit\Aa_H$, then z is non-degenerate as a critical point
iff $A_z$ is surjective iff $z$ is non-degenerate as a $1$-periodic
trajectory; see~(\ref{eq:def-non-deg}).
\end{exercise}

\begin{lemma}\label{le:non-deg=>isolated}
A non-degenerate critical point $z\in\Lambda M$ of
$\Aa_H$ is isolated.
\end{lemma}

\begin{proof}[Proof v1.]
The critical points of $\Aa_H$ near $z$ are in bijection with
the zeroes near the origin of the map $f_z$ given
by~(\ref{eq:F_z-dA}) and $z$ corresponds to the origin.
But no point other than the origin gets mapped to zero,
because $f_z$ is a local diffeomorphism near the origin
by the inverse function theorem. The latter applies since
the linearization $df_z(0)=A_z$ is a bijection: It is injective
by the non-degeneracy assumption, hence surjective by
self-adjointness.
\end{proof}

\begin{proof}[Proof v2.]
Recall the bijection $z\mapsto z(0)=:p$
from Exercise~\ref{exc:Pp-Fix-FH}.
Note that $\zeta(t):=d\psi_t(p)\zeta_0$ already
lies in the 'kernel' of the differential equation~(\ref{eq:Hess-A})
for any $\zeta_0\in T_{z(0)} M$. But $\zeta(t)$ must
close up at time one in order to lie in the kernel of $A_z$.
This happens precisely if $\zeta_0$ is
eigenvector of $d\psi_1(z(0))$ associated to the
eigenvalue $1$. Thus there is an isomorphism
\begin{equation}\label{eq:ker_A=Eig_1}
     \ker A_z\cong\Eig_1 d\psi_1(z(0)),\quad
     \zeta\mapsto\zeta(0).
\end{equation}
Now recall Exercise~\ref{exc:nondeg->isolated}.
\end{proof}

\subsubsection{The Hessian with respect to a unitary trivialization}
Given a loop $z$, pick a unitary trivialization~(\ref{eq:unitary-triv})
of the symplectic vector bundle $z^*TM\to\SS^1$, namely
a smooth family $\Phi$ of vector space isomorphisms
$\Phi(t):\R^{2n}\to T_{z(t)} M$ intertwining the
Hermitian triples $\omega_0,J_0,\inner{\cdot}{\cdot}_0$
and~$\omega,J,g_J$; cf.~(\ref{eq:R2n-comp-triples}).
Conjugation transforms the Hessian $A_z$
into the unbounded linear operator on $L^2=L^2(\SS^1,\R^{2n})$
with dense domain $W^{1,2}$~given~by
\begin{equation}\label{eq:Hess-A-triv}
\begin{split}
     A(z):=\Phi^{-1} A_z\Phi:L^2\supset W^{1,2}
   &\to L^2=L^2(\SS^1,\R^{2n})
     \\
     \zeta
   &\mapsto -J_0\dot \zeta-S_t\zeta
\end{split}
\end{equation}
where $S_t$ is a $1$-periodic family of symmetric matrizes
(Exercise~\ref{exc:symm-Hess}), namely
\begin{equation}\label{eq:S-triv}
     S_t v=
     \Phi^{-1}\bigl(
        J_t(z)\left(\Nabla{t}\Phi\right) v
       +(\Nabla{\Phi v} J)(z) \,\dot z
       +\Nabla{\Phi v}\nabla H_t(z)
     \bigr),\quad v\in\R^{2n}.
\end{equation}

\begin{example}[Infinite Morse index]\label{ex:inf-M-index}
The operator $-i\frac{d}{dt}$ on $C^\infty(\SS^1,\C^n)$
has eigenvectors $\zeta_k=e^{-i2\pi kt}z$ and eigenvalues
$\lambda_k=2\pi k$, $k\in\Z$, $z\in\C\setminus\{0\}$.
\end{example}

\subsection{Baire's category theorem -- genericity}\label{sec:Baire}
Since the notions surrounding Baire's category theorem are
easily confused, we enlist them for definiteness
in more detail than needed here. However,
all you should take with you in Section~\ref{sec:Baire}
is Baire's category Theorem~\ref{thm:Baire} part~(C)
and its first application Theorem~\ref{thm:A-Morse}
($\Aa_H$ Morse for generic $H$).

A \textbf{\Index{topological space}} consists of a set
$X$ together with a collection\footnote{
  It is often useful to call a set alternatively a
  \textbf{\Index{collection}} or a \textbf{\Index{family}}.
  Indeed isn't a ``collection of sets'' so much easier to absorb than
  a ``set of sets''?
}
$\Uu\subset 2^X$ of subsets $U\subset X$
subject to certain axioms.\footnote{
  Axioms: The collection $\Uu$, firstly, must contain the empty set
  $\emptyset$ and $X$ itself and, secondly, it must be closed under
  finite intersections and, thirdly, under arbitrary unions.
}
Such collection $\Uu$
is called a \textbf{\Index{topology}} on
$X$.\index{topology!closed subsets of --}
The\index{topology!open subsets --}
elements $U\in\Uu$ are called the \textbf{\Index{open subsets}}
of the topology, their
\textbf{\Index{complement}s} $\comp{U}:=X\setminus U$ the
\textbf{\Index{closed subsets}}.
We often denote a topological space $(X,\Uu)$ by the shorter symbol $\Tt$
with the understanding that $\Tt$ denotes the collection
of open sets -- the set $X$ being given implicitly as
the largest one of them, the union of all open sets.
\\
A subset $S\subset X$ of a topological space $\Tt$
is called \textbf{\Index{dense}}
if it meets (intersects) every non-empty open set or, equivalently,
if its closure $\bar S=X$ contains every point.
A subset $A$ of a topological space
is called \textbf{\Index{nowhere dense}} \index{subset!nowhere dense}
if its closure $\bar A$ has empty interior, that is
if there is no open set $U\not=\emptyset$ in which $A$ is dense:
For each non-empty open set $U$
there is a non-empty open subset $V\subset U$
disjoint from $A$.\footnote{
  Equivalently, the closure $\bar A$ of a nowhere dense set
  has a dense complement. Equivalently, the complement $A^{\rm C}$
  of a nowhere dense set is a set with dense \emph{interior}.
  (Not every dense set has a nowhere dense complement.)
  }
Any countable union of nowhere dense subsets is called
a\index{subset!meager} 
\textbf{\Index{meager subset}}\,\footnote{
  A meager subset, although a union of nowhere dense sets,
  can be dense: Consider $\Q\subset\R$.
  }
(alternatively, a subset of the \emph{first category in the sense of Baire}).
All other subsets, that is all non-meager subsets, are said to be
of the \textbf{\Index{second category in the sense of Baire}}. 
\index{subset!of $2^{\rm nd}$ category (Baire)}
These are somewhere dense.
\\
As first readings we recommend the Wikipedia article
\href{https://en.wikipedia.org/wiki/Meagre_set}{Meagre set},
the online handout
\href{http://www.ucl.ac.uk/~ucahad0/3103_handout_7.pdf}{The Baire category theorem and its consequences},
the Blog
\href{https://terrytao.wordpress.com/2009/02/01/245b-notes-9-the-baire-category-theorem-and-its-banach-space-consequences/}{The Baire category theorem and its Banach space consequences},
and $\S 1$ of~\cite{{Oxtoby:1980a}}.

\begin{figure}
  \centering
  \includegraphics
                             [height=3cm]
                             {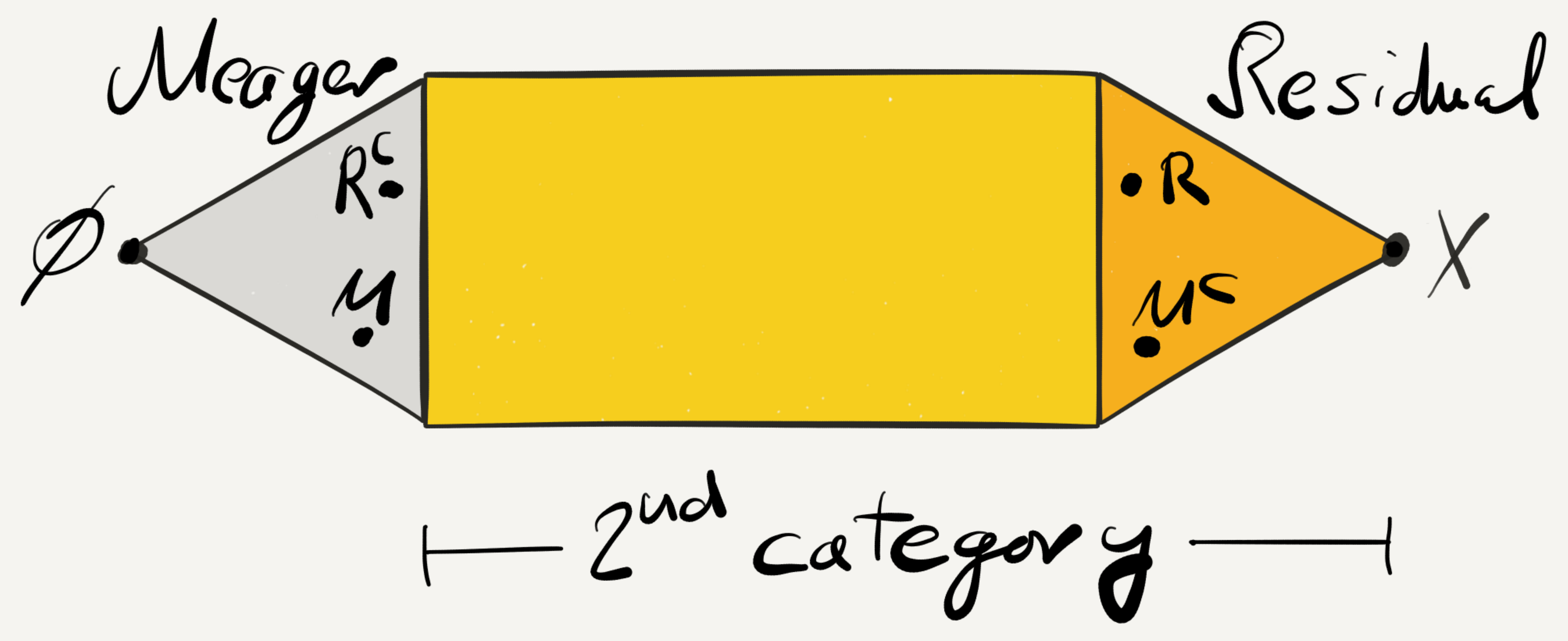}
  \caption{Partition of power set $2^X$ of
                \emph{complete} metric space $\Xx=(X,d)$}
  \label{fig:fig-meager-residual}
\end{figure}

\begin{exercise}
The subcollection $\Mm\subset\Tt$ of meager
subsets\index{$\sigma$-ideal}\index{ideal!$\sigma$- --}
is a \textbf{$\mbf{\sigma}$-ideal} of subsets: Subsets of a meager set
are meager, countable unions of meager sets are meager.
[Hint: Countable unions of countable unions are countable.]
\end{exercise}

How about the complements $R:=\comp{M}\subset X$ of meager subsets 
$M$? These are called \textbf{\Index{residual subsets}} 
or\index{subset!co-meager}\index{subset!residual}  
\textbf{co-meager subsets}.
Let's call them \textbf{\Index{residuals}}.
Let $\Rr$ be the set of residuals. Can the complement $R$ of
a meager set be meager? Or is it always non-meager, i.e.
of the second category? 

\begin{exercise}\label{exc:Baire-exc}
Suppose $U$ and $R$ are subsets of a topological space $\Tt$, show:
\begin{itemize}
\item[\rm (i)]
  $U$ open and dense $\:\Leftrightarrow\:$
  complement $U^{\rm C}$ closed and nowhere dense;
\item[\rm (ii)]
  $R$ residual $:\Leftrightarrow$
  $R^{\rm C}$ meager $\Leftrightarrow$
  $R$ $\supset$ countable intersect. of dense open sets;
\item[\rm (iii)]
  Countable intersections of residuals are residuals.
\end{itemize}
[Hint: (i)~Interior $\INT U=\bigl.\overline{U^{\rm C}}\bigr.^{\rm C}$.
(ii)~ '$\Rightarrow$'  Families
de\,Morgan~\cite[Thm.~0.3]{Kelley:1955a}.
'$\Leftarrow$' Suffices to show: Complement $R^{\rm C}$ is
contained in a meager set.]
\end{exercise}

\subsubsection{Complete metric spaces}

A metric space $\Xx=(X,d)$ comes with a natural topology $(X,\Uu_d)$,
called the \textbf{\Index{metric topology}}, where the collection of
open sets $\Uu_d$\index{topology!metric --}
consists of the family of open balls about the points of $X$
and arbitrary unions of such.
A subset $S\subset X$ of a metric space is \emph{\Index{dense}} if and
only if every point of $X$ is the limit of a sequence in $S$.

\begin{theorem}[Baire category theorem]\label{thm:Baire}
In a complete metric space $\Xx$
\begin{enumerate}
\item[\rm (A)]
  meager sets (countable unions of nowhere dense sets)
  have empty interior;
\item[\rm (B)]
  the complement of any meager set is dense,
  that is residual sets are dense;
\item[\rm (C)]
  countable intersections of dense open sets,
  therefore residuals, are dense.
\end{enumerate}
\end{theorem}

\begin{exercise}
Show that the three assertions (A), (B), (C) are equivalent.
\newline
[Hint: Take complements.]
\end{exercise}

By the exercise it suffices to prove part~(C) of the
theorem.\footnote{
  For a proof see e.g.~\cite[Thm.~2.2]{Rudin:1991b}
  or~\cite[Thm.~6.34]{Kelley:1955a}
  or~\cite[Thm.~9.1]{Oxtoby:1980a}.
  See also~\cite[\S 9]{Kelley:1976a}
  and for references to the original papers
  see~\cite[Notes to Ch.~III]{reed:1980a}.
  }
Concerning a slightly different direction
see Remark~\ref{rmk:Baire-new}.
In applications one often gets away with
a weaker form of the Baire theorem
(replace 'dense' by 'non-empty'):

\begin{corollary}[Baire category theorem -- weak form]\label{cor:Baire}
For a non-empty complete metric space $\Xx=(X,d)$ the following is true.
\begin{enumerate}
\item[\rm (b)]
  One cannot write $X$ as a countable union of nowhere dense sets.
  \newline
  (A non-empty complete metric space is non-meager in itself.)
\item[\rm ($\widetilde{{\rm b}}$)]
  If $X$ is written as a countable union of closed sets,
  then at least one of them has non-empty interior.
\item[\rm (c)]
  Countable intersections of dense open sets are non-empty.
\end{enumerate}
\end{corollary}

So in a non-empty complete metric space,
any set with non-empty interior is of the second category
(non-meager) by~(A). Moreover, the complement $R$
of a meager set $M$ cannot be meager: Otherwise
$\Xx=M\cup M^{\rm C}$ contradicting~(b).
Thus $\Rr\cap\Mm=\{M^{\rm C}\mid M\in\Mm\}\cap\Mm=\emptyset$
which answers the introductory questions and
is illustrated by Figure~\ref{fig:fig-meager-residual}.

Note that the properties \emph{nowhere dense},
\emph{dense}, and \emph{somewhere dense}
do not correspond to the sets $\Mm$, $\Rr$, and
$2^X\setminus (\Mm\cup\Rr)$.
While all elements of $\Rr$ are dense subsets of $\Xx$
and all nowhere dense subsets are located in $\Mm$,
it is possible that even dense subsets are elements of $\Mm$,
e.g. $M=\Q\subset \R=X$.

\begin{remark}[Warning]
Obviously, not all subsets of the second category
are dense. For example, the non-dense subset
$A=[-1,1]\subset\R=X$ is of the second category:
It is not meager since its complement is not dense.
In view of this, the in the literature not uncommon
wording \textit{``every set of the second category is dense
by Baire's category theorem''} is rather misleading,
often based on defining the sets of the second category
as those that contain countable intersections of dense open sets,
that is on confusing second category and residual.
\end{remark}

\subsubsection{Genericity}
It is common to call a property P \textbf{\Index{generic}}
if it is 'achievable by small perturbation'. More precisely, in our
context a generic property is one that is true for the elements
of the set $\Rr$ of residual subsets of a complete metric space
$\Xx$. In this case the property is shared by the elements
of a dense set. In other words, by a small perturbation of an
arbitrary pick one can get the desired property.\footnote{
  Saying \emph{``pick a generic element $x$ of $\Xx$''}
  actually means that one chooses $x\in\Rr$.
  }

For closed manifolds~$Q$ the set of smooth functions
equipped with the metric
\begin{equation}\label{eq:Cinfty-d}
     d(f,g):=\sum_{k=0}^\infty\frac{1}{2^k}\,
     \frac{\norm{f-g}_{C^k}}{1+\norm{f-g}_{C^k}}
\end{equation}
is a complete metric space $(C^\infty(Q),d)$, often
abbreviated by $C^\infty(Q)$; see e.g. footnote to~(\ref{eq:complete-45}),
\cite[1.46]{Rudin:1991b}, or \cite[IV.2]{Conway:1985a}.
The\index{complete metric space!of smooth functions} 
$C^k$ norm is the sum of the $C^0$ norms
of all partial derivatives up to order $k$
where $\norm{f}_{C^0}:=\sup_{Q} \abs{f}$.

\subsubsection{Non-degeneracy is a generic property}
\begin{theorem}[$\Aa_H$ Morse for generic $H$]
\label{thm:A-Morse}
Let $(M,\omega)$ be a \underline{closed} symplectic manifold. Then
there is a dense \underline{open} subset $\Hhreg$ of the complete
metric space $\Hh=(C^\infty(\SS^1\times M),d)$
such that the symplectic action functional $\Aa_H:\Ll_0 M\to\R$
is Morse whenever $H\in\Hhreg$.
\end{theorem}

The proof, given in Example~\ref{ex:generic-Morse} below, will serve
to illustrate abstract Thom-Smale transversality theory
(Section~\ref{sec:Thom-Smale-transversality}) in a simple setting.

\begin{definition}[Non-degenerate case, Morse-regular Hamiltonians]
\label{def:non-deg-case}
The terminology \textbf{\Index{non-degenerate case}}
refers to the situation $H\in\Hhreg$, that is all
$1$-periodic orbits $z\in\Pp_0(H)=\Crit \Aa_H$ are non-degenerate,
that is $\Aa_H$ is Morse.
Call\index{regular!M- --}\index{regular!Morse- --}\index{M-regular}
the
elements of\index{Hamiltonian!M(orse)-regular --}
$\Hhreg$\index{$\Hhreg$ M(orse)-regular Hamiltonians}
\textbf{\Index{Morse-regular}} or
\textbf{\Index{M-regular}} Hamiltonians.
\end{definition}

\begin{proposition}[Finite set]\label{prop:fin-crit-pts}
Given a closed symplectic manifold $(M,\omega)$
and an M-regular Hamiltonian $H\in\Hhreg$, then
$\Crit\Aa_H=\Pp_0(H)$ is a finite set.
\end{proposition}

\begin{proof}
\textit{v1.} $\Crit\Aa_H$ is compact
(Lemma~\ref{le:compactness-crit}) and
discrete (Lemma~\ref{le:non-deg=>isolated}).
\newline
\textit{v2.} Exercise~\ref{exc:finite-set}.
\end{proof}

\subsubsection{Variant: Metrizable complete topological vector spaces}

A pair $(E,\Uu)$ is called\index{vector space!topological --}
a \textbf{\Index{topological vector space}} if it consists of a
vector space $E$ together with a\index{Hausdorff topology}
Hausdorff\,\footnote{
  \textbf{Hausdorff} means
  that\index{topological space!Hausdorff}
  any two distinct points are separable
  by open sets (there are two disjoint open sets such that each
  of them contains precisely one of the two points).
  }
topology $\Uu$ for which the vector space operations are continuous.
Examples are normed vector spaces.
A sequence $(x_n)\subset E$ is called
a\index{Cauchy sequence!in topological vector space}
\textbf{Cauchy sequence} if each
open neighborhood $U\in\Uu$ of $0$
contains the difference $x_n-x_m\in\Uu$ for all $n,m$ sufficiently large.
If every Cauchy sequence converges to an element $x\in E$, then
$(E,\Uu)$\index{topological vector space!complete --}
is called\index{complete!topological vector space}
a \textbf{complete topological vector space}.
We recommend the concise presentation~\cite[\S 1.1]{Zimmer:1990a}.

A topological space $(X,\Uu)$ is called a \textbf{\Index{Baire space}}, or
simply \textbf{Baire}, if countable intersections of dense open sets,
hence residuals, are dense.
Examples are complete metric spaces, for instance Banach spaces,
hence Banach manifolds.

\begin{remark}\label{rmk:Baire-new}
Complete topological vector spaces $(E,\Uu)$ which are
\textbf{metrizable}\footnote{
  A topological space is called \textbf{metrizable}
  if\index{topological space!metrizable}
  it\index{metrizable topological space}
  admits a metric $d$ whose open balls
  \emph{generate the given topology}, that is every open set is a union of such balls.
}
are Baire spaces; for details of the proof see
e.g.~\cite[Prop.\,1.1.8~(B')]{Infusino:2016a}.
Theorem~\ref{thm:Baire} part~(C) holds
for metrizable complete topological vector spaces.
\end{remark}

\section{Downward gradient equation}\label{sec:DGF}

Throughout $(M,\omega)$ is a closed
symplectic manifold.
To emphasize time dependence
we denote time $1$-periodic Hamiltonians
$H\in C^\infty(\SS^1\times M)$~by~$H_t$.~Let
\begin{equation}\label{eq:J-fam-comp}
     J=\{J_t=J_{t+1}\}\subset\Jj(M,\omega)
\end{equation}
be a $1$-periodic family
of $\omega$-compatible almost complex structures
with associated $1$-periodic families of Riemannian metrics $g_J=\{g_{J_t}=:g_t\}$
and Levi-Civita connections $\nabla=\{\nabla(g_t)=:\nabla^t\}$.
Let $\Ll_0 M=C^\infty_{\rm contr}(\SS^1,M)$ be
the space of free contractible loops in $M$
and $\Pp_0(H)=\Crit\Aa_H$ the set of $1$-periodic
\emph{contractible} trajectories of the Hamiltonian flow $\psi$
given by~(\ref{eq:Ham-flow}).

As we indicated in Remark~\ref{rem:no-flow}
the initial value problem of the $L^2$ gradient $\grad\Aa_H$
is ill-posed on the loop space, no matter which Hilbert
or Banach completion one takes into consideration;
see also Remark~\ref{rem:dyn-comp}.
The way out is to interpret a curve $\R\to\Ll M$ in the loop
space asymptotic to critical points $z^\mp\in\Crit \Aa_H$
as a cylinder in $M$ and the formal downward
gradient equation on $\Ll M$ as a PDE for the cylinder
$\R\times \SS^1\to M$ in the manifold with
asymptotic boundary conditions given by
the two $1$-periodic orbits $z^\mp$.
The key property that makes the analysis work
is non-degeneracy of the critical points $z^\mp$:
This assumption leads to a Fredholm problem,
hence to solution spaces of \emph{finite} dimension.
In Section~\ref{sec:DGF}
we follow mainly~\citerefFH{salamon:1999a,Floer:1995a}
and~\cite[\S 6.5]{hofer:2011a}.
\begin{equation*}
\begin{gathered}
  \textsf{Throughout Section~\ref{sec:DGF} we 
  fix $J$ as in~(\ref{eq:J-fam-comp})
  and pick $H$ Morse-regular.}   
\end{gathered}
\end{equation*}

\subsection{Connecting trajectories}\label{sec:conn-traj}

Fix $J$ as in~(\ref{eq:J-fam-comp}).
Pick $H\in\Hhreg$, that is the functional $\Aa_H$ is Morse, so
its critical points are isolated.
A smooth map $u:\R\times \SS^1\to M$, $(s,t)\mapsto u(s,t)$,
is called a \textbf{\Index{trajectory}}
or a \textbf{\Index{Floer cylinder}}
if it satisfies the perturbed non-linear Cauchy-Riemann type elliptic
PDE, also called \textbf{\Index{Floer's equation}}, given by
\begin{equation}\label{eq:FLOER}
     \Ff(u):=\p_su+\grad\Aa_H(u_s)
     =\p_s u-J_t(u)\p_t u-\nabla H_t(u)=0 .
\end{equation}
Here $u_s$ denotes the loop $u(s,\cdot)$, for
the $L^2$ gradient $\grad\Aa_H(u_s)$ see~(\ref{eq:dA_H}).
Floer's equation generalizes three theories, as indicated in
Figure~\ref{fig:fig-Floers-interpol-eq}.
\begin{figure}[h]
  \centering
  \includegraphics
                             [height=4cm]
                             {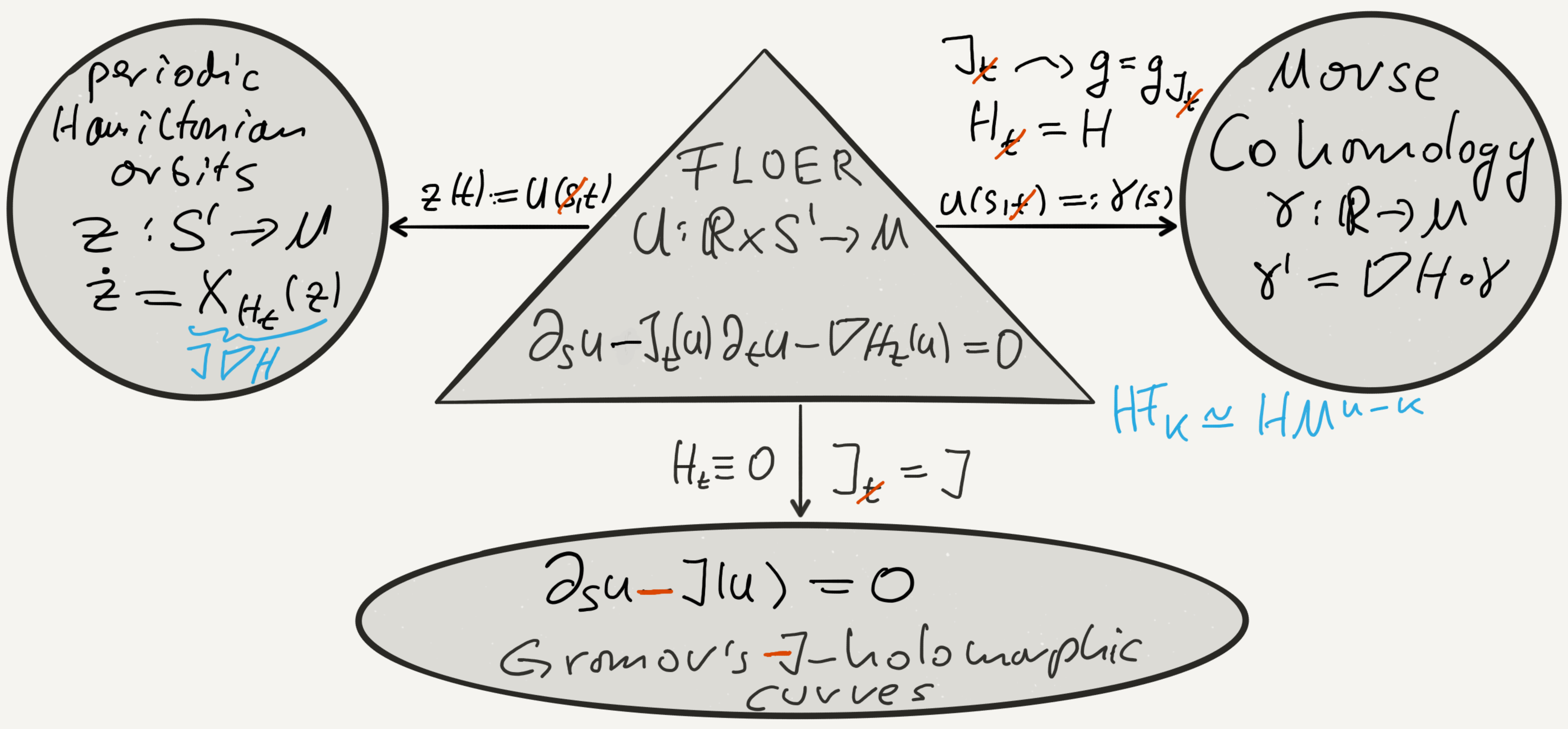}
  \caption{Floer's equation interpolates between three theories}
  \label{fig:fig-Floers-interpol-eq}
\end{figure}

For $H\in\Hhreg$ and $z^\mp\in\Crit\Aa_H$ the
\textbf{set of connecting trajectories} or the
\textbf{\Index{connecting manifold}}\footnote{
  Although at this stage $\Mm(z^-,z^+)$ is not yet a manifold.
  }
or\index{connecting trajectories!set of --} 
the \textbf{\Index{moduli space}} of connecting solutions
\begin{equation}\label{eq:conn-mf}
     \Mm(z^-,z^+)=\Mm(z^-,z^+;H,J)
\end{equation}
consists of all Floer cylinders $u$ with asymptotic limits
\begin{equation}\label{eq:LIMITS}
     \lim_{s\to\mp\infty} u(s,t)=z^\mp(t)
\end{equation}
where the convergence is uniform in $t\in\SS^1$,
in other words in $C^0(\SS^1)$.

\begin{remark}[Asymptotic convergence: $C^0$ versus $W^{1,2}$]
Only asking $C^0(\SS^1)$ convergence in the boundary
condition of a connecting trajectory $u$
may come as a surprise, given that the
natural domain of the symplectic action functional $\Aa_H$
is the Banach manifold of $W^{1,2}$ loops in $M$.
For instance, one would expect that the energy
of a connecting trajectory is the difference
of the action values of the two asymptotic boundary conditions
whenever the asymptotic convergence happens
with respect to the functional's natural topology, in this case $W^{1,2}$.
However, Lemma~\ref{prop:ENERGY-ACTION}
below shows that $C^0$ asymptotic convergence~(\ref{eq:LIMITS})
is already sufficient to enforce finite energy of $u$.
But is this really that surprising?\,\footnote{
  As the $L^2$ gradient does not generate a flow,
  the cylinder substitutes are already special to start with.
  Secondly, being solutions to a PDE,
  as opposed to an ODE, they are rather exotic, hence rare, creatures.
  But from exceptional objects one may expect
  exceptional behavior.
  }
Indeed, together with non-degeneracy of the boundary conditions
$z^\mp$, Theorem~\ref{thm:finite-ENERGY=LIMITS-exist} below
guarantees even exponential convergence -- based on just a $C^0$ convergence 
assumption, but not to forget that $u$ solves an elliptic PDE..
\end{remark}

\subsubsection{Energy}
A useful notion concerning gradient type equations is
the \textbf{energy} of arbitrary
paths,\index{energy $E(u)$ of curve}
that is arbitrary cylinders in $M$ in our case. It is defined by
\begin{equation}\label{eq:ENERGY-Floer}
     E(u):=\frac12 \int_0^1\int_{-\infty}^\infty\left(
     \Abs{\p_su}^2 +\Abs{J_t(u)\p_t u+\nabla H_t(u)}^2\right) ds\, dt
     \ge 0
\end{equation}
for \emph{any} smooth cylinder $u:\R\times\SS^1\to M$
and where the integrand is evaluated at $(s,t)$, of course.
Note the following.
\begin{itemize}
\item
  The energy $E(u)\ge 0$ vanishes precisely
  \emph{on the periodic orbits} viewed as \textbf{constant trajectories}
  $u(s,\cdot)\equiv z\in\Pp_0(H)$.
\item
  The energy of a \emph{trajectory} is the square of the $L^2$ norm, namely
  \begin{equation}\label{eq:E=L2-norm}
     E(u)=\norm{\p_s u}_2^2.
  \end{equation}
\item
  Among all smooth cylinders $w:\R\times\SS^1\to M$
  subject to the same asymptotic boundary conditions~(\ref{eq:LIMITS}),
  whether $z^\mp\in\Crit\Aa_H$ are non-degenerate or not,
  it is precisely the Floer cylinders $u$ that \emph{minimize
  the energy} with $E(u)=\Aa_H(z^-)-\Aa_H(z^+)$; see
Exercise~\ref{prop:ENERGY-ACTION-general}.
\end{itemize}

\begin{lemma}[Connecting trajectories are of finite energy]
\label{prop:ENERGY-ACTION}
Given\index{connecting trajectories!finite energy}
a\index{finite energy!connecting trajectories}
connecting trajectory $u\in\Mm(z^-,z^+)$,
non-degeneracy of $z^\mp$ is actually not needed,
then\footnote{
  By definition $\int_{\R\times\SS^1} v^*\omega
  :=\int_0^1\int_{-\infty}^\infty \omega(\p_su,\p_tu)\,dsdt$, mind that
  the order of $s,t$ in $\R\times\SS^1$ must be the same as of
  $\p_s,\p_t$ when inserted into $u^*\omega$.
  }
\begin{equation}\label{eq:ENERGY-FLoer-flow-line}
\begin{split}
     E(u)
   &=\int_{\R\times\SS^1} u^*\omega
     -\int_0^1 H_t(z^-(t))\, dt 
     +\int_0^1 H_t(z^+(t))\, dt\\
   &=\Aa_H(z^-)-\Aa_H(z^+).
\end{split}
\end{equation}
\end{lemma}

\begin{proof}
By~(\ref{eq:E=L2-norm}) and the gradient nature of the Floer equation
we obtain that
\begin{equation*}
\begin{split}
     E(u)
   &=\int_{-\infty}^\infty\inner{\p_s u_s}{\p_s u_s}_{L^2(\SS^1)}\, ds\\
   &=\lim_{T\to\infty}
     \underbrace{\bigl(\Aa_H(u_{-T})-\Aa_H(u_T)\bigr)}_{\le\Aa_H(z^-)-\Aa_H(z^+)}\\
   &=\lim_{T\to\infty}\left(\int_{[-T,T]\times\SS^1}
     u^*\omega-\int_0^1 H\circ u_{-T}\, dt+\int_0^1 H\circ u_T\, dt\right).
\end{split}
\end{equation*}
To obtain the final step we fixed an extension $\widebar{v}:\D\to M$ of
$z^-:\SS^1\to M$. Then we chose the extensions of the loops $u_{-T}$
and $u_T$ required by the definition of $\Aa_H$
by simply connecting these loops along the cylinder $u$ to $z^-$ and
its spanning disk $\widebar{v}$.
In the difference $\Aa_H(u_{-T})-\Aa_H(u_T)$ each of the integrals
$\int_\D\widebar{v}^*\omega$ and $\int_{(-\infty,-T]\times\SS^1}u^*\omega$
appears twice, but with opposite signs.
Because the difference is uniformly bounded from above, 
the limit as $T\to\infty$ exists and
is given by the RHS of the first identity
in~(\ref{eq:ENERGY-FLoer-flow-line});
here $C^0$ convergence~(\ref{eq:LIMITS}) enters.
\\
The second identity in~(\ref{eq:ENERGY-FLoer-flow-line})
holds by the earlier argument: Given an extension $\widebar{v}:\D\to
M$ of the loop $z^-$ choose the natural extension $u\# v$ of $z^+$.
\end{proof}

\begin{remark}[Circumventing non-boundedness of $\Aa_H$ -- finite energy]
\label{rem:non-bound-Aa}
To construct a Morse complex one needs that the trajectories used to
define the boundary operator have precisely one critical point sitting
asymptotically at each of the two ends.
In our case, suppose $u:\R\times \SS^1\to M$ is a trajectory,
how can we guarantee existence of asymptotic limits $z^\mp\in\Crit\Aa_H$?
Well, if they exist, the energy identity~(\ref{eq:ENERGY-Floer-no-sol})
shows that $u$ is of \emph{finite energy}.
Indeed it turns out, see Theorem~\ref{thm:finite-ENERGY=LIMITS-exist},
that finite energy of a trajectory $u$ is sufficient to enforce existence of
asymptotic limits $z^\mp\in\Crit\Aa_H$ -- under the assumption that
our functional $\Aa_H$ is Morse. This is why we pick $H\in\Hhreg$ in
Section~\ref{sec:DGF}.
\end{remark}

\begin{exercise}\label{prop:ENERGY-ACTION-general}
Any smooth cylinder $w:\R\times\SS^1\to M$
subject to the asymptotic boundary conditions~(\ref{eq:LIMITS}),
whether $z^\mp\in\Crit\Aa_H$ are non-degenerate or not,
satisfies the identity
\begin{equation}\label{eq:ENERGY-Floer-no-sol}
\begin{split}
     E(u)
   &=\frac12 \int_0^1\int_{-\infty}^\infty
     \Abs{\p_s u-J_t(u)\p_t u-\nabla H_t(u)}^2\, ds\, dt\\
   &\quad
     +\Aa(z^-)-\Aa(z^+).
\end{split}
\end{equation}
[Hint: Start at the integral term in~(\ref{eq:ENERGY-Floer-no-sol}).
Permute the integrals, so the integral over $t$ becomes
the $L^2(\SS^1)$ inner product of $\p_s u_s+\grad\Aa_H(u_s)$ with itself.
Use the gradient nature of the Floer equation to end up with
$E(u)+\lim_{T\to\infty}\left(\Aa_H(u_{T})-\Aa_H(u_{-T})\right)$
which is equal to $E(u)+\Aa_H(z^+)-\Aa_H(z^-)$,
as shown in the proof of Lemma~\ref{prop:ENERGY-ACTION}.]
\end{exercise}

\subsubsection{Finite energy trajectories} 
Consider the \textbf{set of finite energy trajectories}
$$
     \Mm:=
     \{\text{solutions $u:\R\times\SS^1\to M$ of Floer's
     equation~(\ref{eq:FLOER})} \mid E(u)<\infty\}.
$$
If $\Aa_H$ is Morse, then by
Theorem~\ref{thm:finite-ENERGY=LIMITS-exist}
below every finite energy trajectory is connecting
and this non-trivial fact contributes the inclusion $\subset$
in the identity
\begin{equation}\label{eq:Mm=Mm_xy}
     \Mm=\bigcup_{z^\mp\in\Pp_0(H)} \Mm(z^-,z^+).
\end{equation}
The other inclusion $\supset$ already
holds without the Morse assumption
by~(\ref{eq:ENERGY-FLoer-flow-line}).
In the non-degenerate case, still assuming
that $\omega$ and $c_1(M)$ vanish over $\pi_2(M)$,
counting with appropriate signs the 1-dimensional components
appearing on the RHS of~(\ref{eq:Mm=Mm_xy})
defines the Floer boundary operator.
In the special case of an autonomous
$C^2$ small Morse function $H$,
see~Proposition~\ref{prop:HZ-C2small},
and autonomous $J$, the count defines the Morse
boundary operator and the RHS of~(\ref{eq:Mm=Mm_xy})
is naturally homeomorphic to $M$ itself.
Consequently Floer homology
represents the singular integral
co/homology of $M$.\footnote{
  As $M$ is closed and automatically orientable,
  Poincar\'{e} duality identifies $\Ho^{2n-k}\cong\Ho_k$.
  }
These remarks show the significance~of

\begin{theorem}\label{thm:finite-ENERGY=LIMITS-exist}
Let $u:\R\times\SS^1\to M$ be a Floer cylinder
in the non-degenerate~case,\footnote{
  That is $u$ satisfies the Floer equation~(\ref{eq:FLOER}) on the
  closed manifold $M$ and $H\in\Hhreg$.
  }
then the following are equivalent.
\begin{enumerate}
\item[\rm \texttt{(finite energy)}]
  $E(u)<\infty$.
  \index{\rm \texttt{(finite energy)}}
\item[\rm \texttt{(asymp.limits)}]
  There are periodic orbits $z^\mp\in\Pp_0(H)$ which are
  the $C^0$ limits of the family of loops
  $s\mapsto u_s$; see~(\ref{eq:LIMITS}).
  Moreover, the derivative $\p_su_s(t)\to 0$ decays, as $s\mp\infty$,
  again uniformly in~$t$. \index{\rm \texttt{(asymp.limits)}}
\item[\rm \texttt{(exp.decay)}]
  There exist constants $\delta,c>0$ such that
  $$
     \Abs{\p_su(s,t)}\le ce^{-\delta\abs{s}}
  $$
  at every point $(s,t)$ of the cylinder $\R\times\SS^1$.
  \index{\rm \texttt{(exp.decay)}}
\end{enumerate}
\end{theorem}

\begin{proof}[Outline of proof.]
The proof is non-trivial, for details we
recommend~\citerefFH{salamon:1999a}. Roughly speaking,
\texttt{(finite energy)}, namely by~(\ref{eq:E=L2-norm})
finiteness of the integral $\norm{\p_s u}^2_2=E(u)<\infty$
over the whole cylinder $\R\times\SS^1$, enforces via
additivity of the integral with respect to disjoint union of the
domain that the integrals over annuli $(T,T-1)\times\SS^1$
must converge to zero, as $T\to\mp\infty$.
But a mean value inequality for
$e=\abs{\p_su(s,t)}^2$ based on a differential
inequality of the form
$\Delta e:={\p_s}^2+{\p_t}^2\ge -c_1-c_2 e^2$
provides a pointwise estimate of $\abs{\p_su(s,t)}$
in terms if the $L^2$ norm over some annulus
$A_s$ which does not depend on $s$.
Thus $\abs{\p_su_s}=\abs{\p_tu_s-X_{H_t}(u_s)}$ converges to zero,
as $s\to\mp\infty$, uniformly in $t$.
From this one already concludes
existence of a sequence $s_k\to\infty$
such that $u_{s_k}$ converges uniformly
to some periodic orbit $z^+$.
But by non-degeneracy all periodic orbits are isolated
which implies that any sequence diverging
to $+\infty$ leads to the same limit; similarly
for $s\to-\infty$. One has confirmed
\texttt{(asymptotic limits)}.
But \texttt{(asymptotic limits)} implies \texttt{(finite energy)}
by Lemma~\ref{prop:ENERGY-ACTION}.

Obviously \texttt{(exponential decay)}
immediately leads to \texttt{(finite energy)}
by~(\ref{eq:E=L2-norm}) and explicit integration.
Conversely, how \texttt{(finite energy)}
leads to \texttt{(exponential decay)}
is hard to illustrate, have a look
at~\citerefFH[\S 2.7]{salamon:1999a}.
A key observation is that $\xi_s(t)=\p_s u_s(t)$
lies in the kernel of the trivialization
$D=\frac{d}{d s}+A(s)=\p_s-\Jbar_0\p_t-S(s,t)-C(s,t)$
of the linearized operator $D_u$,
see~(\ref{eq:D_u-lin-Floer})
and~(\ref{eq:triv-D_u-lin-Floer}),
and that for kernel elements the function
$f(s):=\frac12\int_0^1\abs{\xi_s(t)}^2\, dt$
satisfies a differential inequality
$f^{\prime\prime}(s)\ge\delta^2 f(s)$
for $\abs{s}$ sufficiently large; this is based on
invertibility of the Hessian operators $A(s)$
at $\mp\infty$, hence near $\mp\infty$.
It is here where non-degeneracy of the asymptotic boundary
conditions $z^\mp$ enters.
This way one arrives at an $L^2$ version
of the desired estimate, namely
$f(s)\le c^\prime e^{-\delta\abs{s}}$.
Application of the operator
$\p_s+\Jbar_0\p_t$ to $D\xi=0$
leads to a differential inequality
$\Delta \abs{\xi_s(t)}^2\ge -c^{\prime\prime}\abs{\xi_s(t)}^2$,
thus to a mean value inequality
for $\abs{\xi_s(t)}^2$
which together with the formerly obtained
$L^2$ estimate establishes \texttt{(exponential decay)}.
\end{proof}

\subsection{Fredholm theory}\label{sec:Fredholm}
For convenience of the reader
we enlist some basic notions and tools of Fredholm theory.
Concerning details we highly recommend~\cite{mcduff:2004a}.

Suppose throughout that $X,Y,Z$ are Banach
spaces. Denote by $\Ll(X,Y)$ the
\textbf{Banach space of bounded linear operators} $L:X\to Y$
\index{$\Ll(X,Y)$ bounded linear operators}
equipped with the \textbf{\Index{operator norm}}
$\norm{L}:=\sup_{\norm{x}=1}\{\norm{Lx}\colon x\in X\}$.
A bounded linear operator $D:X\to Y$ is a
\textbf{\Index{Fredholm operator}} if it has a closed
range and finite dimensional kernel and cokernel.
The latter is the quotient space $\coker D:= Y/\im D$.
It inherits a Banach space structure from $Y$ since $\im D$ is closed.
The map defined on the space of Fredholm operators
$\Fred(X,Y)$ by
$$
     \INDEX D:=\dim\ker D-\dim\coker D
$$
is called the \textbf{\Index{Fredholm index}} of $D$.

\begin{theorem}[semi-Fredholm estimate]\label{thm:fredholm-estimate}
Given $D\in\Ll(X,Y)$ and a compact operator $K:X\to Z$,
suppose there is a constant $c>0$ such that
\begin{equation}\label{eq:semi-Fredholm}
     \Norm{x}_X\le c\left(\Norm{Dx}_Y+\Norm{Kx}_Z\right)
\end{equation}
for every $x\in X$. Then $D$ has closed range and
finite dimensional kernel.
\end{theorem}

\begin{exercise}
Prove the previous theorem. Use it to show openness $\Fred\subset\Ll$.
[Hint: Concerning finite dimensionality it suffices to show that
the unit ball in $\ker D$ is compact. Use finite dimension
together with the Hahn-Banach theorem to reduce the proof
of closed range to the case in which $D$ is injective;
choose a complement $X_1$ of $\ker D$ and replace $X$ by $X_1$.
Concerning openness use again finiteness of $k=\dim\ker D$
to define an augmentation $(D,K_0):X\to Y\oplus \R^k$
of $D$ which is injective and has closed range $Z$.
Apply the open mapping theorem, cf. proof of Lemma~\ref{le:bhjjg767},
to conclude that the inverse of the bounded linear bijection
$(D,K_0):X\to Z$ is continuous.]
\end{exercise}

\begin{exercise}[Stability properties]\label{exc:Fred-stability}
a)~Show that the subset $\Fred(X,Y)\subset\Ll(X,Y)$ is open with respect to
the operator norm and the index is \textbf{\Index{locally constant}},
that is constant on each component.
Show that\footnote{
  [Hint: Use constant index and the fact that the kernel dimension
  at most decreases locally, see e.g.~\citerefFH[Le.\,3.7]{weber:2002a},
  to conclude that $\dim \coker(D+L)=\dots\le 0$ if $\norm{L}$ is small.]
  }
surjectivity of $D\in\Fred(X,Y)$ is an open property with respect to
the operator norm as well.

b)~A \textbf{\Index{compact operator}} is a (bounded) linear
operator $K:X\to Y$ which takes bounded sets to
\textbf{\Index{pre-compact set}s}, that is sets of compact closure.
Show that the sum $D+K$ of a Fredholm operator $D$
and a compact operator $K$ is Fredholm and of the same index.
\end{exercise}

\begin{lemma}\label{le:bhjjg767}
For a bounded linear operator $D:X\to Y$
are equivalent:
\begin{itemize}
\item 
  $D$ is an injection with closed range.
\item 
  There is a constant $c>0$ such that
  \begin{equation}\label{eq:inj-est}
     \Norm{x}_X\le c\Norm{Dx}_Y
  \end{equation}
  for every $x\in X$.
\end{itemize}
\end{lemma}

\begin{proof}
'$\Rightarrow$'
By the open mapping theorem\footnote{
  \textbf{The \Index{open mapping theorem}.} A bounded linear
  surjection between Banach spaces is open.
  A map is called\index{open map} \textbf{open}
  if it maps open sets to open sets.
  }
\emph{the inverse of a bounded linear bijection between
Banach spaces is continuous}. 
Now pick $c:=\norm{(\tilde D)^{-1}}_{\Ll(\im D,X)}$
where $\tilde D:X\to\im D$, $x\mapsto Dx$.
'$\Leftarrow$' By contradiction the inequality itself shows
that $\ker D=\{0\}$. By Theorem~\ref{thm:fredholm-estimate}
the range of $D$ is closed.
\end{proof}

\begin{definition}
A \textbf{\Index{complement}}, often called
\textbf{\Index{topological complement}},
of a closed linear subspace $X_0\subset X$ is a closed linear
subspace $Z\subset X$ such that $X_0\oplus Z=X$.\footnote{
  Here $\oplus$ denotes the \textbf{internal direct
  sum}\index{direct sum~internal} 
  of two closed subspaces which by definition
  means $X_0+Z=X$ and $X_0\cap Z=\{0\}$.
  }
If $X_0$ admits a complement it is a \textbf{\Index{complemented
subspace}}.
\end{definition}

Examples of complemented subspaces are finite dimensional
subspaces and closed subspaces of finite codimension.
In a Hilbert space every closed subspace is complemented.
We recommend the book by Brezis~\cite[II.4]{Brezis:1983a}.

\begin{definition}
A \textbf{\Index{right inverse}} of a surjective bounded linear
operator $D:X\to Y$ is a bounded linear operator $T:Y\to X$
such that $DT=\1_Y$.
\end{definition}

\begin{exercise}\label{exc:RI<=>complemented}
Given a surjective bounded linear operator $D:X\to Y$, then
$$
     \text{$D$ admits a right inverse $T$}
     \quad\Leftrightarrow\quad
     \text{$\ker D$ is complemented}.
$$
[Hint: '$\Rightarrow$'
A natural try is $Z:=\im T$. Use $DT=\1_Y$ to derive
the injectivity estimate~(\ref{eq:inj-est}) for $T$, so $\im T$ is
closed, and to conclude $\im T\cap\ker D=\{0\}$. Writing $x=x-TDx+TDx$
shows that $X=(\ker D)+Z$.
'$\Leftarrow$'
Note that the restriction $D|_Z:Z\to Y$ to the complement
is a bounded bijection.]
\end{exercise}

The previous exercise shows that any surjective Fredholm
operator admits a right inverse. This generalizes, see part ii),
as follows; for details see e.g.~\citerefFH{weber:2002a}.

\begin{exercise}\label{exc:gen-surj-Fred}
Let $D:X\to Y$ be Fredholm and $L\in\Ll(Z,Y)$. It~holds~that:
\begin{enumerate}
\item[i)] 
  The range of the bounded operator
  $$
     D\oplus L:X\oplus Z\to Y,\quad
     (x,z)\mapsto Dx+Lz,
  $$
  is closed with finite dimensional complement.
  How about $\dim\ker (D\oplus L)$? 
  Give an example in which $\ker (D\oplus L)$
  is infinite dimensional.
\item[ii)] 
  If $D\oplus L$ is surjective, then $\ker(D\oplus L)$ admits a
  complement, thus a right inverse. Moreover, the projection
  $P:\ker (D\oplus L)\to Z$, $(x,z)\mapsto z$,
  to the second component is a Fredholm operator with
  $$
     \ker P\simeq\ker D,\qquad
     \coker P\simeq\coker D,
  $$
  thus $\INDEX P=\INDEX D$.
\end{enumerate}
\end{exercise}

\begin{remark}[Non-linear Fredholm theory]
\label{rem:nonlin-Fred}
A continuously differentiable map $f:X\to Y$
is a \textbf{\Index{Fredholm map}}
if the linearization $df(x):X\to Y$ is a Fredholm operator
for every $x\in X$. In this 
case \index{Fredholm index!of Fredholm map} 
$\INDEX(f):=\INDEX df(x)$ is the
\textbf{Fredholm index} of the Fredholm map~$f$.
If $df(x)$ is not onto one calls $x$ a \textbf{critical point}
of\index{critical point!of Fredholm map}
the\index{Fredholm map!critical point of --}
Fredholm map $f$.
\newline
For a map $f:X\to Y$ of class $C^\ell$ with $\ell\ge 1$,
Fredholm or not, an element $y\in Y$ is called a
\textbf{\Index{regular value}} of $f$ if for each pre-image
$x\in f^{-1}(y)$ the linear operator $D:=df(x):X\to Y$ is onto 
and admits a \emph{right inverse}.\footnote{
  Any $y$ outside the image of $f$, i.e. with empty pre-image
  $f^{-1}(y)=\emptyset$, is a regular value.
  }
Then by the implicit function theorem,
see e.g.~\cite[Thm.~A.3.3]{mcduff:2004a},
the \textbf{\Index{regular level set}}\index{level set!regular --}
\begin{equation}\label{eq:RVT}
     \Mm:=f^{-1}(y)\subset X
\end{equation}
is a $C^\ell$ Banach manifold whose tangent spaces
are given by the kernels, that is
$$
     T_x\Mm=\ker df(x).
$$
This result is called the \textbf{\Index{regular value
theorem}}.\index{theorem!regular value --}
If $f$ is even a Fredholm map, then $\Mm$ is finite
dimensional and $\dim\Mm=\INDEX(f)$.
\end{remark}

\begin{exercise}
a) The Fredholm index of a Fredholm map is well defined.
b)~The set of critical points of a Fredholm map is closed.
\newline
[Hint: b) Exercise~\ref{exc:Fred-stability}\,a).]
\end{exercise}

\subsection{Connecting manifolds}\label{sec:mod-space}
Consider a connecting manifold
$\Mm(z^-,z^+;H,J)$ as defined by~(\ref{eq:conn-mf}).
Recall that $H\in\Hhreg$ is Morse-regular.
In this section we show that it is a smooth manifold for
generic Hamiltonian $H$ and its dimension is
the difference of the canonical Conley-Zehnder indices of $z^\mp$.

To start with denote the left hand side of Floer's
equation~(\ref{eq:FLOER}) by $\Ff_H(u)$, often denoted by $\bar\p_{H,J}$
to emphasize its Cauchy-Riemann type nature. Let
\begin{equation*}
\begin{split}
     D_u=D\Ff_H(u):W^{1,p}_u:=W^{1,p}(\R\times\SS^1,u^*TM)
   &\to
     L^p_u
\end{split}
\end{equation*}
denote linearization at a zero $u$. Here $p>2$ is a constant;
see Remark~\ref{rem:p>2}. As outlined below,
see~(\ref{eq:D_u-lin-Floer-gen}), the linearized operator is of the form
\begin{equation}\label{eq:D_u-lin-Floer}
\begin{split}
     D_u\zeta
   &=\Nabla{s}\zeta-J_t(u)\Nabla{t}\zeta
     -\Nabla{\zeta}\nabla H_t(u)-\left(\Nabla{\zeta} J\right)(u)\p_t u
     \\
   &=\left(\tfrac{D}{ds}+A_{u_s}\right)\zeta
\end{split}
\end{equation}
for every smooth compactly supported vector field $\zeta$ along $u$.
Here $A_{u_s}$ denotes the covariant Hessian operator~(\ref{eq:Hess-A})
of $\Aa_H$ based at a loop $u_s:=u(s,\cdot)$.
Actually formula~(\ref{eq:D_u-lin-Floer}) not only makes sense for
$p>2$, let us allow $p>1$.
The \Index{Sobolev spaces} $L^p_u$ and $W^{1,p}_u$ are the closures of
the\index{$W^{1,p}_u$ Sobolev space of vector fields along $u$}
vector space of smooth compactly supported vector fields $\zeta$
along $u$ with respect to the Sobolev
norms\index{$\norm{\cdot}_{1,p}$ Sobolev $W^{1,p}$ norm}
$$
     \Norm{\zeta}_p
     :=\left(\int_{-\infty}^\infty\int_0^1\Abs{\zeta}^p\right)^{\frac{1}{p}}
     ,\quad
     \Norm{\zeta}_{1,p}
     :=\left(\int_{-\infty}^\infty\int_0^1\Abs{\zeta}^p
     +\Abs{\Nabla{s}\zeta}^p+\Abs{\Nabla{t}\zeta}^p\right)^{\frac{1}{p}}.
$$

\begin{definition}
Abbreviate $\Hh:=C^\infty(\SS^1\times M)$. The elements of the
set\index{$\Hhreg(J)$ MS-regular Hamiltonians}
$$
     \Hhreg(J):=\{H\in\Hh\mid
     \text{$D\Ff_H(u)$ onto
     $\forall u\in\Mm(z^-,z^+;H,J)$
     $\forall z^\mp\in\Pp_0(H)$}\}
$$
are\index{Hamiltonian!M(orse)-S(male) regular --}
called \textbf{\Index{MS-regular Hamiltonian}s}
or \textbf{\Index{Morse-Smale Hamiltonian}s}.
\end{definition}

\begin{exercise}[Constant connecting trajectories, MS-regular $\Rightarrow$ M-regular]
Suppose $H\in\Hhreg(J)$. Show that every
Hamiltonian loop $z\in\Pp_0(H)$ is non-degenerate,
hence $\Aa_H$ is Morse. This shows that $\Hhreg(J)\subset\Hhreg$.
\newline
[Hint: Consider the constant trajectory $u_s\equiv z$.]
\end{exercise}

\begin{exercise}[Morse-Smale condition]
As indicated in Figure~\ref{fig:fig-Floers-interpol-eq},
if $u,H,J$ are independent of $t$, then the Floer equation
recovers the gradient flow of $\nabla H$ on the closed Riemannian
manifold $(M,g_J)$.
Show that in this case the elements of $\Hhreg(J)$ are precisely
those Morse functions on $M$ which satisfy the
\textbf{\Index{Morse-Smale condition}},
i.e. all stable and unstable manifolds intersect transversely.
\end{exercise}

The significance of MS-regular Hamiltonians $H\in\Hhreg(J)$
lies in the fact that their connecting manifolds
are smooth manifolds, even of finite dimension,
as a consequence of the regular value theorem;
cf. Remark~\ref{rem:nonlin-Fred}.
However, to apply that theorem, two assumptions need to be verified:
Firstly, for zero to be a regular value, the kernel of $D_u$
needs to be complemented; cf.
Exercise~\ref{exc:RI<=>complemented}.
Secondly, to get to finite dimension of $\Mm(z^-,z^+;H,J)$
the kernel has to be finite dimensional.
Fredholm operators satisfy both criteria.

Thus, modulo proving that $D_u$ is a Fredholm operator
and calculating its Fredholm index, the regular value theorem
provides part~(ii) of 

\begin{theorem}[Connecting manifolds]\label{thm:A-MS}
Given a closed symplectic manifold $(M,\omega)$
and a family of $\omega$-compatible almost complex
structures $J_t=J_{t+1}$, then
\begin{itemize}
\item[\rm (i)]
  the set $\Hhreg(J)$ of regular Hamiltonians is a residual
  of\, $\Hh:=C^\infty(\SS^1\times M)$;
\item[\rm (ii)]
  for any $H\in\Hhreg(J)$ any 
  $\Mm(z^-,z^+):=\Mm(z^-,z^+;H,J)$ is a smooth
  manifold. The dimension of the
  component of $u$ is given by\,\footnote{
  Do not miss our standing assumptions $\Io_{c_1}=0=\Io_\omega$.
  }
  $$
     \dim \Mm(z^-,z^+)_u=\INDEX\, D\Ff_H(u)
     =\CZcan(z^-)-\CZcan(z^+)
  $$ 
  where $\CZcan$ is the canonical Conley-Zehnder index normalized
  by~(\ref{eq:CZ-normalization-canonical-88}).
\end{itemize}
\end{theorem}

For the proof of part~(i) we refer to~\citerefFH[\S5]{Floer:1995a}.
It utilizes a tool called Thom-Smale transversality
theory, see Section~\ref{sec:Thom-Smale-transversality} below.
It is crucial that the linearized operator $D\Ff_H(u)$
is already Fredholm to start with, but this holds true
precisely for Morse-regular Hamiltonians. This explains
one reason for our standing assumption $H\in\Hhreg$.
Concerning part~(ii) we shall sketch below the proof of the Fredholm
property of $D_u=D\Ff_H(u)$ and the calculation of its Fredholm index denoted by
$\INDEX\, D_u$.

\begin{remark}[Critical points unaffected by MS-perturbation]
\label{rem:M->MS}
By~\citerefFH[Thm.\,5.1\,(ii)]{Floer:1995a} one can $C^\infty$
approximate a given Morse-regular $H$ by MS-regular Hamiltonians $H^\nu$
which $C^2$ agree with $H$ along its (finitely many) $1$-periodic orbits.
Thus $\Crit\Aa_H\subset\Crit_{H^\nu}$. 
In~\citerefFH[Rmk.\,5.2\,(ii)]{Floer:1995a} it is pointed out that it
is\index{{open \color{red} problems \color{black}}}
an \emph{open problem} whether it is sufficient to perturb
$H$ outside some open neighborhood $U\subset M$ of the images
of the $1$-periodic orbits of $H$. This would guarantee equality
$\Crit\Aa_H=\Crit\Aa_{H^\nu}$ for large $\nu$.
However, the authors point out that it is possible to perturb
the family $J$ outside such neighborhood $U$ to achieve
MS-regularity for the given M-regular Hamiltonian $H$ itself,
just with respect to a perturbed $1$-periodic family of
$\omega$-compatible structures.
\end{remark}

\subsubsection*{Linearization at general cylinders $\mbf{u}$}

When it comes to gluing, in Section~\ref{sec:FH-gluing},
it is necessary to linearize $\Ff=\Ff_H$
given by~(\ref{eq:FLOER})
not only at zeroes of $\Ff_H$, but at more general smooth cylinders
$u:\R\times\SS^1\to M$.
The procedure is completely analogous
to the definition of the covariant Hessian~(\ref{eq:Hess-A}),
just replace the map in~(\ref{eq:F_z-dA}) by the map
\begin{equation}\label{eq:D_u-lin-Floer-gen}
     f_u:L^p_u\supset W^{1,p}_u\to L^p_u,\quad
     \zeta\mapsto\Tt_u(\zeta)^{-1}\Ff_H(\exp_u \zeta),
\end{equation}
given any cylinder $u\in \Bb^{1,p}(z^-,z^+)$.
Here $\Bb=\Bb^{1,p}(z^-,z^+)$ denotes the Banach manifold
which, roughly speaking, consists of all
continuous cylinders $u:\R\times\SS^1\to M$
which are locally of Sobolev class\footnote{
  Here it is crucial to choose $p>2$, because in this case
  $W^{1,p}$ implies continuity, so one can work with local coordinate
  charts on $M$ to analyze $u$.
  }
$W^{1,p}$ and converge asymptotically in a suitable way
to the given periodic orbits $z^\mp$. Convergence
of the elements $u$ of $\Bb$
must be such that the tangent space $T_u \Bb$
coincides with the space $W^{1,p}_u:=W^{1,p}(\R\times\SS^1,u^*TM)$
of $W^{1,p}$ vector fields along $u$;
see~\citerefFH[Thm.~5.1]{Floer:1995a} for details.
Note that $\Ff_H(u)\in L^p_u$ is a vector field
along $u$ of class $L^p$. All these spaces fit together
in the form of a Banach space bundle $\Ee\to\Bb$
whose fiber over $u\in\Bb$ is $\Ee_u=L^p_u$;
see Figure~\ref{fig:fig-L2-Hilbert-bdle} for a similar case.
Now $\Ff_H$ is a section of $\Ee$ whose
regularity depends on the regularity of the function $H$;
which is smooth here.
To show that 
\begin{equation}\label{eq:linearization}
     D\Ff_H(u)\zeta:=d f_u(0)\zeta
     =\left.\frac{d}{d\tau}\right|_0 f_u(\tau\zeta)
\end{equation}
is equal to the operator $D_u$ displayed in~(\ref{eq:D_u-lin-Floer})
one can either utilize local coordinates on $M$ or work with global
notions.\footnote{
  We use the notation $D\Ff_H(u)\zeta$ of a \Index{Fr\'{e}chet derivative},
  but\index{derivative!Fr\'{e}chet --}\index{derivative!G\^{a}teaux --}
  define it by \Index{G\^{a}teaux derivative}
  $\tfrac{D}{d\tau}\mid_{\tau=0}\Ff_H(\exp_u\tau\zeta)$.
  See e.g.~\cite[\S\,3.1]{Lebedev:2003a} for these notions and when they coincide.
  }
Details of both possibilities
can be found in~\citesymptop[App.\,A]{weber:1999a}
in the slightly different context of Section~\ref{sec:FH-T^*M}.

\begin{remark}[The condition $p>2$]\label{rem:p>2}
Roughly speaking, maps on a ${\color{blue} 2}$-dimensional domain of
Sobolev class $W^{k,p}$ are continuous, so one can localize,
and well behaved with respect to relevant compositions and products
whenever $kp>{\color{blue} 2}$; for details see in~\cite{mcduff:2004a}
the paragraphs prior to Prop.~3.1.9 and App.\,B,
see also the Blog
\href{https://symplecticfieldtheorist.wordpress.com/2015/05/08/lp-or-not-lp-that-is-the-question/}{$L^p$ or not $L^p$, that is the question}.
\newline
For instance, the non-linear Fredholm theory,
see Remark~\ref{rem:nonlin-Fred},
requires to equip the domain $\Bb$
of the section $\Ff=\Ff_H$ in~(\ref{eq:FLOER})
with the structure of a differentiable Banach manifold.
Choosing $\Bb=\Bb^{1,p}(z^-,z^+)$ with $p>{\color{blue} 2}$,
as we did in the previous remark, does the job
(and is the usually selected option), but
leads into the realm of $L^p$ estimates
which are much harder to obtain than $L^2$ estimates.
Another option would be to choose $k={\color{cyan} 2}$
and $p>\frac{{\color{blue} 2}}{{\color{cyan} 2}}$,
of course $p=2$, and view $\Ff_H$ as a section of the Hilbert
space bundle $\Ee^{1,2}\to\Bb^{2,2}(z^-,z^+)$.
Unfortunately, this choice brings in higher derivatives.
\end{remark}

\subsection*{The Fredholm operator $\mbf{D_u}$}
Consider two (non-degenerate) critical points
$z^\mp\in\Crit\Aa_H$ and suppose that
$u\in\Mm(z^-,z^+;H,J)$ is a connecting trajectory.

\begin{theorem}[Fredholm operator]\label{thm:D_u-Fredholm}
By non-degeneracy of $z^\mp$
the linear operator $D_u:W^{1,p}_u\to L^p_u$
given by~(\ref{eq:D_u-lin-Floer}) is Fredholm
for $1<p<\infty$ and
\begin{equation}\label{eq:D_u-F_index-CZ}
     \INDEX D_u=\CZcan(z^-)-\CZcan(z^+).
\end{equation}
\end{theorem}

To prove the theorem it is convenient
to represent $D_u$ by an operator $D$
acting on vector fields that take values in $\R^{2n}$.

\subsubsection{Trivialization} 
The connecting trajectory $u\in\Mm(z^-,z^+;H,J)$ extends continuously
from the open cylinder $Z=\R\times\SS^1$ to its compactification
$\widebar Z=(\R\cup\{\mp\infty\})\times\SS^1$, because there are the two
periodic orbits $z^\mp$ sitting at the ends by the asymptotic $C^0$ limit
condition~(\ref{eq:LIMITS}).
Thus we have in fact a hermitian vector bundle $\widebar{u}^*TM$
over the compact cylinder-with-boundary $\widebar Z$. Pick a canonical
unitary trivialization $\Phi(s,t):\R^{2n}\to T_{\widebar u(s,t)}M$
according to Proposition~\ref{prop:unitary-triv}.
In fact, for ease of notation, let us right away
agree to omit any 'bars' from now on completely.
Proceed as in~(\ref{eq:Hess-A-triv})
to represent $D_u:W^{1,p}_u\to L^p_u$ by the operator
\[
     D:W^{1,p}:=W^{1,p}(\R\times\SS^1,\R^{2n})
     \to L^{p}:=L^{p}(\R\times\SS^1,\R^{2n})
\]
given by
\begin{equation}\label{eq:triv-D_u-lin-Floer}
     D=\tfrac{d}{ds}+A(s)=\p_s-J_0\p_t-S(s,t)+\text{\st{$C(s,t)$}}.
\end{equation}
Here the family of symmetric matrizes
$S(s,t)=S(s,t)^T\in\R^{2n\times 2n}$ is given by
replacing in~(\ref{eq:S-triv}) the loop $z$ by the family
of loops $s\mapsto u_s$ and $\Phi(t)$ by $\Phi(s,t)$.
The only new element is the matrix family
$C(s,t)=\Phi(s,t)^{-1}\left(\Nabla{s}\Phi(s,t)\right)$
which converges to zero, as $s\mp\infty$, uniformly in $t$,
because $\Nabla{s}$ actually stands for $\Nabla{\p_s u}$
and $\p_su$ vanishes asymptotically
by~\texttt{(asymptotic limits)}.\footnote{
  In fact, for any cylinder $u\in\Bb^{1,p}(z^-,z^+)$
  asymptotic convergence $\p_su\to 0$ holds.
  }
The bad news is that in general $C$ is not symmetric,
thereby destroying self-adjointness of $A(s)$.
The good news is that $C$ only amounts to a compact
perturbation of $D$ (cf.\citerefFH[Le.~3.18]{robbin:1995a}),
so Fredholm property and index, if any, will not change by
Exercise~\ref{exc:Fred-stability} if we simply ignore $C$
(while investigating \emph{these two} properties).

\begin{figure}
\hfill
\begin{minipage}[b]{.32\linewidth}                
  \centering
  \includegraphics
                             [height=3.3cm]                 
                             {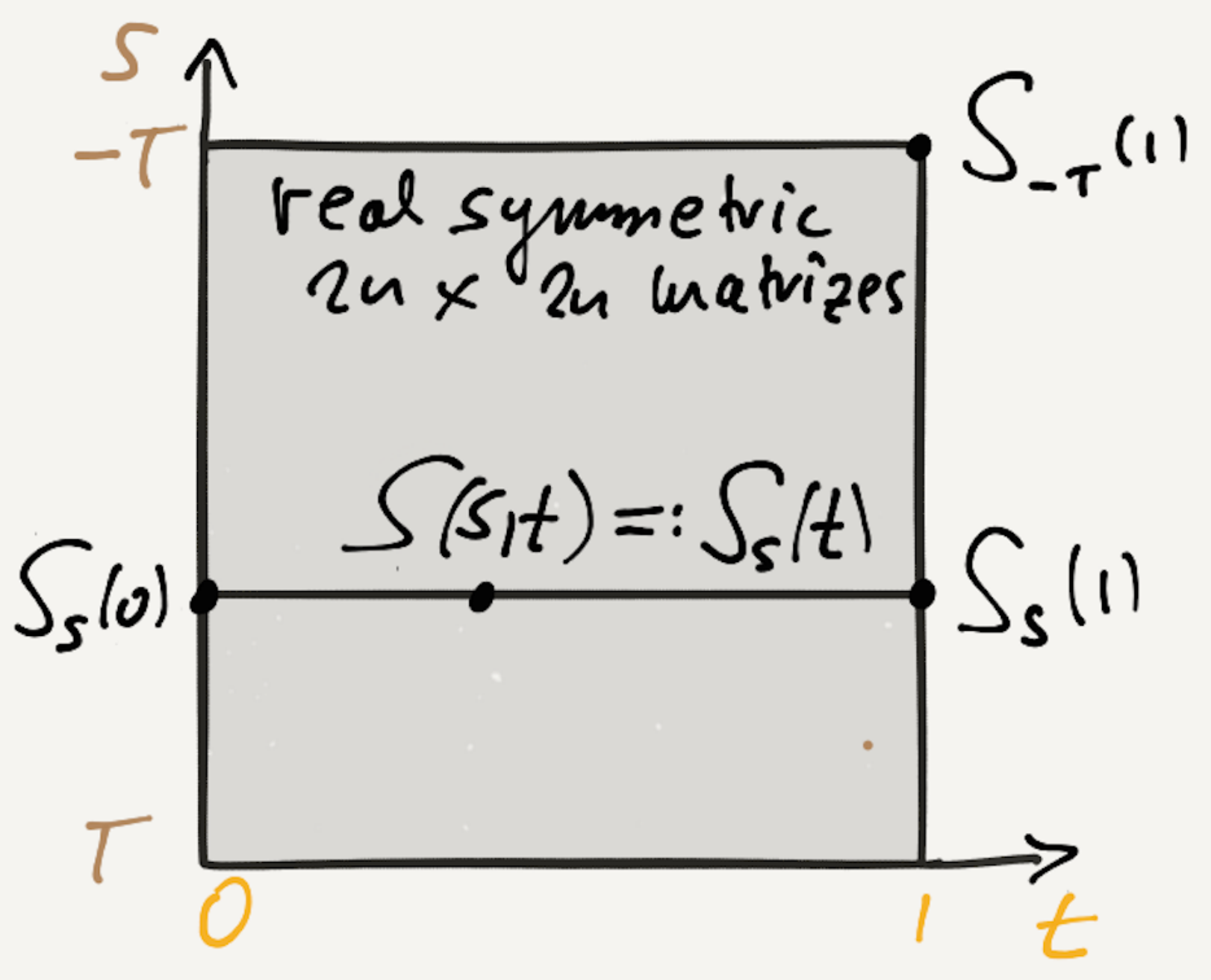}
\end{minipage}
\hfill
\begin{minipage}[b]{.32\linewidth}                
  \centering
  \includegraphics
                              [height=3.3cm]                
                             {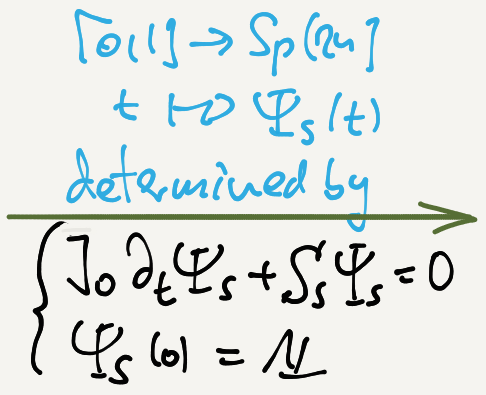}
\end{minipage}
\hfill
\begin{minipage}[b]{.32\linewidth}                
  \centering
  \includegraphics
                             [height=3.3cm]                
                             {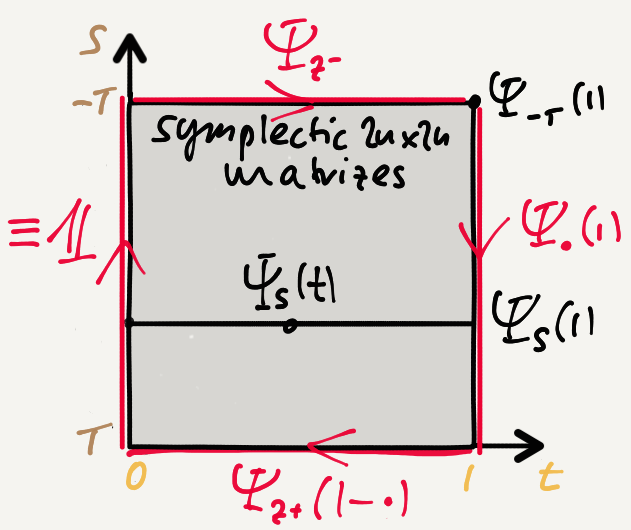}
\end{minipage}
\hfill
\vspace{-.4cm}
\caption{Rectangles worth of matrices: Symmetric $\longrightarrow$
  symplectic}
\label{fig:fig-SmSp}
\end{figure}

From now on we set $C=0$.
For fixed $s\in\R$ the path of symmetric matrizes
$S_s:[0,1]\to\R^{2n\times 2n}$, $t\mapsto S(s,t)$,
determines a symplectic path $t\mapsto \Psi_s(t)$ by
$$
     \dot\Psi_s
     =J_0 S_s\Psi_s
     =\Jbar_0 (-S_s)\Psi_s
     ,\qquad
     \Psi_s(0)=\1,
$$
see Figure~\ref{fig:fig-SmSp} and Exercise~\ref{exc:symm-symp}. The
asymptotic limit matrizes
$
     S^\mp(t):=\lim_{s\to\mp\infty} S(s,t)
$
are given by~(\ref{eq:S-triv}) with $z=z^\mp$
and their corresponding symplectic paths $\Psi^\mp$ lie
in $\SsPp^*(2n)$: Indeed the paths $\Psi^\mp$ coincide with
the paths $\Psi_{z^\mp,v^\mp}$ in~(\ref{eq:CZ-well-def-Ham}),
because they satisfy the same ODE and initial condition,
but the latter lie in $\SsPp^*(2n)$ by non-degeneracy of $z^\pm$;
cf. Exercises~\ref{exc:CZ-well-def-Ham}
and~\ref{exc:non-deg-Crit-pp}.
Thus $\Psi^\mp$ do admit Conley-Zehnder indices
for which we shall choose the canonical
normalization~(\ref{eq:CZ-normalization-canonical-88}), of course.
By the uniform limit condition~(\ref{eq:LIMITS})
these indices are actually already shared
by the paths $\Psi_{\mp T}\in\SsPp^*(2n)$ whenever
$T>0$ is sufficiently large. For simplicity we set
\begin{equation}\label{eq:pm-infty-Psi}
     \Psi_{z^\mp}:=\Psi_{\mp T}.
\end{equation}

\subsubsection{Fredholm property}
We aim to show that the bounded linear operator
\[
     D=\p_s-J_0\p_t-S:W^{1,p}\to L^p
\]
in~(\ref{eq:triv-D_u-lin-Floer})
is Fredholm whenever $1<p<\infty$. To ease notation we shall omit
domain $\R\times\SS^1$ and target space $\R^{2n}$ in our notation, in
general. Pick $p\in(1,\infty)$.

The first step is to establish for $D$ the semi-Fredholm
estimate~(\ref{eq:semi-Fredholm}) with some suitable compact
operator, say the operator
\[
     K:W^{1,p}\to W^{1,p}([-T,T]\times\SS^1,\R^{2n})
     \hookrightarrow L^p([-T,T]\times\SS^1,\R^{2n})=:L^p_T
\]
given by composing restriction (continuous) with inclusion (compact for
large $T$ by~\citerefFH[Le.\,3.8]{robbin:1995a}).
Details are non-trivial, see~\citerefFH[\S\,2.3]{salamon:1999a}.
One gets that
\[
     \Norm{\xi}_{W^{1,p}}\le c\left(\Norm{D\xi}_{L^p}+\Norm{\xi}_{L^p_T}\right)
\]
for all $\xi\in W^{1,p}$. The constant $c=c(p)$
depends on the choice of $p\in(1,\infty)$.\footnote{
  Here we do \emph{linear} Fredholm theory.
  The need for $p>2$ only arises later when we
  use \emph{non-linear} Fredholm theory,
  cf. Remark~\ref{rem:p>2},  to show the manifold
  property of the spaces of connecting flow lines.
  The non-linear application
  is based on the present linear findings.
  }
Thus $D:W^{1,p}\to L^p$ has finite dimensional kernel
and closed range by Theorem~\ref{thm:fredholm-estimate}.
Now consider the \textbf{\Index{formal adjoint operator}}
\[
     \Hdual{D}=-\p_s-J_0\p_t-S: W^{1,q}\to L^q,\qquad
     \tfrac{1}{p}+\tfrac{1}{q}=1,\quad
     1<p<\infty,
\]
given by Proposition~\ref{prop:formal-adjoint}
whose construction is detailed in 
Appendix~\ref{sec:formal-adjoint}.\footnote{
  The idea is to view $D$ as an unbounded operator on $L^p$ with dense
  domain $W^{1,p}$ and construct the unbounded functional analytic
  adjoint $\Bdual{D}$ on the dual space $\Hdual{(L^p)}$
  with suitably defined domain.
  Calculate the representative of $\Bdual{D}$ under the duality
  isomorphism, notation $\Hdual{D}$ on $L^q$.
  It turns out that $\dom \Hdual{D}=W^{1,q}$;
  the inclusion '$\subset$' relies on \emph{elliptic regularity}.
  Forget the ambient space $L^q$ of the unbounded operator
  $\Hdual{D}$ and interpret $\Hdual{D}\in\Ll(W^{1,q},L^q)$.
  }
The semi-Fredholm estimate for $D$
is equally valid for $\Hdual{D}$, just carry out the
variable transformation $s\mapsto-s$ in the integral
appearing in the term $\norm{D\xi}_{L^p}$.
Hence $\Hdual{D}$ has finite dimensional kernel and closed range,
as well. The cokernel of $D$ is analyzed in
Appendix~\ref{sec:annihilators} where we introduce the
annihilator $\pann{(\im D)}$ and show that
there are two isometric isomorphisms of Banach
spaces, namely
\begin{equation}\label{eq:kerD*=cokerD}
\begin{aligned}
     \ker\Hdual{D}
     \simeq\pann{(\im D)}
     \simeq\Hdual{(\coker D)}
     :=\Ll(\coker D,\R).
\end{aligned}
\end{equation}
The first isomorphism is due to the identity~(\ref{eq:formal-adjoint-identity})
that characterizes $\Hdual{D}$, together with
restricting the duality isomorphism
$L^q\to\Hdual{(L^p)}$, $\eta\mapsto\inner{\eta}{\cdot}$
to $\ker\Hdual{D}$, see~(\ref{eq:iso-ker-coker}), and the second isomorphism
$\Xi$ is defined by~(\ref{eq:iso-annih-coker}). So
\[
     \dim\Hdual{(\coker D)}=\dim \ker \Hdual{D}<\infty,
\]
hence $\dim\coker D<\infty$. Thus the operator $D:W^{1,p}\to L^p$ is Fredholm.

Analogously the corresponding two isometric Banach space isomorphisms
\begin{equation}\label{eq:kerD=cokerD*}
\begin{aligned}
     \ker D
     \simeq\pann{(\im\Hdual{D})}
     \simeq\Hdual{(\coker\Hdual{D})}
\end{aligned}
\end{equation}
show that $\dim\coker\Hdual{D}<\infty$.
Thus $\Hdual{D}:W^{1,q}\to L^q$ is Fredholm, too.

\subsubsection{Fredholm index} 

Let's get back to the rectangle $[0,1]\times[-T,T]\ni(t,s)$ worth of
symplectic matrizes $\Psi_s(t)\in\Sp(2n)$ introduced prior
to~(\ref{eq:pm-infty-Psi}) and illustrated by Figure~\ref{fig:fig-SmSp}.
Let ${\color{red} \Gamma}$ denote the {\color{red} loop of symplectic matrizes}
obtained by cycling along the rectangle's boundary once.
In~\citesymptop{Robbin:1993a} Robbin and Salamon introduced a
Conley-Zehnder type index for rather general symplectic paths in the
sense that there are no restrictions on initial and endpoint; see
Section~\ref{sec:RS-index}.
Among the most useful features of the Robbin-Salamon index $\RS$
is that it is additive under concatenations of paths.
Furthermore, constant paths have index zero
and paths homotopic with fixed endpoints
share the same index. Moreover, the Robbin-Salamon index
coincides with the Conley Zehnder index $\CZ$ ($=-\CZcan$)
on the set $\SsPp^*(2n)$ of admissible paths;
see Section~\ref{sec:CZ}. Thus
\begin{equation}\label{eq:mu-RS-CZ}
     0=\RS(\Gamma)
     =\underbrace{\RS(\1)}_{0}+\underbrace{\RS(\Psi_{z^-})}_{-\CZcan(\Psi_{z^-})}
     +\RS\left(\Psi_{\cdot}(1)\right)
     +\underbrace{\RS(\Psi_{z^+}(1-\cdot))}_{\CZcan(\Psi_{z^+})}.
\end{equation}
So to conclude the proof of the index
formula~(\ref{eq:D_u-F_index-CZ}) it remains to identify the yet
anonymous term in the sum with the Fredholm index of $D_u$. 
We relate the unkown term in an intermediate step to another quantity
called spectral flow.

The Robbin-Salamon index counts with multiplicities
the intersections (called crossings) of a symplectic path
$s\mapsto \Psi(s)$ with the Maslov cycle $\Cc$
in the symplectic linear group $\Sp(2n)$.
Let us repeat the definition given
in~\citerefFH[\S\,2.4]{salamon:1999a} that starts from
$\Phi^\prime=J_0 S\Phi$.
Slightly perturbing the path, if necessary,
leads to finitely many \emph{regular crossings} $s_i$,
that is crossings at which the following quadratic form is
non-degenerate. Suppose that $\Psi$ has only regular crossings.
The multiplicity at a regular crossing $s$
is measured by the signature of the quadratic form
$$
     \Gamma(\Psi,s):\ker\left(\1-\Psi(s)\right)\to\R,\quad
     \zeta_0\mapsto\omega_0\left(\zeta_0,\Psi^\prime(s)\zeta_0\right)
     =\langle\zeta_0,S(s)\zeta_0\rangle_0,
$$
called the \emph{crossing form of the symplectic path}
at $s$. The Robbin-Salamon index, in case there are no crossings at
the endpoints, is the sum of signatures
$$
     \RS(\Psi):=\sum_{s_i}\sign\Gamma(\Psi,s_i)
$$
over all crossings. As long as endpoints are fixed,
the definition does not depend on perturbations, if any,
required to obtain regular crossings.

Recall that by~(\ref{eq:ker_A=Eig_1})
the domain of the crossing form $\Gamma(\Psi,s)$
is isomorphic to the kernel of the unbounded self-adjoint operator
$A(s)=-J_0\p_t-S(s)$ on the Hilbert space $L^2=L^2(\SS^1,\R^{2n})$.
Thus in the present context a crossing $s$
corresponds to $A(s)$ having non-trivial kernel.
The quadratic form
$$
     \Gamma(A,s):\ker A(s)\to\R,\quad
     \zeta\mapsto\langle\zeta,A^\prime(s)\zeta\rangle_{L^2},
$$
is called the \textbf{crossing form of the family of selfadjoint
operators} $A(s)$.\index{crossing form!spectral flow}
One\index{spectral flow!crossing form}
can show, see~\citerefFH[Le.~2.6]{salamon:1999a},
that the two crossing forms are isomorphic
under the natural isomorphism~(\ref{eq:ker_A=Eig_1}).
Thus the integer
$$
     \SF(A):=\sum_{s_i}\sign\Gamma(A,s_i)
     =\RS\left(s\mapsto \Psi_{s}(1)\right),
$$
called the \textbf{\Index{spectral flow} of the operator family} $A=\{A(s)\}$, fits
in~(\ref{eq:mu-RS-CZ}).\index{$\SF(A)$ spectral flow}

\begin{exercise}
a)~Check that the signature of the crossing form $\Gamma(A,s_i)$
at a regular crossing $s_i$ measures the number of
eigenvalues of $A(s)$, with multiplicities, that change from
negative to positive minus those changing from positive
to negative. So the total change $\SF(A)$
is a \textbf{\Index{relative Morse index}}. 
\index{Morse index!relative}
\newline
b)~Suppose the asymptotic operators $A^\mp$
have finitely many negative eigenvalues -- the number of
which, including multiplicities, is called the
\textbf{\Index{Morse index}} of $A^\mp$ and denoted by
$\IND(A^\mp)$. Check that $\SF(A)=\IND(A^-)-\IND(A^+)$.
\end{exercise}

To conclude the proof of the index formula~(\ref{eq:D_u-F_index-CZ})
we cite the result in~\citerefFH[Thm.~4.21]{robbin:1995a}
that the spectral flow of the family $A=\{A(s)\}$
is equal to the Fredholm index of the operator
$D=\frac{d}{ds}+A(s)$ given by~(\ref{eq:triv-D_u-lin-Floer}).\footnote{
  The sign conventions in~\citerefFH{robbin:1995a}
  differ from ours in two locations
  neutralizing each other.
  }

This concludes the (sketch of
the) proof of Theorem~\ref{thm:D_u-Fredholm}.

\newpage 
\subsection{Thom-Smale transversality theory}
\label{sec:Thom-Smale-transversality}
It\index{transversality!Thom-Smale --}
happens frequently that one is interested in the set, say $\Mm_V$, of
solutions $u$ to some differential equation which depends on some
parameter, say a function $V$. Bringing all terms of the differential equation to
the same, say left hand, side the problem takes on the form $\Ff_V(u)=0$.
So we have reformulated $\Mm_V$ as the zero set of a map
$\Ff_V:\Uu\to\Ee$ between suitable spaces.
Ideally one hopes that a) zero is a regular value of $\Ff_V$
and b) the linearization of $\Ff_V$ is a Fredholm operator, so that
a) the zero set $\Mm_V$ inherits the structure of a manifold
whose dimension b) is given by the Fredholm index.

In Section~\ref{sec:Thom-Smale-transversality} we will detail the
standard technique to achieve
this, called Thom-Smale transversality theory.\footnote{
  Wassermann~\citerefFH{Wasserman:1969a} established
  equivariant transversality theory in finite dimensions.
  Example: genericity of Morse-Bott functions;
  cf. Theorem~\ref{thm:RF-2nd-cat} in infinite dimensions.
  }
We illustrate the theory in Example~\ref{ex:generic-Morse} by working
out the proof of Theorem~\ref{thm:A-Morse} which, roughly speaking,
asserts that the symplectic action functional $\Aa_H$ is Morse for
generic Hamiltonian $H$. In the example we will utilize certain
Banach spaces and manifolds whose lengthy constructions
we postpone to an appendix at the end of the present
Section~\ref{sec:Thom-Smale-transversality}.

\subsubsection{Thom-Smale transversality -- abstract theory}
From now on and
throughout\index{$\mathrm{TS-}\pitchfork$ Thom-Smale transversality}
this section suppose $\ell\ge 1$ is an integer
and $\Uu$ and $\Vv$ are Banach manifolds
of class $C^\ell$ each of which admits a \emph{countable atlas}
and is modeled on a \emph{separable}\footnote{
  A topological space, e.g. a metric space, is called \textbf{separable}
  if it admits a dense\index{separable space}
  sequence.\index{topological space!separable}\index{metric space!separable}
  While non-open sub\emph{spaces} $S\subset X$ (subsets with induced
  structure) of separable topological spaces $X$ need not be separable
  (counter-example Sorgenfrey plane), separability of general sub\emph{spaces}
  $S\subset X$ holds if either $S$ is open or $X$ is a separable
  \emph{metric} space.
  }
Banach space.\footnote{
  Such topological spaces $\Uu$ and $\Vv$ are
  second countable (topology has countable base),
  cf. Remark~\ref{rem:sep-B-countable},
  a property stronger than separable and also called
  \textbf{completely separable}.\index{separable!completely --}
  }
Moreover, suppose that
\[
     \Ee\to\Uu\times\Vv
\]
is a Banach space bundle with a section $\Ff$,
both of class $C^\ell$. Let $\Ee_{(u,V)}$ denote the fiber over $(u,V)$.
It is convenient to introduce the notation
$$
     \Ff_V(u):=\Ff(u,V)=:\Ff_u(V),\quad
     \text{$u\in\Uu$, $V\in\Vv$}.
$$
Recall that the tangent bundle of a vector bundle
splits naturally along the zero section.
In what follows we use the notation $D$
to denote the \textbf{\Index{linearization of a section}}
at a point of the zero section: The (Fr\'{e}chet)
differential composed with projection onto the fiber.
For instance, at a zero $(u,V)\in\Ff^{-1}(0)$ the linearization is denoted by
$$
     D\Ff(u,V): T_{(u,V)}\left(\Uu\times\Vv\right)
     \simeq T_u\Uu\oplus T_V\Vv
     \to\Ee_{(u,V)}
$$
and defined by composing the differential
$$
     d\Ff(u,V):
     T_{(u,V)}\left(\Uu\times\Vv\right)
     \to T_{\Ff(u,V)}\Ee
     \simeq
     T_{(u,V)}\left(\Uu\times\Vv\right)
     \oplus\Ee_{(u,V)}
$$
with projection onto the second component $\Ee_{(u,V)}$.
Note that restriction of $\Ee$ to $\Uu\times\{V\}$ 
yields a bundle $\Ee^V\to\Uu$ of which $\Ff_V$ is a section
and with fibers $\Ee^V_u=\Ee_{(u,V)}$;
similarly $\Ee^u\to\Vv$ with section $\Ff_u$.\footnote{
  To make this precise, consider the inclusion
  $\iota_V:\Uu\to\Uu\times\Vv$, $u\mapsto (u,V)$, and denote
  the pull-back bundle $(\iota_V)^*\Ee\to\Uu$ by $\Ee^V$;
  analogously for $\Ee^u$.
  }
Note also that
\begin{equation}\label{eq:split-D}
\begin{split}
     D\Ff(u,V)=\overbrace{D\Ff_V(u)}^{=:D}\oplus
     \overbrace{D\Ff_u(V)}^{=:L} :
     T_u\Uu\oplus T_V\Vv
   &\to \Ee_{(u,V)}
     \\
     (\xi,v)
   &\mapsto D\xi+Lv
\end{split}
\end{equation}
at every zero $(u,V)$ of $\Ff$.

\begin{theorem}[{\Index{Thom-Smale transversality}}]
\label{thm:TS-transversality}
Let $\Ff$ be a $C^\ell$ section of $\Ee$ with
\begin{enumerate}
\item[{\rm\bf (F)}]
  $D=D\Ff_V(u): T_u\Uu\to\Ee^V_u$ is Fredholm
  and $\ell\ge\max\{1,1+\INDEX D\Ff_V(u)\}$
  for each $V\in\Vv$ and every $u\in\Ff_V^{-1}(0)$.
\item[{\rm\bf (S)}]
  $D\Ff(u,V) =D\oplus L: T_u\Uu\oplus T_V\Vv\to\Ee_{(u,V)}$ is surjective,
  $\forall(u,V)\in\Ff^{-1}(0)$.
\end{enumerate}
Then the subset of the \emph{parameter manifold} $\Vv$
given by\index{parameter manifold in TS-$\pitchfork$}
$$
     \Vvreg:=\{V\in\Vv\mid\text{$D\Ff_V(u)$ surjective
     $\forall u\in\Ff_V^{-1}(0)$}\}
$$
is residual, hence dense, in $\Vv$.
\end{theorem}

\begin{remark}\label{rem:cond-(A)}
In practice, instead of verifying the conditions (F) \& (S) for a
given section $\Ff$, it is often more convenient to verify the
Fredholm condition~(F) and the trivial-annihilator condition (A),
cf. Appendix~\ref{sec:annihilators}, namely:
\begin{enumerate}
\item[{\rm\bf (A)}]
  At every zero $(u,V)$ of the section $\Ff$
  the \textbf{\Index{annihilator}}
  \begin{equation*}
  \begin{split}
     \Ann_{(u,V)}
  :&=\pann{\left(\im D\Ff(u,V)\right)}\\
  :&=\{\eta\in\Ll(\Ee_{(u,V)},\R)\mid
     \text{$\eta(\xi)=0$ for every $\xi\in\im D\Ff(u,V)$}\}\\
   &=\{0\}
  \end{split}
  \end{equation*}
  of the image of the linearization $D\Ff(u,V)$ is trivial.
\end{enumerate}
\end{remark}

\begin{lemma}\label{le:cond-(A)}
If $\Ff$ satisfies (F) \& (A), then it satisfies (F) \& (S).
\end{lemma}

\begin{proof}
Suppose $\Ff(u,V)=0$.
It is a consequence of the geometric form of the Hahn-Banach
theorem, see e.g.~\cite[Thm.\,1.7 and Cor.\,1.8]{brezis:2011a},
that triviality of the annihilator $\pann{(\im D\Ff(u,V))}$
implies density of the image of $D\Ff(u,V)$. (Reversely, density implies
triviality by the extension by continuity principle, Exercise~\ref{exc:ECP}.)
But the image is also closed by Exercise~\ref{exc:gen-surj-Fred}~i).
\end{proof}

The proof of Theorem~\ref{thm:TS-transversality} rests on two
pillars. Firstly, the slightly miraculous equality of the set
$\Vvreg$ of parameters for which the Fredholm operator
$D=D\Ff_V(u)$ is surjective for all $u\in\Ff_V^{-1}(0)$
and the set of regular values of the projection
$\pi:\Uu\times\Vv\supset\Ff^{-1}(0)\to\Vv$
onto the second component.
Secondly, the generalization of Sard's theorem
to infinite dimensions, due to Smale~\citerefFH{smale:1965a}.
Applied locally to $\pi$, Sard-Smale yields the assertion of
Theorem~\ref{thm:TS-transversality}.

\begin{theorem}[Sard-Smale]\label{thm:Sard-Smale}
Suppose\index{Sard-Smale theorem}\index{theorem!Sard-Smale --}
$W$ is a separable Banach space and
$U\subset W$ an open connected\,\footnote{
  If $U$ is not connected, simply impose one condition
  $\ell\ge 1+\INDEX (f)$ for each component.
  }
subset. Let $Z$ be a Banach space.
If $f:U\to Z$ is a Fredholm map of class $C^\ell$ with
$\ell\ge 1+\INDEX (f)$,\footnote{
  Recall that being a Fredholm map already requires $\ell\ge 1$.
  }
then the set
$$
     \Zreg:=\{z\in Z\mid \text{$\im df(x)=Z$ for every $x\in U$
     with $f(x)=z$}\},
$$
of regular values is a residual subset of $Z$, thus dense, hence non-empty.\footnote{
  However, note that $\Zreg$ contains any element $z\in Z$ with empty
  pre-image $f^{-1}(z)=\emptyset$.
  }
\end{theorem}

\begin{remark}\label{rem:sep-B-countable}
a) The Fredholm condition is necessary; cf.~\citerefFH{smale:1965a}.

b) Sometimes separability of the target Banach space $Z$ is added to
the assumptions, see e.g.~\citerefFH{smale:1965a} and~\cite{mcduff:2004a},
although seemingly it is not used in the proofs.
References without that assumption
include~\cite{Abraham:1967a,Chang:2005a}.

c) Separability of the domain Banach space $W$
is essential, because the proof uses that separable metric spaces
are \textbf{\Index{Lindel\"of}}
(\emph{every open cover admits a countable subcover};
see e.g.~\cite[Prop.~A.5.2]{mcduff:2004a}).\footnote{
  In metric spaces second-countable, separable, and Lindel\"of
  are three equivalent properties.
  }
\\
To prove the theorem one shows that for every point $p\in W$ there is
a closed neighborhood $A_p$ which provides a dense open subset
$\Zreg(A_p)\subset Z$. One uses Lindel\"of to extract a countable
subcover $\{A_i\}$ of the collection of all $A_p$'s. Then one shows
that $\Zreg=\cap_i\Zreg(A_i)$; well explained in~\cite{Abraham:1967a}.
By countability of the intersection the Baire Theorem~\ref{thm:Baire} applies.

d) Banach \emph{manifolds}:
Theorem~\ref{thm:Sard-Smale} extends to Banach manifolds
$W$ and $Z$ by applying the Banach space version to local
coordinate representations of $f$. Here one adds the assumptions
that $W$ is modeled on a separable Banach space and that both
$W$ and $Z$ can be covered by countably many coordinate charts.
In this case $\Zreg$ can be written as a \emph{countable} intersection of the
dense open sets obtained from the coordinate representations; cf. Step III below.
\end{remark}


\vspace{.05cm}\noindent
\textbf{\boldmath Proof of Theorem~\ref{thm:TS-transversality}
  (Thom-Smale transversality).}
Suppose $\Ff$ is a $C^\ell$ section with $\ell\ge 1$
of the Banach space bundle
$\Ee\to\Uu\times\Vv$ and it satisfies conditions
(F) \& (S). Let us call $\Ff$ a
\textbf{\Index{universal section}}, just to emphasize
that $\Ff_V$ and $\Ff_u$ are its restrictions.
The proof takes four steps I--IV.

\vspace{.1cm}
\textsc{Step~I.} Suppose $\Ff(u,V)=0$ and recall the notation
$D\Ff(u,V)=D\oplus L$ from~(\ref{eq:split-D}) where
$D=D\Ff_V(u)$ is Fredholm from $X=T_u\Uu$ to $Y=\Ee_{(u,V)}$
by assumption (F) and $L=D\Ff_u(V)$ is bounded from $Z=T_V\Vv$
to $Y$ due to $\ell\ge 1$. As $D\Ff(u,V)$ is surjective by assumption (S),
Exercise~\ref{exc:gen-surj-Fred}~ii) asserts that
\begin{itemize}
\item[a)]
  $\ker D\Ff(u,V)=\ker(D\oplus L)$ admits a complement,
  thus a right inverse;
\item[b)]
  projection onto the second component
  $$
     \Pi=\Pi_{(u,V)}:
     T_u\Uu\oplus T_V\Vv\supset
     \ker D\Ff(u,V)=\ker(D\oplus L)\to T_V\Vv
  $$
  is a Fredhom operator with $\INDEX(\Pi)=\INDEX(D)$.
\end{itemize}

\textsc{Step~II.} Zero is a regular value of $\Ff$ by a). Thus by the
regular value theorem, see Remark~\ref{rem:nonlin-Fred}, 
the \textbf{\Index{universal moduli space}},
namely the zero set 
$$
     \Mm:=\Ff^{-1}(0)\subset\Uu\times\Vv
$$
of the universal section $\Ff$, is a $C^\ell$ Banach manifold
that admits a countable atlas. Moreover, each component of $\Mm$
is modeled on a separable Banach space,\footnote{
  A finite sum, say $T_u\Uu\oplus T_V\Vv$, of separable Banach spaces
  is a separable Banach space and so is any subspace,
  say $T_{(u,V)}\Mm=:W$.
  }
say $W$, and the tangent spaces are given by
$$
     T_{(u,V)}\Mm
     =\ker D\Ff(u,V)
     =\{(\xi,v)\mid D\xi+Lv=0\}
     \subset T_u\Uu\oplus T_V\Vv.
$$
Since $\Mm$ is a $C^\ell$ manifold,
projection to the second component
\begin{equation}\label{eq:pi}
     \pi:\Mm\to\Vv,\quad
     (u,V)\mapsto V,
\end{equation}
is a map of class $C^\ell$. But the linearization of this
map at $(u,V)$ is precisely the Fredholm operator $\Pi_{(u,V)}$.
Thus $\pi$ is a $C^\ell$ Fredholm map whose Fredholm index 
along the \textbf{component} $\Mm^{(u,V)}$ of $(u,V)$ is equal to the
index of $D\Ff_V(u)$.

\vspace{.1cm}
\textsc{Step~III.} For each of the countably
many local coordinate charts $\varphi:\Mm\supset C\to W$
of $\Mm$ and $\psi:\Vv\supset B_\psi\to Z$ of $\Vv$ the
coordinate representative\footnote{
  Here $W$ and $Z$ are the separable Banach spaces on which $\Mm$
  and $\Vv$ are modeled.
  }
$$
     \pi_{\varphi,\psi}
     :=\psi\circ\pi\circ\varphi^{-1}: W\supset\varphi(C)\to Z
$$
of the $C^\ell$ Fredholm map $\pi$ satisfies the
assumptions of the Sard-Smale Theorem~\ref{thm:Sard-Smale}.
Therefore the set $\Rr(\pi_{\varphi,\psi})$ of regular values of
$\pi_{\varphi,\psi}$ is residual in $Z$. Hence $\psi^{-1}\Rr(\pi_{\varphi,\psi})$
is residual in the open subset $B_\psi\subset\Vv$ since $\psi$ is a
homeomorphism onto its image. Thus the union
$A_{\varphi,\psi}:=\psi^{-1}\Rr(\pi_{\varphi,\psi})\cup\comp{B_\psi}$
with the complement of $B_\psi$ is residual in $\Vv$ and represents
the regular values of the map $\pi\circ\varphi^{-1}: W\supset\varphi(C)\to\Vv$.
(Recall that elements outside an image are regular values automatically.)
The regular values of $\pi$ are those elements of $\Vv$
which are regular for all of these maps, that is $\Rr(\pi)
=\cap_{\varphi,\psi} A_{\varphi,\psi}=\cap_{\varphi,\psi}\psi^{-1}\Rr(\pi_{\varphi,\psi})$.
But a countable intersection of residuals $A_{\varphi,\psi}$ is residual by
Exercise~\ref{exc:Baire-exc}.
Thus $\Rr(\pi)$ is residual in the Banach manifold $\Vv$ which --
being locally modeled on a Banach space -- is a Baire space.
Hence $\Rr(\pi)$ is dense in $\Vv$.

\vspace{.1cm}
\textsc{Step~IV.} Theorem~\ref{thm:TS-transversality}
is now reduced to the following
Wonder-Lemma:\index{Wonder-Lemma}\index{lemma!wonder-}

\begin{lemma}\label{le:wonder-lemma}
For $\Ff$ as in Theorem~\ref{thm:TS-transversality}
the map $\pi:\Mm\to\Vv$ is defined~and
\begin{equation*}
\begin{split}
     \Rr(\pi):
   &=\{\text{regular values of $\pi$}\}\\
   &=\{V\in\Vv\mid \text{$D\Ff_V(u)$
     surjective $\forall u\in\Ff_V^{-1}(0)$}\}\\
   &=:\Vvreg.
\end{split}
\end{equation*}
\end{lemma}

\begin{proof}
Suppose $\Ff(u,V)=0$. As the linearization $d\pi(u,V)=\Pi_{(u,V)}$
is Fredholm, surjectivity implies existence of a right inverse.
Thus it suffices to show
$$
     \text{$d\pi(u,V)$ surjective}\quad\Longleftrightarrow\quad
     \text{$D:=D\Ff_V(u)$ surjective}.
$$
Let us use the notation $D\Ff(u,V)=D\oplus L$ with $L:=D\Ff_u(V)$;
see~(\ref{eq:split-D}).
\newline
'$\Rightarrow$'
Pick $\eta\in\Ee_{(u,V)}$, then $\eta=D\Ff(u,V)(\xi_0,v)=D\xi_0+Lv$
for some tangent vector $(\xi_0,v)\in T_{(u,V)}(\Uu\times\Vv)$
by the surjectivity assumption (S).
Given $v$, then $v=d\pi(u,V)(\xi_1,v)$
by surjectivity of $d\pi(u,V)$ for some element
$(\xi_1,v)\in T_{(u,V)}\Mm=\ker D\Ff(u,V)$, that is
$D\xi_1+Lv=0$. Hence $D(\xi_0-\xi_1)=\eta$.
\newline
'$\Leftarrow$'
Pick $v\in T_V\Vv$, then $Lv=D\xi$ for some
$\xi\in T_u\Uu$ by surjectivity of $D$.
Thus the pair $(-\xi,v)$ lies in $T_{(u,V)}\Mm$.
It gets mapped to $v$ under $d\pi(u,V)$.
\end{proof}

\subsubsection{Thom-Smale transversality -- proper implies open}

\begin{theorem}\label{thm:TS-transversality-openness}
\index{Thom-Smale transversality!openness}
Consider a $C^\ell$ section $\Ff:\Uu\times\Vv\to\Ee$
satisfying the assumptions of Theorem~\ref{thm:TS-transversality}.
Suppose, in addition, that the restriction $\pi_a$ of the projection
$\pi:\Mm\to\Vv$, see~(\ref{eq:pi}), to some open subset
$\Mm^a\subset\Mm=\Ff^{-1}(0)$ is a proper map.\footnote{
  \textbf{\Index{Proper map}}: Pre-images of compact sets are compact.
  Actually the proof only uses that pre-images under $\pi_a$
  of convergent sequences in $\Vv$, limit included, are compact.
  }
Then the set of regular parameters
$$
     \Vvreg^a:=\{V\in\Vv\mid\text{$D\Ff_V(u)$ is surjective
     whenever $(u,V)\in \Mm^a$}\}
$$
is not only dense but also open in the Banach manifold $\Vv$ of all parameters.
\end{theorem}

\begin{proof}
The set $\Vvreg^a$ is dense in $\Vv$, as by
Theorem~\ref{thm:TS-transversality}~already~its~subset~$\Vvreg$~is.

To prove openness pick $V\in \Vvreg^a$ and
assume by contradiction that there is no open neighborhood
of $V$ in $\Vv$ contained in $\Vvreg^a$.
Then there is a sequence $V_i\in\Vv$ converging to $V$
and a sequence $u_i\in\Uu$ such that every pair $(u_i,V_i)$ lies in
$\Mm^a$ and the corresponding Fredholm operator
$D_i:=D\Ff_{V_i}(u_i)$ is not surjective.
Consider the compact set $K:=\{V_i\}_i\cup\{V\}$ that consists of
the convergent sequence together with its limit. The pre-image
$\pi_a^{-1}(K)\subset\Mm^a$ is compact by assumption
and it contains the sequence $(u_i,V_i)$.
Thus there is a convergent subsequence, still denoted by $(u_i,V_i)$,
with limit $(u,V)$ for some $u\in\Uu$. By continuity of the section
$\Ff$ the limit $(u,V)$ of the zeroes $(u_i,V_i)$ also lies in the
zero set $\Mm=\Ff^{-1}(0)$.
So the operator $D:=D\Ff_V(u)$ is Fredholm by assumption~(F)
and it is surjective by our choice $V\in\Vvreg^a$.
\newline
Recall from Exercise~\ref{exc:Fred-stability}\,a) that
surjectivity of $D$ is an open property with respect to the
operator norm. On the other hand, the assumption
$\Ff\in C^1$ implies that the partial derivatives
$D_i$ converge to $D$ in the operator norm; use local
coordinates about $(u,V)$. None of the $D_i$ is surjective. Contradiction.
\end{proof}

Given the nice assertion of the theorem, the proof is elementary.
This indicates that the hypothesis might be non-trivial
to verify in applications.

\subsubsection{Morse functions on a closed manifold
form an open and dense set}
\begin{exercise}
Given a closed manifold $Q$, show that the set
$\Vvreg^2$ of $C^2$ Morse functions on $Q$ is open and dense in the Banach
space $\Vv:=C^2(Q)$.
\newline
[Hint: How about the manifold $\Uu\cong Q$ that consists of constant
maps $\gamma:\R\to Q$, $\gamma(t)\equiv q$ for some $q\in Q$, and the
map $\Ff_f(\gamma):=\dot\gamma+\nabla f(\gamma)=\nabla f(\gamma)$
taking values in $\Ee\cong TQ$?
Here $\nabla f$ is the gradient of $f$ with respect to any fixed
auxiliary Riemannian metric $g$ on $Q$.]
\end{exercise}

Sharpen the exercise by showing that, given a $C^2$-function
$f_0:Q\to\R$ and an open neighborhood $U$ of $\Crit f$ in $Q$,
then $f_0+f$ is Morse for some $C^2$-small function $f$ supported in $U$.

\subsubsection{Symplectic action $\mbf{\Aa_H}$ is Morse for
  generic Hamiltonian $\mbf{H}$}
\begin{example}[Proof of Theorem~\ref{thm:A-Morse}]\label{ex:generic-Morse}
Let $(M,\omega)$ be a closed symplectic manifold.
Pick a compatible almost complex structure $J$, not depending
on time $t$, and let $g=g_J$ be the associated Riemannian metric
and $\nabla$ the Levi-Civita connection.
Recall that \textbf{\Index{smooth}} means $C^\infty$ smooth.
Let $\Hhreg$ be the set of all smooth Hamiltonians
$H:\SS^1\times M\to\R$ for which the symplectic action functional 
$\Aa_H:\Ll_0 M\to\R$, given by~(\ref{eq:action-A}) on the space of
smooth contractible loops $z:\SS^1\to M$, is Morse. Our goal is to
show that $\Hhreg$ is a dense open subset of the complete metric space 
\begin{equation}\label{eq:complete-45}
     \Hh:=\left(C^\infty(\SS^1\times M),d\right).
\end{equation}
Here the metric $d$ is given by~(\ref{eq:Cinfty-d}) and it is complete.\footnote{
  To see that $d$ is a metric carry over the arguments
  in~\cite[\S 1.1]{Infusino:2016a}.
  Completeness of $d$ follows from the fact that convergence with respect to $d$
  means convergence with respect to every $C^k$ sup-norm
  $p_k:=\norm{\cdot}_{C^k}:C^k(\SS^1\times M)\to[0,\infty)$
  which includes partial derivatives up to order $k$.
  But $p_k$ is complete since the domain $\SS^1\times M$
  is, firstly, compact by assumption and, secondly, Hausdorff by
  definition of a manifold; see e.g.
  Examples\,1.1.6 and\,1.1.10 in~\cite{Zimmer:1990a}.
  Now apply Lemma~1.1.14 in~\cite{Zimmer:1990a}.
  }

Thom-Smale transversality requires Banach manifolds, so the $C^\infty$
category is not suitable. As $\Uu$ pick the completion
$
     \Lambda M
$
of the space $\Ll_0M:=C^\infty_{\rm contr}(\SS^1,M)$
with respect to the Sobolev $W^{1,2}$ norm;
cf.~Proposition~\ref{prop:B-mf}.\footnote{
  $\Uu$ is the Hilbert manifold of absolutely continuous
  contractible maps $z:\SS^1\to M$ with square integrable derivative $\dot z$.
  The tangent space $T_z\Uu=W^{1,2}(\SS^1,z^*TM)=:W_z^{1,2}$ consists of
  absolutely continuous vector fields $\zeta$ along $z$ 
  with square integrable derivative $\Nabla{t}\zeta$.
  }
As parameter manifold $\Vv$ of Hamiltonians $H$
we choose the separable Banach space\footnote{
  Cf. Proposition~\ref{prop:complete-separable}; compactness of
  $M$ is crucial for completeness of $\norm{\cdot}_{C^k}$.
  }
\begin{equation}\label{eq:Vv-45}
     \Hh^k:=(C^k(\SS^1\times M),\norm{\cdot}_{C^k}),\quad k\ge 2,
\end{equation}
and as $C^1$ section the $L^2$ gradient $\Ff=\grad\Aa:
\Lambda M\times\Hh^k\to\Ee$. Let $\Hhreg^k\subset\Hh^k$
be the subset whose elements $H$ make $\Aa_H:\Lambda M\to\R$ Morse.

Observe that $\Hhreg^k$ coincides with the set of regular values
of the section $\Ff_H$ that appears in the Thom-Smale transversality
Theorem~\ref{thm:TS-transversality}. To see this note that
$\Crit\Aa_H=\Ff_H^{-1}(0)$ since $\Ff_H=\grad\Aa_H$ and since
the zeros $z\in\Lambda M$ of $\Ff_H$, although
apriori only of class $W^{1,2}$, automatically inherit
$C^k$-smoothness from $H$ by Lemma~\ref{le:regularity-crit}.
Recall that, by Exercise~\ref{exc:non-deg-Crit-pp},
non-degeneracy of $z$ as a critical point of $\Aa_H$
is equivalent to surjectivity of the linear operator
$A_z$ that represents the Hessian of $\Aa_H$ at $z$.
But $A_z=D\Ff_H(z)$ since $\Ff_H=\grad\Aa_H$. 

To summarize $\Hhreg=\cap_k\Hhreg^k$ and both definitions
\begin{equation*}
\begin{split}
     \Hhreg^k
   &:=\{H\in\Hh^k\mid\text{$\Aa_H$ Morse}\}\\
   &=\{H\in\Hh^k\mid\text{$D\Ff_H(z)$ surjective
     $\forall z\in\Ff_H^{-1}(0)$}\}\\
   &=:\Hhreg^k
\end{split}
\end{equation*}
agree.
Theorem~\ref{thm:TS-transversality} will ultimately
prove, see Step III, that
\begin{equation}\label{eq:backdoor}
     \emph{
     the set $\Hhreg^{k}$ is {\color{cyan} dense}
     and {\color{red} open} in $\Hh^k$ for $k\ge 2$.
     }
\end{equation}
Suppose that~(\ref{eq:backdoor}) has already been established.
Then we can approach the $C^\infty$ smooth scenario $\Hh$ through the
back door, namely, by approximation:

\textbf{\boldmath Density of $\Hhreg$ in $\Hh$.}
Given a smooth Hamiltonian $H\in\Hh$, viewed as being of class $C^k$, let
us exploit {\color{cyan} density}, in order to $C^k$-approximate $H$ by some
${\color{cyan} h_k}\in\Hhreg^{k}$ as illustrated by
Figure~\ref{fig:fig-TS-approx}.
\begin{figure}[h]
  \centering
  \includegraphics
                             [height=4.5cm]
                             {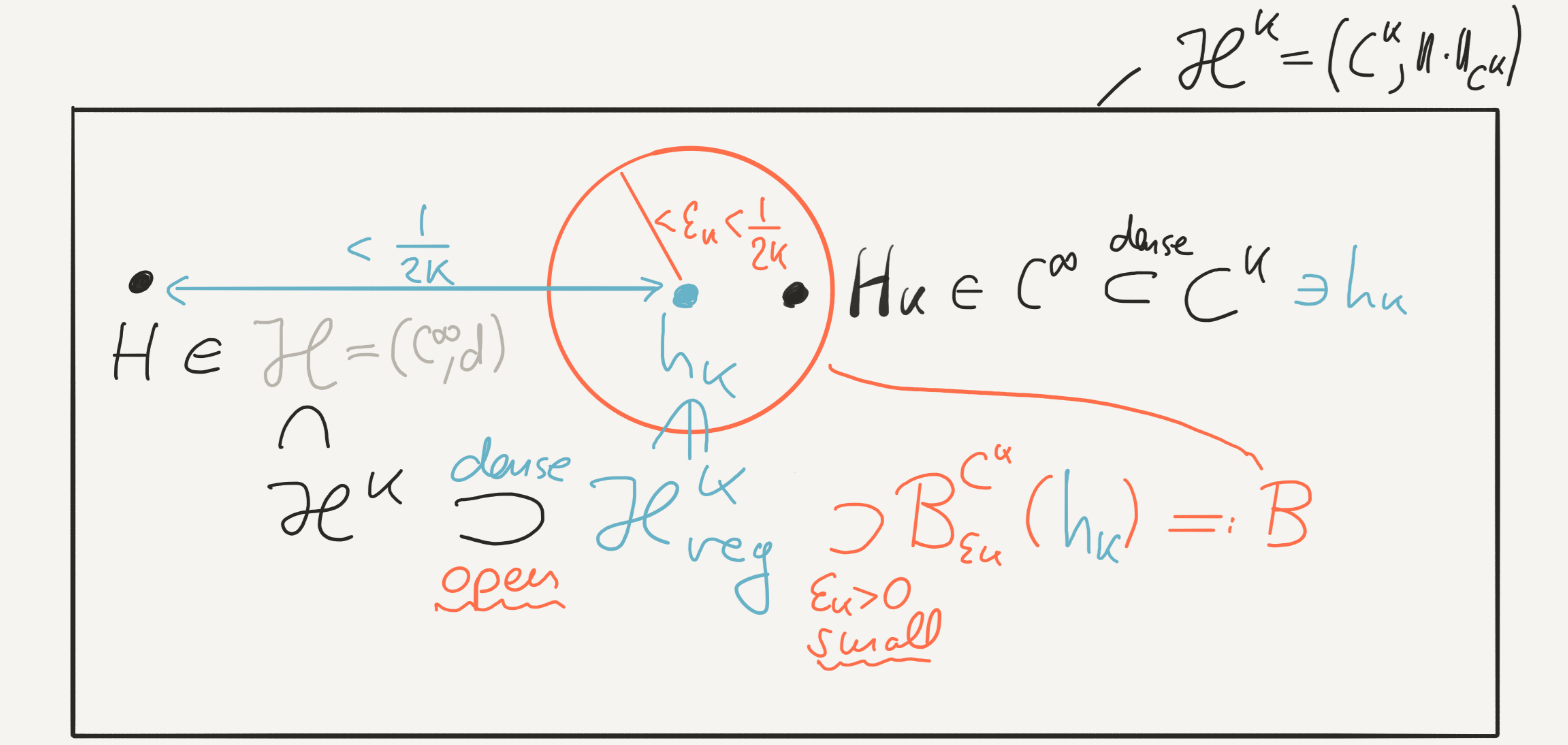}
  \caption{Any $H\in\Hh$ is approximated by a sequence
                $H_k\in\Hhreg=C^\infty\cap\Hhreg^{k}$}
  \label{fig:fig-TS-approx}
\end{figure}
By {\color{red} openness} $\Hhreg^{k}$
contains a whole open $C^k$-ball ${\color{red} B}$ about $h_k$.
Now $C^k$-approximate $h_k$ by a smooth Hamiltonian
${\color{red} H_k}\in B$ using the standard fact that $C^\infty$-functions
are dense in the Banach space of $C^k$-functions,
see e.g.~\cite[Thm.\,2.6]{hirsch:1976a}, hence in the open
subset $B$.
Choosing in each of the two approximation steps a $C^k$-distance less than
$\frac{1}{2k}$ we get that $\norm{H-H_k}_{C^k}<\frac{1}{k}$ for every $k\ge 2$;
set $H_0:=H_1:=H_2$.
At first sight somewhat amazingly, this is already enough to obtain
convergence $d(H,H_k)\to 0$, despite not
having any control whatsoever on the higher $C^\nu$-norms of
$H-H_k$ that appear in the infinite sum $d$. 
However, a second look reveals that the weights $1/2^\nu$ take care of
the higher $C^\nu$-norms:
\[
     d(H,H_k)
     =\underbrace{\sum_{\nu=0}^{\nu_\eps}\frac{1}{2^\nu}}_{\le 2}
     \underbrace{\frac{\norm{H-H_k}_{C^\nu}}{1+\norm{H-H_k}_{C^\nu}}}
          _{\text{$<\frac{1}{k}$ whenever $\nu\le k$}}
     +\underbrace{\sum_{\nu=\nu_\eps+1}^\infty\frac{1}{2^\nu}}_{=:\chi(\nu_\eps)}\,
     \underbrace{\frac{\norm{H-H_k}_{C^\nu}}{1+\norm{H-H_k}_{C^\nu}}}_{\le 1}.
\]
Given $\eps>0$, choose the integer $\nu_\eps$ sufficiently large
such that $\chi(\nu_\eps)<\frac{\eps}{2}$, then for every integer
$k>k_0:=\max\{\nu_\eps,\frac{4}{\eps}\}$ it holds that
$d(H,H_k)<\frac{2}{k}+\frac{\eps}{2}<\eps$.
Since each $H_k$ lies in $\Hhreg^{k}\cap C^\infty=:\Hhreg$ this proves
density in $\Hh$ of the set $\Hhreg$ of \emph{smooth} regular Hamiltonians.

\textbf{\boldmath Openness of $\Hhreg$ in $\Hh$.}
Pick $H_*\in\Hhreg\subset\Hhreg^2\:{\color{red} \subset}\:\Hh^2$.
Then by {\color{red} openness}~(\ref{eq:backdoor}) for $k=2$ there is
a small radius $\delta>0$ such that the whole open $C^2$ ball
$B_\delta^{C^2}(H_*)$ about $H_*$ lies in $\Hhreg^2$.
Set $\eps_*:=\frac{1}{4}\frac{\delta}{1+\delta}$ and {\color{blue} pick}
${\color{blue} H} \in\Hh$ with
$$
     {\color{blue}\eps_* >}\:\,
     d(H,H_*)
     =\sum_{\nu=0}^{\infty}\frac{1}{2^\nu}
     \frac{\norm{H-H_*}_{C^\nu}}{1+\norm{H-H_*}_{C^\nu}}
     >\frac{1}{4} \cdot \frac{\norm{H-H_*}_{C^2}}{1+\norm{H-H_*}_{C^2}}.
$$
But this inequality implies that $\norm{H-H_*}_{C^2}<\delta$ and therefore
$$
     {\color{blue} H} \in\Bigl(B_{\delta}^{C^2}(H_*)\cap\Hh\Bigr)
     \subset\Bigl(\Hhreg^2\cap C^\infty\Bigr)
     =\Hhreg.
$$
Thus the whole open ball ${\color{blue} B_{\eps_*}^{d}(H_*)}$ is contained
in $\Hhreg$. So $\Hhreg\subset\Hh$ is open.

Let us now prove~(\ref{eq:backdoor}), hence
Theorem~\ref{thm:A-Morse}, in three steps I-III.

\vspace{.05cm}\noindent
\textbf{I. Setup and Fredholm condition~(F).}
Pick an integer $k\ge 2$. We apply Thom-Smale transversality theory to
these choices: The Banach manifolds
\[
     \Uu:=\Lambda M=W^{1,2}_{\rm contr}(\SS^1,M),\qquad
     \Vv:=\Hh^k:=(C^k(\SS^1\times M),\norm{\cdot}_{C^k}),
\]
each admits a countable atlas and is modeled on a separable
Banach space, see Appendix~\ref{sec:B-spaces-B-mfs},
and the $C^1$ section $\Ff:\Lambda M\times\Hh^k\to\Ee$
given by the $L^2$ gradient
\[
     \Ff(z,H):=\grad\Aa_H(z)=-J(z)\dot z-\nabla H_t(z)
\]
of the symplectic action functional $\Aa_H:\Lambda M\to\R$ defined
by~(\ref{eq:action-A}).
Here
$$
     \Ee\to\Lambda M\times\Hh^k,\qquad
     \Ee_{(z,H)}=L^2(\SS^1,z^*TM)=:L^2_z,
$$
is the Hilbert bundle whose fiber over $(z,H)$ consists of the square
integrable vector fields along the loop $z$ in $M$; fibers do not
depend on $H$.
Bundle setup and the calculation of the $L^2$ gradient
have been carried out in Section~\ref{sec:FH-gradient}.

Observe that $H\mapsto\Ff(z,H)$ is linear, so together with our
previous calculation in Section~\ref{sec:Hess-A} of the Hessian $A_z$
of $\Aa_H$ at a critical point $z$, see~(\ref{eq:Hess-A}), we observe
that the linearization of $\Ff$ at a zero
\[
     (z,H)\in\Mm:=\Ff^{-1}(0)
     =\bigcup_{H\in\Hh}\underbrace{\Crit\Aa_H\times\{H\}}_{=\Ff_H^{-1}(0)}
     \subset\Lambda M\times\Hh^k
\]
is given by
\begin{equation}
\begin{split}
     D\Ff(u,H)=D\oplus L:W^{1,2}_z\oplus \Hh^k&\to L^2_z
     \\
     (\zeta,h)&\mapsto
     A_z\zeta-\nabla h_t(z)
\end{split}
\end{equation}
where $D\zeta:=D\Ff_H(z)\zeta=A_z\zeta$ is given by~(\ref{eq:Hess-A})
and where
\[
     L:=D\Ff_z(H):\Hh^k\to L^2_z,\quad
     h\mapsto -\nabla h_t(z).
\]
Note that
$$
     \norm{Lh}_{L^2_z} =\norm{\nabla h_t(z)}_{L^2_z}
     \le\norm{h}_{C^1(\SS^1\times M)}\le\norm{h}_{\Hh^k},
$$
so $\norm{L}_{\Ll(\Hh^k,L^2_z)}\le 1$.
By Exercise~\ref{exc:symm-Hess},
as an unbounded operator on $L^2_z$ with dense domain $W^{1,2}_z$,
the Hessian $A_z$ has a finite dimensional
kernel which, by self-adjointness, coincides with the cokernel.
As an operator from $W^{1,2}_z$ to $L^2_z$ the Hessian
is bounded, notation $D$, thus a Fredholm operator with $\INDEX D=0$.

Hence assumption (F) in Theorem~\ref{thm:TS-transversality} is
satisfied once $\Ff$ is of class~$C^1$.
But this is readily verified by estimating,
based on $k\ge 2$ and after choosing suitable trivializations, firstly,
the $L^2$ norm of the difference of $\nabla h_t(z)$ and
$\nabla h_t(\tilde z)$ in terms of $\norm{h}_{C^2}$
and, secondly, the $L^2$ norm of the difference of
$D\Ff_H(z)$ and $D\Ff_{\tilde H}(\tilde z)$ applied to $\zeta$.
The second estimate should be in terms of $\norm{\zeta}_{1,2}$.
Continuity of the two partial derivatives
then implies that the derivative of $\Ff$ exists
at each point and varies continuously itself
(in~\cite[Ch.\,1]{ambrosetti:1993a} combine Thm.\,4.3 with Prop.\,4.2).

\vspace{.05cm}\noindent
\textbf{\boldmath II. Transversality condition~(S).}
By Lemma~\ref{le:cond-(A)} it suffices to verify condition~(A):
At every zero $(z,H)$ of $\Ff$ the annihilator
$\Ann_{(z,h)}\subset (L^2_z)^*\simeq L^2_z$ of the image of
$D\Ff(z,H):W^{1,2}_z\oplus\Hh^k\to L^2_z$ is the trivial vector space.
In other words, given $\eta\in L^2_z$, we have to show that the
two assumptions
\begin{equation*}
\begin{cases}
     \inner{\eta}{D\zeta}_{L^2_z}=0&\text{, $\forall\zeta\in W^{1,2}_z$,}\\
     \inner{\eta}{L h}_{L^2_z}=0&\text{, $\forall h\in\Hh^k$,}
\end{cases}
\end{equation*}
together imply that $\eta=0$. As usual, the assumption involving
the Fredholm operator will improve the regularity of $\eta$
which one then exploits to invalidate the other assumption
whenever $\eta\not=0$ is not the trivial vector field.

As $D\zeta=A_z\zeta$, assumption one means that
${A_z}^*\zeta=0$, thus by self-adjointness
$\zeta\in\ker {A_z}^*=\ker A_z\subset W^{1,2}_z$.
Hence the $L^2$ vector field $\eta$ is in fact continuous by the
Sobolev embedding theorem, see e.g.~\cite[Thm.\,B.1.11]{mcduff:2004a}.
From now on suppose by contradiction that
\[
     \eta(t_*)\not=0,\quad\text{for some $t_*\in\SS^1=\R/\Z$.}
\]
Below we are going to construct an element
$h\in\Hh^k$ which has the property that
$\inner{\eta}{L h}_{L^2_z}=\inner{\eta}{\nabla h_t(z)}_{L^2_z}>0$. 
Contradiction.

The construction of $h$ will be local. Let $\iota$ be the injectivity
radius of $(M,g)$. We utilize coordinates about the point
$z_*=z(t_*)\in M$ provided by the exponential map
$\exp:T_{z_*M}\supset B_\iota(0)\to M$ and modeled on the
radius $\iota$ ball $B_\iota(0)\subset T_{z_*}M$. The tangent space
comes with the inner product $g_{z_*}$ denoted by $\inner{\cdot}{\cdot}$.
The piece of the loop $z$ inside the coordinate patch determines a
curve $Z(t)$ in $T_{z_*}M$ via the identity $z(t)=\exp_{z_*} Z(t)$.
\begin{itemize}
\item
  As $Z(t_*)=0$ and $z$ is continuous, there is a time interval $I_1$
  about $t_*$ with
  \[
     \Abs{Z(t)}\le\iota/2,\qquad
     \forall t\in I_1=[t_*-\delta_1,t_*+\delta_1].
  \]
\item
  As $\eta(t_*)\not=0$ and $\eta$ is continuous,\footnote{
    Let $E_2(z_*,\xi):=d(\exp_{z_*})_\xi$ be the derivative of $\exp_{z_*}$ at $\xi$.
    As $E_2(z_*,0)=d(\exp_{z_*})_0=\1$ is invertible, also $E_2(z_*,\xi)$
    remains invertible for all $\xi$ near the origin.
  }
  there is a time
  interval $I_2$ about $t_*$~with
  \[
     \INNER{E_2(z_*,Z(t))^{-1}\eta(t)}{\eta(t_*)}>0,\qquad
     \forall t\in I_2=[t_*-\delta_2,t_*+\delta_2].
  \]
\item
  With $\delta:=\min\{\frac12,\delta_1,\delta_2\}$ pick a cutoff function
  $\gamma\in C^\infty_0((t_*-\delta,t_*+\delta),[0,1])$
  as in Figure~\ref{fig:fig-cutoff-gamma}. Denote the $1$-periodic
  extension of $\gamma$ to $\R$ again by $\gamma$.\footnote{
    Viewing $\eta$ as a map on $\R$ such that $\eta(t+1)=\eta(t)$
    $\forall t\in\R$ we may suppose that $t_*\in[0,1)$.
    }
\item
  Pick a cutoff function $\beta\in C^\infty_0((-\iota,\iota),[0,1])$
  as in Figure~\ref{fig:fig-cutoff-beta}.
\item
  Define a smooth function $h:\R\times M\to\R$ with $h_{t+1}\equiv
  h_t$ for all $t$ by
  \[
     h_t(p):=
     \begin{cases}
       \gamma(t)\beta(\abs{v}^2)\INNER{\eta(t_*)}{v}&\text{,
       $p=\exp_{z_*} v$ with $\abs{v}^2<\iota^2$,}
       \\
       0&\text{, else.}
     \end{cases}
  \]
\end{itemize}
\begin{figure}
\hfill
\begin{minipage}[b]{.49\linewidth}
  \centering
  \includegraphics[width=0.95\textwidth]
                             {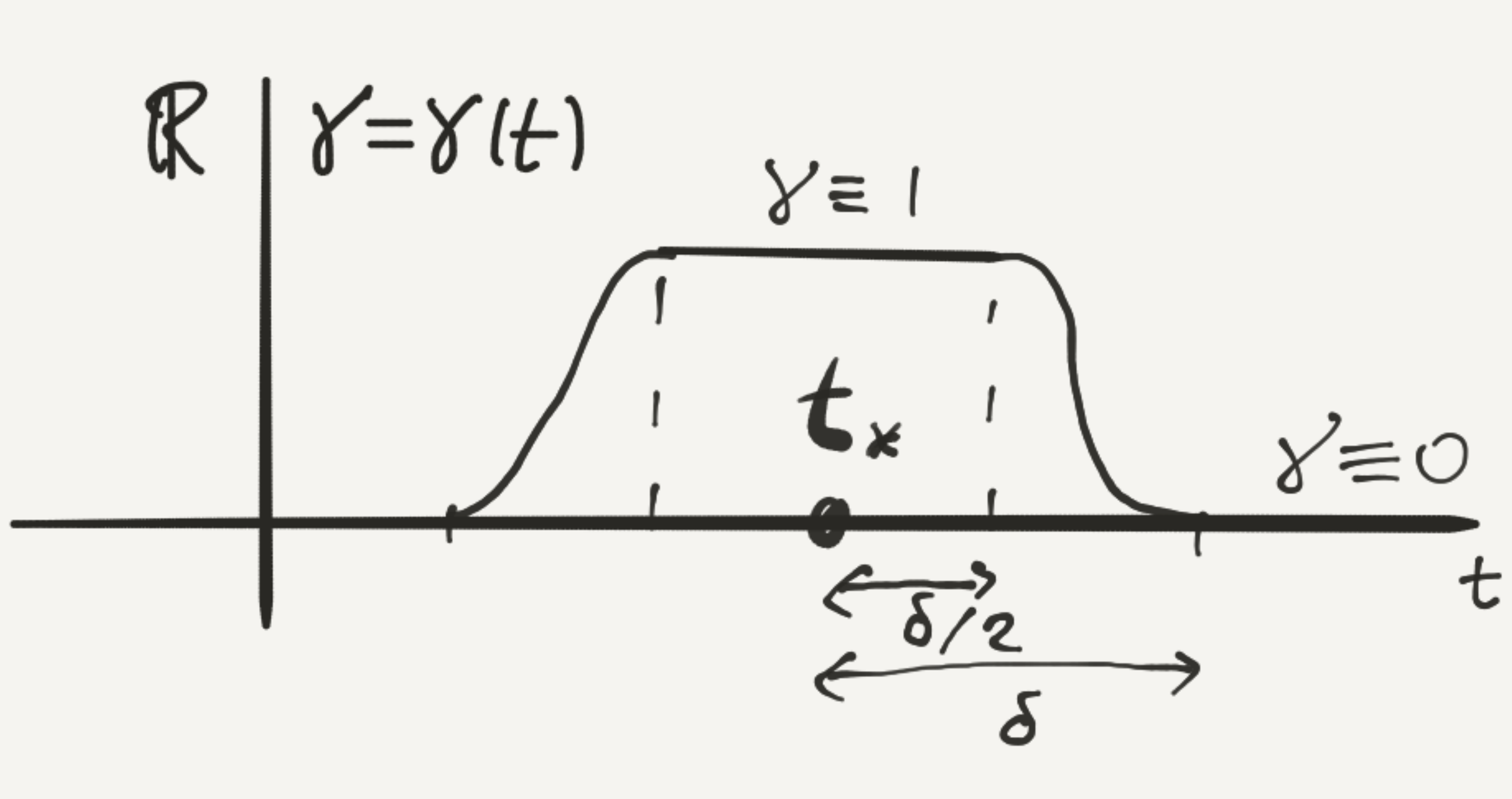}
  \caption{$\gamma\in C^\infty_0((t_*-\delta,t_*+\delta))$}
  \label{fig:fig-cutoff-gamma}
\end{minipage}
\hfill
\begin{minipage}[b]{.49\linewidth}
  \centering
  \includegraphics[width=0.95\textwidth]
                             {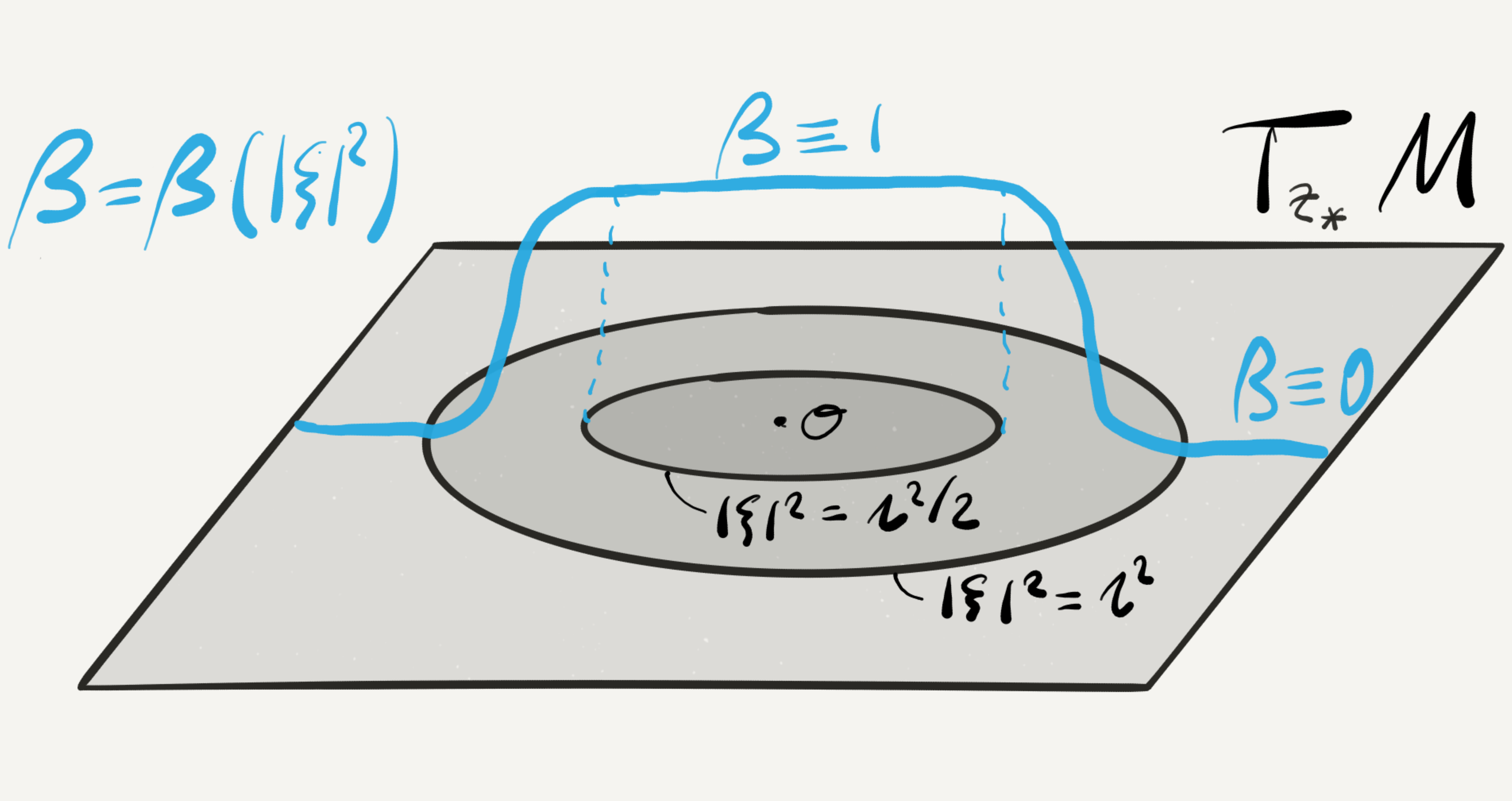}
  \caption{$\beta\in C^\infty_0((-\iota,\iota))$}
  \label{fig:fig-cutoff-beta}
\end{minipage}
\hfill
\end{figure}
Straightforward calculation shows that
\begin{equation*}
\begin{split}
     \INNER{\eta}{\nabla h_t(z)}_{L^2_z}
   &=\int_0^1 dh_t(z(t))\, \eta(t)\: dt
     \\
   &=\int_{T:=\{t\,:\,\abs{Z(t)}<\iota\}}
     \left.\tfrac{d}{d\tau}\right|_{\tau=0}
     \underbrace{h\bigl(\overbrace{\exp_{z(t)}\tau\eta(t)}^{=\exp_{z_*}v_\tau(t)}\bigr)}
     _{\gamma(t)\beta(\abs{v_\tau}^2)\INNER{\eta(t_*)}{v_\tau}} \: dt
     \\
   &=\int_{t_*-\delta}^{t_*+\delta}\biggl(
     2\gamma(t)\overbrace{\beta^\prime(\abs{Z(t)}^2)}^{\equiv 0}
     \,\INNER{Z(t)}{\left.\tfrac{D}{d\tau}\right|_0 v_\tau(t)}
     \,\INNER{\eta(t_*)}{Z(t)}
     \\
   &\qquad\qquad\quad
     +\gamma(t)\overbrace{\beta(\underbrace{\abs{Z(t)}^2}_{\le\iota^2/4})}^{\equiv 1}\,
     \langle
       \eta(t_*),\hspace{-.4cm}
       \underbrace{\left.\tfrac{D}{d\tau}\right|_0 v_\tau(t)}_{E_2(z_*,Z(t))^{-1}\eta(t)}
       \hspace{-.35cm}
     \rangle
     \biggr) \: dt
     \\
   &\ge\int_{t_*-\delta/2}^{t_*+\delta/2}
     \underbrace{\gamma(t)}_{\equiv 1}
     \underbrace{\INNER{\eta(t_*)}{E_2(z_*,Z(t))^{-1}\eta(t)}}_{>0}\: dt 
     >0.
\end{split}
\end{equation*}
But this is the desired contradiction to assumption two.
The second equality holds since for values outside $T$ the function $h$ is
identically zero by definition. Equality three holds since $\gamma$ is
supported in $[t_*-\delta,t_*+\delta]$. Observe that $v_\tau(t)$ is defined
by the overbraced identity; set $\tau=0$ to get that $v_0(t)=Z(t)$ and apply
$\left.\tfrac{D}{d\tau}\right|_0$ to both sides to get that
$E_2(z_*,v_0(t)) \left.\tfrac{D}{d\tau}\right|_0 v_\tau(t)
=E_2(z(t),0)\,\eta(t)=\eta(t)$.

\vspace{.1cm}\noindent
\textbf{\boldmath III. Density and openness in the $C^k$ category.}
To conclude the proof of Theorem~\ref{thm:A-Morse} we
show~(\ref{eq:backdoor}), that is density and openness of $\Hhreg^k$ in $\Hh^k$.

\begin{proof}[Proof of Theorem~\ref{thm:A-Morse}]
\emph{Density}: By Steps I and II the Thom-Smale transversality
Theorem~\ref{thm:TS-transversality} applies to
$\Ff=\grad\Aa:\Lambda M\times \Hh^k$ with $\ell=1$.
\newline
\emph{Openness}: By Theorem~\ref{thm:TS-transversality-openness}
it remains to show that the projection $\pi:\Mm\to\Vv^k$ in~(\ref{eq:pi})
is a proper map. The $C^1$ Banach manifold $\Mm$ is given by
\begin{equation*}
     \Mm
     :=\Ff^{-1}(0)
     =\{(z,H)\in\Lambda M\times\Hh^k\mid\dot z=X_{H_t}(z)\}.
\end{equation*}
\emph{Properness}: 
Given a compact subset $K\subset\Hh^k$
and a sequence $(z^\nu,H^\nu)$ in
$$
     \pi^{-1}(K) =\bigcup_{H\in K}
     \left(\Crit\Aa_H\times\{ H\}\right)
     \subset\left(\Lambda M\times\Hh^k\right),
$$
we need to extract a convergent subsequence (the limit of which then
lies automatically in the zero set $\Mm$ by continuity of $\Ff$).
\newline
Observe that the sequence $H^\nu$ lives in $K$.
Compactness of $K\subset\Hh^k$ allows to extract a convergent
subsequence, still denoted by $H^\nu$, and a limit $H\in\Hh^k$.
This has the consequence that the right hand sides of the sequence
of differential equations $\dot z^\nu=X_{H_t^\nu}(z^\nu)$,
each $z^\nu$ being $C^k$ by Lemma~\ref{le:regularity-crit},
are well behaved in the sense that they have a continuous limit
if the sequence $z^\nu$ has one.

By compactness of $M$ the sup norm over $\SS^1\times M$
of the $C^{k-1}$ vector field $X_H$ is finite.
Hence the families $\Ff_0=\{z^\nu\}$ and $\Ff_1=\{\dot z^\nu\}$ are
pointwise bounded.
Since $z^\nu\in C^k(\SS^1,M)$ with $k\ge 2$
both families are equicontinuous by Exercise~\ref{exc:AA}\,b).
Now apply the Arzel\`{a}-Ascoli Theorem~\ref{thm:AA}
to extract subsequences, without changing the notation,
and continuous loops $x$ and $y$
such that $z^\nu\to z$ and $\dot z^\nu\to y$ in $C^0$.
In this case $y=\dot z$,  see e.g.~\cite[Le.\,1.1.14]{Zimmer:1990a},
hence $z\in C^1$.
On the other hand, using the equations, the sequence $\dot z^\nu$
converges in $C^0$ to the vector field $X_{H_t}(z)$.\footnote{
  Set $z^\nu=\exp_z\zeta^\nu$, then
  $\norm{X_{H_t}(z)-\Tt_z(\zeta^\nu)^{-1} X_{H^\nu}(z^\nu)}
  _{C^0}\to 0$, as $\nu\to\infty$; cf.~(\ref{eq:F_z-dA}).
  }

Thus we found a subsequence of $(z^\nu,H^\nu)$ which admits a limit
$(z,H)\in \Crit\Aa_H\times\{H\}\subset \pi^{-1}(H)$.
The latter inclusion holds since $H\in K$.
\end{proof}
This concludes Example~\ref{ex:generic-Morse}.
\end{example}

\subsubsection{Classical action $\mbf{\Ss_V}$ is Morse for
generic potential $\mbf{V}$}
\begin{example}
Given a closed Riemannian manifold $Q$,
the classical action functional $\Ss_V$ defined
by~(\ref{eq:classical-action}) is Morse for generic smooth potential
functions $V:\SS^1\times Q\to\R$; for details
see~\citerefFH{weber:2002a}.
Here the critical sets $\Crit\Ss_V$ are not necessarily compact,
but there parts in sublevel sets $\{\Ss_V<a\}$ are
for each regular value $a$ of $\Ss_V$.
Hence Step III above will require refined arguments.
\end{example}

\newpage 
\section{Floer chain complex and homology}\label{sec:FC}

In this section we define the Floer complex associated to a
regular pair $(H,J)$ on a \emph{closed} symplectic manifold
$(M,\omega)$. We only consider the simplest case in which
$\omega$ and $c_1(M)$ vanish over $\pi_2(M)$;
cf.~(\ref{eq:symp-asph-c_1}).
The corresponding homology, called Floer homology,
transforms by isomorphisms when changing regular pairs
and in the end of the day represents, again naturally, the
singular homology of $M$.
For simplicity we take $\Z_2$ coefficients in all homology
theories in Section~\ref{sec:FC}.

\begin{definition}\label{def:regular-pair}
A \textbf{\Index{regular pair}} $(H,J)$ consists of a
Hamiltonian $H\in C^\infty(\SS^1\times M)$, that is a periodic
family of functions $H_{t+1}=H_t:=H(t,\cdot)$ of functions on $M$,
and a periodic family $J_{t+1}=J_t\in\Jj(M,\omega)$
of $\omega$-compatible almost complex structures
such that the following is true: The action functional
$\Aa_H:\Ll M\to\R$ given by~(\ref{eq:action-A}) is Morse and
the linearized operators $D_u$ given by~(\ref{eq:D_u-lin-Floer})
are surjective for all connecting trajectories $u\in\Mm(x,y;H,J)$
and all Hamiltonian loops $x,y\in\Pp_0(H)=\Crit\Aa_H$.
\end{definition}

Earlier we showed how to obtain a regular pair:
To satisfy the Morse condition, pick an element $H^0$
of the residual subset $\Hhreg\subset C^\infty(\SS^1\times M)$
provided by Theorem~\ref{thm:A-Morse}.
Now pick a family $J_{t+1}=J_t$ of $\omega$-compatible almost complex
structures on $M$.
Then, either perturb $H^0$ away from its critical points,
see Theorem~\ref{thm:A-MS}, or stay with $H^0$ and perturb
the family $J_t$, see~\citerefFH[Thm.~5.1]{Floer:1995a},
to obtain a regular pair denoted by $(H,J)$.

\begin{definition}\label{def:CF}
Pick a regular pair $(H,J)$. Then the
$\Z_2$ vector spaces
$$
     \CF_k(H)=\CF_k(M,\omega,H)
     :=\bigoplus_{z\in\Pp_0(H)\atop \CZcan_H(z)=k}
     \Z_2 z
$$
graded by the canonical, that is clockwise normalized, Conley-Zehnder index,
see~(\ref{eq:CZ-normalization-canonical-88}) and (\ref{eq:def-CZ-orbit}),
are called the \textbf{\Index{Floer chain group}s}
associated to the Hamiltonian $H$.\index{chain group!Floer --}
By convention the empty set generates the trivial vector space.
\newline
The set $\Pp_0(H)=\Crit:=\Crit \Aa_H$ of all contractible $1$-periodic
Hamiltonian orbits is finite by Proposition~\ref{prop:fin-crit-pts}.
The subset $\Crit_k$ of those of index $k$ is a basis of
$\CF_k(H)$, called the \textbf{\Index{canonical basis} (over $\mbf{\Z_2}$)}.

The \textbf{Floer boundary operator}\index{Floer!boundary operator}\index{boundary operator!Floer --}
is given on a basis element $x\in\Crit_k$ by
\begin{equation}\label{eq:Floer-boundary-operator}
\begin{split}
     \p=\p^{\rm F}(H,J):
     \CF_*(H)&\to\CF_*(H)
     \\
     x&\mapsto
     \sum_{y\in\Crit_{k-1}} \#_2(m_{xy})\, y
\end{split}
\end{equation}
where $\#_2(m_{xy})$ is the number {\rm (mod 2)}
of \textbf{connecting flow lines}.\footnote{
  A \textbf{\Index{connecting flow line}}
  is\index{flow line!connecting --}
  an unparametrized solution curve,
  the image of a flow trajectory $\R\to\Ll M$, $s\mapsto u_s$,
  between two critical points $x,y$.
  Connecting flow lines are in bijection with the set $m_{xy}$ of those $u\in\Mm(x,y;H,J)$
  with\index{$m_{xy}$ connecting flow lines}
  $\Aa_H(u_0)=\frac12(\Aa_H(x)+\Aa_H(y))$.
  This is a finite set by Exercise~\ref{exc:finite-set-flow-lines-ind-diff-1}.
  Let $\#_2(m_{xy})$ be the number of elements~modulo~2.
  }
\end{definition}

\begin{proposition}[Boundary operator]\label{prop:HF-bound-op}
It holds that $\p^2=0$.
\end{proposition}

\begin{exercise}\label{exc:HF-bound-op}
Given $x\in\Crit_{k+1}$, show that $\p^2 x$
is equal to the sum over all $z\in\Crit_{k-1}$
where the coefficient of each $z$ is the number (mod 2)
of 1-fold broken flow lines $(u,v)\in m_{xy}\times m_{yz}$ that start at $x$, end at $z$,
and pass an intermediate critical point $y$ of index $k$
at which $u$ and $v$ meet; cf. footnote.
\end{exercise}

So to prove $\p^2=0$ it suffices to show that
the number of such 1-fold broken flow lines $(u,v)$ between
$x$ and $z$ is even. To see this one shows that
for each $(u,v)$ there exists precisely one partner pair
$(\tilde u,\tilde v)$ which is determined by the property
that there is a 1-dimensional connected non-compact manifold,
i.e. an open interval, of trajectories running straight from $x$ to $z$
and whose two ends correspond to the two partner pairs;
see Figure~\ref{fig:fig-partner-pairs}.
The sense in which the two ends correspond to partner pairs
is introduced and detailed in Section~\ref{sec:FH-comp} on compactness
up to broken flow lines. The gluing procedure developed in
Section~\ref{sec:FH-gluing} excludes that two families
converge to the same broken flow line $(u,v)$.
To summarize, the set of 1-fold broken flow lines from $x$ to $z$
is in bijection with the ends of finitely many open intervals,
so the number of them is even.

\begin{definition}\label{def:HF}
The chain complex
$$
     \CF(H)=\CF(M,\omega,H,J):=\left(\CF_*(H),\p^{\rm F}(H,J)\right)
$$
is called the \textbf{Floer complex}\index{Floer!complex}
associated to a regular pair $(H,J)$. The corresponding
homology groups, called
\textbf{Floer homology groups},\index{Floer!homology}
are graded $\Z_2$ vector spaces. They are denoted by
$$
     \HF_*(H)=\HF_*(M,\omega,H;J).
$$
\end{definition}

\begin{theorem}[Continuation]\label{thm:HF-continuation}
For any two regular pairs $(H^\alpha,J^\alpha)$ and
$(H^\beta,J^\beta)$ there is a natural\,\footnote{
  Here \textbf{\Index{natural isomorphism}} means that there
  are no (further) choices involved.
  }
isomorphism
$$
     \Psi^{\beta\alpha}:\HF_*(\alpha)\to\HF_*(\beta).
$$
Furthermore, given a third regular pair $(H^\gamma,J^\gamma)$, then
$$
     \Psi^{\gamma\beta}\Psi^{\beta\alpha}=\Psi^{\gamma\alpha}
     ,\qquad
     \Psi^{\alpha\alpha}=\1.
$$
\end{theorem}

\begin{theorem}[Calculation]\label{thm:HF=H-CSF}
Suppose $(M,\omega)$ is a closed symplectic manifold
such that $\omega$ and $c_1(M)$ vanish over $\pi_2(M)$.
Then for any regular pair $(H^\alpha,J^\alpha)$ there is an
isomorphism of degree $n$ denoted by
$$
     \Psi^\alpha:\HF_{\ell-n}(\alpha)\to \Ho_{\ell}(M)
$$
and these isomorphisms are natural in the sense that
$$
     \Psi^\beta \Psi^{\beta\alpha}=\Psi^\alpha.
$$
\end{theorem}

\begin{corollary}\label{cor:AC}
The weak non-degenerate {Arnol\,$^\prime$d} conjecture~(\ref{eq:non-deg-AC}) is true.
\end{corollary}

\subsection{Compactness -- bubbling off analysis}
\label{sec:FH-comp}
Throughout $\Aa_H$ is Morse.
Let $\Mm_{x,y}$ be the space $\Mm(x,y;H,J)$,
see~(\ref{eq:conn-mf}), of connecting trajectories
between contractible $1$-periodic orbits $x,y\in\Crit\Aa_H$.

\begin{definition}[Convergence to broken trajectory]
\label{def:convergence-to-broken-line}
We say that a sequence $(u^\nu)\subset\Mm_{xy}$ of connecting trajectories
\textbf{\boldmath converges to a $(k-1)$-fold broken trajectory}\footnote{
  or, alternatively,
  \textbf{converges to a broken trajectory with $\mbf{k}$ components}
  }
$(u_k,\dots,u_1)$
if\index{$C^\infty_{\rm loc}$ convergence}
the\index{convergence!in $C^\infty_{\rm loc}$}
following holds
true.\index{convergence to broken trajectory}
There\index{broken trajectory!convergence to --}
are\index{broken trajectory}
\begin{itemize}
\item[-]
  pairwise different periodic orbits $z_k=x,z_{k-1},\dots,z_0=y$,
\item[-]
  (non-constant) connecting trajectories $u_j\in\Mm_{z_j z_{j-1}}$ for
  $j=1,\dots,k$,
\item[-]
  sequences $(s^\nu_j)\subset\R$ of reals for $j=1,\dots,k$,
\end{itemize}
such that the sequence of maps $\R\times\SS^1\to M$,
$(s,t)\mapsto u^\nu(s+s^\nu_j,t)$ converges to $u_j$ in
$C^\infty_{\rm loc}$, as $\nu\to\infty$, i.e. uniformly with all
derivatives on compact sets. By
$$
     u^\nu\to(u_k,\dots,u_1),\quad \text{as $\nu\to \infty$,}
$$
we henceforth denote \textbf{convergence to a broken trajectory};
see Figure~\ref{fig:fig-conv-broken-flow-line}.
\end{definition}

\begin{figure}[h]
  \centering
  \includegraphics
                             [height=4cm]
                             {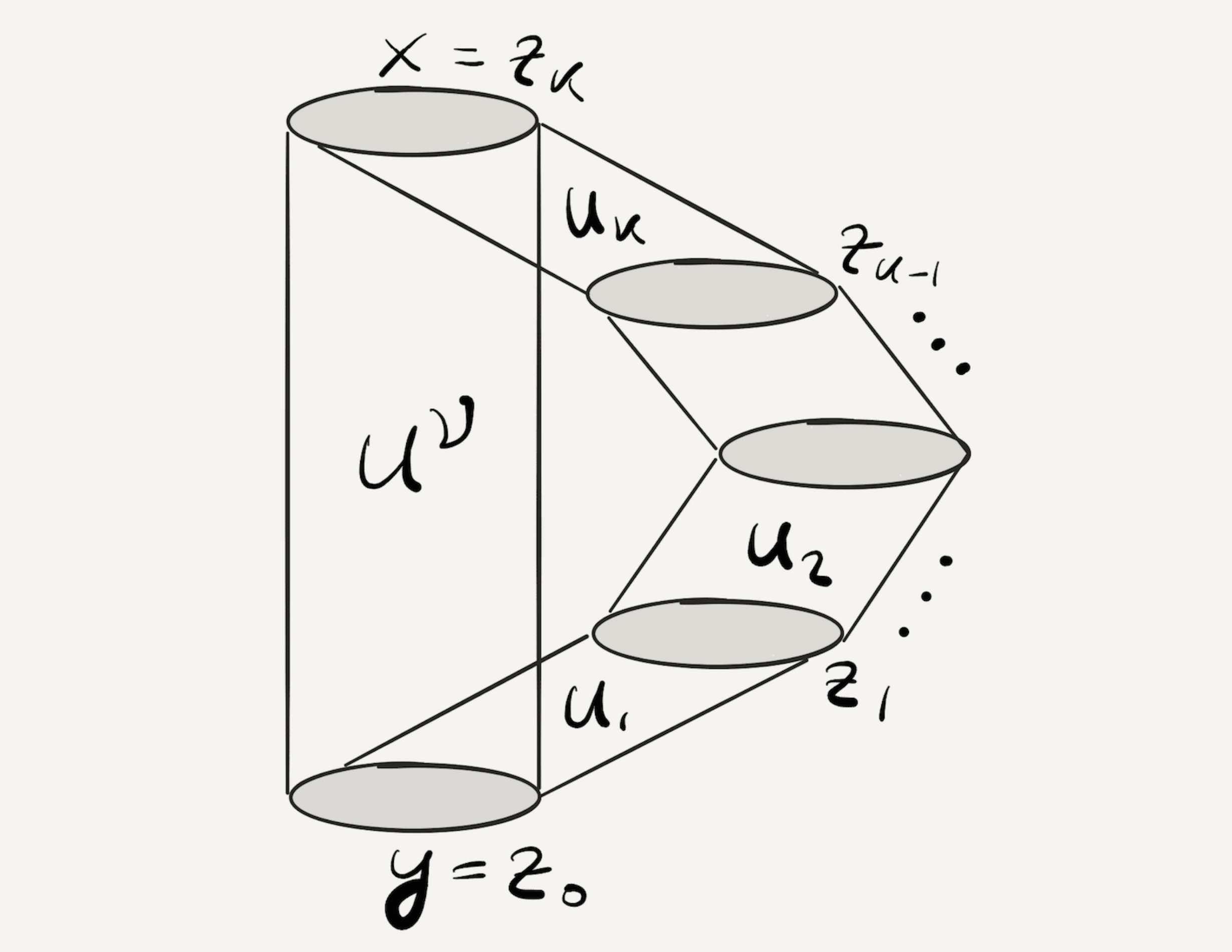}
  \caption{Convergence of sequence $u^\nu$ to broken trajectory
                 $(u_k,\dots,u_1)$}
  \label{fig:fig-conv-broken-flow-line}
\end{figure}

\begin{exercise}[Strictly decreasing index in case of a regular pair]
\label{exc:FH-comp}
Show that the index $\CZcan$,
strictly decreases
along the members of a broken trajectory $(u_k,\dots,u_1)$
whenever $(H,J)$ is a regular pair.
\newline
[Hint: By non-constancy each $u_j$ comes in a family
of dimension at least $1$. But this dimension
is a Fredholm index. Now recall~(\ref{eq:D_u-F_index-CZ}).]
\end{exercise}

\begin{proposition}[Compactness up to broken trajectories]
\label{prop:compact-up-to-broken-line}
Suppose $\Aa_H$ is Morse.
Then any sequence $(u^\nu)\subset\Mm_{xy}$ of connecting trajectories
admits a subsequence which converges to a broken trajectory
$(u_k,\dots,u_1)$, as $\nu\to\infty$.
\end{proposition}

It is interesting to observe that by, and in the sense of, the
proposition the space $\Mm$ of finite energy trajectories,
see~(\ref{eq:Mm=Mm_xy}), is compact for a regular pair.

\begin{exercise}[Finite set in case of a regular pair $(H,J)$]
\label{exc:finite-set-flow-lines-ind-diff-1}
Show that in case of index difference one, that is $\CZcan(x)-\CZcan(y)=1$,
the set $m_{xy}$ of those connecting trajectores $u\in\Mm_{xy}$ whose
initial loop $u_0=u(0,\cdot)$ lies on a fixed intermediate action level,
say $r=\frac12(\Aa_H(x)+\Aa_H(y))$, is a
finite\index{$u_{(\sigma)}:=u(\sigma+\cdot,\cdot)$ time-shift}
set.\footnote{
  The set $m_{xy}$ parametrizes the flow lines from $x$ to $y$,
  in other words, the set of connecting trajectores $u$ modulo time shift.
  As maps $u$ and $u_{(\sigma)}:=u(\sigma+\cdot,\cdot)$ are different, but their
  image curve in the loop space, i.e. their flow line, is the same.
  }
\end{exercise}

\subsubsection{Proof of Proposition~\ref{prop:compact-up-to-broken-line} (Compactness up to broken trajectories)}
The proof has four steps the first of which is highly trivial,
at least in the present case of a \emph{compact} manifold $M$.
We only sketch the main ideas. Actually the first three steps
do not require that $\Aa_H$ is Morse, not even that
$x,y$ are non-degenerate.
The Morse requirement enters in Step~IV
through Theorem~\ref{thm:finite-ENERGY=LIMITS-exist}.

\vspace{.2cm}\noindent
\textsc{Step I.} (Uniform $C^0$ bound for $u^\nu$)

\vspace{.1cm}\noindent
Obvious by compactness of $M$.

\vspace{.2cm}\noindent
\textsc{Step II.} (Uniform $C^0$ bound for $\p_su^\nu$ -- bubbling
off analysis)

\vspace{.1cm}\noindent
In fact, to carry out Step~III below a uniform $W^{1,p}$ bound, for
some constant $p>2$, for the sequence of connecting trajectories
$u^\nu$ would be sufficient. Note that we get a uniform
$W^{1,2}$ bound for free due to the energy
identity~(\ref{eq:ENERGY-FLoer-flow-line}).
But this is not good enough to get uniform $C^0$ bounds;
cf. also Remark~\ref{rem:p>2}.

On the positive side, investing some work even leads to
a uniform $C^1$ bound, that is
$$
     \sup_{\nu}\Norm{\p_s u^\nu}_{L^\infty}<\infty.
$$
To prove this, assume by contradiction there was a sequence
of points $\zeta^\nu$ along which the derivative
$
     \abs{\p_s u^\nu(\zeta^\nu)}\to\infty
$
explodes, as $\nu\to \infty$. It is convenient to
view maps defined on $\R\times\SS^1$ likewise
as maps on $\C\ni s+it$ which are 1-periodic in the imaginary part $t$.
Exploiting the invariance of $\Mm_{xy}$ under shifts in the $s$ variable
in order to replace $u^\nu(s,t)$ by the shifted sequence $u^\nu(s-s^\nu,t)$
and then by compactness of $\SS^1$ picking a subsequence,
we assume without loss of generality that
$\zeta^\nu=0+i \tau^\nu\to 0+i t_0=:z_0$, as $\nu\to\infty$.
Now there appears yet another \Index{Wonder-Lemma}, called the
\textbf{\Index{Hofer-Lemma}}~\cite[\S 6.4 Le.~5]{hofer:2011a}.
When\index{lemma!Hofer-}\index{lemma!wonder-}
applied for each $\nu$ to the continuous non-negative function
$g:=\abs{\p_s u^\nu}$ on the complete metric space
$X:=[-1,1]\times\SS^1$, the point $x_0:=\zeta^\nu\in X$,
and the constant $\eps_0:=\abs{\p_s u^\nu(\zeta^\nu)}^{-1/2}>0$,
the Hofer-Lemma yields a sequence of points
$z^\nu\in X$ and constants $\eps^\nu>0$ such that
\begin{itemize}
\item
     $0<\eps^\nu\le \abs{\p_s u^\nu(\zeta^\nu)}^{-1/2}\to 0$
\item
     $R^\nu\eps^\nu:=\Abs{\p_s u^\nu(z^\nu)}\eps^\nu
     =g(z^\nu)\eps^\nu
     \ge\frac{g(\zeta^\nu)}{\abs{\p_s u^\nu(\zeta^\nu)}^{1/2}}
     =\Abs{\p_s u^\nu(\zeta^\nu)}^{1/2}\to\infty$
\item
     $\Norm{z^\nu-\zeta^\nu}
     \le\frac{2}{\abs{\p_s u^\nu(\zeta^\nu)}^{1/2}}\to 0$
\item
     $\sup_{{\color{cyan} B_{\eps^\nu}(z^\nu)}} \Abs{\p_s u^\nu}
     \le 2\Abs{\p_s u^\nu(z^\nu)}=2R^\nu$
\end{itemize}
Note that $R^\nu:=\abs{\p_s u^\nu(z^\nu)}\to\infty$,
so the derivative explodes as well along the new sequence of points
$z^\nu$ which also converges to $z_0$, but for which we have
more information than we had for $\zeta^\nu$.
The key step is then to consider the sequence of
rescaled smooth maps
$$
     v^\nu:\C\to M,\quad
     z\mapsto u^\nu(z^\nu+(R^\nu)^{-1} z).
$$
These maps have the property that they are non-constant
and uniformly $C^1$ bounded on balls
$B_{R^\nu\eps^\nu}(0)$ whose radius $R^\nu\eps^\nu$ tends to infinity.
More precisely,
\begin{itemize}
\item
  $\abs{\p_s v^\nu(0)}=1$
\item
  $\abs{\p_s v^\nu}\le 2$ on {\color{cyan} $B_{R^\nu\eps^\nu}(0)$},
  see Figure~\ref{fig:fig-bubbling}
\item
  $\p_s v^\nu+J_{t^\nu+(t/R^\nu)}(v^\nu)\p_t v^\nu
  =\frac{1}{R^\nu}\nabla H_{t^\nu+(t/R^\nu)}(v^\nu)$
\end{itemize}
Now one shows that there is a smooth map $v:\C\to M$
and a subsequence, still denoted by $v^\nu$, such that
$v^\nu\to v$ in $C^\infty(\C,M)$, that is uniformly with all
derivatives on $\C$.
In view of the Arzel\`{a}-Ascoli Theorem~\ref{thm:AA}
it suffices to establish uniform in $\nu$
bounds on compact sets for $v^\nu$ and each of its derivatives.
In view of the Sobolev embedding theorem, see
e.g.~\cite[Thm.~B.1.11]{mcduff:2004a}, it suffices
to establish uniform $W^{k,p}_{\rm loc}$ bounds for $v^\nu$,
that is for each $k\in\N$ and each compact set $K\subset\C$
find a $W^{k,p}$ bound $c$ for the restriction $v^\nu|_K$ to $K$
such that $c$ serves for all $\nu$.
The key input is the Calder\'{o}n-Zygmund
inequality~(\ref{eq:CaldZyg-delbar}) for the
Cauchy-Riemann operator $\delbar$.
The desired bounds are established by induction on $k$.
For details see e.g.~\citerefFH[Le.~5.2]{salamon:1990a}, using
elliptic bootstraping techniques, or~\cite[\S 6.4 Le.~6]{hofer:2011a},
using the \textbf{\Index{Gromov trick}} to first get rid off the
Hamiltonian term followed by a proof by contradiction.

The limit map $v:\C\to M$ satisfies the equation
\begin{equation}\label{eq:J-hol-plane}
     \p_s v+J(v)\p_t v=0
\end{equation}
where $J(v)=J_{t_0}(v)$. The solutions $w:\C\to M$ of this
elliptic PDE are called \textbf{\Index{$J$-holomorphic planes}}
or \textbf{\Index{pseudo-holomorphic planes}}.
They have been introduced in Gromov's
1985 landmark paper~\citeintro{gromov:1985a}.
Since
$$
     \abs{\p_s v(0)}=1,\qquad
     \abs{\p_s v}\le 2,\qquad
     \int_\C v^*\omega=\norm{\p_s v}_{L^2}^2>0,
$$
we arrive at a contradiction as soon as we can show that
$v$ extends continuously from $\C$ to the Riemann
sphere, that is to a continuous map
$\tilde v:\SS^2=\C\cup\{\infty\}\to M$.
Indeed in this case the proof of
Proposition~\ref{prop:compact-up-to-broken-line} is complete since
$$
     0=\int_{\SS^2} \tilde v^*\omega
     =\int_\C v^*\omega>0.
$$
Here the first identity uses that by our standing assumption~(C2)
the evaluation map $\Io_\omega=0$ of $\omega$ over $\pi_2(M)$
vanishes, see~(\ref{eq:I-omega}), and the second identity
holds since a point has measure zero. 

So it remains to construct
the continuous extension $\tilde v$. For each radius $R>0$
the map $v:\C\to M$ gives rise to a loop in $M$ by restriction
to the radius-$R$ sphere in $\C$ centered at the origin;
notation $v_R:\C\supset {\color{red} \{\abs{z}=R\}}\to M$;
indicated {\color{red}red} in Figure~\ref{fig:fig-bubbling}.
In~\citerefFH[Pf. of Prop.~4.2]{salamon:1990a} it is shown,
see also~\cite[paragraph before~(6.94)]{hofer:2011a}, that
the lengths of the image circles {\color{red} $\gamma_R$} in $M$
of the maps $v_R$ tends to zero, as $R\to\infty$.
\begin{figure}
  \centering
  \includegraphics
                             [height=4.5cm]
                             {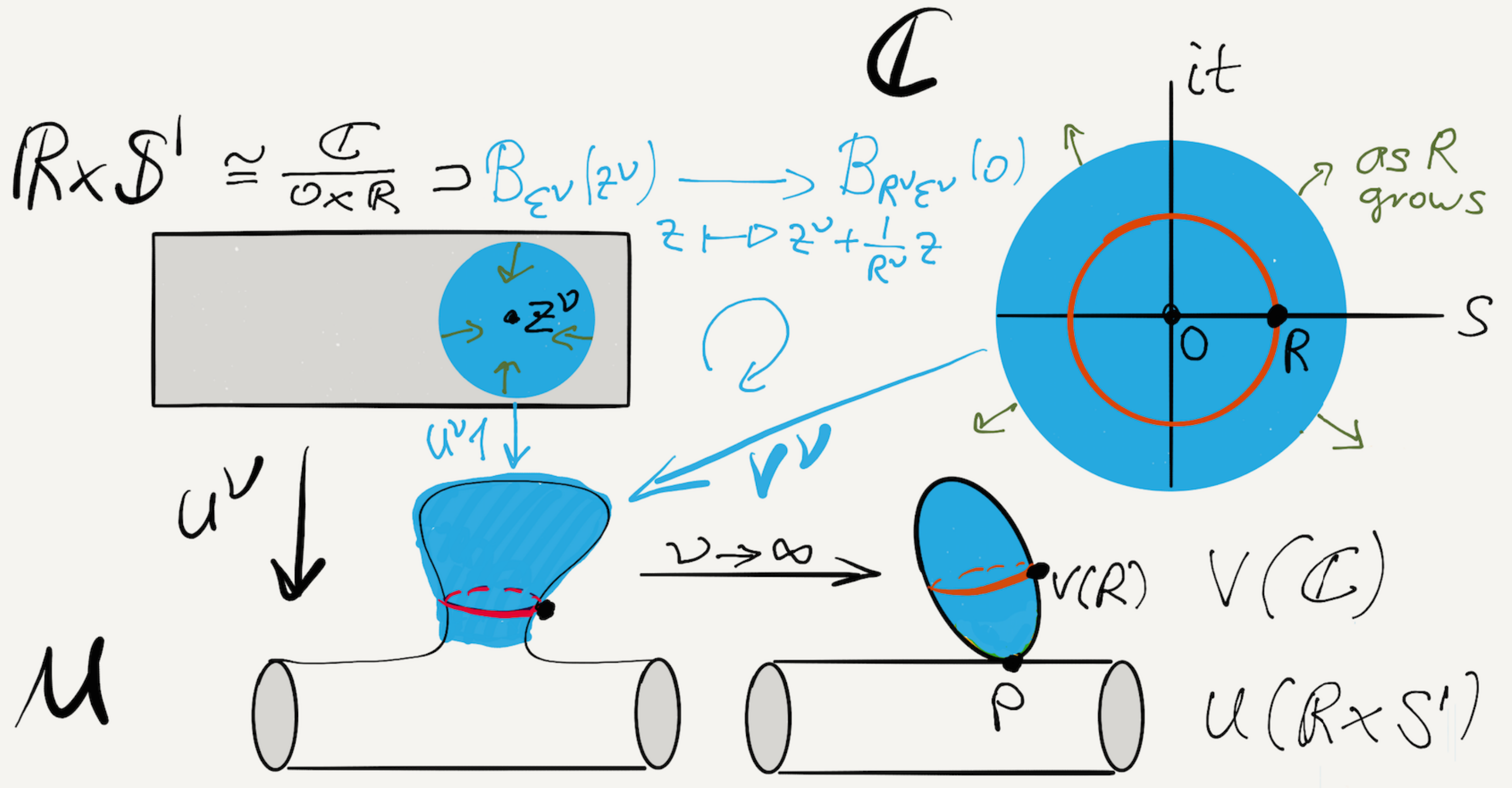}
  \caption{Bubbling off of spheres requires that $\omega$ does
                not vanish over $\pi_2(M)$}
  \label{fig:fig-bubbling}
\end{figure}
While along any sequence $R_j\to\infty$
the family of points $\{v(R)=v(R+i 0)\}_{R>0}$
admits a convergent subsequence by compactness of $M$,
continuity of $v$ together with the length shrinking property
imply uniqueness of the limit $p$ and independence of the
choice of the sequence $R_j\to\infty$.
Clearly $p:=\lim_{R\to\infty} v(R)$ completes the image
$v(\C)$ to be a 2-sphere and we are done.\footnote{
  As $p$ lies in the closure of the image of $v$,
  it lies in the closure of the union of all images
  $u^\nu(\R\times\SS^1)$. So the bubble is attached
  at $p$ to whatever is the limit, see Step~IV, of the $u^\nu$'s.
  }

\vspace{.2cm}\noindent
\textsc{Step III.} (Uniform $C^\infty_{\rm loc}$ bound for
$u^\nu$ -- elliptic bootstrapping)

\vspace{.1cm}\noindent
First of all, in view of the uniform $C^1$ bound
for our sequence of connecting trajectories $u^\nu$,
obtained in Steps~I and~II,
the Arzel\`{a}-Ascoli Theorem~\ref{thm:AA}
provides a continuous map $u:\R\times\SS^1\to M$
to which some subsequence, still denoted by $u^\nu$,
converges uniformly on compact sets. This allows
to analyze the problem in local coordinates on $M$,
hence for maps taking values in $\R^{2n}$.
\newline
To derive uniform $C^\infty$ estimates for the sequence $u^\nu$
on compact sets is, surely without surprise, an iterative procedure.
To illustrate the basic mechanism behind,
let us \emph{incorrectly oversimplify} things by assuming
that~(i) each $u^\nu$ is a map $\C\to(\R^{2n},J_0)$,
$s+i t\mapsto u^\nu(s,t)$, supported in a compact set $K$
-- which it clearly isn't, given the non-linear target $M$
and the asymptotic boundary conditions $x$ and $y$ --
and that~(ii) each $u^\nu$ satisfies the much simpler PDE
$$
     \delbar u^\nu=\nabla H(u^\nu),\qquad
     \delbar:=\p_s+ J_0\p_t.
$$
The iteration rests on
the\index{Calder\'{o}n-Zygmund inequality}
\textbf{Calder\'{o}n-Zygmund type inequality}
\begin{equation}\label{eq:CaldZyg-delbar}
     \Norm{v}_{W^{1,p}}\le c_p\Norm{\delbar v}_{L^p}
\end{equation}
and its immediate consequence (replace $v$ by derivatives of $v$)
$$
     \Norm{v}_{W^{k,p}}\le c_{p,k}\Norm{\delbar v}_{W^{k-1,p}}.
$$
These hold true for any compactly supported
map $v\in C^1_0(\C,\R^{2n})$ whenever $k\in\N$ and
$1<p<\infty$ and where the constant
$c_p>0$ only depends on~$p$;
see e.g.~\citerefFH[III \S 1 Prop.~4]{Stein:1970a}
or~\citerefCG[Thm.~2.6.1]{Wendl:2015a}. By continuity the estimates
continue to hold for $v$ in the closure $W^{k,p}_0$
of $C_0^\infty$ with respect to the $W^{k,p}$ norm.

Application to our oversimplified maps $u^\nu\in W^{1,p}_0(K,\R^{2n})$
shows that
$$
     \Norm{u^\nu}_{W^{k,p}}
     \le c_p\Norm{\nabla H(u^\nu)}_{W^{k-1,p}}
     \le C(p,k,K,H)
$$
where the constant $C$ does not depend on $\nu$.
More precisely, starting with $k=2$
we get a uniform $W^{2,p}$ bound on $K$ which then leads to
a uniform $W^{3,p}$ bound and so on for every $k\in\N$.
The Sobolev embedding theorem, see
e.g.~\cite[Thm.~B.1.11]{mcduff:2004a}, then provides
uniform $C^k$ bounds on $K$ for every $k$.
Now apply the Arzel\`{a}-Ascoli Theorem~\ref{thm:AA}
to each derivative of $u^\nu$.

However, the general case is much harder, of course, mainly
due to the non-linearity $J(u^\nu)$ in front of the
highest order term $\p_t u^\nu$. The fact that the maps
$u^\nu$ are not of compact support at all, requires the use of
cutoff functions leading to additional terms as well.
For details see e.g.~\cite[\S 6.4 Le.~6]{hofer:2011a}
or~\citerefFH[Le.~5.2]{salamon:1990a}.

\vspace{.2cm}\noindent
\textsc{Step IV.} (The limit broken trajectory)

\vspace{.1cm}\noindent
Pick $T>0$ and consider the restrictions
of $u^\nu$ to $Z_T:=[-T,T]\times\SS^1$.
By Step~III there are uniform $C^\infty$ bounds for $u^\nu$
on $Z_T$, thus there is a subsequence, still denoted by $u^\nu$
converging uniformly with all derivatives to some
smooth map $u:Z_T\to M$ which solves
the Floer equation~(\ref{fig:fig-Floers-interpol-eq}) as well.
Replacing $T$ by $2T$ and choosing subsequences if
necessary one concludes that $u^\nu$ restricted to $Z_{2T}$
converges to a smooth solution $Z_{2T}\to M$ which coincides with $u$
on $Z_T$. Iteration leads to a smooth limit solution again denoted by
$u:\R\times\SS^1\to M$. If we knew that $u$ was of finite energy
we would apply Theorem~\ref{thm:finite-ENERGY=LIMITS-exist}
to obtain existence of periodic orbits $z^\mp\in\Pp_0(H)$
sitting at the ends of $u$, that is $u\in\Mm_{z^- z^+}$.
With this understood apply the finite\footnote{
  The downward procedure can only end at $y$
  and it ends after at most $\abs{\Crit\Aa_H}$ many
  steps: The action strictly decreases
  from $z_\ell$ to $z_{\ell-1}$
  iff $z_\ell\not=z_{\ell-1}$ iff $E(u_\ell)>0$, but
  there are only finitely many periodic orbits by
  Proposition~\ref{prop:fin-crit-pts}.
  }
iteration detailed in~\citerefFH[Pf. of Prop.~4.2]{salamon:1990a}
to get the desired limit broken trajectory $(u_k,\dots,u_1)$.

Let us check that the energy of the limit solution
$u:\R\times\SS^1\to M$ is indeed finite.
By the energy identity~(\ref{eq:ENERGY-FLoer-flow-line})
for connecting trajectories we get
\begin{equation*}
\begin{split}
     \Aa_H(x)-\Aa_H(y)&=E(u^\nu)\\
   &\ge\int_0^1\int_{-T}^T \Abs{\p_su^\nu(s,t)}^2 ds\, dt=:E_{[-T,T]}(u^\nu)\\
   &=\Aa_H(u^\nu_{-T})-\Aa_H(u^\nu_T)
\end{split}
\end{equation*}
where $u^\nu_T:=u^\nu(T,\cdot)$ and the last identity is true since
$\p_s u^\nu$ is the downward $L^2$ gradient of $\Aa_H$.
By uniform $C^\infty$ convergence we can take the limit
over $\nu$ to get the estimate
\begin{equation}\label{eq:A-diff}
     \Aa_H(x)-\Aa_H(y)
     \ge\Aa_H(u_{-T})-\Aa_H(u_T)=E_{[-T,T]}(u)
\end{equation}
for every $T$, in particular, for $T=\infty$.
Thus $\infty >\Aa_H(x)-\Aa_H(y)\ge E(u)$.\footnote{
  The case $E_{[-T,T]}(u)=0$ for every $T>0$
  is not excluded: Given $v\in\Mm_{xy}$,
  consider the upward shifts $u^\nu(s,t):=v(s-\nu,t)$. So
  $u^\nu_{-T}=v(-\nu-T,\cdot)$ and
  $u^\nu_T=v(-\nu+T,\cdot)$. Then
  $\Aa_H(u^\nu_{-T})$ and $\Aa_H(u^\nu_{T})$
  both converge to $\Aa_H(x)$, as $\nu\to\infty$.
  Indeed on \emph{compact} sets $u^\nu$ converges uniformly
  with all derivatives to the constant
  trajectory $u(s,\cdot)=x(\cdot)$.
  }

This concludes the proof of Step~IV and
Proposition~\ref{prop:compact-up-to-broken-line}.

\subsection{Gluing}\label{sec:FH-gluing}

Suppose $(H,J)$ is a regular pair and
pick periodic orbits $x,y,z$ of $\CZcan$-indices
$k+1,k,k-1$, respectively.
To conclude the proof of Proposition~\ref{prop:HF-bound-op}
($\p^2=0$) let us construct a continuous map\footnote{
  \emph{Warning.} Here $u\#_R v$ denotes the true zero,
  in~\citerefFH[\S 3.4]{salamon:1999a} the approximate zero $\tilde w_R$.
  }
$$
     \cdot\#_\cdot \cdot :m_{xy}\times[R_0,\infty)\times m_{yz}
     \to m_{xz},\quad
     (u,R,v)\mapsto u\#_R v,
$$
called the \textbf{\Index{gluing map}}, whose image lies in
one component of the 1-dimensional manifold $m_{xz}$
of which it covers one 'end' in the sense that
$u\#_R v\to(u,v)$, as $R\to\infty$. In other words, the
broken trajectory $(u,v)$ represents the boundary point of that end.
Furthermore, and most importantly, as it concludes the proof of
$\p^2=0$, no sequence in $m_{xz}\setminus u\#_{[R_0,\infty)} v$,
that is no sequence away from the {\color{cyan} image} of the glued family
$R\mapsto u\#_R v$, converges to $(u,v)$;
see Figure~\ref{fig:fig-gluing-map}.
\begin{figure}
  \centering
  \includegraphics
                             [height=3cm]
                             {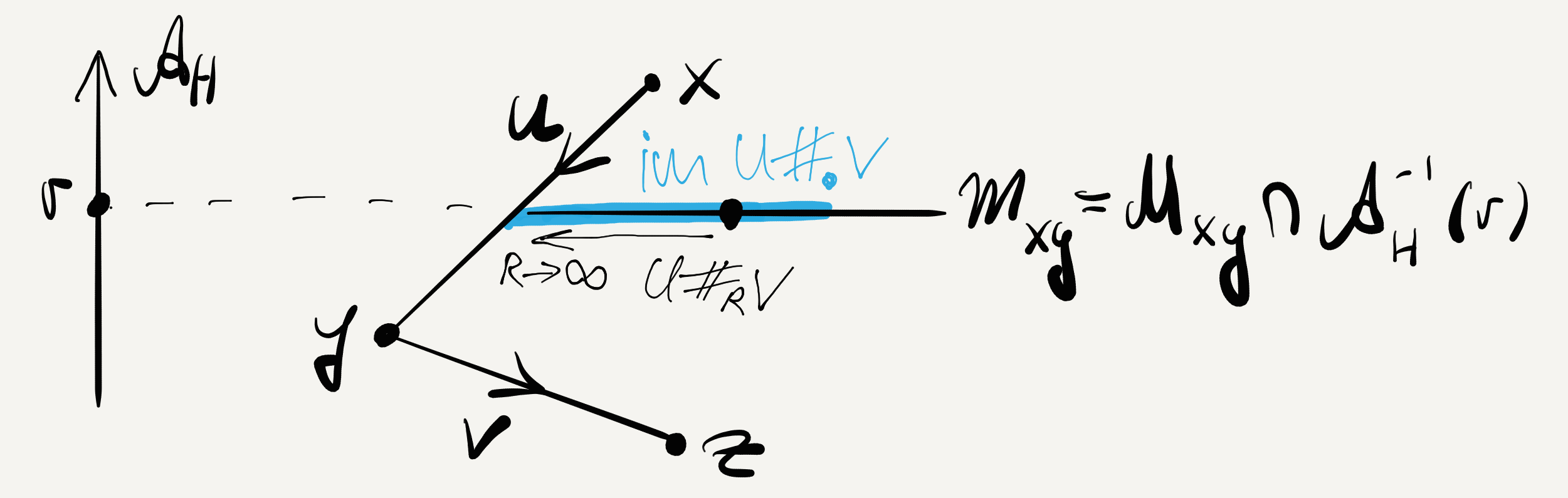}
  \caption{(Gluing map) Convergence $u\#_R v\to (u,v)$,
                 as $R\to\infty$}
  \label{fig:fig-gluing-map}
\end{figure}

The construction is by the \textbf{\Index{Newton method}}
to find a zero of a map $\Ff_H$ near a given approximate
zero $\tilde w_R$: Roughly speaking, one needs that
$\Ff_H(\tilde w_R)$ is 'small', the derivative
$D_R:=D_{\tilde w_R}:=D\Ff_H(\tilde w_R)$ at $\tilde w_R$ is
'steep',\footnote{
  More precisely, in our context $D_R$
  needs to admit a uniformly bounded right inverse $T$.
  }
and does not vary 'too much' near the approximate zero;
see~Figure~\ref{fig:fig-Newton-method}.
\begin{figure}[h]
  \centering
  \includegraphics
                             [height=3cm]
                             {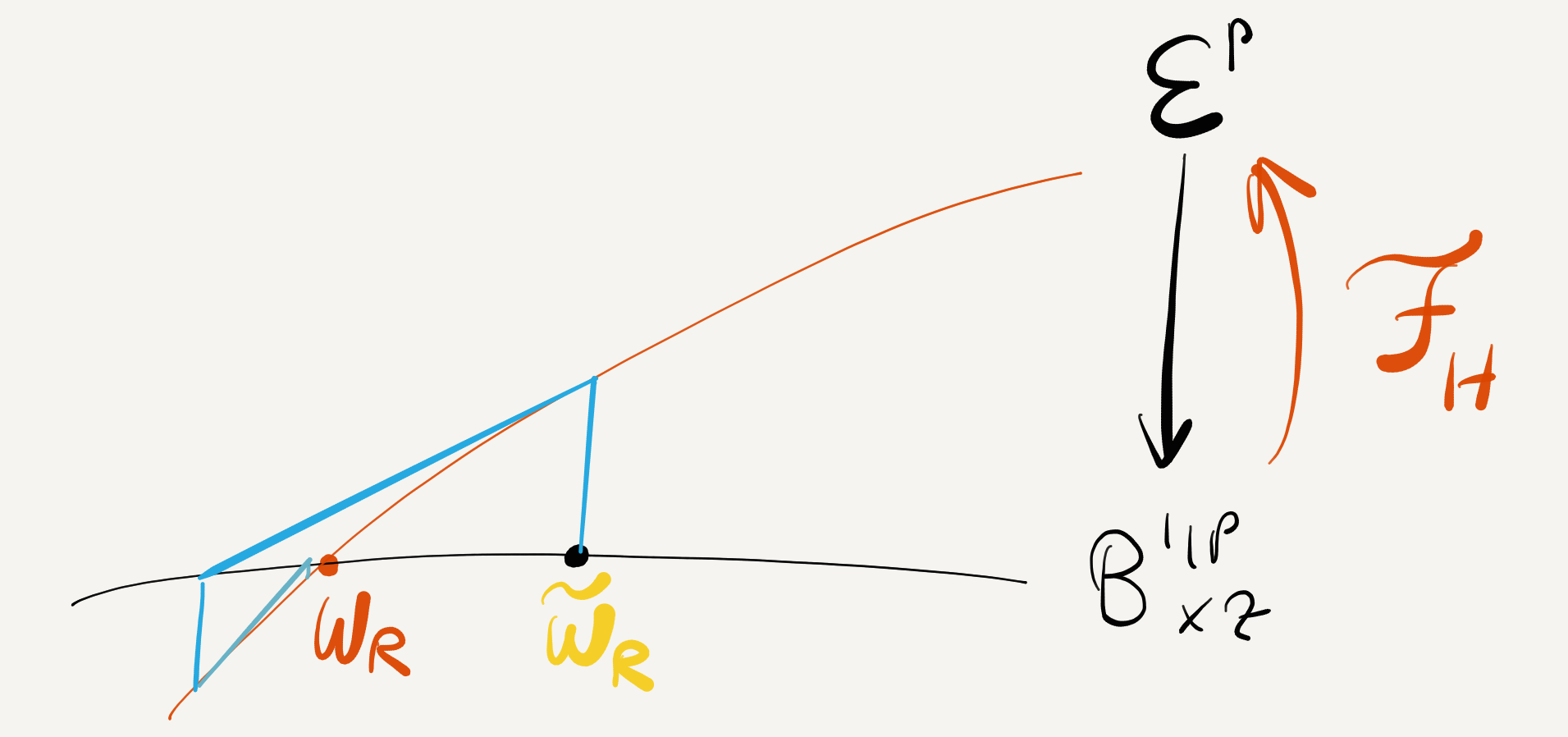}
  \caption{(Newton method) Find true zero $w_R$ nearby
                 approximate zero $\tilde w_R$}
  \label{fig:fig-Newton-method}
\end{figure}

We only outline the construction and refer
to~\citerefFH[\S 3.3]{salamon:1999a} for details.
Obviously $\Ff_H$ is the Floer section~(\ref{eq:FLOER})
of the Banach bundle $\Ee^p\to\Bb^{1,p}(x,z)$;
cf.~(\ref{eq:D_u-lin-Floer-gen}).
To start with pick $u\in m_{xy}$ and $v\in m_{yz}$
and consider the approximate zero $\tilde w_R=\tilde w_R(u,v)$
used in~\citerefFH{salamon:1999a} and illustrated by
Figure~\ref{fig:fig-approx-zero}.
Next, as we wish to use the implicit function
theorem, we need to move from Banach bundles
to Banach spaces. So, we replace $\Ff_H$ by the~map
$$
     f_R:=f_{\tilde w_R}:L^p_{\tilde w_R}\supset W^{1,p}_{\tilde w_R}
     \to L^p_{\tilde w_R},\quad
     \xi\mapsto\Tt_{\tilde w_R}(\xi)^{-1}\Ff_H(\exp_{\tilde w_R} \xi),
$$
where $\Tt$ denotes parallel transport; cf.~(\ref{eq:F_z-dA})
and~(\ref{eq:D_u-lin-Floer-gen}).

\begin{exercise}
a)~The map $\xi\mapsto \exp_{\tilde w_R} \xi$ is
a bijection between the zeroes of $f_R$ and those of $\Ff_H$.
b)~The linearization 
$df_R(0)\xi:=\left.\frac{d}{d\tau}\right|_{\tau =0}f_R(\tau \xi)$
is $D_R$.\newline
[Hint: Compare~\citesymptop[Pf. of Thm.~A.3.1]{weber:1999a}.]
\end{exercise}
So the tasks at hand are to
\begin{itemize}
\item[(a)]
  show that $f_R$ admits a unique zero $\eta_R$ whenever $R$
  is sufficiently large;
\item[(b)]
  show that $w_R:=\exp_{\tilde w_R}\eta_R\to(u,v)$, as $R\to\infty$;
\item[(c)]
  define $u\#_R v:=w_R$.
\end{itemize}

Obviously we start with task~(c). Next, concerning task~(b), let us
argue geometrically by looking at Figure~\ref{fig:fig-approx-zero}
where $\beta:\R\to[0,1]$ is a cutoff function that equals zero for $s\le 0$
and one for $s\ge 1$.
Observe that the curve $s\mapsto \tilde w_R(s\cdot)$
follows more and more, the larger $R$, all of $u$ and all of $v$.
Thus $\tilde w_R$ restricted to a fixed compact subdomain,
say of the form $[-T,T]\times \SS^1$, runs into
the constant in $s$ solution $y$, unless it is shifted appropriately
backward or forward in $s$ in which case it runs towards a piece of
$u$ or $v$, respectively.
\begin{figure}[h]
  \centering
  \includegraphics
                             [height=6cm]
                             {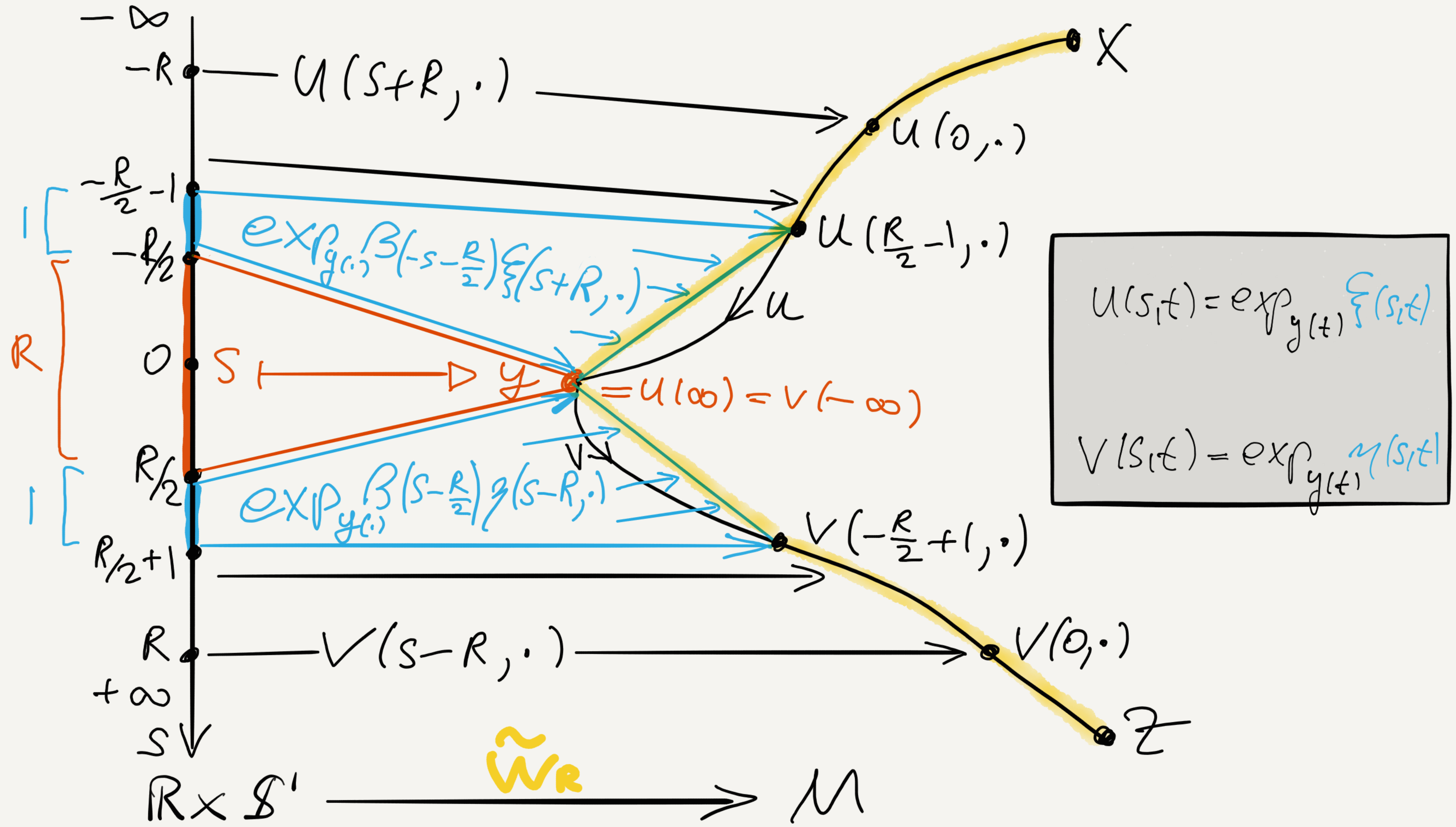}
  \caption{The approximate zero
                 ${\color{yellow} \tilde w_R}=\tilde w_R(u,v)$
                 of the Floer section $\Ff_H$}
  \label{fig:fig-approx-zero}
\end{figure}

Concerning task~(a) let us discuss informally the three
ingredients needed to carry out Newton's method.

\vspace{.1cm}\noindent
\textsc{I. Approximate zero.}
The $L^p$ norm of $f_R(\tilde w_R)$,
hence of $\Ff_H(\tilde w_R)$ since parallel transport
$\Tt$ is an isometry, is equal to the $L^p$ norm of $\Ff_H$
applied to the $\exp$ parts of $\tilde w_R$ along the length
1 domain $[-\frac{R}{2}-1,-\frac{R}{2}]\times\SS^1$
and an analogous length 1 domain in the positive half cylinder.
Let us pick the $\p_s$-piece of $\Ff_H$ to illustrate
what happens.\footnote{
  Concerning the other piece of $\Ff_H$, check how
  $J(\cdot)\p_t\cdot-\nabla H_t(\cdot)$ applied to
  $\exp_y\beta\xi$ approaches $J(x)\dot x-\nabla H_t(x)=0$,
  as $s\to\infty$.
  }
The derivative $\p_s(\exp_y\beta\xi)$
is a sum of two terms: Term one approaches zero, as
$s\to\infty$, since $u_s$ approaches $y$ uniformly in $t$
and term two approaches zero, as
$s\to\infty$, since $\p_su_s$ does.
By \texttt{(exp.~decay)} in
Theorem~\ref{thm:finite-ENERGY=LIMITS-exist}
both terms converge to $0$ exponentially, as $s\to\infty$.
So the integral over the length 1 interval $[-\frac{R}{2}-1,-\frac{R}{2}]$
-- which moves to $+\infty$ with $R$ --  becomes as small as desired by
picking $R$ large.

\vspace{.1cm}\noindent
\textsc{II. Right inverse.}
By assumption both Fredholm operators $D_u$ and $D_v$,
see~(\ref{eq:D_u-lin-Floer}), are
surjective. Salamon proves in~\citerefFH[Prop.~3.9]{salamon:1999a}
that there are constants $c>0$ and $R_0>2$ such that
for any $R>R_0$ the Fredholm operator $D_R:=D_{\tilde w_R}$
based on the approximate zero is surjective as well and,
moreover, there is an injectivity
estimate for $D_R$ on the range of ${D_R}^*$, namely
$$
     \Norm{{D_R}^*\eta}_{W^{1m,p}}\le c\Norm{D_R{D_R}^*\eta}_{L^p}
$$
for every $\eta\in W^{2,p}_{\tilde w_R}$. It is, of course, crucial that
the estimate is uniform in $R$.

\begin{exercise}
Suppose $R>R_0$.
(i)~Show that $D_R{D_R}^*:W^{2,p}_{\tilde w_R}\to L^{p}_{\tilde w_R}$
is a bijection and admits a continuous inverse.
(ii)~Show that
$$
     T:={D_R}^*\left(D_R{D_R}^*\right)^{-1}:
     L^{p}_{\tilde w_R}\to W^{1,p}_{\tilde w_R}
$$
is a right inverse of $D_R$ and calculate its operator norm.
[Hint: Recall~(\ref{eq:kerD=cokerD*}).]
\end{exercise}

\vspace{.1cm}\noindent
\textsc{III. Quadratic estimates.}
To see what is meant by quadratic estimates
have a look at~\citerefFH[Prop.~4.5]{salamon:2006a}.

\vspace{.1cm}\noindent
With all three preparations I-III in place, apply to $f_R$ the
implicit function theorem in the form
of~\citerefFH[Prop.~A.3.4]{salamon:2006a} with $x_0=x_1=0$
to obtain a unique $\eta_R\in\im T$ such that
$f_R(\eta_R)=0$, equivalently, such that
$\Ff_H(w_R)=0$~where
$$
     w_R:=\exp_{\tilde w_R}\eta_R.
$$

This concludes the construction of the gluing map
and thereby the proof of
Proposition~\ref{prop:HF-bound-op} ($\p^2=0$).

\subsection{Continuation}\label{sec:FH-continuation}

Suppose throughout that $J\in\Jj(M,\omega)$
is fixed and that $(H^\alpha,J)$, $(H^\beta,J)$,
and $(H^\gamma,J)$ are \emph{regular pairs}.\footnote{
  E.g. pick three elements $H^\alpha,H^\beta,H^\gamma$
  of the set $\Hhreg(J)$ provided by Theorem~\ref{thm:A-MS}.
  }
By $H^{\alpha\beta}$ or $\{H^{\alpha\beta}_{s,t}\}$ we denote
a \textbf{\Index{homotopy between Hamiltonians}}, that is
a smooth map $\R\times\SS^1\times M\to\R$
such that
$$
     H^{\alpha\beta}_{s,t}=
     \begin{cases}
        H^\alpha_t&\text{, $s\le -1$,}\\
        H^\beta_t&\text{, $s\ge +1$.}
     \end{cases}
$$
The key idea is to replace the $s$-independent
Hamiltonians in the Floer equation~(\ref{eq:FLOER})
by $s$-dependent homotopies, thereby destroying the
occasionally disturbing symmetry under $s$-shifts; see
footnote to~(\ref{eq:A-diff}).
Consider the PDE
\begin{equation}\label{eq:s-FLOER}
     \p_s u+J_t(u)\p_t u-\nabla H^{\alpha\beta}_{s,t}(u)=0
\end{equation}
for smooth cylinders $u:\R\times\SS^1\to M$.
It is called the\index{Floer equation!homotopy --} 
\textbf{homotopy Floer equation}\index{homotopy!Floer equation}
and its solutions $u$\index{trajectory!homotopy --}
\textbf{homotopy trajectories}.\index{homotopy!trajectory}
Impose the usual asymptotic boundary conditions~(\ref{eq:LIMITS})
for two periodic orbits $z^-=x^\alpha\in\Pp_0(H^\alpha)$ and
$z^+=x^\beta\in\Pp_0(H^\beta)$ of \emph{different} Hamiltonians
and denote the set of~such~$u$~by
$$
     \Mm_{x^\alpha x^\beta}=\Mm(x^\alpha,x^\beta;H^{\alpha\beta}).
$$
Just as before, for a generic, called \textbf{regular}, \textbf{homotopy}
this\index{homotopy!regular}\index{regular!homotopy}
moduli space is a smooth manifold for any choice
of $x^\alpha,x^\beta$ and the dimension is
the index difference $\CZcan(x^\alpha)-\CZcan(x^\beta)$.
There are also analogous compactness and gluing properties.
The difference is that due to the
missing invariance under shifts in the $s$-variable
one uses the \emph{index difference zero}
moduli spaces $\Mm_{x^\alpha x^\beta}$ (compact, thus
finite, sets) to define maps which are
given on $x^\alpha\in\Crit_k\Aa_{H^\alpha}$ by
\begin{equation}\label{eq:Floer-continuation-maps}
\begin{split}
     \psi^{\beta\alpha}(H^{\alpha\beta}):\CF_*(H^\alpha)
   &\to\CF_*(H^\beta)
     \\
     x^\alpha
   &\mapsto\sum_{x^\beta\in\Crit_k\Aa_{H^\beta}}\#_2(\Mm_{x^\alpha
     x^\beta}) \, x^\beta.
\end{split}
\end{equation}
The \emph{index difference one} moduli spaces
$\Mm_{x^\alpha y^\beta}$ lead to the chain map property
$$
     \psi^{\beta\alpha}\p^\alpha=\p^\beta \psi^{\beta\alpha}.
$$

\begin{figure}[h]
  \centering
  \includegraphics
                             [height=3cm]
                             {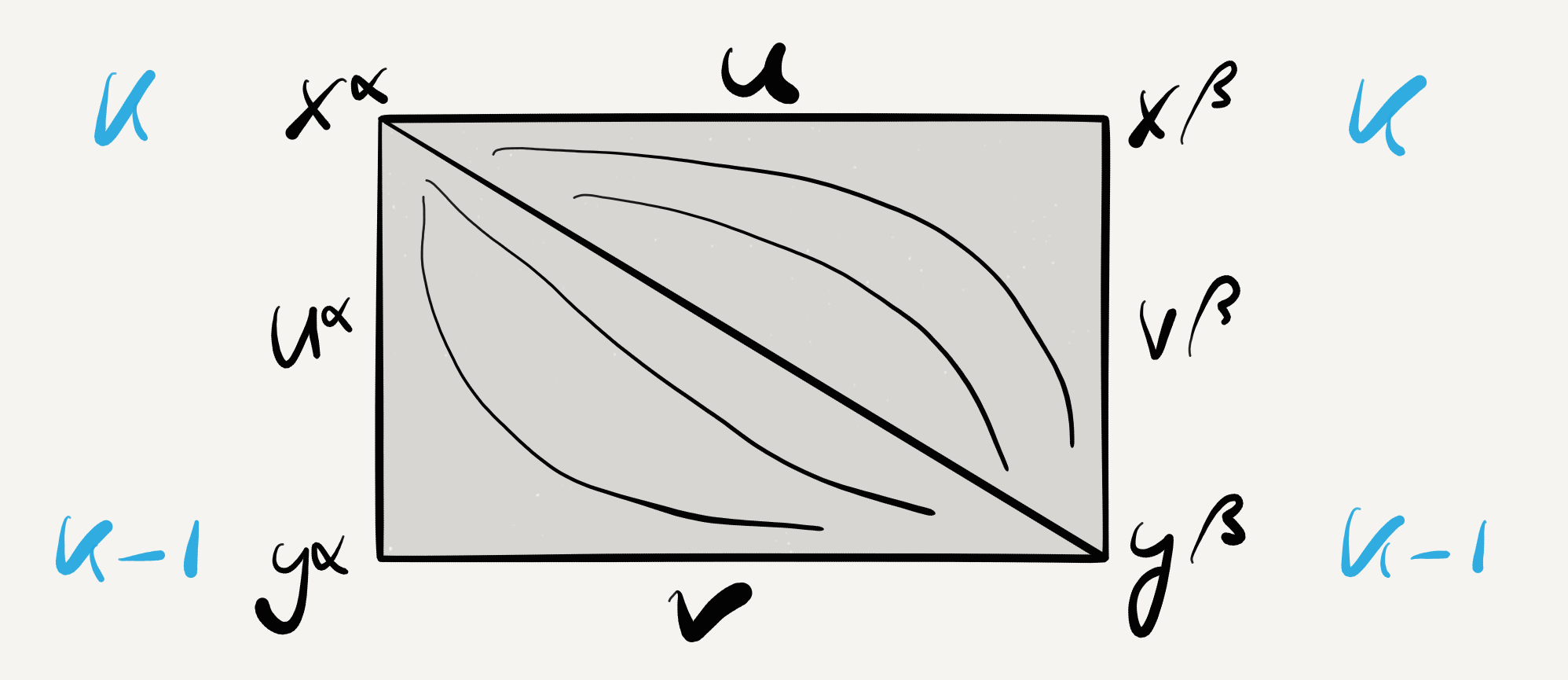}
  \caption{Partner pair property $(u^\alpha,v)\sim(u,v^\beta)$ implies
               $\psi^{\beta\alpha}\p^\alpha=\p^\beta \psi^{\beta\alpha}$}
  \label{fig:fig-partner-pair-s}
\end{figure}

This identity is equivalent to all 1-fold broken trajectories
from $x^\alpha$ to $y^\beta$ appearing as partner pairs
as indicated by Figure~\ref{fig:fig-partner-pair-s}.
Just as before the partner pair property follows from compactness
up to 1-fold broken orbits and a corresponding gluing construction.
The \textbf{induced morphism on homology}
\begin{equation}\label{eq:Floer-continuation-maps-homology}
     \Psi^{\beta\alpha}
     :=[\psi^{\beta\alpha}(H^{\alpha\beta})]:\HF_*(H^\alpha)
     \to\HF_*(H^\beta)
\end{equation}
is called\index{continuation map!Floer --}
\textbf{Floer continuation map}.\index{Floer!continuation map}
Does it depend on the homotopy?

\begin{exercise}\label{exc:const-htpy}
Denote the 
\textbf{constant homotopy}\index{homotopy!constant -- $H^\alpha$}
$H^{\alpha \alpha}\equiv H^\alpha$ again by $H^\alpha$ for simplicity.
Show that $\psi^{\alpha\alpha}(H^\alpha)=\1$ is the identity on
$\CF_*(H^\alpha)$.
\end{exercise}

\begin{lemma}[Salamon~{\citerefFH[Le.~3.11]{salamon:1999a}}]
\label{le:Salamon-homotopy-s}
Given regular homotopies $H^{\alpha\beta}$ from $H^{\alpha}$ to $H^{\beta}$ 
and $H^{\beta\gamma}$ from $H^{\beta}$ to $H^{\gamma}$, define a map
$$
     H^{\alpha\gamma}_{R}= H^{\alpha\gamma}_{R,s,t}:=
     \begin{cases}
        H^{\alpha\beta}_{s+R,t}&\text{, $s\le 0$,}\\
        H^{\beta\gamma}_{s-R,t}&\text{, $s\ge 0$,}
     \end{cases}
$$
where $R\ge 2$ is a constant; see
Figure~\ref{fig:fig-Salamon-homotopy-s}.
Then there is a constant $R_0>0$ such that for $R>R_0$
the map $H^{\alpha\gamma}_{R}$ is a regular homotopy
from $H^\alpha$ to $H^\gamma$ and the induced morphism
$$
     \psi^{\gamma\alpha}_R:\CF_*(H^\alpha)\to\CF_*(H^\gamma)
$$
is given by 
\begin{equation}\label{eq:Salamon-homotopy-s}
     \psi^{\gamma\alpha}_R=\psi^{\gamma\beta}\circ\psi^{\beta\alpha}.
\end{equation}
\end{lemma}
\begin{figure}[h]
  \centering
  \includegraphics
                             [height=3cm]
                             {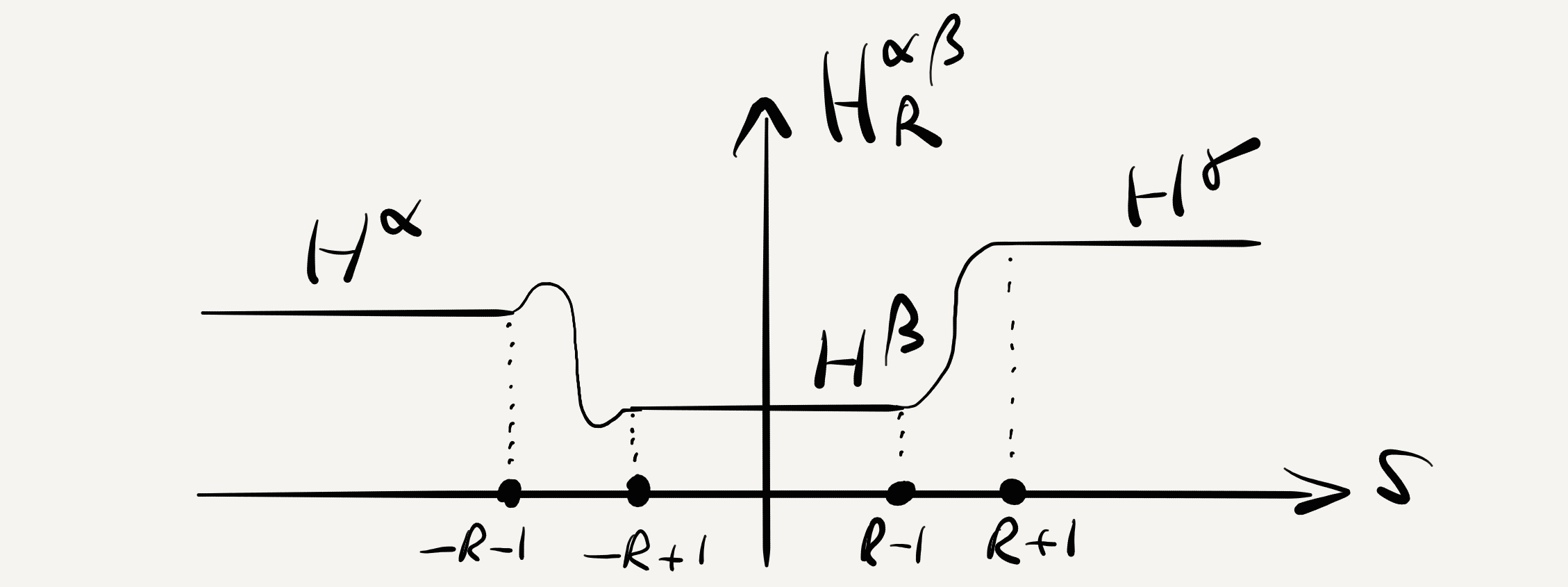}
  \caption{The homotopy $H^{\alpha\gamma}_R$ from $H^\alpha$ to
                $H^\gamma$}
  \label{fig:fig-Salamon-homotopy-s}
\end{figure}\index{partner pair property}

\begin{exercise}[Continuation maps $\Psi^{\beta\alpha}$ independent of homotopy]
\label{exc:htps-of-htps}
Given two regular homotopies $H^{\alpha\beta}_0$ and
$H^{\alpha\beta}_1$ from $H^{\alpha}$ to $H^{\beta}$,
show that $\psi^{\beta\alpha}_0$ and $\psi^{\beta\alpha}_1$
are chain homotopy equivalent: Define a homomorphism
$$
     T:\CF_*(H^\alpha)\to\CF_*(H^\beta)
$$
such that
\begin{equation}\label{eq:chain-htpy-equiv}
     \psi^{\beta\alpha}_1-\psi^{\beta\alpha}_0=
     \p^\beta T+T\p^\alpha.
\end{equation}
Note that such $T$ raises the grading by $+1$.
\newline
[Hint: Pick a regular homotopy $\{H^{\alpha\beta}_{\lambda}\}$ of
homotopies $H^{\alpha\beta}_{\lambda}=H^{\alpha\beta}_{\lambda,s,t}$
from $H^\alpha$ to $H^\beta$ which agrees with $H^{\alpha\beta}_0$
for $\lambda=0$ and with $H^{\alpha\beta}_1$ for $\lambda=1$.
In case of \emph{index difference $-1$} the
parametrized moduli
spaces\index{moduli spaces!parametrized}
$$
     \Mm(y^\alpha,x^\beta;\{H^{\alpha\beta}_{\lambda}\})
     :=\{(\lambda,u)\mid\lambda\in[0,1],
     u\in\Mm(y^\alpha,x^\beta;H^{\alpha\beta}_{\lambda})\}
$$
are 0-dimensional manifolds, in fact finite sets.
Count them appropriately to define $T$.
Analyze compactness up to 1-fold broken orbits
of the (1-dimensional) \emph{index difference $0$} moduli spaces
$\Mm(x^\alpha,x^\beta;\{H^{\alpha\beta}_{\lambda}\})$
and set up corresponding gluing maps
to prove the desired identity~(\ref{eq:chain-htpy-equiv});
cf.~\citerefFH[Pf. of Le.~3.12]{salamon:1999a}.]
\end{exercise}

\subsubsection{Proof of Theorem~\ref{thm:HF-continuation} (Continuation)}
To prove that $\Psi^{\beta\alpha}$ is an isomorphism with
inverse $\Psi^{\alpha\beta}$ pick a regular homotopy
$H^{\alpha\beta}=H^{\alpha\beta}_{s,t}$ from $H^\alpha$ to $H^\beta$
and denote the (regular) reverse homotopy by
$H^{\beta\alpha}:=H^{\alpha\beta}_{-s,t}$. With
the associated homotopy $H^{\alpha\alpha}_R$
of Lemma~\ref{le:Salamon-homotopy-s} we get
\begin{equation*}
\begin{split}
     \psi^{\alpha\beta}(H^{\alpha\beta}_{-s,t})\circ
     \psi^{\beta\alpha}(H^{\beta\alpha}_{s,t})
   &=\psi^{\alpha\alpha}(H^{\alpha\alpha}_R)\\
   &=\psi^{\alpha\alpha}(H^{\alpha})+\p^\alpha T+T\p^\alpha\\
   &=\1+\p^\alpha T+T\p^\alpha
\end{split}
\end{equation*}
where identity two is by~(\ref{eq:chain-htpy-equiv}) for
the two regular homotopies $H^{\alpha\alpha}_R$ and $H^{\alpha}$
from $H^{\alpha}$ to itself.
Identity three is by Exercise~\ref{exc:const-htpy}.
Hence $\Psi^{\alpha\beta}\Psi^{\beta\alpha}=\1$.
Repeat the argument starting with a homotopy
from $H^\beta$ to $H^\alpha$ to get
$\Psi^{\beta\alpha}\Psi^{\alpha\beta}=\1$. This shows that
$\Psi^{\beta\alpha}$ is an isomorphism with inverse $\Psi^{\alpha\beta}$.
That $\Psi^{\alpha\alpha}=\1$ follows from
Exercises~\ref{exc:const-htpy} and~\ref{exc:htps-of-htps}.
The identity $\Psi^{\gamma\beta}\Psi^{\beta\alpha}=\Psi^{\gamma\alpha}$
holds by~(\ref{eq:Salamon-homotopy-s}).
This proves Theorem~\ref{thm:HF-continuation}.

\subsection{Isomorphism to singular homology}
\label{sec:iso-sing}

Recall that we use $\Z_2$ coefficients.
After introducing in great length a new homology theory
for the data $(M,\omega,H,J)$ --
which does not depend on $(H,J)$
as we saw in the previous section on continuation --
it is natural to ask if Floer homology relates
to any known homology theory and if so, how?
The answer is that Floer homology relates to \emph{singular
homology of the closed manifold $M$ itself} by isomorphisms
$\Psi^{\eps H^\alpha}=\PsiPSShom$ compatible with the
continuation~maps~$\Psi^{\beta\alpha}$.

\subsubsection{Method 1 ($\mbf{C^2}$ small Morse functions)}
Floer observed in~\citerefFH[Thm.~2]{floer:1989c} that,
roughly speaking, if one chooses a sufficiently $C^2$ small
Morse function $h:M\to\R$ as Hamiltonian,
then not only the $1$-periodic orbits of $h$ are precisely
its critical points, see Proposition~\ref{prop:HZ-C2small},
but also the connecting Floer trajectories $u=u(s,t)$ will not depend on $t$
and turn into connecting Morse trajectories $\gamma=\gamma(s)$.

\begin{remark}[$C^2$ small: Floer trajectories reduce to
Morse -- not always!]\label{rem:t-indep-flow-lines}
Actually the situation is a bit more complex
as we said right above: Hofer and Salamon proved
in~\citerefFH[\S 7]{Hofer:1995a} that for sufficiently
$C^2$ small Morse functions all connecting trajectories
of index difference \emph{one or less}
are independent of $t$ and therefore connecting Morse
trajectories. But index difference one (and zero) \emph{is all that is
needed} in either chain complex. (See~\citerefFH[Ex.~7.2]{Hofer:1995a}
for an example how $t$-independence  fails
in case of index difference two or larger.
See also~\citerefFH[Rmk.~7.5]{Hofer:1995a} saying
that their proof does not work for symplectic manifolds of
minimal Chern number $n-1$.)
\end{remark}

Suppose from now on that $h:M\to\R$ is Morse and $C^2$ small.
A closer look shows that the Floer equation~(\ref{eq:FLOER})
actually turns for $t$-independent trajectories
$u(s,\text{\st{$t$}})=:\gamma(s)$ into the \emph{up}ward
gradient equation
\begin{equation}\label{eq:FC-equation-eq}
     \gamma^\prime=\p_su=-\grad\Aa_h(u)=\nabla h
\end{equation}
for curves $\gamma=:\R\to M$; cf. Figure~\ref{fig:fig-Floers-interpol-eq}.
Note that $\Crit\Aa_h=\Crit h$ by Proposition~\ref{prop:HZ-C2small}
and that $\Aa_h=-h$ along the critical set.
Furthermore, by~(\ref{eq:mu_H=ind_-H}) the Floer grading
$z$ as a constant periodic orbit and the Morse index of $z$
as a critical point of $h$ are related by
$\CZcan(z)=n-\IND_h(z)$.

Recall from Section~\ref{sec:Morse-cohom},
see~\cite{weber:2015-MORSELEC-In_Preparation} for details,
that the Morse cochain groups of a Morse function $h$ on a closed
manifold are generated by the critical points of $h$,
graded by their Morse index $\IND_h$,
and the coboundary operator is given by counting
the flow lines of the upward gradient $\nabla h$,
that is from critical points $x$ of index, 
say $2n-k-1=\IND_h(x)$, equivalently $k+1-n=\CZcan(x)$
to those of index $2n-k=\IND_h(y)$, equivalently $k-n=\CZcan(y)$.

\begin{figure}[h]
  \centering
  \includegraphics
                             [height=3cm]
                             {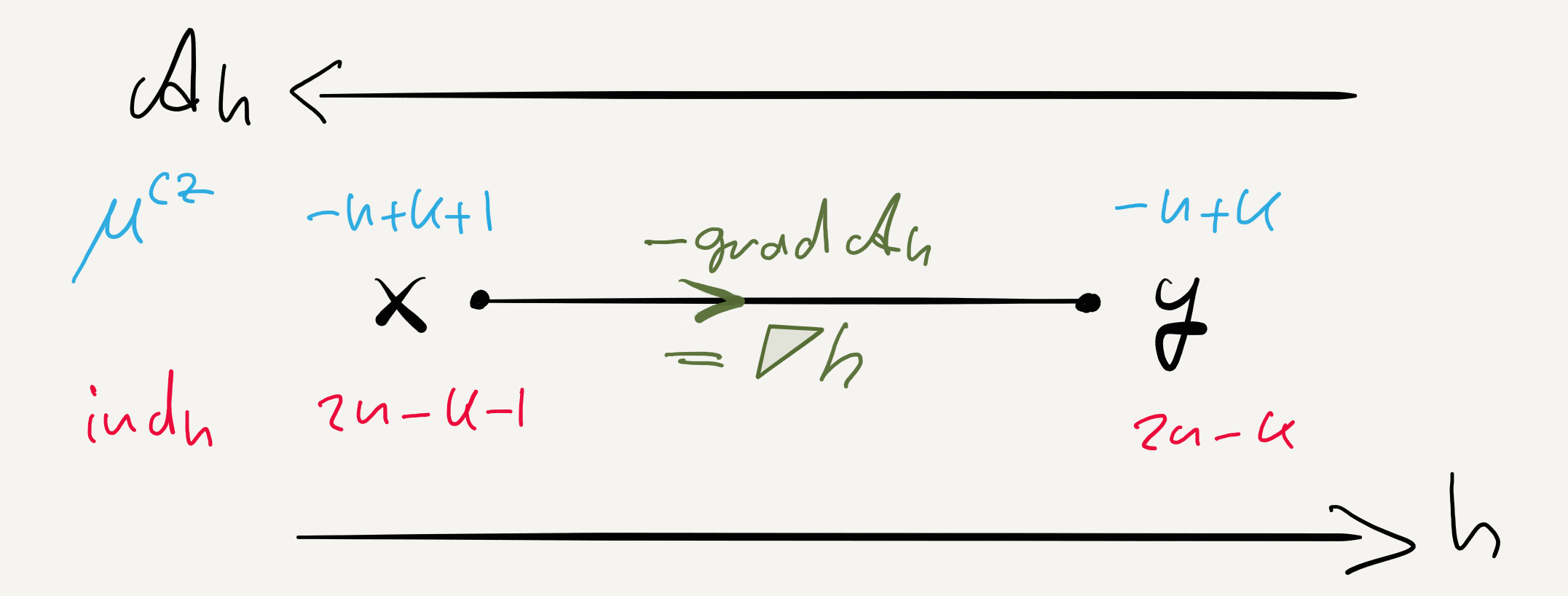}
  \caption{Downward Floer gradient degenerates to upward Morse gradient}
  \label{fig:fig-Floer-iso}
\end{figure}

Note that $\p^{\rm F}x=\delta^{\rm M}x$ by~(\ref{eq:FC-equation-eq}).
Consequently for any sufficiently
$\mbf{C^2}$ \textbf{small Morse function}, denoted for
distinction\index{$\eps h$: $C^2$ small Morse function}
by\index{Morse function!$C^2$ small --}
$$
     \eps h:M\to \R,
$$
the following chain and cochain complexes are naturally equal
\begin{equation}\label{eq:FC-M-Coho}
     \left(\CF_{*-n}(\eps h),\p^{\rm F}(\eps h)\right)
     \equiv\left(\CM^{2n-*}(\eps h),\delta^{\rm M}(\eps h)\right),
\end{equation}
up to shifting the grading, that is
$\CF_{n-\ell}(\eps h)$ is naturally equal to $\CM^{\ell}(\eps h)$.

To prove Theorem~\ref{thm:HF=H-CSF}
consider a regular pair $(H^\alpha,J^\alpha)$
and define the desired isomorphism, say denoted by $\Psi^{\eps H^\alpha}$,
by composing the following three isomorphisms:
Firstly, the Floer continuation map
$$
     \HF_{k-n}(H^\alpha)\to\HF_{k-n}(\eps h)
$$
for some $C^2$ small Morse function $\eps h$. Secondly, the
isomorphism
$$
     \Psi^{\eps h}:\HF_{k-n}(\eps h)\to \HM^{2n-k}(\eps h)
$$
induced by the chain level identity~(\ref{eq:FC-M-Coho}), followed by
Poincar\'{e} duality\footnote{
  Any symplectic manifold is naturally \emph{oriented}
  and our $M$ is \emph{closed} by assumption.
  }
\begin{equation*}\label{eq:FH-M-Coho}
\begin{split}
     \HM^{2n-k}(\eps h)
     \simeq\HM_{k}(\eps h).
\end{split}
\end{equation*}
Thirdly, the fundamental isomorphism~(\ref{eq:fund-Morse-hom})
of Morse homology given by
$$
     \HM_{k}(\eps h)
     \simeq\Ho_{k}(M)
$$
and compatible with the Morse continuation maps.
We use $\Z_2$ coefficients.

As a consequence $\Psi^{\eps H^\alpha}$ is compatible
with the Floer continuation maps
as stated in Theorem~\ref{thm:HF=H-CSF}.
See also~\citerefFH{Salamon:1992a}
and~\cite[Thm.~12.1.4]{mcduff:2004a}.
For simplicity and in view of~(\ref{eq:nat-iso-Floer-sing})
let us denote $\Psi^{\eps H^\alpha}$ by $\Psi^\alpha$ from now on.

\subsubsection{Method 2 (Spiked disks --\Index{PSS isomorphism})}
Piunikhin, Salamon, and Schwarz came up
in~\citerefFH{Piunikhin:1996a} with a rather
different idea to construct an isomorphism
between Floer and Morse homology, denoted~by
$$
     \PhiPSShom:\HM_*(M)\to\HF_{*-n}(M,\omega).
$$
Fix a Morse function $f$ and a (generic) Riemannian
metric on $M$ such that the Morse complex
is defined and pick a generator $x\in\Crit_k f$.
Suppose $(H,J)$ is a regular pair,
so Floer homology is defined, and pick a
generator $z\in\Crit_{k-n} \Aa_H$.
The idea is to relate Morse and Floer
flow lines by first following
a Morse flow line $\gamma$ that comes from $x$
and then, say at time $s=0$, change
over to a Floer flow line that goes to $z$.
Of course, the transition from a family of points
to a family of circles requires some interpolation, i.e. some
thought. The idea is to look at $J$-holomorphic
planes $v:\C\to M$, see~(\ref{eq:J-hol-plane}), and
homotop, as the polar radius $s$ of points in $\C$
traverses the interval $[1,2]$, the zero Hamiltonian inside the
unit circle (polar radius $s=1$) to the given Hamiltonian $H_t$
on and outside the circle of polar radius $s=2$.
Denote such a homotopy by $H_s=H_{s,t}$.
Moreover, one requires $v(e^{2\pi(s+i\cdot)})$ to converge to
the given periodic orbit $z$, as $s\to\infty$.
The key condition that couples the two worlds
is then that $\gamma$ meets $v$ at $s=0$.
Such a configuration, called a spiked disk,
is illustrated by Figure~\ref{fig:fig-spiked-disk}.
\begin{figure}[h]
  \centering
  \includegraphics
                             [height=3cm]
                             {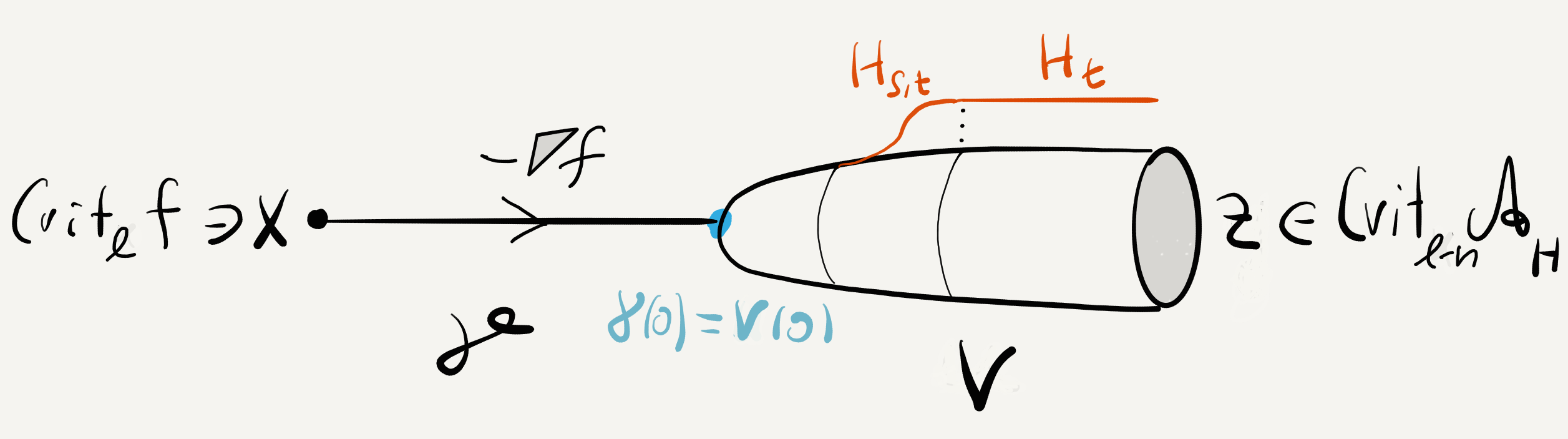}
  \caption{Counting spiked disks defines the
                 PSS chain map $\phiPSShom(H,f)$}
  \label{fig:fig-spiked-disk}
\end{figure}
Now the codimension of the set of all possible points $\gamma(0)$
is $2n-\IND_f(x)$. And the set of all $v$
which satisfy the homotopy Floer equation~(\ref{eq:s-FLOER})
for $H_s$ and converge to $z$, as $s\to\infty$,
has dimension $n-\CZcan(z)$.\footnote{
  The dimension formula $2n-\mu_H(z)$ in~\citerefFH[\S\,3.5]{salamon:1999a}
  involves $\mu_H(z):=n-\CZ(z)=n+\CZcan(z)$.
  }
Thus the moduli space of spiked disks is (for generic homotopy $H_s$)
a 0-dimensional manifold in case the Morse index of $x$
is equal to $\CZcan(z)+n$. Also there is 'no index left' for broken
flow lines, so one has compactness, thus finiteness of index
difference zero moduli spaces.
It is now clear that on the chain level the homomorphism
$$
     \phiPSShom=\phiPSShom(H,f):\CM_\ell(f)\to\CF_{\ell-n}(H)
$$
is defined by counting (modulo 2) the finitely many spiked
disks from $x$ to $z$.
The index difference one moduli spaces are
compact up to 1-fold broken orbits,
the breaking can happen on either side,
which together with a gluing construction
shows that the 1-fold broken orbits
come in pairs, namely as partner pairs.
This shows that $\phiPSShom(H,f)$ is a chain map.
But why is the induced homomorphism
$\PhiPSShom$ on homology an isomorphism?

\begin{exercise}
Replace spiked disks by disks with spikes
to define~chain~maps
$$
     \psiPSShom=\psiPSShom(f,H):\CF_{k}(H)\to\CM_{k+n}(f).
$$
Show by picture that both compositions
\[
     \phiPSShom(H,f)\circ\psiPSShom(f,H)\sim\1_{\CF}
\]
and
\[
     \psiPSShom(f,H)\circ\phiPSShom(H,f)\sim\1_{\CM}
\]
are chain homotopic to the identity.
\newline
[Hint: Draw a configuration for one of the compositions
and see how it can degenerate, that is identify the configurations
in the boundary of moduli space.
Consult~\cite[\S 12.1]{mcduff:2004a} in case you get stuck.]
\end{exercise}

\begin{exercise}
Show that both methods lead to the same isomorphisms
\begin{equation}\label{eq:nat-iso-Floer-sing}
     \PsiPSShom=\Psi^{\eps H^\alpha}=:\Psi^\alpha.
\end{equation}
[Hint: Consult~\cite[\S 12.1]{mcduff:2004a}, prior to Rmk.~12.1.7,
in case you get stuck.]
\end{exercise}

\newpage
\subsection{Action filtered Floer homology}
\label{sec:action-filtered-Floer homology}

We summarize the main features in the form of an exercise;
use $\Z_2$ coefficients.

\begin{exercise}
\label{exc:action-filtered-Floer homology}
Prove the \textbf{energy identity for connecting homotopy trajectories},
that\index{energy identity!homotopy trajectories}
is\index{homotopy trajectories!energy identity}
$$
     E(u)=\Norm{\p_s u}_{L^2}^2
     =\Aa_{H^\alpha}(x^\alpha)-\Aa_{H^\beta}(x^\beta)
     -\int_0^1\int_{-\infty}^\infty\left(\p_s H_{s,t}\right)(u)\, ds\, dt
$$
for every connecting homotopy trajectory
$u\in \Mm(x^\alpha,x^\beta;H^{\alpha\beta})$ where $E(u)$
is defined by~(\ref{eq:ENERGY-Floer}) for $H_{s,t}$.
A \textbf{\Index{monotone homotopy}}\index{homotopy~monotone}
is a homotopy $H_{s,t}$ such that $\p_s H_{s,t}\ge 0$ pointwise on
$\SS^1\times M\times\R$.
Show that for monotone homotopies
the action of $x^\alpha$ is strictly larger than that of $x^\beta$,
unless the connecting homotopy flow line $u$ is constant in $s$
in which case $u\equiv x^\alpha=x^\beta$.
Suppose that $a$ and $b$ are regular values of $\Aa_{H^\alpha}$
and define the \textbf{\Index{action filtered Floer chain groups}}\index{Floer chain groups!action filtered --}
$$
     \CF_*^{(a,b)}(H^\alpha)
$$
as usual except for only employing critical points
whose actions lie in the \textbf{\Index{action window}} $(a,b)$.
Define the boundary operator as before
in~(\ref{eq:Floer-boundary-operator}) except for only employing
critical points of action in $(a,b)$.
Show that $\p^2=0$. The homology $\HF_*^{(a,b)}(H^\alpha)$ of this
chain complex is called 
\textbf{\Index{action filtered Floer homology}}.\index{Floer homology!action filtered --}

Given a monotone homotopy $H^{\alpha\beta}$,
define the corresponding continuation map
$\psi^{\beta\alpha}(H^{\alpha\beta})$ as before
in~(\ref{eq:Floer-continuation-maps}) except for only employing
critical points of action in $(a,b)$.
Check that this defines a chain map. The induced
homomorphism $\Psi^{\beta\alpha}$ on homology is called
the\index{continuation map!monotone --}
\textbf{\Index{monotone continuation map}}.
\newline
Now consider Hamiltonians that have the property that
$H^\alpha\le H^\beta$ pointwise on $\SS^1\times M$.
Pick a smooth cutoff function $\rho:\R\to [0,1]$ which is
$0$ for $s\le-1$ and $1$ for $s\ge1$. Check that for
any Hamiltonians with $H^\alpha\le H^\beta$ the
\textbf{\Index{convex combination}}
$H^{\alpha\beta}_\rho:=(1-\rho) H^\alpha+\rho H^\beta$
is a monotone homotopy. Show that the corresponding
monotone continuation maps have composition properties
$$
     \Psi^{\gamma\beta}\Psi^{\beta\alpha}=\Psi^{\gamma\alpha}
    ,\qquad
     \Psi^{\alpha\alpha}=\1,
$$
whenever $H^\alpha\le H^\beta\le H^\gamma$
analogous to Theorem~\ref{thm:HF-continuation}.
\end{exercise}

Recall from Proposition~\ref{prop:HZ-C2small}
that on a closed symplectic manifold $M$ there
are Hamiltonians $H:\SS^1\times M\to\R$ without
non-contractible $1$-periodic orbits. So by continuation
Floer homology is trivial on components of the
free loop space of $M$ other than the component $\Ll_0 M$
consisting of contractible loops.
However, sometimes combining a suitable action window
with a geometric constraint leads to nontrivial results;
see e.g.~\citerefFH{weber:2006a} for such a situation, although
for a class of non-closed symplectic manifolds, namely cotangent bundles.

\subsection{Cohomology and Poincar\'{e} duality}\label{sec:FH-cohomology}
Recall that we work with $\Z_2$ coefficients.
Cohomology arises\index{cohomology!Floer}\index{Floer cohomology}
from homology by dualization. In Section~\ref{sec:Morse-cohom}
this is explained in detail including the geometric realization,
just replace Morse co/homology by Floer co/homology.
So here we just summarize the geometric realization of the Floer
cochain complex. Given a regular pair $(H,J)$,
the \textbf{Floer cochain group}\index{Floer!cochain group}
is the $\Z_2$ vector space
$$
     \CF^*(H):=\Hom(\CF_*(H),\Z_2).
$$
Obviously the canonical basis $\Bb_H$ of the $\Z_2$ vector space $\CF_*(H)$
is the set of generators $\Crit\Aa_H=\Pp_0(H)$, namely the finite set
of contractible $1$-periodic Hamiltonian orbits.
Hence the (finite) dual set
$$
     \Bb_H^\#:=\{\eta^x\mid x\in\Pp_0(H)\}
$$
of Dirac $\delta$-functionals\footnote{
  Define $\eta^x$ on basis elements $y\in\Bb_H$ by $\eta^x(y)=1$, if
  $y=x$, and by $\eta^x(y)=0$, otherwise.
  }
on $\CF_*(H)$ is a basis of $\CF^*(H)$, called the
\textbf{\Index{canonical basis of $\CF^*(H)$}},
and $\CF^*(H)$ is a $\Z_2$ vector space of finite dimension.
Let $\CF^k(H)$ be the subspace generated by
those Dirac functionals whose canonical Conley-Zehnder index
$\CZcan(\eta^x):=\CZcan(x)$ is equal to $k$.
As we saw in~(\ref{eq:formal-sum-678}),
the \textbf{\Index{Floer coboundary operator}}
$$
     \delta^k:=\left(\p_{k+1}\right)^\#:\CF^k(H)\to\CF^{k+1}(H)
$$
acts on basis elements by
$$
     \delta^k\eta^y
     =\sum_{x\in\Crit_{k+1} \Aa_H} \#_2(m_{xy})\,\eta^x.
$$
Here $\#_2(m_{xy})$ is the number {\rm (mod 2)}
of connecting Floer flow lines,
cf.~(\ref{eq:Floer-boundary-operator}), and $\Crit_{k} \Aa_H$
is the set of critical points $x$ with $\CZcan(x)=k$.
By Proposition~\ref{prop:HF-bound-op} we get that
$\delta^2=(\p^2)^\#=0$.

Tu put it in a nut shell, the Floer cochain groups are
generated by the contractible $1$-periodic orbits and graded by the
canonical Conley-Zehnder index, whereas the Floer coboundary
operator is given by the {\rm (mod 2)} upward count of connecting Floer
flow lines between critical points of index difference one.

The quotient space
$$
     \HF^k(H)=\HF^k(M,\omega,H;J):=\frac{\ker\delta^k}{\im\delta^{k-1}}
$$
is called the $\mbf{k^{\mathrm{th}}}$ \textbf{\Index{Floer cohomology}}
with $\Z_2$ coefficients associated to $H\in\Hhreg(J)$.
Given regular pairs $(H^\alpha,J^\alpha)$ and $(H^\beta,J^\beta)$,
continuation isomorphisms of degree zero are given by the
transposes
$$
     (\Psi^{\beta\alpha})^\#=[\psi^{\beta\alpha}(H^{\alpha\beta})^\#]
     :\HF^*(\beta)\to\HF^*(\alpha)
$$
of the continuation maps in Theorem~\ref{thm:HF-continuation};
see~(\ref{eq:Floer-continuation-maps}) and~(\ref{eq:cont-maps-trans}).
The transpose of the natural isomorphism $\Psi^\alpha$
to singular homology in Theorem~\ref{thm:HF=H-CSF}
provides the isomorphism
$$
     \left(\Psi^\alpha\right)^\#:
     \Ho^*(M)\to\HF^{*-n}(\alpha)
$$
which is of degree $-n$ and compatible with the continuation maps.

\subsubsection{\Index{Poincar\'{e} duality}}
Suppose $(H,J)$ is a regular pair and $x\in\Crit\Aa_H$ is a
contractible $1$-periodic Hamiltonian orbit.
Consider the maps
$$
     \hat x(t):=x(-t),\qquad
     \hat H_t:=-H_{-t},\qquad
     \hat J_t:=J_{-t}.
$$

\begin{exercise}
Given a regular pair $(H,J)$ and $x,y\in\Crit\Aa_H$,
show that
$$
     \CZcan(\hat x;\hat H)=-\CZcan(x;H),\qquad
     \Aa_{\hat H}(\hat \gamma)=-\Aa_H(\gamma),
$$
for every contractible loop $\gamma:\SS^1\to M$ and that
$$
     u\in\Mm(x,y;H,J)\quad\Leftrightarrow\quad
     \hat u\in\Mm(\hat y,\hat x;\hat H,\hat J)
$$
where $\hat u(s,t):=u(-s,-t)$.
\end{exercise}

\begin{exercise}\label{exc:PD-commute}
Given a regular pair $(H,J)$, pick $x\in\Crit_{k+1}\Aa_H$.
Check that the horizontal maps in the diagram
\begin{equation*}
\begin{tikzcd} [row sep=tiny, column sep=scriptsize]
  &x
    \arrow[r, mapsto]
    &\hat x
      \arrow[r, mapsto]
      &\eta^{\hat x}
\\
\Pd_{k+1}:
  &\CF_{k+1}(H)
    \arrow[r, "\widehat{}"]
    \arrow[ddd, "{\p_{k+1}(H,J)}"', "{-\grad\Aa_H(u_s)}"]
    &\CF_{-k-1}(\hat H)
      \arrow[r, "\#"]
      &\CM^{-k-1}(\hat H)
        \arrow[ddd, "{\hat\delta^{-k-1}(\hat H,\hat J)}"]
\\
\mbox{}&&&\mbox{}
\\
\mbox{}&&&\mbox{}
\\
\Pd_{k}:
  &\CF_{k}(H)
    \arrow[r, "\widehat{}"]
    &\CF_{-k}(\hat H)
      \arrow[uuu, cyan, "{\substack{-\grad\Aa_{\hat H}(\hat u_s)
         \\=\grad\Aa_{H}(u_s)}}"', "{\hat\p_{-k}}"]
      \arrow[r, "\#"]
      &\CF^{-k}(\hat H)
\\
  &\sum_y \#_2 m_{xy} \cdot y
    \arrow[r, mapsto]
    &\sum_{\hat y} \#_2 m_{xy} \cdot\hat y
      \arrow[r, mapsto]
      &\sum_{\hat y} \substack{\#_2 m_{\hat y\hat x}\\\#_2 m_{xy}} \cdot\eta^{\hat y}
\end{tikzcd}
\end{equation*}
are isomorphisms and that the diagram commutes,\footnote{
  The sums in the last line are over all critical points of
  canonical Conley-Zehnder index equal to the grading of the
  corresponding (co)chain groups in the previous line.
  }
that is
$$
     \hat\delta^{-k-1}\circ\Pd_{k+1}=\Pd_{k}\circ\p_{k+1}.
$$
[Hint: Compare the {\rm mod 2} counts $\#_2 m_{xy}(H,J)$ and
$\#_2 m_{\hat y\hat x}(\hat H,\hat J)$; cf.~(\ref{eq:Floer-boundary-operator}).]
\end{exercise}

\begin{definition}[Poincar\'{e} duality]
By Exercise~\ref{exc:PD-commute} the chain level
isomorphisms $\Pd_k^{\hat\alpha\alpha}$, where $\alpha$ abbreviates $(H,J)$,
descend to isomorphisms
$$
     \PD_k^{\hat\alpha\alpha}:=[\Pd_k^{\hat\alpha\alpha}]:
     \HF_k(\alpha)\stackrel{\simeq}{\longrightarrow}
     \HF^{-k}(\hat\alpha)
$$
which together with continuation provide the
Poincar\'{e} duality isomorphisms
$$
     \PD_{k-n}^\alpha
     :=\left(\Psi^{\hat\alpha\alpha}\right)^\#\circ\PD_{k-n}^{\hat\alpha\alpha}:
     \underbrace{\HF_{k-n}(\alpha)}_{\simeq\Ho_{k}(M)}
     \stackrel{\simeq}{\longrightarrow}
     \HF^{n-k}(\hat\alpha)\stackrel{\simeq}{\longrightarrow}
     \underbrace{\HF^{n-k}(\alpha)}_{\simeq\Ho^{2n-k}(M)}
$$
for every $k$ and any regular pair $\alpha=(H,J)$
and where $2n=\dim M$.
\end{definition}

\section{Cotangent bundles and loop spaces}
\label{sec:FH-T^*M}
Suppose $(Q,g)$ is a closed Riemannian manifold.
Pick a smooth function $V$ on $\SS^1\times M$,
called \textbf{\Index{potential energy}},
and set $V_t(q):=V(t,q)$.
For $v\in T_qQ$ we abbreviate $g_q(v,v)$
by $\abs{v}_q^2=:\abs{v}^2$.

The Lagrange function or \textbf{\Index{Lagrangian}}
$L_{V_t}(q,v)=\frac12\abs{v}_q^2-V_t(q)$, defined on $\SS^1\times TQ$,
is the difference of \emph{kinetic} and \emph{potential energy}. The
functional $\Ss_V(x):=\int_0^1 L_{V_t}(x(t),\dot x(t))\, dt$ defined
on the free loop space $\Ll Q:=C^\infty(\SS^1,Q)$
of\index{$\Ll Q:=C^\infty(\SS^1,Q)$ free loop space}
$Q$ is called
the\index{action functional!classical}
\textbf{\Index{classical action} functional}. 
Explicitly it is given by
\begin{equation}\label{eq:classical-action}
     \Ss_V=\Ss_{V,g}:\Ll Q\to\R,\quad
     x\mapsto\int_0^1\frac12\Abs{\dot x(t)}^2-V_t(x(t))\, dt.
\end{equation}
Its extremals, that is the critical points of $\Ss_V$, are the
perturbed\footnote{
  Here the \emph{perturbation} is $V$. If $V\equiv const$,
  the critical points are the closed geodesics.
  }
closed geodesics on the Riemannian manifold $Q$ given by
\[
     \Crit\,\Ss_V=\Pp(V)
     :=\{-\Nabla{t}\dot x-\nabla V_t(x)=0 \}.
\]
On the other hand, by exactness of the symplectic manifold
$(T^*Q,d\lambdacan)$ the corresponding symplectic action functional
$\Aa_H^{\lambdacan}$, see~(\ref{eq:action-can-cot}),
is naturally defined on arbitrary loops, not just contractible ones.
It is convenient to identify $TQ\simeq T^*Q$ via the isomorphism
$\xi\mapsto g(\xi,\cdot)$ provided by the Riemannian metric. For
mechanical Hamiltonians $H_{V_t}(q,v):=\frac12\abs{v}_q^2+V_t(q)$ 
on $\SS^1\times TQ$ of the form kinetic plus potential energy
the (perturbed)\index{action functional!symplectic}
\textbf{\Index{symplectic action functional}}
$$
     \Aa_V=\Aa_{V,g}:\Ll TQ\to\R,\quad
     z=(x,y)\mapsto\int_0^1 g\left( y(t),\dot x(t)\right) -H_{V_t}(z(t))\, dt,
$$
has the same critical points as $\Ss_V$, up to identification under the embedding
$$
     \natinc:\Ll Q\INTO\Ll TQ,\quad
     x\mapsto (x,\dot x)=:z_x ,\qquad
     \natinc\left(\Crit\Ss_V\right)=\Crit\Aa_V.
$$
One has the inequality\footnote{
  To memorize the inequality direction, remember that $\Ss_V$ is
  bounded below, but not $\Aa_V$.
  }
$\Aa_V(x,y)\le\Ss_V(x)$ for general loops $(x,y)$ in $TQ$
with equality $\Aa_V(x,\dot x)=\Ss_V(x)$ along the image of
$\natinc$, in particular, on critical points.
Furthermore, for generic $V$ both functionals are Morse and
the Morse index
\begin{equation}\label{eq:IND=CZ}
     \IND_{\Ss_V}(x)=\CZcan(z_x)
\end{equation}
of a critical point $x$ of $\Ss_V$ coincides with the canonical
Conley-Zehnder index, see~(\ref{eq:CZ-normalization-canonical}),
of the corresponding critical point
$z_x=(x,\dot x)$ of $\Aa_V$ whenever the vector bundle
$x^*TQ\to Q$ is orientable; otherwise a correction term $\sigma(x)=+1$
adds to $\CZcan(z_x)$. For proofs of these facts see~\citerefFH{weber:2002a}.
It turns out, see~\citerefFH{salamon:2006a}, that the downward $L^2$ gradient
equation of $\Ss_V$ is the \textbf{\Index{heat equation}}
\begin{equation}\label{eq:heat}
     \p_su-\Nabla{t}\p_t u-\nabla V_t(u)=0
\end{equation}
for smooth cylinders $u:\R\times\SS^1\to Q$ in the manifold $Q$.
Imposing on $u$ asymptotic boundary conditions
$x^\pm\in\Pp(V)$, similarly to~(\ref{eq:LIMITS}), the operators
$D_u$ obtained by linearizing the heat equation~(\ref{eq:heat}) are
Fredholm whenever the $1$-periodic Hamiltonian orbits
$x^\mp$ are non-degenerate.
In case all $x^\pm\in\Pp(V)$ are non-degenerate and all linearizations $D_u$
are surjective, in other words, if the
\textbf{\Index{Morse-Smale condition}}
holds for~(\ref{eq:heat}), then counting flow lines between critical points
of Morse index difference one, say modulo 2, defines a boundary operator
on the Morse chain groups. With $\Z_2$ coefficients these are defined by
$$
     \CM_*(\Ss_{V,g}):=\bigoplus_{x\in\Pp(V)}\Z_2 x
$$
and they are graded by the Morse index $\IND_{\Ss_V}$, whatever
coefficient ring.
With integer coefficients the Morse complex $\CM(\Ss)=(\CM_*,\p_*)$
has been constructed in~\citerefFH{weber:2013a,weber:2013b}.
In~\citerefFH{weber:2014c} it is shown that there is a natural isomorphism
$$
     \HM_*(\Ss_V)\simeq\Ho_*(\Ll Q)
$$
to singular homology of the free loop space.
In~\citerefFH{salamon:2006a} a natural isomorphism
$$
     \HF_*(T^*Q,H_V)\simeq\HM_*(\Ss_V)
$$
was established (over $\Z_2$);
concerning the integers see subsection below.
Assume $\Z_2$ coefficients.
The isomorphism of homology groups
\begin{equation}\label{eq:FH=LM}
     \HF_*(T^*Q,\omegacan,\Aa_V)
     \simeq\Ho_*(\Ll Q)
\end{equation}
is\index{isomorphism!Viterbo --}
called the \textbf{Viterbo isomorphism}; see~\citerefFH{Viterbo:1998a}
for Viterbo's approach or \citerefFH{salamon:2006a,weber:2014c}
and~\citerefFH{abbondandolo:2006b,abbondandolo:2006a} for others.
For a comparison see~\citerefFH{weber:2005a}.

\begin{exercise}[Pendulum watched by uniformly rotating observer]
\label{ex:pendulum-rotating_observer}
Consider the simplest closed manifold $Q=\SS^1$
and work out explicitly the Morse and Floer chain complexes
leading to~(\ref {eq:FH=LM}).
While one quickly sees that a pendulum subject to gravity
is described by a Hamiltonian of the form
$H_V(q,v)=\frac12\abs{v}^2+V(q)$ on $T^*\SS^1=(\R/\Z)\times\R$,
how could one change the system in order to make the potential $V$ not only
time-dependent, but even time-$1$-periodic?
\newline
[Hint: Consult~\citerefFH{Weber:1996a} in case you get
stuck.]
\end{exercise}

%
%
\subsection*{Orientations}
For general closed manifolds $Q$ the relation between the
Morse and the Conley-Zehnder index of $x\in\Pp(V)$ has
been established in~\citerefFH{weber:2002a}.

That even for orientable closed manifolds $Q$ there is a problem to construct
coherent orientations of the spaces of connecting flow lines has been
discovered by Kragh\citerefFH{2007arXiv0712.2533K};
cf.~\citerefFH{Seidel:2010a}.
The problem arises when the second Stiefel-Whitney class
does not vanish over 2-tori.
In such cases Abouzaid~\citerefFH{Abouzaid:2011a} resolved the
problem by using local coefficients
to construct the Floer homology groups of the cotangent bundle;
see also~\citerefFH{abbondandolo:2014b,Abbondandolo:2015c}
and~\citerefFH{Kragh:2013a}.
The case of general closed manifolds $Q$, orientable or not,
is treated in~\citerefFH{Abouzaid:2015a}.

\begin{remark}
Suppose $Q$ is a closed manifold.
Then $Q$ is orientable
iff the first Stiefel-Whitney class (of its tangent bundle) is
trivial, that is $w_1(Q)=0\in\Ho^1(Q;\Z_2)$;
see e.g.~\cite[p.148]{milnor:1974a}.
An orientable manifold is called \textbf{\Index{spin}}
if it carries what is called a spin structure
and this is equivalent to $w_2(Q)\in\Ho^1(Q;\Z_2)$ being trivial;
see e.g.~\cite[II Thm.~2.1]{Lawson:1989a}.
In other words, a manifold being spin is equivalent to both $w_1(Q)$
and $w_2(Q)$ being trivial.
\end{remark}

\noindent
In particular, the isomorphism~(\ref{eq:FH=LM}) holds true over the integers
\begin{itemize}
\item
  for orientable closed manifolds $Q$ such that $w_2(Q)$ vanishes on all ($2$-)tori;
\item
  in particular, for closed manifolds that carry a spin structure.
\end{itemize}

\noindent
Examples of spin manifolds $Q$ are
\begin{itemize}
\item[-]
  all closed orientable manifolds of dimension $n\le 3$;
\item[-]
  all spheres (which is non-obvious only for two-spheres);
\item[-]
  all odd complex projective spaces $\CP^{2n+1}$,
  e.g. the Riemann sphere $\CP^1$.
\end{itemize}
The even complex projective spaces $\CP^{2n}$ are not
spin, in particular, the complex projective plane $\CP^2$ is not.

\bibliographystylerefFH{alpha}
\cleardoublepage
\phantomsection
\addcontentsline{toc}{section}{References}

\begin{bibliographyrefFH}{}
\end{bibliographyrefFH}

\cleardoublepage
\phantomsection


\cleardoublepage
\phantomsection
\part{Reeb dynamics}\label{sec:Reeb-Dyn}
\chapter{Contact geometry}
\chaptermark{Contact geometry}
\label{sec:contact-geometry}

From now on we consider \emph{autonomous} Hamiltonians
$F:M\to\R$ on symplectic manifolds $(M,\omega)$
and restrict our search for periodic orbits
to closed regular level sets $S=F^{-1}(c)$
equipped with the Hamiltonian vector field $X_F$.
One says that $F$ defines $S$.
If $K$ defines $S$, too, then $X_F=fX_K$ along $S$
for some non-vanishing function $f$ on $S$.
So, given $S$, the set $\Cc(S)$ of closed flow lines $P$, called
closed characteristics of $S$, does not
depend on the defining Hamiltonian $F$.
But the natural parametrizations of these embedded
circles $P$ depend on $F$; just multiplicate $F$ by constants
$\alpha>1$ to run faster, $\alpha\in(0,1)$ to run slower,
or $\alpha<0$ to run in the opposite direction along $P$.
So in this context it doesn't make sense to fix the period.
One looks for periodic orbits, any period $\tau\not=0$.
But, even in $(\R^{2n},\omega_0)$, not every
regular level set $S$ admits a periodic orbit.

To guarantee existence of periodic orbits one imposes
geometric conditions on a closed hypersurface in
$(M,\omega)$. Firstly, co-orientability
and, secondly, existence of a contact form $\alpha$ on the hypersurface
which is compatible with the ambient symplectic manifold in the sense
that the 2-form $d\alpha$ coincides with the restriction of $\omega$.
Such co-orientable closed hypersurfaces are called
\emph{hypersurfaces of contact type} and we denote them by
$\Sigma$ or $(\Sigma,\alpha)$ for distinction
from ordinary energy surfaces $S$.
Our main reference in Chapter~\ref{sec:contact-geometry} is~\cite{hofer:2011a}.
We also recommend the excellent overview and survey,
\citerefCG{Etnyre:2006a} and~\citerefCG{Geiges:2001a}, respectively,
and the very nicely written introduction in~\citerefCG[\S 1.6]{Wendl:2015a}.

\begin{NOTATION}\label{not:mfs-Hams}
We use the notation $(M^{2n},\omega)$ for symplectic and
$(W^{2n-1},\alpha)$ for contact manifolds.
Exact symplectic manifolds are denoted by $(V,\lambda)$
with symplectic form $\omega:=d\lambda$.
For \emph{autonomous} Hamiltonians and their flows
we use the letters $F:M\to\R$ and $\phi=\phi^F$, whereas for
potentially time-dependent quantities we write $H$ and $\psi=\psi^H$;
see also Notation~\ref{not:notations_and_signs}.
Energy surfaces are closed hypersurfaces $S$ of the form $F^{-1}(c)$
where $c$ is a regular value of $F$.
Hypersurfaces of contact type in a symplectic manifold $(M,\omega)$
are denoted by $(\Sigma,\alpha)$, where the contact form $\alpha$
has to satisfy a compatibility condition with $\omega$ which
can be formulated, equivalently, in terms of existence of a Liouville vector field $\LVF$
near $\Sigma$.
We assume that the submanifolds $S$ and $\Sigma$ are closed.
\end{NOTATION}

\section{Energy surfaces in $(\R^{2\lowercase{n}},\omega_0)$}
\label{sec:en-surf-R2n}
Unless mentioned differently, let $\R^{2n}$ be equipped with the standard
symplectic form $\omega_0=d\lambda_0=:dx\wedge dy$;
cf.~(\ref{eq:R2n-omega-standard}).
In fundamental difference to Chapter~\ref{sec:FH} we
consider now \emph{autonomous} Hamiltonians $F:\R^{2n}\to\R$. 
A \textbf{\Index{hypersurface}} is a submanifold of codimension one.

\begin{definition}\label{def:bounding-hypersurface}
A hypersurface $\Sigma$ in a non-compact manifold $V$ is called
\textbf{\Index{bounding}} if it is closed and its complement
$V\setminus\Sigma$ consists of two connected components, one of which,
called the \textbf{inside}, has\index{hypersurface!inside $\interior{M}$ of bounding --}
compact\index{inside of bounding hypersurface}
closure, 
say $M$.\index{hypersurface!bounding $M$}
Then $M$ is a compact manifold-with-boundary
and $\p M=\Sigma$ is connected and closed.
In\index{$\Sigma$ bounds $M$}
this case we say that \textbf{$\mbf{\Sigma}$ bounds $\mbf{M}$}.
\end{definition}

To understand the dynamics of the flow $\phi=\phi^F$ on
$(\R^{2n},\omega_0)$ generated by the Hamiltonian
vector field $X_F$ defined by~(\ref{eq:X_H})
recall that, by level preservation~(\ref{eq:energy-preservation}),
it suffices to understand the dynamics on each level set $F^{-1}(c)$.
To have a realistic goal, still ambitious though, we only consider
compact level sets. In addition, we require that $c$
is a regular value of $F$. Hence $F^{-1}(c)\subset\R^{2n}$
is a closed submanifold of codimension one by
the\index{hypersurface!closed --}
regular value theorem.
Thus $(F^{-1}(c),\phi^F)$ is a compact smooth dynamical system.
This is true for almost every $c\in\R$. (By Sard's
theorem the non-regular values form a measure zero subset of $\R$.)
The converse is somewhat less obvious.

\begin{lemma}\label{le:S=pre-image_of_0}
Every connected closed hypersurface
$S\subset\R^m$ is co/orientable and of the form $f^{-1}(0)$ for
some smooth function $f:\R^m\to\R$ with $\nabla f\pitchfork S$.
\newline
\textbf{Smooth \Index{Jordan Brouwer separation theorem}.}
$S$ is bounding.\footnote{
  False for general manifolds $N\not=\R^{m}$
  e.g. $N=\SS^1\times\SS^1$ or $N=\SS^1\times\R$ with $S=\SS^1\times\{0\}$.}
\end{lemma}

Note that $\nabla f\pitchfork S$ tells that $0$ is a regular value of
$f$ and $[\nabla f]$ co-orients $S$.

\begin{proof}
Based on the fact that any closed hypersurface in
$\R^m$ is orientable, a construction of a function $f$ on
$\R^m$ such that $f^{-1}(0)=S$ and $\nabla f\pitchfork S$ and
such that both sets $\{f<0\}$ and $\{f>0\}$ are connected is given
in~\citerefCG{Lima:1988a}.\footnote{
  The construction of $f$ works if $S$ is any codimension one
  submanifold, compact or not, of any simply-connected
  manifold; see~\citerefCG[Rmk.]{Lima:1988a}.
  Show that 1-connected is \emph{necessary}.
  }
By compactness of $S$ there is a radius $R$ such that $S$ lies
inside the radius $R$ ball $B$ centered at the origin.
Suppose by contradiction that there are elements $x\in\{f<0\}$ and $y\in\{f>0\}$
that \emph{both} lie outside the ball $B$. Connect $x$ and $y$ by a
continuous path that lies outside $B$. Then $f$ must be zero somewhere
along the path. Contradiction.
\end{proof}

To summarize, in $\R^{2n}$ connected compact regular level sets $F^{-1}(c)$ and
connected closed hypersurfaces $S$ are the same. However, one is
related to functions, the other one to geometry. 
Whereas on $F^{-1}(c)$ dynamics arises through the ODE provided by $X_F$,
on the geometry side the dynamical
information has its description as well, namely, through integral submanifolds
of what is called the characteristic line bundle $\Ll_S\to S$.
In the following we discuss both versions.

\subsubsection{Hamiltonian dynamics on energy surfaces --
  periodic orbits}
A Hamiltonian $F:\R^{2n}\to\R$ being \emph{autonomous}
has useful consequences:
\begin{itemize}
\item[(i)]
  Level sets, called \textbf{\Index{energy levels}},
  are preserved by the Hamiltonian~flow
  $$
     \phi=\phi^F=\{\phi^F_t\}_{t\in\R}
  $$
  as we saw in~(\ref{eq:energy-preservation}). Throughout we
  assume compactness
  of level sets, so $\phi$ is indeed a \textbf{\Index{complete flow}},
  i.e. exists for all times in $\R$.
\item[(ii)]
  If $c$ is a \textbf{\Index{regular value}} of $F$,
  that\index{$S=F^{-1}(c)$ energy surface}
  is $dF$ nowhere vanishes on $F^{-1}(c)$,
  and $F^{-1}(c)$ is compact, then we call the closed codimension 1 submanifold
  $$
     S^c:=F^{-1}(c)\subset\R^{2n}
  $$
  an \textbf{\Index{energy surface}}.
  It may have finitely many connected components.
\item[(iii)]
  By Lemma~\ref{le:S=pre-image_of_0} a connected\footnote{
    For a connected hypersurface it is much easier to decide
    whether it is a level set.
    }
  closed hypersurface
  $S\subset\R^{2n}$ is a level set
  \begin{equation}\label{eq:compact-energy-surface}
     S:=S^0=F^{-1}(0)
  \end{equation}
  for some smooth function $F:\R^{2n}\to\R$ with regular value $0$.
  \begin{equation*}
     \textsf{In $\R^{2n}$ connected closed hypersurfaces are
     energy surfaces.}
  \end{equation*}
  Observe that compactness is part of our definition of an energy surface.
  Note that $X_F\not= 0$ everywhere on $S$,
  so $\phi$ admits no stationary points~on~$S$.
\item[(iv)]
  An energy surface $S$ is naturally co-oriented ($\nabla F\perp S$),
  thus oriented.
\item[(v)]
  A non-constant flow line that closes up 
  is an \emph{embedded circle} $P\INTO S$.
  Pick a point $x\in P$ to obtain the natural parametrization of $P$
  given by $\gamma_x:\R\to S$, $t\mapsto \phi_t x$, whose prime
  period $\tau_{\gamma_x}=:\tau_P$ is called
  \textbf{\emph{the} period of the
  \Index{closed flow line}} $P$ on $S$.
  It does not depend on the choice of $x$.
  (Check that the period groups $\Per(\gamma_x)=\Per(\gamma_y)$
  are equal for any two points $x$ and $y=\phi_T x$ on $P$.)
\end{itemize}

Energy preservation shows that Hamiltonian systems
describe physical systems without friction, so for instance
oscillations never decrease. This indicates that Hamiltonian
flows might be rather \emph{complicated}. Indeed, as opposed
to gradient flows, under a Hamiltonian flow any particle
\emph{returns close to its origin} again and again
(for a proof see e.g.~\cite[\S 1.4]{hofer:2011a}):

\begin{theorem}[\Index{Poincar\'{e} recurrence theorem}]
\label{thm:Poincare-recurence}
Under the Hamiltonian flow $\phi$ on a (closed)
energy surface $S$ almost every\,\footnote{
  with respect to the regular measure on $S$
  associated to the induced volume form on $S$
  }
point on $S$ is a \textbf{\Index{recurrent point}}: For almost
every $x\in S$ there is a sequence $t_j\to \infty$ such that
$\lim_{j\to\infty}\phi_{t_j} x=x$.
\end{theorem}

It sounds like some points, if not many, might close up
in finite time, returning exactly to their origins.
For generic Hamiltonians of class $C^2$ this is
indeed true! This result of Pugh and
Robinson~\citerefCG{Pugh:1983a} is called the
\textbf{\Index{Closing-Lemma}}.

As opposed to this \emph{generic} phenomenon,
the general existence question is:
\begin{equation*}
  \textsf{Does any connected energy surface
  $S\subset(\R^{2n},\omega_0)$
  admit a periodic orbit?}
\end{equation*}

This question has only relatively recently been given
the answer
\begin{equation}\label{eq:RF-counter-Ham-Seif}
  \textsf{"No."}
\end{equation}
by Ginzburg and G\"urel~\citerefCG{Ginzburg:2003a}
for a $C^2$ smooth energy surface $S\subset\R^4$ diffeomorphic to
$\SS^3$; see also references therein and the featured
review~\href{http://www.ams.org/mathscinet/search/publdoc.html?pg1=INDI&s1=314002&vfpref=html&r=26&mx-pid=2031857}{MR2031857}.
This answered the Hamiltonian version of
Seifert's question from 1950, commonly known as the
\textbf{\Index{Seifert conjecture}}: \index{conjecture!Seifert --} 
Does any non-vanishing vector field
on the unit sphere $\SS^3$ admit a periodic orbit?
Restricting the classes of vector fields,
starting with the original class of $C^1$ vector fields,
one gets a hierarchy of questions.
For the history of counterexamples,\footnote{
  Examples for ``No'' are referred to as \textbf{counterexamples
  to the Seifert conjecture}.
  }
including references, we refer to the
survey~\citerefCG{Ginzburg:2001a} and also the
review~\href{http://www.ams.org/mathscinet/search/publdoc.html?pg1=INDI&s1=314002&vfpref=html&r=31&mx-pid=1909955}{MR1909955}.

\subsubsection{Geometric reformulation -- closed characteristics
on energy surfaces}
The Hamiltonian vector field $X_F=X_F^{\omega_0}$ is non-zero on an
\emph{energy surface} $S=F^{-1}(0)$ by definition.
Thus $X_F$ generates an \emph{oriented} line bundle
\begin{equation*}
\begin{tikzcd}[row sep=large]
\Ll_F=\Ll_F^{\omega_0}:\R
    \arrow[r,hook]
    &\R\cdot X_F
      \arrow[d]
    &&TS\arrow[d]
\\
  &S
    &&S
          \arrow[u, bend right=50,"X_F\not= 0"']
\end{tikzcd}
\end{equation*}
in other words, a \textbf{\Index{distribution}}\footnote{
  A \textbf{distribution of rank $\mbf{k}$} in a tangent bundle
  is a subbundle of rank $k$.
  }
of rank one in $TS$.

\begin{exercise}
Consider the inclusion $\iota:S\hookrightarrow \R^{2n}$
and the restriction $\omega_0|_S:=\iota^*\omega_0$
of the symplectic form to $S=F^{-1}(0)$. Show that
$$
     \Ll_S=\Ll_S^{\omega_0}:=\ker\omega_0|_S=\R\cdot X_F=:\Ll_F.
$$
[Hint: Show this pointwise.
Inclusion $\supset$ is easy and $\dim \ker (\omega_0|_S)_x=1$:
Exclude $=0$ by odd dimension of $S$ and $\ge 2$ by
non-degeneracy of the 2-form.]
\end{exercise}

So $\Ll_S$ is a line bundle with a non-vanishing section, namely
\begin{equation}\label{eq:char-line-bdle}
\begin{tikzcd}[row sep=large]
\Ll_S:\R
    \arrow[r,hook]
    &\ker \omega_0|_S
      \arrow[d]
\\
  &S
          \arrow[u, bend right=50,"{\text{$X_F\not= 0$ whenever $F^{-1}(0)=S$ regular}}"']
\end{tikzcd}
\end{equation}

\begin{definition}
One calls $\Ll_S$ the \textbf{\Index{characteristic line bundle}}
of the energy surface $S$ in $(R^{2n},\omega_0)$.
A \textbf{\Index{closed characteristic}} of $\Ll_S$ is a closed integral
curve of the distribution $\Ll_S$, i.e. an embedded
circle $C\subset S$, likewise denoted by $P\subset S$,
whose tangent bundle $TC$ is equal to the restriction $\Ll_S|_C$.
\end{definition}

\begin{remark}[Energy surfaces in symplectic manifolds]
The same constructions work if $S=F^{-1}(0)$ is a closed regular level set
in a general symplectic manifold $(M,\omega)$; such $S$ is called an
\textbf{\Index{energy surface} in $\mbf{(M,\omega)}$}.
However, not every connected closed hypersurface
in a manifold is a level set, as the example $S=\SS^1\times\{0\}$ in
$M=\SS^1\times\SS^1$ shows. A sufficient condition
is simply-connectedness of the manifold; see
Exercise~\ref{exc:S-level-set-1-connected}.
\end{remark}

\begin{definition}
Given a symplectic manifold $(M,\omega)$, suppose $S\subset M$
is a closed hypersurface. A function $F:M\to\R$ is called
a\index{Hamiltonian!defining --}
\textbf{\Index{defining Hamiltonian} for $\mbf{S}$}
if $S=F^{-1}(c)$ for some regular value $c$ of $F$.
Let $\Ff(S)$
be\index{defining Hamiltonians!set of --}
the\index{$\Ff(S)$ defining Hamiltonians for $S$}
\textbf{\Index{set of defining Hamiltonians} for $\mbf{S}$}.
\end{definition}

Summarizing, in $\R^{2n}$ any connected closed hypersurface $S$ is an
energy surface for some Hamiltonian, that is $\Ff(S)\not=\emptyset$.
Furthermore, for any energy surface $F^{-1}(0)=:S\subset(\R^{2n},\omega_0)$
closed characteristics of the line bundle $\Ll_S^{\omega_0}$ coincide
with closed flow lines of the Hamiltonian vector field $X_F^{\omega_0}$,
as $X_F$ is a (non-vanishing) section of $\Ll_S$.

\begin{exercise}[Line bundle $\Ll_S^{\omega_0}$
independent of defining Hamiltonian]\label{exc:hj767}
Suppose $S$ is closed hypersurface of $(\R^{2n},\omega_0)$, then
$$
     F,K\in\Ff(S)\quad\Rightarrow\quad
     \text{$X_F=f X_K$ along $S$}
$$
for some non-vanishing function $f$ on $S$.
\newline
[Hint:
  Note that $\nabla F(x)$ and $\nabla K(x)$ are
  both non-zero and orthogonal to the codimension-1
  subspace $T_x S\subset\R^{2n}$. Alternatively, their
  co-vectors are colinear, as the kernel of each is precisely $T_x S$, 
  and non-zero, as $\codim T_x S>0$.]
\end{exercise}

Since the line bundle $\Ll_S$ only depends on the closed (smooth)
hypersurface $S\subset(\R^{2n},\omega_0)$,
we denote the set of closed characteristics of $\Ll_S$ by
$
     \Cc(S)=\Cc(S;\omega_0)
$.
The\index{$\Cc(S)$ closed characteristics of energy surface $S\subset(M,\omega)$}
earlier question can now be reformulated geometrically
as follows:
\begin{equation*}
  \textsf{Is $\Cc(S)$ non-empty for any connected closed hypersurface
  $S\subset(\R^{2n},\omega_0)$?}
\end{equation*}
As we saw above, in this generality the answer is ``No''.

\subsubsection{Energy surfaces of contact type -- Weinstein conjecture}
\begin{figure}
  \centering
  \includegraphics
                             [height=4cm]
                             {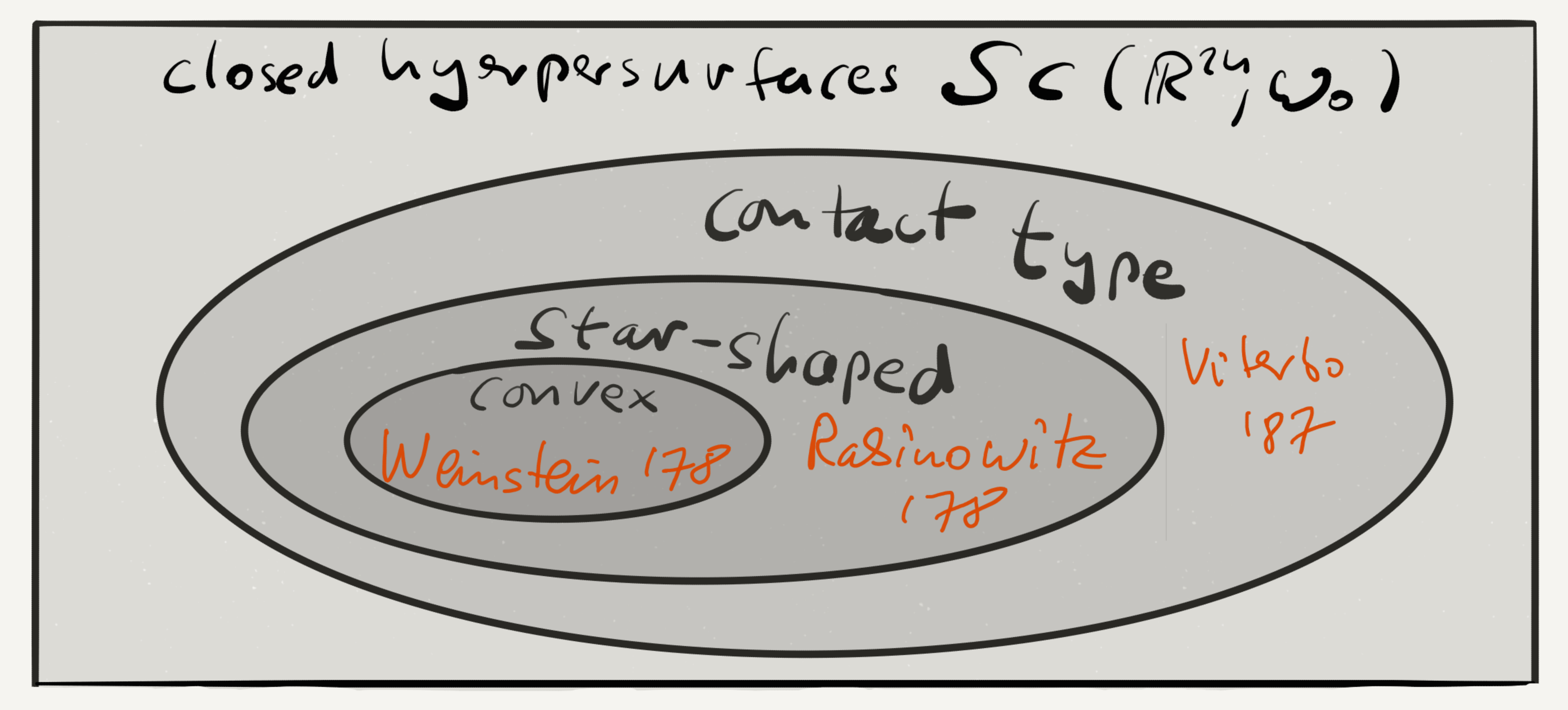}
  \caption{Classes of closed hypersurfaces $\Sigma\subset(\R^{2n},\omega_0)$
                 with ${\color{red} \Cc(\Sigma)\not=\emptyset}$}
  \label{fig:fig-Weinstein}
\end{figure}
The previous question was answered positively for convex and star-shaped
hypersurfaces by~Weinstein\citerefCG{Weinstein:1978a}
and Rabinowitz~\citerefCG{Rabinowitz:1978a}, respectively.
Weinstein isolated key geometric features of a star-shaped
hypersurface in $\R^{2n}$ and introduced the notion of
\emph{hypersurface of contact type} in a symplectic manifold.

\begin{conjecture}[\textbf{\Index{Weinstein conjecture}}~\citeintro{Weinstein:1979a}]
\label{conj:Weinstein}
  A closed hypersurface\,\footnote{
  in general, in any symplectic manifold $(M,\omega)$.
  }
  of contact type with trivial first real cohomology carries a
  closed characteristic.
\end{conjecture}
The Weinstein conjecture in $(\R^{2n},\omega_0)$
\index{conjecture!Weinstein} was confirmed by
Viterbo~\citerefCG{Viterbo:1987a}, even without any assumption
on the first cohomology.
More generally, existence of a Reeb loop for any contact form
on the 3-sphere, even for any closed orientable contact 3-manifold
with trivial $\pi_2$, was shown by Hofer~\citerefCG{Hofer:1993a}
and generalized to arbitrary $\pi_2$ by Taubes~\citerefCG{Taubes:2007a};
cf.~\citerefCG{Hutchings:2010a}.
Let us now follow Weinstein identifying the key geometric features.
Consider the radial vector~field
\begin{equation}\label{eq:LVF}
     \LVF_0:\R^{2n}\to \R^{2n},\quad
     z=(x,y)\mapsto z=\sum_{j=1}^n \left(x_j\p_{x_j}+y_j\p_{y_j}\right).
\end{equation}
The (local) flow $\Lflow=\Lflow^{\LVF_0}=\{\Lflow^{\LVF_0}_t\}$ generated by $\LVF_0$
is called \textbf{Liouville flow}.

\begin{exercise}\label{exc:Liouville-VF}
Show that $\LVF_0$ is a \emph{Liouville vector field}, i.e.
$L_{\LVF_0} \omega_0=\omega_0$.
\end{exercise}

\begin{definition}\label{def:star-shaped}
A closed hypersurface $\Sigma$ in $(\R^{2n}\setminus\{0\},\omega_0)$
is called \textbf{\Index{star-shaped}} (with respect to the origin)
if it is transverse to $\LVF_0$, see Figure~\ref{fig:figstar-shaped},
or equivalently if the projection $\R^{2n}\setminus\{0\}\to\SS^{2n-1}$,
$z\mapsto z/\norm{z}$, restricts to a diffeomorphism $\Sigma\to\SS^{2n-1}$.
\end{definition}

\begin{figure}[h]
  \centering
  \includegraphics
                             [height=3cm]
                             {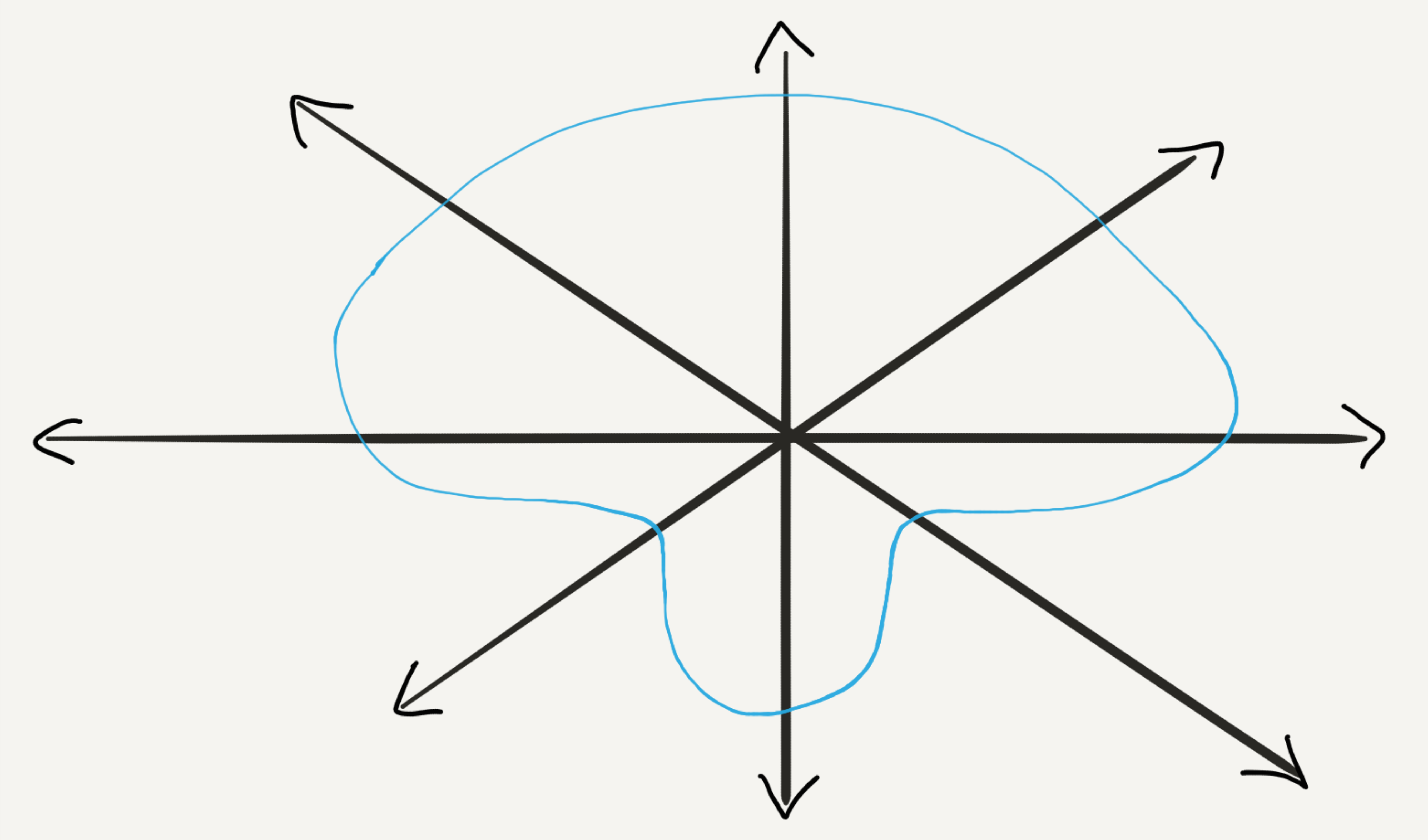}
  \caption{A hypersurface ${\color{cyan} \Sigma}\pitchfork\LVF_0$ not meeting the origin
                is called star-shaped}
  \label{fig:figstar-shaped}
\end{figure}

\begin{remark}[Contact type energy surfaces $\Sigma\subset (\R^{2n},\omega_0)$]
\label{rem:cont-type-R2n}\mbox{ }
\begin{itemize}
\item[(i)]
Existence of the Liouville vector field $\LVF_0$ has strong geometric and
dynamical consequences: Transversality of $\LVF_0$ along an energy surface
$\Sigma$ induces on some neighborhood the structure of a foliation
whose leaves are energy surfaces $\Sigma_\eps:=\Lflow_\eps \Sigma$,
i.e. Liouville flow copies of $\Sigma$.
The apriori slightly obscure condition that $\LVF_0$ is a
\textbf{symplectic dilation}, that is $L_{\LVF_0}\omega_0=\omega_0$,
causes\index{symplectic!dilation}
that the copies $\Sigma_\eps$ are even \emph{dynamical} copies in the sense that
the linearized diffeomorphisms $d\Lflow_\eps:T\Sigma\to T\Sigma_\eps$ identify
the characteristic foliations $\Ll_\Sigma:=\ker\omega_0|_\Sigma$ and $\Ll_{\Sigma_\eps}$
isomorphically; cf.~(\ref{eq:family-modelled-on-S})
and Figure~\ref{fig:fig-contact-type-char-bundles}.
\item[(ii)]
Conclusion: Given an energy surface $\Sigma\subset(\R^{2n},\omega_0)$,
the two key structures are, firstly, existence near $\Sigma$ of a
\textbf{Liouville vector field} $\LVF$,
that is having the dilation property $L_\LVF\omega_0=\omega_0$,
which is, secondly, transverse to $\Sigma$.
Such pair $(\Sigma,\LVF)$ is called an \textbf{energy surface of contact type}.
\item[(iii)]
Alternative definition:
One can show, see Section~\ref{sec:CTH-SM}, that existence of $\LVF$ in~(ii) is
equivalent to existence of a 1-form $\alpha$ \emph{on $\Sigma$ itself}
such that, firstly, the restriction $\omega_0|_\Sigma$ is $d\alpha$
(thus $\ker d\alpha$ is the characteristic line bundle
$\Ll_\Sigma$ in~(\ref{eq:char-line-bdle})),
and such that, secondly, the $1$-form $\alpha$
is non-vanishing on $\Ll_\Sigma$
(evaluation $\alpha_x (\Ll_\Sigma)_x=\R$ is non-trivial
$\forall x\in \Sigma$,\footnote{
  Thus $\rank(\xi:=\ker\alpha\to \Sigma)=2n-2$,
  hence $T \Sigma=\Ll_\Sigma\oplus\xi$, so
  $\Ll_\Sigma\pitchfork\xi$ in $T\Sigma$.
  }
see Figure~\ref{fig:fig-cont_type_hypsurf-cont_str}).
\item[(iv)]
Given data $(\Sigma,\alpha)$ as in the previous item~(iii), the two conditions
$$
     \RVF_\alpha\in\Ll_{d\alpha}:=\ker d\alpha,\qquad
     \alpha(\RVF_\alpha)=1,
$$
uniquely determine a vector field $\RVF_\alpha$ on the contact type energy surface
$\Sigma$, called the \textbf{\Index{Reeb vector field}} associated to $\alpha$.
\end{itemize}
\end{remark}

\begin{figure}
  \centering
  \includegraphics
                             [height=3.8cm]
                             {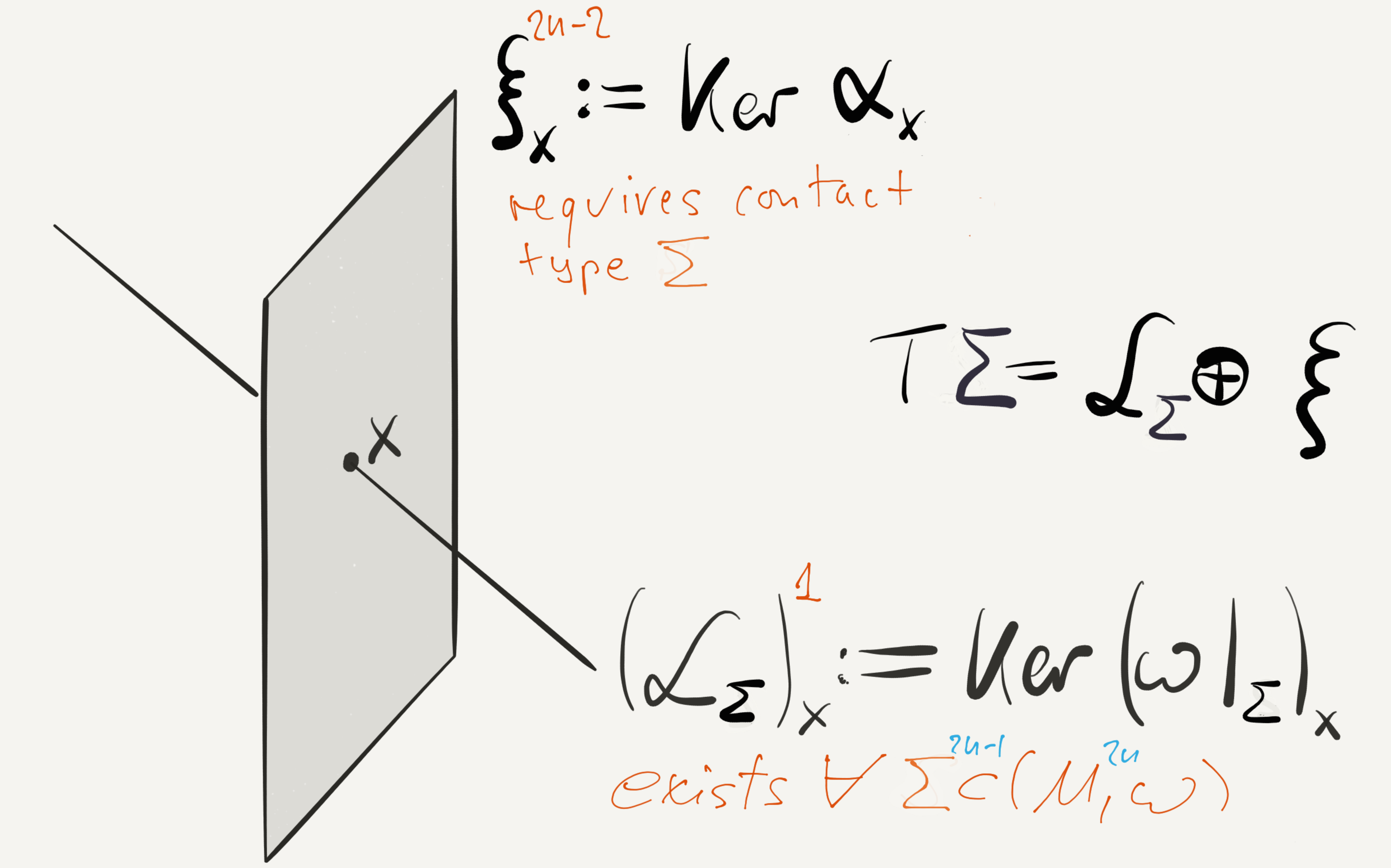}
  \caption{Contact structure $\xi:=\ker \alpha$ of contact type
                 hypersurface $\Sigma$}
  \label{fig:fig-cont_type_hypsurf-cont_str}
\end{figure}

\begin{exercise}[For contact type energy surfaces
Hamiltonian and Reeb dynamics coincide up to reparametrization]
\label{exc:Liouville-VF-en-surf}
Given the data $(\Sigma,\alpha)$ and $\RVF_\alpha$ in the previous item~(iv)
where $\Sigma=F^{-1}(0)$ for some $F:\R^{2n}\to\R$.
Show that $\RVF_\alpha=f\, X_F$ for some
non-vanishing function $f$ on $\Sigma$.
[Hint: $\Ll_F=\Ll_\Sigma=\Ll_{d\alpha}$.]
\end{exercise}

\begin{example}[Non-contact type]
An energy surface $S\subset(\R^{2n},\omega_0)$ diffeomorphic
to the sphere, but not of contact type is shown
in~\cite[\S 4.3 Fig.~4.8]{hofer:2011a}.
\end{example}

\section{Contact manifolds}

Exercise~\ref{exc:Liouville-VF-en-surf}
shows that the dynamics on a contact type hypersurface $\Sigma$
is determined, up to reparametrization,
by the $1$-form $\alpha$ on $\Sigma$ itself,
independent of the ambient symplectic manifold.

\begin{definition}\label{def:contact-mf}
A \textbf{\Index{contact form}} on a ($2n-1$)-dimensional manifold
$W$ is a 1-form $\alpha$ on $W$ such that $d\alpha$ is at any point
$x$ a non-degenerate skew-symmetric bilinear form on the subspace
$\xi_x:=\ker \alpha_x$ of the tangent space. The hyperplane distribution
$\xi$ is called a \textbf{\Index{contact structure}} and $(W,\xi)$
a \textbf{\Index{contact manifold}}.
\end{definition}

The \textbf{\Index{Gray stability theorem}}, see 
\index{theorem!Gray stability --} 
e.g.~\cite[Thm.~2.2.2]{geiges:2008a}, tells that for any given smooth
family $\{\xi_t\}_{t\in[0,1]}$ of contact structures there
is a smooth family $\{\varphi_t\}_{t\in[0,1]}$ of diffeomorphisms of
$W$ such that $(\varphi_t)_*\xi_0=\xi_t$.
But there is no such result for contact forms.
This indicates that contact forms are objects less
geometrical than contact structures.

\begin{exercise}[$\alpha$ contact
$\Leftrightarrow$ $\alpha\wedge (d\alpha)^{n-1}\not= 0$]
\label{exc:contact-volume_form}
Given a contact form $\alpha$, show:
\begin{itemize}
\item[(a)]
At every $x\in W$ the co-vector $\alpha_x$ is non-zero,
thus the vector space $\xi_x$ is necessarily of dimension $2n-2$,
in other words a hyperplane.
\item[(b)]
The defining condition of a contact structure $\xi=\ker\alpha$, namely
$d\alpha$ being non-degenerate on $\xi$, is equivalent to
$\alpha\wedge (d\alpha)^{n-1}$ being a \textbf{\textbf{volume form}} on $W$,
i.e. this $(2n-1)$-form is at no point $x$ of $M$ the zero-form.
\item[(c)]
The kernel $\Ll_{d\alpha}:=\ker d\alpha$ is a line bundle over $W$ and
$
     \Ll_{d\alpha}\oplus\ker\alpha=T W
$.
There is a unique 'unit' section $\RVF_\alpha$ of the line bundle $\Ll_{d\alpha}$
determined by $\alpha(\RVF_\alpha)=1$ and
called the \textbf{Reeb vector field} associated to $\alpha$.
The (local) flow generated by $\RVF_\alpha$ on $W$
is called \textbf{Reeb flow associated to $\alpha$} and denoted by
$\Rflow=\Rflow^{R_\alpha}=\{\Rflow^{R_\alpha}_t\}$.\footnote{
  The dynamical behavior of Reeb vector fields associated to two contact forms
  representing the same contact structure $\xi=\ker\alpha=\ker\alpha^\prime$
  is in general very different.
  }\index{Reeb!flow}\index{Reeb!vector field}\index{$\RVF_\alpha$ Reeb vector field}\index{$\Rflow^{R_\alpha}_t$ Reeb flow}
\item[(d)]
The distribution $\xi=\ker\alpha$ is nowhere integrable.\footnote{
  Hint: Frobenius; see e.g.~\cite[Ch.\,VI]{lang:2001a}
  or\index{theorem!Frobenius' --}
  \cite[Sec.~II.5]{Sternberg:1983a} or
  \cite[Ch.\,1]{warner:1983a}.
  }
\item[(e)]
The distribution $\xi=\ker\alpha$ is co-oriented.
\end{itemize}
\end{exercise}

\begin{exercise}[Contact manifolds are orientable]
Suppose $(W,\xi=\ker\alpha)$ is a contact manifold of
dimension $2n-1$.
Show that $\alpha$ induces an orientation of $W$.
If $n$ is even, then this orientation only depends on the
hyperplane distribution~$\xi$, but not on the choice of
contact form whose kernel is $\xi$.
[Hint: Pick $-\alpha$.]
\end{exercise}

\begin{exercise}[Reeb flow preserves contact structure]
\label{exc:Reeb-preserves-contact}
Show that the Reeb flow preserves the contact form $\alpha$, thus
the contact structure $\xi$.
\newline
[Hint: Check that $L_{R_\alpha}\alpha=0$.]
\end{exercise}

\begin{exercise}[Standard contact structure on $\R^3$]
Consider $\R^3$ with coordinates $(x,y,z)$
and set $\alpha:=dz-ydx$ and $\xi:=\ker \alpha$.
Check that $\alpha$ is a contact form on $\R^3$
and $\xi$ is spanned by the vector fields $\p_y$
and $\p_x+y\p_z$ whose commutator is the vector field
$R_\alpha=\p_z$ that is transverse to the plane field $\xi$.
\end{exercise}

\section{Energy surfaces $S$ in $(M,\omega)$}\label{sec:en-surf-V}

Throughout $(M,\omega)$ denotes a symplectic manifold of dimension $2n$.
This section parallels Section~\ref{sec:en-surf-R2n}
on \emph{energy surfaces $F^{-1}(0)$ in $\R^{2n}$}. We don't repeat proofs.

Define the \textbf{\Index{characteristic line bundle}} of a closed
hypersurface $S\subset M$~by
$$
     \Ll_S:=\ker \omega|_S\to S.
$$

Whereas any connected closed hypersurface in $\R^{2n}$ is a regular level
set of some Hamiltonian $F:\R^{2n}\to\R$, the situation is slightly
different in a manifold.

\begin{exercise}\label{exc:S-level-set-1-connected}
Suppose $N$ is a simply-connected manifold.
Then any connected closed hypersurface $S\subset N$
is a regular level set $S=S^0:=F^{-1}(0)$
for some function $F:B\to\R$.  [Hint: Cf. Lemma.~\ref{le:S=pre-image_of_0}.]
\end{exercise}

\begin{exercise}\label{exc:S-co-orient-ext}
For a closed hypersurface $S\subset (M,\omega)$
are equivalent:
\begin{itemize}
\item[\rm (i)]
$S$ is orientable.
\item[\rm (ii)]
$S$ is co-orientable.
\item[\rm (iii)]
$\Ll_S$ is orientable.
\item[\rm (iv)]
$S=F^{-1}(0)$ is a regular level set with $F$ defined on a neighborhood $U(S)$.
\item[\rm (v)]
There\index{parametrized family of hypersurfaces modeled on~$S$}
exists a
\textbf{parametrized family of hypersurfaces modeled on~$\mbf{S}$},
that is a diffeomorphism
\begin{equation}\label{eq:FHMS}
     \Phi:(-\delta,\delta)\times S\to U\subset M,\quad
     (\eps,x)\mapsto\Phi(\eps,x)=:\Phi_\eps(x),
\end{equation}
onto some neighborhood $U$ of $S$ that has compact closure
and such that $\Phi_0=\id_S$. We abbreviate $S_\eps:=\Phi_\eps S$
and sometimes we denote $\Phi$ by $(S_\eps)$.
\end{itemize}
\end{exercise}

\begin{remark}\label{rem:vhjgj676}
An \textbf{\Index{energy surface} in a symplectic manifold} is a
closed regular level set $S$ of the form $S=S^c:=F^{-1}(c)$ where
$F:M\to\R$ is a function and $c$ is a regular value.
We may assume that $S$ is of the form $F^{-1}(0)$; otherwise, add a
constant to $F$.
Note that an energy surface $S$ is a closed co-/orientable hypersurface
that may consist of finitely many components.
As earlier in Section~\ref{sec:en-surf-R2n},
a non-vanishing section of $\Ll_S$ is provided by the Hamiltonian
vector field $X_F$ of any Hamiltonian having $S$ as regular level set. A
\textbf{\Index{closed characteristic} on an energy surface \boldmath$S$}
is an\index{energy surface!closed flow line on --}
embedded\index{energy surface!closed characteristic on --}
circle $C\subset S$, likewise denoted by the letter $P$, 
such that $TC=\Ll_S|_C$.
Recall from~(\ref{eq:Cc-clos-char-3})
that $\Cc(S)=\Cc(S;\omega)$ denotes the set of closed characteristics
on the energy surface $S$.
\end{remark}

\subsection{Stable hypersurfaces}

\begin{definition}\label{def:stable-hypersurface}
A closed hypersurface $S\subset(M,\omega)$
is called \textbf{stable} if there is
a\index{stable hypersurface}\index{hypersurface!stable}
parametrized family $(S_\eps)$ modeled on $S$
such that for each $\eps$ the linearization
$$
     d\Phi_\eps:\Ll_S\to\Ll_{S_\eps}
$$
is a bundle isomorphism.
\end{definition}

\section{Contact type hypersurfaces $\Sigma$ in $(M,\omega)$}
\label{sec:CTH-SM}

In~\citeintro{Weinstein:1979a} Weinstein introduced the following notion; cf.
Conjecture~\ref{conj:Weinstein}.

\begin{definition}\label{def:equiv-cont-type}
A closed hypersurface $\iota:\Sigma\hookrightarrow(M,\omega)$
in a symplectic manifold is said to be of \textbf{\Index{contact type}}
if $\Sigma$ is co-orientable\footnote{
  Corresponds to $\LVF\pitchfork \Sigma$ in the alternative
  Definition~\ref{def:contact-type}.
  Moreover, co-orientability of $\Sigma$ enables extensions from $\Sigma$ to
  neighborhoods; cf. Exercise~\ref{exc:S-co-orient-ext}.
  }
and there is a 1-form $\alpha$ on $\Sigma$ such that
\begin{itemize}
\item[(i)]
  $d\alpha=\iota^*\omega$, that is the restriction
  $\omega|_\Sigma:=\iota^*\omega$ to $\Sigma$ is exact.
\item[(ii)]
  $\alpha$ is non-vanishing on the characteristic line bundle
  (except zero section)
  $$
     \Ll_\Sigma:=\ker\omega|_\Sigma=\ker d\alpha.
  $$
\end{itemize}
\end{definition}

More precisely, condition~(ii) means that evaluation
$\alpha_x (\Ll_\Sigma)_x=\R$ is non-trivial at any point $x\in\Sigma$,
likewise (cf. Figure~\ref{fig:fig-cont_type_hypsurf-cont_str})
$$
     \Ll_\Sigma\oplus\ker\alpha=T \Sigma.
$$

\begin{exercise}[Contact type hypersurfaces are contact manifolds]
\label{exc:cont-type=xi}
To see that a contact type hypersurface $(\Sigma,\alpha)$
in a symplectic manifold $(M,\omega)$ 
is a contact manifold with contact structure $\xi=\ker\alpha$
show the following.
\newline
a)~The linear functional $\alpha_x\in T_x^*\Sigma$ is non-zero
at any point $x\in\Sigma$.
So $\xi:=\ker \alpha$ is a $(2n-2)$-plane distribution
in the tangent bundle $T\Sigma$. 
\newline 
b)~Show that $d\alpha$ restricts
to a non-degenerate two-form on $\xi$, denoted by $d\alpha|\xi$.
\newline
c)~Show that $\alpha\wedge(d\alpha)^{\wedge(n-1)}$ is a volume form on
$\Sigma$. (Thus such $\alpha$ orients $\Sigma$.)
\newline
[Hint: b)~$\Ll_\Sigma=\ker d\alpha$.
c)~$\ker d\alpha\oplus\ker\alpha=T\Sigma$.
Recall Exercise~\ref{exc:contact-volume_form}~(b).]
\end{exercise}

\begin{remark}[Reeb flow on contact type hypersurface]
The characteristic line bundle of a contact type hypersurface $(\Sigma,\alpha)$
admits by co-orientability a natural section, namely the Reeb vector field $R_\alpha$
normalized by $\alpha(R_\alpha)\equiv 1$.
\end{remark}

There is a second, of course equivalent, definition of contact type
which reveals more of the interaction of contact type with the ambient
symplectic manifold. The key element is a vector field $\LVF$ transverse to
$\Sigma$ -- a useful property to generate copies $\Sigma_\eps$ of
$\Sigma$ -- and which \emph{dilates the symplectic form}:

\begin{definition}\label{def:Liouville-VF}
A vector field $\LVF$ on a symplectic manifold
$(M,\omega)$
is\index{$\LVF$ Liouville vector field}
called\index{Liouville!vector field}
a\index{Euler!vector field}
\textbf{Liouville vector field}, or an \textbf{Euler vector field},
if $L_\LVF\omega=\omega$.
\end{definition}

\begin{exercise}\label{exc:Liouville-VF-2}
If $\LVF$ is Liouville on an open subset $U\subset M$,
then so is $\LVF+X_h$ for any function $h:U\to\R$.
\end{exercise}

\begin{definition}\label{def:contact-type}
A closed\footnote{
  Assuming \emph{closed}, that is compact and without boundary,
  makes several things so much simpler: Firstly, transversality is much
  easier to handle and, secondly, flows on $\Sigma$ are complete.
}
hypersurface $\Sigma\subset (M,\omega)$
is said to be of \textbf{\Index{contact type}} if some
neighborhood $U$ of $\Sigma$ admits a Liouville vector field $\LVF$
transverse to $\Sigma$, in symbols $\LVF\pitchfork \Sigma$.
In case of global existence, namely $U=M$, the hypersurface $\Sigma$
is said to be of \textbf{\Index{restricted contact type}}.
The\index{contact type!restricted --}
\textbf{Liouville flow} is the flow $\Lflow=\Lflow^\LVF$ of $\LVF$
on\index{Liouville!flow}\index{$\Lflow^\LVF_t$ Liouville flow}
$U$. It satisfies $\Lflow_t^*\omega=e^t\omega$ wherever it is
defined.\footnote{
  Thus a closed symplectic manifold cannot admit
  a global Liouville vector field.
  }
\end{definition}

\subsubsection{Both definitions of contact type are equivalent}
{\boldmath$\LVF\mapsto \alpha_\LVF$} :
Given a Liouville vector field $\LVF$ on $U$,
define
\begin{equation}\label{eq:alpha_Y}
     \lambda_\LVF:=i_\LVF\omega,\qquad
     \alpha_\LVF:=\iota^*\lambda_\LVF,
\end{equation}
to obtain the desired 1-form $\alpha_\LVF$ on $\Sigma$.
Abbreviate $\lambda=\lambda_\LVF$ and $\alpha=\alpha_\LVF$.
Then indeed $d\lambda=di_\LVF\omega=L_\LVF\omega=\omega$
and for any non-zero
$v\in(\Ll_\Sigma)_x:=(\ker\omega|_\Sigma)_x\subset T_x\Sigma$,
together with $\LVF(x)\notin T_x\Sigma$ as $\LVF\pitchfork\Sigma$, we have that
\begin{equation}\label{eq:hjhkd34}
     0\not=\omega(\LVF(x),v)=\lambda_x(v)=\alpha_x(v).
\end{equation}
{\boldmath$\alpha\mapsto \LVF_\alpha$} :
This way is harder as it involves in the first step to extend the 1-form $\alpha$
from $\Sigma$ to a 1-form $\lambda=\lambda_\alpha$ on some
neighborhood $U$ of $\Sigma$ such that $d\lambda=\omega|U$;
see~\cite[\S 4.3 Le.~3]{hofer:2011a}. Once one has the extension,
the identity $\lambda_\alpha=\omega(\LVF,\cdot)$ determines $\LVF=\LVF_\alpha$.
On $U$ one has $L_\LVF\omega=di_\LVF\omega=d\lambda=\omega$.
To show $\LVF\pitchfork\Sigma$ observe that now in~(\ref{eq:hjhkd34})
the right hand side is non-zero, thus $\LVF(x)$ cannot be in $T_x\Sigma$
as $v$ already is.

\subsubsection{Geometrical and dynamical consequences of contact type}
Suppose $\Sigma\subset (M,\omega)$ is a closed hypersurface
of contact type. Let $\LVF=\LVF_\lambda$ and $\lambda=\lambda_\LVF$
be the two associated structures, one corresponding to the other one
(see discussion right above), on some neighborhood $U$ of $\Sigma$.

The contact type property has powerful geometrical and dynamical
consequences illustrated by Figure~\ref{fig:fig-contact-type-char-bundles}:
\begin{figure}
  \centering
  \includegraphics
                             [height=4cm]
                             {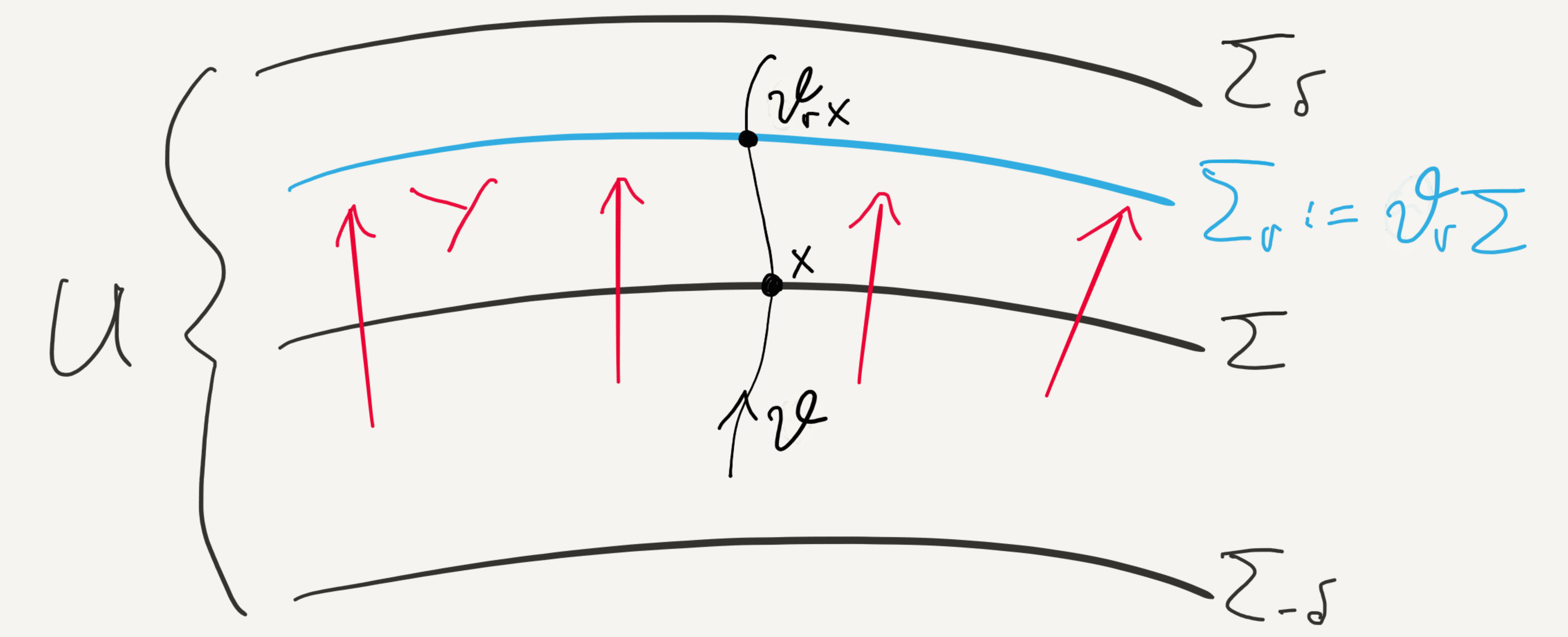}
  \caption{Contact type induces foliation $(\Sigma_r)$ and
                 isomorphisms $\Ll_\Sigma\cong\Ll_{\Sigma_r}$}
  \label{fig:fig-contact-type-char-bundles}
\end{figure}
Transversality $\LVF\pitchfork\Sigma$ leads to a foliation
of a neighborhood of $\Sigma$ by copies of $\Sigma$ under the Liouville
flow $\Lflow=\Lflow^\LVF$, that is by the hypersurfaces defined by
$\Sigma_r:=\Lflow_r \Sigma$ for $r>0$ small. These copies are even
\emph{dynamical copies} in the sense that the linearization of the
diffeomorphism $\Lflow_r:\Sigma\to \Sigma_r$ identifies the characteristic
foliations $\Ll_\Sigma$ and $\Ll_{\Sigma_r}$ isomorphically.

To see this consider the parametrized family
of hypersurfaces modeled on $\Sigma$:
\begin{equation}\label{eq:family-modelled-on-S}
     \Phi:(-\delta,\delta)\times\Sigma\to U\subset M,\quad
     (r,x)\mapsto \Lflow_r x.
\end{equation}
Here $\delta>0$ is a sufficiently small constant whose
existence is guaranteed by compactness of $\Sigma$;
choose for $U$ the image of $\Phi$.
Use $L_\LVF\omega=\omega$ in~(\ref{eq:Lie-derivative})
and $\Lflow_0=\id$ to obtain the identity
$$
     (\Lflow_r)^*\omega=e^r \omega.
$$
Given some non-zero vector $v\in(\Ll_\Sigma)_x=\ker (\omega|_\Sigma)_x$, then
$$
     0
     =\omega\left(v,w\right)
     =e^r \omega\left(v,w\right)
     =(\Lflow_r)^*\omega\left(v,w\right)
     =\omega\bigl(d\Lflow_r(x)v,
     \underbrace{d\Lflow_r(x)w}_{\in T_{\Lflow_r x}\Sigma_r}\bigr)
$$
for every $w\in T_x\Sigma$. So the non-zero vector
$d\Lflow_r(x)v$ lies in $\ker(\omega|_{\Sigma_r})_{\Lflow_r x}$. Thus
\begin{equation}\label{eq:CT=>stable}
     d\Lflow_r:\Ll_\Sigma\to \Ll_{\Sigma_r}
\end{equation}
is an isomorphism of line bundles proving that
\emph{contact type implies stable}. Hence $\Lflow_r$ induces a
bijection $\Cc(\Sigma)\cong\Cc(\Sigma_r)$, $C\mapsto\Lflow_r C$.
This motivates Definition~\ref{def:stable-hypersurface}.

\begin{exercise}[Stable, but not of contact type]
Show that the hypersurface $\Sigma:=\SS^2\times\SS^1$ in the symplectic
manifold $(M,\Omega):=(\SS^2\times\R^2,\omega\oplus \omega_0)$ is
stable, but not of contact type. Here $\omega$ is any symplectic
form on $\SS^2$; cf.~Exercise~\ref{exc:2-sphere}.
\end{exercise}

\subsection{Energy surfaces of contact type}

\begin{proposition}\label{prop:char=per}
Suppose a  closed hypersurface $\Sigma$ in a
symplectic manifold $(M,\omega)$ is both,
firstly, of contact type with respect to some Liouville vector field
$\LVF$ on some neighborhood $U$ of $\Sigma$
and, secondly,\footnote{
  This assumption is void if $\Sigma$ is bounding:
  One easily constructs $F$ even globally on $M$.
  }
a regular level set $\Sigma=F^{-1}(c)$ of some function $F:U\to\R$.
In this case the Reeb vector field and the Hamiltonian vector field
are pointwise co-linear along $\Sigma$. In symbols,
along $\Sigma$ it holds that
\begin{equation}\label{eq:non-van-f}
     \RVF_{\alpha_\LVF}=f\, X_F
\end{equation}
for some non-vanishing function $f$ on $\Sigma$.
In other words, on $\Sigma$ the Reeb flow and the Hamiltonian flow
coincide up to reparametrization.
\end{proposition}

\begin{proof}
Set $\alpha:=\alpha_\LVF$, cf.~(\ref{eq:alpha_Y}), so
$\ker d\alpha=\ker\omega|_{\Sigma}=:\Ll_{\Sigma}=\Ll_{F^{-1}(c)}$.
But $R_\alpha$ is a section of $\ker d\alpha$ by definition and $X_F$
is one of $\Ll_{F^{-1}(c)}$ by Remark~\ref{rem:vhjgj676}.
\end{proof}

\begin{exercise}[Reeb flows on level sets are Hamiltonian near\footnote{
  Reeb flows on \emph{bounding} contact type hypersurfaces are Hamiltonian
  for some $F:M\to\R$.
  }
the level set]\label{exc:Reeb=Ham}
Suppose the regular value $c$ of $F:U\to\R$ in
Proposition~\ref{prop:char=per} is \emph{zero}; otherwise replace $F$
by $F-c$. Firstly, extend the non-vanishing function $f$ in~(\ref{eq:non-van-f})
from $\Sigma=F^{-1}(0)$ to a non-vanishing function on some open
neighborhood, still denoted by $U$ and $f$, constant outside a compact
neighborhood $D\subset U$ of $\Sigma$.
[Hint: Co-orientability of $\Sigma$, tubular neighborhoods $D$.]
Secondly, show that
\begin{itemize}
\item
  zero is a regular value of the product function $fF:U\to\R$;
\item
  the pre-image $(fF)^{-1}(0)$ is still $\Sigma$;
\item
  along $\Sigma$ there are the identities
  $\RVF_{\alpha_\LVF}=f\, X_F=X_{fF}$.
\end{itemize}
[Hint: As $X_{fF}=fX_F+FX_f$ it helps that $\Sigma=F^{-1}(0)$ is the
pre-image~of~\emph{zero}.]
\end{exercise}

\section{Restricted contact type -- exact symplectic}
\label{sec:CTH-ESM}

\begin{exercise}\label{exc:restricted-CT=>exact}
Let $\Sigma$ be a closed hypersurface
in a symplectic manifold $(M,\omega)$.
\\
a)~If $\Sigma$ is of restricted contact type, see
Definition~\ref{def:contact-type}, with respect to a Liouville vector
field $\LVF$ on $M$, then $\omega$ is an exact symplectic 
form with primitive $\lambda_\LVF:=i_\LVF\omega:=\omega(\LVF,\cdot)$.
In particular, the manifold $M$ cannot be closed.
\\
b)~If $\omega=d\lambda$ is exact, then every
simply-connected closed hypersurface of contact type is of
restricted contact type.
\end{exercise}

\begin{definition}\label{def:exact symplectic manifold}\index{symplectic manifold!exact}
An \textbf{\Index{exact symplectic manifold}} $(V,\lambda)$
is a manifold $V$ with a $1$-form $\lambda$ such that
$\omega:=d\lambda$ is a symplectic form.
It automatically comes with
the\index{$\LVF_\lambda$ associated Liouville vector field}
\textbf{associated Liouville vector field} $\LVF_\lambda$
determined by\index{Liouville vector field!associated}
$i_{\LVF_\lambda}d\lambda=\lambda$.
\end{definition}

\begin{exercise}
Show $L_{\LVF_\lambda}d\lambda=d\lambda$
and the following.
An exact symplectic manifold $(V,\lambda)$ is necessarily non-compact.
The boundary $\p M$ of a compact exact symplectic manifold-with-boundary
$(M,\lambda)$ is necessarily non-empy.
\end{exercise}

\begin{exercise}[Liouville vector fields are outward pointing]
\label{exc:symp-mf-with-bdy}
Suppose $(M,\omega=d\lambda)$ is a compact exact
symplectic\index{Liouville vector field!outward pointing}
manifold-with-boundary and consider the associated Liouville vector field
$\LVF_\lambda$ defined by $\omega(\LVF_\lambda,\cdot)=\lambda$ on $U=M$.
Suppose that $\LVF_\lambda\pitchfork\p M$.
(In other words, suppose that the boundary $\p M$ is of (restricted)
contact type with respect to $\LVF_\lambda$.)
Let $\iota:\p M\hookrightarrow M$ be inclusion and set $\alpha:=\iota^*\lambda$.
\begin{itemize}
\item[(i)]
  Use the fact that $\LVF_\lambda$ lies in the kernel of $\lambda$,
  i.e. $\lambda(\LVF_\lambda)=0$, to prove the relation
  $$
     i_{\LVF_\lambda}\omega^n=\lambda\wedge(d\lambda)^{n-1}
  $$
  between the natural volume form $\omega^n$ on $M$ and its primitive.
  Recall that the restriction $\iota^*\left(\lambda\wedge(d\lambda)^{n-1}\right)
  =\alpha\wedge(d\alpha)^{n-1}$ is a volume form on $\p M$.
\item[(ii)]
  Let $M$ be equipped with the orientation provided by $\omega^n$
  and let the boundary $\p M$ be equipped with the induced orientation
  according to the 'put outward normal first' rule;
  cf.~\cite[Ch.\,3 \S 2]{guillemin:1974a}. Verify that
  \begin{equation*}
  \begin{split}
     0
   &<\int_{M}\omega^n\\
   &=\int_{\p M}\iota^*\left(\lambda\wedge(d\lambda)^{n-1}\right)\\
   &=\int_{\p M}\iota^*i_{\LVF_{\lambda}}\omega^n
  \end{split}
  \end{equation*}
  where the first identity is Stoke's theorem;
  see~\cite[Ch.\,4 \S 7]{guillemin:1974a}.
  Let $\nu$ be an outward pointing vector field along $\p M$.
  Then verify that the integral $\int_{M}\omega^n$ is a positive multiple of
  $\int_{\p M}\iota^* i_\nu \omega^n$ due to the 'put outward normal first' rule.
  So $\LVF_\lambda$ points in the same half-space
  as $\nu$, that is the outer one.
  Hence $\LVF_\lambda$ generates a complete backward flow.
\end{itemize}
\end{exercise}

\subsection{Bounding hypersurfaces of restricted contact type}
An exact symplectic manifold $(V,\lambda)$ already comes
equipped with the globally defined 
\textbf{\Index{associated Liouville vector field}} $\LVF_\lambda$
determined\index{$\LVF_\lambda$ associated Liouville vector field}
by the identity\index{Liouville vector field!associated}
$$
     d\lambda(\LVF_\lambda,\cdot)=\lambda.
$$

\begin{proposition}[Defining Hamiltonians]\label{prop:char=per-exact}
Suppose $\Sigma$ is a bounding hypersurface in an exact symplectic manifold
$(V,\lambda)$ transverse $\Sigma\pitchfork\LVF_\lambda$
to the associated Liouville vector field
$\LVF_\lambda$ on~$V$.\footnote{
  Such $\Sigma$ is of (restricted) contact type
  with $\alpha_{\lambda}:=\lambda|_\Sigma$.
}
Denote the closure of the inside of $\Sigma$ by $M$.
Then there is a global Hamiltonian $F:V\to\R$ with regular level set
$F^{-1}(0)=\Sigma$ such that $F$ is negative\footnote{
  If one does not prescribe the same sign for all $F$'s inside
  $\Sigma$ and the opposite sign outside, then one looses convexity of the
  space of such, since $F+(-F)=0$.
  }
inside $\Sigma$, equal to a positive constant outside some compact
neighborhood of $M$, and such that the Hamiltonian
and the Reeb vector field coincide
\begin{equation}\label{eq:Ham=Reeb}
     \RVF_{\alpha_\lambda}=X_F,\quad\text{along $\Sigma=F^{-1}(0)$
     where $\alpha_{\lambda}:=\lambda|_\Sigma$.}
\end{equation}
Such $F$ is called a \textbf{\Index{defining Hamiltonian}} for the
bounding (restricted) contact type hypersurface $\Sigma$;
see Figure~\ref{fig:fig-defining-Hamiltonian}.
The\index{Hamiltonian!defining}\index{defining Hamiltonians!space of --}
\textbf{\Index{space of defining Hamiltonians}}, denoted by
$\Ff(\Sigma)=\Ff(\Sigma,V,\lambda)$,
is\index{$\Ff(\Sigma)=\Ff(\Sigma,V,\lambda)$ space of defining Hamiltonians}
convex.
\end{proposition}

\begin{proof}
Use a tubular neighborhood of $\Sigma$
to define $F$ near $\Sigma$ and then, using that $\Sigma$ bounds,
extend that function, say by $-1/+1$, to the remaining parts
of the inside/outside of $\Sigma$.
Then $R_\alpha=fX_{F}$ along $\Sigma$
for some non-vanishing $f$ on $\Sigma$
by~(\ref{eq:non-van-f}).
Use again that $\Sigma$ bounds
to extend $f$ to a non-vanishing function on $V$.
Then the product function
$
     fF:V\to\R
$
has the desired properties since $d(fF)=f(dF)$ on $\Sigma$;
cf. Exercise~\ref{exc:Reeb=Ham}.
Convexity essentially follows from the identity $X_{F+G}=X_F+X_G$
and the chosen regular value being \emph{zero}.
\end{proof}

\subsection{Convexity}
One gets a natural ambience of Rabinowitz-Floer homology by replacing
\begin{equation*}
\begin{gathered}
  \textsf{bounding hypersurfaces of (restricted) contact type}\\
  \textsf{in exact symplectic manifolds}
\end{gathered}
\end{equation*}
by
\begin{equation*}
\begin{gathered}
  \textsf{convex exact hypersurfaces $\Sigma$}\\
  \textsf{in convex exact symplectic manifolds $(V,\lambda)$.}
\end{gathered}
\end{equation*}
What is the difference?
Whereas the restriction $\lambda|_\Sigma$ to $\Sigma$
needs to become contact only after adding some exact $1$-form on $\Sigma$,
non-compactness of exact symplectic manifolds is tamed and made
'controlable' outside compact parts by requiring what is called the
manifold being \emph{convex}, or \emph{cylindrical}, near infinity;
see~\citeintro{Eliashberg:2000a} and~\citerefRF{Bourgeois:2003a}.
From now on $\Sigma$ and $V$ are connected manifolds.

\begin{definition}\label{def:exact-convex-symplectic-manifold}
A \textbf{\Index{convex exact symplectic manifold}}
$(V,\lambda)$\index{symplectic manifold!convex exact --}
consists of a \underline{connected} manifold
$V$ of dimension $2n$ equipped with a 1-form $\lambda$~such~that
\begin{itemize}
\item[(i)]
  $\omega:=d\lambda$ is a symplectic form on $V$ and
\item[(ii)]
  the exact symplectic manifold $(V,\lambda)$ is \textbf{\Index{convex at infinity}},
  that is there is an exhaustion $V=\cup_k M_k$ of $V$ by compact
  manifolds-with-boundary $M_k\subset M_{k+1}$ such that
  $\alpha_k:=\lambda|_{\p M_k}$ is a contact form on $\p M_k$ for every $k$.
  \index{symplectic manifold!convex at infinity}
\end{itemize}
\end{definition}

\begin{remark}
Suppose $(V,\lambda)$ is a convex exact symplectic manifold.
  Given any \emph{compactly} supported (smooth) function $f$ on $V$,
  then $(V,\lambda+df)$ is also a convex exact symplectic manifold:
  Indeed the 1-forms $\lambda$ and $\lambda+df$, called \textbf{equivalent}
  $1$-forms, generate the same symplectic form $\omega$ on
  $V$. One obtains a suitable exhaustion by forgetting the first
  $M_k$'s, use only those on which $df=0$.
\end{remark}

By Exercise~\ref{exc:symp-mf-with-bdy}
the associated Liouville vector field $\LVF_\lambda$
points out of $M_k$ along $\p M_k$. So the Liouville
flow is automatically \emph{backward} complete.
A convex exact symplectic manifold $(V,\lambda)$ is called
\textbf{complete}\index{exact convex symplectic manifold!complete --}\index{complete!flow} if
the vector field $\LVF_\lambda$ generates a complete flow on $V$.
If $\LVF_\lambda\not= 0$ outside some compact set one says that
$(V,\lambda)$ has \textbf{\Index{bounded topology}}.\footnote{
  To understand this choice of terminology recall that Morse theory
  describes the change of topology of sublevel sets $\{f\le c\}$ when $c$
  crosses a critical level. But critical points of $f$ are the zeroes of the
  gradient vector field $\nabla f$, whatever Riemannian metric one picks.
  }
Call a subset
$A\subset V$ \textbf{displaceable}
if\index{displaceable subset}
$$
     A\cap\psi_1^H A=\emptyset
$$
for some compactly supported Hamiltonian
$H:[0,1]\times V\to\R$.

\begin{exercise}
a)~If $(V,\lambda)$ is a convex exact symplectic manifold,
then so is its \textbf{\Index{stabilization}}
$(V\times\C,\lambda\oplus\lambda_\C)$. The 1-form
$\lambda_\C$ on $\C$ is given by $\frac12(x\,dy-y\,dx)$.
b)~In $(V\times\C,\lambda\oplus\lambda_\C)$
every compact subset is displaceable.
\end{exercise}

Main examples of convex exact symplectic manifolds are
\begin{itemize}
\item
  Euclidean space $\R^{2n}$ equipped with the $1$-form
  $\lambda_0$ given by~(\ref{eq:R2n-omega-standard}).
  Indeed the radial Liouville vector field $\LVF_0(z)=z$
  in~(\ref{eq:LVF}) is transverse to the boundary of each ball $M_k$
  about the origin of radius $k$.
\item
  cotangent bundles $T^* Q$ equipped with the Liouville form
  $\lambdacan$ and the canonical fiberwise radial Liouville vector
  field $\LVFcan$, see~(\ref{eq:Liouv-VF-cot-bdl}).
  These are complete and of bounded topology whenever the base
  manifold is closed. More generally,
\item
  Stein manifolds,
  see~\citerefCG{Eliashberg:1990a,Eliashberg:1991a}
  or~\cite[Thm.~1.5]{Cieliebak:2012a}.
\end{itemize}

\subsubsection{Cylindrical ends}
Symplectic manifolds with cylindrical ends
have been introduced to construct 
\emph{\Index{symplectic field theory (SFT)}}
in~\citeintro{Eliashberg:2000a}.

\begin{exercise}\label{exc:cyl-ends}
Show that a convex exact symplectic manifold
$(V,\lambda)$ is complete and of bounded topology iff
there exists an embedding $\phi:N\times\R_+\to V$,
for some closed, not necessarily connected, manifold $N$,
such that $\phi^*\lambda=e^r\alpha_N$ with contact form
$\alpha_N:=\phi^*\lambda|_{N\times\{0\}}$ and such that
$V\setminus\phi(N\times\R_+)$ is compact.
\newline
[Hint: Apply the Liouville flow to $N:=\p M_k$ for some large $k$;
cf. Figure~\ref{fig:fig-cylindrical-ends}.]
\end{exercise}

Each connected component $N_j\times\R_+$ of $N\times\R_+$ is
called a \textbf{\Index{cylindrical end}} of $(V,\lambda)$ and comes
equipped with the symplectic form $\phi^*d\lambda=d(e^r\alpha_{N_j})$.

\begin{definition}\label{def:acs-cylindrical}
A $d\lambda$-compatible almost complex structure $J$
on $(V,\lambda)$, i.e. $J\in\Jj(V,d\lambda)$, is called
\textbf{cylindrical}
if\index{cylindrical almost complex structure}
it\index{almost complex structure!cylindrical}
is cylindrical on the cylindrical ends: Namely, the corresponding
almost complex structure $\phi^*J$ on $N\times\R_+$
\begin{itemize}
\item
  couples Liouville and Reeb vector field,
  that is $J\p_r=R_{\alpha_N}$ along $N$;
\item
  leaves $\ker\alpha_N$ invariant;
\item
  is invariant under the semi-flow $(x,0)\mapsto(x,r)$
  for $(x,r)\in N\times\R_+$.
\end{itemize}
\end{definition}

Concerning existence of cylindrical almost complex structures
see~\citerefRF[\S 2 \S 3]{Bourgeois:2003a} or~\cite[\S 2.1]{Abbas:2014a}.

\begin{figure}
  \centering
  \includegraphics
                             [height=4cm]
                             {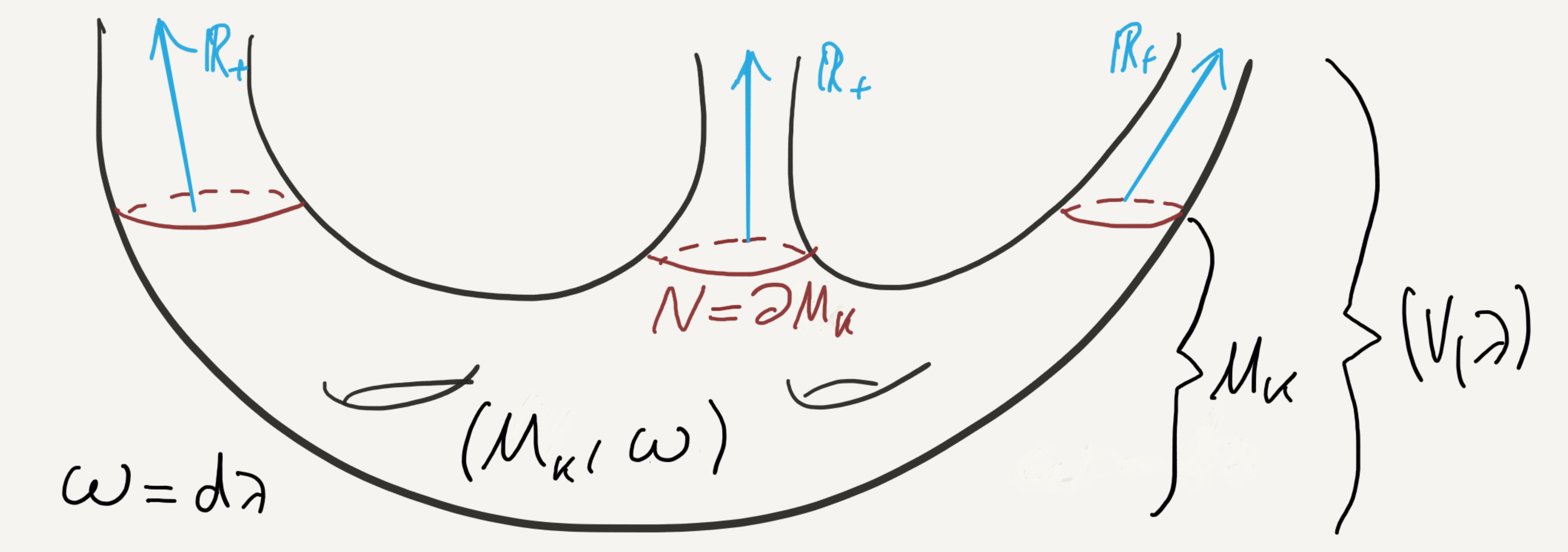}
  \caption{Convex exact symplectic manifold $(V,\lambda)$ with 3 cylindrical ends}
  \label{fig:fig-cylindrical-ends}
\end{figure}

\subsubsection{Convex exact hypersurfaces}
\begin{definition}\label{def:exact-convex-hypersurface}
A \textbf{\Index{convex exact hypersurface}} in a convex exact
symplectic manifold $(V,\lambda)$ is a \underline{connected} closed hypersurface
$\Sigma\subset V$ such that
\begin{itemize}
\item[(i)]
  there is a contact 1-form $\alpha$ on $\Sigma$ such that
  $\alpha-\lambda|_\Sigma$ is exact and
\item[(ii)]
  the hypersurface $\Sigma$ is bounding,\footnote{
    Assuming that $\Sigma$ bounds, in particular being connected and closed,
    together with being of codimension $1$ in a symplectic manifold, is a
    sufficient condition that $\Sigma$ admits a defining Hamiltonian --
    thereby relating Reeb and Hamiltonian dynamics.
    }
  say $M$. (So $V\setminus\Sigma$ has two connected components,
  one of compact closure, namely $M$, the other one not.)
\end{itemize}
\end{definition}

\begin{remark}\label{rem:ghg676}
The next Exercise~\ref{exc:equiv-1-form-restrict}
shows that a convex exact hypersurface $\Sigma$ in a
convex exact symplectic manifold $(V,\lambda)$
and with associated contact form $\alpha$ is of restricted contact
type with respect to an equivalent $1$-form $\mu:=\lambda+dh$
which restricts to the same contact form $\alpha=\mu|_\Sigma$.
Moreover, the new Liouville vector field, given by
$\LVF_\mu=\LVF_\lambda-X_h^\omega$ where $\omega=d\lambda=d\mu$,
is still transverse to $\Sigma=\p M$ and still outward pointing;
see Exercise~\ref{exc:symp-mf-with-bdy}.
\end{remark}

\begin{exercise}\label{exc:equiv-1-form-restrict}
Consider a convex exact hypersurface $\Sigma$ in $(V,\lambda)$ with
associated contact form $\alpha$. Show the following.
\begin{itemize}
\item[(a)]
  There is a compactly supported function $h:V\to\R$
  such that the $1$-form $\mu:=\lambda+ dh$ on $V$
  restricts to $\alpha=\mu |_{\Sigma}$.
\item[(b)]
  The new Liouville vector field is given by
  $\LVF_\mu=\LVF_\lambda-X_h^\omega$ and
  it is transverse to $\Sigma$ whenever $\LVF_\lambda$ is.
\end{itemize}
[Hint: (a)~Consult~\citeintro[p.\,253 Rmk.\,(2)]{Cieliebak:2009a} if you get stuck.]
\end{exercise}

For a list of further useful consequences
see~\citeintro[p.\,253]{Cieliebak:2009a}. For instance,
a closed hypersurface is bounding whenever
$\Ho_{2n-1}(V;\Z)=0$. This holds, for example, if $V$
is Stein of dimension $>2$ or a stabilization.

\subsection{Cotangent bundles}\label{sec:cot-bdls}
Given a closed manifold $Q$ of dimension $n$,
consider the cotangent bundle $(T^*Q,\omegacan=d\lambdacan)$
with its canonical exact symplectic form.

\begin{exercise}\label{exc:T^*Q-contact}
Show that $\LVFcan$ determined by $\omegacan(\LVFcan,\cdot)=\lambdacan$
takes in natural local coordinates the form
of a \emph{fiberwise radial} vector field, namely
\begin{equation}\label{eq:Liouv-VF-cot-bdl}
     \LVFcan=2\sum_{i=1}^n p_i\p_{p_i}=2\LVFfrad.
\end{equation}
It is called\index{canonical!Liouville vector field}
the\index{$\LVFcan$ canonical Liouville vector field}
\textbf{canonical}\index{Liouville vector field!canonical}
or\index{Liouville vector field!fiberwise radial}
\textbf{fiberwise radial Liouville vector field}.
The\index{canonical!fiberwise radial vector field}
\textbf{fiberwise radial vector field}
$\LVFfrad:T^*Q\to TT^*Q$
is given by\index{$\LVFfrad$ fiberwise radial vector field}
the\index{vector field!fiberwise radial --}
derivative
\begin{equation}\label{eq:Liouv-VF-vert-cot-bdl}
     \LVFcan(\eta)=2\LVFfrad(\eta):=2\left.\frac{d}{d\tau}\right|_{\tau=1} \tau\eta
\end{equation}
of the curve $\tau\eta$ in the manifold $T^* Q$ at time $1$.
Note that $\LVFfrad$ exists on any co/tangent bundle;
cf.~\href{https://en.wikipedia.org/wiki/Tangent_bundle}{wiki/Tangent\_bundle}.
\end{exercise}

\begin{definition}\label{exc:T^*Q-fiberwise-star-shaped}
A hypersurface $\Sigma\subset T^*Q$ is called
\textbf{\Index{fiberwise star-shaped}}\index{star-shaped!fiberwise}
(with respect to the zero section)
if $\Sigma$ is bounding, disjoint from the zero section, and transverse
$\Sigma\pitchfork\LVFcan$ to the fiberwise radial vector
field.\footnote{
  Why is \emph{fiberwise} radial fine, whereas
  in $\R^{2n}$ one uses the \emph{fully} radial vector field~(\ref{eq:LVF})?
  }
\end{definition}

\begin{exercise}\label{def:T^*Q-fiberwise-star-shaped}
The intersection of a fiberwise star-shaped hypersurface
$\Sigma\subset T^*Q$ with each fiber $T_q^*Q$ is diffeomorphic to a
sphere of dimension $n-1$.
\end{exercise}

\begin{exercise}\label{exc:T^*Q-kin+potential}
Pick a Riemannian metric $g$ and a smooth function $V$
on $Q$, consider the Hamiltonian
$F(q,p)=\frac12g_q(p,p)+V(q)$.
Show that if $c>\max_Q V$, then $\Sigma^c:=F^{-1}(c)$
is a fiberwise star-shaped hypersurface; cf.~\cite[(4.11)]{hofer:2011a}.
\end{exercise}

\begin{example}[Canonical contact structure on $S^*Q$]\label{ex:S*Q-contact}
The unit sphere cotangent bundle $S^*Q$ of $(Q,g)$ is fiberwise
star-shaped and of restricted contact type in
$(T^*Q,\omegacan=d\lambdacan)$.
In other words, the boundary of the unit disk cotangent bundle
$D^*Q$ of the closed Riemannian manifold $(Q,g)$ is fiberwise
star-shaped and of restricted contact type.
\end{example}

\section{Techniques to find periodic orbits}
For convenience of the reader we enlist and summarize,
following~\cite{hofer:2011a}, some key techniques to find
closed flow lines of (autonomous) Hamiltonian systems.

\subsection{Via finite capacity neighborhoods}
Based on the Hofer-Zehnder capacity function $c_0$ established
in~\cite[Ch.\,3]{hofer:2011a} one derives the following
existence results for closed Hamiltonian flow lines.

\begin{itemize}
\item
  \textsc{Nearby existence}~\cite[Thm.~4.1, p.106]{hofer:2011a}.
  Given a closed regular level set $S=S^1:=F^{-1}(1)\subset (M,\omega)$
  that admits a (bounded) neighborhood $U$ of finite capacity
  $c_0(U,\omega)<\infty$, then the set of closed characteristics
  $\Cc(S^{r_j};\omega)\not=\emptyset$ is non-empty
  for some sequence $r_j\to 1$.
  (There is even a dense subset of $(1-\eps,1+\eps)$ of such $r$'s.)

  Idea of proof: Use freedom in choosing the Hamiltonian representing $S$
  to pick a certain 'radial' one $H=H(r)$.
\item
  \textsc{Existence on $S$ itself.}
  In case the periods $T_j$ of the canonically parametrized\footnote{
    For each $p\in P_j$ the solution $z:\R\to F^{-1}(r_j)$ to $\dot
    z(t)=X_F\circ z(t)$ with $z(0)=p$ traces out $P_j$ and comes back
    to itself for the first time at some positive time, say $T_j$.
    }
  closed characteristics $P_j$ on $S^{r_j}=F^{-1}(r_j)$ are bounded, then $S$
  itself admits a periodic orbit, too.

  Idea of proof: Apply the Arzel\`{a}-Ascoli Theorem~\ref{thm:AA}.
\item
  \textsc{One-parameter families}
  \cite[Prop.~4.2, p.110]{hofer:2011a}.
  Consider a Hamiltonian loop $z^*:\R\to F^{-1}(E^*)\subset
  (M,\omega)$, said of energy $E^*$, of period $T^*$
  which admits precisely two \textbf{\Index{Floquet multipliers}}\footnote{
    the eigenvalues of the linear map $d\phi_{T^*}(p):T_pM\to T_p M$
    }
  equal to $1$.
  Application of the Poincar\'{e} continuation method 
  shows that $z^*$ belongs to a unique smooth family of
  periodic orbits $z^E$ parametrized by their energy $E$
  and whose periods $T^E$ converge to $T^*$, as $E\to E^*$.

  Idea of proof: Construct a Poincar\'{e} section map for $z^*$,
  investigate how the eigenvalues of its linearization along $z^*$
  are related to the Floquet multipliers of $z^*$,
  apply the implicit function theorem.
\end{itemize}

\subsection{Via characteristic line bundles}
Consider the characteristic line bundle $\Ll_S$,
see~(\ref{eq:char-line-bdle}),
over a closed co-orientable hypersurface $S$ in a symplectic
manifold $(M,\omega)$. Such $S=F^{-1}(0)$ is an energy surface of some
Hamiltonian $F$ defined \emph{near} $S$; see Exercise~\ref{exc:S-co-orient-ext}.
While the set of closed characteristics $\Cc(S)$ is empty in certain situations,
e.g. for the \textbf{\Index{Zehnder tori}}~\citerefCG{Zehnder:1987b},
cf.~\cite[\S 4.5]{hofer:2011a}, for large classes of closed
co-orientable hypersurfaces existence of closed characteristics
is guaranteed.

They key concept is that of a parametrized family $(S_\eps)$ of hypersurfaces
modeled on a closed hypersurface $S$, introduced earlier in~(\ref{eq:FHMS}).

\emph{Nearby existence} listed above fits into this framework
whenever $U$ is of finite $c_0$ capacity (construct the diffeomorphism
$\Phi$ using the normalized gradient flow of $F$, say with respect to an
$\omega$-compatible Riemannian metric on $M$).

Two classes of hypersurfaces which do admit periodic orbits are the following.

\begin{itemize}
\item
  \textsc{Bounding hypersurfaces.}
  Suppose $S\subset(M,\omega)$ bounds a compact submanifold-with-boundary
  $B$ and $(S_\eps)$ is a parametrized family modeled on $S$.
  Then each $S_\eps=\Phi_\eps S$ is the boundary of the symplectic
  manifold-with-boundary $B_\eps=\Phi_\eps B$.
  Now the key property is monotonicity\footnote{
    By Lebesgue's last theorem monotonicity
    of a function implies differentiability, thus Lipschitz
    continuity, almost everywhere in the sense of measure theory;
    for a proof see e.g.~\cite{pugh:2002a}.
    }
  of the function $C(\eps):=c_0(B_\eps,\omega)$ which holds
  by the \texttt{(monotonicity)}
  axiom of the Hofer-Zehnder capacity $c_0$. For details
  see~\cite[Thm.~4.3, p.116]{hofer:2011a}.
\item
  \textsc{Stable hypersurfaces.} A closed hypersurface $S\subset(M,\omega)$
  is called a \textbf{\Index{stable hypersurface}} if it admits a
  rather nice parametrized family $(S_\eps)$ modeled on $S$, namely one
  for which each linearization $d\Phi_\eps:TS\to TS_\eps$ restricts to a
  line bundle isomorphism $\Ll_S\to\Ll_{S_\eps}$.

  \textsc{Advantage.} It suffices to detect a closed
  characteristic on some member $S_\eps$ of the family
  in order to obtain a closed characteristic
  of the original dynamical system $(S=F^{-1}(0),X_F)$ itself.
  For instance, if in the nearby existence result mentioned above
  $S$ was stable, e.g. of contact type, then $\Cc(S)\not=\emptyset$.
  Since a closed hypersurface in $(\R^{2n},\omega_0)$
  admits a bounded neighborhood $U$ of finite\footnote{
    To see that $c_0(U)<\infty$,
    pick a ball around $U\subset\R^{2n}$
    and apply the axioms \texttt{(monotonicity)}
    and \texttt{(non-triviality)} of a symplectic capacity
    in~\cite[\S 1.2]{hofer:2011a}.
    }
   $c_0$ capacity, this confirms the Weinstein conjecture for contact
   type hypersurfaces in $(\R^{2n},\omega_0)$; cf. Figure~\ref{fig:fig-Weinstein}.

  As we saw earlier in~(\ref{eq:CT=>stable}), a hypersurface $S$ of
  contact type is stable with respect to the parametrized family $(S_\eps)$
  produced by the Liouville flow $\Phi_\eps:=\Lflow^\LVF_\eps$.
  The linearized flow $d\Lflow^\LVF_\eps:\Ll_S\to\Ll_{S_\eps}$ is a
  bundle isomorphism between the characteristic line bundles
  of $S$ and $S_\eps=\Lflow^\LVF_\eps S$.

  For an example of a stable hypersurface which is not of contact type
  see~\cite[p.122]{hofer:2011a}.
\end{itemize}

\bibliographystylerefCG{alpha}
\cleardoublepage
\phantomsection
\addcontentsline{toc}{section}{References}

\begin{bibliographyrefCG}{}
\end{bibliographyrefCG}

\cleardoublepage
\phantomsection            
\chapter{Fixed energy -- Rabinowitz-Floer homology}
\chaptermark{Fixed energy -- Rabinowitz-Floer}\label{sec:RF}

In 2007 Cieliebak and Frauenfelder considered a version of the action functional
that involves an additional real variable~$\LM$ (detecting periodic
solutions of whatever period), namely
the\index{action functional!Rabinowitz --}
\textbf{\Index{Rabinowitz action functional}}\footnote{
  For an exposition of the classical free period
  action functional see~\citerefRF{Abbondandolo:2013c}.
  }
\begin{equation*}\label{eq:action-Rab-intro-1}
     \Aa^F:\Ll V\times\R\to\R,\quad
     ( \lpz,\LM)\mapsto
     \int_{\SS^1} \lpz^*\lambda-\LM\int_0^1 F( \lpz(t))\, dt
\end{equation*}
associated to \emph{autonomous} Hamiltonians
$F:V\to\R$ on certain exact\footnote{
  Exactness of a symplectic manifold implies non-compactness,
  an inconvenient property which one gets under control
  by imposing additional conditions, for instance \emph{convexity}.  
}
symplectic connected manifolds~$(V,\lambda)$.
While $\Ll V=C^\infty(\SS^1,V)$ is the set of $1$-periodic loops,
the extra parameter $\LM$, together with time independence of~$F$,
causes that a critical point $(\lpz,\LM)$ of $\Aa^F$ corresponds to
\begin{itemize}
\item
  either a $\LM$-periodic Hamiltonian loop that lies on the level set $F^{-1}(0)$;
\item
  or a constant (period $\LM=0$) loop sitting at a
  point $q=z(0)\in F^{-1}(0)$;
\end{itemize}
see (\ref{eq:crit-Reeb}).
Negative periods $\LM<0$ tell that the Hamiltonian loop follows $-X_F$.

Now one can exploit this \emph{fixed energy property} to study dynamical properties
of connected closed hypersurfaces $\Sigma$ in $V$ which can be represented
as regular level-zero sets of autonomous Hamiltonians $F$ on $V$
with compactly supported differentials $dF$,
called defining Hamiltonians for $\Sigma$.
A sufficient condition that $\Sigma$ is of the form $F^{-1}(0)$
is that $\Sigma=\p M$ bounds a compact manifold-with-boundary $M$.
However, not all connected bounding energy surfaces admit
closed orbits, as we saw in~(\ref{eq:RF-counter-Ham-Seif}).
The tool to prove existence of periodic solutions
is a version of Floer homology for the Rabinowitz functional
$\Aa^F$ on the extended domain $\Ll V\times\R$.
Recall from Chapter~\ref{sec:FH} that the analytic key to set up Floer
homology is compactness, up to broken flow lines, of the spaces of
connecting flow lines.
The extra parameter $\LM$ causes non-compactness in certain cases --
giving way to non-existence~(\ref{eq:RF-counter-Ham-Seif}).
A sufficient condition to fix this
is to require the bounding hypersurface $\Sigma$ to be of
\textbf{\Index{restricted contact type}}, i.e. with contact form
$\alpha=\lambda|_\Sigma$.
Consequently by~(\ref{eq:Ham=Reeb}) there is a convex
set $\Ff(\Sigma)\not=\emptyset$ of \textbf{defining Hamiltonians} $F$
whose Hamiltonian vector field is simply equal, along $\Sigma$, to one
and the same Reeb vector field~$R_\alpha$ -- but $\Sigma$ is compact.
The identity $X_F=R_\alpha$ is furthermore extremely beneficial
in the sense that it allows to utilize the analysis carried
out in the Hamiltonian setting of Chapter~\ref{sec:FH}.
In fact, one can allow slightly more general hypersurfaces;
cf. Definitions~\ref{def:exact-convex-symplectic-manifold}
and~\ref{def:exact-convex-hypersurface}
and Remark~\ref{rem:ghg676}.

\begin{assumption}\label{ass:RF}
In Chapter~\ref{sec:RF}, unless mentioned otherwise, we assume that
\begin{itemize}
\item
    $(V,\lambda)$ is a convex exact symplectic manifold
    of bounded topology whose associated
    Liouville vector field $\LVF=\LVF_\lambda$ generates a complete
    flow on $V$;\footnote{
      By~\citeintro[Le.\,1.4]{Cieliebak:2009a} \emph{bounded topology} and
      \emph{complete flow} can be achieved for any convex exact
      symplectic manifold $(V,\lambda)$
      by modifications outside of $\Sigma$.
      }
\item
    $\Sigma\stackrel{\iota}{\hookrightarrow} V$ is a
    convex exact hypersurface. Let $\alpha$ denote the contact form
    on $\Sigma$ and $M$ the compact manifold-with-boundary
    bounded by $\Sigma=\p M$.
\end{itemize}
According to our conventions both $V$ and $\Sigma$ are
\emph{connected}. By Exercise~\ref{exc:symp-mf-with-bdy} the Liouville
vector field $\LVF$ is outward pointing along the boundary $\Sigma$ of $M$.
By Remark~\ref{rem:ghg676} we may assume whenever convenient that
$\alpha=\lambda|_\Sigma:=\iota^*\lambda$ is the restriction of the
primitive $\lambda$ of the symplectic form
$\omega:=d\lambda=d\mu$.\footnote{
  Otherwise, replace $\lambda$ by $\mu:=\lambda+dh$ where
  $(V,\mu)$ inherits the properties of $(V,\lambda)$.
  }
\end{assumption}

\textit{Differences to Chapter~\ref{sec:FH}.} Now the closed orbits
cannot lie anywhere in the symplectic manifold $(V,\omega)$,
they are constrained to a \emph{fixed} (regular) energy surface $\Sigma=F^{-1}(0)$
required to be a contact manifold with respect to the
restriction $\lambda|_\Sigma=:\alpha$;
after changing the primitive $\lambda$ of $\omega$, if necessary.
(Equivalently the Liouville vector field $\LVF$ determined by
$d\lambda(\LVF,\cdot)=\lambda$ is transverse to $\Sigma$.)
In exchange, now the periods $\LM$ are free -- no restriction to
period $1$ any more. Furthermore, whatever defining Hamiltonian
one picks, the Hamiltonian loops are
precisely the Reeb loops of the contact manifold $(\Sigma,\alpha)$
and their images are called \textbf{\Index{closed characteristics}}.
So what one is really counting
are geometric objects associated to the contact manifold
$(\Sigma,\alpha)$ and, as $\alpha=\lambda|_\Sigma$, the way it sits in the
exact symplectic manifold $(V,\lambda)$.
Last, not least, as defining Hamiltonians are autonomous,
non-constant periodic solutions come at least in $\SS^1$-families.
So the functional $\Aa^F$ is at best Morse-Bott, as opposed to Morse.

Under Assumption~\ref{ass:RF}, appropriately taking account of the, 
at best, Morse-Bott nature of the
functional $\Aa^F$, Cieliebak and Frauenfelder proved

\begin{theorem}[Existence and continuation,~\citeintro{Cieliebak:2009a}]
\label{thm:RF}
Under Assumption~\ref{ass:RF} Floer homology
for the Rabinowitz action functional and with $\Z_2$ coefficients
$$
     \HF(\Aa^F)=\HF(\Aa^F;\Z_2)
$$
is defined.
If $\{F_s\}_{s\in[0,1]}$ is a smooth family of defining
Hamiltonians of convex exact hypersurfaces $\Sigma_s$,
then $\HF(\Aa^{F_0})$ and $\HF(\Aa^{F_1})$ are canonically isomorphic.
\end{theorem}

In particular, as the space of defining Hamiltonians is convex,
see~(\ref{eq:Ham=Reeb}),
Floer homology $\HF(\Aa^F)$ does not depend
on the defining Hamiltonian, but on the pair
$(\Sigma,V)$, at most.
%
%
In fact, in~\citerefRF[Prop.\,3.1]{Cieliebak:2010b}
it is shown independence on the unbounded component
of $V\setminus\Sigma$, that is only $\Sigma=\p M$ and its inside,
the compact manifold-with-boundary $M$, are relevant for $\HF(\Aa^F)$.
This justifies the following notation where $\RFH(\Sigma)$ just
serves to abbreviate $\RFH(\p M,M)$.

\begin{definition}
\label{def:RF}
The \textbf{\Index{Rabinowitz-Floer homology}}
of a convex exact hypersurface $\Sigma\subset V$ bounding $M$,
see Assumption~\ref{ass:RF}, is the $\Z_2$
vector\index{{open \color{red} problems \color{black}}}
space\footnote{
  To define $\RFH$ with integer coefficients is an \emph{open problem}.
  }
\begin{equation}\label{eq:def-RFH}
     \RFH(\Sigma)=\RFH(\p M,M):=\HF(\Aa^F)
\end{equation}
where\index{$\RFH(\Sigma=\p M,M)$ Rabinowitz-Floer homology}
$F\in\Ff(\Sigma)$ is a defining Hamiltonian for $\Sigma=F^{-1}(0)$.
\end{definition}

By Theorem~\ref{thm:RF} Rabinowitz-Floer homology
does not change under homotopies of convex exact hypersurfaces.
An integer grading $\mu$ of $\RFH$ exists, see~(\ref{eq:RF-grading}),
whenever $\Sigma$ is simply connected and $c_1(V)$ vanishes over $\pi_2(V)$.
The following deep result has major consequences,
e.g. it reconfirms the Weinstein conjecture for displaceable $\Sigma$;
see~\citeintro[Cor.~1.5]{Cieliebak:2009a} and Section~\ref{sec:Weinstein-RFH}.

\begin{theorem}[Vanishing theorem,~\citeintro{Cieliebak:2009a}]
\label{thm:RF-vanishing}
If $\Sigma$ is displaceable,
then Rabinowitz-Floer homology $\RFH(\Sigma)=0$ vanishes.
\end{theorem}

The idea of proof is to decompose life $[0,1]$ into two parts
$[0,\frac12]$ and $[\frac12,1]$, identifying $\SS^1$ and $[0,1]/\{0,1\}$.
Then push the Hamiltonian flow of $F$ into the first part of life
using a young (support in $[0,\frac12]$) cutoff function $\chi$
and, in the second part of life, allow for elderly (support in
$[\frac12,1]$) but time-experienced Hamiltonian perturbations
$H\in\Hh^\dagger$. This way one gets to the perturbed Rabinowitz action
$\Aa^{F^\chi}_H$ in~(\ref{eq:pert-Rab-action}) and proving
Theorem~\ref{thm:RF-vanishing} reduces to a smart homotopy
argument; cf.~(\ref{eq:RF-thpy-arg}).

Albers and Frauenfelder~\citerefRF{Albers:2010b} realized that the critical points of $\Aa^{F^\chi}_H$ are Moser's~\citerefRF{Moser:1978a} leaf-wise intersection points ($\LIP$s). They obtained existence results for $\LIP$s by associating (Rabinowitz) Floer homology groups $\HF(\Aa^{F^\chi}_H)$ to the perturbed action; see Section~\ref{sec:LIPS}.

\vspace{.1cm}
\textit{Outline.}
In Sections~\ref{sec:Rab-action}--\ref{sec:RF-CC}
we indicate the proof of Theorem~\ref{thm:RF}
following closely the original, excellently written,
paper~\citeintro{Cieliebak:2009a}.
We also recommend the survey~\citerefRF{Albers:2012a}.
Section~\ref{sec:LIPS} is on the
perturbed action functional and $\LIP$s. In Section~\ref{sec:RF-SH}
very briefly we state the relation to loop spaces.

\begin{NOTATION}\label{not:RFH}
Cf. Notation~\ref{not:mfs-Hams}
and Notation~\ref{not:notations_and_signs}.
The elements of $\Ll\Sigma\times\R$,
in particular, the critical points of $\Aa^F$,
are denoted by $(\lpz,\LM)$.
The critical points $(\lpz,\LM)$ correspond, firstly, to the points of
$\Sigma$ via constant loops $\lpz_q\equiv q\in\Sigma$ whenever $\LM=0$
and, secondly, to $\LM$-periodic Reeb loops
$r=(\ror_\lpz)_\LM:\R/\LM\Z=:\SS^1_{\LM}\to\Sigma$;
see~(\ref{eq:Xi}).\footnote{
  Reeb loops are non-constant, as the Reeb vector field is nowhere zero
  by $\alpha(\RVF_\alpha)=1$.
  }
The notation is meant to indicate that for
\emph{negative} $\LM$ the loop $t\mapsto\lpz(t)$, equivalently
$\ror$, follows $-X_F=-R_\alpha$.
Connecting trajectories are pairs ``\textbf{\Index{upsilon}}''
$\upsilon=(u,\LMpath)$ where $u:\R\times\SS^1\to V$ and $\LMpath:\R\to\R$.
\end{NOTATION}

\section[Rabinowitz action functional $\Aa^F$ -- free period]
{Rabinowitz action $\Aa^F$ -- free period}
\sectionmark{Rabinowitz action functional $\Aa^F$}
\label{sec:Rab-action}
Throughout let $\Sigma\stackrel{\iota}{\hookrightarrow} V$
be a convex exact hypersurface in a convex exact symplectic
manifold with symplectic structure $\omega=d\lambda$;
see Assumption~\ref{ass:RF}.
Recall from Section~\ref{sec:CTH-ESM} that $\Sigma=\p M$ bounds a
compact manifold-with-boundary $M$ and comes with the contact form
$\alpha=\lambda|_\Sigma:=\iota^*\lambda$, whereas
$V$ globally carries the Liouville vector field $Y\pitchfork\Sigma$
that is determined by the identity $i_Y\omega=\lambda$.
\newline
Let us repeat from Proposition~\ref{prop:char=per-exact}
the notion of defining Hamiltonian for $\Sigma$.

\begin{definition}
A \textbf{\Index{defining Hamiltonian} for \boldmath$\Sigma$} 
\index{Hamiltonian!defining}
is an autonomous Hamiltonian $F:V\to\R$,
negative on the inside $\interior{M}$ of the hypersurface $\Sigma$,
zero on $\Sigma$, positive outside $\Sigma$,
constant outside some compact set,
and with $X_F=R_\alpha$ along $\Sigma$
(the Hamiltonian vector field extends the Reeb vector field);
see~(\ref{eq:Ham=Reeb}) and Figure~\ref{fig:fig-defining-Hamiltonian}.
Let $\Ff(\Sigma)$ be the space of Hamiltonians defining~$\Sigma$.
\end{definition}

\begin{figure}[h]
  \centering
  \includegraphics
                             [height=3cm]
                             {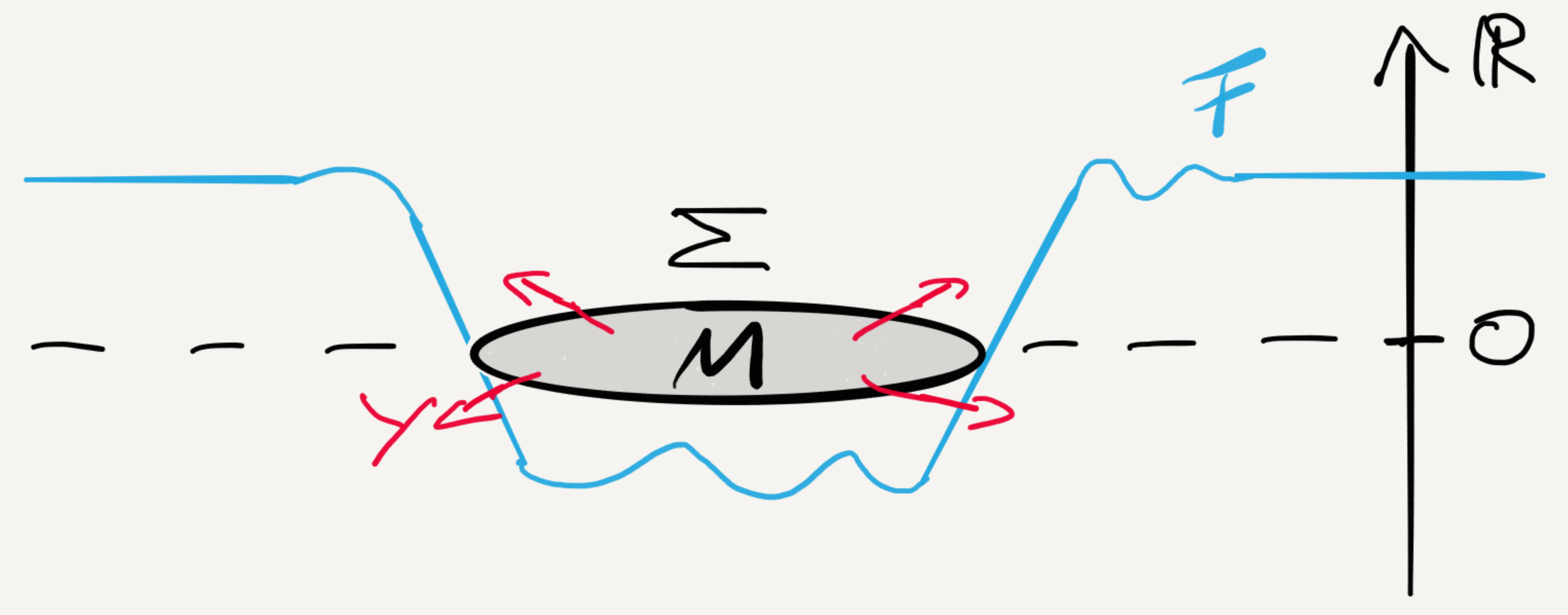}
  \caption{Defining Hamiltonian $F$ for
                 hypersurface $\Sigma=F^{-1}(0)$ bounding $M$}
  \label{fig:fig-defining-Hamiltonian}
\end{figure}

The space of defining Hamiltonians $\Ff(\Sigma)\not=\emptyset$
is non-empty and convex by Proposition~\ref{prop:char=per-exact}.

\begin{remark}
As the Reeb vector field $R_\alpha$ nowhere vanishes, zero is
automatically a regular value of $F\in\Ff(\Sigma)$.
Furthermore, the manifold $V$ and the compact manifold-with-boundary
$M$ bounded by $\Sigma$ are canonically
oriented by the volume form $\omega^n$.
Moreover, the orientation of
$$
     \Sigma=F^{-1}(0)=\p M
$$
as boundary of $M$ by the put-outward-normal-first convention
coincides with the orientation provided by the volume form
$\alpha\wedge(d\alpha)^{n-1}$ on $\Sigma$;
see Exercise~\ref{exc:symp-mf-with-bdy} which also shows that
the Liouville vector field $Y$ points outward along $\Sigma=\p M$.
The gradient of $F$ with respect to any Riemannian metric points also outward.
\end{remark}

\begin{definition} 
For a defining Hamiltonian $F\in\Ff(\Sigma)$
the \textbf{(unperturbed) Rabinowitz action
  functional}\index{Rabinowitz action functional!(unperturbed)
  --}\index{action functional!(unperturbed) Rabinowitz --}\index{$\Aa^F$
  Rabinowitz action (unpert.)}
is defined by
\begin{equation}\label{eq:action-Rab}
     \Aa^F:\Ll V\times\R,\quad
     ( \lpz,\LM)\mapsto
     \int_{\SS^1} \lpz^*\lambda-\LM\int_0^1 F( \lpz(t))\, dt,
\end{equation}
where $\Ll V:=C^\infty(\SS^1,V)$ is the space of $1$-periodic loops in $V$.
\end{definition}

Bringing in the real numbers $\LM$ causes that critical points
$( \lpz,\LM)$ of $\Aa^F$ not only correspond to integral loops of $X_F$,
as was the case in~(\ref{eq:action-A}),
but constrains\footnote{
  To find critical points of a function $f=f(x,y)$, say on $\R^2$ for
  illustration, whose domain is cut out by a constraint, say
  $g(x,y)=c$, one introduces a dummy variable $\lambda\in\R$
  called \textbf{\Index{Lagrange multiplier}} and determines
  the critical points of the function
  $\Lambda(x,y,\lambda)=f+\lambda(g-c)$;
  cf.~\href{https://en.wikipedia.org/wiki/Lagrange_multiplier}{Wikipedia}.
  In our case $\LM$ plays the role of the Lagrange multiplier and $c$ is zero.
  }
them to lie on energy level zero
and allows for periods $\LM$ other than~$1$.

\begin{exercise}[Critical points correspond to Reeb loops and points of $\Sigma$]
\label{exc:Crit-Rab}
Show that the critical points $(\lpz,\LM)$ of $\Aa^F$ are the solutions
$\LM\in\R$ and $\lpz:\SS^1\cong\R/\Z\to V$ of the
ODE and the constraint given by
\begin{equation*}
\begin{cases}
     \dot \lpz(t)=\LM\, X_F(\lpz(t))&\text{, $t\in\SS^1$,}
     \\
      \lpz(t)\in F^{-1}(0)&\text{, $t\in\SS^1$,}
\end{cases}
\end{equation*}
or, equivalently, of
\begin{equation}\label{eq:crit-Reeb}
\begin{cases}
     \dot \lpz(t)=\LM\, \RVF_\alpha(\lpz(t))&\text{, $t\in\SS^1$,}
     \\
     P:=\lpz(\SS^1)\subset\Sigma&.
\end{cases}
\end{equation}
[Hint: Show $\int_0^1F( \lpz(t))\,dt=0$. But $ \lpz(t)=\phi_t \lpz(0)$
where $\phi=\phi^F$ preserves $F$.]
\end{exercise}

Thus a critical point $(\lpz,\LM)$ of $\Aa^F$
corresponds either, in case $\LM>0$, to a $\LM$-periodic\,\footnote{
  Here $\LM$ is just \emph{a} period, not necessarily the
  prime period $\LM_\rorDUMMY$ of the Reeb loop $\rorDUMMY=\lpz(\cdot/\LM)$.
  }
Reeb loop $\lpz(\cdot/\LM)$ on the contact manifold
$(\Sigma,\alpha)$, or in case $\LM<0$ to one that runs backwards
following $-R_\alpha$, or in case $\LM=0$ to a constant loop
$\lpz_p\equiv p=\lpz(0)$
sitting at any point $p$ of $\Sigma$.\footnote{
  The appearance of all point loops in addition to Reeb loops --
  at first glance seemingly an annoying irregularity --
  is actually the \textbf{power plant} of the whole theory;
  cf.\index{power plant of $\RFH$}
  Remark~\ref{rem:powerhouse}~(ii).
  }
A critical point of the form $(\lpz,0)$ is called a \textbf{constant critical point}.
By~(\ref{eq:crit-Reeb})
the critical points of $\Aa^F$ \emph{do not} depend
on the defining Hamiltonian. In the notation~(\ref{eq:conv-or-path})
there is a bijection
\begin{equation}\label{eq:Xi}
     \Crit\,\Aa^F\to\GRO(\Sigma)\cup \Sigma,\quad
     (\lpz,\LM)\mapsto (\rorDUMMY_z)_\LM:=
     \begin{cases}
        \lpz^{\frac{1}{\tau}}=\lpz(\cdot/\LM)&\text{, $\LM\not=0$,}\\
        \lpz(0)&\text{, $\LM=0$,}
     \end{cases}
\end{equation}
onto the set $\GRO(\Sigma)\cup \Sigma$ given by (and then identified with a set of
pairs)\index{$\GRO(\Sigma)$ forward or backward (closed) Reeb loops}
\begin{equation}\label{eq:GRO}
\begin{split}
   &\left\{\rorDUMMY_\pertau:\R/\pertau\Z\to\Sigma\mid
     \text{$\dot \rorDUMMY_\pertau=R_\alpha(\rorDUMMY_\pertau)$, $\LM\not=0$}
     \right\}\cup\Sigma\\
   &\simeq
     \left\{(\rorDUMMY,\pertau)\in\C^\infty(\R,\Sigma)\times\R\mid
        \text{$\dot \rorDUMMY=\sign(\pertau) R_\alpha(\rorDUMMY)$, $\pertau\in\Per(\rorDUMMY)$}
     \right\}.
\end{split}
\end{equation}
Here $\GRO(\Sigma)$ is the
\textbf{set of signed Reeb loops}:\index{Reeb loop!signed --}
These are the forward or backward\footnote{
  Allowing also for periodic solutions of
  $\dot \rorDUMMY={\color{brown}-}R_\alpha(\rorDUMMY)$
  simplifies things, e.g. the above map~(\ref{eq:Xi}) is simply a ``bijection''
  instead of a ``$2:1$ map on non-constant critical points''.
  }
periodic Reeb orbits on $\Sigma$.
To see that the map $\simeq$ given by
$(\rorDUMMY,\pertau)\mapsto\rorDUMMY_\pertau$
is a bijection define \Index{$\sign(0):=0$}
and recall from~(\ref{eq:conv-or-path}) that $\rorDUMMY_\pertau$ is
given by $\rorDUMMY$ restricted to $\R/\pertau\Z$,
subject to direction reversal if $\pertau<0$,
and that $\rorDUMMY_0=\rorDUMMY(0)$.

\begin{definition}[Simple critical points and their covers]
\label{def:critpt-towers}
Pick a \emph{non-constant} critical point $(\lpz,\LM)$
of $\Aa^F$ and consider the (embedded) image circle $P:=\lpz(\SS^1)$.
Suppose $\LM>0$, otherwise take
$(\hat\lpz,\hat\LM)$.
Observe that $1\in\Per(\lpz:\R\to\Sigma)=\LM_\lpz\Z$
where $\LM_\lpz>0$ is the prime period; see~(\ref{eq:orbits-types}).
Thus $1=\ell\LM_\lpz$ for some integer $\ell=\ell(\lpz)\ge 1$.
Rescale $\lpz$ and $\LM$ by
\begin{equation}\label{eq:simple-crit-point}
     \lpz_P:=\lpz^{\LM_\lpz}=\lpz(\LM_\lpz\cdot),\qquad
     \sigma_P:=\LM\LM_\lpz,\quad\LM=\ell\sigma_P.
\end{equation}
The prime period of $\lpz_P:\R\to\Sigma$ is $1$ and
$$
     c_P:=\left(\lpz_P,\sigma_P\right)\in\Crit\,\Aa^F,\quad
     \lpz_P:\SS^1\INTO\Sigma,\quad
     \lpz_P(\SS^1)=P.
$$
As $\lpz_P:\SS^1\INTO\Sigma$ is a simple loop,
we\index{critical point!backward -- of $\Aa^F$}
call\index{critical point!simple -- of $\Aa^F$}
$c_P$\index{simple!critical point of $\Aa^F$}
\textbf{a simple critical point} of $\Aa^F$.\footnote{
  Let us call $\hat c_P:=\left(\lpz_P(-\cdot),-\sigma_P\right)$ the
  \textbf{corresponding \Index{backward simple critical point}}.
  }
The other ones with image $P$ are
obtained by subjecting $\lpz_P$ to time shifts leading to an
$\SS^1$-family denoted by $\SS^1 *c_P$ or $S_{c_P}$.
As $\sigma_P$ divides the speed factor $\LM$ of $\lpz$,
we call it the \textbf{\Index{prime speed}} of the critical points
with image $P$.
\end{definition}

The $k$-fold covers of a simple critical point $c_P$ defined by
\begin{equation}\label{eq:jhjkhkjhkj45}
     c_P^k:=\left(z_P^k,k\sigma_P\right), \quad
      z_P^k:=z_P(k\cdot)\qquad k\in\Z,
\end{equation}
are critical points of $\Aa^F$ as well and, up to the $\SS^1$-action by time
shift, there are no other critical points whose image is $P$.
Observe that $\lpz_P^0\equiv\lpz_P(0)$ is constant.

\begin{exercise}\label{exc:critpt-towers}
Check the assertions in Definition~\ref{def:critpt-towers}.
Show\index{prime!time-$1$ -- speed factor}\index{prime!critical point}
that\index{$\sigma_P$ time-$1$ prime \Index{speed factor}}
$\sigma_P$, modulo time shift also $\lpz_P$, is independent of
$(\lpz,\LM)\in\Crit\,\Aa^F$ as long as $\lpz(0)\in P$.
\newline
[Hints: Let $(y,\chi)$ also be a critical point
with $\chi>0$ and $y(\SS^1)=P$. Check that both paths
$\lpz^{1/\LM}$ and $y^{1/\chi}$ are Reeb solutions
and their prime periods are $\LM\LM_{\lpz}$ and $\chi \LM_y$,
respectively. Hence $\LM\LM_{\lpz}=\chi\LM_y=:\sigma_P$.
But now the paths $\lpz^{\LM_\lpz}$ and $y^{\LM_\chi}$
satisfy the same ODE $\dot x=\sigma_P \RVF_\alpha(x)$,
so they are equal up to time shift.]
\end{exercise}

\begin{exercise}[Simple Reeb loop associated to $(\lpz,\LM)$ via time of first return]
\label{exc:simple-Reeb-1st-return}
Show the assertions of Exercise~\ref{exc:critpt-towers} as follows.
Pick $(\lpz,\LM)\in\Crit\,\Aa^F$ with $\LM\not=0$ and set $P:=\lpz(\SS^1)$.
Now set $p:=\lpz(0)\in P$ and apply the Reeb flow $\Rflow_t$ to get
the Reeb path $\rorDUMMY(t):=\Rflow_t p$ for $t\in\R$ whose
image is $P$. (Hence the images $P=\lpz(\SS^1)$ of non-constant
critical points $(\lpz,\LM)$ are closed characteristics.)
Let $T_P>0$ be the time of first return. Check that
it does not depend on the initial point in $P$, thereby justifying
the notation $T_P$, as opposed to $T_p$.
Since $\Rflow$ is a one-parameter group $T_P$
is a period of $\rorDUMMY$ and as it is the smallest positive one
$T_P=\LM_\rorDUMMY$ is the prime period of $\rorDUMMY$. Thus
to $(\lpz,\LM)$ belongs the simple Reeb loop
\begin{equation}\label{eq:simple-Reeb-loop}
     \ror_P:\R/T_P\Z\to\Sigma,\quad t\mapsto \Rflow_t p,\quad p=\lpz(0),
\end{equation}
a diffeomorphism onto its image $P=\lpz(\SS^1)$.
\begin{itemize}
\item[a)]
  Show that $\LM=k T_P$ for some integer $k\not= 0$.
\item[b)]
  Show that $T_P=\sigma_P$ is the prime speed,
  hence $(\ror_P^{T_P},T_P)=(\lpz_P,\sigma_P)=:c_P$.
\end{itemize}
[Hints: a)~Show that $\rorDUMMY=\lpz^{1/\LM}:\R\to\Sigma$ and
observe that $\LM\in\Per(\lpz^{1/\LM})$.
b)~The pair $(\ror_P^{T_P},T_P)$ is a critical point and
$\ror_P^{T_P}:\SS^1\INTO \Sigma$ is an embedding with image $P$.
Hence $T_P=\ell\sigma_P$ for some $\ell\in\N$
and $\lpz_P:=\ror_P^{T_P}(\cdot/\ell)$ is, in particular, of period $1$.
But $\lpz_P(1)=\lpz_P(0)$, equivalently $\ror_P^{T_P}(1/\ell)=\ror_P^{T_P}(0)$,
implies $\ell=1$.]
\end{exercise}

\begin{exercise}[Action spectrum]\label{exc:action-spectrum}
Show that the action value
\begin{equation}\label{eq:action-growth-by-period}
     \Aa^F(c_P^k)=\Aa^F(\lpz_P^k,k\sigma_P)=k\sigma_P,\qquad
     \Aa^F(\lpz,0)=0,
\end{equation}
of the $k$-fold cover, $k\not=0$, of a simple critical point $c_P$
is given by $k$ times the prime speed $\sigma_P$.
So by Exercise~\ref{exc:simple-Reeb-1st-return}
the\index{spectrum!action --}
\textbf{\Index{action spectrum}}
$\mathfrak{S}(\Aa^F)$,\index{$\mathfrak{S}(\Aa^F)$ action spectrum}
i.e. the set of critical values of $\Aa^F$, consists of all integer
multiples of the prime periods of the Reeb loops, in symbols
$
     \mathfrak{S}(\Aa^F)=\Z\,\mathfrak{S}(\Sigma)
$.
The\index{spectrum!prime period --}
set\index{$\mathfrak{S}(\Sigma)$ periods of simple Reeb loops}
$\mathfrak{S}(\Sigma)$
of\index{spectrum!prime period --}
periods of the \emph{simple} Reeb loops
is the \textbf{\Index{prime period spectrum}}
of the contact manifold $\Sigma$.
\end{exercise}

\begin{figure}
  \centering
  \includegraphics
                             [height=5cm]
                             {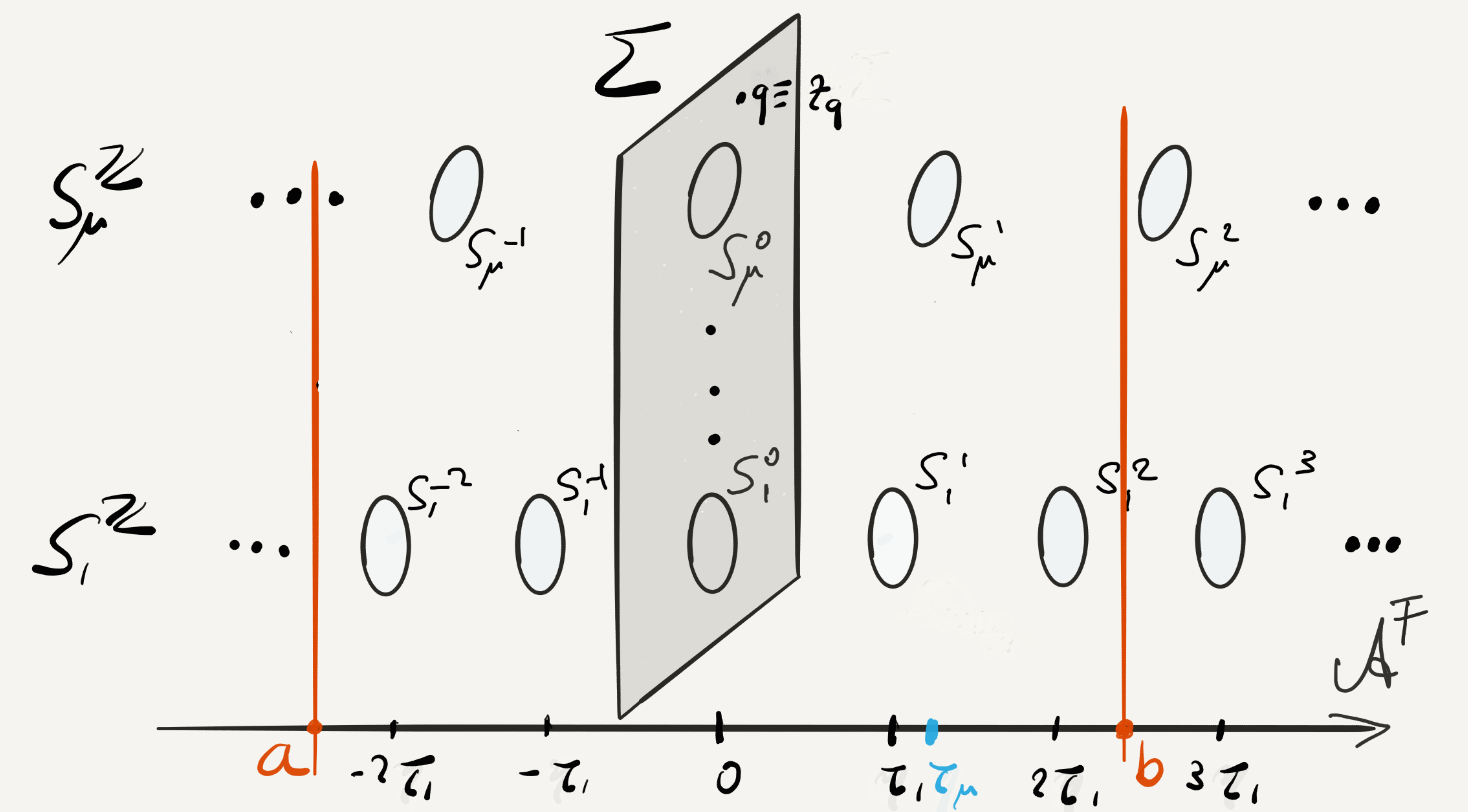}
  \caption{Finitely many critical circle towers:
                $\Crit\,\Aa^F\cong\Sigma\cup S_1^\Z\cup\dots S_\mu^\Z$}
  \label{fig:fig-crit-point-towers}
\end{figure}

\begin{remark}[Critical towers]\label{rem:crit-pt-towers}
By~(\ref{eq:jhjkhkjhkj45})
any\index{towers of critical points}\index{critical point!towers}
simple critical point of $\Aa^F$, that is a pair of the form
$c_P=(\lpz_P,\sigma_P)$, gives rise to
a whole \emph{critical point tower} $c_P^\Z:=\left(\lpz_P^k,k\sigma_P\right)_{k\in\Z}$.
The circle acts on the $k$-fold cover $c_P^k$ of $c_P$ by time shifts
$$
     T * c_P^k:=\left(T* z_P^k,k\sigma_P\right)
     :=\left(z_P^k (T+\cdot ),k\sigma_P\right)
     ,\quad T\in\SS^1=\R/\Z.
$$%
Thus point towers come, at least, as \emph{critical circle towers}
$C_P^\Z:=\left(C^k_P\right)_{k\in\Z}$ where
$$
     C^k_P:=S_P^k\times\{k\sigma_P\}
     ,\quad
     S^k_P
     :=\SS^1 * \lpz_P^k
     :=\left\{T * \lpz_P^k\right\}_{T\in\SS^1}
     ,\quad
     S_P^\Z
     :=\left(S^k_P\right)_{k\in\Z}.
$$
So the action functional $\Aa^F$ can be Morse-Bott, at best, but not Morse.
The set of simple critical points $c_P=(\lpz_P,\sigma_P)$ corresponds
to the set $\Cc(\Sigma)$ of \emph{closed characteristics}
$P$ with distinguished point $p\in P$: a)~Associate to $\lpz_P$ the embedded circle
$P:=\lpz_P(\SS^1)$ with distinguished point $p:=\lpz_P(0)$.
Indeed $P$ is a closed Reeb orbit, thus a closed
characteristic, by Exercise~\ref{exc:simple-Reeb-1st-return}.
b)~Vice versa, given $P$ and $p\in P$, consider the Reeb path
$\ror(t):=\Rflow_tp$. Its prime period is denoted by $\pertau_\ror$.
Then $(\lpz_P,\sigma_P):=(\ror^{\pertau_\ror},\pertau_\ror)$ is a
simple critical point of $\Aa^F$ which gets mapped back to $P$ and $p$ by a).
That the map in~a) followed by the one in~b) is the identity as well
holds by Exercise~\ref{exc:simple-Reeb-1st-return}~b).
In general, the simple critical points can appear in
families larger than circles, unless one has
\end{remark}

\subsection{Transverse non-degeneracy}\label{sec:transverse-ND}

Recall that a function $f$ is called a \textbf{\Index{Morse-Bott function}}
if, firstly, its critical set $C:=\Crit f$ is a submanifold
(whose components might be of different dimensions)
and, secondly, the tangent space $T_p C$ at every point $p$ of $C$
is precisely the kernel of the Hessian $\Hess_p f$ of $f$ at $p$.
In view of invariance of the functional $\Aa^F$
under the time shift circle action,
non-degeneracy is achievable \emph{at most} in directions transverse
to the circles in $\Ll V$ generated by the $\SS^1$-action.

\begin{theorem}[Transverse non-degeneracy,~\citeintro{Cieliebak:2009a}]
\label{thm:RF-2nd-cat}
There is a residual\,\footnote{
  cf. Section~\ref{sec:Baire}
  }
subset\index{$\Ffreg$ regular Rabinowitz-Floer Hamiltonians}
$\Ffreg$ of the complete metric space
\Index{$\Ff:=C^\infty_{0^\prime}(V)$} of smooth functions on $V$ with compactly
supported \emph{differential} such that the following is true.
For every $f\in\Ffreg $ the Rabinowitz action functional $\Aa^f$ is
Morse-Bott and its critical set consists of $f^{-1}(0)$
together with a disjoint union of
(embedded) circles.
\end{theorem}

\begin{proof}
See~\citeintro[Thm.~B.1, p.298]{Cieliebak:2009a}. The proof uses the machinery 
of Thom-Smale transversality, see Section~\ref{sec:Thom-Smale-transversality},
with the \emph{additional difficulty} that the universal
section cannot be surjective due to the unavoidable
critical circles arising by time shifting non-constant critical points.
\end{proof}

\begin{definition}\label{def:RF-non-deg-Sigma}
A contact type hypersurface is called
\textbf{non-degenerate}\index{contact type hypersurface!non-degenerate
--}\index{non-degenerate!contact type hypersurface}
if all Reeb loops $\ror$\index{non-degenerate!transverse}
are\index{transverse non-degenerate!Reeb loop}
\textbf{transverse
non-degenerate}\index{Reeb loop!transverse non-degenerate --}
in the sense that the linearized Reeb flow
$\d\Rflow_T(p):\xi_p\to\xi_p$ at $p=\ror (0)$,
where $\xi_p:=\ker\alpha(p)$ and $T>0$ is the prime period
of $\ror$, cf. Exercise~\ref{exc:Reeb-preserves-contact},
does not have $1$ as eigenvalue.
\end{definition}

From now on suppose Assumption~\ref{ass:RF}.
Hence the bounding hypersurface $\Sigma\subset(V,\lambda)$ is
of restricted contact type. Fix a defining Hamiltonian $F\in\Ff(\Sigma)$.

\begin{exercise}\label{exc:reg_non-deg}
Check that $F\in\Ffreg$ iff $\Sigma$ is non-degenerate.
\newline
[Hint: Cf. Exercise~\ref{exc:non-deg-Crit-pp}.]
\end{exercise}

\begin{exercise}\label{exc:reg_non-deg-nbhd}
a)~Show that there is a convex closed neighborhood $\Uu^F$ of $F$ in
$\Ff=C^\infty_{0^\prime}(V)$ such that the zero set $f^{-1}(0)$ of any
element $f\in\Uu^F$ is a convex exact hypersurface of $(V,\lambda)$.
(In fact, it is of restricted contact type, if $\Sigma$ is.)
b)~Check that $\Uureg^F:=\Uu^F\cap\Ffreg$ is residual in $\Uu^F$.
\newline
[Hint: a)~Given the by $d\lambda(\LVF,\cdot)=\lambda$ globally determined
Liouville vector field~$\LVF$, the contact type property of $\Sigma$
is equivalent to transversality $\LVF\pitchfork \Sigma$. But transversality
$\LVF\pitchfork\tilde\Sigma$ survives, where $\tilde\Sigma:=f^{-1}(0)$,
whenever $f$ is sufficiently close to $F$ in $\Ff$. Define
$\tilde\alpha:=\tilde\iota^*\lambda$; see~(\ref{eq:alpha_Y}).]
\end{exercise}

\begin{remark}\label{rem:generic-Sigma}
Given regular Hamiltonians $F_0,F_1\in\Uureg^F$ near the
defining Hamiltonian $F$ of 
$\Sigma$, by Exercise~\ref{exc:reg_non-deg-nbhd}~a)
convex combination provides a family
\begin{equation}\label{eq:homotopy-cont-RF}
     \Sigma_s:={F_s}^{-1}(0),\quad F_s:=(1-s)F_0+s F_1 ,\quad
     s\in[0,1],
\end{equation}
of bounding restricted contact type hypersurfaces of $(V,\lambda)$.
While the endpoints of the family are non-degenerate,
there is no reason that all members $\Sigma_s$ be.
\end{remark}

\begin{remark}[$\Crit\,\Aa^F$ consists of $\mu<\infty$ circle towers and $\Sigma$]
\label{rem:fin-crit-pt-towers}
If $\Sigma$ is \emph{non-degenerate}, i.e. $F\in\Ffreg$,
then $\Cc(\Sigma)$ consists
of \emph{finitely many} closed characteristics $P_1,\dots,P_\mu$
since $\Sigma$ is compact.
The embedded circles $P_i\hookrightarrow\Sigma$
correspond, up to fixing a point $p_i\in P_i$, likewise modulo
time shifts, to simple Reeb loops
$$
     \ror_i:=\ror_{P_i}=\Rflow_\cdot p_i:\R/T_i\Z\to\Sigma
     ,\quad
     T_i:=T_{P_i},
$$
see~(\ref{eq:simple-Reeb-loop}), so by
Exercise~\ref{exc:simple-Reeb-1st-return}~b)
to simple critical points
\begin{equation}\label{eq:simple-crit-pts}
     c_i:=(\lpz_i,\LM_i)
     :=(\ror_i^{T_i},T_i)=(\lpz_{P_i},\sigma_{P_i})
     \in\Crit\,\Aa^F,\qquad
     i=1,\dots,\mu.
\end{equation}
Observe that $\LM_i:=T_i=\sigma_{P_i}$ is the prime period
(time of first return) of the Reeb loop $\ror_i$ and simultaneously
the prime speed of the simple critical point $z_i=\ror_i^{T_i}$. Time
shifting $\ror_i$ and $\lpz_i$ moves the initial point
$p_i=\ror_i(0)=\lpz_i(0)$ around the circle $P_i$.
Let $c_i^k$ be the $k$-fold cover~(\ref{eq:jhjkhkjhkj45}) of
$c_i$. Let $S_i^k:=S_{P_i}^k=\SS^1 *\lpz_i^k$, so
\begin{equation}\label{eq:crit-circles}
     C_i^k:=C_{P_i}^k=S_i^k\times \{k\LM_i\}
     ,\quad
     k\in\Z^*:=\Z\setminus\{0\} ,
\end{equation}
are the corresponding circles in $\Ll\Sigma\times\R$ arising by time shift;
cf.~Remark~\ref{rem:crit-pt-towers}.
The components of the critical set
$C:=\Crit\,\Aa^F$ are the circles $C_i^k$ along
which the action is constantly $k\LM_i$
and the component $C_0=\Sigma\times\{0\}$
of constant loops in $\Sigma$
where the action is zero; cf.~(\ref{eq:RF-crit-mf-C}).
Note that each critical component is compact and that by
non-degeneracy only finitely many of them lie on action levels
inside any given compact interval $[a,b]$.
This is illustrated by Figure~\ref{fig:fig-crit-point-towers}.
\end{remark}

\section{Upward gradient flow -- downward count}
Pick a defining Hamiltonian $F\in\Ff(\Sigma)$. So $F$ is
constant, i.e. $X_F=0$, outside a compact subset $K$ of $V$.
In order to turn the differential $d\Aa^F$ of the Rabinowitz action
functional into a gradient, we need to choose a metric on $\Ll V\times \R$.
The fact that $V$ is non-compact, as opposed
to the symplectic manifold in Section~\ref{sec:FH-comp}
on fixed-period Floer homology, causes serious difficulties
when it comes to prove compactness of moduli spaces of
connecting trajectories $\upsilon=(u,\LMpath)$, due to the lack of an
apriori $C^0$ bound; cf. Theorem~\ref{thm:RF-compactness}.
The key idea to obtain nevertheless uniform
$C^0$ bounds for connecting trajectories
is to choose a family $J=(J_t)_{t\in\SS^1}$
of $\omega$-compatible almost complex structures $J_t$
which are \emph{cylindrical}, see Definition~\ref{def:acs-cylindrical},
along the cylindrical ends $N\times\R_+$ of $(V,\lambda)$.

\begin{remark}[$X_F=0$ on cylindrical ends --
causes uniform $C^0$ bound for~$u$'s]
\label{rem:X_F=0-on-cyl-ends}
Recall from Exercise~\ref{exc:cyl-ends} that $N=\p M_k$
is chosen as boundary of one of the members of the exhaustion
of $V=\cup_k M_k$. Choose $k$ larger, if necessary,
such that $M_k$ contains $K$. This guarantees that
$X_F$ vanishes on the cylindrical ends.
\end{remark}

For $( \xi_i, \LM_i)\in T_{(\lpz,\LM)}(\Ll V\times\R)$
an $L^2$ metric on $\Ll V\times\R$ is defined by
\begin{equation}\label{eq:RF-L2-inner-product}
     \left\langle( \xi_1, \LM_1),( \xi_2, \LM_2)\right\rangle_J
     :=\int_0^1 \omega\left( \xi_1(t),J_t(\lpz(t)) \xi_2(t)\right) dt
     +\LM_1\LM_2.
\end{equation}

\begin{exercise}\label{eq:nabla-Aa^F}
The gradient of $\Aa^F$ with respect to the $L^2$ metric is given by
\begin{equation}\label{eq:RF-grad-Aa}
     \grad\Aa^F(\lpz,\LM)   
     =\begin{pmatrix}
     -J_t(\lpz)\left(\p_t \lpz-\LM X_F(\lpz)\right)
     \\
     -\int_0^1 F(\lpz(t))\, dt
     \end{pmatrix}.
\end{equation}
\end{exercise}

So for the norm-square of the gradient one gets
\begin{equation}\label{eq:RF-norm-grad-Aa}
\begin{split}
     \Norm{\grad\Aa^F(\lpz,\LM)}^2
   &=\left\langle(\grad\Aa^F(\lpz,\LM),\grad\Aa^F(\lpz,\LM) \right\rangle_J\\
   &=\Norm{\p_t \lpz-\LM X_F(\lpz)}_{2}^2+\mean^2(F\circ \lpz)\\
   &\le\Bigl(\Norm{\p_t \lpz-\LM X_F(\lpz)}_{2}+\Abs{\mean(F\circ \lpz)}\Bigr)^2
\end{split}
\end{equation}
where\footnote{
  Here
  $
     \Norm{\xi}_{2}^2:=\int_0^1\omega\left(\xi(t),J_t(\lpz(t))\xi(t)\right) dt
  $
  for smooth vector fields $\xi$ along the loop $\lpz$.
  }
we used that $J_t$ is compatible
with~$\omega$ \index{$\mean(F\circ \lpz ):=\int F\circ \lpz $ mean value} 
and where
$$
     \mean(F\circ \lpz ):=\int_0^1 F(\lpz(t))\, dt.
$$

The \emph{upward} gradient trajectories of $\grad\Aa^F$ are the
solutions denoted by the letter ``\textbf{\Index{upsilon}}''
$$
     \upsilon=(u,\LMpath)\in
     C^\infty(\R\times\SS^1,V)\times C^\infty(\R,\R)
$$
of\index{$\upsilon=(u,\LMpath)$ is called 'upsilon'}
the elliptic PDE given by
\begin{equation}\label{eq:RF-UGF-ups}
     \p_s\upsilon=\grad\Aa^F(\upsilon),\qquad
     \upsilon=(u,\LMpath),
\end{equation}
or, equivalently, the zeroes of a section,
cf.~(\ref{eq:linearization}), namely
\begin{equation}\label{eq:RF-UGF}
     \Ff_F(u,\LMpath):=
     \begin{pmatrix}
     \p_su+J_t(u)\left(\p_t u-\LMpath X_F(u)\right)
     \\
     \p_s\LMpath+\int_0^1 F(u_s(t))\, dt
     \end{pmatrix}
     =0.
\end{equation}

\begin{remark}[Homology -- upward vs. downward gradient flow]
\label{rem:RF-non-standard-order}
  The advantage of using the \emph{upward} gradient equation
  is that on cylindrical ends, where $X_F=0$, component one
  of~(\ref{eq:RF-UGF}) becomes the well known $J$-holomorphic
  curve equation\footnote{
  The downward gradient leads to the anti $J$-holomorphic
  curve equation $\p_su-J(u)\p_t u=0$.
  }
  in which case to obtain an apriori $C^0$ bound for $u$
  one can simply refer to the literature;
  see~\citeintro[p.268, pf. of Thm.~3.1]{Cieliebak:2009a}.
  In order to define nevertheless \emph{homology}, as opposed to cohomology,
  one must ensure that the action \emph{decreases} along the
  \emph{boundary operator} $\p$. This leads to the non-standard
  order in the coefficients $n(c_-,c_+)c_-$ of $\p c_+$
  in~(\ref{eq:RF-bdy}) below. In other words, to define an action decreasing
  boundary operator $\p c_+$, one can either
\begin{enumerate}
\item[(standard)]
  count downward flow lines emanating from $c_+$ at time $-\infty$ or
\item[(present)]
  count upward flow lines that end at $c_{\color{cyan} +}$
  at time ${\color{cyan} +}\infty$.
\end{enumerate}
  Figure~\ref{fig:fig-RF-UGF} \emph{illustrates}\footnote{
    Be aware that $\p c_+$ will not count connecting flow lines $u/\R:=\tilde u$
    of $\grad\Aa^F$, but connecting \emph{cascade} flow lines
    $\Gamma=\bigl(\gamma^0_-,\tilde\upsilon^1,\gamma^1,\dots,
    \tilde\upsilon^\ell,\gamma^\ell_+\bigr)$;
    cf. Figure~\ref{fig:fig-RF-cascades}. The cascades are unparametrized curves
    $\upsilon^j/\R:=\tilde\upsilon^j=(\tilde u^j,\tilde\LMpath^j)$,
    the $\gamma^j_{(\pm)}$ are (semi-)finite time, thus
    parametrized, Morse gradient trajectories along components of the
    Morse-Bott manifold $C:=\Crit\,\Aa^F$.
  }
  the downward count of upward flows.
\end{remark}

\begin{figure}
  \centering
  \includegraphics
                             [height=4cm]
                             {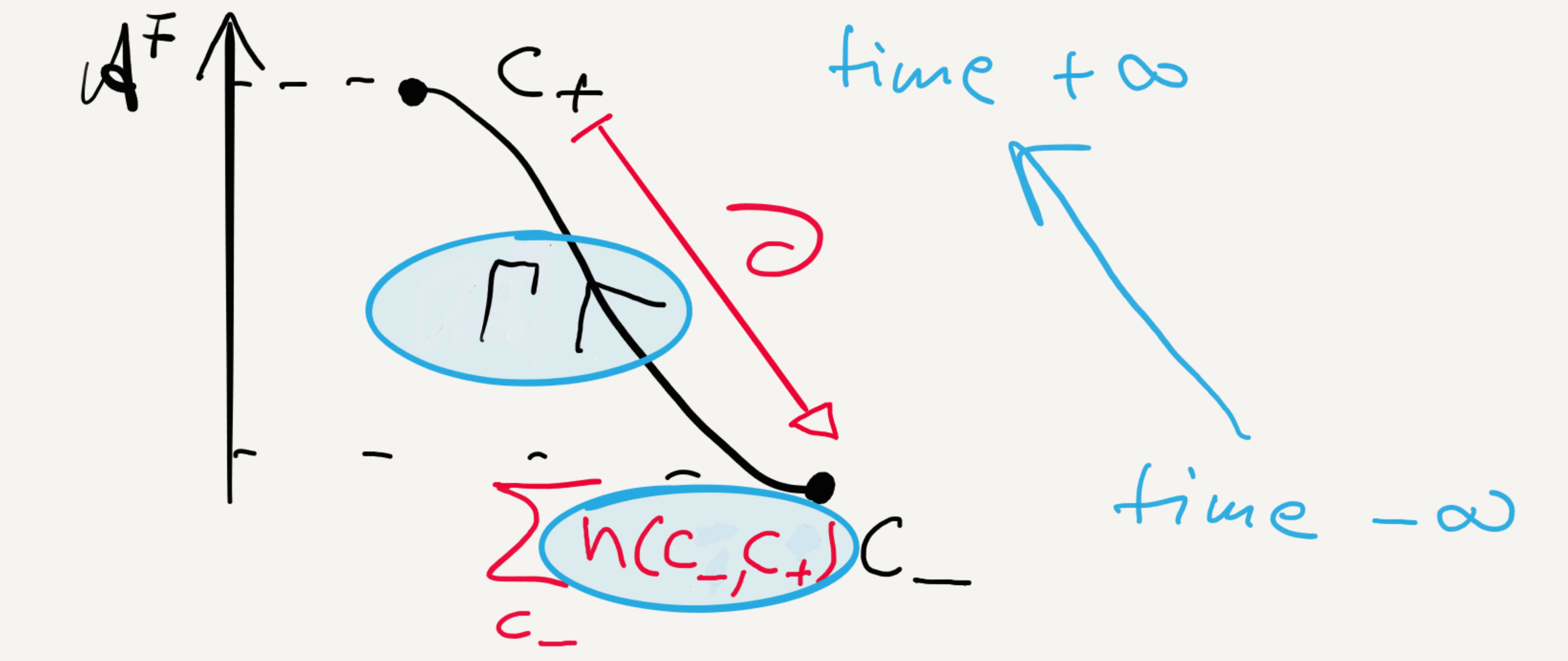}
  \caption{\emph{Downward} count of upward flows $\Gamma$ meeting at $c_+$}
  \label{fig:fig-RF-UGF}
\end{figure}

\section{Rabinowitz-Floer chain complex}\label{sec:RF-CC}

Consider a convex exact hypersurface $\Sigma\subset (V,\lambda)$,
say of restricted contact type, in a
convex exact symplectic manifold, as in Assumption~\ref{ass:RF}.
Fix a family $J=(J_t)_{t\in\SS^1}$ of almost complex structures $J_t$
compatible with the symplectic structure $\omega:=d\lambda$
on $V$ and cylindrical along the cylindrical ends;
see Definition~\ref{def:acs-cylindrical}.
Choose the associated $L^2$ inner product on $\Ll V\times\R$
given by~(\ref{eq:RF-L2-inner-product}).

There are two goals in Section~\ref{sec:RF-CC}. Firstly, to associate to a
defining Hamiltonian $F\in\Ff(\Sigma)$ of $\Sigma$ Floer
homology groups $\HF(\Aa^F)$ with $\Z_2$ coefficients.
Secondly, to show that these are invariant, up to natural
isomorphism, not only under changing the defining Hamiltonian $F$,
but also under convex exact homotopies of the
hypersurface $\Sigma$ itself. Cieliebak and
Frauenfelder~\citeintro{Cieliebak:2009a} defined
$$
     \RFH(\Sigma,V):=\HF(\Aa^F),
$$
called \textbf{\Index{Rabinowitz-Floer homology}}
of the convex exact hypersurface $\Sigma=\p M$ which,
by assumption, bounds a compact manifold-with-boundary, say $M$.

To construct Floer homology one usually slightly perturbs relevant
quantities in a first step, see Section~\ref{sec:Thom-Smale-transversality},
in order to get to a Morse situation, so one can use the then
discrete critical points themselves as generators of
the Floer chain groups. Here this is impossible:
Since $F$, being defining for~$\Sigma$, is necessarily
\emph{time-independent} the Rabinowitz action
functional $(\lpz,\LM)\mapsto \Aa^F(\lpz,\LM)$ will always be invariant under
the $\SS^1$-action on $\Ll V\times \R$ given by time shifting
the loop component $\lpz$. While $\Aa^F$ is therefore never Morse,
it is of the simplest Morse-Bott type for generic, called regular,
defining Hamiltonian $F$ by Theorem~\ref{thm:RF-2nd-cat}.
For $F\in\Ffreg$ the critical set is the union of $C_0\cong\Sigma$ and
$\mu$ critical point towers
\begin{equation}\label{eq:RF-crit-mf-C}
\begin{split}
     C:=\Crit\,\Aa^F
   &=C_0 \mathop{\dot{\cup}} \mathop{\dot{\bigcup}}_{k\in\Z^*} C_1^k\dots
     \mathop{\dot{\cup}} \mathop{\dot{\bigcup}}_{k\in\Z^*} C_\mu^k\\
   &\subset\Ll\Sigma\times \R
\end{split}
\end{equation}
as illustrated by Figure~\ref{fig:fig-crit-point-towers}.
The floors are compact connected manifolds, namely the ground floor
$C_0\cong\Sigma$ together with the upper ($k>0$) and
lower ($k<0$) circle floors $C_i^k\cong\SS^1$ given by~(\ref{eq:crit-circles}).
It is convenient to identify the constant critical points
$(\lpz,0)$ of $\Aa^F$, action zero,  with the points
$p=z(0)$ of the compact hypersurface $\Sigma$ itself. 
Remark~\ref{rem:fin-crit-pt-towers}
describes how each closed characteristic $P_i$ together with a chosen point
$p_i\in P_i\subset\Sigma$
corresponds to a simple critical point $\lpz_i$ of $\Aa^F$ with
$\lpz_i(0)=p_i$ whose prime speed $\sigma_i$ is the prime period
$\LM_i:=\pertau_{\ror_i}$ of the simple Reeb loop
$\ror_i:\R/\LM_{\lpz_i} \Z\INTO\Sigma$ that parametrizes the closed characteristic
$P_i$ and is determined by the initial condition $\ror_i(0)=p_i$; see
also Remark~\ref{rem:fin-crit-pt-towers}.

\begin{remark}[Critical set]\label{rem:RF-crit-set-bd-action}
The restrictions $\Aa^F|_{C_i^k}\equiv k\LM_i$ and $\Aa^F|_{C_0}\equiv 0$ of the
action functional to components are constant
by Exercise~\ref{exc:action-spectrum}.
So it makes sense to speak of the \emph{action value of a critical component}.
As $i=1,\dots,\mu$ only runs through a \emph{finite} set,
caused by compactness and transverse non-degeneracy of $\Sigma$,
there can only be finitely many critical components
with actions in a given bounded interval $[a,b]$;
see Figure~\ref{fig:fig-crit-point-towers}.
In other words, the set of critical points
$$
     C^{[a,b]}:=\Crit^{[a,b]}\Aa^F:=\{a\le\Aa^F\le b\}\cap\Crit\,\Aa^F
$$
whose actions lie in an interval $[a,b]$ form a \emph{closed} submanifold
$C^{[a,b]}\subset\Ll \Sigma\times\R$ diffeomorphic to a finite union
of embedded circles whenever $0\notin[a,b]$; otherwise,
there is in addition one connected component diffeomorphic to $\Sigma$.
\end{remark}

Perturbing $F$ amounts to perturbing $\Sigma=F^{-1}(0)$, of course.
However, by Exercise~\ref{exc:reg_non-deg-nbhd}~a), small perturbations
will not leave the class of convex exact hypersurfaces
and even the restricted contact type property of $\Sigma$ will be
preserved.\footnote{
  The contact condition is open;
  cf. Exercise~\ref{exc:contact-volume_form}~(b).
  }
So from now on
\begin{itemize}
\item
  we assume that $\Sigma$ in Assumption~\ref{ass:RF}
  is, in addition, non-degenerate\footnote{
  The non-degeneracy assumption on $\Sigma$
  is justified by invariance of $\RFH$ under smooth variations
  of $\Sigma$ up to natural isomorphism;
  see Section~\ref{sec:RF-continuation} and
  Remark~\ref{rem:generic-Sigma}.
  }
\item
  with defining Hamiltonian $F\in\Ff(\Sigma)\cap\Ffreg$;
  cf. Exercise~\ref{exc:reg_non-deg-nbhd}~b).
\end{itemize}
Given such $\Sigma=F^{-1}(0)$, the goal is to find suitable
auxiliary data $(h,g)$ that allows to define Floer homology\footnote{
  meaning that the chain groups should be generated by
  some finite, or at least discrete, critical set and the boundary
  should count isolated flow lines connecting
  critical points
  }
and has the property that different choices lead to naturally isomorphic
homology groups.
A suitable candidate is
\begin{equation*}
  \textsf{Frauenfelder's implementation of Morse-Bott theory by
              cascades~\citerefRF{Frauenfelder:2004a}}
\end{equation*}
which, in addition, requires to fix a Morse-Smale pair $(h,g)$ consisting of
\begin{itemize}
\item
  a Morse function $h:C\to\R$ and
\item
  a Riemannian metric $g$ on the critical manifold $C=\Crit\,\Aa^F$.
\end{itemize}
While the critical set $\Crit\, h\subset\Crit\,\Aa^F$ of
$h$ is not necessarily finite, its subset
\begin{equation}\label{eq:RF-bd-crit-FINITE}
     \Crit^{[a,b]} h:=C^{[a,b]}\cap \Crit\, h=\Crit\, h|_{C^{[a,b]}}
\end{equation}
is finite, as the manifold $C^{[a,b]}$ is closed, see
Remark~\ref{rem:RF-crit-set-bd-action}, and as $h$ is Morse.

\subsubsection{Chain groups}
For $F\in\Ff(\Sigma)\cap\Ffreg$ the Floer chain group $\CF(\Aa^{F},h)$
is defined as the vector space over $\Z_2$
that consists of all formal sums 
$$
     x=\sum_{c\in\Crit\, h} x_c c
$$
of critical points of $h$ such that the $\Z_2$-coefficients $x_c=x_c(x)$ in
such a formal sum $x$ satisfy the \textbf{\Index{upward finiteness
condition}}\index{finiteness condition!upward}
\begin{equation}\label{eq:RF-FINITE}
     \Bigl|\bigl\{c\in\Crit\, h\mid\text{$x_c(x)\not= 0$
     and $\Aa^{F}(c)\ge \kappa$}\bigr\}\Bigr|
     <\infty,\quad\forall\kappa\in\R.
\end{equation}
In words, given a formal sum $x$, one requires finiteness of the number
of non-zero coefficients $x_c(x)$ above any given action value $\kappa$.
\begin{definition}
The Floer chain group $\CF(\Aa^{F},h)$ is the $\Z_2$-vector space generated by
\textbf{\Index{upward finite formal sums}}\index{formal sum!upward finite}
$x$ of critical points of the Morse function $h$ on the critical set $C$
of the Rabinowitz action functional $\Aa^{F}$.
\end{definition}

\subsubsection{Connecting cascade flow lines -- upward flows}
\begin{figure}
  \centering
  \includegraphics[width=0.95\textwidth]
                             {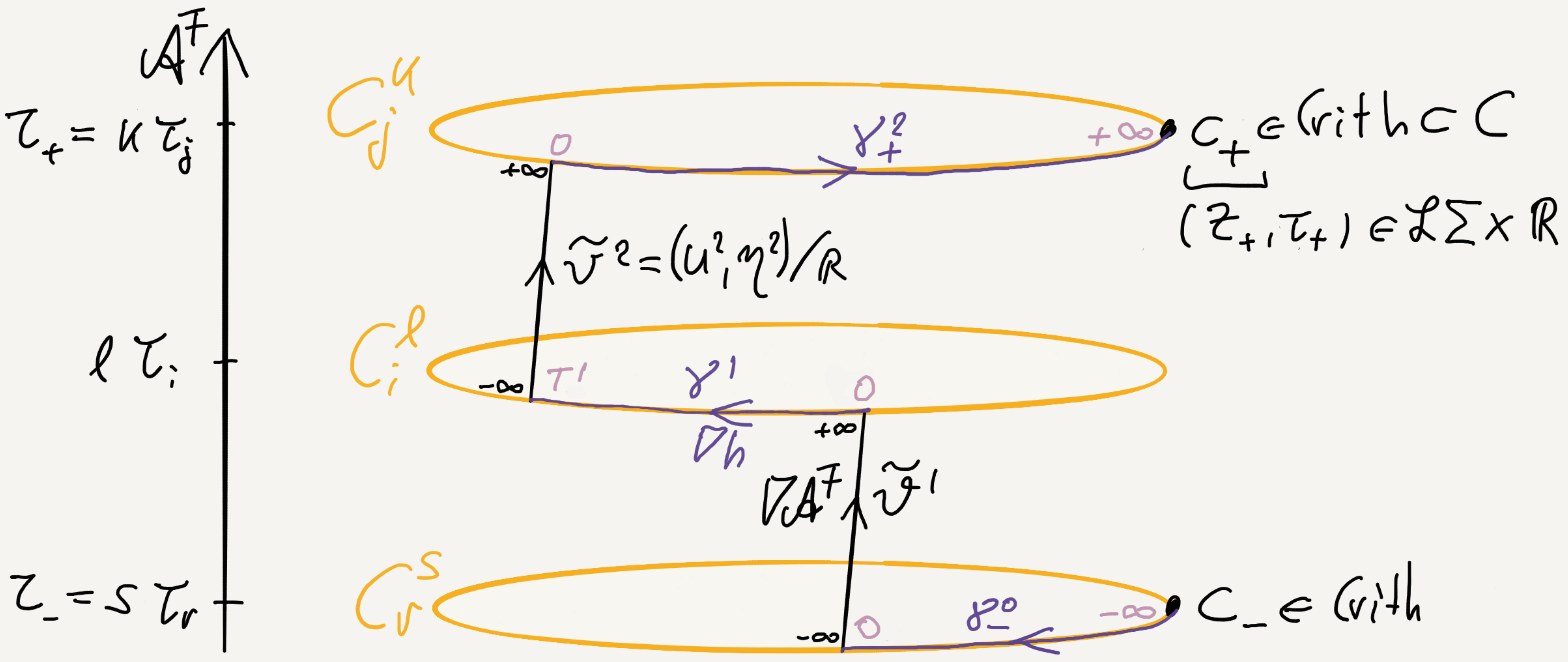}
  \caption{Cascade flow $\Gamma
                 =\bigl(\gamma^0_-,\tilde\upsilon^1,\gamma^1,
                 \tilde\upsilon^2,\gamma^2_+\bigr)\in
                 \Mm_{c_-c_+}(\Aa^{F},h,J,g)$}
  \label{fig:fig-RF-cascades}
\end{figure}
Given $F\in\Ff(\Sigma)\cap\Ffreg$, on the critical manifold $C:=\Crit\,\Aa^F$
(whose components are circles and one component is given by
the closed hypersurface $\Sigma$ according to
Theorem~\ref{thm:RF-2nd-cat}) consider the \emph{Morse gradient flow}
generated by the Morse function $h:C\to\R$ and the
Riemannian metric $g$ on $C$ through the ODE
$\dot\gamma=\nabla^g h(\gamma)$ for smooth maps $\gamma:\R\to\C$.
Given two critical points $c_\pm=(\lpz_\pm,\LM_\pm)\in\Crit\, h\subset C$,
a\index{connecting flow line!with cascades}\index{cascade flow line}
\textbf{connecting upward flow line with cascades}
is a tuple of the form
$$
     \Gamma=\bigl(\gamma^0_-,\tilde\upsilon^1,
     \gamma^1,\dots,\tilde\upsilon^{\ell-1},\gamma^{\ell-1},
     \tilde\upsilon^\ell,\gamma^\ell_+\bigr)
$$
where $\ell\in\N_0$ is the number of cascades $\tilde\upsilon^j$ and such that
$$
     \dot\gamma^j=\nabla h(\gamma^j),\quad
     \p_s\upsilon^j=\grad\Aa^F(\upsilon^j),\quad
     \tilde\upsilon^j=[\upsilon^j:\R\to\Ll V\times \R].
$$
The notation $\tilde\upsilon=[v]="\{\upsilon\}/\R"$ 
is meant to indicate unparametrized flow lines or, equivalently, flow trajectories
that only differ by shifting the $s$ variable are considered equivalent.
See Figure~\ref{fig:fig-RF-cascades} for a
connecting cascade flow line with $\ell=2$ cascades.
Actually each intermediate Morse trajectory comes with a finite time
$T^j\ge 0$ and is defined on a \emph{finite time} interval, namely
$$
     \gamma^j:[0,T^j]\to C.
$$
In contrast, the two ends $\gamma^0_-:(-\infty,0]\to C$ and
$\gamma^\ell_+:[0,\infty)\to C$ are \emph{semi-infinite} Morse
trajectories. In any case, neither of the Morse trajectories
is invariant under the time-$s$ shift $\R$ action.

More precisely, the tuple $\Gamma$ starts with a semi-infinite Morse trajectory
$\gamma^0_-$ backward asymptotic to the given critical point
$c_-\in\Crit\, h$ and whose position at time zero, namely
$\gamma^0_-(0)$, is the backward asymptote $\upsilon^1_-$ of some
$\grad\Aa^F$ flow trajectory $\upsilon^1\in\tilde\upsilon^1=[\upsilon^1]$
defined on the whole real line, see~(\ref{eq:RF-UGF}), where
\begin{equation}\label{eq:RF-asymp-lim}
     \upsilon^j_\mp=\lim_{s\to\mp\infty}\upsilon^j(s)\in C,\quad
     \upsilon^j\in[\upsilon^j],\quad
     \upsilon^j=(u^j,\LMpath^j):\R\to\Ll V\times\R.
\end{equation}
The other, positive, asymptote $\upsilon^1_+\in C$
provides the initial point $\gamma^1(0)$ of a finite time
Morse trajectory $\gamma^1:[0,T^1]\to C$ whose endpoint
$\gamma^1(T^1)$ is a backward asymptote $\upsilon^2_-$.
Continuing this way one reaches the final piece, namely the
semi-infinite Morse trajectory
$\gamma^\ell_+:[0,\infty)\to C$ that starts at the previous
backward asymptote point $\upsilon^\ell_+$ at time $0$ and is itself
forward asymptotic to the second given critical point $c_+$.
The unparametrized Morse-Bott flow lines $\tilde\upsilon^j$ are called
\textbf{\Index{cascades}}.

The \textbf{moduli space of connecting cascade flow lines}
\begin{equation}\label{eq:cascade-mod-space}
     \Mm_{c_-c_+}=\Mm_{c_-c_+}(\Aa^{F},h,J,g)
\end{equation}
consists of all cascade flow lines $\Gamma$ connecting $c_-$
and\index{$\Mm_{c_-c_+}$ connecting cascade flow lines}
$c_+$.\index{moduli space!of connecting cascade flow lines}
The following are non-trivial -- even in finite dimension -- although well known:
See~\citerefRF[App.\,A]{Frauenfelder:2004a} and~\citeintro[App.\,A]{Cieliebak:2009a}.
Firstly, for generic $J$ and $g$ these moduli spaces are smooth manifolds.
Secondly, it is a consequence of the Compactness
Theorem~\ref{thm:RF-compactness} below
that the 0-dimensional part $\Mm^0_{c_-c_+}$ of $\Mm_{c_-c_+}$
is compact, so a finite~set.

\subsubsection{Boundary operator and Floer homology}

Pick $F\in\Ff(\Sigma)\cap\Ffreg$.
Since for generic $J$ and $g$ the 0-dimensional part $\Mm^0_{c_-c_+}$ of
the space of connecting flow lines $\Mm_{c_-c_+}(\Aa^{F},h,J,g)$
is compact, hence a finite set, the number of elements modulo two
$$
     n(c_-,c_+):=\#_2(\Mm^0_{c_-c_+})
$$
is well defined. Then the Floer boundary operator
$$
     \p:\CF(\Aa^{F},h)\to \CF(\Aa^{F},h)
$$
is defined as the linear extension of
\begin{equation}\label{eq:RF-bdy}
     \p c_+:=\sum_{c_-\in\Crit\, h} n(c_-,c_+) c_-
\end{equation}
for $c_+\in\Crit\, h$. This is illustrated  by Figure~\ref{fig:fig-RF-UGF}; cf. also
Remark~\ref{rem:RF-non-standard-order}.

\begin{exercise}
The formal sum in the definition of $\p$ satisfies
the finiteness condition~(\ref{eq:RF-FINITE}).
[Hint: The action non-decreases along cascade flow lines,~(\ref{eq:RF-bd-crit-FINITE}),
Exercise~\ref{exc:action-spectrum}, Arzel\`{a}-Ascoli Theorem~\ref{thm:AA}.]
\end{exercise}

The boundary operator property $\p^2=0$ follows by standard gluing and
compactness arguments, compare Proposition~\ref{prop:HF-bound-op},
once one has the Compactness Theorem~\ref{thm:RF-compactness}
for the 1-dimensional part of moduli space, together with
finiteness of the number $\abs{\Crit^{[a,b]} h}=\abs{\Crit\, h|_{C^{[a,b]}}}$
of critical points of the Morse function $h$ in any finite action interval
$[a,b]$ or, equivalently, on the closed manifold $C^{[a,b]}$.
By definition \textbf{Floer homology of the Rabinowitz action functional}
$\Aa^{F}$ is the homology of this chain complex,
namely \index{$\HF(\Aa^{F})$ (Rabinowitz-) Floer homology}
\begin{equation}\label{eq:RF-homology-F}
     \HF(\Aa^{F}):=\frac{\ker\p}{\im\p},\qquad
     F\in\Ff(\Sigma)\cap\Ffreg.
\end{equation}
We already dropped $(h,J,g)$ from the notation since
$\HF(\Aa^{F})$ does not depend on that choice, up to
canonical isomorphism, as follows by the standard continuation
techniques detailed in Section~\ref{sec:FH-continuation}.
Via continuation one also shows independence of the
regular defining Hamiltonian $F\in\Ff(\Sigma)\cap\Ffreg$.

\subsubsection{Rabinowitz-Floer homology of a convex exact hypersurface}

\begin{remark}\label{rem:RFH-well-defined}
Given $\Sigma$ as in Assumption~\ref{ass:RF},
observe the following: There is no regular defining Hamiltonian $F\in\Ff(\Sigma)$ iff
$\Sigma$ itself is already degenerate; see Exercise~\ref{exc:reg_non-deg}.
Hence fixing non-regularity of $F$ is equivalent to perturbing~$\Sigma$.
In practice pick $F_0$ in the dense subset $\Uureg^F:=\Uu^F\cap\Ffreg$
of the small open neighborhood $\Uu^F$ of $F$ in $\Ff$
provided by Exercise~\ref{exc:reg_non-deg-nbhd}~b).
In particular, the zero set $\Sigma_0:=F_0^{-1}(0)$
is a non-degenerate convex exact hypersurface nearby~$\Sigma$.
The Floer homology of $\Aa^{F_0}$ is then defined
by~(\ref{eq:RF-homology-F}).
For any two such choices $F_0$ and $F_1$ there is
the standard continuation isomorphism on homology.
\end{remark}

\textbf{\Index{Rabinowitz-Floer homology}}
of\index{$\RFH(\Sigma)$ RF homology of convex exact hypersurface}
a convex exact hypersurface $\Sigma\subset(V,\lambda)$
with defining Hamiltonian $F\in\Ff(\Sigma)$ is defined by
\begin{equation}\label{eq:RF-homology}
     \RFH(\Sigma)
     :=\HF(\Aa^{F_0})
     ,\qquad
     F_0\in\Uureg^F.
\end{equation}
By continuation $\RFH(\Sigma)$ does not depend on the choice of $F_0$;
cf. Remark~\ref{rem:generic-Sigma}.
Observe that $F_0$ lies in $\Ffreg$ and in $\Ff(\Sigma_0)$: It is defining
for the non-degenerate convex exact hypersurface
$\Sigma_0:=F_1^{-1}(0)$ nearby $\Sigma$.
For $\SS^1$-equivariant Rabinowitz-Floer homology
see~\citerefRF{Frauenfelder:2016a}.

\subsection{Compactness of moduli spaces}\label{sec:RF-compactness}
Definition and property $\p^2=0$ of the boundary operator $\p$
both hinge on compactness properties of the moduli spaces
of connecting flow lines with cascades.
Recall that a connecting cascade flow line consists of (semi-)finite
Morse gradient trajectories\footnote{
  in general not connecting ones:
  Morse trajectories have finite life time, except the
  initial and ending one which are semi-infinite, also called
  semi-connecting.
  }
$\gamma^j$ along the critical set $C$ of $\Aa^F$ and the cascades
$\tilde\upsilon^j$ themselves. Cascades are, modulo $s$-shift,
connecting trajectories $\upsilon=(u,\LMpath)$ of $\grad\Aa^F$
between two critical sets in $\Ll V\times\R$.\footnote{
  \emph{Connecting} trajectories necessarily live on the infinite domain
  $(\R\times\SS^1)\times\R$ and their moduli spaces are subject to division
  by the free $\R$-action given by shifting the $s$-variable.
  }
While compactness up to broken trajectories
for Morse trajectories can even be handled within
finite dimensional dynamical systems,
see e.g.~\cite{weber:2015-MORSELEC-In_Preparation},
in the case of cascades compactness of the loop space
component $u$ is also standard, but this
is not so for the Lagrange multiplier component $\LMpath$.
Let us detail this: To prove compactness up to broken
trajectories of the set of $\grad\Aa^F$ trajectories $\upsilon=(u,\LMpath)$
that are subject to a uniform action bound
\begin{equation}\label{eq:RF-unif-action-bd-ass}
     a\le\Aa^F(\upsilon(s))\le b
\end{equation}
for all times $s\in\R$, the following apriori bounds are sufficient: Namely,
\begin{itemize}
\item[(i)]
  a uniform $C^0$ bound on $u$;
\item[(ii)]
  a uniform $C^0$ bound on $\LMpath$;
\item[(iii)]
  a uniform $C^0$ bound on $\p_su$ (thus on $\p_tu$).
\end{itemize}
(i) is fine
due to our choice of \emph{cylindrical} almost complex
structures; see Remark~\ref{rem:RF-non-standard-order}.
(ii)~was new and required new techniques when
it was established in~\citeintro[\S 3.1]{Cieliebak:2009a}.
Thus we shall outline their argument below.
(iii)~follows from the exactness assumption of the symplectic form
$\omega=d\lambda$ via standard bubbling-off analysis as shown
in Section~\ref{sec:FH-comp}.

\begin{exercise}
Doesn't one need a uniform $C^0$ bound on the derivative $\LMpath^\prime$
as well? And why is the $C^0$ bound requirement (iii) for the
derivative $\p_s u$ not listed directly after condition (i) for $u$ itself?
\end{exercise}

Let us state the compactness theorem
in a rather general form that will serve simultaneously
\begin{itemize}
\item
  the continuation problem, where $s$-dependent Hamiltonians appear;
\item
  the proof of the Vanishing Theorem~\ref{thm:RF-vanishing}, where
  the rescalings $\chi$ help.
\end{itemize}
Of course, the following compactness theorem
for families requires three $C^0$ bounds as in (i--iii)
which are, in addition, uniform in the family parameters.

\begin{theorem}[Cascade compactness]\label{thm:RF-compactness}
Suppose that $F:[0,1]\times V\to\R$ is a smooth function
such that each Hamiltonian $F_\sigma:=F(\sigma,\cdot)$ defines
a convex exact hypersurface $\Sigma_\sigma$
and $\chi:[0,1]\times\SS^1\to[0,\infty)$ is a smooth function such that each
$\chi_\sigma:=\chi(\sigma,\cdot)$ integrates to one.
Let $\eps>0$ and $c<\infty$ be constants as provided
by~{\citeintro[Prop.~3.4]{Cieliebak:2009a}},\footnote{
  {\citeintro[Prop.~3.4]{Cieliebak:2009a}} is a generalization
  of Proposition~\ref{prop:control-eta} to families
  $(F_\sigma,\Sigma_\sigma,\chi_\sigma)$.
  }
and suppose that the following inequality holds
\begin{equation}\label{eq:ass-casc-comp}
     \left(c+\frac{\norm{F}_{\infty}}{\eps^2}\right)\cdot
     \left(\norm{\p_\sigma F}_{\infty}+\norm{\p_\sigma\chi}_{\infty}\cdot
     \norm{F}_{\infty}\right)\le\frac{1}{8}.
\end{equation}
Then the following is true.
Assume that $\upsilon_\nu=(u_\nu,\LMpath_\nu)\in C^\infty(\R\times\SS^1,V)
\times C^\infty(\R,\R)$ is a sequence of trajectories of the $s$-dependent
gradient $\grad\Aa^{\chi_s F_s}$ and there are bounds $a,b\in\R$ such that
$$
     a\le\lim_{s\to-\infty}\Aa^{\chi_s F_s}(\upsilon_\nu(s))
$$
for all times $s\in\R$ and that
$$
     \lim_{s\to\infty}\Aa^{\chi_s F_s}(\upsilon_\nu(s))\le b.
$$
Then there is a subsequence, still denoted by $\upsilon_\nu$,
and a trajectory $\upsilon$ of $\grad\Aa^{\chi_s F_s}$ such that
the subsequence $\upsilon_\nu=(u_\nu,\LMpath_\nu)$ converges to
$\upsilon=(u,\LMpath)$ in the $C_{\rm loc}^\infty$-topology.
\end{theorem}

\begin{proof}
\citeintro[Thm.~3.6]{Cieliebak:2009a}.
\end{proof}

\subsubsection{Uniform bounds for multiplier paths $\mbf{\LMpath}$ --
  contact type enters}

For simplicity we only consider the case of
trajectories of the $s$-independent gradient $\grad\Aa^F$.
Moreover, we only sketch proofs; for details
see~\citeintro[\S 3.1]{Cieliebak:2009a}.

\begin{remark}[Contact type of $\Sigma$ enters]
\label{rem:powerhouse}
(i)~So far the geometric condition on $\Sigma$ to be not just any
bounding hypersurface in $(V,\lambda)$, but a convex exact one,\footnote{
  We assume contact type for $\Sigma$, otherwise, add an exact
  form to $\lambda$; cf. Exercise~\ref{exc:equiv-1-form-restrict}.
  }
has not been used. That will change now in order to obtain a uniform
$C^0$ bound for the Lagrange multiplier components $\LMpath$
of all trajectories $\upsilon=(u,\LMpath)$ subject to the same
action bounds $a$ and $b$ as in~(\ref{eq:RF-unif-action-bd-ass}).

(ii)~Some condition on the energy surface $\Sigma$
is indeed necessary, given that there are energy surface
counterexamples $\Sigma^\prime$ in $\R^{2n}$ to the Hamiltonian
Seifert conjecture; cf.~(\ref{eq:RF-counter-Ham-Seif}).
But even in the absence of closed characteristics the constant loops
corresponding to the points of $\Sigma$ are still critical points of $\Aa^F$
thereby giving rise to nontriviality
$\RFH(\Sigma^\prime)\simeq\Ho(\Sigma;\Z_2)\not=0$;
cf. Exercise~\ref{exc:RFH-Sigma}
and~\citeintro[Pf. of Cor.~1.5]{Cieliebak:2009a}.
But this would contradict the Vanishing Theorem~\ref{thm:RF-vanishing}
since in $\R^{2n}$ any compact subset is displaceable.
\end{remark}

\begin{proposition}\label{prop:control-eta}
Let $F\in\Ff(\Sigma)$ be a defining Hamiltonian.
Then there are constants $\eps>0$ and $c<\infty$ such that
for pairs $(\lpz,\LM)\in\Ll V\times\R$ it holds
$$
     \Norm{\grad\Aa^F(\lpz,\LM)}\le\eps\quad\Rightarrow\quad
     \Abs{\LM}\le c\left(\abs{\Aa^F(\lpz,\LM)}+1\right).
$$
\end{proposition}

The proposition tells that near critical points
the multiplier part of a pair $(\lpz,\LM)\in\Ll V\times\R$
is bounded in terms of the pair action.
So if $(\lpz,\LM)$ is one element of a whole trajectory
$(u,\LMpath):\R\to \Ll V\times\R$ subject to the
bound~(\ref{eq:RF-unif-action-bd-ass}), thus
\begin{equation}\label{eq:ac-bound-7687}
     \abs{\Aa^F(u(s),\LMpath(s))}\le\kappa:=\max\{\abs{a},\abs{b}\},\quad s\in\R,
\end{equation}
then the proposition provides a uniform bound on $\abs{\LMpath(s)}$ for $s$
near $\pm\infty$. So it remains to deal with $\abs{\LMpath}$ along compact
intervals in $\R$, one compact interval for each trajectory of the, generally
non-compact, family under consideration.

\begin{corollary}[Uniform $\LMpath$ bound]\label{cor:C0-eta}
Let $\upsilon=(u,\LMpath)\in C^\infty(\R\times\SS^1,V)
\times C^\infty(\R,\R)$ be a trajectory of
$\grad\Aa^{F}$ along which the action remains in a compact interval,
say in $[a,b]$; cf.~(\ref{eq:RF-unif-action-bd-ass}).
Then the $L^\infty$-norm of $\LMpath$ is bounded uniformly in terms of
a constant $c$ that depends on $a,b$, but not on $\upsilon$.
\end{corollary}

\begin{proof}[Sketch of proof of Proposition~\ref{prop:control-eta}]
(For details see~{\citeintro[Prop.~3.2]{Cieliebak:2009a}}.)
The key input, due to $\Sigma$ being both an energy surface $F^{-1}(0)$
\emph{and} of contact type with respect to $\alpha$, is the
\underline{coupling $X_F=R_\alpha$} of Hamiltonian and Reeb dynamics
along $\Sigma$; see~(\ref{eq:Ham=Reeb}). To illustrate the effect of the
coupling note that for a critical point $(\lpz,\LM)$ of $\Aa^F$
\begin{equation}\label{eq:RF-warm-up}
     \Aa^F(\lpz,\LM)=\int_0^1
     \Bigl(\underbrace{\lambda|_{\lpz}(\dot\lpz)}_{\alpha(\LM R_\alpha)\equiv \LM}
     -\LM\underbrace{F\circ \lpz}_{0} \Bigr) dt
     =\LM.
\end{equation}
This actually solves Exercise~\ref{exc:action-spectrum}.
Let's see how much of this identity survives
for a general pair $(u,\LMpath)$ whose only restriction is
that the loop part $u$ must stay in a small
neighborhood $U_\delta$ of $\Sigma$.

\vspace{.1cm}\noindent
\textbf{I.}
\textit{There are constants $\delta>0$ and $c_\delta<\infty$
with the following significance. For every pair $(\lpz,\LM)\in\Ll V\times\R$
whose loop part $\lpz$ remains $\delta$-near to $\Sigma$
in the sense that $\lpz(\SS^1)\subset U_\delta:=F^{-1}(-\delta,\delta)$
the Lagrange multiplier satisfies the estimate}
$$
     \Abs{\LM}\le2\Abs{\Aa^F(\lpz,\LM)}+c_\delta\Norm{\grad \Aa^F(\lpz,\LM)}.
$$

The key step is to obtain the two constants.
By compactness of $\Sigma=F^{-1}(0)$
and zero being a regular value of $F$ the closure of $U_\delta$
is compact for sufficiently small $\delta>0$. Now choose $\delta>0$
smaller, if necessary, such that
$$
     \lambda_p \left(X_F(p)\right)\ge\tfrac12+\delta,\quad p\in U_\delta.
$$
Such $\delta$ exists since 
\underline{$\lambda(X_F)=\alpha(R_\alpha)\equiv 1$}
along $\Sigma=F^{-1}(0)$: To see this use
contact type, see~(\ref{eq:Ham=Reeb}), and the definition
of $R_\alpha$, see Exercise~\ref{exc:contact-volume_form}~(c).
The constant $c_\delta:=2\norm{\lambda|_{U_\delta}}_\infty$
is finite since  $U_\delta$ is of compact closure.
The desired estimate is a rather mild generalization
of~(\ref{eq:RF-warm-up}), see~\citeintro[p.264]{Cieliebak:2009a}.

\vspace{.1cm}\noindent
\textbf{II.}
\textit{For each $\delta>0$ there is a constant $\eps=\eps(\delta)>0$
such that whenever a pair $(\lpz,\LM)$ satisfies
$\norm{\grad \Aa^F(\lpz,\LM)}\le\eps$, then
the loop part $\lpz$ remains in $U_\delta$.
}

\vspace{.1cm}
To show II one first analyzes $\lpz$ in two cases, in each case
forgetting one of the two components of $\norm{\grad \Aa^F(\lpz,\LM)}$;
see~(\ref{eq:RF-norm-grad-Aa}). Excluding both~cases~yields~II.
\newline
\textsc{Case~1:}
There are times $t_0,t_1\in\SS^1$ with
$\abs{F(\lpz(t_0))}\ge\delta$ and $\abs{F(\lpz(t_1))}\le\delta/2$.
In this case by periodicity of $\lpz$ and continuity of $F$ there are two points,
again denoted by $t_0,t_1\in\SS^1$, with $t_0<t_1$ and
$\abs{F\circ \lpz(t)}\in[\frac{\delta}{2},\delta]$ on $[t_0,t_1]$.
Set $\mu:=\max_{x\in \widebar{U}_\delta, t\in\SS^1}\abs{\nabla^{J_t} F(x)}_{g_{J_t}}$
and forget component \emph{two} in~(\ref{eq:RF-norm-grad-Aa}) to get\footnote{
  Use that by Cauchy-Schwarz 
  $\norm{f}_2=\norm{f}_2\norm{1}_1\ge \langle f,1\rangle_{L^2}=\norm{f}_1$
  for $f\in C^\infty(\SS^1,\R)$.
  }
\begin{equation*}
\begin{split}
     \Norm{\grad \Aa^F(\lpz,\LM)}
   &\ge\Norm{\dot \lpz-\LM X_F(\lpz)}_2\\
   &\ge\int_{t_0}^{t_1}\Abs{\dot \lpz-\LM X_F(\lpz)} dt\\
   &\ge\dots\\
   &\ge\frac{\delta}{2\mu}.
\end{split}
\end{equation*}
The omitted steps, using e.g. $dF(X_F)=0$,
are detailed in~\citeintro[p.265]{Cieliebak:2009a}.
\newline
\textsc{Case~2:}
The loop $\lpz$ lives outside $U_{\delta/2}$.
Forgetting component \emph{one}
in~(\ref{eq:RF-norm-grad-Aa})\footnote{
  To see the identity note that $F(\lpz(t))$ is non-zero, hence will
  not change sign; cf.~(\ref{eq:jhjkhkj7876}).
  }
$$
     \Norm{\grad \Aa^F(\lpz,\LM)}
     \ge\Bigl|\int_0^1 F(\lpz(t))\, dt\Bigr|
     =\int_0^1 \underbrace{\abs{F(\lpz(t))}}_{\ge\delta/2}\, dt
     \ge\frac{\delta}{2}.
$$

To prove Step~II pick any $\delta>0$, set
$$
     \eps:=\frac{\delta}{4\max\{1,\mu\}}
     <\min\{\frac{\delta}{2},\frac{\delta}{2\mu}\},
$$
and assume $(\lpz,\LM)\in\Ll V\times\R$ satisfies
$\norm{\grad \Aa^F(\lpz,\LM)}\le\eps$.
Neither case 1 nor 2 applies to $\lpz$.
So $\lpz$ hits $U_{\delta/2}$ ($\neg$ case 2),
but then it cannot leave $U_\delta$ ($\neg$ case 1).

\vspace{.1cm}\noindent
\textbf{III.}
\textit{We prove the proposition.
}

\vspace{.1cm}
Choose the constants $\delta$ and $\eps=\eps(\delta)$
of Steps~I and~II, respectively, and set $c:=\max\{2,c_\delta\eps\}$.
Suppose $\norm{\grad \Aa^F(\lpz,\LM)}\le \eps$. Then by Step~II the loop part
$\lpz$ remains in $U_\delta$, hence Step~I applies and yields
\begin{equation*}
\begin{split}
     \Abs{\LM}
   &\le2\Abs{\Aa^F(\lpz,\LM)}+c_\delta\Norm{\grad\Aa^F(\lpz,\LM)}\\
   &\le c\Abs{\Aa^F(\lpz,\LM)}+c_\delta\eps.
\end{split}
\end{equation*}
This concludes the outline of the proof of Proposition~\ref{prop:control-eta}.
\end{proof}

\begin{proof}[Proof of Corollary~\ref{cor:C0-eta} (Trajectories)]
Pick $\eps>0$ as in Proposition~\ref{prop:control-eta}.\footnote{
  By Step~II above $\eps=\eps(\delta)$ given any sufficiently small
  constant $\delta>0$.
  }
Set $u_s:=u(s,\cdot)$ and $\LMpath_s:=\LMpath(s)$.
The key tool is the quantity given for $\sigma\in\R$ by
$$
     \Tt_\sigma
     :=\inf\{s\ge 0\mid
     \norm{\grad\Aa^F((u,\LMpath)_{\sigma+s})}<\eps\}
     ,\quad
     (u,\LMpath)_{s}:= (u_s,\LMpath_s).
$$
This quantity helps twice. Firstly, given a trajectory $(u,\LMpath)$
and a time $\sigma$ element $(u,\LMpath)_\sigma$ of gradient norm
$\ge\eps$ (otherwise $\Tt_\sigma=0$), the function $\Tt_\sigma$
measures for how long the gradient norm will remain $\ge\eps$, so
for how long the trajectory will not get too close to a critical point. By
assumption~(\ref{eq:RF-unif-action-bd-ass}) we get the estimate
\begin{equation*}
\begin{split}
     b-a
   &\ge\lim_{s\to\infty}\Aa^F((u,\LMpath)_s)-\lim_{s\to-\infty}\Aa^F((u,\LMpath)_s)\\
   &=\int_{-\infty}^\infty\norm{\grad\Aa^F(u,\LMpath)}^2\, ds\\
   &\ge\int_\sigma^{\sigma+\Tt_\sigma}
     \underbrace{\norm{\grad\Aa^F(u,\LMpath)}^2}
        _{\text{$\ge\eps^2$ on $(\sigma,\sigma+\Tt_\sigma)$}}\, ds\\
   &\ge \Tt_\sigma\eps^2
\end{split}
\end{equation*}
which holds true for every $\sigma\in\R$ and, by the way,
also shows finiteness $\Tt_\sigma<\infty$.
Secondly, for any $\sigma\in\R$ the gradient norm of
the trajectory element $(u,\LMpath)_{\sigma+\Tt_\sigma}$
at time $\sigma+\Tt_\sigma$ is $\le\eps$.
Hence the pair $(u,\LMpath)_{\sigma+\Tt_\sigma}$ satisfies the
assumption of  Proposition~\ref{prop:control-eta}, so together with
the action bound $\kappa$ in~(\ref{eq:ac-bound-7687}) we get that
\begin{equation*}
\begin{split}
     \Abs{\LMpath(\sigma+\Tt_\sigma)}
   &\le c\left(\abs{\Aa^F((u,\LMpath))_{\sigma+\Tt_\sigma}}+1\right)\\
   &\le c(\kappa+1)=:c_\kappa
\end{split}
\end{equation*}
for every $\sigma\in\R$. Putting things together one obtains the
desired estimate
\begin{equation*}
\begin{split}
     \Abs{\LMpath(\sigma)}
   &=\Bigl|\LMpath(\sigma+\Tt_\sigma)-\int_{\sigma}^{\sigma+\Tt_\sigma}
     \p_s\LMpath(s)\, ds\Bigr|\\
   &\le \Abs{\LMpath(\sigma+\Tt_\sigma)}
     +\int_{\sigma}^{\sigma+\Tt_\sigma}
     \underbrace{\abs{\p_s\LMpath(s)}}_{\le\Norm{F}_\infty\int_0^1dt}\, ds\\
   &\le c_\kappa+\Tt_\sigma\Norm{F}_\infty\\
   &\le c_\kappa+\frac{\norm{F}_\infty(b-a)}{\eps^2}
\end{split}
\end{equation*}
for every $\sigma\in\R$ using that $\p_s\LMpath$ satisfies the
gradient flow equation~(\ref{eq:RF-UGF}).
\end{proof}

\subsection{Continuation}\label{sec:RF-continuation}
To prove invariance of Floer homology $\HF(\Aa^F)$ defined
by~(\ref{eq:RF-homology-F}) for a
convex exact hypersurface $\Sigma$ with regular defining Hamiltonian $F$,
not only under change of the regular defining Hamiltonian, but even under convex
exact deformations of $\Sigma$ itself, see Theorem~\ref{thm:RF},
suppose $\{F_s\}_{s\in[0,1]}$ is a smooth family of defining
Hamiltonians of convex exact hypersurfaces $\Sigma_s$ in
$(V,\lambda)$.
The construction of natural continuation maps follows precisely
the same steps as in Section~\ref{sec:FH-continuation},
see~\citeintro[\S 3.2]{Cieliebak:2009a} for details,\footnote{
  The $r$-homotopy of $s$-homotopies $\widebar H^{\bar\chi,r}_s$
  in~\citeintro[p.\,276]{Cieliebak:2009a} should be
  $H^{\chi_{1-s}}_{r(1-s)}$, not $H^{\chi_{1-s}}_{1-rs}$,
  in order that for each $r$ the initial point (at $s=0$)
  coincides with the endpoint (at $s=1$) of the $s$-homotopy
  $H^{\chi,r}_s:=H^{\chi_s}_{rs}$ and so the two homotopies can be
  concatenated.
  }
once appropriate compactness properties of the spaces of
connecting cascade flow lines for the $s$-dependent
gradient $\grad \Aa^{F_s}$ have been established.
\\
The only problem is that condition~(\ref{eq:ass-casc-comp})
on smallness of the product $\norm{F}_\infty\norm{\p_s F}_\infty$
might not be satisfied in general.
A common technique is to carry out the homotopy $\{F_s\}_{s\in[0,1]}$
in $N$ steps, namely sucessively via the homotopies
$$
     F^j_s:=F_{\frac{j-1+s}{N}},\quad s\in[0,1],\quad j=1,\dots,N,
$$
and show that the continuation map provided by each of them is an
isomorphism. Since 
$$
     \Norm{\p_s F^j_s}_\infty
     =\frac{1}{N}\Norm{\p_s F_{\frac{j-1+\sigma}{N}}}_\infty
     \le\frac{1}{N} \Norm{\p_s F}_\infty
$$
condition~(\ref{eq:ass-casc-comp}) is satisfied indeed for each
homotopy $F_j$ whenever $N$ is chosen sufficiently large.
But the composition of the continuation isomorphisms
provided by each individual $F^j$ is the continuation map, hence
isomorphism, provided by $F$; cf.~(\ref{eq:Salamon-homotopy-s}).

To show that $\RFH(\Sigma)$ is well defined by~(\ref{eq:RF-homology})
requires to pick $F_0,F_1\in\Uureg^F$ as in
Remark~\ref{rem:RFH-well-defined} and show that
$\HF(\Aa^{F_0})\simeq\HF(\Aa^{F_1})$ by continuation.
For this it is sufficient, as mentioned above, to have
a smooth family $\{F_s\}_{s\in[0,1]}$ of \emph{defining
Hamiltonians} of \emph{convex exact} hypersurfaces $\Sigma_s$ in
$(V,\lambda)$ interpolating $F_0$ and $F_1$. Note that such
family is obtained simply by convex combination of $F_0$ and $F_1$;
see Remark~\ref{rem:generic-Sigma}.

In~\citerefRF[Prop.\,3.1]{Cieliebak:2010b}
it is even shown independence on the unbounded component
of $V\setminus\Sigma$, that is only $\Sigma=\p M$ and its inside,
the compact manifold-with-boundary $M$, are relevant for $\HF(\Aa^F)$.
This leads to the notation $\RFH(\p M,M)$ for $\HF(\Aa^F)$,
often abbreviated by $\RFH(\Sigma)$; see Definition~\ref{def:RF}.

\subsection{Grading}\label{sec:RF-grading}
Suppose $\Sigma$ is a convex exact hypersurface, say of restricted
contact type, and $F\in\Ff(\Sigma)$ is defining.
Throughout Section~\ref{sec:RF-grading} suppose that
\begin{itemize}
\item[(i)]
  the contact manifold $(\Sigma,\alpha)$ is simply-connected, that is
  $\pi_1(\Sigma)=0$;
\item[(ii)]
  $\Aa^F:\Ll V\times\R\to\R$ is Morse-Bott;
\item[(iii)]
  $(V,d\lambda)$ has trivial first Chern class
  over $\pi_2(V)$, that is $\Io_{c_1}=0$;
\item[(iv)]
  $h:C\to\R$ is Morse on the critical manifold
  $C:=\Crit\,\Aa^F\subset\Ll\Sigma\times\R$.
\end{itemize}
Under these conditions there exists an integer grading
$\mu=\RStrans+\INDsign_h$ taking values in $\tfrac12\Z$,
see~(\ref{eq:RF-grading}), of the Rabinowitz-Floer complex
in terms of the sum of the transverse Robbin-Salamon index
of (rescaled) Reeb loops and the signature index of a critical
point of the Morse function $h$.
For non-degenerate $\Sigma$ the transverse Robbin-Salamon
index reduces to the transverse Conley-Zehnder index $\CZtrans$
and $\mu$ will be half-integer valued.

\subsubsection{Transverse Robbin-Salamon index}
Pick a critical point $(\lpz,\LM)$ of $\Aa^F$, i.e. $\lpz:\SS^1\to
\Sigma$ satisfies $\dot \lpz=\LM R_\alpha(\lpz)$
by~(\ref{eq:crit-Reeb}). In words, the loop
$\lpz:\R/\Z\to\Sigma$ integrates the rescaled Reeb vector field
$\LM\RVF_\alpha$.\footnote{
  Alternatively consider the corresponding $\LM$-periodic
  Reeb path $\ror(t)=\Rflow_t p$ with $p=\lpz(0)$ on the
  $\LM$-dependent interval $[0,\LM]$;
  see Exercise~\ref{exc:simple-Reeb-1st-return}.
  }
Recall that $\xi:=\ker\alpha\to\Sigma$ defines a -- with respect to
$d\alpha$ symplectic -- vector bundle of rank $2n-2$
and that $T\Sigma=\xi\oplus\R\RVF_\alpha$.
By~(i) pick a smooth extension $\widebar \lpz:\D\to\Sigma$ of the loop $\lpz$
and choose a unitary trivialization of the symplectic vector bundle
$(\widebar \lpz^*\xi,\widebar \lpz^*d\alpha)$.
Then the linearization of the flow generated by
$\LM\RVF_\alpha$,\index{Reeb flow!linearized rescaled --}
informally called the \textbf{linearized rescaled Reeb flow},
provides along the trajectory $\lpz:[0,1]\to\Sigma$ by~(ii) a path in the
symplectic linear group $\Sp(2n-2)$ with initial point $\1$
and Robbin-Salamon index denoted by
\begin{equation*}\label{eq:RStrans}
     \RStrans(\lpz,\LM) 
     \in\tfrac12\Z.
\end{equation*}
It is\index{index!transverse Robbin-Salamon --}
called\index{$\RStrans$ transverse Robbin-Salamon index}
the \textbf{\Index{transverse Robbin-Salamon index}}
of the critical point $(\lpz,\LM)$ or, alternatively,
the $1$-periodic solution $\lpz$ of the rescaled Reeb vector field
$\LM\RVF_\alpha$.

\begin{exercise}
Show that the definition of $\RStrans$ is independent of the choice of,
firstly, the extending disk by~(iii) and, secondly, of the unitary trivialization.
\end{exercise}

\begin{exercise}
Calculate $\RStrans(\lpz,0)$
for constant critical points $(\lpz,0)$ of $\Aa^F$.
\end{exercise}

\subsubsection{Transverse Conley-Zehnder index}
Suppose a critical point $(\lpz,\LM)$ of $\Aa^F$ is \emph{transverse
non-degenerate};\footnote{
  Call $(\lpz,\LM)$ \textbf{transverse non-degenerate}
  iff $\LM\not=0$ and the corresponding Reeb
  loop\index{critical point!transverse non-degenerate --}
  is.\index{transverse non-degenerate!critical point}
  }
cf. Definition~\ref{def:RF-non-deg-Sigma}
and Exercise~\ref{exc:reg_non-deg}.
So the image $\lpz(\SS^1)$ is an
isolated closed characteristic $P\cong\SS^1$ of $\Sigma$.
Moreover, the corresponding path in $\Sp(2n-2)$ ends away
from the Maslov cycle, i.e. the path is an element of $\SP^*(2n-2)$,
so the Robbin-Salamon index is nothing but the Conley-Zehnder
index of this path. In order to implicitly signalling the assumption
of transverse non-degeneracy, as opposed to just general Morse-Bott,
we use the notation
\begin{equation*}\label{eq:RStrans-ND}
     \CZtrans(\lpz,\LM)
     \in\tfrac12+\Z
\end{equation*}
for transverse non-degenerate critical points and call
this index\index{$\CZtrans$ transverse Conley-Zehnder index}
the\index{index!transverse Conley-Zehnder --}
\textbf{\Index{transverse Conley-Zehnder index}}.

\begin{exercise}[Half-integers]
Show that $\CZtrans$ is half-integer valued.
\end{exercise}

\subsubsection{Signature index and Morse index}
Let $f:N\to\R$ be a Morse-Bott function
on a finite dimensional manifold~$N$.
Given a  critical point $x$ of $f$, recall that
the \textbf{\Index{Morse index}} $\IND_f(x)$ is the
number\index{index!Morse --}
of negative eigenvalues, with multiplicities,
of the Hessian of $f$ at $x$. The\index{index!signature --}
\textbf{\Index{signature index}} of $x\in\Crit f$ is defined by
$$
     \INDsign_f(x):=-\frac12\sign \Hess_x f.
$$

\begin{exercise}[Morse-Bott]\label{exc:sign-index-MB}
Let $f:N\to\R$ be Morse-Bott. Show that\footnote{
  Set $\IND_f(C_i):=\IND_f(x)$ and $\INDsign_f(C_i):=\INDsign_f(x)$
  for some, hence any, $x\in C_i$.
  }
$$
     \INDsign_f(C_i)=\IND_f(C_i)-\tfrac12\left(\dim N-\dim C_i\right)
$$
for every connected component $C_i$ of the critical manifold $C=\Crit f$.
\end{exercise}

\begin{exercise}[Morse]\label{exc:sign-index-M}
If $x\in\Crit f$ is a non-degenerate critical point, then
$$
     \INDsign_f(x)=\IND_f(x)-\tfrac12\dim N\in
     \begin{cases}
        \Z&\text{, $\dim N$ even,}\\
        \frac12+\Z&\text{, $\dim N$ odd.}
     \end{cases}
$$
\end{exercise}

\subsubsection{Morse-Bott grading of the Rabinowitz-Floer complex}

\begin{definition}[Grading]
For $c\in\Crit\, h\subset C=\Crit\,\Aa^F$, define
\begin{equation}\label{eq:RF-grading}
     \mu(c):=\RStrans(c)+\INDsign_h(c)\in\tfrac12\Z.
\end{equation}
\end{definition}

\begin{exercise}[$\Sigma$ non-degenerate $\Rightarrow$ half-integers]
For non-degenerate $\Sigma$, show that
$
     \mu(c) \in\tfrac12+\Z
$
is half-integer valued for every $c\in\Crit\, h$.
Is this also true if $c=(\lpz,0)$ is a constant critical point?
\end{exercise}

Given $c_-,c_+\in\Crit\, h$, consider a connecting cascade flow line
$\Gamma$, an element of the moduli space
$\Mm_{c_-c_+}(\Aa^{F},h,J,g)$ defined by~(\ref{eq:cascade-mod-space}).
For generic $J$ and $g$ this space is a smooth manifold
whose \textbf{\Index{local dimension}}\footnote{
  dimension of the component that contains $\Gamma$
  }
at $\Gamma$ is given by\index{dimension!local --}
\begin{equation*}
\begin{split}
     \dim_\Gamma \Mm_{c_-c_+} 
   &=\RStrans(c_+)+\INDsign_h(c_+)
     -\left(\RStrans(c_-)+\INDsign_h(c_-)\right)-1\\
   &=\mu(c_+)-\mu(c_-)-1.
\end{split}
\end{equation*}
The first step is non-obvious
even for the Morse-Bott cascade complex
in finite dimensions; see~\citeintro[(65)]{Cieliebak:2009a}.
One crucial ingredient to obtain the first step,
cf.~\citeintro[(64)]{Cieliebak:2009a}, is the following even less trivial
formula. Let $\Mm$ be the moduli space of finite energy
trajectories $\upsilon$ of $\grad\Aa^F$.
By Morse-Bott, firstly, the asymptotic limits $\upsilon^\mp$
in~(\ref{eq:RF-asymp-lim}) exist, let $C_\mp$ be their
components of $C$, and secondly the linearization $D\Ff_F(\upsilon)$
of the gradient equation~(\ref{eq:RF-UGF}) at a flow trajectory $\upsilon$,
cf.~(\ref{eq:linearization}), is Fredholm between suitable spaces.
By~\citeintro[Prop.\,4.1]{Cieliebak:2009a} 
\begin{equation*}
\begin{split}
     \dim_\upsilon\Mm
   &=\INDEX D\Ff_F(\upsilon)+\dim C_++\dim C_-\\
   &=\RStrans(\upsilon_+)-\RStrans(\upsilon_-)+\tfrac{\dim C_++\dim C_-}{2}.
\end{split}
\end{equation*}
For further details concerning index computations
see~\citeintro{Cieliebak:2009a} and~\citerefRF{Merry:2011b}.

\subsection{Relation to homology of $\Sigma$}

\begin{exercise}\label{exc:RFH-Sigma}
If $\Sigma$ 
carries no Reeb loops that are contractible in $V$, then
$$
     \RFH_*(\Sigma)\simeq \Ho_*(\Sigma;\Z_2)
$$
is a grading preserving isomorphism
where $\RFH_*(\Sigma):=\HF_*(\Aa^F)$ for any defining Hamiltonian
$F\in\Ff(\Sigma)$; see~(\ref{eq:def-RFH}).
Why should there be a grading without assuming triviality
of $\pi_1(\Sigma)$ and $\Io_{c_1}$?
\end{exercise}

\subsection{Example: Unit cotangent bundle of spheres}\label{sec:RF-ex}
Consider the unit cotangent bundle $\Sigma:=S^*\SS^n$
of the unit sphere $Q:=\SS^n\subset\R^{n+1}$ in euclidean
space equipped with the induced Riemannian metric.
Observe that $\Sigma$ is a hypersurface of restricted contact
type of the convex exact symplectic manifold
$(V,\lambda)=(T^*\SS^n,\lambdacan)$; see Example~\ref{ex:S*Q-contact}.
In particular, the restriction $\alpha=\lambdacan|\Sigma$
is a contact form on the hypersurface $\Sigma =S^*\SS^n$ that bounds
the (compact) unit disk cotangent bundle $M=D^*\SS^n$.

\begin{theorem}[Unit cotangent bundle of unit sphere,~\citeintro{Cieliebak:2009a}]
For $n\ge 4$
$$
     \RFH_k(S^*\SS^n) 
     =
     \begin{cases}
     \Z_2&\text{, $k\in\left\{-n+\frac12,-\frac12,\frac12,n-\frac12\right\}
               +\Z\cdot(2n-2)$,}\\
     0&\text{, else.}
     \end{cases}
$$
\end{theorem}
\begin{figure}[h]
  \centering
  \includegraphics[width=\textwidth]
                             {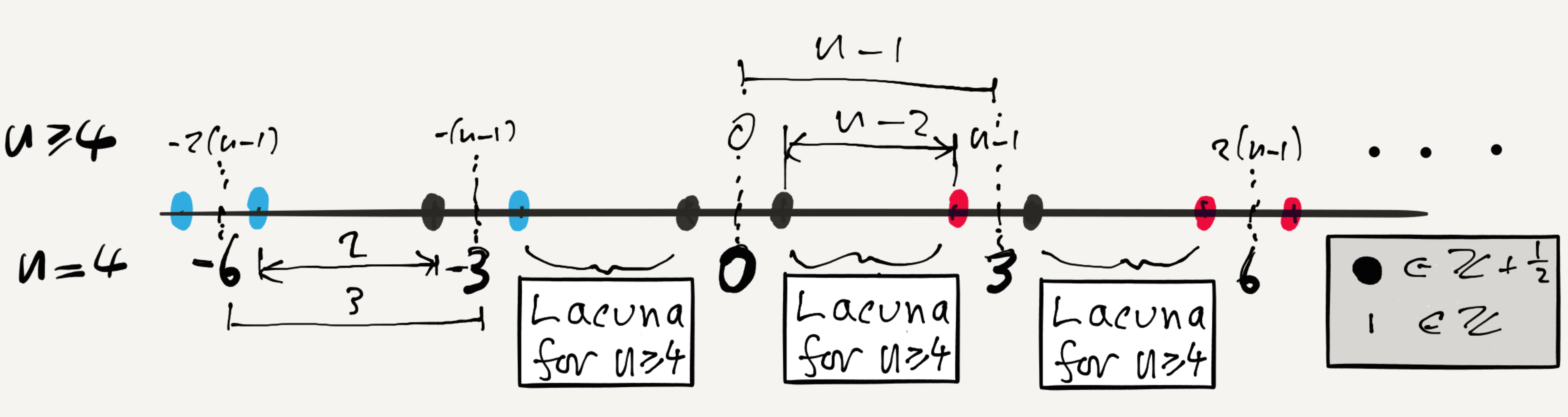}
  \caption{Indices ${\color{cyan}\bullet\color{black}\bullet\color{red}\bullet}
                 \in\Z+\frac12$ for which
                 $\RFH_{\color{cyan}\bullet\color{black}\bullet\color{red}\bullet}(S^*S^{n})=\Z_2$,
                 $n\ge4$}
  \label{fig:fig-indices}
\end{figure}

\textit{Idea of proof (Lacunary principle -- Floer homology equals
chain complex).}
One exploits the facts that for the round metric on $\SS^n$,
suitably normalized, all geodesics are periodic with prime period~$1$,
that the functional $\Aa^F$ is Morse-Bott,\footnote{
  Do not confuse with the stronger notion ``transverse non-degenerate''
  in Theorem~\ref{thm:RF-2nd-cat}.
  }
and that the critical manifold $C=\Crit\,\Aa^F$ is given by
$\Sigma=S^*\SS^n$ (the constant loops -- we don't call them geodesics),
together with $\Z^*$ copies of $S^*\SS^n$ where $\Z^*$ labels the
periods of the periodic geodesics.
Now one picks a Morse function $h_0$ on $S^*\SS^n$ with precisely four
critical points of Morse indices $0,n-1,n,2n-1$ and defines $h$ to be
the Morse function on $C$ which coincides with
$h_0$ on each component. Then
$$
     \Crit\, h\cong\Crit\, h_0\times\Z.
$$
Now one pits action knowledge~(\ref{eq:action-growth-by-period})
in terms of the prime period and positivity of the action difference
of two critical points $c_-$ and $c_+$ sitting on the two ends
of a non-constant connecting trajectory
against the facts that the boundary operator decreases the grading
$\RStrans$ precisely by $1$ and that $\RStrans$ can be related
via the Morse index theorem for geodesics to the particular Morse
indices $0,n-1,n,2n-1$ encountered among the critical points of $h$
as mentioned above. In the end, for $n\ge 4$, one gets to the
conclusion that there cannot be a connecting
trajectory between critical points of $\RStrans$-index difference one.
For details see~\citeintro[p.\,293]{Cieliebak:2009a}.

\section{Perturbed Rabinowitz action $\Aa^F_H$}
To prove the Vanishing Theorem~\ref{thm:RF-vanishing} motivates
to allow for more general Hamiltonians in the action functional;
cf.~\citeintro[\S 3.3]{Cieliebak:2009a}. 
Pick a defining, thus autonomous, Hamiltonian $F\in\Ff(\Sigma)$ and a
\textbf{young}\index{cutoff function!young --}\index{young cutoff function}
cutoff function\footnote{
  The \textbf{\Index{support} of \boldmath$\chi$}, notation
  $\supp\chi$, is the closure of its non-vanishing locus $\{\chi\not=0\}$.
  }
\begin{equation}\label{eq:chi-young}
     \chi:\SS^1\to [0,\infty)
     ,\qquad
     \supp\chi\subset(0,\tfrac12),\qquad
     \int_0^1\chi(t)\, dt=1.
\end{equation}
The assumption $\chi\,{\color{magenta} \ge 0}$ is used
in~(\ref{eq:jhjkhkj7876}). Let us call the Hamiltonian
\begin{equation}\label{eq:F-young}
     F^\chi:=\chi F:\SS^1\times V\to\R
\end{equation}
a \textbf{\Index{young Hamiltonian}},\index{Hamiltonian!young --}
as its flow is active only in the first half $[0,\tfrac12]$ of life.
Suppose $H\in C^\infty(\SS^1\times V)$ is a possibly non-autonomous
$1$-periodic Hamiltonian.
The \textbf{\Index{perturbed Rabinowitz action} functional}
on $\Ll V\times\R$ is\index{action functional!perturbed --}
defined\index{$\Aa^{F^\chi}_H$ perturbed Rabinowitz action}
by\index{Rabinowitz action functional!perturbed --}
\begin{equation}\label{eq:pert-Rab-action}
     \Aa^{F^\chi}_H(\lpz,\LM)
     :=\Aa^{F^\chi}(\lpz,\LM)-\int_0^1 H_t(\lpz(t))\, dt.
\end{equation}

In the notation of~(\ref{eq:RF-UGF}) and taking the same choices,
such as a cylindrical almost complex structure $J$ with induced
$L^2$ metric $g_J$ on $\Ll V\times \R$, the upward gradient
trajectories of $\grad \Aa^{F^\chi}_H$ are the solutions
$\upsilon=(u,\LMpath)$ of the PDE
\begin{equation}\label{eq:RF-perturbed-UGF}
\begin{split}
   &\p_s\upsilon-\grad \Aa^{F^\chi}_H(\upsilon)
     \\
   &=
     \begin{pmatrix}\p_su\\\p_s\LMpath\end{pmatrix}
     +
     \begin{pmatrix}
     J_t(u)\left(\p_t u-\LMpath\chi X_F(u)-X_H(u)\right)
     \\
     \int_0^1 \chi(t) F(u_s(t))\, dt
     \end{pmatrix}
     \\
   &=0.
\end{split}
\end{equation}
Compactness in Section~\ref{sec:RF-compactness},
thus the definition of Floer homology
in Section~\ref{sec:RF-CC}, goes through for
$\Aa^{F^\chi}$ with only minor modifications;
see~\citeintro[\S 3.1]{Cieliebak:2009a}.

\subsection{Proof of Vanishing Theorem}
The Vanishing Theorem~\ref{thm:RF-vanishing}
asserts triviality $\RFH(\Sigma)=0$ of Rabinowitz-Floer homology
whenever $\Sigma$ is displaceable.

\subsubsection{Proof of Vanishing Theorem~\ref{thm:RF-vanishing} -- v1}
The first version of the proof is short and rather illustrative,
hence well suited to communicate the main ideas. But  -- 
it has the disadvantage that it involves a class of Hamiltonians
for which we haven't introduced Floer homology
$\HF(\Aa^{F^\chi}_{H^\rho})$; however,
this is done in~\citerefRF{Albers:2010b}.
In other words, in proof v1 the required technical work is moved elsewhere,
so we are just left with the nice bits.

The idea is to find Hamiltonians such that there are no critical points
\begin{equation}\label{eq:Crit=empty}
     \Crit\,\Aa^{F^\chi}_{H^\rho}=\emptyset.
\end{equation}
In this case there are no generators of the
corresponding Floer complex. Hence
\begin{equation}\label{eq:VThm-v1}
     0=\HF(\Aa^{F^\chi}_{H^\rho})\simeq\HF(\Aa^{F^\chi})\simeq\HF(\Aa^F)=:\RFH(\Sigma).
\end{equation}
The isomorphisms are by continuation.
Section~\ref{sec:RF-CC} establishes $\HF(\Aa^F)$ only, but
the same construction goes through for $F^\chi$ replacing
$F$, including the construction of the second continuation isomorphism.
The first continuation isomorphism and $\HF(\Aa^{F^\chi}_{H^\rho})$
itself are constructed in~\citerefRF[\S 2.3]{Albers:2010b};
see Section~\ref{sec:LIPS}.

In preparation of the proof of~(\ref{eq:Crit=empty})
we take the following choices, given a compactly supported
Hamiltonian $H:{\color{red} [0,1]}\times V\to\R$ that displaces $\Sigma$.
By compactness of $\Sigma$ the map $\psi_1^H$ not
only displaces $\Sigma$ from itself, but also a small
neighborhood $U$ of $\Sigma$. Observe that $\Sigma$
is contained in $\supp X_F$ for any of its defining Hamiltonians $F$,
since 0 is a regular value of $F$. Change or modify $F$, if necessary,
such that $\supp X_F\subset U$. Hence this support also gets displaced:
\begin{equation}\label{eq:supp-X_F-displace}
     \supp X_F\cap \psi_1^H(\supp X_F)=\emptyset.
\end{equation}
Now pick an \textbf{elderly} cutoff
function\index{cutoff function!elderly --}\index{elderly cutoff function}
\begin{equation}\label{eq:rho-elderly}
     \rho:[0,1]\to[0,1],\qquad
     \rho|_{[0,1/2]}\equiv 0,\qquad
     \text{$\rho\equiv 1$ near $t=1$},
\end{equation}
as illustrated by Figure~\ref{fig:fig-young-elderly} and consider the
\textbf{\Index{elderly Hamiltonian}},\index{Hamiltonian!elderly --}
defined by
\begin{equation}\label{eq:H-elderly}
     H^\rho=H_t^\rho:=\dot\rho(t) H_{\rho(t)} \in C^\infty(V),\quad t\in{\color{red} \SS^1}.
\end{equation}
\emph{Elderly} means that the flow of $H^\rho$ is active
only in the second half $[\tfrac12,1]$ of life.

\begin{figure}[h]
  \centering
  \includegraphics
                             [height=3cm]
                             {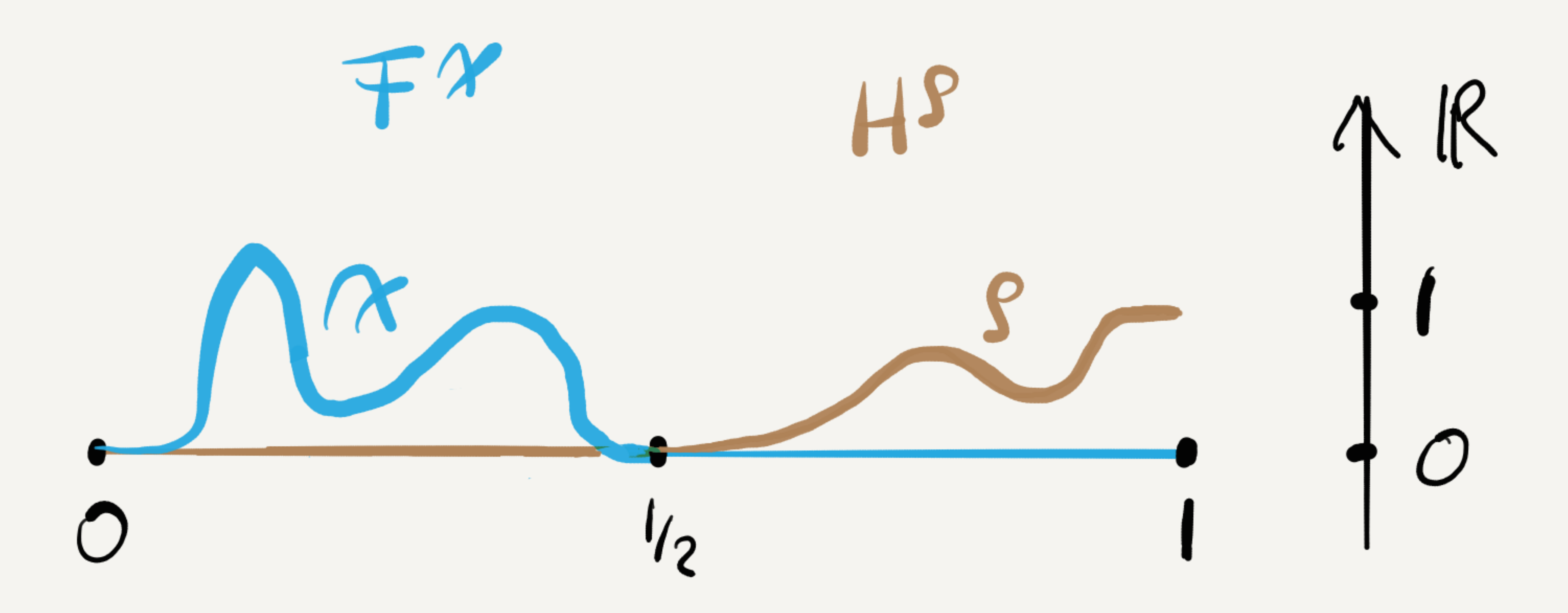}
  \caption{{\color{blue} Young} and {\color{brown} elderly}
                 cutoff functions and Hamiltonians}
  \label{fig:fig-young-elderly}
\end{figure}

\begin{exercise}\label{exc:flow-rho-H}
Show a) $\psi_t^{H^\rho}=\psi_{\rho(t)}^H$ for $t\in\R$
and b) $\psi_1^{\LM F^\chi}=\psi_{\LM}^F$ for $\LM\in\R$.
[Hint: a) Footnote.\footnote{
  $\frac{d}{dt}\psi^H_{\rho(t)}=\frac{d\,\psi^H_{\rho(t)}}{d\,\rho(t)}\dot\rho(t)
  =X_{H^\rho}\circ\psi^H_{\rho(t)}$
  }
b)~Note that $\LM F^\chi=F^{\LM\chi}=\dot \rho F$ with
$\rho(t):=\LM\int_0^t\chi(\sigma)\, d\sigma$.]
\end{exercise}

The critical points of $\Aa^{F^\chi}_{H^\rho}$
are the solutions $(\lpz,\LM)\in\Ll V\times\R$ of the equations
\begin{equation}\label{eq:Crit-emptyness}
\begin{cases}
     \dot \lpz(t)=\LM\chi(t) X_F(\lpz(t))+\dot\rho(t)
     X_{H_{\rho(t)}}(\lpz(t))&\text{, $t\in\SS^1$,}
     \\
     \int_0^1\chi(t) F(\lpz(t))\,dt=0\text{.}&
\end{cases}
\end{equation}

It remains to prove emptyness~(\ref{eq:Crit=empty}).
By contradiction suppose that $(\lpz,\LM)$ is a solution of~(\ref{eq:Crit-emptyness}).
There are two cases.

\vspace{.1cm}
\noindent
\textbf{I. \boldmath$\lpz(0)\notin \supp X_F$.}
On $[0,\tfrac12]$ the first equation in~(\ref{eq:Crit-emptyness})
describes a reparametrization of the Hamiltonian flow of $F$, so $F$
is preserved. Thus
$$
     F\circ \lpz|_{[0,\tfrac12]}\equiv F(\lpz(0))
     =:c\not= 0
$$
is non-zero since by assumption
$\lpz(0)\notin\supp X_F\supset\Sigma=F^{-1}(0)$. Thus
$$
     \int_0^1\chi F(\lpz)\, dt
     =\int_0^{\frac12}\chi F(\lpz)\, dt
     =c\int_1^{\frac12}\chi\, dt=c\not= 0
$$
as $\supp \chi$ lies in $[0,\tfrac12]$ where $F\circ \lpz\equiv c$.
This contradicts equation two in~(\ref{eq:Crit-emptyness}).

\vspace{.1cm}
\noindent
\textbf{II. \boldmath$\lpz(0)\in \supp X_F$.}
As $X_{F^\chi}$ is young and $X_{H^\rho}$ is elderly,
these two vector fields are supported in disjoint time intervals.
Thus the flow of their sum is the composition of their individual flows.
Together with Exercise~\ref{exc:flow-rho-H} we get that
\begin{equation*}
\begin{split}
     \overbrace{\supp X_F\ni \lpz(0)}^{\rm hypothesis}
   &=\lpz(1)\\
   &=\psi_1^{H^\rho}\circ\psi_1^{\LM F^\chi}  \lpz(0)\\
   &=\psi_1^H\circ\hspace{-.45cm}
     \underbrace{\psi_\LM^F \lpz(0)}_{\text{hyp.}\Rightarrow\in\supp X_F}
     \hspace{-.45cm} .
\end{split}
\end{equation*}
But this is impossible since $\psi_1^H$ displaces
$\supp X_F$ by~(\ref{eq:supp-X_F-displace}). Contradiction.

\subsubsection{Proof of Vanishing Theorem~\ref{thm:RF-vanishing} -- v2}
This version of proof gets away with Floer
homology as introduced in Section~\ref{sec:RF-CC}
and is based on continuation and the following
stronger version of the absence~(\ref{eq:Crit=empty})
of critical points of the perturbed action $\Aa^{F^\chi}_{H^\rho}$
associated to the young Hamiltonians $F^\chi$ in~(\ref{eq:F-young})
and the elderly ones $H^\rho$ in~(\ref{eq:H-elderly}).

\begin{lemma}[No critical points]\label{le:no-crit-pts-Aa-FH}
There is a constant $\gamma=\gamma(J)>0$ such that
$$
     \bigl\|\grad \Aa^{F^\chi}_{H^\rho}(\lpz,\LM)\bigr\|\ge\gamma
$$
for every $(\lpz,\LM)\in\Ll V\times\R$.\footnote{
  Both $\grad$ and $\norm{\cdot}$ depend on $J$;
  see Exercise~\ref{eq:nabla-Aa^F}.
  }
\end{lemma}

Fix a smooth monotone cutoff function $\beta:\R\to[0,1]$ with
$\beta(s)=1$ for~$s\ge 1$ and $\beta(s)=0$ for $s\le -1$
in order to define a homotopy $H_{\mbf{\cdot}}$, and its
reverse~$\widebar{H}_{\mbf{\cdot}}$, between the zero Hamiltonian $0$
and the Hamiltonian $H^\rho=H^\rho_t$, namely
$$
     H_s:=\beta(s)\, H^\rho,\qquad
     \widebarsub{H}{s} :=(1-\beta(s))\, H^\rho,\qquad
     s\in\R.
$$
For each real parameter value $R\ge 1$ consider the homotopy
$\widebarsub{H}{\mbf{\cdot}}\#_R H_{\mbf{\cdot}}$ from
the zero Hamiltonian $0$ back to itself which is defined by the concatenation
$$
     \R\ni s\mapsto\widebarsub{H}{s}\#_R H_s
     :=\begin{cases}
        H_{s+R}&\text{, $s\le 0$,}\\
        \widebarsub{H}{s-R}&\text{, $s\ge 0$,}
     \end{cases}
$$
of the homotopy $H_{\mbf{\cdot}}$ followed by its reverse
$\widebarsub{H}{\mbf{\cdot}}$, as
illustrated by Figure~\ref{fig:fig-RF-concatenation-H}.

\begin{figure}
  \centering
  \includegraphics
                             [height=3cm]
                             {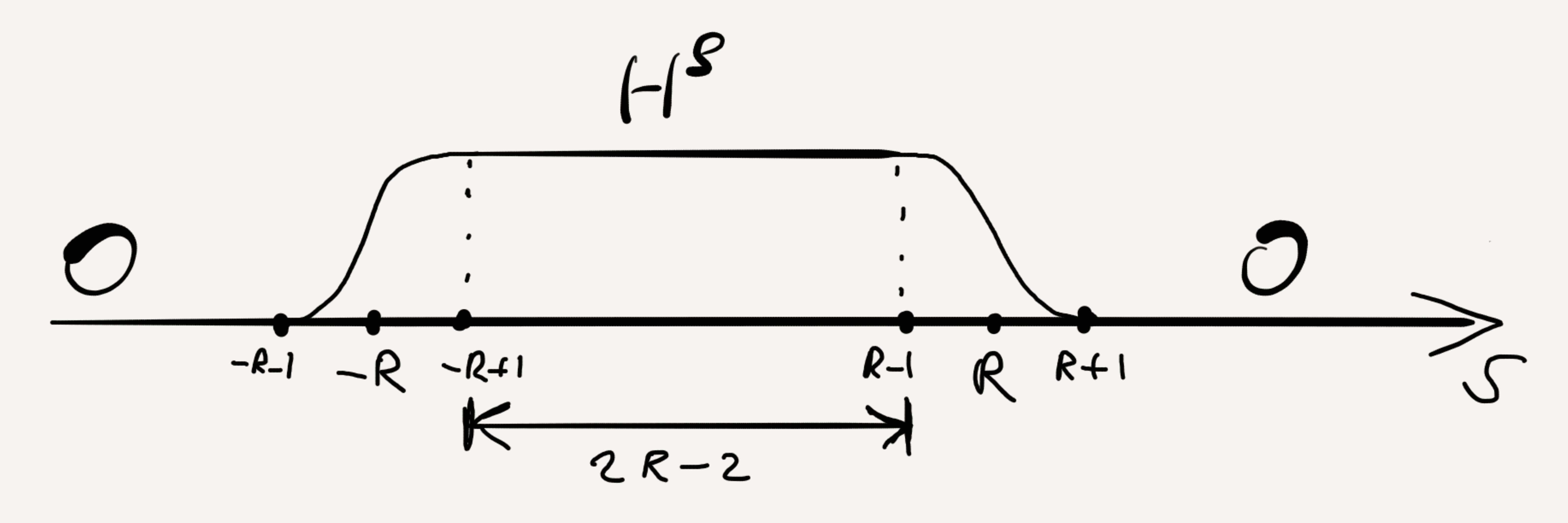}
  \caption{Concatenation homotopy $s\mapsto{\bar{H}}_{{\mspace{-2mu} s}}\#_R H_s$
                from $0$ over $H^\rho$ back to $0$}
  \label{fig:fig-RF-concatenation-H}
\end{figure}

Now consider the homotopy, in $r\in[0,1]$, of homotopies
of Hamiltonians $s\mapsto F^\chi + r\widebarsub{H}{s}\#_R\, rH_s$
and their corresponding action functionals
$$
     \Aa_{r,s}:=\Aa^{F^\chi}_{r\widebarsub{H}{s}\#_R\, rH_s}.
$$
\begin{figure}[h]
  \centering
  \includegraphics
                             [height=4cm]
                             {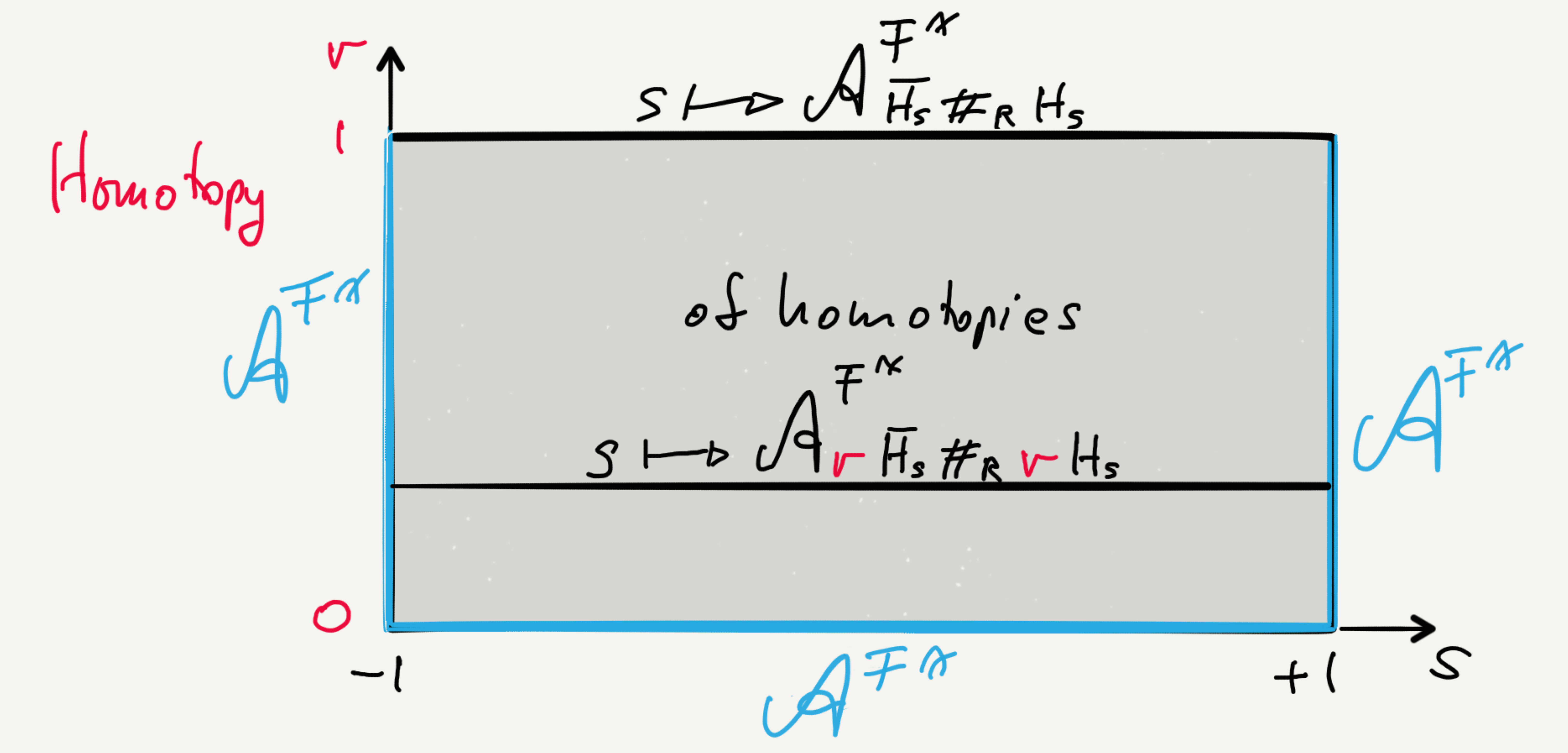}
  \caption{Homotopy, in $r$, of homotopies
                 $s\mapsto \Aa_{s,r}$, each from $\Aa^{F^\chi}$ to $\Aa^{F^\chi}$}
  \label{fig:fig-RF-htpy-of-htps-Aa}
\end{figure}

The homotopies at $r=0$ and $r=1$ have special properties.
The one at $r=0$ is constant, given by $s\mapsto \Aa^{F^\chi}$,
and the one at $r=1$ has -- as a consequence of the no-critical-points
Lemma~\ref{le:no-crit-pts-Aa-FH} -- no connecting~flow~lines:

\begin{lemma}[No finite energy trajectories]\label{le:RF-no-flow-lines}
There is a constant $R_0$ depending only on $F^\chi$, $\Norm{H^\rho}_\infty$,
and the action values $\Aa^{F^\chi}(\upsilon^\pm)$ of two fixed
critical points such that the following is true. 
For each real $R\ge R_0$ there are no trajectories of the $s$-dependent gradient
$\grad \Aa_{1,s}$ converging asymptotically to $\upsilon^\pm$.
\end{lemma}

\begin{lemma}[Uniform $C^0$ bound]\label{le:RF-compactness-htps}
Pick $R\ge 1$ and assume that $\upsilon(s)=\left(u(s),\LMpath(s)\right)$ is a
trajectory~of $\grad\Aa_{r,s}$, for some
$r\in[0,1]$, converging asymptotically to
$\upsilon^\pm=(u^\pm,\LMpath^\pm)\in\Crit\,\Aa^{F^\chi}$.
Then $\LMpath(s)$ is uniformly bounded by a constant that depends only on
$F^\chi$, $\Norm{H^\rho}_\infty$, $R$, and the action values of $\upsilon^\pm$.
\end{lemma}

Lemma~\ref{le:RF-compactness-htps} implies compactness of the
relevant components of the moduli spaces appearing in the
definition~(\ref{eq:Floer-continuation-maps-homology})
of the continuation homomorphisms
$$
     \Psi^r=[\psi(r\widebarsub{H}{\mbf{\cdot}}\#_R\, rH_{\mbf{\cdot}})]
     :\HF(\Aa^{F^\chi})\to\HF(\Aa^{F^\chi}),\quad r\in[0,1].
$$
But the $\Psi^r$ are all equal, as the defining homotopies of any two are
homotopic, cf. Exercise~\ref{exc:htps-of-htps},
and $\Psi^0=\1$, because it is induced by the constant homotopy,
cf. Exercise~\ref{exc:const-htpy}. But $\Psi^1=0$ since by
Lemma~\ref{le:RF-no-flow-lines} there are just no connecting
trajectories of $\grad \Aa^{F^\chi}_{\widebarsub{H}{s}\#_R H_s}$
whenever $R\ge R_0$. Thus we conclude
\begin{equation}\label{eq:RF-thpy-arg}
     \1=\Psi^0=\Psi^1=0:\HF(\Aa^{F^\chi})\to\HF(\Aa^{F^\chi}).
\end{equation}
But this is only possible if the domain $\HF(\Aa^{F^\chi})=0$ is
trivial which is what is claimed by the Vanishing Theorem~\ref{thm:RF-vanishing}.
Full details are given in~\citeintro[\S 3.3]{Cieliebak:2009a}.

\vspace{.2cm}
Let us look at the ideas, at least, how to prove the three lemmas.
For details see~\citeintro[p. 281-284]{Cieliebak:2009a}.

\vspace{.1cm}
\noindent
\textsc{No critical points.}
The proof of Lemma~\ref{le:no-crit-pts-Aa-FH}
takes three steps and uses compactness of $\Sigma$,
that $H$ displaces $\Sigma$ and that $F^\chi$ and $H^\rho$
are young and elderly, respectively. Furthermore, it is crucial
that $\grad\Aa$ has two components, see~(\ref{eq:RF-grad-Aa}),
whose norms can be played out, one against the other one.

Step~1: There is a constant $\eps_1(J)$
such that if $(\lpz,\LM)\in\Ll V\times\R$ satisfies
\begin{equation}\label{eq:hyp-eps-1}
     \Norm{\p_t \lpz-\LM\chi X_F(\lpz)-\dot\rho X_{H_{\rho}}(\lpz)}_2
     \le\eps_1
\end{equation}
then 
\begin{equation*}
     \left(\lpz(0),\lpz(\tfrac12)\right)\notin\supp X_F\times\supp X_F.
\end{equation*}
In words, for such loops $\lpz$ not both, birth and \Index{midlife crisis}, can
happen under $F$-action, at least one of them requires a rest from
the defining Hamiltonian $F$.
The proof of this uses that $\Sigma$ is compact and is displaced by
$H^\rho$ only at the end of life $t=1$ where the young $F^\chi$ is not supported.

Step~2: There are $\eps_2,\delta>0$ such that if
$(\lpz,\LM)$ satisfies~(\ref{eq:hyp-eps-1}) for $\eps_2$, then
$$
     \Abs{F\circ \lpz}\ge\frac{\delta}{2},\qquad
     \text{on $[0,\tfrac12]$.}
$$
That is $F$ stays away from zero during the first part of life of any
such loop $\lpz$.
The proof exploits that $H^\rho$ is elderly
and uses Step~1.

Step~3: Pick  $(\lpz,\LM)\in\Ll V\times\R$.
To prove that $\bigl\|\grad\Aa^{F^\chi}_{H^\rho}(\lpz,\LM)\bigr\|\ge\gamma
=:\min\{\eps_2,\delta/2\}$ involves two cases in each of which one
simply forgets one of the two gradient components,
cf.~(\ref{eq:RF-norm-grad-Aa}).
\newline
If $I:=\norm{\p_t \lpz-\LM\chi X_F(\lpz)-\dot\rho X_{H_{\rho}}(\lpz)}_2\ge\gamma$,
we are done: Just forget component two.
Otherwise, if $I<\gamma$, Step~2 applies. Forget component one to
obtain that
\begin{equation}\label{eq:jhjkhkj7876}
\begin{split}
     \bigl\|\grad\Aa^{F^\chi}_{H^\rho}(\lpz,\LM)\bigr\|
   &\ge\Abs{\mean(F^\chi\circ \lpz)}\\
   &=\biggl|\int_0^{\tfrac12}\underbrace{\chi}_{\color{magenta} \ge 0}\cdot\, 
       \underbrace{F(\lpz)}_{\not=0}\, dt\biggr|\\
   &=\int_0^{\tfrac12}\chi
      \cdot\underbrace{\Abs{F(\lpz)}}_{\ge\delta/2}\, dt
     \ge\frac{\delta}{2}\ge\gamma.
\end{split}
\end{equation}
Here we used that $\chi$ is young and integrates to one.
Concerning the second identity use the non-negativity
assumption $\chi\,{\color{magenta} \ge 0}$ in~(\ref{eq:chi-young})
together with the fact that by Step~2 the function
$F\circ \lpz\not= 0$ is non-zero along the interval $[0,\frac12]$, in fact,
it is either $\ge\frac{\delta}{2}$ or $\le-\frac{\delta}{2}$.

\begin{exercise}
It\index{{open \color{red} problems \color{black}}}
seems to be an \emph{open question}, at least we couldn't localize
a proof in the literature, if the conclusion
$\norm{\grad\Aa^{F^\chi}_{H^\rho}(\lpz,\LM)}\ge\frac{\delta}{2}$
in~(\ref{eq:jhjkhkj7876}) remains valid for real-valued
cut-off functions $\chi:\SS^1\to{\color{magenta}\R}$
in~(\ref{eq:chi-young}).
\end{exercise}

\vspace{.1cm}
\noindent
\textsc{No finite energy trajectories.}
This is essentially an integrated version of 
Lemma~\ref{le:no-crit-pts-Aa-FH}.
In the $s$-independent case the proof is trivial:
Any connecting trajectory $(u,\LMpath)$ has finite energy
$E(u,\LMpath):=\int_{-\infty}^\infty\norm{\grad \Aa(u,\LMpath)}^2\, ds$,
namely, it is given by the action difference of the asymptotic limits.
But a positive lower gradient bound as in Lemma~\ref{le:no-crit-pts-Aa-FH}
contradicts finiteness, so such trajectories cannot exist.
The $s$-dependent case is rather similar;
see~\citeintro[p.283, Step~1]{Cieliebak:2009a}.

\vspace{.1cm}
\noindent
\textsc{Uniform bound.}
The proof is similar to the one of the
uniform bound on $\abs{\LMpath}$ in Section~\ref{sec:RF-compactness},
again built on the quantity $\Tt_\sigma$.
That $H$ displaces $\Sigma$ is not used.

\subsection{Weinstein conjecture}\label{sec:Weinstein-RFH}

\begin{theorem}
A displaceable convex exact hypersurface $\Sigma$
in a convex exact symplectic manifold $(V,\lambda)$
carries a closed Reeb loop that is contractible in $V$.
\end{theorem}

\begin{proof}
Vanishing Theorem~\ref{thm:RF-vanishing}
and Exercise~\ref{exc:RFH-Sigma}.
\end{proof}

The theorem was established in~\citerefRF{Schlenk:2006a}
for stably displaceable hypersurfaces of contact type.
Recall that if a closed contact type hypersurface $\Sigma\subset
(\R^{2n},\omega_0)$ is simply-connected, then
it is of restricted contact type by
Exercise~\ref{exc:restricted-CT=>exact}.

\begin{exercise}
Consider a connected closed hypersurface $\Sigma\subset
(\R^{2n},\omega_0)$.\mbox{ }
\\
a)~Show that $\Sigma$ is bounding and displaceable.
\\
b)~Show that 
if $\Sigma$ is, in addition, transverse to some
global\footnote{
  A locally near $\Sigma$ defined $\LVF$ is sufficient
  whenever $\pi_1(\Sigma)=0$;
  cf. Exercise~\ref{exc:restricted-CT=>exact} b).
  }
  Liouville vector field $\LVF$, then it is convex exact.
\end{exercise}

\subsection{Leaf-wise intersections}\label{sec:LIPS}
Before giving the formal definition of ``leaf-wise intersection''
let us first switch on light by looking at the \textbf{circular restricted
three body problem} in celestial
mechanics;\index{three body problem}
cf.~\cite[Ch.\,10]{abraham:1978a}.

\subsubsection*{Motivation: Satellite perturbed by comet}
Following~\citerefRF{Albers:2012b}
consider an almost massless particle, the satellite~$s$,
moving in the gravitational field of two huge massive bodies
called primaries, say earth~$E$ and moon~$M$.
By assumption the whole system is restricted to a given
fixed plane and each of the primaries
moves along a circle about their common center of mass.
So the configuration space of the system is rather restricted. 
Now suppose the satellite, so far moving peacefully on its energy surface
$\Sigma=F^{-1}(0)$ in phase space, gets temporarily influenced by a
comet passing by during a time interval, say of length one.
Suppose the previously autonomously via $\phi=\phi^F$ on $\Sigma$ moving
satellite receives extra kinetic energy by gravitational attraction
as the comet appears in front at time zero and looses energy when
the comet disappears behind at time one. The comet's
presence is described by a time-dependent Hamiltonian $H$ with
Hamiltonian flow $\psi=\psi^H$. In other words,
at time zero, say at the phase space location $x\in\Sigma$,
the satellite gets lifted off of $\Sigma$ following $\psi$
for one unit of time after which it gets dropped back onto $\Sigma$
at $\psi_1 x$.
Let $L_x:=\phi_{\R} x\subset\Sigma$, called the
\textbf{leaf of $\mbf{x}$}, be the whole phase space trajectory of
the satellite as it would happen without the comet's appearance.
\\
One would probably not expect that the satellite will get dropped back
at $\psi_1 x$ to its original trajectory $L_x$. However if it happens,
the unexpected penomenon
$$
     \psi_1 x\in L_x,\qquad L_x:=\phi_{\R} x\subset\Sigma,
$$
is called a \textbf{\Index{leaf-wise intersection}}, see
Figure~\ref{fig:fig-RF-LIP}, and $x$ is called
a\index{\LIP~leaf-wise intersection point}
\textbf{\Index{leaf-wise intersection point (\LIP)}}.
\begin{figure}
  \centering
  \includegraphics
                             [height=3.5cm]
                             {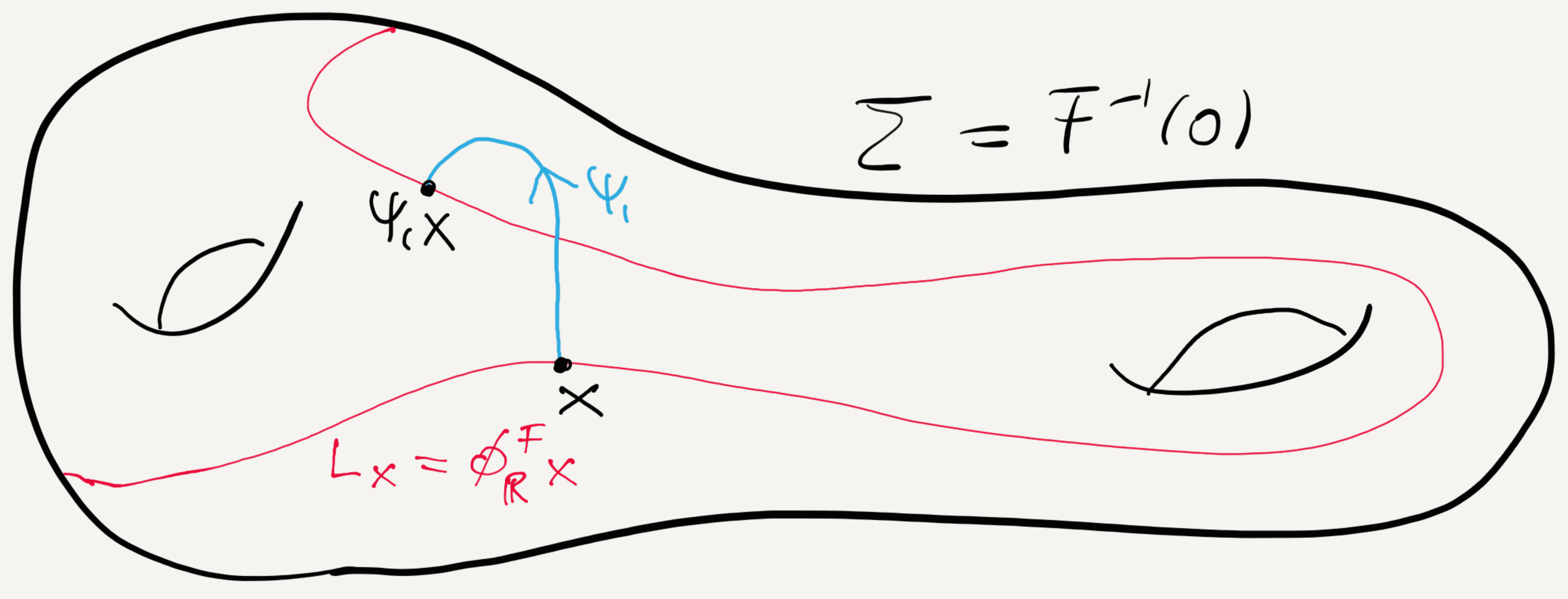}
  \caption{Leaf-wise intersection $\psi_1 x\in L_x$:}
                \centering
                Satellite at $x$ deviates under sudden comet influence
                $\psi_t=\psi_t^{H_t}$,
                \\ \centering
                but happens to end up on its unperturbed trajectory
                $\phi^F_{\R} x$ afterwards
  \label{fig:fig-RF-LIP}
\end{figure}

Surprisingly, the phenomenon indeed happens:
In certain hypersurfaces of exact symplectic manifolds, for instance such
of non-zero Rabinowitz-Floer homology, \emph{for any} comet there
is by~\citerefRF[Thm.\,C]{Albers:2010b} a satellite position $x$ which
ends up on its own unperturbed trajectory afterwards.
In certain cotangent bundle situations there are, for generic
comets, even infinitely many undestroyable satellite trajectories
by~\citerefRF{Albers:2010c};
see also~\citerefRF[Thm.\,1]{Albers:2012b}
and \S2.2 in the survey~\citerefRF{Albers:2012a}.

\subsubsection{Rabinowitz-Floer homology for perturbed action
$\mbf{\Aa^{F^\chi}_H}$}
Let $(V,\lambda)$ and $\Sigma=F^{-1}(0)$
satisfy Assumption~\ref{ass:RF} where $F\in\Ff(\Sigma)$ is defining.
Consider the Rabinowitz action functional $\Aa^{F^\chi}_H$
in~(\ref{eq:pert-Rab-action}) for possibly non-autonomous
\textbf{\Index{elderly Hamiltonians}}
$$
     H\in\Hh^\dagger:=\{H\in C^\infty(\SS^1\times V)\mid
     \text{$H_t=0$ for $t\in[0,\tfrac12]$}\}
$$
also called \textbf{elderly perturbations} of $\Aa^{F^\chi}$.
The\index{elderly!perturbations}\index{perturbation!elderly}
perturbed\index{$\Hh^\dagger$ elderly perturbations}
functional $\Aa^{F^\chi}_H$
has a number of useful properties:
  As time-dependence is allowed for $H$, the functional
  $\Aa^{F^\chi}_H$ is Morse for generic $H\in\Hh^\dagger$,
  as shown in~\citerefRF[Thm.\,2.13]{Albers:2010b}. Thus
  no Morse-Bott complex will be needed at all.
  The critical points of $\Aa^{F^\chi}_H$ correspond to leaf-wise
  intersections.\footnote{
    The map $\Crit\to\{{\LIP}s\}$, $(\lpz,\LM)\mapsto \lpz(\frac12)$,
    is injective, unless $L_x\cong \SS^1$ for some $\LIP$ $x$.
    }

\begin{proposition}[Critical points and {\LIP}s,~{\citerefRF{Albers:2010b}}]
If $(\lpz,\LM)\in\Crit\,\Aa^{F^\chi}_H$, then $x:=\lpz(\tfrac12)$ lies in
$\Sigma=F^{-1}(0)$ and $\psi^H_1 x$ lies on the Reeb leaf
$L_x=\phi^F_\R x$.
\end{proposition}

The definition of Floer homology $\HF(\Aa^{F^\chi}_H)$ proceeds
pretty much as in Chapter~\ref{sec:FH}
with the little extra twist of an upward finiteness
condition, completely analogous to~(\ref{eq:RF-FINITE}),
that takes care of infinitely many critical points.
More precisely, given a defining Hamiltonian $F\in\Ff(\Sigma)$, pick
a young cutoff function $\chi$ as in~(\ref{eq:chi-young})
and a generic elderly perturbation $H\in\Hh^\dagger$
such that $\Aa^{F^\chi}_H$ is Morse.
Let $\CF(\Aa^{F^\chi}_H)$ be the $\Z_2$ vector space
that consists of all formal sums
$$
     \xi=\sum_{c\in\Crit\,\Aa^{F^\chi}_H} \xi_c c
$$
such that, given $\xi$, for each $\kappa\in\R$ there is only a finite
number of non-zero $\Z_2$-coefficients $\xi_c$ that belong to
critical points $c$ of action $\ge\kappa$; cf.~(\ref{eq:RF-FINITE}).
Let $\Mm(c_+,c_-)$ be the space of connecting (upward) gradient
trajectories, that is solutions $\upsilon=(u,\LM)$ of the
PDE~(\ref{eq:RF-perturbed-UGF}), with asymptotic
boundary conditions $c_\pm\in\Crit\,\Aa^{F^\chi}_H$
sitting at $s=\pm\infty$; cf. Figure~\ref{fig:fig-RF-UGF}.
For generic $\SS^1$-families of cylindrical almost complex
structures $J_t$ the space $\Mm(c_+,c_-)$ is a smooth finite
dimensional manifold that carries a free $\R$-action by $s$-shift.
Let
$$
     n(c_-,c_+):=\#_2(m_{c_-c_+}),\qquad
     m_{c_-c_+}:=\Mm(c_+,c_-)/\R,
$$
be the number {\rm (mod 2)} of zero-dimensional components of
the moduli space of connecting flow lines;
cf.~(\ref{eq:Floer-boundary-operator}). Define the Floer boundary
operator on the chain groups $\CF(\Aa^{F^\chi}_H)$ analogous
to~(\ref{eq:RF-bdy}) by linear extension of
\begin{equation*}
     \p c_+:=\sum_{c_-\in\Crit\, \Aa^{F^\chi}_H} n(c_-,c_+) c_-
\end{equation*}
for $c_+\in\Crit\,\Aa^{F^\chi}_H$; cf. Figure~\ref{fig:fig-RF-UGF} and
Remark~\ref{rem:RF-non-standard-order}.

By definition \textbf{Floer homology of the perturbed Rabinowitz
action functional} $\Aa^{F^\chi}_H$ is the homology of this chain complex,
namely\index{$\HF(\Aa^{F^\chi}_H)$ perturbed Rabinowitz Floer homology}
$$
     \HF(\Aa^{F^\chi}_H):=\frac{\ker\p}{\im\p}.
$$

\begin{theorem}[Invariance of $\RFH$ under elderly
perturbations,~{\citerefRF{Albers:2010b}}]\mbox{ }\\
If $\Aa^{F^\chi}_H$ is Morse for an elderly perturbation $H\in\Hh^\dagger$, then
$$
     \HF(\Aa^{F^\chi}_H)\simeq \HF(\Aa^{F^\chi}_0)=\RFH(\Sigma).
$$
\end{theorem}

The theorem actually concludes the proof of version one of the
Vanishing Theorem~\ref{thm:RF-vanishing}; cf.~(\ref{eq:VThm-v1}).

\begin{remark}[Rabinowitz Floer for Reeb chords and Voyager missions]
Coming back to the previous space travel motivation, there is
practical interest in so-called consecutive collision orbits --
of course, in small perturbations of them.
In~\citerefRF{Frauenfelder:2017a}
consecutive collision orbits, interpreted as Reeb chords,
are encoded as critical points of an adequate version of
the\index{Rabinowitz Floer homology!for Reeb chords}
Rabinowitz action functional. Calculation of the corresponding
Rabinowitz\index{orbits!collision --}
Floer homology then leads to infinitely many \Index{collision orbits}.
\end{remark}

\subsubsection{The general picture: Coisotropic intersections}
The previous situation of a closed codimension one submanifold
$\Sigma=F^{-1}(0)$ being foliated by flow lines, that is $1$-dimensional
leaves $L_x$, is the rather special case $r=1$ of the general
leaf-wise intersection problem described in the
masterpiece~\citerefRF{Moser:1978a}.
It is amazing to see how important results fall off as
special cases for particular values of the codimension $r$ of a closed
coisotropic submanifold $\Sigma$ of a simply-connected exact
symplectic manifold $(V,\lambda)$ of dimension $2n$;
see the presentation in~\citerefRF{Moser:1978a}
of consequences 1.--4. of the main theorem that asserts existence of
leaf-wise intersections.

A codimension $r$ submanifold $\Sigma$ of a symplectic manifold
of dimension $2n$ is\index{submanifold!coisotropic}
called a \textbf{\Index{coisotropic submanifold}}
if every tangent space of $\Sigma$ is a coisotropic
subspace of the corresponding tangent space of $V$.
This implies $r\le n$.
The collection of symplectic complements
$(T_p\Sigma)^\omega$ turns out to provide an integrable
distribution of rank $r=\dim (T_p\Sigma)^\omega$ in the tangent bundle
$T\Sigma$. Thus by Frobenius the $2n-r$ dimensional manifold $\Sigma$
is foliated by leaves of dimension $r$.

For $r=0$ the main theorem in~\citerefRF{Moser:1978a}
proves existence of at least two fixed points
of a symplectic diffeomorphism of a closed simply-connected
symplectic manifold, see Remark~\ref{rem:2FIX};
the assumption of being simply-connected
was removed in~\citerefRF{Banyaga:1980a}.

For $r=1$ one gets to the previously described situation of integral
curves of the characteristic line bundle of an energy surface, hence
to Reeb dynamics if $\Sigma$ is of contact type.

For $r=n$ one recovers the Lagrangian intersection problem.

\section{Symplectic homology and loop spaces}\label{sec:RF-SH}

Running out of time and pages, let us just briefly mention that,
given a closed Riemannian manifold $Q$,
Rabinowitz-Floer homology of the unit sphere cotangent bundle
$\Sigma=S^*Q$, bounding the unit disk cotangent bundle $M=D^*Q$,
in the cotangent bundle $(V,\lambda)=(T^*Q,\lambdacan)$
encodes both homology and cohomology of the free loop space.\footnote{
  Recall that in the present text we work with $\Z_2$ coefficients;
  for field coefficients see the convention prior to Thm.\,1.10
  in~\citerefRF{Cieliebak:2010b}.
  }
It was shown in~\citerefRF[Thm.\,1.10]{Cieliebak:2010b} that
\begin{equation}\label{eq:RFH=LOOP}
     \RFH_{*^\prime}(S^*Q)\simeq
     \begin{cases}
        \Ho_{*^\prime}(\Ll Q)&\text{, $*^\prime>1$,}\\
        \Ho^{-*^\prime+1}(\Ll Q)&\text{, $*^\prime<0$,}
     \end{cases}
\end{equation}
and that for $*^\prime=0,1$ there are isomorphisms involving the Euler class
of the vector bundle $T^*Q\to Q$. Now the grading $*^\prime$
of $\RFH$ is different from the half-integer grading $*$ defined earlier
in~(\ref{eq:RF-grading}), namely
\begin{equation*}
     *^\prime:=
     \begin{cases}
        *+\frac12&\text{, on generators $c$ of positive action
          $\Aa^F(c)>0$,}\\
        *-\frac12&\text{, on generators $c$ of negative action
          $\Aa^F(c)<0$.}
     \end{cases}
\end{equation*}
The isomorphism is obtained by relating via a long exact sequence
Rabinowitz-Floer homology to symplectic homology and cohomology of
$D^*Q$, aka Floer co/homology of the cotangent bundle,
and then use the isomorphism~(\ref{eq:FH=LM}).
A proof of~(\ref{eq:RFH=LOOP}) by a direct construction
is given in~\citerefRF{Abbondandolo:2009a}
and a generalization to twisted cotangent bundles
in~\citerefRF{Merry:2011a}; see~\citerefRF{Bae:2011a} for an
alternative method.

\bibliographystylerefRF{alpha}
\cleardoublepage
\phantomsection
\addcontentsline{toc}{section}{References}

\begin{bibliographyrefRF}{}
\end{bibliographyrefRF}

\cleardoublepage
\phantomsection

\appendix
\part{Appendices}\label{sec:appendices}
\chapter{Function spaces}\label{sec:function-spaces}


\section{Some Banach spaces and manifolds}\label{sec:B-spaces-B-mfs}
We shall provide the Banach manifolds $\Uu$ and $\Vv$,
each admitting a countable atlas modeled on a separable Banach space,
that are used in Example~\ref{ex:generic-Morse}.

\subsection*{The separable Banach space $\mbf{\Vv=C^k(Q)}$
for closed $\mbf{Q}$}
\begin{proposition}[The separable Banach space $\Vv=C^k(Q)$]
\label{prop:complete-separable}
Fix\index{completeness!of $C^k(Q)$}
$k\in\N_0$.\index{$C^k(Q)$}
For\index{separability!of $C^k(Q)$}
a closed manifold $Q$ the vector
space $C^k(Q)$ of $k$ times continuously differentiable functions
$f:Q\to\R$ equipped with the sum $\norm{\cdot}_{C^k}$ of the sup norms
of a function and its derivatives up to order $k$ is
complete and separable. This remains valid for vector valued
functions, that is for $C^k(Q,\R^\ell)$.
\end{proposition}

\begin{proof}[Idea of proof]
  \textbf{Case \boldmath$k=0$.}
  Asking any topological space, say $Q$, just to be compact and Hausdorff
  already implies that $C^0(Q)$ is \emph{complete} under the sup norm;
  cf.~\cite[Thm.\,IV.8]{reed:1980a}.
  \emph{Separability:}\footnote{
    In case of a bounded open subset $\Omega\subset\R^n$ separability
    of $C^0(\widebar{\Omega})$ follows from density of the (countable) space
    of polynomials with rational coefficients;
    see e.g.~\cite[Cor.\,1.32]{Adams:2003a}.
    }
  A manifold is metrizable; either exploit second
  countable or use a Whitney embedding into some $\R^N$.
  So pick a dense sequence $(x_r)$ in $Q$
  and for $r,s\in\N$ define the function
  $f_{r,s}(x):=\frac{1}{s}-\dist(x,x_r)$ if $\dist(x,x_r)\le\frac{1}{s}$
  and $f_{r,s}(x):=0$ otherwise. Then the (countable) collection
  given by the $\{f_{r,s}\}$ together with a constant function generates a
  countable algebra $\Aa$ by taking linear combinations
  with rational coefficients of finite products of members of the collection.
  The algebra $\Aa$ separates the points of $Q$,
  thus the closure of $\Aa$ in $C^0(Q)$ is $C^0(Q)$ itself
  by the Stone-Weierstrass Theorem; cf.~\cite[Thm.\,IV.9]{reed:1980a}.
  \\
  \textbf{Case \boldmath$k\ge 1$.} Reduction to $k=0$:
  Consider the isometric linear injection
  $C^k(Q)\to C^0(Q,\R^\ell)$ that maps $f$ to the vector whose
  components are $f$ and its derivatives up to order $k$.
    A Cauchy sequence in the image converges by completeness of the
    target space. The first component of the limit vector consists of
    a continuous function $f$ and the
    other components are the partial derivatives of $f$;
    see e.g.~\cite[Le.\,1.1.14]{Zimmer:1990a}.
  So the injection image is closed.
  Every\index{subspace!of metric space}
  subspace\footnote{
    A \textbf{subspace} of a metric space is a subset equipped with the
    \textbf{induced metric}, namely, the restriction of the ambient metric.
    }
  $A$ of\index{metric space!subspace of --}
  a separable \emph{metric} space $X$ is separable.\footnote{
    Consider the set of open balls of rational radii centered
    at the elements of the countable dense subset of $X$.
    Consider the collection of intersections of these balls with $A$.
    In each such intersection select one element. The set $S$
    of selected elements is countable and dense in $A$.
    }
  Every closed subspace of a complete metric space is complete.
  Thus separability and completeness of the target
  are inherited by the closed image, hence via the inverse isometry by
  the domain. But $Q$ is a manifold.
  To make sense of derivatives
  one can either use covariant derivatives $\nabla^j$ or cover $Q$
  by a finite atlas of local coordinates, cf.~\cite[App.\,B]{Wehrheim:2004a},
  or one employs the manifold
  $J^k(Q,\R)$ of $k$-jets, cf.~\cite[Ch.\,2 \S 4]{hirsch:1976a},
  and the associated map $j^k:C^k(Q)\to C^0(Q,J^k(Q,\R))$
  which is continuous, injective, and of closed image by~\cite[Ch.\,2
  Thm.\,4.3]{hirsch:1976a}.
\end{proof}

\subsection*{The separable Banach space $W^{k,p}(Q)$
for closed $\mbf{Q}$}
An alternative choice for $\Vv$ would be $W^{k,p}_0(Q)$, say with $p=2$
in order to even have a Hilbert space, which we introduce in the following.

We recommend~\cite{Adams:2003a}
for a concise, yet detailed, presentation of spaces of functions
on open subsets $\Omega$ of $\R^n$ -- e.g. continuous functions, Lebesgue
integrable functions, Sobolev functions -- including
the analysis of properties such as completeness and separability.

\begin{definition}[The separable Banach space $W^{k,p}_0(\Omega)$]
\label{def:Wkp}
Let $\Omega$ be an open subset of $\R^n$.
Pick a real $p\in[1,\infty)$.
A \textbf{\Index{multi-index}} is an $n$-tuple
$\alpha=(\alpha_1,\dots,\alpha_n)$ of non-negative integers $\alpha_j$.
The corresponding partial derivative is denoted by
$\p^\alpha:=\p_1^{\alpha_1}\dots\p_n^{\alpha_2}$ where
$\p_j:=\p/\p x_j$. The sum $\abs{\alpha}$ of the
components $\alpha_j$ is called the order of $\alpha$.
The \textbf{\Index{Sobolev space}}
\[
     \mbf{W^{k,p}(\Omega)}
\]
consists of all Lebesgue $p$-integrable functions $f:\Omega\to\R$,
notation $f\in L^p(\Omega)$, that admit \underline{\textbf{w}}eak
derivatives $\p^\alpha f$ up to order $k$ and all of them are in $L^p(\Omega)$.
The norm under which $W^{k,p}(\Omega)$ is complete, see
e.g.~\cite[Thm.\,3.3]{Adams:2003a}, is a sum of the
$L^p$ norms of the weak derivatives, namely
\begin{equation}\label{eq:sobolev-norm}
     \Norm{f}_{k,p}:=\Biggl(\sum_{0\le\abs{\alpha}\le k}
     \Norm{\p^\alpha f}_p^p\Biggr)^{1/p}.
\end{equation}
The Banach space $W^{k,p}(\Omega)$
is\index{$W^{k,p}(\Omega)$ Sobolev space (definition via weak derivatives)}
separable;
see e.g.~\cite[Thm.\,3.6]{Adams:2003a}.
As we are heading towards \emph{closed} manifolds $Q$, we are interested
in\index{$W^{k,p}_0(\Omega):=\overline{C_0^\infty(\Omega)}^{k,p}$
  Sobolev space (comp. supp. closure)}
the subspace
\[
     \mbf{W^{k,p}_0(\Omega)}
     :=\overline{C_0^\infty(\Omega)}^{k,p} 
\]
which by definition, see e.g.~\cite[\S 3.2 (c)]{Adams:2003a},
is the closure of the space $C_0^\infty(\Omega)$ of compactly
supported smooth functions with respect to the $W^{k,p}$ norm.
It is a fact that for smooth functions weak and ordinary derivatives
coincide (if an ordinary derivative is continuous it coincides with
the weak derivative). Thus $C_0^\infty(\Omega)\subset W^{k,p}(\Omega)$,
so we can \emph{forget weak derivatives in our context}.
By definition $W^{k,p}_0(\Omega)$ is a closed, thus complete,
linear subspace of $W^{k,p}(\Omega)$. It inherits separability;
cf. proof of Proposition~\ref{prop:complete-separable}.
\end{definition}

\begin{proposition}[The separable Banach space $W^{k,p}_0(Q)$]
\label{prop:complete-separable-W}
Let\index{completeness!of $W^{k,p}(Q)$}
$k\in\N_0$ and\index{separability!of $W^{k,p} (Q)$}
$p\in[1,\infty)$. Then for any closed manifold
$Q$\index{$W^{k,p}(Q)$ Sobolev space of functions on closed manifold $Q$}
the\index{Sobolev space!of functions on a closed manifold}
Sobolev space $W^{k,p}_0(Q)$ is complete and separable.
This remains valid for vector valued functions, that is for
$W^{k,p}_0(Q,\R^\ell)$.
(As $Q$ itself is compact, the notation $W^{k,p}(Q)$ is common.)
\end{proposition}

\begin{proof}[Idea of proof]
Let us stick to the notation $W^{k,p}_0(Q)$, so nothing reminds weak
derivatives. To actually define this space pick a finite local
coordinate cover $\{(\varphi_i,U_i)\}_{i=1}^r$ of the closed
manifold $Q$ such that each $U_i$ has \emph{compact closure}.
Pick a \textbf{subordinated partition of unity}
$\rho=\{\rho_i\}_{i=1}^r$,
that\index{partition of unity!subordinated --}
is $\rho_i\in C^\infty_0(U_i)$ and $\sum_i\rho_i\equiv 1$.
Define open subsets $\Omega_i:=\varphi_i(U_i)$ of $\R^n$.
Define a vector space
\begin{equation}\label{eq:Sob-closed-mfs}
     W^{k,p}_0(Q):=\{f:Q\to\R\mid
     f^\rho_i :=(\rho_i f)\circ \varphi_i^{-1}\in W^{k,p}_0(\Omega_i)
     \;\forall i\} .
\end{equation}
Since $W^{k,p}_0(\Omega_i):=\widebar{C_0^\infty(\Omega_i)}^{k,p}$
there are indeed no weak derivatives involved in this definition.
A norm is given by summing up the $W^{k,p}$ norms of the local coordinate
representatives, let us choose the finite sum
\begin{equation}\label{eq:sobolev-norm-Q}
     \Norm{f}_{k,p}:=\left(\sum_{i=1}^r
     \Norm{f^\rho_i}_{k,p}^p\right)^{1/p},\qquad
     f\in W^{k,p}_0(Q).
\end{equation}
Different choices of such local coordinates or subordinated
partitions of unity produce the same set $W^{k,p}_0(Q)$
and equivalent norms. So the induced topology on $W^{k,p}_0(Q)$ is the
same. The normed space $W^{k,p}_0(Q)$ is complete,
as each $W^{k,p}_0(\Omega_i)$ is.
Let $S_i$ be a countable dense subset of $W^{k,p}_0(\Omega_i)$
and consider the corresponding subset
$S_{U_i}:=\{ s\circ\varphi_i\mid s\in S_i\}$
of the set of compactly supported functions on $U_i$.
Note that $S_{U_i}\subset W^{k,p}_0(Q)$.
Let $S$ be the set of finite sums of elements of
$S_{U_1}\cup\dots\cup S_{U_r}$.
Then $S$ is not only countable (one even can allow rational
coefficients), but also a dense subset of the Banach space
$W^{k,p}_0(Q)$: To see this write $f\in W^{k,p}_0(Q)$ as the finite
sum $f=\sum_i \rho_i f$ and approximate each $f^\rho_i\in
W^{k,p}_0(\Omega_i)$ by a sequence $(s_i^\nu)_{\nu\in\N}\subset S_i$.
Then $f^\nu:=s_1^\nu\circ\varphi_1+\dots+s_r^\nu\circ\varphi_r\in S$
converges to $\rho_1 f+\dots+\rho_r f=f$
in the $W^{k,p}$ norm~(\ref{eq:sobolev-norm-Q}).
\end{proof}

\subsection*{The Banach manifold $\mbf{\Uu=W^{k,p}(\SS^1,Q)}$ for $\mbf{kp>1}$}

\begin{proposition}[The Banach manifold $\Uu=W^{k,p}_0(\SS^1,Q)$]
\label{prop:B-mf}
Let $Q$ be a closed manifold of dimension $n$. Given an integer
$k$ and a real $p\in[1,\infty)$ such that $kp>1$, then $W^{k,p}_0(\SS^1,Q)$
is a smooth Banach manifold which admits a countable
atlas modeled on the separable Banach space $W^{k,p}_0(\SS^1,\R^n)$
whenever $Q$ is orientable.\footnote{
  For non-orientable $Q$ use model spaces
  $W^{k,p}_\sigma([0,1],\R^n)$ as defined in~\citerefFH[(8)]{weber:2002a}.
  }
(Since $\SS^1$ is compact, the notation $W^{k,p}(\SS^1,Q)$ is common.)
\end{proposition}

\begin{proof}[Idea of proof]
Use Whitney's Theorem to fix an embedding $Q\INTO\R^\ell$.
for some sufficiently large integer $r$.
Consider the separable Banach space $W^{k,p}_0(\SS^1,\R^\ell)$
provided by Proposition~\ref{prop:complete-separable-W}.
Actually here (by periodicity of the domain) separability follows
since the set of Fourier series with coefficients in $\Q^r$
is dense in $C^\infty_0(\SS^1,\R^\ell)$, hence in its closure
$W^{k,p}_0(\SS^1,\R^\ell)$.
Now define $\mbf{W^{k,p}_0(\SS^1,Q)}$ as the (metric) subspace of
$W^{k,p}_0(\SS^1,\R^\ell)$ that consists of those elements
which take values in $Q$.
Thus $W^{k,p}_0(\SS^1,Q)$ inherits separability and it is complete
since it is closed: Closedness follows from compactness of
$Q$ in combination with the Sobolev embedding theorem;
see e.g.~\cite[Thm.\,4.12 III (6)]{Adams:2003a}.
It applies since $kp>\dim \SS^1$ and implies that the elements
of $W^{k,p}_0(\SS^1,\R^\ell)$ are continuous loops.

Fix a Riemannian metric on $Q$. Using the associated
exponential map one can construct,
about each smooth element $\gamma$ of $W^{k,p}_0(\SS^1,Q)$,
a local coordinate chart taking values in
$W^{k,p}_0(\SS^1,\gamma^*TQ)\simeq W^{k,p}_0(\SS^1,\R^n)$.
The isomorphism is natural once one fixes a trivialization
of the pull-back vector bundle $\gamma^*TQ\to Q$;
a trivialization exists since $Q$ is orientable.
The collection of charts associated to smooth
loops is an open cover of $W^{k,p}_0(\SS^1,Q)$.
Every open cover of a separable metric space
admits a countable subcover;
see e.g.~\citerefFH[Rmk.\,3.5]{weber:2002a}.
\end{proof}

\cleardoublepage
\phantomsection
\chapter{Functional analysis}\label{sec:functional-analysis}


\section{Duality}\label{sec:sub-dual}
The summaries and preparations in Subsection~\ref{sec:annihilators}
and~\ref{sec:formal-adjoint} are directed
towards the analysis of the Fredholm operator
$D_u$ in Section~\ref{sec:mod-space}.

\subsection{Annihilators and cokernels}\label{sec:annihilators}
For convenience we enlist some useful basics; for details we recommend,
for instance, the excellent
presentations~\cite[\S 5 \S 8]{lax:2002a} and~\cite{Rudin:1991b,brezis:2011a}.

\vspace{.1cm}
\textit{Linear spaces.} One defines the quotient $X/Y$ of a (as always real)
linear space by a linear subspace as the set $\{x+Y\mid x\in X\}$ of
cosets equipped with addition defined by choosing representatives, add
them, and form their coset, similarly for multiplication by reals.
The dimension of the linear space $X/Y$ is called the
\textbf{\boldmath\Index{codimension} of the subspace $Y$ in $X$},
in symbols\index{$\codim Y:=\dim X/Y$}
\[
     \codim Y:=\dim X/Y.
\]

\vspace{.1cm}
\textit{Normed linear spaces.}
Suppose that $X$ carries a \textbf{\Index{norm}}\footnote{
  \textbf{Norm.} Positive: $\norm{x}=0$ $\Rightarrow$ $x=0$,
  subadditive: $\norm{x+y}\le\norm{x}+\norm{y}$, homogeneous:
  $\norm{\alpha x}=\abs{\alpha}\cdot\norm{x}$.
  }
$\norm{\cdot}: X\to[0,\infty)$, so an associated topology. Whereas any linear
subspace $Y$ inherits a norm by restriction, the quotient space $X/Y$
{\color{cyan}\underline{{\color{black}inherits a norm}}} in case
\underline{$Y$ is closed}, namely
\[
     \Norm{x+Y}:=\inf_{z\in x+Y}\Norm{z},\qquad
     x+Y\in X/Y.
\]
In case the normed space $X$ is in addition complete, in other words, a
\textbf{\Index{Banach space}}, and the linear subspace $Y$ is closed,
thus a Banach space itself, well, in this case the norm on the
quotient space $X/Y$ is complete, too.
\newline
A linear map $D:X\to Y$ between normed linear spaces is called a
\textbf{\Index{bounded linear operator}}\index{linear operator!bounded --}
if there is a constant $c$ such that $\norm{Dx}\le c\norm{x}$ for
all $x\in X$.\footnote{
  Boundedness is equivalent to continuity for
  linear maps between normed linear spaces.
  }
The smallest such $c$ is called the
\textbf{\Index{operator norm}} or\index{norm!operator --}
just the \textbf{norm of $\mbf{D}$}, i.e.
\begin{equation*}
\begin{split}
     \Norm{D}=\Norm{D}_{\Ll(X,Y)}
   :&=\inf\{c : \text{$\norm{Dx}\le c\norm{x}$ $\forall x\in X$}\}\\
   &=\sup_{x\in X\setminus\{0\}}\frac{\norm{Dx}}{\norm{x}}.
\end{split}
\end{equation*}
By $\Ll(X,Y)$ one denotes the set of all bounded linear operators
$D:X\to Y$. It is naturally a normed linear space, even
complete whenever the target $Y$ is complete.
In case of the target space $\R$ one denotes the (Banach)
space $\Ll(X,\R)$ by $\mbf{\Hdual{X}}$ and calls it the \textbf{\Index{dual
space}} of the normed linear space $X$.\index{$\Hdual{X}:=\Ll(X,\R)$ dual space}
The elements of $\Hdual{X}$ are called \textbf{bounded}, likewise
\textbf{continuous}, \textbf{linear functionals} on $X$.

\begin{exercise}[Extension by continuity principle]\label{exc:ECP}
Consider\index{extension by continuity principle}
a Banach space $Y$ and a normed linear space $X$
with a dense linear subspace $U$.
Then any bounded linear map $A:U\to Y$ extends uniquely to a bounded
linear map $\overline{A}:X\to Y$ which has the same operator norm.
\newline
[Hint: There are three crucial assumptions, boundedness of $A:U\to Y$,
completeness of $Y$, and density of $U\subset X$.
Consequently $A$ takes Cauchy sequences to Cauchy sequences,
the latter admit a limit, and the extension is unique.
In case you get stuck consult~\cite[Thm.\,1.12]{Teschl:2016a}.]
\end{exercise}

\begin{figure}[h]
\begin{equation*}
\begin{tikzcd} [row sep=normal, column sep=huge] 
\left(X,\Norm{\cdot}_X\right)
\arrow[r, dashed, "{\bar{A}\in\Ll(X,Y)}"]
  &\left(Y,\Norm{\cdot}_Y\right)\; \text{complete}
\\
U
\arrow[u, white, shift left=2.5, "\scriptstyle{\color{black}\text{dense}}"]
\arrow[u, phantom, "\bigcup"]
\arrow[ur, "A", "\text{linear $\&$ bounded}"']
  &
\end{tikzcd}
\end{equation*}
\caption{Extension by continuity principle}
\label{fig:ECP}
\end{figure}
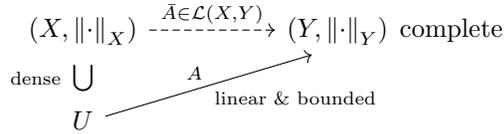

The \textbf{\Index{annihilator}} of a linear subspace $Y\subset X$
of\index{$\pann{Y}\subset\Hdual{X}$!annihilator of subspace $Y\subset X$}
a normed linear space consists of all continuous linear functionals on $X$
which vanish on $Y$, in symbols
\[
     \pann{Y}:=\{\lambda\in \Hdual{X}\mid \text{$\lambda(y)=0$ $\forall y\in Y$}\}.
\]
Occasionally it is more convenient to use the notation $\Ann\, Y$.
The annihilator is a closed linear subspace of the Banach space $\Hdual{X}$.
It holds that $\pann{(Y^\perp)}=\overline{Y}$;
see e.g.~\cite[Prop.\,1.9]{brezis:2011a}.
There is an isometric isomorphism $\Hdual{Y}\simeq \Hdual{X}/\pann{Y}$ of Banach
spaces; cf.~\cite[\S 8 Thm.\,$7^\prime$]{lax:2002a}.
For a proof of this isomorphism, as well as the following one,
see~\cite[Thm.\,4.9]{Rudin:1991b}.
\newline
If in addition \underline{$Y$ is closed}, then there is an isometric Banach space
isomorphism\footnote{Well defined:
  $\lambda$ vanishes on $Y$. Linear: $\lambda$ is linear.
  Injective: Equivalence classes
  partition $X$, so $\Lambda_\lambda=0$ means $\lambda$
  evaluated on any $x\in X$ is zero. Onto: Set $\pi:X\to X/Y$,
  $x\mapsto x+Y$. Given $L\in\Hdual{(X/Y)}$, set $\lambda:=L\circ \pi$,
  check that $\Lambda_\lambda=L$.
  Isometry: Hint~\cite[\S 8 Thm.\,$7$]{lax:2002a}.
}
\begin{equation}\label{eq:iso-annih-quotient}
     \Xi:\pann{Y}\stackrel{\simeq}{\longrightarrow}\Hdual{\left( X/Y\right)},\quad
     \lambda\mapsto \Lambda_\lambda,\qquad
     \text{$Y\subset X$ closed.}
\end{equation}
Here\index{$\Xi:\pann{Y}\simeq\Hdual{( X/Y)}$ isometric isomorphism of
Banach spaces}
$\Lambda_\lambda$ is defined on a coset by evaluating $\lambda$ on
any\index{$\Lambda_\lambda$, $\Lambda_\eta$}
representative, that is
\[
     \Lambda_\lambda(x+Y):=\lambda(z),\quad
     z\in x+Y.
\]

\vspace{.1cm}
\textit{Linear spaces associated to a bounded linear operator.}
Let $X$ and $Y$ be normed linear spaces.
By continuity the \textbf{\Index{kernel} of} $\mbf{D}\in\Ll(X,Y)$ defined
as\index{$\ker D=\{x\in X\mid Dx=0\}$ kernel}
the pre-image $\ker D:=D^{-1}(0)$ of the zero vector is always a
\emph{closed} linear subspace.
In\index{image subspace}\index{linear operator!image of --}
contrast,\index{linear operator!kernel of --}
the image $\im D:=D(X)$ need not be closed, in general.
However, the image is a linear subspace of $Y$, so one can form the linear
quotient space
\[
     \coker D:=\frac{Y}{\im D}
\]
called the\index{linear operator!cokernel of --}
\textbf{\Index{cokernel} of $\mbf{D}$}.
If the image of $D$ is closed the cokernel is a normed space
whose dual is isometrically isomorphic to the annihilator of the image:
\begin{equation}\label{eq:iso-annih-coker}
     \Xi:\pann{(\im D)}\stackrel{\simeq}{\longrightarrow}\Hdual{\left(\coker
       D\right)},\quad
     \lambda\mapsto \Lambda_\lambda,\qquad
     \text{$\im D\subset Y$ closed.}
\end{equation}

\subsection{Formal adjoint operator}\label{sec:formal-adjoint}
The purpose of this section is to formalize the construction
of what is commonly called the \emph{formal adjoint}
by emplyoing unbounded functional analytic adjoints.

Consider the linear space
$C^\infty_0:=C^\infty_0(\R\times\SS^1,\R^{2n})$
of\index{$C_0^\infty$ with inner product $\inner{\cdot}{\cdot}$}
compactly supported smooth functions with values
in euclidean space $\R^{2n}$. The \textbf{\Index{pairing}}
\begin{equation}\label{eq:pairing}
     \INNER{f}{g}
     :=\int_{-\infty}^\infty\int_0^1
     \INNER{f(s,t)}{g(s,t)}_{\R^{2n}}\: dt ds
     \qquad f,g\in C^\infty_0
\end{equation}
defines an inner product on $C^\infty_0$.\footnote{
  \textbf{Inner product.} Real valued, bi-linear,
  symmetric:\index{inner product}
  $\inner{f}{g}=\inner{g}{f}$, non-degenerate:
  $\inner{f}{g}=0$ $\forall g$ $\Rightarrow$ $f=0$.
  One proves non-degeneracy by contradiction
  using continuity of $f$. Continuity plus compact support
  implies real valuedness: $\inner{f}{g}<\infty$.  
  }
It is extremely useful that $C_0^\infty$ is dense\footnote{
  Dense ($p\in[1,\infty)$): See~{\cite[Ch.\,4]{Salamon:2016a}} for domains
  which are locally compact Hausdorff spaces, e.g. $\R\times\SS^1$,
  and a class of measures $\mu$ that includes the Lebesgue measure used here.
  }
in the separable\footnote{
  The metric space $L^p(\R^n)$ is separable for $p\in[1,\infty)$,
  cf.~\cite[Thm.\,4.13]{brezis:2011a}, thus so are subspaces.
  Embedd $L^p(\R\times\SS^1)\INTO L^p(\R^2)$
  via zero-extension of $f:\R\times[0,1)\to\R$.
  }
reflexive\footnote{
  True for $p\in(1,\infty)$, cf.~\cite[Thm.\,4.10]{brezis:2011a}.
  Reflexivity transfers to closed linear subspaces.
  }
Banach\footnote{
  Banach space for $p\in[1,\infty]$,
  cf.~\cite[Thm.\,4.9]{Salamon:2016a}.
  }
spaces $L^p:=L^p(\R\times\SS^1,\R^{2n})$ for any real $p\in[1,\infty)$.
On\index{real numbers!extended --}
the other hand, giving up on real valuedness and allowing for
values in the \textbf{\Index{extended reals}}
\Index{$\overline{\R}:=\R\cup\{-\infty,+\infty\}$},
the pairing makes sense on the larger vector space
of all measurable functions. Indeed H\"older's inequality
\[
     \INNER{f}{g}\le\Norm{f}_{L^p}\Norm{g}_{L^q},\qquad
     \frac{1}{p}+\frac{1}{q}=1,\qquad
     1<p<\infty,
\]
holds true for merely measurable $f$ and $g$; see
e.g.~\cite[Thm.\,4.1]{Salamon:2016a}.
Let us call the reals $p$ and $q$ a pair\index{pair of conjugate exponents}
of \textbf{conjugate exponents}, alternatively, call any one of them
the\index{$1/p+1/q=1$!conjugate exponents}
\textbf{\Index{conjugate exponent}} of the other one.
Observe that the pairing $\inner{\cdot}{\cdot}$ extends
to $L^q\times L^p$ by H\"older and, based on density of $C_0^\infty$,
the extension by continuity principle. The extension is still non-degenerate.
H\"older\index{conjugate exponents $1/p+1/q=1$}
guarantees real valuedness (finiteness) of the
pairing whenever $f\in L^q$ and $g\in L^p$. This trivial observation
turns rather powerful when translated to boundedness, i.e. continuity,
of the linear functionals $f^*:=\inner{f}{\cdot}$ arising by fixing one function
and considering variable only the other one.

For a pair of conjugate exponents $p,q>1$ the map between Banach
spaces\footnote{
  Throughout all domains are cylinders $\R\times\SS^1$,
  targets $\R^{2n}$, e.g. $L^q$ means $L^q(\R\times\SS^1,\R^{2n})$.
  }
\begin{equation}\label{eq:duality-iso}
     L^q:\to (L^p)^*,\quad
     v\mapsto v^*:=\inner{v}{\cdot},
\end{equation}
is an isometric isomorphism, the\index{isomorphism!duality --}
\textbf{\Index{duality isomorphism}};
cf.~\cite[Thm.\,4.35]{Salamon:2016a}.
What follows will use the for $1<p<\infty$ separable reflexive Banach
spaces\footnote{
  The $L^p$ spaces are reflexive, separable Banach
  spaces for $p\in(1,\infty)$, hence so is $E:=L^p\times L^p\times L^p$.
  The map $T:W^{1,p}\to E$, $\xi\mapsto(\xi,\p_s \xi,\p_t \xi)$,
  mapping $\xi$ to its weak derivatives, is an isometry into a Banach space,
  so the image is closed. But completeness, reflexivity, and separability
  are transferred to closed subspaces; for details
  see e.g.~\cite[Proposition~8.1]{brezis:2011a}.
  }
\begin{equation*}
\begin{split}
     W^{1,p}:
   &=\{\xi\in L^p_\loc\mid\text{$\exists$ weak first derivatives in
     $L^p_\loc$ and $\norm{\xi}_{W^{1,p}}<\infty$}\}
     \\
   &=\overline{C_0^\infty}^{1,p}\\
   &=: W^{1,p}_0 .
\end{split}
\end{equation*}
The
equality is a theorem and $\norm{\cdot}_{W^{q,p}}=\norm{\cdot}_{1,p}$
is defined by~(\ref{eq:sobolev-norm}).

\subsubsection*{Goal}
Given a pair of conjugate exponents $p,q>1$,
we wish to associate to the \emph{bounded} linear operator $D:W^{1,p}\to L^p$
in~(\ref{eq:triv-D_u-lin-Floer}) a bounded linear operator
$\Hdual{D}:W^{1,q}\to L^q$ which satisfies for each $\eta\in W^{1,q}$ the identity
\[
     \INNER{\eta}{D\xi}=\INNER{\Hdual{D}\eta}{\xi},\qquad \forall \xi\in W^{1,p}.
\]
If existent such $\Hdual{D}$ will be unique by non-degeneracy of the pairing.
Trying the functional analytic adjoint fails, as $W^{1,q}$
is not the dual of $L^p$, and $L^q$ not the one of $W^{1,p}$.
A way out is to view $D$ as an \emph{unbounded} linear operator
\[
    D:L^p\supset\dom D=W^{1,p}\to L^p
\]
and determine\index{$\Bdual{D}$!functional analytic adjoint}
its\index{adjoint!unbounded functional analytic --}
\textbf{unbounded functional analytic adjoint}
\begin{equation*}
\begin{tikzcd} [row sep=tiny]
\Bdual{D}:\Hdual{(L^p)}\supset\dom \Bdual{D}
\arrow[r]
      &\Hdual{(L^p)}
\\
\Hdual{D}:\;\; L^q\;\;\supset\dom \Hdual{D}
\arrow[r]
      &L^q.
\end{tikzcd}
\end{equation*}
Here $\Hdual{D}$ denotes the representative of $\Bdual{D}$
under the duality isomorphism, i.e.
\begin{equation}\label{eq:Dprime=D*}
     \Bdual{D}\eta^*=\INNER{\Hdual{D}\eta}{\cdot}\in\Hdual{(L^p)}.
\end{equation}

\subsubsection*{\boldmath Functional analytic adjoint $\Bdual{D}$ of the unbounded linear operator $D$}
Following~\cite[\S 2.6]{brezis:2011a} we will determine the unbounded
dual operator, in particular, its domain.
Once succeeded we simply forget the ambient spaces and consider both
operators as bounded operators on their respective domains.
Set
\[
     \dom\Bdual{D}
     :=\{\eta^*\in\Hdual{(L^p)}\mid\text{$\exists c_\eta\ge 0$ with
       $\abs{\inner{\eta}{D\xi}}\le c_\eta\norm{\xi}_{L^p}$, $\forall\xi\in\dom D$}\}
\]
where $\eta^*=\inner{\eta}{\cdot}$ for a unique $\eta\in L^q$ by the duality isomorphism.
The defined domain is a linear subspace of $\Hdual{(L^p)}$.
Appearance of the $L^p$ norm, as opposed to the $W^{1,p}$ norm,
is consistent with the fact that $D$ is now viewed as an (unbounded)
operator on $L^p$. Before showing that $\dom\Bdual{D}$ corresponds
under duality to $W^{1,q}$, let us define $\Bdual{D} \eta^*$: Pick
$\eta^*=\inner{\eta}{\cdot}\in\dom\Bdual{D}$ and define the linear functional
\[
     \ell_\eta:\dom D\to\R,\quad
     \xi\mapsto\INNER{\eta}{D\xi}
\]
on the domain of $D$. It is $L^p$ bounded. Indeed
\[
     \Abs{\ell_\eta(\xi)}:=\Abs{\INNER{\eta}{D\xi}}
     \le c_\eta\Norm{\xi}_{L^p},\quad\forall\xi\in\dom D=W^{1,p}.
\]
In our particular situation $\dom D$ is
$W^{1,p}$ which is dense in $L^p$ since already its subset
$C_0^\infty$ is.
So the extension by continuity principle, Exercise~\ref{exc:ECP},
applies and shows that $\ell_\eta$ extends uniquely
to a bounded linear functional $L_\eta$ on the
ambient space $L^p$.\footnote{
  The domain of a general unbounded linear operator $D$ on
  a Banach space $X$ might not be dense in $X$.
  In such a case invoke Hahn-Banach, see e.g.\cite[Thm.\,1.1]{brezis:2011a},
  which provides a \emph{linear} extension of
  $\ell_\eta$, no continuity asserted, but subject to the same bound
  function $c_\eta\norm{\cdot}_{X}$ (which indeed makes sense on the
  whole space $X$). To summarize, one gets a function
  \[
     \text{$L_\eta:X\to\R$ linear and subject to
     $\abs{L_\eta(x)}\le c_\eta\norm{x}_{X}$, $\forall x\in X$.}
  \]
  But this means that $L_\eta$ is continuous, hence element of
  $\Hdual{X}$. However, without density of $\dom D$ in $X$
  one cannot expect uniqueness of the extension.
  See also~\cite[\S 2 Rmk.\,14]{brezis:2011a}.
  }
The \textbf{unbounded functional analytic adjoint of $\mbf{D}$}
is by definition the unbounded linear operator
\[
     \Bdual{D} 
     :\Hdual{(L^p)}\supset\dom\Bdual{D}
     \to\Hdual{(L^p)},\quad
     \eta^*\mapsto L_\eta,\qquad
     L_\eta|_{\dom D}=\INNER{\eta}{D\cdot}.
\]

\subsubsection*{\boldmath Calculation of $\dom\Bdual{D}$}
Translated to $L^q$ via the duality isomorphism
the domain of $\Bdual{D}$ corresponds to
\begin{equation*}
\begin{split}
     \dom\Hdual{D}
   &=\{\eta\in L^q\mid\text{$\exists c_\eta\ge 0$ with
     $\abs{\inner{\eta}{D\xi}}\le c_\eta\norm{\xi}_{L^p}$, $\forall \xi\in W^{1,p}$}\}
     \\
   &=W^{1,q}(\R\times\SS^1,\R^{2n})
\end{split}
\end{equation*}
where equality two is to be proved:
It suffices to show, given $\eta\in L^q$, that
\[
     \eta\in W^{1,q}\quad\Leftrightarrow\quad
     \exists c\ge 0:\;\Abs{\INNER{\eta}{(\p_s-J_0\p_t-S)\varphi}}
     \le c\Norm{\varphi}_{L^p},\:\:\forall\varphi\in C_0^\infty.
\]
Observe that we replaced on the right hand side $W^{1,p}$ by the dense subset
$C_0^\infty$. This is justified by the extension by continuity principle.

'$\Rightarrow$' Pick $\eta\in W^{1,q}$ and $\varphi\in
C_0^\infty$, then by definition of weak differentiability of $\eta$
(let $\p_s\eta,\p_t\eta\in L^q$ denote the weak derivatives)
and by symmetry of each matrix $S(s,t)$ we get the identity,
whereas H\"older provides the inequality, in
\[
     \Abs{\INNER{\eta}{D\varphi}}
     =\Abs{\INNER{-\p_s\eta-J_0\p_t\eta-S\eta}{\varphi}}
     \le\underbrace{\left(1+\Norm{S}_{L^\infty}\right)\Norm{\eta}_{W^{1,q}}}_{=:c_\eta}
     \Norm{\varphi}_{L^p}.
\]
The continuous family of symmetric matrices $S(s,t)$ has the
asymptotic limits $S^\mp(t)$, as $s\to\mp\infty$, uniformly in
$t$; cf. text following~(\ref{eq:triv-D_u-lin-Floer}).
So $\Norm{S}_{L^\infty}<\infty$.

'$\Leftarrow$'
Fix $\eta\in L^q$. The linear functional defined by
\[
     \lambda_\eta:C_0^\infty\to\R,\quad
     \varphi\mapsto\INNER{\eta}{(\p_s-J_0\p_t)\varphi}
\]
is bounded with respect to the $L^p$ norm. Indeed by assumption and H\"older
\[
     \Abs{\lambda_\eta(\varphi)}=\Abs{\INNER{\eta}{D\varphi+S\varphi}}
     \le c\Norm{\varphi}_{L^p}
     +\Norm{S}_{L^\infty}\Norm{\eta}_{L^q}\Norm{\varphi}_{L^p}.
\]
By the extension by continuity principle $\lambda_\eta$ extends
uniquely to a bounded linear functional $\Lambda_\eta$ on $L^p$.
By the duality isomorphism there is an element $v_\eta\in L^q$ such
that $\Lambda_\eta=\inner{v_\eta}{\cdot}:L^p\to\R$. Putting things
together we get that
\[
     \INNER{\eta}{(\p_s-J_0\p_t)\varphi}
     =:\lambda_\eta(\varphi)
     =\Lambda_\eta(\varphi)
     =\INNER{v_\eta}{\varphi},\quad
     \forall \varphi\in C_0^\infty.
\]
For $\varphi=(\p_s+J_0\p_t)\psi$ and with
$\Delta:=\p_s\p_s+\p_t\p_t$ we get the identity
$
     \INNER{\eta}{\Delta\psi}
     =\INNER{v_\eta}{(\p_s+J_0\p_t)\psi}
$
for every $\psi\in C_0^\infty$. Elliptic regularity tells that
$\eta\in W^{1,q}$.\footnote{
  Apply~\cite[Thm.\,B.3.2]{mcduff:2004a} which works for
  $C_0^\infty(\R^2)$. In view of our cylinder domain,
  localize the problem using suitable cutoff functions.
  One gets that $\eta$ lies in $W^{1,q}_\loc$ and
  \[
     \Norm{\eta}_{W^{1,q}([k,k+1]\times\SS^1)}
     \le C\left(\Norm{v_\eta}_{L^q([k-1,k+2]\times\SS^1)}
     +\Norm{\eta}_{L^q([k-1,k+2]\times\SS^1)}\right)
  \]
  where $C$ does not depend on $k$. Sum over $k$ to obtain that along
  the whole
  $\R\times\SS^1$ the $W^{1,q}$ norm of $\eta$ is bounded above by
  constants and the $L^q$ norms of $v_\eta$ and $\eta$, but these are finite.
  }
This concludes the proof of the equivalence,\footnote{
  More information around the proof of the equivalence
  can be found, for domains being open subsets $\Omega\subset\R^N$
  for both cases $W^{1,p}$ and $W^{1,p}_0$, in~\cite[Propositions~8.3,
  9.3, 9.18]{brezis:2011a}.
}
hence of $\dom\Hdual{D}=W^{1,q}$.

\subsubsection*{\boldmath Formal adjoint $\Hdual{D}$ of the bounded linear operator $D$}
Given a pair of conjugate exponents $1<p,q<\infty$,
consider the bounded linear operator
$D=\p_s-J_0\p_t-S:W^{1,p}\to L^p$ arising in~(\ref{eq:triv-D_u-lin-Floer}).
Viewing $D$ as an unbounded operator on $L^p$ with dense domain
$W^{1,p}$ we have above determined, firstly, the unbounded functional
analytic adjoint $\Bdual{D}$ on $\Hdual{(L^p)}$ with domain
$\dom\Bdual{D}$ and, secondly, its representative under duality,
namely, the unbounded linear operator $\Hdual{D}$ on $L^q$ whose
domain is the dense subspace $W^{1,q}$, as we saw.

\begin{definition}\label{def:formal-adjoint}
By definition the \textbf{\Index{formal adjoint operator}}
of\index{$\Hdual{D}$!formal adjoint}\index{adjoint!formal --}
the bounded linear operator
$D=\p_s-J_0\p_t-S:W^{1,p}\to L^p$\index{operator!formal adjoint --}
is the representative $\Hdual{D}$ of the unbounded functional analytic
adjoint, but viewed as a linear operator from the
Banach space $W^{1,q}$ to $L^q$.
\end{definition}

\begin{proposition}\label{prop:formal-adjoint}
The formal adjoint of $D=\p_s-J_0\p_t-S:W^{1,p}\to L^p$
is the linear operator\,\footnote{
  Recall that $L^p$ abbreviates $L^p(\R\times\SS^1,\R^{2n})$
  and analogously for the other spaces.
  }
\[
     \Hdual{D}=\p_s-J_0\p_t-S:W^{1,q}\to L^q ,\qquad
     \frac{1}{p}+\frac{1}{q}=1,\quad
     1<p<\infty.
\]
The formal adjoint 
is characterized by satisfying, for
each\index{characterizing identity for $\Hdual{D}$}
$\eta\in W^{1,q}$, the identity
\begin{equation}\label{eq:formal-adjoint-identity}
     \INNER{\eta}{D\varphi}
     =\INNER{\Hdual{D}\eta}{\varphi},\qquad \forall \varphi\in C_0^\infty.
\end{equation}
In our case the formal adjoint $\Hdual{D}$ is bounded of the same operator
norm as $D$.
\end{proposition}

\begin{proof}
Pick $\eta\in W^{1,q}$, then for $\xi\in C_0^\infty$ and
by~(\ref{eq:Dprime=D*}) we get that
\[
     \INNER{\Hdual{D}\eta}{\xi}
     =\left(\Bdual{D}\eta^*\right)\xi
     :=L_\eta(\xi)
     =\ell_\eta(\xi)
     :=\INNER{\eta}{D\xi}.
\]
Here we used that $L_\eta$ restricts to $\ell_\eta$ on
$\dom D=W^{1,p}$ which contains $C_0^\infty$.
Thus $\Hdual{D}$ satisfies~(\ref{eq:formal-adjoint-identity}),
uniqueness holds by non-degeneracy of the pairing.

Let $\sigma\in L^q$ be the weak derivative of $\eta$ with respect
to $s$ and $\tau\in L^q$ the one with respect to $t$.
Then by definition of the weak derivatives
$\forall \varphi\in C_0^\infty$ we get
\[
     \INNER{-\sigma+{J_0}^T\tau-S^T\eta}{\varphi}
     =\INNER{\eta}{(\p_s-J_0\p_t-S)\varphi}
     =\INNER{\eta}{D\varphi}
     =\INNER{\Hdual{D}\eta}{\varphi}.
\]
As $J_0$ is antisymmetric and $S$ is pointwise symmetric,
non-degeneracy of the pairing implies that
$\Hdual{D}\eta=-\sigma-{J_0}\tau-S\eta$.
For simplicity one writes $\Hdual{D}=-\p_s-J_0\p_t-S$ with
the understanding that these symbols indicate weak derivatives.
(Applied to maps of class $C^1$ they coincide with the usual derivatives.)

For $\eta\in W^{1,q}$ one immediately sees that
\[
     \Norm{\Hdual{D}\eta}_{L^q}
     =\Norm{D\eta}_{L^q}
     \le\left(1+\Norm{S}_{L^\infty}\right)\Norm{\eta}_{W^{1,q}}
\]
where the identity follows by the variable transformation
$s\mapsto -s$.
\end{proof}

By density $C_0^\infty\subset W^{1,p}$ the characterizing identity continues
to hold for all elements $\xi\in W^{1,p}$ and then, together with
non-degeneracy of the pairing, shows that the duality isomorphism
$L^q\simeq\Hdual{(L^p)}$ restricts to an isomorphism
\begin{equation}\label{eq:iso-ker-coker}
     \ker\Hdual{D}\simeq\pann{(\im D)}.
\end{equation}

\cleardoublepage
\phantomsection

\backmatter 

\bibliographystyle{alpha}
\cleardoublepage
\phantomsection
\addcontentsline{toc}{chapter}{Textbook references}
\bibliography{$HOME/Dropbox/0-Libraries+app-data/Bibdesk-BibFiles/library_math}{}

\begin{thebibliography}{{Lan}01}

\bibitem[Abb14]{Abbas:2014a}
Casim Abbas.
\newblock {\em An introduction to compactness results in symplectic field
  theory}.
\newblock Springer, Heidelberg, 2014.

\bibitem[AF03]{Adams:2003a}
Robert~A. Adams and John J.~F. Fournier.
\newblock {\em Sobolev spaces}, volume 140 of {\em Pure and Applied Mathematics
  (Amsterdam)}.
\newblock Elsevier/Academic Press, Amsterdam, second edition, 2003.

\bibitem[AM78]{abraham:1978a}
Ralph Abraham and Jerrold~E. Marsden.
\newblock {\em Foundations of mechanics}.
\newblock Benjamin/Cummings Publishing Co., Inc., Advanced Book Program,
  Reading, Mass., 1978.
\newblock Second edition, revised and enlarged, With the assistance of Tudor
  Ra{\c{t}}iu and Richard Cushman.

\bibitem[AP93]{ambrosetti:1993a}
Antonio Ambrosetti and Giovanni Prodi.
\newblock {\em A primer of nonlinear analysis}, volume~34 of {\em Cambridge
  Studies in Advanced Mathematics}.
\newblock Cambridge University Press, Cambridge, 1993.

\bibitem[AR67]{Abraham:1967a}
Ralph Abraham and Joel Robbin.
\newblock {\em Transversal mappings and flows}.
\newblock An appendix by Al Kelley. W. A. Benjamin, Inc., New York-Amsterdam,
  1967.

\bibitem[Arn78]{Arnold:1978a}
V.~I. Arnold.
\newblock {\em Mathematical methods of classical mechanics}.
\newblock Springer-Verlag, New York-Heidelberg, 1978.
\newblock Translated from the Russian by K. Vogtmann and A. Weinstein, Graduate
  Texts in Mathematics, 60.

\bibitem[Bre83]{Brezis:1983a}
Ha{\"{\i}}m Brezis.
\newblock {\em Analyse fonctionnelle}.
\newblock Collection Math\'ematiques Appliqu\'ees pour la Ma\^\i trise.
  [Collection of Applied Mathematics for the Master's Degree]. Masson, Paris,
  1983.
\newblock Th{\'e}orie et applications. [Theory and applications].

\bibitem[Bre11]{brezis:2011a}
Ha{\"{\i}}m Brezis.
\newblock {\em Functional analysis, {S}obolev spaces and partial differential
  equations}.
\newblock Universitext. Springer, New York, 2011.

\bibitem[BT82]{bott:1982a}
Raoul Bott and Loring~W. Tu.
\newblock {\em Differential forms in algebraic topology}, volume~82 of {\em
  Graduate Texts in Mathematics}.
\newblock Springer-Verlag, New York, 1982.

\bibitem[CE12]{Cieliebak:2012a}
Kai Cieliebak and Yakov Eliashberg.
\newblock {\em From {S}tein to {W}einstein and back}, volume~59 of {\em
  American Mathematical Society Colloquium Publications}.
\newblock American Mathematical Society, Providence, RI, 2012.
\newblock Symplectic geometry of affine complex manifolds.

\bibitem[Cha05]{Chang:2005a}
Kung-Ching Chang.
\newblock {\em Methods in nonlinear analysis}.
\newblock Springer Monographs in Mathematics. Springer-Verlag, Berlin, 2005.

\bibitem[Con85]{Conway:1985a}
John~B. Conway.
\newblock {\em A course in functional analysis}, volume~96 of {\em Graduate
  Texts in Mathematics}.
\newblock Springer-Verlag, New York, 1985.

\bibitem[Gei08]{geiges:2008a}
Hansj{{\"o}}rg Geiges.
\newblock {\em An introduction to contact topology}, volume 109 of {\em
  Cambridge Studies in Advanced Mathematics}.
\newblock Cambridge University Press, Cambridge, 2008.

\bibitem[GH78]{Griffiths:1978a}
Phillip Griffiths and Joseph Harris.
\newblock {\em Principles of algebraic geometry}.
\newblock Wiley-Interscience [John Wiley \& Sons], New York, 1978.
\newblock Pure and Applied Mathematics.

\bibitem[GP74]{guillemin:1974a}
Victor Guillemin and Alan Pollack.
\newblock {\em Differential topology}.
\newblock Prentice-Hall, Inc., Englewood Cliffs, N.J., 1974.

\bibitem[Hir76]{hirsch:1976a}
Morris~W. Hirsch.
\newblock {\em Differential topology}.
\newblock Springer-Verlag, New York-Heidelberg, 1976.
\newblock Graduate Texts in Mathematics, No. 33.

\bibitem[HZ11]{hofer:2011a}
Helmut Hofer and Eduard Zehnder.
\newblock {\em Symplectic invariants and {H}amiltonian dynamics}.
\newblock Modern Birkh{\"a}user Classics. Birkh{\"a}user Verlag, Basel, 2011.
\newblock Reprint of the 1994 edition.

\bibitem[Inf16]{Infusino:2016a}
Maria Infusino.
\newblock {Topological Vector Spaces, last accessed 13/10/2017 on}
  \href{http://www.math.uni-konstanz.de/~infusino/Note.pdf}{Webpage}, {Lecture
  Notes}, 2016.

\bibitem[Kel55]{Kelley:1955a}
John~L. Kelley.
\newblock {\em General topology}.
\newblock D. Van Nostrand Company, Inc., Toronto-New York-London, 1955.

\bibitem[KN76]{Kelley:1976a}
John~L. Kelley and Isaac Namioka.
\newblock {\em Linear topological spaces}.
\newblock Springer-Verlag, New York-Heidelberg, 1976.
\newblock With the collaboration of W. F. Donoghue, Jr., Kenneth R. Lucas, B.
  J. Pettis, Ebbe Thue Poulsen, G. Baley Price, Wendy Robertson, W. R. Scott,
  and Kennan T. Smith, Second corrected printing, Graduate Texts in
  Mathematics, No. 36.

\bibitem[{Lan}01]{lang:2001a}
Serge {Lang}.
\newblock {\em {Fundamentals of differential geometry}}.
\newblock Springer-Verlag, New York, corr. printing 2nd edition, 2001.

\bibitem[Lax02]{lax:2002a}
Peter~D. Lax.
\newblock {\em Functional analysis}.
\newblock Pure and Applied Mathematics (New York). Wiley-Interscience [John
  Wiley \& Sons], New York, 2002.

\bibitem[LM89]{Lawson:1989a}
H.~Blaine Lawson, Jr. and Marie-Louise Michelsohn.
\newblock {\em Spin geometry}, volume~38 of {\em Princeton Mathematical
  Series}.
\newblock Princeton University Press, Princeton, NJ, 1989.

\bibitem[LV03]{Lebedev:2003a}
L.~P. Lebedev and I.~I. Vorovich.
\newblock {\em Functional analysis in mechanics}.
\newblock Springer Monographs in Mathematics. Springer-Verlag, New York, 2003.
\newblock Revised and extended translation of the Russian original.

\bibitem[Mil63]{milnor:1963a}
John Milnor.
\newblock {\em Morse theory}.
\newblock Based on lecture notes by M. Spivak and R. Wells. Annals of
  Mathematics Studies, No. 51. Princeton University Press, Princeton, N.J.,
  1963.

\bibitem[Mil65]{milnor:1965a}
John Milnor.
\newblock {\em Lectures on the {$h$}-cobordism theorem}.
\newblock Notes by L. Siebenmann and J. Sondow. Princeton University Press,
  Princeton, N.J., 1965.

\bibitem[MS74]{milnor:1974a}
John~W. Milnor and James~D. Stasheff.
\newblock {\em Characteristic classes}.
\newblock Princeton University Press, Princeton, N. J., 1974.
\newblock Annals of Mathematics Studies, No. 76.

\bibitem[MS98]{mcduff:1998a}
Dusa McDuff and Dietmar Salamon.
\newblock {\em Introduction to symplectic topology}.
\newblock Oxford Mathematical Monographs. The Clarendon Press, Oxford
  University Press, New York, second edition, 1998.

\bibitem[MS04]{mcduff:2004a}
Dusa McDuff and Dietmar Salamon.
\newblock {\em {$J$}-holomorphic curves and symplectic topology}, volume~52 of
  {\em American Mathematical Society Colloquium Publications}.
\newblock American Mathematical Society, Providence, RI, 2004.

\bibitem[Oxt80]{Oxtoby:1980a}
John~C. Oxtoby.
\newblock {\em Measure and category}, volume~2 of {\em Graduate Texts in
  Mathematics}.
\newblock Springer-Verlag, New York-Berlin, second edition, 1980.
\newblock A survey of the analogies between topological and measure spaces.

\bibitem[PP09]{Palais:2009a}
Richard~S. Palais and Robert~A. Palais.
\newblock {\em Differential equations, mechanics, and computation}, volume~51
  of {\em Student Mathematical Library}.
\newblock American Mathematical Society, Providence, RI; Institute for Advanced
  Study (IAS), Princeton, NJ, 2009.
\newblock IAS/Park City Mathematical Subseries.

\bibitem[Pug02]{pugh:2002a}
Charles~Chapman Pugh.
\newblock {\em Real mathematical analysis.}
\newblock Undergraduate Texts in Mathematics. {Springer, xii+437 p.}, New York,
  2002.

\bibitem[RS80]{reed:1980a}
Michael Reed and Barry Simon.
\newblock {\em Methods of modern mathematical physics. {I}}.
\newblock Academic Press, Inc. [Harcourt Brace Jovanovich, Publishers], New
  York, second edition, 1980.
\newblock Functional analysis.

\bibitem[Rud91]{Rudin:1991b}
Walter Rudin.
\newblock {\em Functional analysis}.
\newblock International Series in Pure and Applied Mathematics. McGraw-Hill,
  Inc., New York, second edition, 1991.

\bibitem[Sal16]{Salamon:2016a}
Dietmar~A. Salamon.
\newblock {\em Measure and integration}.
\newblock EMS Textbooks in Mathematics. European Mathematical Society (EMS),
  Z\"urich, 2016.

\bibitem[SS05]{Stein:2005a}
Elias~M. Stein and Rami Shakarchi.
\newblock {\em Real analysis}, volume~3 of {\em Princeton Lectures in
  Analysis}.
\newblock Princeton University Press, Princeton, NJ, 2005.
\newblock Measure theory, integration, and Hilbert spaces.

\bibitem[Ste83]{Sternberg:1983a}
Shlomo Sternberg.
\newblock {\em Lectures on differential geometry}.
\newblock Chelsea Publishing Co., New York, second edition, 1983.
\newblock With an appendix by Sternberg and Victor W. Guillemin.

\bibitem[Tes16]{Teschl:2016a}
Gerald Teschl.
\newblock Topics in real and functional analysis.\newline
\newblock
  \href{http://www.mat.univie.ac.at/~gerald/ftp/book-fa/}{www.mat.univie.ac.at/$\sim$gerald/ftp/book-fa/},
  2016.

\bibitem[War83]{warner:1983a}
Frank~W. Warner.
\newblock {\em Foundations of differentiable manifolds and {L}ie groups},
  volume~94 of {\em Graduate Texts in Mathematics}.
\newblock Springer-Verlag, New York-Berlin, 1983.
\newblock Corrected reprint of the 1971 edition.

\bibitem[Web]{weber:2015-MORSELEC-In_Preparation}
Joa Weber.
\newblock {Lecture Notes on Morse and Conley Theory}.
\newblock \href{http://www.math.sunysb.edu/~joa/PUBLICATIONS/MORSELEC.pdf}{In
  progress}.

\bibitem[Weh04]{Wehrheim:2004a}
Katrin Wehrheim.
\newblock {\em Uhlenbeck compactness}.
\newblock EMS Series of Lectures in Mathematics. European Mathematical Society
  (EMS), Z\"urich, 2004.

\bibitem[Zim90]{Zimmer:1990a}
Robert~J. Zimmer.
\newblock {\em Essential results of functional analysis}.
\newblock Chicago Lectures in Mathematics. University of Chicago Press,
  Chicago, IL, 1990.

\end{thebibliography}


\begin{thebibliography}{EGH00}

\bibitem[Arn63]{Arnold:1963a}
V.~I. Arnol{$'$}d.
\newblock Proof of a theorem of {A}. {N}. {K}olmogorov on the preservation of
  conditionally periodic motions under a small perturbation of the
  {H}amiltonian.
\newblock {\em Uspehi Mat. Nauk}, 18(5 (113)):13--40, 1963.

\bibitem[Arn76]{Arnold:1976a}
V.~I. Arnol{$'$}d.
\newblock In Felix~E. {Browder}, editor, {\em {Mathematical developments
  arising from Hilbert problems. Proceedings of the symposium in pure
  mathematics of the American Mathematical Society, held at Northern Illinois
  University, DeKalb, Illinois, May 1974.}}, volume XXVIII, page~66, 1976.

\bibitem[Ban93]{Bangert:1993a}
Victor Bangert.
\newblock On the existence of closed geodesics on two-spheres.
\newblock {\em Internat. J. Math.}, 4(1):1--10, 1993.

\bibitem[Bir13]{Birkhoff:1913a}
George~D. Birkhoff.
\newblock Proof of {P}oincar\'e's geometric theorem.
\newblock {\em Trans. Amer. Math. Soc.}, 14(1):14--22, 1913.

\bibitem[CF09]{Cieliebak:2009a}
Kai Cieliebak and Urs~Adrian Frauenfelder.
\newblock A {F}loer homology for exact contact embeddings.
\newblock {\em Pacific J. Math.}, 239(2):251--316, 2009.

\bibitem[CZ83]{Conley:1983a}
C.~C. Conley and E.~Zehnder.
\newblock The {B}irkhoff-{L}ewis fixed point theorem and a conjecture of {V}.
  {I}. {A}rnol$'$d.
\newblock {\em Invent. Math.}, 73(1):33--49, 1983.

\bibitem[EGH00]{Eliashberg:2000a}
Y.~Eliashberg, A.~Givental, and H.~Hofer.
\newblock Introduction to symplectic field theory.
\newblock {\em Geom. Funct. Anal.}, Special Volume, Part II:560--673, 2000.
\newblock GAFA 2000 (Tel Aviv, 1999).

\bibitem[Flo86]{Floer:1986a}
Andreas Floer.
\newblock Proof of the {A}rnol$'$d conjecture for surfaces and generalizations
  to certain {K}\"ahler manifolds.
\newblock {\em Duke Math. J.}, 53(1):1--32, 1986.

\bibitem[Flo88]{floer:1988a}
Andreas Floer.
\newblock Morse theory for {L}agrangian intersections.
\newblock {\em J. Differential Geom.}, 28(3):513--547, 1988.

\bibitem[Flo89]{floer:1989a}
Andreas Floer.
\newblock Symplectic fixed points and holomorphic spheres.
\newblock {\em Comm. Math. Phys.}, 120(4):575--611, 1989.

\bibitem[Fra92]{Franks:1992a}
John Franks.
\newblock Geodesics on {$S^2$} and periodic points of annulus homeomorphisms.
\newblock {\em Invent. Math.}, 108(2):403--418, 1992.

\bibitem[Gro85]{gromov:1985a}
M.~Gromov.
\newblock Pseudo holomorphic curves in symplectic manifolds.
\newblock {\em Invent. Math.}, 82:307--347, 1985.

\bibitem[Ham35]{Hamilton:1835a}
William~Rowan Hamilton.
\newblock Second essay on a general method in dynamics.
\newblock {\em Philosophical Transactions of the Royal Society of London},
  125:95--144, 1835.

\bibitem[Jac09]{Jacobi:2009a}
C.~G.~J. Jacobi.
\newblock {\em Jacobi's lectures on dynamics}, volume~51 of {\em Texts and
  Readings in Mathematics}.
\newblock Hindustan Book Agency, New Delhi, revised edition, 2009.
\newblock Delivered at the University of K\"onigsberg in the winter semester
  1842--1843 and according to the notes prepared by C. W. Brockardt, Edited by
  A. Clebsch, Translated from the original German by K. Balagangadharan,
  Translation edited by Biswarup Banerjee.

\bibitem[Kol54]{Kolmogorov:1954a}
A.~N. Kolmogorov.
\newblock On conservation of conditionally periodic motions for a small change
  in {H}amilton's function.
\newblock {\em Dokl. Akad. Nauk SSSR (N.S.)}, 98:527--530, (Russian) [English
  translation in: Lectures Notes in Physics 93, Springer, 1979.], 1954.

\bibitem[Lef26]{Lefschetz:1926a}
Solomon Lefschetz.
\newblock Intersections and transformations of complexes and manifolds.
\newblock {\em Trans. Amer. Math. Soc.}, 28(1):1--49, 1926.

\bibitem[LS30]{lusternik:1930a}
L.~{Lusternik} and L.~{Schnirelmann}.
\newblock {M\'ethodes topologiques dans les probl\`emes variationnels.}
\newblock {Moskau: Issledowatelskij Institut Mathematiki i Mechaniki pri J. M.
  G. U}, 1930.

\bibitem[Mos62]{Moser:1962a}
J.~Moser.
\newblock On invariant curves of area-preserving mappings of an annulus.
\newblock {\em Nachr. Akad. Wiss. G\"ottingen Math.-Phys. Kl. II}, 1962:1--20,
  1962.

\bibitem[{Nel}16]{2016arXiv161102676N}
J.~{Nelson}.
\newblock {From Dynamics to Contact and Symplectic Topology and Back}.
\newblock {\em \href{https://arxiv.org/abs/1611.02676}{ArXiv e-prints}},
  November 2016.

\bibitem[{Nik}74]{Nikishin:1974a}
N.A. {Nikishin}.
\newblock {Fixed points of diffeomorphisms of two-dimensional spheres
  preserving an oriented plane.}
\newblock {\em {Funct. Anal. Appl.}}, 8:77--79, {(Translation from Funkts.
  Anal. Prilozh. 8, No.1, 84--85 (1974))}, 1974.

\bibitem[{Poi}95]{Poincare:1895a}
H.~{Poincar\'e}.
\newblock {Analysis situs.}
\newblock {J. de l'\'Ec. Pol. (2) I. 1-123}, 1895.

\bibitem[{Poi}12]{Poincare:1912a}
H.~{Poincar\'e}.
\newblock {Sur un th\'{e}or\`{e}me de G\'{e}om\'{e}trie}.
\newblock {\em Rend. Circolo Math. Palermo}, 33:375--407, 1912.

\bibitem[Rab79]{Rabinowitz:1979a}
Paul~H. Rabinowitz.
\newblock Periodic solutions of a {H}amiltonian system on a prescribed energy
  surface.
\newblock {\em J. Differential Equations}, 33(3):336--352, 1979.

\bibitem[Rab86]{rabinowitz:1986a}
Paul~H. Rabinowitz.
\newblock {\em Minimax methods in critical point theory with applications to
  differential equations}, volume~65 of {\em CBMS Regional Conference Series in
  Mathematics}.
\newblock Published for the Conference Board of the Mathematical Sciences,
  Washington, DC; by the American Mathematical Society, Providence, RI, 1986.

\bibitem[Sim74]{Simon:1974a}
Carl~P. Simon.
\newblock A bound for the fixed-point index of an area-preserving map with
  applications to mechanics.
\newblock {\em Invent. Math.}, 26:187--200, 1974.

\bibitem[Wei79]{Weinstein:1979a}
Alan Weinstein.
\newblock On the hypotheses of {R}abinowitz' periodic orbit theorems.
\newblock {\em J. Differential Equations}, 33(3):353--358, 1979.

\end{thebibliography}


\begin{thebibliography}{{Wen}15}

\bibitem[EG91]{Eliashberg:1991a}
Yakov Eliashberg and Mikhael Gromov.
\newblock Convex symplectic manifolds.
\newblock In {\em Several complex variables and complex geometry, {P}art 2
  ({S}anta {C}ruz, {CA}, 1989)}, volume~52 of {\em Proc. Sympos. Pure Math.},
  pages 135--162. Amer. Math. Soc., Providence, RI, 1991.

\bibitem[Eli90]{Eliashberg:1990a}
Yakov Eliashberg.
\newblock Topological characterization of {S}tein manifolds of dimension
  {$>2$}.
\newblock {\em Internat. J. Math.}, 1(1):29--46, 1990.

\bibitem[Etn16]{Etnyre:2006a}
John~B. Etnyre.
\newblock {Contact Manifolds}.
\newblock In Jean-Pierre Fran{\c{c}}oise, Gregory~L. Naber, and Tsou~Sheung
  Tsun, editors, {\em Encyclopedia of mathematical physics. {V}ol. 1, 2, 3, 4,
  5}, volume~1, pages 631--636. Academic Press/Elsevier Science, Oxford, 2006.
  {\href{http://people.math.gatech.edu/~etnyre/preprints/papers/phys.pdf}{Preprint
  (last accessed on 22/10/2016)}}.

\bibitem[Gei01]{Geiges:2001a}
Hansj{\"o}rg Geiges.
\newblock A brief history of contact geometry and topology.
\newblock {\em Expositiones Mathematicae}, 19(1):25 -- 53, 2001.

\bibitem[GG03]{Ginzburg:2003a}
Viktor~L. Ginzburg and Ba{\c{s}}ak~Z. G{\"u}rel.
\newblock A {$C^2$}-smooth counterexample to the {H}amiltonian {S}eifert
  conjecture in {$\mathbb R^4$}.
\newblock {\em Ann. of Math. (2)}, 158(3):953--976, 2003.

\bibitem[Gin01]{Ginzburg:2001a}
Viktor~L. Ginzburg.
\newblock The {H}amiltonian {S}eifert conjecture: examples and open problems.
\newblock In {\em European {C}ongress of {M}athematics, {V}ol. {II}
  ({B}arcelona, 2000)}, volume 202 of {\em Progr. Math.}, pages 547--555.
  Birkh\"auser, Basel, 2001.

\bibitem[Hof93]{Hofer:1993a}
H.~Hofer.
\newblock Pseudoholomorphic curves in symplectizations with applications to the
  {W}einstein conjecture in dimension three.
\newblock {\em Invent. Math.}, 114(3):515--563, 1993.

\bibitem[Hut10]{Hutchings:2010a}
Michael Hutchings.
\newblock Taubes's proof of the {W}einstein conjecture in dimension three.
\newblock {\em Bull. Amer. Math. Soc. (N.S.)}, 47(1):73--125, 2010.

\bibitem[Lim88]{Lima:1988a}
Elon~L. Lima.
\newblock The {J}ordan-{B}rouwer separation theorem for smooth hypersurfaces.
\newblock {\em Amer. Math. Monthly}, 95(1):39--42, 1988.

\bibitem[PR83]{Pugh:1983a}
Charles~C. Pugh and Clark Robinson.
\newblock The {$C^{1}$} closing lemma, including {H}amiltonians.
\newblock {\em Ergodic Theory Dynam. Systems}, 3(2):261--313, 1983.

\bibitem[Rab78]{Rabinowitz:1978a}
Paul~H. Rabinowitz.
\newblock Periodic solutions of {H}amiltonian systems.
\newblock {\em Comm. Pure Appl. Math.}, 31(2):157--184, 1978.

\bibitem[Tau07]{Taubes:2007a}
Clifford~Henry Taubes.
\newblock The {S}eiberg-{W}itten equations and the {W}einstein conjecture.
\newblock {\em Geom. Topol.}, 11:2117--2202, 2007.

\bibitem[Vit87]{Viterbo:1987a}
Claude Viterbo.
\newblock A proof of {W}einstein's conjecture in {${\bf R}^{2n}$}.
\newblock {\em Ann. Inst. H. Poincar\'e Anal. Non Lin\'eaire}, 4(4):337--356,
  1987.

\bibitem[Wei78]{Weinstein:1978a}
Alan Weinstein.
\newblock Periodic orbits for convex {H}amiltonian systems.
\newblock {\em Ann. of Math. (2)}, 108(3):507--518, 1978.

\bibitem[{Wen}15]{Wendl:2015a}
C.~{Wendl}.
\newblock {Lectures on Holomorphic Curves in Symplectic and Contact Geometry}.
\newblock {Manuscript on}
  \href{http://www.homepages.ucl.ac.uk/~ucahcwe/pub/jhol_bookv33.pdf}{Webpage},
  last accessed on 15/09/2016, v3.3, May 2015.

\bibitem[Zeh87]{Zehnder:1987b}
E.~Zehnder.
\newblock Remarks on periodic solutions on hypersurfaces.
\newblock In {\em Periodic solutions of {H}amiltonian systems and related
  topics ({I}l {C}iocco, 1986)}, volume 209 of {\em NATO Adv. Sci. Inst. Ser. C
  Math. Phys. Sci.}, pages 267--279. Reidel, Dordrecht, 1987.

\end{thebibliography}


\begin{thebibliography}{{Sma}65}

\bibitem[Abo11]{Abouzaid:2011a}
Mohammed Abouzaid.
\newblock A cotangent fibre generates the {F}ukaya category.
\newblock {\em Adv. Math.}, 228(2):894--939, 2011.

\bibitem[Abo15]{Abouzaid:2015a}
Mohammed Abouzaid.
\newblock Symplectic cohomology and {V}iterbo's theorem.
\newblock In {\em Free loop spaces in geometry and topology}, volume~24 of {\em
  IRMA Lect. Math. Theor. Phys.}, pages 271--485. Eur. Math. Soc., Z\"urich,
  2015.

\bibitem[AM06]{abbondandolo:2006a}
Alberto Abbondandolo and Pietro Majer.
\newblock Lectures on the {M}orse complex for infinite-dimensional manifolds.
\newblock In {\em Morse theoretic methods in nonlinear analysis and in
  symplectic topology}, volume 217 of {\em NATO Sci. Ser. II Math. Phys.
  Chem.}, pages 1--74. Springer, Dordrecht, 2006.

\bibitem[AS06]{abbondandolo:2006b}
Alberto Abbondandolo and Matthias Schwarz.
\newblock On the {F}loer homology of cotangent bundles.
\newblock {\em Comm. Pure Appl. Math.}, 59(2):254--316, 2006.

\bibitem[AS14]{abbondandolo:2014b}
Alberto Abbondandolo and Matthias Schwarz.
\newblock Corrigendum: {O}n the {F}loer homology of cotangent bundles
\newblock {\em Comm. Pure Appl. Math.}, 67(4):670--691, 2014.

\bibitem[AS15]{Abbondandolo:2015c}
Alberto Abbondandolo and Matthias Schwarz.
\newblock The role of the {L}egendre transform in the study of the {F}loer
  complex of cotangent bundles.
\newblock {\em Comm. Pure Appl. Math.}, 68(11):1885--1945, 2015.

\bibitem[Bot56]{Bott:1956a}
Raoul Bott.
\newblock On the iteration of closed geodesics and the {S}turm intersection
  theory.
\newblock {\em Comm. Pure Appl. Math.}, 9:171--206, 1956.

\bibitem[El{\u\i}67]{Eluiasson:1967a}
Halld{\'o}r~I. El{\u\i}asson.
\newblock Geometry of manifolds of maps.
\newblock {\em J. Differential Geometry}, 1:169--194, 1967.

\bibitem[FHS95]{Floer:1995a}
Andreas Floer, Helmut Hofer, and Dietmar Salamon.
\newblock Transversality in elliptic {M}orse theory for the symplectic action.
\newblock {\em Duke Math. J.}, 80(1):251--292, 1995.

\bibitem[Flo89]{floer:1989c}
Andreas Floer.
\newblock Witten's complex and infinite-dimensional {M}orse theory.
\newblock {\em J. Differential Geom.}, 30(1):207--221, 1989.

\bibitem[FO99]{Fukaya:1999a}
Kenji Fukaya and Kaoru Ono.
\newblock Arnold conjecture and {G}romov-{W}itten invariant.
\newblock {\em Topology}, 38(5):933--1048, 1999.

\bibitem[GG15]{Ginzburg:2015b}
Viktor~L. Ginzburg and Ba\c sak~Z. G{\"u}rel.
\newblock The {C}onley conjecture and beyond.
\newblock {\em Arnold Math. J.}, 1(3):299--337, 2015.

\bibitem[Gin10]{Ginzburg:2010b}
Viktor~L. Ginzburg.
\newblock The {C}onley conjecture.
\newblock {\em Ann. of Math. (2)}, 172(2):1127--1180, 2010.

\bibitem[Gom98]{Gompf:1998a}
Robert~E. Gompf.
\newblock Symplectically aspherical manifolds with nontrivial {$\pi_2$}.
\newblock {\em Math. Res. Lett.}, 5(5):599--603, 1998.

\bibitem[Hof85]{Hofer:1985a}
Helmut Hofer.
\newblock Lagrangian embeddings and critical point theory.
\newblock {\em Ann. Inst. H. Poincar\'e Anal. Non Lin\'eaire}, 2(6):407--462,
  1985.

\bibitem[HS95]{Hofer:1995a}
H.~Hofer and D.~A. Salamon.
\newblock Floer homology and {N}ovikov rings.
\newblock In {\em The {F}loer memorial volume}, volume 133 of {\em Progr.
  Math.}, pages 483--524. Birkh\"auser, Basel, 1995.

\bibitem[Kak43]{Kakutani:1943a}
Shizuo Kakutani.
\newblock Topological properties of the unit sphere of a {H}ilbert space.
\newblock {\em Proc. Imp. Acad. Tokyo}, 19:269--271, 1943.

\bibitem[{Kra}07]{2007arXiv0712.2533K}
T.~{Kragh}.
\newblock {The Viterbo Transfer as a Map of Spectra}.
\newblock {\em ArXiv e-prints}, December 2007.

\bibitem[Kra13]{Kragh:2013a}
Thomas Kragh.
\newblock Parametrized ring-spectra and the nearby {L}agrangian conjecture.
\newblock {\em Geom. Topol.}, 17(2):639--731, 2013.
\newblock With an appendix by Mohammed Abouzaid.

\bibitem[KRT08]{Kedra:2008a}
Jarek K\c{e}dra, Yuli Rudyak, and Aleksy Tralle.
\newblock Symplectically aspherical manifolds.
\newblock {\em J. Fixed Point Theory Appl.}, 3(1):1--21, 2008.

\bibitem[Kui65]{Kuiper:1965a}
Nicolaas~H. Kuiper.
\newblock The homotopy type of the unitary group of {H}ilbert space.
\newblock {\em Topology}, 3:19--30, 1965.

\bibitem[Lon02]{Long:2002a}
Yiming Long.
\newblock {\em Index theory for symplectic paths with applications}, volume 207
  of {\em Progress in Mathematics}.
\newblock Birkh\"auser Verlag, Basel, 2002.

\bibitem[LT98]{Liu:1998a}
Gang Liu and Gang Tian.
\newblock Floer homology and {A}rnold conjecture.
\newblock {\em J. Differential Geom.}, 49(1):1--74, 1998.

\bibitem[Mos76]{Moser:1976a}
J.~Moser.
\newblock Periodic orbits near an equilibrium and a theorem by {A}lan
  {W}einstein.
\newblock {\em Comm. Pure Appl. Math.}, 29(6):724--747, 1976.

\bibitem[MW10]{Mawhin:2010a}
Jean Mawhin and Michel Willem.
\newblock {Origin and evolution of the Palais--Smale condition in critical
  point theory}.
\newblock {\em Journal of Fixed Point Theory and Applications}, 7(2):265--290,
  2010.

\bibitem[Pal66]{Palais:1966b}
Richard~S. Palais.
\newblock Lusternik-{S}chnirelman theory on {B}anach manifolds.
\newblock {\em Topology}, 5:115--132, 1966.

\bibitem[PSS96]{Piunikhin:1996a}
S.~Piunikhin, D.~Salamon, and M.~Schwarz.
\newblock Symplectic {F}loer-{D}onaldson theory and quantum cohomology.
\newblock In {\em Contact and symplectic geometry ({C}ambridge, 1994)},
  volume~8 of {\em Publ. Newton Inst.}, pages 171--200. Cambridge Univ. Press,
  Cambridge, 1996.

\bibitem[RS95]{robbin:1995a}
Joel {Robbin} and Dietmar {Salamon}.
\newblock {The spectral flow and the Maslov index.}
\newblock {\em {Bull. Lond. Math. Soc.}}, 27(1):1--33, 1995.

\bibitem[Sal90]{salamon:1990a}
Dietmar Salamon.
\newblock Morse theory, the {C}onley index and {F}loer homology.
\newblock {\em Bull. London Math. Soc.}, 22(2):113--140, 1990.

\bibitem[{Sal}99a]{salamon:1999a}
Dietmar {Salamon}.
\newblock {Lectures on Floer homology}.
\newblock In {\em {Symplectic geometry and topology. Lecture notes from the
  graduate summer school program, Park City, UT, USA, June 29--July 19, 1997}},
  pages 145--229. Providence, RI: American Mathematical Society, 1999.

\bibitem[Sal99b]{salamon:1999b}
Dietmar Salamon.
\newblock {\em {Spin Geometry and Seiberg-Witten Invariants}}.
\newblock
  \href{https://people.math.ethz.ch/~salamon/PREPRINTS/witsei.pdf}{Unpublished
  manuscript}, 1999.

\bibitem[Sar42]{Sard:1942a}
Arthur Sard.
\newblock The measure of the critical values of differentiable maps.
\newblock {\em Bull. Amer. Math. Soc.}, 48:883--890, 1942.

\bibitem[Sch93]{schwarz:1993a}
Matthias Schwarz.
\newblock {\em Morse homology}, volume 111 of {\em Progress in Mathematics}.
\newblock Birkh\"auser Verlag, Basel, 1993.

\bibitem[Sei10]{Seidel:2010a}
Paul Seidel.
\newblock {A remark on the symplectic cohomology of cotangent bundles, after
  Kragh}.
\newblock Unpublished notes, 2010.

\bibitem[{Sma}65]{smale:1965a}
Stephen {Smale}.
\newblock {An infinite dimensional version of Sard's theorem.}
\newblock {\em {Am. J. Math.}}, 87:861--866, 1965.

\bibitem[Ste70]{Stein:1970a}
Elias~M. Stein.
\newblock {\em Singular integrals and differentiability properties of
  functions}.
\newblock Princeton Mathematical Series, No. 30. Princeton University Press,
  Princeton, N.J., 1970.

\bibitem[SW06]{salamon:2006a}
Dietmar Salamon and Joa Weber.
\newblock Floer homology and the heat flow.
\newblock {\em Geom. Funct. Anal.}, 16(5):1050--1138, 2006.

\bibitem[SZ92]{Salamon:1992a}
Dietmar Salamon and Eduard Zehnder.
\newblock Morse theory for periodic solutions of {H}amiltonian systems and the
  {M}aslov index.
\newblock {\em Comm. Pure Appl. Math.}, 45(10):1303--1360, 1992.

\bibitem[Vit98]{Viterbo:1998a}
Claude Viterbo.
\newblock {Functors and computations in Floer homology with applications, II.}
\newblock {Preprint Universit{\'e} Paris-Sud no. 98-15}, 1998.

\bibitem[Was69]{Wasserman:1969a}
Arthur~G. Wasserman.
\newblock Equivariant differential topology.
\newblock {\em Topology}, 8:127--150, 1969.

\bibitem[Web93]{weber:1993a}
Joa Weber.
\newblock Der {M}orse-{W}itten {K}omplex
  \href{http://www.math.sunysb.edu/~joa/PUBLICATIONS/1993-madip.pdf}{\rm
  (access pdf)}.
\newblock Master's thesis, {Department of Mathematics, TU~Berlin}, February
  1993.

\bibitem[Web96]{Weber:1996a}
Joa Weber.
\newblock {Morse theory on the loop space of flat tori and symplectic Floer
  theory}.
\newblock In {\em eprint \href{https://arxiv.org/abs/dg-ga/9612012}{{\rm
  arXiv:dg-ga/9612012}}}, December 1996.

\bibitem[Web02]{weber:2002a}
Joa Weber.
\newblock {Perturbed closed geodesics are periodic orbits: Index and
  transversality}.
\newblock {\em Math. Z.}, 241(1):45--82, 2002.

\bibitem[Web05]{weber:2005a}
Joa Weber.
\newblock {T}hree approaches towards {F}loer homology of cotangent bundles.
  {Conference on Symplectic Topology}.
\newblock {\em J. Symplectic Geom.}, 3(4):671--701, 2005.
\newblock 

\bibitem[Web06a]{weber:2006b}
Joa Weber.
\newblock The {M}orse-{W}itten complex via dynamical systems.
\newblock {\em Expo. Math.}, 24(2):127--159, 2006.

\bibitem[Web06b]{weber:2006a}
Joa Weber.
\newblock {Noncontractible periodic orbits in cotangent bundles and Floer
  homology.}
\newblock {\em Duke Math. J.}, 133(3):527--568, 2006.

\bibitem[Web13a]{weber:2013b}
Joa Weber.
\newblock {M}orse homology for the heat flow.
\newblock {\em Math. Z.}, 275(1-2):1--54, 2013.

\bibitem[Web13b]{weber:2013a}
Joa Weber.
\newblock {M}orse homology for the heat flow -- {L}inear theory.
\newblock {\em Math. Nachr.}, 286(1):88--104, 2013.

\bibitem[Web15]{Weber:2015c}
Joa Weber.
\newblock Contraction method and {L}ambda-{L}emma.
\newblock {\em S{\~a}o Paulo Journal of Mathematical Sciences}, 9(2):263--298,
  2015.

\bibitem[Web17]{weber:2014c}
Joa Weber.
\newblock {S}table foliations and semi-flow {M}orse homology.
\newblock {\em Ann. Sc. Norm. Super. Pisa Cl. Sci. (5)}, Vol. XVII(3):853--909,
  2017.

\end{thebibliography}


\newcommand{\etalchar}[1]{$^{#1}$}
\begin{thebibliography}{BEH{\etalchar{+}}03}

\bibitem[Abb13]{Abbondandolo:2013c}
Alberto Abbondandolo.
\newblock Lectures on the free period {L}agrangian action functional.
\newblock {\em J. Fixed Point Theory Appl.}, 13(2):397--430, 2013.

\bibitem[AF10a]{Albers:2010b}
Peter Albers and Urs Frauenfelder.
\newblock Leaf-wise intersections and {R}abinowitz {F}loer homology.
\newblock {\em J. Topol. Anal.}, 2(1):77--98, 2010.

\bibitem[AF10b]{Albers:2010c}
Peter Albers and Urs Frauenfelder.
\newblock Spectral invariants in {R}abino-witz-{F}loer homology and global
  {H}amiltonian perturbations.
\newblock {\em J. Mod. Dyn.}, 4(2):329--357, 2010.

\bibitem[AF12a]{Albers:2012b}
Peter Albers and Urs Frauenfelder.
\newblock Infinitely many leaf-wise intersections on cotangent bundles.
\newblock {\em Expo. Math.}, 30(2):168--181, 2012.

\bibitem[AF12b]{Albers:2012a}
Peter Albers and Urs Frauenfelder.
\newblock Rabinowitz {F}loer homology: a survey.
\newblock In {\em Global differential geometry}, volume~17 of {\em Springer
  Proc. Math.}, pages 437--461. Springer, Heidelberg, 2012.

\bibitem[AS09]{Abbondandolo:2009a}
Alberto Abbondandolo and Matthias Schwarz.
\newblock Estimates and computations in {R}abinowitz-{F}loer homology.
\newblock {\em J. Topol. Anal.}, 1(4):307--405, 2009.

\bibitem[Ban80]{Banyaga:1980a}
Augustin Banyaga.
\newblock On fixed points of symplectic maps.
\newblock {\em Invent. Math.}, 56(3):215--229, 1980.

\bibitem[BEH{\etalchar{+}}03]{Bourgeois:2003a}
F.~Bourgeois, Y.~Eliashberg, H.~Hofer, K.~Wysocki, and E.~Zehnder.
\newblock Compactness results in symplectic field theory.
\newblock {\em Geom. Topol.}, 7:799--888, 2003.

\bibitem[BF11]{Bae:2011a}
Youngjin Bae and Urs Frauenfelder.
\newblock Continuation homomorphism in {R}abinowitz {F}loer homology for
  symplectic deformations.
\newblock {\em Math. Proc. Cambridge Philos. Soc.}, 151(3):471--502, 2011.

\bibitem[CFO10]{Cieliebak:2010b}
Kai Cieliebak, Urs Frauenfelder, and Alexandru Oancea.
\newblock Rabinowitz {F}loer homology and symplectic homology.
\newblock {\em Ann. Sci. \'Ec. Norm. Sup\'er. (4)}, 43(6):957--1015, 2010.

\bibitem[Fra04]{Frauenfelder:2004a}
Urs Frauenfelder.
\newblock The {A}rnold-{G}ivental conjecture and moment {F}loer homology.
\newblock {\em Int. Math. Res. Not.}, 42:2179--2269, 2004.

\bibitem[FS16]{Frauenfelder:2016a}
Urs Frauenfelder and Felix Schlenk.
\newblock {$S^1$}-equivariant {R}abinowitz-{F}loer homology.
\newblock {\em Hokkaido Math. J.}, 45(3):293--323, 2016.

\bibitem[FZ17]{Frauenfelder:2017a}
U.~{Frauenfelder} and L.~{Zhao}.
\newblock {Existence of infinitely many consecutive collision orbits in the
  planar circular restricted three-body problem}.
\newblock {\em ArXiv e-prints}, May 2017.

\bibitem[Mer11]{Merry:2011a}
Will~J. Merry.
\newblock On the {R}abinowitz {F}loer homology of twisted cotangent bundles.
\newblock {\em Calc. Var. Partial Differential Equations}, 42(3-4):355--404,
  2011.

\bibitem[Mos78]{Moser:1978a}
J.~Moser.
\newblock A fixed point theorem in symplectic geometry.
\newblock {\em Acta Math.}, 141(1--2):17--34, 1978.

\bibitem[MP11]{Merry:2011b}
Will~J. Merry and Gabriel~P. Paternain.
\newblock Index computations in {R}abinowitz {F}loer homology.
\newblock {\em J. Fixed Point Theory Appl.}, 10(1):87--111, 2011.

\bibitem[Sch06]{Schlenk:2006a}
Felix Schlenk.
\newblock Applications of {H}ofer's geometry to {H}amiltonian dynamics.
\newblock {\em Comment. Math. Helv.}, 81(1):105--121, 2006.

\end{thebibliography}


\begin{thebibliography}{BRCF05}

\bibitem[AG01]{Arnold:2001b}
V.~I. Arnol{$'$}d and A.~B. Givental{$'$}.
\newblock Symplectic geometry.
\newblock In {\em Dynamical systems, {IV}}, volume~4 of {\em Encyclopaedia
  Math. Sci.}, pages 1--138. Springer, Berlin, 2001.

\bibitem[AKN06]{arnold:2006a}
Vladimir~I. Arnol$'$d, Valery~V. Kozlov, and Anatoly~I. Neishtadt.
\newblock {\em Mathematical aspects of classical and celestial mechanics},
  volume~3 of {\em Encyclopaedia of Mathematical Sciences}.
\newblock Springer-Verlag, Berlin, third edition, 2006.
\newblock [Dynamical systems. III], Translated from the Russian original by E.
  Khukhro.

\bibitem[Arn67]{Arnold:1967a}
V.~I. Arnol{$'$}d.
\newblock On a characteristic class entering into conditions of quantization.
\newblock {\em Funkcional. Anal. i Prilo\v zen.}, 1:1--14, 1967.

\bibitem[Bot85]{Bott:1985a}
Raoul Bott.
\newblock On some recent interactions between mathematics and physics.
\newblock {\em Canad. Math. Bull.}, 28(2):129--164, 1985.

\bibitem[BRCF05]{Barros:2005a}
Manuel Barros, Alfonso Romero, Jos\'e~L. Cabrerizo, and Manuel Fern\'andez.
\newblock The {G}auss-{L}andau-{H}all problem on {R}iemannian surfaces.
\newblock {\em J. Math. Phys.}, 46(11):112905, 15, 2005.

\bibitem[CMP04]{Contreras:2004a}
Gonzalo Contreras, Leonardo Macarini, and Gabriel~P. Paternain.
\newblock Periodic orbits for exact magnetic flows on surfaces.
\newblock {\em Int. Math. Res. Not.}, (8):361--387, 2004.

\bibitem[CZ84]{conley:1984a}
Charles~C. Conley and Eduard Zehnder.
\newblock Morse-type index theory for flows and periodic solutions for
  {H}amiltonian equations.
\newblock {\em Comm. Pure Appl. Math.}, 37(2):207--253, 1984.

\bibitem[Des81]{Deschamps:1981a}
G.~A. Deschamps.
\newblock Electromagnetics and differential forms.
\newblock {\em Proceedings of the IEEE}, 69(6):676--696, June 1981.

\bibitem[FLS64]{Feynman:1964a}
Richard~P. Feynman, Robert~B. Leighton, and Matthew Sands.
\newblock {\em The {F}eynman lectures on physics. {V}ol. 2: {M}ainly
  electromagnetism and matter}.
\newblock Addison-Wesley Publishing Co., Inc., Reading, Mass.-London, 1964.

\bibitem[Gin96]{Ginzburg:1996a}
Viktor~L. Ginzburg.
\newblock On closed trajectories of a charge in a magnetic field. {A}n
  application of symplectic geometry.
\newblock In {\em Contact and symplectic geometry ({C}ambridge, 1994)},
  volume~8 of {\em Publ. Newton Inst.}, pages 131--148. Cambridge Univ. Press,
  Cambridge, 1996.

\bibitem[Gin01]{Ginzburg:2001b}
Viktor~L. Ginzburg.
\newblock {A charge in a magnetic field: Arnol$'$d's problems 1981-9, 1982-24,
  1984-4, 1994-14, 1994-35, 1996-17, and 1996-18, last accessed 15/10/2016 on}
  \href{http://ginzburg.math.ucsc.edu/ARNOLD/mag-post.pdf}{Webpage}, preprint,
  December 2001.

\bibitem[GL58]{Gelfand:1958a}
I.~M. Gel{$'$}fand and V.~B. Lidski{\u\i}.
\newblock On the structure of the regions of stability of linear canonical
  systems of differential equations with periodic coefficients.
\newblock {\em Amer. Math. Soc. Transl. (2)}, 8:143--181, 1958.

\bibitem[Gom01]{Gompf:2001a}
Robert~E. Gompf.
\newblock The topology of symplectic manifolds.
\newblock {\em Turkish J. Math.}, 25(1):43--59, 2001.

\bibitem[Gut14]{Gutt:2014a-arXiv-link}
Jean Gutt.
\newblock Generalized {C}onley-{Z}ehnder index.
\newblock {\em Ann. Fac. Sci. Toulouse Math. (6)}, 23(4):907--932, (cf.
  \href{http://www.math.sunysb.edu/~joa/PUBLICATIONS/1999--prom.pdf}{arXiv:1201.3728}),
  2014.

\bibitem[HK99]{Hofer:1999a}
Helmut Hofer and Markus Kriener.
\newblock Holomorphic curves in contact dynamics.
\newblock In {\em Differential equations: {L}a {P}ietra 1996 ({F}lorence)},
  volume~65 of {\em Proc. Sympos. Pure Math.}, pages 77--131. Amer. Math. Soc.,
  Providence, RI, 1999.

\bibitem[HMSa15]{Hryniewicz:2015a}
Umberto Hryniewicz, Al~Momin, and Pedro A.~S. Salom\~ao.
\newblock A {P}oincar\'e-{B}irkhoff theorem for tight {R}eeb flows on {$S^3$}.
\newblock {\em Invent. Math.}, 199(2):333--422, 2015.

\bibitem[HWZ95]{Hofer:1995b}
H.~Hofer, K.~Wysocki, and E.~Zehnder.
\newblock Properties of pseudo-holomorphic curves in symplectisations. {II}.
  {E}mbedding controls and algebraic invariants.
\newblock {\em Geom. Funct. Anal.}, 5(2):270--328, 1995.

\bibitem[HWZ03]{Hofer:2003a}
H.~Hofer, K.~Wysocki, and E.~Zehnder.
\newblock Finite energy foliations of tight three-spheres and {H}amiltonian
  dynamics.
\newblock {\em Ann. of Math. (2)}, 157(1):125--255, 2003.

\bibitem[Maz12]{mazzucchelli:2012a}
Marco Mazzucchelli.
\newblock {\em Critical point theory for {L}agrangian systems}, volume 293 of
  {\em Progress in Mathematics}.
\newblock Birkh{\"a}user/Springer Basel AG, Basel, 2012.

\bibitem[RS93]{Robbin:1993a}
Joel Robbin and Dietmar Salamon.
\newblock The {M}aslov index for paths.
\newblock {\em Topology}, 32(4):827--844, 1993.

\bibitem[Sal13]{Salamon:2013a}
Dietmar Salamon.
\newblock Uniqueness of symplectic structures.
\newblock {\em Acta Math. Vietnam.}, 38(1):123--144, 2013.

\bibitem[Web99]{weber:1999a}
Joa Weber.
\newblock {\em J-holomorphic curves in cotangent bundles and the heat flow
  \href{http://www.math.sunysb.edu/~joa/PUBLICATIONS/1999-prom.pdf}{\rm (access
  pdf)}}.
\newblock PhD thesis, TU Berlin, June 1999.

\bibitem[WR14]{Warnick:2014a}
K.F. Warnick and P.H. Russer.
\newblock Differential forms and electromagnetic field theory (invited paper).
\newblock {\em Progress In Electromagnetics Research}, 148:83--112, 2014.

\end{thebibliography}

\cleardoublepage
\phantomsection
\addcontentsline{toc}{chapter}{Index}
\printindex

\end{document}